\PassOptionsToPackage{svgnames}{xcolor}
\documentclass[11pt,fleqn,twoside]{book}
\usepackage{style}
\usepackage{commands}
\makeindex
\newcommand{\thesistitle}{
    On Integrable Systems\\ \& \\Rigidity for PDEs with Symmetry
}
\newcommand{\dutchtitle}{
    Over Integreerbare Systemen\\ \& \\Rigiditeit voor PDV's met Symmetrie
}
\newcommand{\thesisauthor}{
    Roy Wang
}
\newcommand{\thesisdate}{
    13 september 2017
}
\title{\thesistitle}
\author{\thesisauthor}
\date{\thesisdate}

\begin{document}

\frontmatter

\thispagestyle{empty}
\begin{center}
    \vspace*{0.2\textheight}
    \begin{minipage}{\textwidth}
        \begin{center}
            \renewcommand{\baselinestretch}{1.2}
            \bfseries\LARGE\thesistitle
        \end{center}
    \end{minipage}
\end{center}

\clearpage

\thispagestyle{empty}
\noindent $\phantom{0}$
\vfill

\noindent On Integrable Systems \& Rigidity for PDEs with Symmetry, \thesisauthor.\\
\noindent Ph.D. Thesis, Utrecht University, \thesisdate\\
\noindent Printed by Ipskamp Printing \\
\noindent ISBN: 978-90-393-6850-3\\
\vspace{0.2cm}

\clearpage

\thispagestyle{empty}
\begin{center}
    \vspace*{3em}
    
    \begin{minipage}{\textwidth}
    \begin{center}
        \renewcommand{\baselinestretch}{1.2}
        \bfseries\LARGE\thesistitle
    \end{center}
    \end{minipage}

    \vspace{5em}

    \begin{minipage}{\textwidth}
    \begin{center}
        \renewcommand{\baselinestretch}{1.2}
        \bfseries\large\dutchtitle
    \end{center}
    \end{minipage}

    \vspace{\baselineskip} (met een samenvatting in het Nederlands)
    \vfill
    {\Large Proefschrift}
    \vspace{3em}

    \begin{minipage}{0.95\textwidth}
    \begin{center}
        ter verkrijging van de graad van doctor aan de Universiteit Utrecht op gezag van
        de rector magnificus, prof.\,dr.\ G.J.\ van der Zwaan,
        ingevolge het  besluit van
        het college voor promoties in het openbaar te verdedigen op
        woensdag 13 september 2017 des ochtends te 10.30 uur
    \end{center}
    \end{minipage}
    \vspace{\baselineskip}
    door 
    \vspace{\baselineskip}
    {\Large Roy Lee Marvin Wang} 
    \vspace{\baselineskip}
    geboren op 20 november 1987 te Ede
    \vspace{\baselineskip}
\end{center}

\clearpage

\thispagestyle{empty}
\begin{tabular}{ll}
    Promotor:  & Prof.\ dr. M.N.~Crainic \\
    Copromotor: & Dr. I.T.~M\u{a}rcu\cb{t} \\
\end{tabular}

\vfill
\cleardoublepage

\setcounter{secnumdepth}{2}
\etocsettocdepth{1}
\tableofcontents
\etocsettocstyle{\subsubsection*{Local Contents}}{}

\mainmatter
\chapter{Background, motivation \& outline}
\label{chapter:Intr-and-outline}
\etocsettocdepth{2}
\localtableofcontents

{\color{white}.}
\\
\noindent
In this introductory chapter we explain some theory leading up to the motivation and outline of this thesis. 
In the first section we describe several different view points on completely integrable systems. 
In the second section we discuss Eliasson's theorem, 
which is one of many normal form theorems motivating our research. 
Most of this is standard, 
except for a minor new result (Proposition \ref{prop:nondeg_function}). 
In the last two sections we motivate and outline the rest of the thesis, 
as their titles suggest.

\section{Integrable systems and underlying structures}
Let $M^{2n}$ be a smooth connected manifold. For us, a manifold is always second countable and Hausdorff (unless stated otherwise).

\begin{definition}
A \textbf{completely integrable system}\index{integrable system!completely} on ${M^{2n}}$ is a pair ${(\omega,\mu)}$ consisting of a symplectic structure ${\omega}$, and a smooth map 
    \begin{equation}\label{def:cis}
        \mu
        =
        (\mu_1,\ldots, \mu_n)
        :
        M\to \R^n
    \end{equation}
such that the ${\mu_i}$ are pairwise in \emph{involution} with respect to the Poisson bracket ${\set{\,,\,}_\omega}$ of ${\omega}$:
\begin{equation}
    \set{ 
        \mu_i, \mu_j 
    }_\omega 
    =0,
    \quad 
    \forall\; i,j,
\end{equation}
and the ${\mu_i}$ are \emph{independent}, i.e.\  
$
    {{\dd\mu_1\wedge\ldots\wedge \dd\mu_n}}
$
is nonvanishing on a dense \emph{open} set.
Moreover, by an \textbf{integrable system} we mean a pair $(\omega, \mu)$ which satisfies the involution condition but may fail the independence condition.
\end{definition}
\noindent
\noindent
The second condition means that ${\dd_x\mu}$ is of maximal rank on a dense open set of points ${x}$. Such points ${x}$ are called regular (and the condition means that the set of regular points contains a dense open set).
Note that some authors \cite{PelaNgoc11, MN05} do not require the openess in the definition of completeness, but instead only require that the functions are independent on a dense set.
More classical works (such as Duistermaat \cite{Dui80}) study the case where $\mu$ is a submersion. 
We require the set of regular points to contain a dense open set, and the above definition is common in the literature (see for example \cite{CaMiVa11}).

The nonregular points are called singular. At the extreme case of singular points one finds the \textbf{fixed} points: the ones at which ${(\dd\mu)_x= 0}$.
A large part of the theory of integrable systems is devoted to the study of local normal forms around singular points \cite{Eli90, Zun96, Zun97, Mir03, Mir14, Cha13, NgocWa13}. We will stare at the definition of integrable systems, with the aim of revealing some of the hidden structure, and their relevance to normal forms: 
infinitesimal actions, Hamiltonian actions, Poisson structures, and singular foliations. 

Concerning the local study of integrable systems (neighborhoods of points), the Darboux-Carathéodory and Arnol'd-Liouville theorems \cite{Arnie} fully describe integrable systems near regular points by providing a normalized coordinate system (so-called `action-angle' coordinates).
Moreover, Duistermaat \cite{Dui80} generalized this result to the semi-local case (neighborhoods of compact orbits).
The local study near singular points is richer (there are multiple normal forms and equivalence relations), and we will focus on this.
Among singular points, there are the fixed points, at which all the differentials ${(d\mu_i)_x}$ vanish. Moreover, among fixed points, there 
are the so called nondegenerate ones, at which the second jet of ${\mu}$ has a rather rigid behavior (see Definition \ref{def-weired-nondegen}).

Here is the list of the main examples of such nondegenerate singularities of integrable systems,
as they arise from Williamson's classification \cite{williamson, williamson40} (for more details, see Remark \ref{rk:Williamson} below).

\begin{myenumerate}
\item The \textbf{elliptic}\index{normal model!elliptic} normal form on ${\R^2}$ with ${\omega_\text{can}}$ is   
    \begin{align*}
        \mu_\text{ell}(x,y)
        & \defeq
        \tfrac12 (x^2+y^2),
        \\
        X_\text{ell}
        & \defeq
        X^\omega_{\mu_\text{ell}}
        =
        y \tfrac{\partial}{\partial x}
        -
        x \tfrac{\partial}{\partial y}.
    \end{align*}
\item The \textbf{hyperbolic}\index{normal model!hyperbolic} normal form on ${\R^2}$ with ${\omega_\text{can}}$ is
    \begin{align*}
        \mu_\text{hyp}(x,y)
        & \defeq
        xy,
        \\
        X_\text{hyp}
        & \defeq
        X^\omega_{\mu_\text{hyp}}
        =
        x \tfrac{\partial}{\partial x}
        -
        y \tfrac{\partial}{\partial y}.
    \end{align*}
\item The \textbf{focus-focus}\index{normal model!focus-focus} normal form on ${\R^4}$ with ${\omega_\text{can}}$ is
    \begin{align*}
        \mu_{\text{ff},1}(x,y)
        & \defeq
        x_1y_1+x_2y_2,
        \\
        X_{\text{ff},1}
        & \defeq
        X_{\mu_{\text{ff},1}}^{\omega_\text{can}}
        =
        -
        x_1 \tfrac{\partial}{\partial x_1}
        +
        y_1 \tfrac{\partial}{\partial y_1}
        -
        x_2 \tfrac{\partial}{\partial x_2}
        +
        y_2 \tfrac{\partial}{\partial y_2},
        \\
        \mu_{\text{ff},2}(x,y)
        & \defeq
        x_1y_2 -x_2y_1,
        \\
        X_{\text{ff},2}
        & \defeq
        X_{\mu_{\text{ff},2}}^{\omega_\text{can}}
        =
        x_2 \tfrac{\partial}{\partial x_1}
        + 
        y_2 \tfrac{\partial}{\partial y_1}
        -
        x_1 \tfrac{\partial}{\partial x_2}
        -
        y_1 \tfrac{\partial}{\partial y_2}.
    \end{align*}
\end{myenumerate}

\noindent
The main point of this list is that, when working formally (i.e.\  with formal power series), around a nondegenerate point any integrable system is
formally equivalent to a direct product of the ones from the list. 
This result has also been proven in the analytic case \cite{Vey78}, where as usual the first step is proving the formal case.

Similar statements in the smooth context are much more subtle and harder to prove,
such as Eliasson's theorems mentioned at the end of this introductory chapter. These theorems will serve as motivation for a large part of this thesis.

\subsection{Infinitesimal Actions}
\label{ssec:InfActs}

We now start looking at some of the structures underlying the notion of integrable systems. We fix an integrable system ${\mu: (M, \omega)\rightarrow \R^n}$. To any function ${f}$ one associates the Hamiltonian vector field ${X_f^\omega}$, i.e.\  the vector field on ${M}$ defined by 
    \[
        i_{X_f}\omega 
        = 
        -df.
    \]
In particular, this gives the Hamiltonian vector fields of the ${\mu_i}$, denoted by the shorthand
    \[
        X_i
        \defeq
        X_{\mu_i}.
    \]
The involutivity condition thus implies
    \[
        [ X_i, X_j ]
        =
        0,
        \quad
        \forall\; i,j.
    \]
This is the first indication of the presence, and importance, of the abelian Lie algebra
    \[
        \gerh \defeq \R^n,
    \]
and of infinitesimal actions. More precisely, recall that given a Lie algebra ${\gerg}$, an \textbf{infinitesimal action}\index{infinitesimal action} of ${\gerg}$ on a manifold ${M}$ is a Lie algebra map
    \[
        \rho
        :
        \gerg
        \to
        \gerX(M)
    \]
from ${\gerg}$ to the Lie algebra of vector fields on ${M}$. For ${x\in M}$, we denote by 
    \[
        \rho_x
        :
        \gerg
        \to 
        T_xM
    \]
the linear map obtained by evaluating ${\rho}$ at ${x}$.  Of course, such infinitesimal actions usually arise by differentiating actions of a Lie group ${G}$ on ${M}$ (when ${\mathfrak{g}}$ is the Lie algebra of ${G}$). While if ${M}$ is compact any infinitesimal action integrates to (i.e.\  comes from) a Lie group action, such a property fails in the noncompact case.  

We see that any integrable system ${\mu:(M,\omega)\to \gerh^*}$ comes with an underlying infinitesimal action ${\rho_\mu=\rho_{\omega,\mu}}$ of ${\gerh}$ on ${M}$ given by
    \[
        \rho_\mu
        :
        \gerh
        \to
        \gerX(M),
        \quad
        e_i
        \mapsto
        X_{\mu_i}.
    \]
 In particular, it induces a partition of ${M}$: by the orbits of the action. The following is immediate:
 
\begin{lemma} 
    Each orbit is inside a fiber of ${\mu}$.
\end{lemma}
 
However, it should serve as a warning that the precise relationship between the orbits and the (connected components of the) fibers is subtle. 
  
\begin{example}
Let us look at the list of the normal forms recalled above:
\begin{myenumerate}
\item  For the elliptic normal form on ${\R^2}$ the fibers of ${\mu_\text{ell}}$ and the orbits of ${X_\text{ell}}$ are the same. They consist of the origin, as a fixed point, and the concentric circles.
\item For the hyperbolic normal form on ${\R^2}$ the fibers of ${\mu_\text{hyp}}$ consist of the cross defined by the axis and the coaxal hyperbolas. The orbits of ${X_\text{hyp}}$ further divides the cross into the origin, as a fixed point, and the half axes.
\item For the focus-focus normal form on ${\R^4}$ one sees a mixture of the elliptic and hyperbolic behavior which cannot be `disentangled' over the real numbers.
\end{myenumerate}
\end{example}
\noindent
\noindent
Returning to our discussion, although the Lie algebra involved may look uninteresting, it is instructive to have a look at infinitesimal actions of a general Lie algebra ${\gerg}$. 

Here is one type of general result that is also interesting for integrable systems: the stability of fixed points. Recall that a \textbf{fixed point}\index{fixed point!of an infinitesimal action} of an infinitesimal action ${\rho}$ is a point ${x}$ at which
    \[
        \rho_x
        =
        0.
    \]
The question is: given a fixed point ${x}$, and deforming a bit the action, do the nearby actions still have fixed points close to ${x}$, and how many? We need the \textbf{normal representation}\index{normal representation} at the fixed point ${x}$, which is the representation of ${\gerg}$ on ${T_xM}$ characterized by:
    \begin{equation}
    \label{lin-isotrpy-reps}
        \rho^{\textrm{norm}}_x
        :
        \gerg
        \to
        \gergl(T_xM),
        \quad
        \rho^{\textrm{norm}}_x(\xi)v_x
        \defeq
        [\rho(\xi), v]_x, 
    \end{equation}
for all $\xi\in \gerg$ and $v\in \gerX(M)$.  
When the action comes from an action of a Lie group ${G}$, then one obtains a similar action of ${G}$ on ${T_xM}$ (inducing the previous one by differentiation). The normal representation at ${x}$ should be seen as the linearization of the action of ${\gerg}$ on ${M}$ near ${x}$ (linearization in the sense that the vector fields defining the action become linear). Therefore, the action of ${\gerg}$ on the vector space ${T_xM}$ should be seen as a normal model around ${x}$ for the action of ${\gerg}$ on ${M}$. 

Besides stability, but related to it, another interesting question is whether ${(M,x)}$ and ${(T_xM,0)}$ are locally isomorphic as germs of ${\gerg}$-spaces. Recall the standard case where this happens:

\begin{proposition}[Bochner \cite{Boc45}]
\label{slice-sing-pts}
If the infinitesimal action comes from the action of a compact Lie group ${G}$ on ${M}$, and ${x}$ is a fixed point, then a neighborhood of ${x}$ in ${M}$ is equivariantly diffeomorphic to a neighborhood of ${0}$ in ${T_xM}$ (endowed with the normal representation).
\end{proposition}
\noindent
\noindent
Such a result is not so interesting when applied to integrable systems because of the compactness conditions on ${G}$, although so-called elliptic singularities seem to be related to actions of the ${n}$-torus. Moreover, this result does not consider the symplectic nature of the infinitesimal action defined by an integrable system. We will return to this aspect in a while.

Much more interesting is the question of the existence of fixed points mentioned above. 
The local model is still useful since it provides an indication of the answer: 
for a representation of a Lie algebra ${\gerg}$ on a vector space ${V}$, the fixed points of the associated (linear) infinitesimal action are precisely the ${\gerg}$-invariant elements of ${V}$, forming
    \[  
        V^\gerg
        =
        \set{
            v\in V
            : 
            \rho(\xi)v
            =0
            \quad\forall\;
            \xi \in \gerg
        }.
    \]
Or in fancier language, 
    $
        V^\gerg
        =
        H^0(\gerg,V),
    $
the zeroeth cohomology group of ${\gerg}$ with coefficients in ${V}$.  
Trying to deform the linear action in such a way that the fixed point ${0}$ is destroyed reveals the relevance of ${H^1(\gerg,V)}$ as well. With this insight one has:

\begin{proposition}
\label{stability-fixed-pts}
Assume that ${x}$ is a fixed point of the infinitesimal action of ${\mathfrak{g}}$ on ${M}$ and assume that, for the normal representation at ${x}$, 
\[
    H^1(\mathfrak{g}, T_xM)= 0.
\] 
Then ${x}$ is a stable fixed point. 
More precisely, any other infinitesimal action close to the original one admits at least a fixed point close to ${x}$. Moreover, one can find an entire family of fixed points parametrized by 
    \[
        H^0(\mathfrak{g}, T_xM)
        = 
        (T_xM)^{\mathfrak{g}}.
    \] 
In particular, the original infinitesimal action has at least a copy of ${(T_xM)^{\mathfrak{g}}}$ worth of fixed points, and if ${(T_xM)^{\mathfrak{g}}= \{0\}}$, then the fixed point must be isolated.
\end{proposition}
\noindent
The previous cohomological condition can also be expressed using the one-connected Lie group ${G(\gerg)}$ integrating ${\gerg}$, because
    \[ 
        H^i(\mathfrak{g}, T_xM)
        = 
        H^i(G(\mathfrak{g}), T_xM)
        \quad 
        \textrm{for}\ i\in \{0, 1\}.
    \]
In particular, if ${\gerg}$ is semisimple of compact type (or equivalently if ${G(\gerg)}$ is compact) then ${H^1}$ vanishes automatically and the previous proposition applies. In this case a lot more is true:

\begin{local_theorem}[Fernandes \& Monnier \cite{FM04}]
\label{Lie-alg-lin-thm} 
If ${\gerg}$ is semisimple of compact type then any infinitesimal action of ${\gerg}$ on ${M}$ can be linearized around any singular point. 
\end{local_theorem}

Note that this proposition is related to, but it is not an immediate consequence of, proposition \ref{slice-sing-pts} since it is not clear at all that the infinitesimal action comes from a Lie group action. 
Returning to proposition \ref{stability-fixed-pts}, in comparison with the other propositions stated above, it is much more applicable to commutative Lie algebras ${\gerg}$, i.e.\  to integrable systems. See \cite{FM04} for more details on these statements.

\subsection{Poisson Brackets}

Of course, one of the important structures entering the definition of integrable systems is the symplectic structure. 
However, the only role of ${\omega}$ is to define the Poisson bracket ${\set{ -, -}_\omega}$, and this is the first indication of the Poisson nature of the definition. 
Here we recall some basic, Poisson-geometric notions and recast the notion of integrable systems in this framework. 

By a \textbf{Poisson bracket}\index{Poisson bracket} ${\set{ -, -}}$ on a Manifold ${M}$ we mean a Lie algebra structure on the algebra ${C^\infty(M)}$ of smooth, real-valued functions satisfying the Leibniz identity
    \[
        \set{f,g\, h}
        =
        \set{f,g}\, h
        +
        g\, \set{f,h},
        \quad
        \forall\; f, g, h \in C^\infty(M).
    \]
This definition makes sense on general algebras, giving rise to the notion of Poisson algebras. The Leibniz identity tells us that for each ${f\in C^\infty(M)}$ the operation ${\set{f, -}}$ is a derivation on ${C^\infty(M)}$, hence defines a vector field on ${M}$. This is called the \textbf{Hamiltonian vector field}\index{Hamiltonian vector field} associated to ${f}$ and is denoted by ${X_f=X_f^\pi}$.

The starting point of the bivector field formulation of Poisson structures is the simple remark that ${\R}$-bilinear, antisymmetric operations ${\set{ -, -}}$ on ${C^\infty(M)}$ satisfying the Leibniz identity are in one-to-one correspondence with bivector fields 
    \[
        \pi
        \in
        \gerX^2(M)
        =
        \Gamma(\wedge^2TM).
    \]
This correspondence is characterized by
    \[
        \pi(df,dg) = \set{f,g}.
    \]
The fact that ${\set{ -, -}}$ satisfies the Jacobi identity is equivalent to the involutivity condition 
    \[
        [\pi,\pi]=0,
    \]
where ${[ -, -]}$ is the Schouten-Nijenhuis bracket on multivector fields. The notation from the previous equation indicates that we think of (and use) multivector fields as multilinear maps whose arguments are one-forms. For a bivector ${\pi}$, fixing such a one-form ${\theta_1}$, the expression ${\pi(\theta_1, -)}$ is a linear map on ${T^*M}$ and hence a vector field. This defines
    \[
        \pi^\sharp
        :
        T^*M
        \to
        TM,
        \quad
        \theta_2(\pi^\sharp(\theta_1))
        = 
        \pi(\theta_1,\theta_2).
    \]
Of course, this is just the reinterpretation of antisymmetric, bilinear maps on a vector space ${V}$ as skew adjoint linear maps ${V\to V^*}$. Applied to an exact one-form ${df}$, one obtains the associated Hamiltonian vector field already described above
    \[
        X_f^\pi
        =
        \pi^\sharp (df).
    \] 
Here are two basic examples of Poisson brackets:

\begin{myenumerate}

\item \textit{Nondegenerate Poisson brackets:} A symplectic two -form ${\omega}$, since nondegenerate, can be inverted into a bivector
    \[ 
        \pi_{\omega}
        \defeq
        \omega^{ -1}.
    \]
To justify the notation, interpret ${\pi}$ and ${\omega}$ as skew adjoint linear maps between the cotangent bundle ${T^*M}$ and and tangent bundle ${TM}$. Via this correspondence the condition ${d\omega=0}$ corresponds to ${[\pi_\omega,\pi_\omega]=0}$.

\item \textit{Linear Poisson brackets:} For a Lie algebra ${\gerg}$, the dual ${\gerg^*}$ carries a canonical Poisson bracket uniquely characterized by the condition that 
    \[ 
        \{\mathrm{ev}_u, \mathrm{ev}_v\}
        = 
        \mathrm{ev}_{[u, v]}
        \quad\forall\ u, v\in \gerg,
    \]
where ${\mathrm{ev}_u}$ is the linear function on ${\gerg^*}$ that sends ${\xi}$ to ${\xi(u)}$. 
Or in coordinates, if ${c_{i, j}^{k}}$ are the structure constants of ${\gerg}$, then the components ${\pi_{i, j}}$ in the writing 
    $
        \pi= 
        \sum_{i, j} 
        \pi_{i, j} 
        \,\partial x_i \wedge \partial x_j
    $ 
are the linear functions
    \[ 
        \pi_{i, j}
        = 
        \sum\nolimits_{k} c_{i, j}^{k} x_k.
    \]
\end{myenumerate}

\noindent
The second class of examples shows that interesting Poisson structures can be quite degenerate. At the extreme, one has the zero Poisson structure on any manifold. The involutivity condition for the moment map ${\mu}$ of a completely integrable system can be interpreted as saying that
    \[ 
        \mu: (M, \omega^{ -1}) \rightarrow (\gerh^*, 0) 
    \]
is a Poisson map.

Of course, symplectic geometry comes with various basic results that exploit the nondegeneracy condition, or that are not so clear or easy to extend to general Poisson structures. 
One class of such results are the neighborhood theorems for various types of submanifolds of a symplectic manifold: Lagrangian, isotropic, coisotropic, or symplectic submanifolds.
Some of these results do extend to Poisson geometry, but the proofs are more involved and the conditions for the theorems are more subtle. 

Let us briefly mention here the case of isotropic submanifolds, i.e.\  submanifolds ${N}$ of a symplectic manifold ${(M, \omega)}$ satisfying the condition 
    \[ 
        T_xN \subset (T_xN)^{\perp_\omega},
        \quad\forall\ x\in N,
    \]
where ${\perp_\omega}$ denotes the symplectic orthogonal with respect to ${\omega}$. The key remark is that, while in classical geometry a neighborhood of ${N}$ is determined by its normal bundle ${\calN=\calN(N,M)}$, in symplectic geometry there is an adapted normal bundle that is relevant. In the case of isotropic manifolds, this so-called \textbf{symplectic normal bundle}\index{symplectic normal bundle} is given by
    \[ 
        \calN^{\omega}= \calN^\omega(N,M)
        \defeq
        (TN)^{\perp_\omega}/TN.
    \]
This is a vector bundle over ${M}$ and each fiber caries a canonical symplectic structure: the restriction of ${\omega}$ to ${(TN)^{\perp_\omega}}$ has ${TN}$ as its kernel, hence it descends to the quotient as a nondegenerate two -form. In other words, ${\NN^{\omega}}$ is a symplectic vector bundle. The main point is that, when ${N}$ is compact, an entire neighborhood of ${N}$ in ${M}$, together with its symplectic structure, is completely determined by the symplectic normal bundle.

Moving to general Poisson structures, there are various other local forms that become interesting. They continue to be related to symplectic geometry, but in a more subtle way. For instance, one can wonder about local forms around symplectic leaves. Here we mention one such result: the case of fixed points of Poisson structures for which one has Conn's linearization theorem \cite{Conn85}. Recall that a \textbf{fixed point of a Poisson manifold}\index{fixed point!of a Poisson manifold} ${(M, \pi)}$ is any ${x\in M}$ at which ${\pi_x= 0}$. In this case the derivative ${d_x\pi:T_xM\to \wedge^2 T_xM}$ defines a linear Poisson structure on ${T_xM}$, and hence a Lie algebra structure on ${\gerg_{\pi,x}\defeq T_x^*M}$.

\begin{local_theorem}[Conn's linearization]\label{conn-lin-thm}
If ${\gerg_{\pi,x}}$ is semisimple of compact type, then there are neighborhoods ${U\sub M}$ of ${x}$ and ${V\sub \gerg_{\pi,x}^*}$ of ${0}$, and a Poisson diffeomorphism
    \[
        \phi: (U,\pi) \xrightarrow{\,\simeq\,} (V,d_x\pi).
    \]
\end{local_theorem}

\section{Eliasson's linearization theorems}

\subsection{Local normal models around fixed points}

In this section we describe the local normal models for fixed points of integrable systems. Most of this can be found in the literature, see for example \cite{Eli90, Mir03}. The only original contributions by the author concerns the exposition through jet bundles, which easily generalizes to coordinate free descriptions of normal models along orbits, and proposition \ref{prop:nondeg_function}, which the author could not find in the literature.

If we are interested in the local behavior of an integrable system 
$\mu: (M, \omega)\rightarrow \mathbb{R}^n$ around a point ${x\in M}$, the story is interesting only when ${x}$ is a singular point, i.e.\  a point at which ${\mu}$ fails to be a submersion. There are several types of such singular points,  depending on the rank of ${(d\mu)_x}$. We now concentrate on the case when ${x}$ is a \textbf{fixed point}\index{fixed point!of an integrable system}, i.e.\  where ${(d\mu)_x= 0}$ (which is equivalent with the fact that ${x}$ is a fixed point for the induced action of ${\gerh= \R^n}$ on ${M}$). Even for fixed points, there are various possibilities depending on the order at which ${d\mu}$ vanishes at ${x}$.

\begin{definition} 
Given an integrable system $\mu: (M, \omega)\rightarrow \R^n$, a fixed point of order (at least) $k$ of $(\omega, \mu)$ is any point $x\in M$ at which $\mu$ vanishes up to order $x$ (i.e.\  all its partial derivatives at $x$ up to order $k$ vanish).
\end{definition}

\noindent
Around such points one can build a local, much simpler (namely polynomial) `local model', from the jet data of the original integrable system at ${x}$: the local model for ${\mu}$ around ${x}$ is basically the homogeneous polynomial of degree ${k+1}$ made out of the partial derivatives of ${\mu}$ of order ${k+1}$, viewed as the moment map of an integrable system on ${\mathbb{R}^{2n}}$. For a coordinate free description,
we will work on the symplectic vector space ${(T_xM, \omega_x)}$ (thought of as the linear approximation of ${(M, \omega)}$ around ${x}$). Then we obtain symmetric polynomials 
\[ 
h_x(\mu_i)\in S^{k+1}T^{*}_{x}M, \quad h_x(\mu)\in S^{k+1}T^{*}_{x}M\otimes \mathbb{R}^n. 
\]
Conceptually, these are just the ${(k+1)}$-st jets of the ${\mu_i}$ and ${\mu}$ combined with the remark that, since ${\mu}$ vanishes up to order ${k}$, the ${(k+1)}$-st jet lives in the symmetric algebra, which is usually illustrated by the standard short exact sequence
\[
        0
        \to
        S^{k+1} T^*M \otimes \R^n
        \to
        J^{k+1}(M,\R^n)
        \to
        J^k(M,\R^n)
        \to
        0.
\]
We interpret the ${h_x(\mu_i)}$ as polynomial functions, obtaining 
\begin{equation}\label{norm -fiorm -localmodel -is} 
h_x(\mu)\defeq (h_x(\mu_1), \ldots, h_x(\mu_n)): (T_xM, \omega_x)\rightarrow \R^n.
\end{equation}

\begin{lemma}
If ${\mu: (M, \omega)\rightarrow \R^n}$ is an integrable system and ${x}$ a fixed point of order ${k}$, then then 
(\ref{norm -fiorm -localmodel -is}) 
is an integrable system (but independence may fail).
\end{lemma}

\begin{proof}
Let ${U\sub M}$ be an open neighborhood of ${x}$ and let ${\phi: U \to T_xM}$ be a chart that maps ${\omega\vert_U}$ to ${\omega_x}$, and for which ${\phi(x)=0}$ and ${\dd_x\phi = \id_{T_xM}}$. We first study what happens on the chart.
There, for any functions ${f,g \in C^\infty(T_xM)}$, we find
    \[
        j^k_0 \big(
            \big\{ 
                f, g 
            \big\}_{\omega_x}
        \big)
        =
        \big\{ 
            j^{k+1}_0(f), j^{k+1}_0(g) 
        \big\}_{\omega_x}.
    \]
Moreover, that ${j^k_x(\mu)=0}$ and ${j^1_x(\phi)=\id_{T_xM}}$ implies that 
    \[
        j^{k+1}_0( \mu\circ \phi^{ -1} )
        =
        j^{k+1}_x( \mu ).
    \]
Combining these observations gives us
    \begin{align*}
        \big\{ 
            j^{k+1}_x(\mu_i), j^{k+1}_x(\mu_j) 
        \big\}_{\omega_x}
        & =
        \big\{ 
            j^{k+1}_0(\mu_i\circ\phi^{ -1}),
            j^{k+1}_0(\mu_j\circ\phi^{ -1})
        \big\}_{\omega_x}
        \\
        & =
        j^k_0 \big(
            \big\{
                \mu_i \circ \phi^{ -1},
                \mu_j \circ \phi^{ -1}
            \big\}_{\omega_x}
        \big)
        \\
        & =
        j^k_0 \big(
            \big\{ 
                \mu_i, \mu_j 
            \big\}_\omega \circ \phi^{ -1}
        \big)
        =
        0
        \qedhere
    \end{align*}
\end{proof}
\noindent
Although it is interesting to think about the general case, we will restrict to the case ${k= 1}$ (just ordinary fixed points). Then ${h_x(\mu)}$ becomes the ordinary Hessian. 

\begin{definition}
\label{def-chi-lineariz} 

Given an integrable system ${\mu: (M, \omega)\rightarrow \R^n}$ and fixed point ${x\in M}$, the resulting integrable system 
\[ \hess_x(\mu)\defeq (\hess_x(\mu_1), \ldots, \hess_x(\mu_n)): (T_xM, \omega_x)\rightarrow \R^n\]
will be called the \textbf{normal model}\index{normal model} associated to ${(\omega, \mu)}$ around ${x}$. 

\end{definition}

\noindent
The constant symplectic form ${\omega_x}$ induces a Poisson bracket
\[ \{\cdot, \cdot\}_{\omega_x}: S^{\bullet}T^{*}_{x}M \otimes S^{\bullet}T^{*}_{x}M\rightarrow S^{\bullet}T^{*}_{x}M.\]
Actually, interpreting it as a (constant) symplectic form on the manifold ${T_xM}$, it induces a Poisson bracket on ${\calC^{\infty}(T_xM)}$ which preserves polynomials. Note that
the bracket between a polynomial of degree ${p}$ and one of degree ${q}$ is one of degree ${(p+q -2)}$. In particular, for ${p= q= 2}$, one obtains a Lie algebra
\[ (S^{2}T^{*}_{x}M, \{\cdot, \cdot\}_{\omega_x}).\]
It contains the Hessians ${\hess_x(\mu_i)}$, which span an abelian Lie subalgebra:

\begin{definition}
\label{def -hessian -lie -alg} 
The Hessian Lie algebra of the integrable system $\mu: (M, \omega)\rightarrow \mathbb{R}^n$ at the fixed point ${x}$ is the Lie \emph{sub}algebra 
\[
        \gerh_{x,\mu}
        \defeq
        \lspan_\R \big(
            \hh_x\mu_1, \ldots, \hh_x\mu_n
        \big)
        \subset (S^{2}T^{*}_{x}M, \{\cdot, \cdot\}_{\omega_x}).
\]
Let us emphasize: it is the Lie algebra ${\gerh_{x, \mu}}$, together with the inclusion.
\end{definition}
\noindent
It turns out that the behavior of the integrable system around ${x}$ can be controlled algebraically by this Lie algebraic inclusion. This fits in perfectly with the fact that, for an action of a Lie algebra ${\gerh}$ on a manifold ${M}$, with fixed point ${x}$, the behavior of ${M}$ around ${x}$ is controlled by the normal representation of ${\gerh}$ on ${T_xM}$ (see (\ref{lin-isotrpy-reps}) in subsection \ref{ssec:InfActs}):
\[ 
    \rho^{\textrm{norm}}_x: \gerh\rightarrow \gergl(T_xM).
\]
Actually, the Hessian Lie algebra is basically just this normal representation (or, more precisely, the inclusion ${\rho^{\textrm{norm}}_x(\gerh)\hookrightarrow \gergl(T_xM)}$) combined with the following remarks:

\begin{myenumerate}

\item 
The normal representation ${\rho^{\textrm{norm}}}$ acts by symplectic linear maps with respect to $\omega$, i.e.\  maps ${A: T_xM\rightarrow T_xM}$ 
satisfying 
\[ 
    \omega_x(A(X), Y)
    = 
    \omega_x(X, A(Y))\quad \forall\ X, Y\in T_xM.
\]
These form a Lie algebra, usually denoted ${\gersp(T_xM, \omega_x)}$, so that the normal representation becomes a Lie algebra map
\[  \rho^{\textrm{norm}}_x: \gerh\rightarrow  \gersp(T_xM, \omega_x).\]

\item 
There is a standard canonical identification, valid for all symplectic vector spaces (and used here for ${(T_xM, \omega_x)}$): 
\begin{equation}\label{eq:idtfct -S2 -sp} 
S^2T^{*}_{x}M\cong \gersp(T_xM, \omega_x), \quad h\longleftrightarrow A,
\end{equation}
characterized by 
\[ \omega(A(u), v)= h(u, v), \quad \forall\ u, v\in T_xM.\]

\end{myenumerate}

\noindent
In other words, what we are saying here is that ${\gerh_{x, \mu}}$ is just the image of
\[ \rho^{\textrm{norm}}_x: \gerh\rightarrow  \gersp(T_xM, \omega_x) \cong S^2T^{*}_{x}M.\]
To check this, one remarks that, since the action of ${\gerh}$ on ${M}$ is given by the vector fields ${X_{\mu_i}}$, we see that ${\rho^{\textrm{norm}}_x}$ is given by the operators ${D_x(X_{\mu_i})}$, the vertical derivative of ${X_{\mu_i}}$ at ${x}$ (see also  Definition \ref{vert -D -for -X} below). It is immediate to see that, under the identification (\ref{eq:idtfct -S2 -sp}), one has 
\[ S^2T^{*}_{x}M\ni \hess_x(\mu_k) \longleftrightarrow D_{x}(X_{\mu_k}) \in \gersp(T_xM, \omega_x),\]
so that everything can be arranged in a commutative diagram
\[
\begin{tikzcd}[row sep=small]
    \gerh \arrow{d}\arrow{r}{\rho^{\textrm{norm}}_x}&
    \gersp(T_xM, \omega_x) \arrow{d}{\sim}
    \\
    \gerh_{x,\mu}\arrow[hookrightarrow]{r}&
    S^2T^{*}_{x}M
     \end{tikzcd}
\]
\noindent
In particular, one can, and we will occasionally do without further notice, identify ${\gerh_{x,\mu}}$ with a subalgebra of the symplectic Lie algebra:
\[  
    \gerh_{x,\mu}\cong \lspan_\R \big(
    D_{x}(X_{\mu_1}), \ldots, D_{x}(X_{\mu_n})
    \big)
    \subset \gersp(T_xM, \omega_x).
\]
\noindent        
Since the notation ${D_x}$ for `the vertical derivative' as ${x}$ will be used several time, we want to make it more visible:

\begin{definition}
\label{vert -D -for -X} 
If ${X\in \gerX(M)}$ is a vector field vanishing at a point ${x}$, we denote by 
\[
    {D_x(X): T_xM\rightarrow T_xM}
\]
the induced vertical derivative of ${X}$ at ${x}$ (called the linearization of ${X}$ at ${x}$). Therefore, intrinsically
$
    {D_x(X)(Y)= [X, Y](x)}
$
while in local coordinates it is represented by the matrix ${\frac{\partial X_i}{\partial x_j}(x)}$, where ${X^i}$ are the components of ${X}$.
\end{definition}
\noindent
From the previous discussion combined with the last part of Proposition \ref{stability-fixed-pts}, we deduce:

\begin{corollary}
\label{cor:fixed-point-set}
For the fixed point set of the linear actions on ${T_xM}$,
\[ (T_xM)^{\gerh}= (T_xM)^{\gerh_{x,\mu}}= \bigcap\nolimits_i
        \ker( D_xX_{\mu_i} )\]
and, if this is zero, then ${x}$ is an isolated fixed point. 
\end{corollary}

Note also that, for the dual action (on ${T^*M}$) we obtain 
\[ (T_x^*M)^{\gerh}= \{ \xi\in T_x^*M: \{\xi, \gerh_{x,\mu}\}_{\omega_x}= 0\}.\]
Moreover, the canonical isomorphism ${\omega^{\sharp}_{x}: T_xM\rightarrow T^*_xM}$ induced by ${\omega_x}$ intertwines the two linear actions, hence it induces a linear isomorphism
\begin{equation}\label{intertwine-fixed-points}
\omega^{\sharp}_{x}: (T_xM)^{\gerh} \stackrel{\sim}{\rightarrow} (T_x^*M)^{\gerh}.
\end{equation}

\begin{remark}\label{rk -hess -as -invariant} It is interesting to think of the inclusion 
\[ \gerh_{x,\mu}\subset \gersp(T_xM, \omega_x)\]
as an algebraic invariant of the integrable system ${(\omega, \mu)}$ at ${x}$. 
Combining with an identification of  ${(T_xM, \omega_x)}$ with ${(\mathbb{R}^{2n}, \omega_{\textrm{can}})}$, one obtains 
a Lie subalgebra of ${\gersp(2n)}$, which is uniquely defined up to conjugation. This can be thought of as a slightly weaker invariant. 
\end{remark}
\noindent
And here is a stronger illustration of the relevance of ${\gerh_{x,\mu}}$, as a Lie subalgebra of ${\gersp(T_xM, \omega_x)}$, on the behavior of the integrable system around ${x}$. While the conditions of type ${\{\mu^i, f\}= 0}$ or ${(d\mu)(X)= 0}$ are vacuous at the fixed point ${x}$ (and one may think that they do not provide any information on ${f}$ or ${X}$ at ${x}$), ${\gerh_{x,\mu}}$ captures higher order information. This is formulated in the following proposition, which will be used later on.

\begin{proposition}
\label{prop:nondeg_function} 
Assume that ${x}$ is a fixed point of the integrable system ${\mu: (M, \omega)\rightarrow \R^n}$ and denote by 
${C_{\gerh_{x,\mu}}}$ the centralizer of 
 \[{\gerh_{x,\mu}} \sub \gersp(T_xM, \omega_x),\] 
where $\gersp(T_xM, \omega_x)$ is the symplectic Lie algebra at ${x}$.

For local vector fields ${X}$ and functions ${f}$ defined near ${x}$, with 
    \[{X\in \gerX_{\mu}} \,\text{ and }\, {f\in \ucalC_{\mu}}\] 
i.e.\  ${X}$ is killed by ${d\mu}$ and ${f}$ commutes with all the ${\mu_i}$'s, one has:
\begin{itemize}
\item ${(df)_x\in (T^{*}_{x}M)^{\gerh}}$ and ${X_x\in (T_{x}M)^{\gerh}}$.
\item If ${(df)_x= 0}$, then ${h_{x}(f)\in C_{\gerh_{x,\mu}}}$.
\item If ${X_x= 0}$ and ${D_x(X)\in \gersp(T_xM, \omega_x)}$, then ${D_x(X)\in C_{\gerh_{x,\mu}}}$.
\end{itemize}
\end{proposition}
\noindent
\begin{proof} 
It suffices to prove the final statement (for the penultimate one takes ${X= X_f}$). One can write down a tedious but straightforward local computation. Here is a more formal argument, based on 
the properties of the vertical derivative ${D_x}$. The idea is obvious: instead of evaluating the condition ${\dd\mu_i(X)= 0}$ at ${x}$, one first takes derivatives and then evaluates at ${x}$. And similarly for second order derivatives. The undesired terms disappear because ${(d\mu)_x=0}$. The first derivative of ${g_i=\dd\mu_i(X)}$ at ${x}$ gives
    \begin{align*}
        0 
        =
        (d g_i)_x(Y_x)
        & =
        \hess_x(\mu)(X_x, Y_x) + \big((d\mu)_x \circ D_x(X)\big) (Y_x) 
        \\  
        & =
        \hess_x(\mu)(X_x, Y_x)  
    \end{align*}
for any ${Y \in \gerX_M}$. Using the identity 
    \[
        \hess_x(\mu_i)(Y_x, Z_x) 
        = 
        \omega_x(D_x(X_{\mu_i})(Y_x), Z_x),
    \]
one has ${D_x(X_{\mu_i})(X_x) = 0}$ i.e.\ , by the previous discussion, ${X(x)}$ is in the fixed point set. Assume now ${X(x)= 0}$. This also shows that the Hessians of the ${g_i}$ at ${x}$:
    \begin{align*}
        \hess_x(g_i)(Y_x,Z_x)
        & =
        \hess_x(\mu_i)( D_x(X)(Y_x), Z_x )
        \\
        & +
        \hess_x(\mu_i)( Y_x, D_x(X)(Z_x) ).
    \end{align*}
When ${g_i= 0}$, passing in the left hand side again from Hessians to vertical derivatives using ${\omega_x}$ (${h_x(g)}$ corresponds to ${D_x(X_g)}$) we find:
    \[
        0 = [ D_x(X_{\mu_i}), D_x(X) ].
    \]\
Since this holds for all ${i}$, the conclusion follows. 
\end{proof}
\noindent

\subsection{Nondegenerate fixed points \& Eliasson's theorem(s)}

We continue to look at an integrable system ${\mu: (M, \omega)\rightarrow \mathbb{R}^n}$ around a fixed point ${x}$. 
The previous proposition gives an indication that conditions on the commutator relations inside ${\gersp(T_xM,\omega_x)}$ involving ${\gerh_{x,\mu}}$ forces the original system to have a behavior similar to its linearization around ${x}$. This slowly brings us to the notion of nondegenerate fixed point \cite{Vey78}:

\begin{definition}
\label{def-weired-nondegen}
A fixed point ${x}$ of ${(\omega,\mu)}$ is called \textbf{nondegenerate}\index{fixed point!nondegenerate} if ${\gerh_{x,\mu}}$
is a Cartan subalgebra, i.e.\ it is abelian (which always holds), ${n}$-dimensional and self-normalizing:
\begin{quote}
    ${[X,Y]\in \gerh_{x,\mu}}$ for all ${Y\in \gerh_{x,\mu}}$ and ${X\in\gersp(T_xM,\omega_x)}$ $\implies$ ${X\in \gerh_{x,\mu}}$.
    \qedhere
\end{quote}
\end{definition}
\noindent
\begin{remark}\label{rk:Williamson} In the spirit of Remark \ref{rk -hess -as -invariant}, it is interesting to recall the result of Williamson \cite{williamson, williamson40} who classified all the 
Cartan subalgebras of ${\gersp(2n,\R)}$ up to conjugation: any Cartan subalgebra of ${\gersp(2n,\R)}$ is conjugate to an ${n}$-dimensional abelian subalgebra spanned by matrices of the following three types:
\begin{myitemize}
    \item 
Elliptic blocks ${A^i_\text{ell}}$. They are zero everywhere except on the ${(x_i,y_i)}$-plane, where
    \[
        A^i_\text{ell}
        =
        \big(
        \begin{smallmatrix}
            0 & -1 
            \\
            1 & 0
        \end{smallmatrix}
        \big).
    \]
They correspond to quadratic functions of the form 
        \[
            \hh^i_\text{ell}(x,y) = \tfrac12(x_i^2+y_i^2).
        \]
    \item 
Hyperbolic blocks ${A^j_\text{hyp}}$. They are zero everywhere except on the ${(x_j,y_j)}$-plane, where 
    \[
        A^i_\text{hyp}=
        \big(
        \begin{smallmatrix}
            -1 & 0
            \\
            0 & 1
        \end{smallmatrix}
        \big).
    \]
They correspond to quadratic functions of the form     
    \[
            \hh^j_\text{hyp}(x,y) = x_jy_j.
    \]
    
\item Focus -focus pairs ${A^k_\text{ff,1}}$ and ${A^k_\text{ff,2}}$. They are zero everywhere except on the ${(x_k,y_k,x_{k+1},y_{k+1})}$-plane, where
    \[
        A^k_\text{ff,1}=
        \left(
        \begin{smallmatrix}
            -1 & 0 & 0 & 0
            \\
            0 & 1 & 0 & 0
            \\
            0 & 0 & -1 & 0
            \\
            0 & 0 & 0 & 1
        \end{smallmatrix}
        \right),
        \quad
        A^k_\text{ff,2}=
        \left(
        \begin{smallmatrix}
            0 & 0 & 1 & 0
            \\
            0 & 0 & 0 & 1
            \\
            -1 & 0 & 0 & 0
            \\
            0 & -1 & 0 & 0
        \end{smallmatrix}
        \right).
    \]
They correspond to quadratic functions of the form
    \begin{align*}
        \hh^k_\text{ff,1}(x,y)
        & =
        x_ky_k + x_{k+1}y_{k+1},
        \\
        \hh^k_\text{ff,2}(x,y)
        & =   
        x_ky_{k+1} - x_{k+1}y_k.
    \end{align*}
\end{myitemize}

\noindent
Hence, the ${\gerh_{x, \mu}}$ up to conjugacy becomes a numeric invariant, called 
the \textbf{Williamson type}\index{Williamson type} of ${(\omega,\mu)}$ at the nondegenerate fixed point:
the triple ${(e,h,f)}$ of natural numbers that classifies the conjugacy class of ${\gerh_{x,\mu}}$. 
\end{remark}
\noindent
Let us point out (although this is likely not the best argument) that Williamson's classification directly implies that the fixed points of the linear action are isolated (see Corollary \ref{cor:fixed-point-set}).

\begin{corollary} 
\label{cor-no-fixed-from-willianson}
At nondegenerate fixed points ${x}$ one has ${(T_xM)^{\gerh}= 0}$. In particular, nondegenerate fixed points
are isolated. 
\end{corollary}

Of course, nondegenerate fixed points of ${(\omega,\mu)}$ are not isolated from the \emph{singular set} at all. The first and simplest example is the normal form of Williamson type ${(2,0,0)}$: 
    \[
        \mu:(\R^4,\omega_\text{can})\to \R^2,
        \quad
        \mu_i(x,y) = \tfrac12 (x_i^2 + y_i^2).
    \]
It has two planes of singular points intersecting in a fixed point at the origin.

However, quite a lot more can be said about the local behavior of integrable systems around nondegenerate fixed points, and this brings us to Eliasson's normal form theorems (`linearization'), which compare ${(\omega, \mu)}$ around ${x}$ with the normal model around the origin (see Definition \ref{def-chi-lineariz}). The best scenario is when the Hessian Lie algebra ${\gerh_{x, \mu}}$ is \textbf{elliptic} in the sense that its Williamson type is of the form ${(e, h=0, f=0)}$. In this case, Eliasson proved the following \cite{Eli90}.

\begin{local_theorem}[Eliasson]
\label{thm -El1} 
Let ${\mu: (M, \omega)\rightarrow \mathbb{R}^n}$ be an integrable system and let ${x\in M}$ be a nondegenerate fixed point of elliptic type. Then, around ${x}$, ${(\omega, \mu)}$ is equivalent to its linearization in the following sense: there exists a symplectomorphism ${\Phi}$ from a neighborhood of ${0}$ in ${(T^*M, \omega_x)}$ and one of ${x}$ in ${(M, \omega)}$, 
\[ 
\Phi: (T_xM, \omega_x, 0)\rightarrow (M, \omega, x),
\]
and a diffeomorphism ${\phi}$ between open neighborhoods of ${0}$ and ${\mu(x)}$ in ${\mathbb{R}^n}$, 
\[ \phi: (\mathbb{R}^n, 0)\rightarrow (\mathbb{R}^n, \mu(x))\] 
such that:
\[
        \begin{tikzcd}
            (M, \omega, x)
            \arrow{d}{\mu}
            & 
            (T_xM, \omega_x, 0)
            \arrow{l}[swap]{\Phi}
            \arrow{d}{\hess_x(\mu)} \\
            (\mathbb{R}^n, \mu(x))
            & 
            (\mathbb{R}^n, 0)
            \arrow{l}{\phi}
        \end{tikzcd}
\]
\begin{align*}
    \Phi^*(\omega) &= \omega_x,
    \text{ and }
    \mu\circ \Phi = \phi\circ (\hess_x(\mu_1), \ldots, \hess_x(\mu_n)).
\end{align*}
\end{local_theorem}

\noindent
Note that this can be seen as a analogue for integrable systems of Conn's linearization theorem from Poisson Geometry (see Theorem \ref{conn-lin-thm}) hence, in particular, as an analogue of the linearizability of infinitesimal actions from Theorem \ref{Lie-alg-lin-thm}.
For more general nondegenerate fixed points the story is more subtle. However, it seems to be generally accepted (see for example \cite{PelaNgoc11}) that Eliasson's work implies the following `linearization result'. 

\begin{local_theorem}[implicit in Eliasson \cite{Eli84}]
\label{thm -El2} 
Let $\mu: (M, \omega)\rightarrow \mathbb{R}^n$ be an integrable system and let ${x\in M}$ be a nondegenerate fixed point.

Then, around ${x}$, the Lagrangian foliation of ${(\omega, \mu)}$ can be linearized around ${x}$ in the following sense:
there exists a symplectomorphism ${\Phi}$ between a neighborhood of ${0}$ in ${(T^*M, \omega_x)}$ and one of ${x}$ in ${(M, \omega)}$ such that 
\[ 
\big\{
    \mu_i, 
    \Phi_*(\hess_x(\mu_j))
\big\}_{\omega} = 0, 
\quad\forall\, i,j.
\]
\end{local_theorem}

\section{Motivation}
\label{sec:outline-thesis}

One may say that this thesis is motivated by understanding these theorems.
First, here is a sketch of the proof, as far as we understood it from the literature, with some remarks.
Afterwards, we will explain the proof we intended to find.

\begin{enumerate}

\item
Prove vanishing of the deformation cohomology associated with each of the normal models. 
This has been done by Vey \cite{Vey78, Vey79} for each model separately.
The idea behind this step is the following:
in deformation theory, one can associate a chain complex $\calC^*(\omega, \mu)$ to any integrable system $(\omega, \mu)$. Morally, if the analysis works out, vanishing of the first cohomology group $H^1(\calC^*(\omega, \mu))$ (i.e.\ infinitesimal rigidity) means that $(\omega,\mu)$ is rigid: any nearby integrable system is equivalent. A general philosophy for normal form theory is that (infinitesimal) rigidity of the normal form should lead to a normal form theorem.

\item
Prove the normal form theorem for each (Williamson) normal form separately. The elliptic case has been proven by Eliasson \cite{Eli90}, as stated in Theorem \ref{thm -El1}, for a much stronger equivalence relation than Theorem \ref{thm -El2} and for arbitrary dimension. In his thesis \cite{Eli84} he claims his method of proof also applies to the focus-focus case, although the analysis would be more involved, and he leaves several non-trivial gaps.

Vũ Ng\d oc and Wacheux \cite{NgocWa13} fill in all gaps in dimension four (i.e.\ exactly one copy of the focus-focus normal form) using several Moser path arguments and functional analysis on a space of flat functions. It is stated this generalizes to arbitrary dimension $4k$, but the proof is left to the reader. See also Chaperon \cite{Cha13} for another proof in dimension four.

The hyperbolic case is proven by Verdière \& Vey \cite{Vey79} for dimension two (i.e.\ exactly one copy of the hyperbolic normal form). 
Their argument concerns volume forms, hence does not directly apply to aribtrary dimension. 
A proof of the higher dimensional case is presented in Miranda \cite{Mir03}.

\item
Induction on the Williamson type $(e, h, f)$, which is done in two steps:

\begin{enumerate}

\item 
Decompose the integrable system into blocks corresponding to the Williamson type, i.e. find an integrable sytem $(\omega', \mu')$ on $\R^{2n}$ and a symplectomorphism $\phi$ between an open neighborhood of zero in $\R^{2n}$ and an open neighborhood of $x$ in $(M, \omega)$ such that 
    \[
        \big\{ \mu_i, \phi_*\mu'_j \big\} = 0 \quad\forall i,j,
    \]
and such that $(\omega',\mu')$ is the direct product of integrable systems of dimension two (for elliptic and hyperbolic) or four (for focus-focus) with a singularity at the origin of Williamson types $(1,0,0)$, $(0,1,0)$ or $(0,0,1)$.
The proof presented in \cite{Mir03} is not clear to the author. In particular,  this is due the fact that the theorem is formulated differently, using the expressions `singular (Lagrangian) foliation' and `$\phi$ preserves the foliation'. 
There are multiple interpretations of these concepts, which we study in detail in Chapter \ref{chapter:sheaves_behind}.

\item
Apply the results of step two to each factor of $(\omega', \mu')$ separately, and confirm that applying the normal form result to one factor does not affect the normal form for another factor. This has been proven by Miranda \cite{Mir03} and later generalized to the equivariant case \cite{Mir14}.
\end{enumerate}

\end{enumerate}

\noindent
We believe this is not the `best' proof and, in a way, this thesis is about searching for a more natural one. We hope our study of (singular) Lagrangian foliations, the Moser path method for integrable systems, and rigidity theorems for PDEs with symmetry will contribute to this.
More specifically, the proof of theorem \ref{thm -El2} can be improved upon in the following ways.

\begin{enumerate}

\item
Directly prove the vanishing of deformation cohomology of the nondegenerate normal form.
It has been proven by Miranda and Vũ Ng\d oc \cite{MN05} that the first cohomology group $H^1(C^*(\omega_\mathrm{can}, \mathrm{Hess}_x\mu))$ vanishes, but the proof depends on induction on the Williamson type. 
It would insightful to find a proof that is more canonical, without using the coordinate system from Williamson's theorem.

\item
Directly prove the normal form theorem \ref{thm -El2} using a Moser path argument, which would provide a purely geometric proof. We believe two small but important steps have been made towards this goal in Chapter \ref{sec:Equivalence of integrable systems}: a study of the different notions of equivalence between integrable systems in order to determine how the Moser path method manifests itself, and a more canonical choice of path connecting the integrable system to its normal form, inspired by \cite{CrFe} that seems absent from the literature.

The missing ingredient is leveraging the vanishing of the deformation cohomology of the normal form to prove vanishing of cohomology along the entire path. Once improvement 1.\ is understood, we should be able to adopt the arguments of \cite{CrFe}.

\item
Alternatively, or rather additionally, prove the normal form theorem directly using the Nash-Moser fast convergence argument from Chapter \ref{chapter:PDEs with Symmetry}. This requires specific estimates for the vanishing of the deformation cohomology of the nondegenerate normal form and possibly a slight generalization of the \nameref{General Main Theorem} on page \pageref{General Main Theorem}. We believe this result can be obtained in the near future for the purely elliptic case. This idea is inspired by \cite{Mar14}, and would show that the normal form is a consequence of rigidity.

\end{enumerate}

\section{Outline}
\label{sec:outline-thesis}

Again, one may say that this thesis is motivated by our desire to understand Eliasson's theorems.
Here are the main points:
\begin{enumerate}
\item 
\textbf{Foliations (Chapter 2):} the next chapter is, in large part,  
devoted to understanding the meaning of `the Lagrangian foliation of an integrable system' and why the linearization of `the Lagrangian foliation' is related to the commutativity relation from Theorem \ref{thm -El2}. One problem is, of course, the fact that ${\mu}$ may have singularities, hence one has to go beyond the well-understood case of regular foliations. The outcome will be that `the Lagrangian singular foliation' of ${(M, \mu)}$ should be interpreted as the sheaf ${\ucalC_{\mu}}$ on ${M}$ of functions that commute with the ${\mu_i}$'s. We will encounter quite a few different other interesting sheaves, but we will explain why ${\ucalC_{\mu}}$ is the best candidate for encoding `the Lagrangian foliation'. 

\item 
\textbf{Equivalences (Chapter 3):} the first part of the Chapter \ref{sec:Equivalence of integrable systems} (namely section \ref{ssec:equivalences -} on equivalences between integrable systems), starts from the desire to understand the equivalences that are hidden behind the previous two theorems, as well as a few more natural ones (most notably the equivalence  that is about preserving the orbits of the integrable system). We will see that each one of them will correspond to one of the sheaves that show up at the previous point. The sheaf ${\ucalC_{\mu}}$ encoding `the Lagrangian foliation' will be the one that allows us to make precise the fact that `the Lagrangian foliation is linearized'. The resulting notion of equivalence, called weak equivalence, is precisely the one that is implicit in Theorem \ref{thm -El2}. The equivalence that appears in \ref{thm -El1}, called `functional equivalence' will result from the (much more natural, but generally less well behaved) sheaf of functions on ${M}$ that are locally of type ${\phi\circ (\mu_1, \ldots, \mu_n))}$.

\item 
\textbf{Moser path method (Chapter 3, continued):} as we have already pointed out, we see Eliasson's theorems (or at least Theorem \ref{thm -El1}) as analogous to Conn's linearization theorem from Poisson Geometry. Originally proven by Conn \cite{Conn85} using hard analysis (Nash-Moser fast convergence method), later on a more geometric proof became available \cite{CrFe}: a Moser path argument that reduced the problem to the vanishing of certain cohomology group (Poisson cohomology), then a combination of standard results from Lie groups (such as averaging) extended to the realm of Lie groupoids. 

It is therefore natural to see whether such a proof can be found for the previous two theorems. 
The entire Chapter \ref{sec:Equivalence of integrable systems} is devoted to the first step of this plan: using the Moser path method to reduce the problem to the vanishing of certain cohomology groups. The relevant cohomology is `the deformation cohomology', which we discuss in section \ref{cis:cohomology}. For that we build on previous work of Miranda-Vu Ngoc \cite{MN05} who introduced the deformation complex adapted to weak equivalences. We should also mention here that, in this thesis, we do not carry out the second step in this plan, namely the vanishing of the cohomology, but we do believe that it is a matter of time until a proof will be found, and we do mean here a rather simple geometric proof.
 
\item 
\textbf{Rigidity theorem for PDEs with symmetries (Chapters 4--7):}
 the entire second half of the thesis
can be motivated by the search for an analytic proof, using the Nash-Moser techniques, for the two theorems. Note the analogy with the story of Conn's linearization theorem.

However, our actual motivation is wider: we would like to have a framework that can be applied not only to Eliasson's theorem, but also other local normal forms results in Differential Geometry (ranging from the classical Newlander-Nirenberg theorem, and including Conn's linearization theorem). We should mention here that one can find a functional analytic framework in the literature (due to Monnier-Zung \cite{MZ04} and then Miranda-Monnier-Zung \cite{MMZ12}), but what we have in mind here is a more geometric one, in which `the geometric structure' of interest are not abstract elements in a Hilbert space, but actual sections of a bundle which are a solution of a PDE (`integrability equation'). We believe that our resulting framework and theorem makes it easier to exploit Nash-Moser techniques (now absorbed inside our proof) to various geometric contexts, concentrating on the specifics of the context (e.g.\  without proving over and over again, in each case, some rather general estimates). 

We would also like to point out that we do pay special attention, and that complicates a bit the story, to achieving a framework that allows us to treat normal forms not only around points, but also around more general submanifolds (or just subspaces). Indeed, many geometric problems come with such normal forms (even for integrable systems, when passing from fixed points to more general singular points, one is looking at normal forms around the orbit through the singular point). 

All these subjects are discussed in Chapters \ref{chap:pseudogroups} and \ref{chapter:PDEs with Symmetry}, while the last chapter (Chapter \ref{chapter:Proof of the Main Theorem}) is devoted to the proof of the rigidity theorem. In Chapter \ref{chapter: Applications} we present some illustrations of the rigidity result: the baby example of Symplectic Geometry (Darboux theorem) and the Newlander-Nirenberg theorem. It is clear that another rather immediate application is Conn's linearization theorem (see also \cite{Mar14}).  We plan to treat the elliptic case of Eliasson's theorem in the article version of our work.    
\end{enumerate}

\chapter{The sheaves behind integrable systems}
\label{chapter:sheaves_behind}
\etocsettocdepth{2}
\localtableofcontents

\noindent
{\color{white}.}
\newline
In this chapter we discuss the sheaf theory that can be used to understand some aspects of integrable systems. The background idea is very simple: while an integrable systems comes with various different structures/aspects one can look at, and one sometimes concentrates only on some of those (e.g.\  the underlying linear action, or just the orbits, or `the Lagrangian foliation'), sheaf theory allows one to proceed in an unified manner: each such aspect can be encoded in (controlled by) a certain sheaf. Our concrete motivation for considering such sheaves is two folded.

\paragraph{Motivation 1: the \emph{`Lagrangian foliation'} of an integrable system.}
Another piece of structure that comes together with the notion of a completely integrable system ${(\omega,\mu)}$ on ${M}$ is that of `Lagrangian foliation'. Roughly speaking, that is the partition of ${M}$ into the level sets of ${\mu}$, and the statement that one often encounters in the literature is that this is `a Lagrangian {\it singular} foliation'. The main problem is that, while there is a well-established theory of regular foliations (i.e.\  for which all leaves have the same dimension), singular foliations are much more subtle; actually, there are several different (and unequivalent) ways to make sense of `singular foliations'. Here we insist on using the quotes since the name `the Lagrangian foliation of an integrable system' is often used in the existing literature in a rather vague way. 
\newline\newline
\noindent
The precise meanings of `singular foliations' that we will use are:
\begin{myenumerate}
\item partitions of ${M}$- where the leaves (${=}$ the equivalence classes) are required to be connected, but we do not even required them to be submanifolds. 
\item Stefan-Sussman foliations \cite{Stef, Sus} - which are partitions by immersed submanifolds, satisfying a smoothness condition (but without excluding different leaves having different dimensions).
\item Androulidakis-Skandalis foliations \cite{An-Sk} - which are defined by modules of vector fields which are involutive and locally finitely generated - see Remark \ref{Androulidakis and Skandalis} below. These are closely related to Stefan-Sussman foliations: any such foliation induces a Stefan-Sussman one. 
\item our framework (to be discussed below): a strongly involutive sheaf of vector fields or, equivalently, a closed sheaf of smooth function. These induce partitions, but they are not Stefan-Sussman foliations in general (the leaves may even fail to be smooth); they model connected (components of) fibers of smooth map (which, in general, do not fit in the Androulidakis-Skandalis framework). 
\end{myenumerate}

\noindent
Now, starting with an integrable system ${\mu: (M, \omega)\rightarrow \mathbb{R}^n}$ there are several, and different, interesting `singular foliation' to consider. It is important to keep in mind that the difference does not come only from the fact that they belong to different frameworks, but especially because they model different phenomena even from the intuitive point of view (it is that the different frameworks model different phenomena). Using our precise terminology, they are:

\begin{myenumerate}
\item[(O)] The `orbit foliation'. Intuitively, this corresponds to the partition of ${M}$ by the orbits of the induced infinitesimal action of ${\mathfrak{h}= \mathbb{R}^ n}$ on ${M}$. This is induced by a canonical local module of vector fields associated to ${(M, \omega, \mu)}$- the one spanned by the ${X_{\mu_i}}$'s. It makes sense as {\it a Androulidakis-Skandalis foliation}   (hence, in particular, as a {\it Stefan-Sussman foliation}).  
\item[(F)] The `fiber foliation'. This is about the partition of ${M}$ by the connected components of the fibers of ${\mu}$. This does not seem to fit in any of the existing frameworks - hence we will only be talking about it as {\it a partition}. It can be encoded/controlled by a sheaf (the sheaf ${\ucalC^{\textsc{fib}}}$ from Definition \ref{def-c-fib-sheaf}), but the main problem here is that this sheaf is not closed in general.
\item[(C)] The `central foliation'. This arises when trying to fix the previous problem, by passing to the closure ${\CM_{\mu}}$ or, equivalently, the commutant (all these will be discussed in detail in this section) of the sheaf induced by the ${\mu_i}$'s. This turns out to be {\it strongly involutive} - therefore it fits in the last framework mentioned above. 
However, one should be aware that this may fail to describe the fibers of ${\mu}$: there may exist two points ${x}$ and ${y}$ in a same (connected component of a) fiber of ${\mu}$, but such that ${f(x)\neq f(y)}$ for some ${f\in \CM_{\mu}}$, i.e.\  there may be two points that are in the same (connected component of the) fiber, without belonging to the same leaf.  We expect that this phenomena cannot occur under mild assumptions on ${\mu}$, but we believe it can occur in general. 
\end{myenumerate}
  
\noindent
The foliation from (C) is clearly the least intuitive one. However, it is the one that is used most of the times to (be able to) prove precise results; that may serve itself as (technical) motivation to consider it - but it still does not make it more intuitive. Our study from this chapter may be seen as an attempt to explain the foliation from (C) a bit more intuitively and to explain why and why it behaves best from a conceptual point of view, and the only one that deserves the name `the Lagrangian foliations of the integrable system'.

\paragraph{Motivation 2: equivalences of integrable systems} Another reason is that of understanding `equivalences' between integrable systems (to be discussed in detail in the first part of the next chapter); the question is how much of the underlying structure one would like to preserve. 

In principle, we would call two integrable systems
${\mu: (M, \omega) \rightarrow \mathbb{R}^n}$ and ${\mu': (M', \omega') \rightarrow \mathbb{R}^n}$ to be equivalent if there exists a symplectomorphism \[{\Phi: (M, \omega)\rightarrow (M', \omega')}\] such that 
\begin{equation}\label{gen -equiv -is} 
\mu'\circ \Phi= \phi\circ \mu.
\end{equation}
for `some bijection' ${\phi}$. Deciding how much of the structure underlying the integrable system one would preserve is, most of the times, about fixing the meaning of `some bijection'.  One could require it to be a diffeomorphism from ${\mathbb{R}^n}$ to itself (or from ${\mu(M)}$ to ${\mu(M')}$), or one could require it to be linear, or even the identity. Each such choice gives rise to another notion of equivalence. Moreover, it is desirable to eliminate the use of ${\phi}$. Sheaves of functions can sometimes be used to do so. This happens especially when the class of functions ${\phi}$ one is looking at is characterized by the fact that they preserve a certain sheaf ${\ucalA}$ on ${\mathbb{R}^n}$; in other words, one may say that each such sheaf on ${\mathbb{R}^n}$ gives rise to a notion of equivalence (for instance, for the sheaf of linear functions one is looking at ${\phi}$'s which are linear). In that case, one can associate to any integrable system ${\mu: (M, \omega) \rightarrow \mathbb{R}^n}$ `the ${\ucalA}$-span of ${\mu}$' or, more precisely, the pullback:
\[ \ucalA_{\mu}\defeq \mu^{*}\ucalA.\]
The functoriality of the pul -back construction implies that for any such equivalence ${(\Phi, \phi)}$ one must have
\[ \Phi^*(\ucalA_{\mu'})= \ucalA_{\mu}.\]
While this consequence of (\ref{gen -equiv -is}) does not make reference to ${\phi}$ anymore, in good cases it is actually equivalent to it.

\newpage
\section{Foliations via vector fields}

\etocsettocstyle{\subsubsection*{Local Contents}}{}
\etocsettocdepth{2}
\localtableofcontents

\noindent
{\color{white}.}
\newline
In this section we start from regular foliations (which, for integrable systems, corresponds to the situation when  ${\mu}$ has no singularities) and characterize them in terms of (sheaves of) vector fields in a way that can be used to model also singular foliations. The author has recently been made aware of similar discussions by Marco Zambon \cite{marco} and San Vũ Ng\d oc \cite{Ngoc06}.

\subsection{Regular foliations} 

We start by recalling the notion of (regular) foliations in terms of partitions by leaves. 

\begin{definition}\label{def -reg -foliations}
A \textbf{regular foliation}\index{foliation!regular} on ${M}$ is a partition ${\calP}$ into immersed connected submanifolds with the property that, locally on small enough open subsets ${U\sub M}$, the restriction ${\calP\vert_U}$ is given by the fibers of a submersion ${\pi: U\rightarrow \mathbb{R}^q}$ with connected fibers on ${U}$. Here ${\calP\vert_U}$ is the partition of ${U}$ by the connected components of all ${L\cap U}$ with ${L\in \calP}$. A submersion ${\pi}$ as in the definition is called a \textbf{trivializing submersion}\index{trivializing submersion} for ${\calF}$ over ${U}$. The foliation is called \textbf{simple}\index{foliation!simple} if it is globally induced by a submersion with connected fibers. 
\end{definition}
\noindent
The Frobenius theorem (recalled below) allows for an infinitesimal approach to foliations, via vector fields. Given a foliation ${\calP}$ on ${M}$, the the tangent bundle of ${\calP}$ on ${M}$ is the sub bundle 
\[ \calF\defeq T\calP \sub TM\]
consisting of vectors that are tangent to the leaves of ${\calP}$. The regularity of ${\calP}$ ensures that ${T\cal P}$ is a constant rank (smooth) sub bundle of ${TM}$. One should think of ${T\calP}$ as the infinitesimal counterpart of the foliation ${\calP}$. The Frobenius theorem identifies the main properties of ${T\calP}$ and allows one to study foliations via their infinitesimal counterpart:

\begin{local_theorem} T
he correspondence ${\calP\mapsto T\calP}$ defines a bijection between foliations on ${M}$ and vector sub bundles of ${\calF \subset TM}$ which are involutive (i.e.\  with the property that its sections give rise to a Lie subalgebra ${\gerX(M)}$).  
\end{local_theorem}

For this reason, from now on, when we will talk about foliations we will refer directly to the sub bundle ${\calF}$. 

\begin{remark} 
It is clear that singular foliations give rise to sub bundles of ${TM}$ which are far from being smooth; they deserve the name of `naive tangent bundles'. Think, for instance, of the partition of ${\mathbb{R}^2}$ by concentric circles (including the origin, viewed as a degenerate circle). 

Let us mention that one way to approach singular foliations is by modeling them via `generalized tangent bundle' or, more precisely, Lie algebroids. This notion encodes the main properties of ${\calF}$ (including the fact that it is a smooth vector bundle); the price to pay is that the bundle itself, call it ${A}$, is no longer a sub bundle of ${TM}$; instead, it is related to ${TM}$ via map ${\rho: A\rightarrow TM}$ (called the anchor of the algebroid). Its image ${\text{Im}(\rho)\subset TM}$ is usually singular and plays the role of the naive tangent bundle of the singular foliation.  From the point of view of partitions, the leaves are the maximal immersed submanifold ${L\subset M}$ which are tangent to ${\text{Im}(\rho)}$; their existence follow from the properties of ${A}$ (including its smoothness).

While the Frobenius theorem tells us that regular foliations can be viewed as Lie algebroids with injective anchor map, to model singular foliations it is natural to restrict to Lie algebroids whose anchor map is almost injective, in the sense that it is injective on a dense open. For instance, for the singular foliation of ${\mathbb{R}^2}$ by concentric circles, the natural Lie algebroid to consider is the one associated to the action of ${S^1}$ by rotations: as a vector bundle it is the trivial line bundle, and the anchor sends the trivializing section of ${A}$ to ${x\partial_y - y\partial_x}$. However, for the foliation of ${\mathbb{R}^3}$ by concentric spheres, while there is a similar Lie algebroid around (associated to the action of ${SO(3)}$ on ${\mathbb{R}^3}$), its anchor map is nowhere injective, and one can actually show that there is no Lie algebroid with almost injective anchor map that induces this foliation. 
\end{remark}

\noindent  
To find approaches to foliations that can be adapted to the singular case, the key is to encode them more algebraically. In the approach via vector fields, that means we should concentrate on 
the space of sections of ${\calF= T\calP}$:
    \[ \gerX_\calF(M)\defeq \Gamma(\calF) \]
and understand its main properties. 

\begin{local_theorem} 
For any (regular) foliation ${\calF}$ on ${M}$: 
\begin{myenumerate}  
    \item ${\gerX_\calF(M)}$ is a Lie subalgebra of ${\gerX(M)}$; in particular it is itself a Lie algebra. 
    \item ${\gerX_\calF(M)}$ is a ${\calC^{\infty}(M)}$-submodule of ${\gerX(M)}$ (so it is itself a ${\calC^{\infty}(M)}$-module). 
    \item as a ${\calC^{\infty}(M)}$-module, ${\gerX_\calF(M)}$ is finitely generated and projective. 
    \item If ${X\in \gerX_\calF(M)}$ vanishes at a point ${x_0\in M}$, then ${X}$ can be written as a sum of vector fields ${f\cdot X'}$ with ${X'\in  \gerX_\calF(M)}$ and ${f'}$ vanishing at ${x_0}$.
\end{myenumerate}  
Moreover, the construction 
\[ \calF\mapsto \gerX_\calF(M)\]
defines a one-to-one correspondence between regular foliation on ${M}$ and subspaces ${\calV\subset \gerX(M)}$ satisfying there conditions. 
\end{local_theorem}

\begin{proof} 
The main ingredient is Swan's theorem, which says that the functor of global sections, ${E\mapsto \Gamma(E)}$, defines an equivalence between the category of vector bundles over ${M}$ and the one of finitely generated projective ${\calC^{\infty}(M)}$-modules, but with one subtlety however: while morphisms, i.e.\  vector bundle maps ${f: E\rightarrow F}$ give rise to (and can be recovered from) morphisms of ${\calC^{\infty}(M)}$-modules, ${f_*: \Gamma(E)\rightarrow \Gamma(F)}$, the injectivity of ${f_*}$ is not equivalent to the injectivity of ${f}$, but to its almost injectivity. Hence, while a finitely generated projective ${\calC^{\infty}(M)}$ submodule ${X\subset \gerX(M)}$ must come from a a vector bundle ${A}$ together with a map ${\rho: A\rightarrow TM}$, to ensure that ${\rho}$ is injective (i.e.\  that ${X}$ comes from a regular foliation), we need the last condition in the previous list. 
\end{proof}

\subsection{Spaces of vector fields \& locality} 
To deal with singular foliations via spaces of vector fields ${\calV\subset \gerX(M)}$,  
we will have to give up on some of the properties listed above. Note that the first two exhibit the algebraic structure of ${\gerX_\calF(M)}$, while the last two reveal conditions on these structures. Therefore it is more natural to question the last two properties. On the other hand, giving up just the last condition, we end up with Lie algebroids with injective anchor map. As shown in the previous remark, this setting does not allow to handle even the foliation of ${\mathbb{R}^3}$ by concentric spheres (and even less so the ones coming from integrable systems). With these motivations in mind, we will only insist on the first two properties in the previous list. Let us fix the terminology:

\begin{definition} 
By a space of vector fields on ${M}$ \index{subspace of vector fields} we mean a linear subspace 
\[\calV \subset \gerX(M).\]
We say that:
\begin{myenumerate}
\item ${\calV}$ is a ${\calC^{\infty}(M)}$-module, or a \textbf{submodule}\index{subspace of vector fields!submodule} of ${\gerX(M)}$, if it is a ${\calC^{\infty}(M)}$-submodule of $\gerX(M)$.
\item ${\calV}$ is \textbf{involutive}\index{subspace of vector fields!involutive} if it is a Lie subalgebra of ${\gerX(M)}$.

\item 
${\calV}$ is of vector bundle type if it is the space of sections of a vector bundle over ${M}$. Swan's theorem states that any finitely generated projective ${\calC^{\infty}(M)}$-module is of vector bundle type (and of course also conversely). 

\end{myenumerate}
\end{definition}

\noindent
Here is one more interesting property that we like to point out: locality. This is the property that explains how such a global (algebraic) object such as ${\gerX_\calF(M)}$ encodes also local phenomena. 

\begin{definition} 
We say that a subspace ${\calV\subset \gerX(M)}$ satisfies the locality condition, or that ${\calV}$ is a \textbf{local subspace}\index{subspace of vector fields!local}, if: any ${X\in \gerX(M)}$ with the property that for each ${x\in M}$ there exists a vector field ${X^x\in \calV}$ such that ${X= X^x}$ in a neighborhood of ${x}$, must belong to ${\calV}$.
\end{definition}

\noindent
Any subset ${\calV\subset \gerX(M)}$ has an associated completion to a local one, ${\calV^{\textrm{loc}}\subset \gerX(M)}$, consisting of those ${X\in \gerX(M)}$ that satisfy the condition appearing in the previous definition. In general, ${\calV\subset \calV^{\textrm{loc}}}$, with equality if and only if ${\calV}$ is local.

Under the assumption that ${\calV}$ is a submodule (which, in principle, we want to keep), ${\calV^{\textrm{loc}}}$, and therefore the locality of ${\calV}$, has some more insightful descriptions:

\begin{lemma} 
For any  ${\calC^{\infty}(M)}$-submodule ${\calV\subset \gerX(M)}$, one has
\begin{eqnarray}
\calV^{loc} & = & \{ X\in \gerX(M): \phi\cdot X \in \calV,\quad\forall\, \phi\in \calC_{c}^{\infty}(M) \}\\
                 & = & \{ \textrm{locally finite sums} \sum\nolimits_i X_i: X_i\in \calV \}.
\end{eqnarray}

\end{lemma}

\begin{proof} 
Assume first that ${X\in \calV^{loc}}$. Hence we find an open cover ${\{U_i\}_i}$ and vector fields ${X_i}$ such that, over each ${U_i}$, ${X= X_i}$. We may assume that the cover comes with a partition of unity ${\{\eta_i\}}$. Then, for any ${\phi\in \calC_{c}^{\infty}(M)}$, the support of ${\phi}$ will intersect only a finite number of ${U_i}$'s, so 
${\phi\cdot X= \sum_i \phi\eta_i  \cdot X_i}$ is a finite sum of elements in ${\calV}$. Hence ${\phi\cdot X\in \calV}$. Conversely now, assume that ${\phi\cdot X\in \calV}$ for any compactly supported ${\phi}$. Then for any ${x\in M}$ choose ${\phi}$ compactly supported that is ${1}$ in a neighborhood of ${x}$. Then ${X= \phi \cdot X}$ on that neighborhood, while ${\phi\cdot X\in \calV}$, and hence ${X\in \calV^{loc}}$. 
The last equality is immediate (and similar). 
\end{proof}
\noindent

\subsection{From locality to sheaves of vector fields}

Locality is intimately related to the notion of sheaf. Let us first fix the terminology. 

\begin{definition} 
By a \textbf{sheaf of vector fields}\index{sheaf of vector fields} on ${M}$, or subsheaf ${\ucalV}$ of ${\gerX_M}$, we mean any subsheaf ${\ucalV}$ of ${\gerX_M}$ in the category of sheaves of vector spaces over ${\mathbb{R}}$. 
We say that:
\begin{myenumerate}
\item ${\ucalV}$ is \textbf{involutive}\index{sheaf of vector fields!involutive} if it is closed under the Lie bracket of vector fields.
\item ${\ucalV}$ is a \textbf{sheaf of submodules}\index{sheaf of vector fields!submodules} (or ${\ucalC^{\infty}_{M}}$-submodules, if we want to be more precise) if it is closed under multiplication by smooth functions (over each open). 
\item ${\ucalV}$ is of vector bundle type if it is the sheaf of sections of a vector sub bundle of ${TM}$.
\end{myenumerate}
\end{definition}
\noindent
Of course, the basic example is the sheaf ${\gerX_{\calF}}$ of vectors tangent to a foliation ${\calF}$. Here is the relationship with locality:

\begin{proposition} \label{prop -basic -sheaves}There is a one-to-one correspondence between
\begin{myenumerate}
\item ${\calC^{\infty}(M)}$-submodules ${\calV\subset \gerX(M)}$ which are local. 
\item sheaves of ${\ucalC^{\infty}_{M}}$-submodules ${\ucalV\subset \gerX_M}$. 
\end{myenumerate}
In this correspondence 
\[ \calV= \Gamma(M, \ucalV) \]
and, conversely, for a local submodule ${\calV\subset \gerX(M)}$, ${\ucalV}$ is defined by:
\[ U\mapsto \Gamma(U, \ucalV)\defeq \{ X\in \gerX(U): \phi\cdot X\in \calV\quad\forall\, \phi\in \calC_{c}^{\infty}(U) \}.\]
\end{proposition}
\noindent
\begin{proof} 
We first make a very general remark about subsheaves of sets, in the particular case ${\ucalV\subset \gerX_M}$: they have the property that if ${X\in \gerX(U)}$ and ${U}$ has an open cover ${\{U_i\}}$ such that ${X|_{U_i}\in \Gamma(U_i, \ucalV)}$ for all ${i}$, then ${X\in \Gamma(U, \ucalV)}$. 
Indeed, ${X^i\defeq X|_{U_i}}$ are local sections of ${\ucalV}$ that coincide on the overlaps, hence there is ${\tilde{X}\in \Gamma(U, \ucalV)}$ gluing them. At the level of ${\gerX_M}$, we have two sections ${X}$ and ${\tilde{X}}$ over ${U}$ that coincides over the ${U_i}$'s, hence they must coincide. Therefore ${X\in\Gamma(U, \ucalV)}$. 
Note also that, as in the previous lemma, the sheaf ${\ucalV}$ associated to ${\calV}$ can also be described without using the module structure:
\begin{equation}\label{general-associated-sheaf}
\Gamma(U, \ucalV)= \big\{X\in \gerX(U):  \forall x\in U  \ \exists X^x\in \calV:   X= X^x\ \textrm{near}\ x \big\}.
\end{equation}

Back to the proposition: we show that the two maps are well defined and are inverse to each other. First: ${\Gamma(M, \ucalV)}$ is a local submodule for any sheaf ${\ucalV}$ as in the statement. The main issue is locality. To check it, let ${X\in \gerX(M)}$
as in the definition of locality: ${X|_{U_x}= X^{x}|_{U_x}}$ with ${X^x\in \Gamma(M, \ucalV)}$. But then we have ${X|_{U_x}\in \Gamma(U_x, \ucalV)}$ and then the remark at the beginning implies that ${X\in \Gamma(M, \ucalV)}$. 

Second, one has to check that, starting with ${\calV}$, ${\ucalV}$ is a sheaf (clearly of submodules if ${\calV}$ is a submodule). But that is immediate. Also, it is clear that ${\calV \mapsto \ucalV \mapsto \Gamma(M, \ucalV)}$ gives back ${\calV}$, as ${\calV}$ is local. 

We are left with showing that the other composition is the identity. I.e., for a sheaf ${\ucalV}$ as above, the sheaf ${\ucalW}$ associated to ${\calW\defeq \Gamma(M, \ucalV)}$, coincides with ${\ucalV}$.  Using the description (\ref{general-associated-sheaf}) for ${\ucalW}$, the remark we started with immediately implies that
${\ucalW\subset \ucalV}$. For the converse, we start with ${X\in \Gamma(U, \ucalV)}$ and we have to show that for any ${x\in M}$ we find ${X^x\in \Gamma(M, \ucalV)}$ so that ${X= X^x}$ in a neighborhood of ${x}$. For that we choose a function ${\eta}$ supported in ${U}$ and that is ${1}$ near ${x}$, and we set ${X^x\defeq \eta\cdot X}$. 
\end{proof}
\noindent 

\begin{remark}
\label{lemma-with-strange-condition} 
The previous proof was written in such a way that the module structure is really used only at the end (multiplying by a bump function). All we need there is the that ${\ucalV}$ has \textbf{enough global sections}\index{enough global sections} in the sense that the canonical maps from global sections to germs 
\begin{equation}
\label{C0}
    (C0): \quad \textrm{germ}_x: \Gamma(M, \ucalV) \rightarrow \ucalV_x\quad \textrm{is surjective for all}\quad x\in M .
\end{equation}
Taking (\ref{general-associated-sheaf}) as definition of ${\ucalV}$, we find a one-to-one correspondence between:
\begin{myenumerate}
\item Local subspaces ${\calV\subset \gerX(M)}$.
\item subsheaves ${\ucalV\subset \gerX_M}$ with enough global sections. 
\end{myenumerate}
\end{remark}
\noindent
Of course, for our basic example of submodule, ${\calV= \gerX_{\calF}(M)}$, the previous result returns the sheaf ${\gerX_{\calF}}$ consisting of local sections of ${\calF}$. This indicates that, to handle foliations via vector fields, one may avoid using sheaves. The reason we use sheaves here is very simple: as soon as we will pass (in the next section) to the dual point of view (characterizing foliations via spaces of functions), sheaves will become inevitable. However, even when interested on the point of view of vector fields, we do believe that the sheaf-theoretic point of view is more insightful (see below) and powerful, and may allow for better frameworks.

\subsection{Finite generation: Androulidakis-Skandalis} 
Here is another basic property that is very common when dealing with sheaves modeling foliations (regular or not):

\begin{definition}
\label{sheaf-vf -finite -gen} 
A subsheaf ${\ucalV}$ of ${\gerX_M}$ is said to be \textbf{finitely generated}\index{sheaf of vector fields!finitely generated} if it is a sheaf of submodules and any point ${x\in M}$ admits a neighborhood ${U}$ and a finite number of sections ${X^1, \ldots, X^p\in \Gamma(U, \ucalV)}$ such that
\[ X|_{V}\in \textrm{Span}_{\calC^{\infty}(V)}\{ X^{1}|_{V}, \ldots,  X^{p}|_{V}\}\]
for any ${X\in \Gamma(U, \calV)}$ and any open neighborhood ${V}$ with ${\overline{V}\subset U}$.
\end{definition}
\noindent
One reason to further restrict to ${V}$'s (and not require only that ${\Gamma(U, \calV)}$ is generated as a module by ${X^1, \ldots, X^p}$) is to obtain a local notion. We can avoid using smaller ${V}$'s in the definition, provided we use compact supports:

\begin{lemma} 
For a sheaf ${\calV}$ of submodules of ${\gerX_M}$ the finite generation condition is equivalent to: any point ${x\in M}$ admits a neighborhood ${U}$ and a finite number of sections ${X^1, \ldots, X^p\in \Gamma(U, \calV)}$ such that any ${X\in \Gamma_{c}(U, \calV)}$ 
is in the ${\calC^{\infty}_{c}(U)}$-span of ${X^1, \ldots, X^p}$.
\end{lemma}

\begin{proof} 
We may assume that the ${U}$'s appearing above are relatively compact. Then we claim that the same ${U}$'s and the ${X^i}$'s as in the definition also work in the lemma, and the other way around. For the direct implication, we assume that ${X|_{V}}$ is in the ${\calC^{\infty}(V)}$-span of ${X^{1}|_{V}, \ldots,  X^{p}|_{V}}$ for any ${X}$ and ${V}$ as in the definition. Starting now with ${X\in \calV(U)}$ with compact support, call it ${K\subset U}$ we first choose ${V}$ relatively compact with ${K\subset V\subset \overline{V}\subset U}$, Applying the hypothesis to ${X}$ over ${V}$, we can write
\[ X|_{V}= f_1 \cdot X^{1}|_{V}+ \ldots + f_p \cdot X^{p}|_{V}, \quad \textrm{with} \quad f_i\in C^{\infty}(V). \]
Choose now ${\eta\in C_{c}^{\infty}(U)}$ supported inside ${V}$ and with ${\eta= 1}$ on ${K}$. We set
\[ \tilde{f}_i= \eta\cdot f_i \in \calC^{\infty}(U) .\]
Looking at ${\sum_i \tilde{f}_i \cdot X^i}$ (over ${U}$) we see that inside ${V}$ it is ${\eta\cdot X= X}$ and outside ${V}$ it is zero (as ${X}$ is). Hence ${X=  \sum_i \tilde{f}_i \cdot X^i}$ over the entire ${U}$.

Conversely, assume that the condition in the lemma holds. Let ${X\in \Gamma(U, \ucalV)}$, and ${V\subset U}$ open with ${V\subset \overline{V}\subset U}$. Choose now ${\eta\in \calC_{c}^{\infty}(U)}$ any function with ${\eta= 1}$ on ${V}$ (this is where we use ${U}$ relatively compact). Since ${\eta\cdot X}$ has compact supports, we find ${\tilde{f}_i \in \calC_{c}^{\infty}(U)}$ so that ${\eta\cdot X= \sum_i \tilde{f}_i \cdot X^i}$ over the entire ${U}$. Restricting to ${V}$, we find that ${X=  \sum_i f_i \cdot X^{i}|_{V}}$ where the ${f_i}$'s are the restrictions of the ${\tilde{f}_i}$.
\end{proof}
\noindent
\begin{remark}
\label{Androulidakis and Skandalis} 
From the existing approaches to singular foliations, the one that seems to be most solid (at least in terms of results that can be proven) is the one used by Androulidakis and Skandalis \cite{An-Sk}: a singular foliation in their sense is an involutive  subspace ${\calM}$ of ${\calX_{c}(M)}$ (compactly supported vector fields on ${M}$) which is a ${\calC^{\infty}(M)}$-submodule and, as such, it is locally finitely generated. Since they choose to use compact supports, we want to clarify the relationship with our discussion in the following sequence of remarks:
\begin{myenumerate}
\item one has a one-to-one correspondence between:
   \begin{myenumerate}
   \item subspaces ${\calM \subset \calX_{c}(M)}$ which are ${\calC^{\infty}(M)}$-submodules (${\equiv}$ ${\calC^{\infty}_{c}(M)}$ ones). 
   \item local submodules ${\calV\subset \calX(M)}$
   \end{myenumerate}
   In this correspondence, ${\calV= \calM^{loc}}$ and ${\calM= \Gamma_c(M, \ucalV)}$. 
   \item ${\calM}$ is involutive if and only if ${\calM^{loc}}$ is. 
\item the definition of ${\calM}$ being locally finitely generated is identical with the condition appearing in the previous lemma, for 
${\calM^{loc}}$. 
\end{myenumerate}
Hence, in our terminology, having a singular foliation in the sense of Androulidakis and Skandalis is the same as having an involutive local ${\calC^{\infty}(M)}$-submodule ${\calV\subset \gerX(M)}$ which is locally finitely generated. Or, by Proposition \ref{prop -basic -sheaves}, involutive sheaves of submodules ${\ucalV\subset \gerX_M}$ which are locally finitely generated. We should also recall (as mentioned in \cite{An-Sk} that such a foliation induces a partition of ${M}$ into leaves, which forms a foliation in the sense of Stefan-Sussman \cite{Stef, Sus} (however, for this one needs to correct the original arguments of Stefan and Sussman, and change some of their main statements -- please see \cite{Balan}). 

Unfortunately, the framework of Androulidakis and Skandalis  does not seem to be general enough for handling the singular foliations that arise from integrable systems.
\end{remark}
\noindent

\subsection{Strong involutiveness}

And here is one more interesting property of ${\calV= \gerX_{\calF}(M)}$ which is meaningful in the more general setting corresponding to singular foliations.

\begin{definition} \label{def -closed -sheaf}
A sheaf of vector fields ${\ucalV\subset \gerX_M}$ is said to be 
    \textbf{strongly involutive}
    \index{sheaf of vector fields!strongly involutive} 
if around any point in ${M}$ one can find a collection ${\Theta}$ of closed 1-forms (all defined on an open neighborhood ${U}$ of the point) such that, at each ${y\in U}$, the stalk  
${\ucalV_y}$ coincides with the zero-set of ${\Theta}$:
\[  \ucalV_y = \{ X\in \gerX_{M, y}: \theta_y(X)= 0\quad\forall\, \theta\in \Theta,\, y\in U \}.\]
\end{definition}
 \noindent 

As the terminology indicates, this notion implies involutiveness. This follows from the fact that ${\theta([X, Y])= \gerL_X(\theta(Y)) - \gerL_{Y}(\theta(X))}$ for any closed one-form ${\theta}$. What is not indicated by the terminology but is important conceptually (though obvious) is the that this condition also implies that we deal with sheaves of submodules. As a summary:

\begin{lemma} \label{lemma-closed-sheaf}
Any strongly involutive subsheaf ${\ucalV\subset \gerX_M}$ is automatically involutive, and a sheaf of ${\ucalC^{\infty}_M}$-submodules. 
\end{lemma}

The difference between involutivity and strong involutivity is subtle. The fact that, for a regular foliation ${\calF}$, ${\gerX_{\calF}}$ has this property follows immediately from the local triviality of foliations: if ${f: U\rightarrow \mathbb{R}^q}$ is a submersion inducing ${\calF|_{U}}$, with ${U}$ an open neighborhood of ${x}$, then the germs at ${x}$ of the 1-forms ${df_i}$ have as zero-set precisely ${\gerX_{\calF, x}}$. Therefore one may say that this condition is related to the very definition of a foliation (local triviality of the foliation) rather than to the involutivity of ${\calF}$. In some sense, the Frobenius theorem (and the infinitesimal approach to foliations) is possible precisely because the two notions coincide when dealing with sheaves of vector bundle type (i.e.\  of sections of a sub bundle of ${TM}$). Therefore we make a distinct statement.

\begin{proposition}
For subsheaves ${\ucalV\subset \gerX_M}$ which are sheaves of sections of a vector sub bundle of ${TM}$, strong involutivity is equivalent to involutivity. 
\end{proposition}
\noindent
\begin{remark}[More insight into the strong involutiveness condition:] In the Androulidakis-Skandalis approach to singular foliations, when the sheaves are involutive and locally finitely generated, strong involutiveness does not hold even in some very simple examples (see e.g.\  Remark \ref{rk-str-inv-counter}). In some sense, that framework is suited to study nice geometric examples, which have smooth leaves. However, it is not suited to study examples coming from integrable systems or, more generally, from fibers of smooth maps (fibers that may fail to be smooth). The strong involutivity is precisely designed to take care of such examples (in particular, one should talk about non-smooth leaves).

For yet another insight, note that Lemma \ref{lemma-closed-sheaf} holds under a weaker `involutiveness assumption': it suffices to require that, for any ${x\in M}$, ${\ucalV_x}$ coincides with the zero-set of a collection ${\Theta_x}$ of germs at ${x}$ of 1-forms, which are not necessarily closed but satisfy the condition that 
\[ d \Phi_x \subset \Phi_x\wedge \Omega^{1}_{x},\]
i.e.\  the differentials of the 1-forms on ${\Phi_x}$ are inside the submodule of ${\Omega^{*}_{x}}$ generated by ${\Phi_x}$. Equivalently, the algebraic ideal of ${\Omega^{*}_{x}}$ generated by ${\Phi_x}$ is a differential ideal. This condition is more natural from the point of view of exterior differential systems (see \cite{Bryant}).  But please note that, while our discussion (and, in general, the study of `singular foliations') as well as the theory of EDS both depart from, and are generalizations of, regular foliations, they go into very different directions: while we are interested on allowing singularities but insist on having leaves that give rise to a partition of ${M}$, the EDS are objects with may have many leaves (or, better: solutions) through each point. 
\end{remark}

\subsection{Conclusion}
While it seems that the general frameworks devoted to singular foliations that can be found in the existing literature are not general enough to handle the ones arising from integrable systems, in order to model such foliations we should use local sub spaces ${\calV\subset \gerX(M)}$ or, more generally, subsheaves ${\ucalV}$ of ${\gerX_{M}}$, that satisfy some of the properties discussed above such as:
\begin{myenumerate}
\item[1.\ ] it is involutive (closed under the Lie bracket).
\item[2.\  ] is is a ${\calC^{\infty}}$- submodule. 
\item[3.] it is strongly involutive. 
\item[4.] it is locally finitely generated.
\end{myenumerate}
In this list, 1.\  is a must, 2.\   is strongly desirable, and 3. and 4. are desirable.

One should also keep in mind that, despite this rather abstract/algebraic approach, one can still consider induced (singular) distribution and leaves. More precisely, for any subsheaf ${\ucalV\subset \gerX_{M}}$, the evaluation maps ${\textrm{ev}_x: \ucalV_x\rightarrow T_xM}$ define 
\begin{equation}\label{FxV}
\calF_x(\ucalV)\defeq \textrm{ev}_x(\ucalV_x)\subset T_xM,
\end{equation}
i.e.\  a singular distribution on ${M}$. Moreover, using integral curves of local vector fields in ${\ucalV}$ one generates an equivalence relation on ${M}$ by starting from the condition that ${x\cong y}$ if there exists such an integral curve between ${x}$ and ${y}$. The resulting equivalence classes deserve to be called the leaves induced by ${\ucalV}$. 
It is not clear to us which are the reasonable conditions under which these leaves are (immersed) submanifolds. The results of Stefan-Sussman \cite{Stef, Sus} (with the statements corrected - see \cite{Balan}) implies that that is the case if ${\ucalV}$ is a sheaf of submodules which is  involutive and locally finitely generated. 
\newpage

\section{Foliations via functions} 
\etocsettocstyle{\subsubsection*{Local Contents}}{}
\etocsettocdepth{2}
\localtableofcontents

\subsection{Basic functions}

Another way of encoding foliations, somehow dual to the point of view of vector fields, is obtained by looking at smooth functions. Roughly speaking, we are looking at smooth functions on the leaf space ${M/\calF}$. As a topological space (with the quotient topology) this may be very pathological, but
one can still make sense of smooth functions by working up on ${M}$. 

\begin{definition} 
Given a (regular) foliation ${\calF}$ on ${M}$, the \textbf{algebra} ${\calC_{\calF}^{\infty}(M)}$ of \textbf{${\calF}$-basic functions} is the sub-algebra of ${\calC^{\infty}(M)}$ consisting of those smooth functions that are constant on the leaves of ${\calF}$. 
\end{definition}
\noindent
However, unlike the case of ${\gerX_\calF (M)}$, ${\calC_{\calF}^{\infty}(M)}$ can not be used to recover ${\calF}$. This happens especially when the leaf space is very pathological. For example, for a foliation which admits a dense leaf (e.g.\  a Kronecker foliation on the torus), any basic function must be constant. Hence ${\calC_{\calF}^{\infty}(M)= \mathbb{R}}$ and this does not distinguish ${\calF}$ from the trivial foliation (i.e.\  with a single leaf). It is for this reason that we have to pass to, and work with, sheaves.

\begin{definition} 
The \textbf{sheaf} ${\ucalC_\calF^{\infty}}$ of \textbf{${\calF}$-basic functions} is the sheaf on ${M}$: 
\[ 
    U\mapsto \ucalC_{\calF}^{\infty}(U)\defeq \calC_{\calF|_{U}}^{\infty}(U).
\]
The sections over an open set ${U\subset M}$ are the basic functions wrt the restriction of ${\calF}$ to ${U}$. 
\end{definition}
\noindent
We stress however that the sheaf  ${\calC_\calF}$ does have much weaker properties than the sheaf ${\gerX_\calF}$. Above all, it usually does not have enough global sections (in the precise sense of Remark \ref{lemma-with-strange-condition}), see e.g.\  the previous example. In contrast, ${\gerX_\calF}$ is soft, and even a fine sheaf (and the same is true for any sheaf ${\ucalV}$ induced by any ${\calC^{\infty}(M)}$-submodule ${\calV\subset \gerX(M)}$).

We now concentrate on making the duality between the spaces ${\calC^{\infty}_\calF(M)}$ and ${\gerX_\calF(M)}$ more precise. The first remark is the following:

\begin{lemma} 
For any regular foliation ${\calF}$ on ${M}$, one has 
\[ 
    \calC_{\calF}^{\infty}(M)
    = 
    \{ 
        f\in \calC^{\infty}(M): 
        \gerL_X(f)=0 \quad\forall\, X\in \gerX_\calF(M)
    \}.
\]
\end{lemma}

\begin{proof}

For the direct inclusion note for any ${X\in \gerX_\calF(M)}$ each integral curve ${t\mapsto \phi_{X}^{t}(x)}$ of ${X}$ stays in one leaf. Therefore, if ${f}$ is basic and ${x\in X}$, then ${f \phi_{X}^{t}(x)}$ is constant and then ${\gerL_X(f)(x)= 0}$. Conversely, if ${f}$ is in the right hand side, it follows that it must be constant on all the integral 
curves of vector fields ${X\in \gerX_\calF(M)}$. But any two points that belong to the same lead can be connected by a sequence of paths that are integral curves of such vector fields. Hence ${f}$ must be constant on the leaves.
\end{proof}

\noindent
\begin{remark} 
It is interesting to switch the roles of ${\calC_\calF^{\infty}(M)}$ and ${\gerX_\calF(M)}$. The outcome is a bit different: 
it is not true that
\[ \gerX_\calF(M) = \{ X\in \gerX(M): \gerL_X(f)=0 \quad\forall\, f\in \calC_{\calF}^{\infty}(M)\}.\]
Think for instance of the Kronecker foliation mentioned above. 
Instead, 
\[ 
    \gerX_\calF(M) = \{ X\in \gerX(M): \gerL_{X_{|U}}(f)=0 \quad\forall\, U\subset M \,\,\text{open}, \quad\forall\, f\in \calC_{\calF}^{\infty}(U)\}.
\]
\end{remark}
\noindent
It is interesting to understand these phenomena in the more general context of spaces of vector fields/smooth functions. Given a subspace
$\calV\subset \gerX(M),$
the natural space (actually an algebra) of functions to consider is 
\[ \Fun^{\textrm{naive}}(\calV)\defeq \{ f\in \calC^{\infty}(M): \gerL_X(f)=0 \quad\forall\, X\in \calV \}.\]
As we shall explain below, it is better to use a sheaf-theoretic perspective. That means that we require the equations ${\gerL_X(f)=0}$ to hold only germ -wise. This leads to the following space of functions (see also below): 
\[ \Fun(\calV)\defeq \{f\in \calC^{\infty}(M): \gerL_{X}(f_{|V})=0 \quad\forall\, V\subset M \,\,\text{open}, \quad\forall\, X\in \Gamma(V, \ucalV)\},\]
where ${\ucalV}$ is the sheaf associated to ${\calV}$. 

\begin{remark} 
One does have ${\Fun^{\textrm{naive}}(\calV)= \Fun(\calV)}$ if ${\calV}$ is local. While the locality is an assumption that we make most of the time, we owe an explanation for the proceeding in this greater generality: in the dual version of this discussion (when the role of functions and vector fields are switched), the dual version of the locality is just too strong and cannot be ignored. Actually, we already know that in the dual picture we do have to use sheaves. From the sheaf theoretic point of view the problem with the naive spaces is that they do not defines sheaves (see below). 
\end{remark}

\noindent
 \subsection{Sheaves of functions \& duality with vector fields}
It is time to pass to general sheaves of functions. We first fix the terminology.

\begin{definition} 
By a sheaf of (smooth) functions on ${M}$ we mean a subsheaf ${\ucalC \subset \ucalC^{\infty}_{M}}$ in the category of sheaves of vector spaces over ${\mathbb{R}}$. 
\end{definition}

\noindent
Of course, we will mainly encounter subsheaves of rings (or, equivalently, of algebras) of ${\ucalC^{\infty}_{M}}$. Even more, many of the subsheaves ${\ucalC}$ we will encounter will be closed not only under the ring operations, but under all ${C^{\infty}}$-ring operations: for any ${f_1, \ldots, f_p\in \Gamma(U, \ucalC)}$ and any smooth functions ${F:\mathbb{R}^p \rightarrow \mathbb{R}}$, ${F\circ (f_1, \ldots, f_p)}$ is again a (local) section of ${\ucalC}$. In this case we will say that ${\ucalC\subset \ucalC_{M}^{\infty}}$ is \textbf{a subsheaf of ${C^{\infty}}$-rings}. To explain the terminology, recall that a ${C^{\infty}}$-ring $R$ comes not only with the usual ring operations, but with operations 
${\textrm{op}_F: R^p\rightarrow R}$ - one for each smooth function ${F:\mathbb{R}^p \rightarrow \mathbb{R}}$ (${F(x, y)= x+ y}$ induces the addition of ${R}$, and ${F(x, y)= xy}$ induces the multiplication). The natural axioms to require are:
\begin{myenumerate}
\item If ${F'= F\circ (F_1, \ldots, F_p)}$ with ${F_i: \mathbb{R}^k\rightarrow \mathbb{R}}$, the induced operation is 
\[ \textrm{op}_{F'}=  \textrm{op}_{F}\circ (\textrm{op}_{F_1}, \ldots, \textrm{op}_{F_p});\]
\item The operation induced by the projection ${\textrm{pr}_k: \mathbb{R}^p\rightarrow \mathbb{R}}$ on the ${k}$-th component is the similar projection ${R^p\rightarrow R}$.
\end{myenumerate}
We refer \cite{Joyce2} for more details. Of course, the basic example is ${\calC^{\infty}(M)}$, with 
\[
    {\textrm{op}_F(f_1, \ldots, f_p)= F\circ (f_1, \ldots, f_p)}.
\]
Hence, saying that ${\ucalC\subset \ucalC_{M}^{\infty}}$ is a subsheaf of ${C^{\infty}}$-rings can be thought of as saying that ${\ucalC}$ is closed under these ${C^{\infty}}$-ring operations. 

For us, the most interesting concept that arises from the ${C^{\infty}}$-ring point of view is that of `finite generation'. While it should be clear what it means for a subsheaf of rings ${\ucalC\subset \ucalC_{M}^{\infty}}$ to be finitely generated:

\begin{definition}
\label{functionally-fg} 
We say that ${\ucalC}$ is \textbf{functionally finitely generated} if any point ${x\in M}$ admits an neighborhood ${U\subset M}$ and sections \[{f_1, \ldots, f_p\in \Gamma(U, \ucalC)}\] such that, for any other local section ${g\in \Gamma(V, \ucalC)}$ defined over an open subset ${V\subset U}$, one has
\[ g= F\circ (f_1|_{V}, \ldots, f_p|_{V})\]
for some smooth function ${F: \mathbb{R}^p\rightarrow \mathbb{R}}$ (compare with Definition \ref{sheaf-vf -finite -gen}). 
\end{definition}
\noindent
We now return to the interplay between sheaves of functions and sheaves of vector fields.

\begin{definition}\label{def -Fun}
Given a subsheaf ${\ucalV}$ of ${\gerX_M}$, \textbf{the induced sheaf of functions ${\Fun(\ucalV)}$} is defined by 
\[ U\mapsto 
\{f\in \calC^{\infty}(U): \gerL_{X}(f_{|V})=0 \quad\forall\, V\subset U \,\,\text{open}, \quad\forall\, X\in \Gamma(V, \ucalV)\}.\]
\end{definition}
\noindent
\noindent
The problem with the naive version of this sheaf, 
\begin{equation}\label{naive -fund} 
U\mapsto \{ f\in \calC^{\infty}(U): \gerL_X(f)=0 \quad\forall\, X\in \Gamma(U, \ucalV)\}
\end{equation} 
is that, in general, it does not make sense even as a presheaf: starting with such a function ${f}$ over ${U}$, its restriction to a smaller ${U'\subset U}$ may fail to belong to satisfy the similar condition for all ${X\in \Gamma(U', \ucalV))}$. However, this situation does not arises under reasonable assumptions (however, the discussion is not as pedantic as it may seem at first, since we will soon pass to the dual point of view, and then the duals of the `reasonable assumptions' become too strong). Let us summarize the outcome:

\begin{proposition}
\label{lemma-with-strange-condition2} 
For any subsheaf ${\ucalV}$ of ${\gerX_M}$, ${\Fun(\ucalV)}$ is a subsheaf of ${C^{\infty}}$-rings of ${\ucalC^{\infty}_{M}}$. 

Its naive version (\ref{naive -fund}) 
is a subsheaf of ${\ucalC^{\infty}_{M}}$ if and only if it coincides with ${\Fun(\ucalV)}$. This happens automatically when ${\ucalV}$ has enough global sections (see Remark \ref{lemma-with-strange-condition}) - in particular when
${\ucalV}$ is a sheaf of ${\ucalC^{\infty}_{M}}$-modules (or just a soft sheaf). 
\end{proposition}

\noindent
\begin{proof} 
For the sheaf part, we only have to check that the restriction map from an open set ${U}$ to a smaller set ${U'}$ works well (indeed, since we work with functions on ${M}$, the remaining sheaf conditions are immediate). Start with 
${f\in \calC^{\infty}(U)}$ belonging to (\ref{naive -fund}). To check that ${f|_{U'}}$ belongs to (\ref{naive -fund}) with ${U}$ replaced by ${U'}$, let ${X'\in \Gamma(U', \ucalV)}$ and we show that ${\gerL_{X'}(f_{|U'})}$ vanishes. 
Look at its germ at an arbitrary point ${x\in U'}$. Using the hypothesis, we find ${X\in \Gamma(U, \ucalV)}$ which has the same germ at ${x}$ as ${X'}$. But then the germ at ${x}$ of ${\gerL_{X'}(f_{|U'})}$ coincides with the one
of ${\gerL_{X}(f)= 0}$. 

To see that ${\Fun(\ucalV)}$ is closed under all the ${C^{\infty}}$-ring operations one remarks that, for ${F:\mathbb{R}^p \rightarrow \mathbb{R}}$ and (local) functions ${f_i}$ on ${M}$, one has
\[ 
    \gerL_X(F\circ (f_1, \ldots, f_p))
    = \sum\nolimits_k
    \frac{\partial F}{\partial x_k}\circ (f_1, \ldots, f_p) \gerL_X(f_k).
    \qedhere
\]
\end{proof}

\noindent
It should now be clear that we can also proceed dually, interchanging the role of the functions and of the vector fields. Therefore, one now starts with a subsheaf ${\ucalC}$ of the sheaf ${\ucalC^{\infty}_M}$ of smooth functions on ${M}$.
The `naive' way to induce a sheaf of vector fields on ${M}$ would be to consider 
\[ U\mapsto \{ X\in \gerX(U): \gerL_{X}(f)= 0 \quad\forall\, f\in \ \Gamma(U, \ucalC)\} .\]
As before, this is not well defined as a presheaf, unless we consider conditions like in the previous lemma (but for the sheaf ${\ucalC}$). However, such conditions are way too strong for our purpose: they are not satisfied even by
the sheaves ${\calC_{\calF}}$ of basic functions associated to regular foliations.  

\begin{definition}
\label{def -Vect}
For a subsheaf ${\ucalC}$ of the sheaf ${\ucalC^{\infty}_M}$ of smooth functions on ${M}$, \textbf{the sheaf of vector fields on ${M}$ induced by ${\ucalC}$}, denoted \textbf{${\Vect(\ucalC)}$}, is defined by 
\[ U\mapsto \{ X\in \gerX(U): \gerL_{X_{|U'}}(f)= 0 \quad\forall\, U'\subset U \quad\forall\, f\in \ \Gamma(U', \ucalC) \}. \]
\end{definition}
\noindent
\noindent
Again, the defining condition can be interpreted as saying that \[{\gerL_{\textrm{germ}_x(X)}(f_x)= 0}\] for all ${x\in U}$ and each germ ${f_x\in \ucalC_x}$, but one should be aware that, at each ${x}$, the solutions of these equations does not describe the stalk of 
${\Vect(\ucalC)}$ at ${x}$ (which is smaller).

\subsection{The strong involutive closure}
The previous constructions are related to strong involutiveness, e.g.\ the following is immediate: 

\begin{lemma} 
For any subsheaf ${\ucalC}$ of the sheaf ${\ucalC^{\infty}_M}$, ${\Vect(\ucalC)}$ is a subsheaf of ${\gerX_M}$ which is strongly involutive (in the sense of Definition \ref{def -closed -sheaf}). 
\\
In particular (cf. Lemma \ref{lemma-closed-sheaf}), it is an involutive sheaf of ${\ucalC^{\infty}_M}$-submodules. 
\end{lemma} 

However, the relationship with strong involutiveness is more substantial. 
To explain this, we consider the following operation:

\begin{definition} 
For a subsheaf ${\ucalV\subset \gerX_M}$, we define the \textbf{strong involutive closure}, or the \textbf{Inv-closure} of ${\ucalV}$, as the following sheaf of vector fields:
\[ \textrm{Inv}(\ucalV)\defeq \Vect(\Fun(\ucalV)).\]
\end{definition}
 \noindent 

As usual, the involutive closure of ${\ucalV}$ is the smallest involutive subsheaf containing ${\ucalV}$. This and the following proposition will be proved together with Proposition \ref{prop -closures -C} and their corollaries. 

\begin{proposition}\label{prop -closures -V}
For any subsheaf ${\ucalV\subset \gerX_M}$, the following are equivalent:
\begin{myenumerate}
\item ${\ucalV}$ is strongly involutive.
\item ${\ucalV= \textrm{Inv}(\ucalV)}$.
\item ${\ucalV= \Vect(\ucalC)}$ for some subsheaf ${\ucalC\subset\ucalC^{\infty}_M}$.
\end{myenumerate}
In general, ${\ucalV}$ is a subsheaf of ${\textrm{Inv}(\ucalV)}$ and ${\textrm{Inv}(\ucalV)}$ is always strongly involutive. 
\end{proposition}
\noindent 

Of course, one can also interchange the roles of ${\ucalV}$ and ${\ucalC}$. For a sheaf ${\ucalC}$ of functions the condition ${\ucalC= \Fun(\ucalV)}$ for some sheaf of vector fields ${\ucalV}$ says that, locally, ${\ucalC}$ is defined as the common zero set of a of a collection of first order differential operators. For this reason, we use the name `closed' below.

\begin{definition}
\label{def -cl -sheaves -ftcs} 
For a subsheaf ${\ucalC \subset \ucalC^{\infty}_{M}}$ define its (differential) closure as
\[ \overline{\ucalC}\defeq \Fun(\Vect(\ucalC)).\]
We say that ${\ucalC}$ is \textbf{closed} if ${\overline{\ucalC}= \ucalC}$. 
\end{definition}
\noindent
In the spirit of the previous proposition, we have:

\begin{proposition}
\label{prop -closures -C} 
For any subsheaf ${\ucalC \subset \ucalC^{\infty}_{M}}$ , the following are equivalent:
\begin{myenumerate}
\item ${\ucalC}$ is closed.
\item ${\ucalC= \overline{\ucalC}}$,
\item ${\ucalC= \Fun(\ucalV)}$ for some subsheaf ${\ucalV\subset \gerX_M}$.
\end{myenumerate}
In general, ${\ucalC}$ is a subsheaf of ${\overline{\ucalC}}$ and ${\overline{\ucalC}}$ is always closed. 
\end{proposition}
\noindent

Restricting to strongly involutive/closed sheaves, we obtain: 

\begin{corollary}\label{one-to-one -corr -sheaves} The constructions 
\[ \ucalV\mapsto \Fun(\ucalC), \quad  \ucalC\mapsto \Vect(\ucalV)\]
define one-to-one correspondences (and are inverse to each other) between:
\begin{myenumerate}
\item strongly involutive subsheaves of ${\gerX_{M}}$.
\item closed subsheaves of ${\ucalC^{\infty}_{M}}$.
\end{myenumerate}
\end{corollary}

\begin{proof}  
({\it of Propositions \ref{prop -closures -V}, 
\ref{prop -closures -C} and the Corollary})
For the moment, let us take the equivalence between 1.\  and 2.\   of Proposition \ref{prop -closures -V} as definition of strong involutivity. 

We are then in a general setup that is very common and is probably well known in category theory. Since we will be using it later on, let us describe it here using the previous notations, but in greater generality. 
First of all, for an ordered set ${F}$, by a closure operator we mean an increasing function 
\[ \textrm{Cl}: F\to F\]
satisfying ${f\leq \textrm{Cl}(f)}$ for all ${f\in F}$ and which is a projection, i.e.\  \[{\textrm{Cl}\circ \textrm{Cl}(f)= \textrm{Cl}(f)}\] for all ${f}$. Any such operator gives rise to the notion of ${f\in F}$ being ${\textrm{Cl}}$-closed by the condition ${f= \textrm{Cl}(f)}$. With this we immediately see that 
${\textrm{Cl}(f)}$ appears as the smallest closed element that is larger than ${f}$.

Our closure operators arise from the following setting. We deal with two ordered sets ${F}$ and ${V}$ (in our case the ones of sheaves of functions, and of vector fields, respectively) together with two functions
\[ \Vect: F \rightarrow V, \quad \Fun: J\rightarrow I,\]
both decreasing, satisfying the adjunction condition: for ${v\in V}$ and ${f\in F}$,
\[ v\leq \Vect(f) \Longleftrightarrow f\leq \Fun(v).\]
The first claim is that 
\[ \textrm{Cl}\defeq \Fun\circ \Vect: F \to F, \quad \textrm{Inv}\defeq \Vect\circ \Fun: V\to V\]
are closure operators. For instance, for ${\textrm{Cl}}$, the condition ${f\leq \textrm{Cl}(f)}$ follows from the adjunction condition applied to ${v= \Vect(f)}$. To see that ${\textrm{Cl}}$ is a projection, we start from the basic inequalities:
\[ f\leq \Fun\circ \Vect(f), \quad v\leq \Vect\circ \Fun(v) .\]
In the first one we replace ${f}$ by ${\Vect(v)}$ and we apply ${\Vect}$ to the second one, to conclude that
\[ \Fun= \Fun\circ \Vect\circ \Fun .\]
This clearly implies that ${\textrm{Cl}}$ is a projection. Moreover, we also find that each element of type ${\Vect(v)}$ is ${\textrm{Cl}}$-closed, while any closed element is clearly of this type. The same apply to the other closure operator ${\textrm{Inv}}$ and we also see that ${\Vect}$ and ${\Fun}$ induces bijections, inverse to each other, between ${\textrm{Cl}= \overline{( -)}}$-closed and ${\textrm{Inv}}$-closed elements. These imply all the statements from the propositions and the corollary.

We are left with proving that, for a subsheaf ${\ucalV\subset \gerX_M}$, it is ${\textrm{Inv}}$-closed if and only if it is strongly involutive in the sense of Definition \ref{def -closed -sheaf}. For the direct implication, if we write ${\ucalV= \ucalC}$ for some sheaf of functions ${\ucalC}$, we see that ${\ucalV}$ is locally defined 
as the common zero set of the 1-forms which are the differentials of the functions in ${\ucalC}$. Hence ${\ucalV}$ is involutive. For the converse, we consider ${\ucalC\defeq \Fun(\ucalV)}$ and it suffices to show that ${\ucalV= \Vect(\ucalC)}$. By the general nonsense, `${\subset}$' is immediate. For the converse, let us assume that 
${X}$ is a section over ${U}$ of ${\Vect(\ucalC)}$ and we show that any ${x\in U}$ admits a neighborhood ${U_x}$ such that ${X|_{U_x}\in \Gamma(U_x, \ucalV)}$. Fix ${x}$, and consider a family of closed 1-forms ${\Theta}$ as in Definition \ref{def -closed -sheaf}, defined over some open neighborhood ${U'\subset U}$ of ${x}$. Choosing ${U'}$ contractible, the 1-forms will be exact hence we find a family ${\Psi}$ of functions (defined over ${U'}$) such that ${\calV_y}$ is the common zero-set of all ${\textrm{germ}_y(df)}$, for all ${y'\in U'}$. On one hand, this implies that ${\Psi\subset \Gamma(U', \ucalC)}$. Since ${X\in \Gamma(U, \Vect(\ucalC))}$, the restriction of ${X}$ to ${U'}$ will be killed (via Lie derivative) by all sections over ${U'}$ of ${\Gamma(U', \ucalC)}$, hence also by ${\Psi}$, This implies that for each ${y\in U'}$, ${\textrm{germ}_y(X)}$ is killed by ${\Theta_y}$, hence ${X|_{U'}\in \Gamma(U', \ucalV)}$ as desired.
\end{proof}
\noindent

\subsection{Conclusion} 
    To model singular foliations from using functions, we see that: 
\begin{myenumerate}

\item 
We should use subsheaves ${\ucalC}$ of ${\ucalC^{\infty}_{M}}$ and not just subspaces of ${\calC^{\infty}(M)}$. 

\item 
It is desirable that the sheaves ${\ucalC}$ are sheaves of algebras. 

\item 
Even more, it is desirable that they are closed. 

\item 
When we encounter subsheaves ${\ucalC}$ which do not satisfy the previous two properties, one should pass to their closures ${\overline{\ucalC}}$ (which do satisfy them). 

\item 
We should keep in mind the relationship with the point of view of vector fields (the constructions ${\ucalV\mapsto \Fun(\ucalV)}$ and ${\ucalC\mapsto \Vect(\ucalC)}$), which, at the level of strongly involutive subsheaves of vector fields and closed subsheaves of functions, becomes an equivalence. 
\end{myenumerate}
As at the end of the previous section, and inspired by the case of regular foliations, one can still talk about the `leaves' associated to a sheaf of functions ${\ucalC}$: these are the classes of the equivalence relation on ${M}$ given by: ${x\cong y}$ iff there exists a continuous path ${\gamma}$ from ${x}$ to ${y}$ with the property that for any local section ${f}$ of ${\ucalC}$, ${f\circ \gamma}$ is locally constant (on its maximal domain). However, the relationship between the leaves of ${\ucalC}$ and the ones of ${\Vect(\ucalC)}$ is not clear to us, even when ${\ucalC}$ is closed. 

Finally, one should also have in mind that closed sheaves of functions (and strongly involutive sheaves of vector fields) model `foliations that are induced by smooth maps' (and these are precisely the ones that arise from integrable systems). E.g., for a precise statement: for a strongly involutive sheaf of vector fields ${\ucalV= \Vect(\ucalC)}$, if ${\ucalC}$ is functionally finitely generated (cf. Definition \ref{functionally-fg}) then, locally, ${\ucalV}$ is given by kernels of smooth maps, i.e.\  any point in ${M}$ admits an open neighborhood ${U}$ and a smooth function ${\mu:U \rightarrow \mathbb{R}^k}$ with the property that ${\Gamma(U, \calV)}$ consists of the vector fields on ${U}$ that are killed by ${(d\mu)}$.  

\begin{remark}[Sheaves of functions are not always appropriate] 
\label{rk-str-inv-counter} One should also be aware that, although the approach via functions (and actually the entire discussion on the strong involutiveness) is appropriate when looking at foliations arising from integrable systems, it may be not so appropriate for other natural examples of singular foliations. For instance, for the `foliation' of ${\mathbb{R}}$ with leaves the origin and the intervals ${( -\infty, 0)}$ and ${(0, \infty)}$, there is an obvious choice of a space/sheaf ${\ucalV_0}$ of vector fields: the one consisting of vector fields that vanish at ${0}$. The induced singular distribution is trivial at the origin and it is the entire tangent space elsewhere, so the induced leaves are precisely the ones we wanted. On the other hand ${\ucalV_0}$ is not closed, and its closure is the entire ${\gerX(\mathbb{R})}$ (in particular there exists no ${\ucalC}$ such that ${\Fun(\ucalV)= \ucalV_0}$). It is hard to imagine a sheaf of functions ${\ucalC}$ that models this foliation. 
\end{remark}

\noindent 
\begin{remark}[Sheaves of functions are sometimes better] 
\label{rk -functions -better}
Finally, let us point out two advantages of the point of view of sheaves of functions. The first one is that there is a natural pullback operation along smooth maps ${\mu: M\rightarrow N}$: for a subsheaf ${\ucalC\subset \ucalC^{\infty}_{N}}$, there is a natural 
\textbf{pullback subsheaf} 
\[ \mu^*\ucalC \subset \ucalC^{\infty}_{M}.\] 
Its sections over an open subset ${U\subset M}$ are those functions ${f: U\rightarrow \mathbb{R}}$ with the property that 
for all ${a\in U}$ there are open neighborhoods ${U_a}$ of ${a}$ in ${M}$ and ${V_a}$ of ${\mu(a)}$ in ${N}$, and a smooth function
${h: V_a\rightarrow \mathbb{R}}$ that is a section of ${\ucalC}$, 
such that
\[ \mu(U_a)\subset V_a, \quad f|_{U_a}= h\circ \mu|_{U_a} .\]
One should be aware that this is not the usual pullback of abstract sheaves, denoted ${\mu^{ -1}\ucalC}$. There is a canonical map 
${\mu^*\ucalC \rightarrow \mu^{ -1}\ucalC}$, but it may fail to be injective. Alternatively, ${\mu^{\star}\ucalC}$ is the image sheaf of the canonical map ${\mu^{ -1}\ucalC \rightarrow \calC^{\infty}_{M}}$. On the other hand, the operation ${\ucalC\mapsto \mu^*\ucalC}$ has the usual functorial properties. 

A second advantage of using functions is that it applies to more general `smooth spaces' aside from smooth manifolds. 
In the literature there are many notions of `smooth spaces' (e.g.\  the differential spaces of \cite{Joao} or the smooth schemes of \cite{Joyce2}, see \cite{Joao-thesis} for a nice overview). Here is a simple example that is relevant when looking at integrable systems: any subspace ${A\subset \mathbb{R}^n}$ comes equipped with a `sheaf of smooth functions on ${A}$', denoted ${\ucalC_{A}^{\infty}}$- which can be thought of as the pullback, in the previous sense, of the sheaf of smooth functions on ${\mathbb{R}^n}$. Explicitly, a function ${f:A\rightarrow \mathbb{R}}$ (and similarly on opens inside ${A}$) will be called smooth if any ${a\in A}$ admits an open neighborhood ${V_a\subset \mathbb{R}^n}$ and a smooth function ${f_a: V_a\rightarrow \mathbb{R}^n}$ such that 
\[ f|_{A\cap V_a}= f_{a}|_{A\cap V_a}.\]
Of course, ${\ucalC_{A}^{\infty}}$ is a sheaf of ${C^{\infty}}$-rings. One can talk about subsheaves ${\ucalC\subset \ucalC^{\infty}_{A}}$ and any subsheaf of ${\ucalC^{\infty}_{\mathbb{R}^n}}$ restricts to a subsheaf of ${\ucalC^{\infty}_{A}}$. 
\end{remark}

\newpage
\section{Lagrangian foliations} 
\etocsettocstyle{\subsubsection*{Local Contents}}{}
\etocsettocdepth{2}
\localtableofcontents

\noindent
{\color{white}.}
\newline
Now we introduce the symplectic structure into our discussion. Our main purpose is to understand the slogan: integrable systems give rise to Lagrangian foliations. Hence, first of all, one has to make sense of `Lagrangian foliations' on a symplectic manifold ${(M, \omega)}$. As before, we start with regular foliations. Then the meaning is clear:

\subsection{The regular case}
 
\begin{definition} 
A regular foliation ${\calF}$ on a symplectic manifold ${(M, \omega)}$ is called \textbf{Lagrangian} if all the leaves of ${\calF}$ are Lagrangian submanifolds of ${(M, \omega)}$. The notions of isotropic and coisotropic foliation is defined similarly. 
\end{definition}
\noindent
Note that this notion is directly related to that of integrable systems:

\begin{proposition}\label{reg -fol -locally -is}
 A regular foliation ${\calF}$ on a symplectic manifold ${(M, \omega)}$ is Lagrangian if and only if is is locally induced by an integrable system without singularities, i.e.\  the trivializing submersions ${\pi: U\rightarrow \mathbb{R}^q}$ (cf. Definition \ref{def -reg -foliations}), endowed with ${\omega|_{U}}$, are integrable systems. 
\end{proposition}
\noindent  

The proof is postponed for the end of this section. Passing to the formulation in terms of tangent bundles (via the Frobenius theorem), we immediately find:

\begin{lemma} 
A regular foliation ${\calF}$ on a symplectic manifold ${(M, \omega)}$ is Lagrangian if and only if each tangent space ${T_x\calF}$ is a Lagrangian subspace of the symplectic vector space ${(T_xM, \omega_x)}$. 
\end{lemma}

More generally, one can talk about isotropic or coisotropic foliations on a symplectic manifold, and they are relevant when studying more general moment maps. In general, for a foliation ${\calF}$ (understood as a sub bundle of ${TM}$) on a symplectic manifold ${(M, \omega)}$, its symplectic orthogonal is rarely a foliation. That is one more reason to consider sheaves of vector fields that are not necessarily involutive. The `good case scenario' is described by the following result which is well-known:

\begin{local_theorem}
\label{thm -orthog -fol} 
For a foliation ${\calF}$ on a symplectic manifold ${(M,\omega)}$, ${\calF^{\perp_{\omega}}}$ is a foliation if and only if the sheaf ${\ucalC_{\calF}^{\infty}}$ of ${\calF}$-basic functions on ${M}$ is involutive, i.e.\  it is closed under the Poisson bracket ${\{\cdot, \cdot\}_{\omega}}$ induced by ${\omega}$ on ${\ucalC^{\infty}_{M}}$. 
\end{local_theorem}

It will be particularly interesting to try to understand how this theorem can survive in more singular cases.

Next, we pass to interpretations in terms of the corresponding spaces of vector fields (and later using smooth functions), interpretations that can be used to handle also singular foliations. Assuming now that the symplectic manifold ${(M, \omega)}$ is fixed, for a sub space ${\calV\subset \gerX(M)}$ we define its symplectic orthogonal
\[ \calV^{\perp_\omega}\defeq \{X\in \gerX(M): \omega(X, V)= 0\quad\forall\, V\in \calV\}.\]
We say that ${\calV}$ is Lagrangian (w.r.t.\  ${\omega}$) if ${\calV^{\perp_\omega}= \calV}$. Similarly we define the notion of isotropic and coisotropic. Note that ${\calV^{\perp_\omega}}$ is always a ${\calC^{\infty}(M)}$-submodule (and local). Hence for ${\calV}$ to be Lagrangian (even if ${\calV}$ is not), it first has to be a (local) submodule. On the other hand, ${\calV^{\perp_\omega}}$ is in general not a Lie sub module (even if ${\calV}$ is). 

\begin{proposition} 
A regular foliation ${\calF}$ on a symplectic manifold ${(M, \omega)}$ is Lagrangian if and only if the subspace ${\gerX_{\calF}(M)\subset \gerX(M)}$ is Lagrangian.

More generally, for any regular foliation ${\calF}$, one has ${\gerX_{\calF}(M)^{\perp_{\omega}}= \Gamma(M, \calF^{\perp_{\omega}})}$.
\end{proposition}
\noindent
\begin{proof}
It suffices to prove the last equality. But that is immediate since every vector in a fiber of a vector bundle can be extended to a smooth section of that bundle.\end{proof}
\noindent
 \subsection{Via sheaves of vector fields \& symplectic orthogonals}
We now pass to sheaves. For a subsheaf ${\ucalV\subset \gerX_M}$, we would like to consider `the sheaf ${U\mapsto \Gamma(U, \ucalV)^{\perp_\omega}}$'. However, as with the definition of ${\Fun(\ucalV)}$, in full generality one has to be a bit careful.

\begin{definition}
For a subsheaf ${\ucalV\subset \gerX_M}$, the orthogonal of ${\ucalV}$ with respect to ${\omega}$ is the new sheaf of vector fields ${\ucalV^{\perp_{\omega}}}$ defined by:
\[ 
    \Gamma(U, \ucalV^{\perp_{\omega}})
    \defeq 
    \{X\in \gerX(U): \omega(X|_{V}, Y)= 0 \quad\forall\, V\subset U
    \,\,\text{open}, \quad\forall\, Y\in \Gamma(V, \ucalV) \}.\]
We say that ${\ucalV}$ is Lagrangian if ${\ucalV^{\perp_{\omega}}= \ucalV}$. 
\end{definition}
\noindent
Similar to Proposition  \ref{lemma-with-strange-condition2}, we have:

\begin{lemma} 
For any subsheaf ${\ucalV}$ of ${\gerX_M}$, ${\ucalV^{\perp_{\omega}}}$ is a sheaf of ${\ucalC^{\infty}_{M}}$-submodules. Moreover, under the following condition on the evaluation maps from global sections to tangent vectors:
\begin{equation}\label{C1}
(C1): \quad \textrm{ev}_x: \Gamma(M, \ucalV) \rightarrow \calF_x(\ucalV)\quad \textrm{is surjective for all}\quad x\in M 
\end{equation}
(see \ref{FxV}), one has  ${\ucalV^{\perp_{\omega}}(U)= \ucalV(U)^{\perp_{\omega}}}$  for all ${U}$. 
\end{lemma}

\begin{remark} 
\label{rk -C1}
Obviously the condition (C1) is satisfied if ${\ucalV}$ has enough global sections in the sense of Remark \ref{lemma-with-strange-condition}. In that case we know that ${\ucalV}$ corresponds to a local subspace ${\calV\subset \gerX(M)}$ 
and we see that the orthogonal sheaf ${\ucalV^{\perp_{\omega}}}$ corresponds to the submodule ${\calV^{\perp}}$. 
\end{remark}
\noindent
From the previous lemma we immediately conclude: 

\begin{proposition} 
A regular foliation ${\calF}$ on a symplectic manifold ${(M, \omega)}$ is Lagrangian if and only if the sheaf ${\gerX_\calF}$ is Lagrangian. 
More generally, for any regular foliation ${\calF}$, ${\gerX_{\calF}^{\perp_{\omega}}}$ is the sheaf of sections of ${\calF^{\perp_{\omega}}}$.
\end{proposition}

\noindent
Symplectic orthogonals gives rise to an interesting operation on sheaves.  
\begin{definition}
\label{def -omega -closed} 
For a sub sheaf ${\ucalV\subset \gerX_M}$, we define the closure of ${\ucalV}$ with respect to ${\omega}$ as:
\[ \overline{\ucalV}^{\omega}\defeq  (\ucalV^{\perp_{\omega}})^{\perp_{\omega}}.\]
We say that ${\ucalV}$ is ${\omega}$-closed if it coincides with its closure with respect to ${\omega}$. 
\end{definition}

\noindent
We now summarize the main properties of this operation:

\begin{lemma}
For any subsheaf ${\ucalV\subset \gerX_M}$:

\begin{myenumerate}

\item 
${\ucalV \subset \overline{\ucalV}^{\omega}}$.

\item 
${\overline{\ucalV}^{\omega}}$ is closed. Actually, the ${\omega}$-orthogonal of any subsheaf is ${\omega}$-closed. 

\item 
When ${\ucalV}$ is induced by a local subspace ${\calV\subset \gerX(M)}$ (cf. Proposition \ref{prop -basic -sheaves} and the remark following it), ${\overline{\ucalV}^{\omega}}$ is induced by the subspace ${(\calV^{\perp_{\omega}})^{\perp_{\omega}}}$.

\end{myenumerate}
\end{lemma}

 \begin{proof} 
The last part follows from the last part of Remark \ref{rk -C1}. The first two follow by the formal arguments mentioned already in the proof given after Corollary \ref{one-to-one -corr -sheaves}: this time the two functors are equal, and given by ${\ucalV\mapsto \ucalV^{\perp}}$. 
 The fact that this reverses the order of inclusions is clear, while the adjunction condition, i.e.\ 
\[ 
    \ucalV \subset \ucalW^{\perp} 
    \Longleftrightarrow 
    \ucalW \subset \ucalV^{\perp},
\]
 is immediate as well. 
\end{proof}
\noindent

\subsection{Via sheaves of functions \& commutants}

We now pass to the characterization of the Lagrangian condition in terms of (sheaves of) functions. The way to use the symplectic form ${\omega}$ on functions is via the induced Poisson bracket ${\{\cdot, \cdot\}}$. The basic operation, dual to that of taking symplectic orthogonals, is that of
commutants in Poisson algebras. 
Recall that, given a Poisson algebra ${(A, \{\cdot, \cdot\})}$, \textbf{commutant}\index{commutant} of a subspace ${C\subset A}$ is defined as
\[ C'\defeq \{a\in A: \{a, c\}= 0\quad\forall\, c\in C \}.\]
One says that ${C}$ is \textbf{self-centralizing} in the Poisson algebra $A$ if ${C= C'}$. In principle, we are interested in this property applied to ${\calC_{\calF}(M)}$ inside the Poisson algebra ${(\calC^{\infty}(M), \{\cdot, \cdot\})}$. However, as we have already seen, when working with functions we have to be more careful. First of all:

\begin{definition} 
For a subsheaf ${\ucalC\subset \ucalC_{M}^{\infty}}$, we define the commutant of ${\ucalC}$ w.r.t.\  ${\omega}$ as the subsheaf  ${\ucalC'\subset \ucalC_{M}^{\infty}}$ defined by
\[ \Gamma(U, \ucalC')= \{f\in \ucalC^{\infty}(U): \{f|_{V}, g\}= 0 \quad\forall\, V\subset U \,\,\text{open}, \quad\forall\, g\in \Gamma(V, \ucalC) \}.\]
We say that ${\calC}$ is \textbf{strongly ${\omega}$-involutive}\index{sheaf of functions!strongly involutive} if ${\ucalC= \ucalD'}$ for some subsheaf ${\ucalD\subset \ucalC_{M}^{\infty}}$.
\end{definition}
\noindent
Similar to previous remarks, the equation in the definition of the commutants is that ${\{\textrm{germ}_x(f), g_x\}= 0}$ for all ${x\in U}$ and ${g_x\in \ucalV_x}$. 
Continuing the reasoning from the previous discussions: ${\calC'(U)= \calC(U)'}$ if ${U}$ is an open cover which ${\calC|_{U}}$ has enough sections, i.e.\  satisfies (C0) from  Remark \ref{lemma-with-strange-condition} on ${U}$ (requiring (C0) over the entire ${M}$ would be too strong). 
But, again, the stalks of ${\calC'}$ are not the commutants of the stalks of ${\calC}$ (the latter ones may fail to form an étale space), even under the condition of having enough sections. 
In particular, the condition for a sheaf to be self-centralizing (see below) will not be equivalent to its stalks being self-centralizing.

\begin{definition} 
For a subsheaf ${\ucalC\subset \ucalC_{M}^{\infty}}$ we define its \textbf{${\omega}$-involutive closure} as
\[ \textrm{Inv}_{\omega}(\ucalC)\defeq (\ucalC')'.\]
We say that ${\ucalC}$ is \textbf{${\textrm{Inv}_{\omega}}$-closed} if ${\ucalC= \textrm{Inv}_{\omega}(\ucalC)}$. 
\end{definition}
\noindent
As before, and with the same arguments, one has that ${\textrm{Inv}_{\omega}}$ is a closure operator and being strongly ${\omega}$-involutive is equivalent to being ${\textrm{Inv}_{\omega}}$-closed . 
Also the following is immediate from the Jacobi identity of the Poisson bracket:

\begin{lemma} 
Any strongly ${\omega}$-involutive subsheaf ${\ucalC\subset \ucalC_{M}^{\infty}}$ is ${\omega}$-involutive in the sense that its local sections are closed under the Poisson bracket ${\{\cdot, \cdot\}_{\omega}}$.
\end{lemma}

We now pass to the property that is relevant for Lagrangian foliations. 

\begin{definition}
We say that a subsheaf ${\ucalC\subset \ucalC_{M}^{\infty}}$ is 
\textbf{self-centralizing}\index{sheaf of functions!self-centralizing} 
if ${\ucalC'= \ucalC}$. 
\end{definition}

\noindent
Now, the fact that a regular foliation is Lagrangian can be characterized as follows: 

\begin{local_theorem}
\label{thm-the-tricky-one0} 
A regular foliation ${\calF}$ on a symplectic manifold ${(M, \omega)}$ is Lagrangian if and only if the sheaf ${\ucalC_{\calF}^{\infty}}$ is self-centralizing. 
\end{local_theorem}

\subsection{Comparing the two approaches}
We now concentrate on the relationship between the Lagrangian condition in terms of sheaves of function and sheaves of vector fields. We start from the last theorem, restated for our purpose as:
 
\begin{local_theorem}
\label{thm-the-tricky-one} 
For a regular foliation ${\calF}$ on a symplectic manifold ${(M, \omega)}$, the sheaf ${\gerX_{\calF}}$ is Lagrangian if and only if the sheaf ${\ucalC_{\calF}^{\infty}}$ is self-centralizing. 
\end{local_theorem}

The rest of this section is motivated by understanding this theorem from a general perspective. Of course, it can be proven in a `down to earth fashion', making direct use of ${\calF}$ and its properties, but we would like to emphasize that the proof is not trivial (and, as the previous remark indicates, the main difficulty is to prove that, under the the self-centralizing condition, ${\calF^{\perp}}$ is isotropic). Moreover, it is interesting and even enlightening to see how much of such proofs hold in general (hence, in principle, for singular foliations as well) and pinpoint where/how the regularity is actually used.

\begin{remark}
\label{important -remark}
Here is one remark that one should keep in mind. It is natural to break the Lagrangian condition into:
\begin{myenumerate}
\item ${\calF}$ is coisotropic: i.e.\  ${\calF^{\perp}\subset \calF}$ or, using sheaves, ${\gerX_{\calF}^{\perp}\subset \gerX_{\calF}}$, 
\item ${\calF}$ is isotropic: i.e.\  ${\calF\subset \calF^{\perp}}$ or, using sheaves, ${\gerX_{\calF}\subset \gerX_{\calF}^{\perp}}$. 
\end{myenumerate}
And it is true (and will be proven below) that ${\calF}$ is coisotropic if and only if ${(\ucalC_{\calF}^{\infty})'\subset \ucalC_{\calF}^{\infty}}$. However, it is not true that ${\calF}$ being isotropic is equivalent to 
${\ucalC_{\calF}^{\infty}\subset (\ucalC_{\calF}^{\infty})'}$. What makes the Lagrangian case work (or what one may forget when passing to the general isotropic case) is that the condition ${\calF^{\perp}= \calF}$ includes the fact that
${\calF^{\perp}}$ is a foliation. Under this extra assumption, it is indeed true that ${\calF}$ is isotropic if and only if ${\ucalC_{\calF}^{\infty}\subset (\ucalC_{\calF}^{\infty})'}$. 

Even looking at simple isotropic foliations, i.e.\  at foliations induced by submersions 
    \[
        {\pi: (M, \omega)\rightarrow B}
    \]
with isotropic fibers, the fact that ${\calF^{\perp}}$ is again a foliation does not follow automatically, and is a very natural condition to impose. Actually, that is precisely the notion of 
\textbf{symplectically complete isotropic fibration} introduced by Dazord and Delzant \cite{Da-De} when extending the action-angle variables from Lagrangian fibrations. This notion is also very interesting point of view: for a fibration ${\pi: (M, \omega)\rightarrow B}$ with isotropic fibers, the condition that ${\calF^{\perp}}$ is again a foliation is equivalent to the fact that ${M}$ admits a Poisson structure (necessarily unique) such that ${\pi}$ becomes a Poisson map. Starting with a Poisson manifold, one then talks about \textbf{isotropic realizations} of it (see also section 7 of \cite{PMCT2} and the references therein).

In the line of the previous terminology, by a \textbf{symplectically complete isotropic foliation} on a symplectic manifold ${(M, \omega)}$ we mean any isotropic foliation ${\calF}$ with the property that its symplectic orthogonal ${\calF^{\perp}}$ is a foliation as well. 
\end{remark}
\noindent
\begin{proposition} \label{prop -gen -Lagr -sheaves} Let ${\ucalV\subset \gerX_M}$ be strongly involutive and let ${\ucalC\subset \ucalC_{M}^{\infty}}$ be the corresponding closed subsheaf in the one-to-one correspondence from Corollary \ref{one-to-one -corr -sheaves}. Then 
\begin{myenumerate}
\item ${\ucalV^{\perp}\subset \ucalV}$ if and only if ${\ucalC\subset \ucalC'}$.
\item ${\ucalV\subset \ucalV^{\perp}}$ implies that ${\ucalC'\subset \ucalC}$, and the converse is true under the hypothesis that also ${\ucalV^{\perp}}$ is strongly involutive. In general, ${\ucalC'\subset \ucalC}$ is equivalent to ${\ucalV\subset \textrm{Inv}(\ucalV^{\perp})}$. 
\end{myenumerate}
In particular, ${\ucalV}$ is Lagrangian implies that ${\ucalC}$ is self-centralizing, and the converse is true if also 
${\ucalV^{\perp}}$ is strongly involutive.   
\end{proposition}
\noindent
As explained in Remark \ref{important -remark}, for the isotropic case, the condition on the involutivity of ${\ucalV^{\perp}}$ is necessary even in the case of regular foliations. 

Before proving the previous proposition, we collect some of the basic properties of the operations of taking orthogonals and commutants.

\begin{lemma}
For any subsheaf ${\ucalV\subset \gerX_M}$ one has
\begin{equation}\label{eq(2)} 
\Fun(\ucalV^{\perp})\subset \Fun(\ucalV)'
\end{equation}
and this becomes an equality if ${\ucalV}$ is strongly involutive.

On the other hand, for any subsheaf ${\ucalC\subset \ucalC_{M}^{\infty}}$ one has
\begin{equation}\label{eq(1)}
\Vect(\ucalC)^{\perp} \subset \Vect(\ucalC')
\end{equation}
and the right hand side is actually the involutive closure of the left one: ${\Vect(\ucalC')= \textrm{Inv}(\Vect(\ucalC)^{\perp})}$. 
Moreover, the commutant ${\ucalC'}$ is closed and ${\overline{\ucalC}}$ has the same commutant as ${\ucalC}$:
\begin{equation}\label{eq(lemma1)} 
\overline{\ucalC'}= \overline{\ucalC}'= \ucalC' .
\end{equation}
\end{lemma}

\begin{proof}
The inclusions (\ref{eq(2)}) and (\ref{eq(1)}) follow by a direct check. For instance, for the first one (and without writing the domains of the sections), we have to show that if a function ${f}$ satisfies 
\[ L_X(f)= 0 \quad \forall \quad X\in \ucalV^{\perp}\]
then ${\{f, g\}= 0}$ for all functions ${g}$ with the property that
\[ L_Y(g)= 0, \quad \forall \quad X\in \ucalV.\]
For such a ${g}$ we note that ${X\defeq X_g}$ is orthogonal to ${\ucalV}$. Hence we can use it in the condition on ${f}$ to deduce that
${L_{X_g}(f)= 0}$, i.e.\  ${\{f, g\}= 0}$. Similarly the other inclusion.

We now prove the last part, i.e.\  (\ref{eq(lemma1)}). Applying ${\Fun}$ to (\ref{eq(1)}), and next to it 
using (\ref{eq(2)}) applied to ${\ucalV= \ucalV(\ucalC)^{\perp}}$ we deduce that  
\begin{equation}\label{strong -inclusion}
\Fun(\Vect(\ucalC')) \subset \Fun(\Vect(\ucalC)^{\perp})\subset \Fun(\Vect(\ucalC))'= \overline{\ucalC}'.
\end{equation}
In particular, ${\overline{\ucalC'}\subset \overline{\ucalC}'}$. But ${\ucalC\subset \overline{\ucalC}}$ implies that ${\overline{\ucalC}'\subset \ucalC'}$ hence 
\[ \overline{\ucalC'}\subset \overline{\ucalC}'\subset \ucalC' .\]
Using ${\ucalD\subset \overline{\ucalD}}$ for ${\ucalD= \ucalC'}$, we find that all these inclusions become equalities. 

Moreover, also the inclusions (\ref{strong -inclusion}) used along the way must become equalities. One of them is:
\[ \Fun(\Vect(\ucalC)^{\perp})= \Fun(\Vect(\ucalC))';\]
this holds for any ${\ucalC}$, i.e.\  one has equality in (\ref{eq(2)}) for any ${\ucalV}$ of type ${\Fun(\ucalC)}$, i.e.\  for any strongly involutive ${\ucalV}$. 

The other one is: 
\[ \Fun(\Vect(\ucalC')) = \Fun(\Vect(\ucalC)^{\perp}).\]
But, in general, ${\Fun(\ucalV)= \Fun(\textrm{Inv}(\ucalV))}$. Applying this to ${\ucalV\defeq \Fun(\ucalC)^{\perp}}$ and using 
that ${\Fun}$ is injective on strongly involutive sheaves, we deduce the equality from the statement related to (\ref{eq(1)}).
\end{proof}
\noindent
We now break Proposition \ref{prop -gen -Lagr -sheaves} into parts and prove some more general statements.

\begin{lemma} 
For any subsheaves ${\ucalC\subset \ucalC^{\infty}_M}$ and ${\ucalV\subset \gerX_M}$, one has the following implications
\begin{myenumerate}
\item ${\ucalV^{\perp}\subset V \Longrightarrow \Fun(\ucalV)\subset \Fun(\ucalV)'}$.
\item ${\ucalC \subset \ucalC'  \Longrightarrow \Vect(\ucalC)^{\perp}\subset \Vect(\ucalC)}$.
\item If ${\ucalV}$ is involutive: ${\ucalV\subset \ucalV^{\perp}   \Longrightarrow \Fun(\ucalV)'\subset \Fun(\ucalV)}$.
\item ${\ucalC'\subset \ucalC \Longrightarrow \Vect(\ucalC)\subset \textrm{Inv}(\Vect(\ucalC)^{\perp})}$.
\end{myenumerate}
\end{lemma}

\begin{proof}
For 1, apply ${\Fun}$ to the given inclusion, followed by (\ref{eq(2)}). 
For 2, apply ${\Vect}$ to the given inclusion and put (\ref{eq(1)}) in front.
For 3, apply ${\Fun}$ to the given inclusion and then use the equality in (\ref{eq(2)}) (due to the fact that ${\ucalV}$ is strongly involutive). 
For 4, apply ${\Vect}$ to the given inclusion and use ${\Vect(\ucalC')= \textrm{Inv}(\Vect(\ucalC)^{\perp})}$ (from the previous lemma).
\end{proof}
\noindent
It is clear that the previous lemma implies Proposition \ref{prop -gen -Lagr -sheaves}.

We see that, to prove Theorem \ref{thm-the-tricky-one} using the general setting, we are left with clarifying whether 
${\ucalV^{\perp}}$ is strongly involutive, under the assumption that ${\ucalV}$ is (which, by Remark \ref{important -remark}, is not enough even for regular foliations) together with ${\ucalC= \ucalC'}$. This should remind us also of Theorem \ref{thm -orthog -fol} which answers precisely this type of question, for regular foliations. Inspired by the case of regular foliations, one natural attempt is to show that the (strong) involutivity of ${\ucalV^{\perp}}$ is related to the (strong) involutivity of ${\ucalC}$ (which is automatic when ${\ucalC= \ucalC'}$). The start looks promising:

\begin{lemma}\label{promissing -lemma} 
For any subsheaf ${\ucalV\subset \gerX_M}$, if ${\ucalV^{\perp}}$ is strongly involutive, then so is ${\Fun(\ucalV)}$.
\end{lemma}

\begin{proof} 
Since ${\ucalV}$ is strongly involutive, we know that the inclusion (\ref{eq(2)}), when applied to ${\ucalV^{\perp}}$, becomes equality. For the resulting term ${\Fun(\ucalV^{\perp})'}$ we apply again (\ref{eq(2)}) (to ${\ucalV}$ now) and then the commutant. We find (remembering that twice the orthogonal gives ${\overline{V}}$):
\[ \Fun(\overline{\ucalV})= \Fun(\ucalV^{\perp})'\supset \Fun(\ucalV)''\]
Applying now ${\Fun}$ to the basic inclusion ${\ucalV\subset \overline{\ucalV}}$ we find
\[ \Fun(\ucalV) \supset \Fun(\overline{\ucalV}).\]
Combining the two centered equations, we find that ${\Fun(\ucalV)''\subset \Fun(\ucalV)}$, which is the reverse of the fundamental inclusion for the double commutant. Hence equality must hold, and ${\Fun(\ucalV)}$ will be strongly involutive. 
\end{proof}
\noindent
\subsection{Intermezzo: the Hamiltonian envelope}

Although the previous lemma looks promising, we would need the converse of it, and that seems very hard to prove directly even in the case of regular foliations (and may very well fail in general). 
This discussion 
brings us to another basic construction that applies to general sheaves on symplectic manifolds. 

\begin{definition}
\label{def -Ham -env} 
For a subsheaf ${\ucalC\subset \ucalC^{\infty}_{M}}$, we define the \textbf{Hamiltonian envelope} of ${\ucalC}$ as the sheaf of vector fields 
\[ \textrm{Ham}(\ucalC)\subset \gerX_M\] 
which is the local ${\ucalC^{\infty}_{M}}$ span of the Hamiltonian vector fields ${X_f}$ with ${f\in \ucalC}$. More precisely, sections of ${\textrm{Ham}(\ucalC)}$ over an open set ${U}$ are the vector fields ${X\in \gerX(U)}$ with the property that any point ${x\in U}$ admits a neighborhood ${U_x\subset U}$ such that 
\[ X|_{U_x}\in \textrm{Span}_{\calC^{\infty}(U_x)}\{X_f: f\in \Gamma(U_x, \ucalC)) \}.\]
\end{definition}
 \noindent

The Hamiltonian envelope is more appropriate (than ${\Vect(\ucalC}$) when trying to obtain foliations in the Androulidakis-Skandalis sense - see Proposition \ref{AS-fol-curious} below. 
On the other hand, the construction of the Hamiltonian envelope is conceptually interesting because it show that the basic operations we discussed (the general ones ${\ucalV\mapsto \Fun(\ucalV)}$ and ${\ucalC\mapsto \Vect(\ucalC)}$, as well as the one induced, ${\omega}$- ${\ucalV\mapsto \ucalV^{\perp}}$ and ${\ucalC\mapsto \ucalC'}$) admit certain factorization. 
One discovers this construction also when wondering whether strongly involutive sheaves are ${\omega}$-closed: trying to write ${\Vect(\ucalC)}$ as the symplectic orthogonal of another sheaf, one finds ${\textrm{Ham}(\ucalC)}$. Here is the summary:

\begin{lemma}
\label{basic -Ham} 
For any subsheaf ${\ucalC\subset \ucalC^{\infty}_{M}}$, one has
\[ \Vect(\ucalC)= \textrm{Ham}(\ucalC)^{\perp},\]
\[ \ucalC'= \Fun(\textrm{Ham}(\ucalC)).\]
\end{lemma}

\begin{proof} 
In principle, these follow by a straightforward check, using \[ {\omega(X_f, X)= \gerL_X(f)}.\] However, given the subtleties involved in the definitions (e.g.\  the stalks of ${\ucalC'}$ are not the commutants of the stalks of ${\ucalC}$, etc), let us give here all the details for the more involved one: the second equality. 

For the direct inclusion, start with ${f}$ a section of ${\ucalC'}$ over ${U}$. Hence ${f\in \calC^{\infty}(U)}$ and ${\{f|_V, g\}= 0}$ for all ${V\subset U}$ open and ${g\in \Gamma(V, \ucalC)}$. To prove that it is a section over ${U}$ of the other sheaf, we have to show that ${\gerL_{X}(f|_V)= 0}$ 
for all ${V\subset U}$ open and ${X\in \Gamma(V, \textrm{Ham}(\ucalC))}$. Fix such a ${V}$ and an ${X}$. To prove the last equation it suffices to show that each ${x\in V}$ has an open neighborhood ${U_x\subset V}$ so that ${\gerL_{X}(f)= 0}$ on ${U_x}$. Choose ${U_x}$ as in the definition of the Hamiltonian envelope, i.e.\  so that ${X|_{U_x}}$ is in the ${\calC^{\infty}(U_x)}$-span of ${\{X_g: g\in \Gamma(U_x, \ucalC)\}}$. Hence, to prove our equation over ${U_x}$, it suffices to show that it is satisfied by ${X= X_g}$ for any ${g\in \Gamma(U_x, \ucalC)}$ but, for those, ${\gerL_{X}(f)}$ (over ${U_x}$) gives the brackets ${\{f|_{U_x}, g\}}$ which are zero because ${f\in \Gamma(U, \ucalC')}$, ${U_x\subset U}$ open and ${g\in \Gamma(U_x, \ucalC)}$. 

For the reverse inclusion, we start with ${f}$ in the right hand side, over some open ${U}$ and we have to show that ${\{ f|_{V}, g\}= 0}$
for all ${V\subset U}$ open and ${g\in \Gamma(V, \ucalC)}$. Fix such a ${V}$ and ${g}$. We now use that ${f\in \Gamma(U, \Fun(\textrm{Ham}(\ucalC))}$, which implies that ${\gerL_{X}(f|_{V})= 0}$ for all ${X\in \Gamma(V, \textrm{Ham}(\ucalC))}$. Note that ${X\defeq X_g}$ is such a section, and then we obtain the desired equation.
\end{proof}
\noindent
We now claim that, for Theorem \ref{thm-the-tricky-one}, the most important point about regular foliations that is missing from the general discussion is the following:

\begin{local_theorem}
\label{thm -main -property -fols} 
For any regular foliation ${\calF}$ on a symplectic manifold ${(M, \calF)}$, locally, the vectors tangent to the orthogonal of ${\calF}$ are spanned by Hamiltonian vector fields ${X_f}$ induced by ${\calF}$-basic functions ${f}$. More precisely:
\begin{equation}\label{main -sheaves -fol} 
\gerX_{\calF^{\perp}}= \textrm{Ham}(\ucalC_{\calF}^{\infty}).
\end{equation} 
In contrast, for strongly involutive subsheaves ${\ucalV\subset \gerX_M}$, one always has
\begin{equation}\label{main -sheaves -fol2}
\ucalV^{\perp} \subset \textrm{Ham}(\ucalC),
\end{equation}
where ${\ucalC= \Fun(\ucalV)}$, but the inclusion may be strict. Actually, the equality is equivalent to the fact that 
${\textrm{Ham}(\ucalC)}$ is ${\omega}$-closed (as in Definition \ref{def -omega -closed}).  
\end{local_theorem} 

Of course, it is tempting the prove the statement about regular foliations as a consequence of the general part of the theorem: all we would have to show is be that
${\textrm{Ham}(\ucalC_{\calF}^{\infty})}$ is ${\omega}$-closed, e.g.\  proving that it comes from a vector bundle. But that looks difficult 
without proving first the entire theorem.

Note that the first part of the theorem also implies Theorem \ref{thm -orthog -fol}.

\begin{proof} 
First we show that any local Hamiltonian vector field ${X_f}$ with ${f}$ a basic function belongs to ${\gerX_\calF^{\perp_\omega}}$. Indeed, by definition a vector field ${X}$ belongs to ${\gerX_\calF^{\perp_\omega}}$ if and only if
    $
        \omega(X,Y)=0
    $
for all local ${Y\in \gerX_\calF}$, whereas a local smooth function ${f}$ lies in ${\ucalC_\calF^{\infty}}$ if and only if 
    \[
        \omega(X_f,Y)
        = 
        \gerL_Y(f) 
        = 
        0,
        \quad\forall\,
        Y\in\gerX_\calF.
    \]
The last part of the theorem just remarks that this part of the proof is completely general. The fact that equality hold only when ${\textrm{Ham}(\ucalC)}$ is closed follows from the first equality in Lemma \ref{basic -Ham} (if we take its orthogonal we find that the closure of ${\textrm{Ham}(\ucalC)}$ is precisely ${\Fun(\ucalV)}$. 

We are left with showing that these Hamiltonian vector fields locally generate ${\gerX_\calF^{\perp_\omega}}$ as a sheaf of ${C^\infty}$-modules. We show that any ${x\in M}$ admits a neighborhood ${U}$ and a finite number of functions
${f^1, \ldots, f^k\in C^{\infty}(U)}$ such that any ${X\in \gerX_{\calF}^{\perp_{\omega}}(U)}$ is a ${\calC^{\infty}(U)}$-linear combination of ${X_{f^1}, \ldots, X_{f^k}}$.  

Since ${\calF}$ is a \textit{regular foliation}, near any point ${x\in M}$ there is an open neighborhood ${U}$ and a submersion with connected fibers 
    \[
        \pi
        :
        U   \to \R^{2n -k}
    \]
that defines ${\calF\vert_U}$ as its fibers. In particular, any ${X\in\gerX_\calF}$ is ${\pi}$-vertical, i.e.\  ${d\pi(X)=0}$. Hence ${\pi_i\in \calC_\calF}$ for all ${i}$. By the previous, the Hamiltonian vector fields ${X_{\pi_i}}$ belong to ${\gerX_\calF^{\perp_\omega}}$. The ${X_{\pi_i}}$ are everywhere ${\R}$-linearly independent, since ${\pi}$ is a submersion, so they define a regular vector bundle ${V_\pi}$ over ${U}$ with fibers
    \[
        (V_\pi)_x
        \defeq
        \lspan_\R \big(
            (X_1)_x,
            \ldots,
            (X_{2n -k})_x
        \big).
    \]  
This is a subbundle of ${T\calF^{\perp_\omega}\vert_U}$ of constant rank equal to ${2n -k}$. Note that ${T\calF^{\perp_\omega}\vert_U}$ itself has rank ${2n -k}$, so the two vector bundles are identical. We are left with showing that any local section of ${V_\pi}$ is a ${C^\infty}$-linear combination of the ${X_i}$. For this, define a vector bundle map by
    \[
        U \times \R^k 
        \to
        V_\pi,
        \quad
        (x,r^1,\ldots,r^k)
        \mapsto
        r^i \, (X_i)_x.
    \]
This map is an isomorphism on each fiber. Hence, by the inverse function theorem, it has a smooth inverse. It is clear that any smooth map ${U\to \R^k}$ is a ${C^\infty}$-linear combination of the standard basis of ${\R^k}$, and this clearly implies the desired conclusion. 
\end{proof}
\noindent
Therefore, we are left with showing the relevance of the condition (\ref{main -sheaves -fol}) to the proof of Theorem  \ref{thm-the-tricky-one}. Again, we try to understand the general situation, but being aware now that the condition ${\ucalV^{\perp}= \textrm{Ham}(\ucalC)}$ may be imposed. Therefore, given Proposition \ref{prop -gen -Lagr -sheaves}, {\it all that is left to be understood is: if ${\ucalC'= \ucalC}$ and we have the extra condition ${\ucalV^{\perp}= \textrm{Ham}(\ucalC)}$, can we prove that ${\ucalV^{\perp}}$ is (strongly) involutive?}

Trying to prove this in full generality, one remarks the following:

\begin{lemma} 
Let ${\ucalV\subset \gerX_M}$ be a strongly involutive subsheaf, with associated sheaf of functions ${\ucalC= \Fun(\ucalV)}$. Assume that ${\ucalC= \ucalC'}$. Then the following are equivalent:
\begin{myenumerate}
\item ${\textrm{Ham}(\ucalC)}$ is Lagrangian (or just coisotropic),
\item ${\textrm{Ham}(\ucalC)}$ is strongly involutive. 
\end{myenumerate}
Moreover, these do imply both ${\ucalV= \ucalV^{\perp}}$, as well as ${\ucalV^{\perp}= \textrm{Ham}(\ucalC)}$. 
\end{lemma}

\begin{proof}
We claim that, without the condition ${\ucalC= \ucalC'}$, one has:
\[ \textrm{Ham}(\ucalC)\subset \ucalV^{\perp}\subset \Fun(\ucalC') = \textrm{Ham}(\ucalC')^{\perp}\]
and the ${\textrm{Inv}}$-closure of the first two terms coincides with the last two. 

The first inclusion follows from the remark made in the previous proof - that ${\ucalV^{\perp}}$ is the closure of ${\textrm{Ham}(\ucalC)}$. The second one is the basic inclusion (\ref{eq(1)}). The last part (the equality) 
is the first equality of Lemma \ref{basic -Ham}  applied to ${\ucalC'}$. Moreover, computing the ${\textrm{Inv}}$-closure of ${\textrm{Ham}(\ucalC)}$ we find, by applying ${\textrm{Inv}= \Vect\circ \Fun}$ and then the second equation from Lemma \ref{basic -Ham}, 
\[ \textrm{Inv}(\textrm{Ham}(\ucalC))= \Vect( \Fun(\textrm{Ham}(\ucalC)))= \Vect(\ucalC').\]
Since ${\ucalV^{\perp}}$ is in between, it must have the same ${\textrm{Inv}}$-closure. These prove the original claim. 

Assume now that ${\ucalC= \ucalC'}$. Using the claim, we see that ${\textrm{Hom}(\ucalC)}$ being strongly involutive (i.e.\  ${\textrm{Inv}}$-closed) is equivalent to ${\textrm{Hom}(\ucalC)= \ucalV^{\perp}= \ucalV}$. 
The same claim also implies that ${\textrm{Hom}(\ucalC)}$ is isotropic, hence the condition of being Lagrangian is equivalent to that of being coisotropic, as well as to ${\textrm{Hom}(\ucalC)= \ucalV^{\perp}= \ucalV}$. 
\end{proof}
\noindent
\subsection{End of proof of Theorem \ref{thm-the-tricky-one}}
Although the last lemma may sound promising, let us note that checking any of those conditions on ${\textrm{Ham}(\ucalC)}$ is usually hard: even for foliations when we can use ${\textrm{Ham}(\ucalC)= \gerX_{\calF^{\perp}}}$, they actually amount to proving the remaining step of Theorem \ref{thm-the-tricky-one}. Instead we remark that the idea behind Lemma \ref{promissing -lemma} has an even simpler version (and even an `if and only if' statement), provided we use the Hamiltonian envelope.  More precisely, since for sheaves of vector fields coming from regular foliations involutive is the same as strongly involutive, the missing step in the proof of Theorem \ref{thm-the-tricky-one} is provided by the following:

\begin{proposition} \label{last-step-vf-funct-thm}
A subsheaf ${\ucalC\subset \ucalC^{\infty}_{M}}$ is involutive if and only if 
${\textrm{Ham}(\ucalC)}$ is involutive.
For a strong involutive sheaf ${\ucalV\subset \gerX_M}$ with ${\ucalV^{\perp}= \textrm{Ham}(\ucalC)}$, the condition  
${\ucalC= \ucalC'}$ implies that ${\ucalV^{\perp}}$ is involutive, where ${\ucalC= \Fun(\ucalV)}$.
\end{proposition}
\noindent
\begin{proof} The key remark is that the map that associates Hamiltonian vector fields with functions is a morphism of sheaves of Lie algebras:
\[ \textrm{Ham}: \ucalC^{\infty}_{M} \rightarrow \gerX_M \]
which takes ${\ucalC}$ to ${\textrm{Ham}(\ucalC)\subset \ucalV(\ucalC)^{\perp}}$ and has the property that, for any local function ${f}$,
\[ f\in \ucalC \Longleftrightarrow \textrm{Ham}(f)\in \textrm{Ham}(\ucalC).\]
This clearly implies that ${\ucalC}$ is involutive if ${\textrm{Ham}(\ucalC)}$ is. Also the converse is immediate since, locally, the elements of ${\textrm{Ham}(\ucalC)}$ are ${\ucalC_{M}^{\infty}}$-linear combinations of ${X_f}$'s with ${f}$ local sections of ${\ucalC}$ and 
\[ \left[\sum_i \phi_i X_{f_i}, \sum_j \psi_j X_{g_j}\right] = \sum_{i, j} \left( \phi_i\psi_j X_{\{f_i, g_j\}}+ \phi_i \{f_i, \psi_j\} X_{g_j} - \psi_j \{g_j, \phi_i\} X_{f_i}\right)\]
are again such combinations. 
\end{proof}

\noindent
 \subsection{Conclusion}
Therefore, a large part of Theorem \ref{thm-the-tricky-one} holds for general sheaves of vector fields, and that is the content of Proposition \ref{prop -gen -Lagr -sheaves}. The only thing we were not able to prove in general (and is probably not true in full generality) is that the condition ${\ucalC= \ucalC'}$ implies that ${\ucalV^{\perp}}$ is strongly involutive (which would then imply the problematic inclusion, namely ${\ucalV\subset \ucalV^{\perp}}$). What we found out is what makes the case of regular foliations work, namely the following two properties:
\begin{myenumerate}
\item the fact that ${\ucalV^{\perp}= \textrm{Ham}(\ucalC)}$ is automatic for (strongly) involutive sheaves coming from regular foliations. 
\item for sheaves of sections of a sub bundle of ${TM}$ the involutivity is equivalent to strong involutivity.
\end{myenumerate}
These are the points that have to be addressed separately when dealing with singular foliations. 
As a curiosity (for now), let us point out that the Hamiltonian envelope is more appropriate (than ${\Vect(\ucalC}$) when trying to obtain foliations in the Androulidakis-Skandalis sense (see Remark \ref{Androulidakis and Skandalis}). More precisely, one has:
 
 \begin{proposition}
\label{AS-fol-curious}  
For any involutive subsheaf ${\ucalC\subset \ucalC^{\infty}_{M}}$ which is functionally finitely generated (cf. Definition \ref{functionally-fg}), ${\Ham(\ucalC)}$ defines a foliation in the sense of Androulidakis and Skandalis.
 \end{proposition}

\noindent
\begin{proof} With the discussion from Proposition \ref{last-step-vf-funct-thm} and Remark \ref{Androulidakis and Skandalis} in mind, we are left with remarking that, if ${\ucalC}$ is functionally finitely generated, then ${\Ham(\ucalC)}$ is finitely generated, but that follows immediately using that, for ${g= F\circ (f_1, \ldots, f_p)}$ (as in Definition \ref{functionally-fg}), one has
\[ 
    X_g = \sum\nolimits_k 
    \frac{\partial F}{\partial x_k}\circ (f_1, \ldots, f_p) \, X_{f_k}\in   \textrm{Span}_{\ucalC^{\infty}}(X_{f_1}, \ldots, X_{f_p}).
    \qedhere
\]
\end{proof}

\noindent
Note that our discussion also proves the intermediate statements described in Remark \ref{important -remark}. Moreover, as a generalization of the fact that a fibration ${\pi: (M, \omega)\rightarrow B}$ with isotropic fibers is a symplectically complete isotropic fibration if and only if ${B}$ admits a Poisson structure such that ${\pi}$ becomes a Poisson map, we obtain:

\begin{corollary} 
An isotropic foliation ${\calF}$ on a symplectic manifold is symplectically complete if and only if the sheaf ${\ucalC_{\calF}^{\infty}}$ of ${\calF}$-basic functions is closed under the Poisson bracket.
\end{corollary}

\begin{proof} 
({\it of Proposition \ref{reg -fol -locally -is}}) For completeness, we also include the proof of Proposition \ref{reg -fol -locally -is}. 
Since this is a local statement we may assume that we deal with a foliation given by a submersion ${\pi: U\to \mathbb{R}^k}$, where ${k}$ is the codimension of the foliation. Let also ${n}$ be the dimension of ${M}$. Note that the basic functions are precisely the ones that are pull backs via ${\pi}$ of functions on ${\mathbb{R}^k}$. 

We have to show that ${\calF}$ is Lagrangian if and only if ${n= 2k}$ and ${\{\pi, \pi_j\}= 0}$ for all ${i}$ and ${j}$. For the direction implication, since ${\ucalC^{\infty}_{\calF}}$ coincides with it commutant, it must be abelian - hence the ${\pi_i}$'s must commute. Since in a symplectic vector space Lagrangian subspaces have half the dimension of the space, we must have ${(n - k)= \frac{1}{2}n}$.

For the converse, a linear algebra argument (and the assumption ${n= 2k}$) shows that it suffices to show that ${\calF}$ is coisotropic or, equivalently, that ${\ucalC^{\infty}_{\calF}}$ is abelian. But since the ${\pi_i}$ functionally generate ${\ucalC^{\infty}_{\calF}}$ and they commute, the statement follows.
\end{proof}

\newpage
\section{The `singular foliations' of an integrable system}
\label{sub:sec:sing -Lag -fol}
\etocsettocstyle{\subsubsection*{Local Contents}}{}
\etocsettocdepth{2}
\localtableofcontents

\subsection{The sheaves ${\CM_{\mu}}$ and ${\XX_{\mu}}$ (the Lagrangian foliation)}
In this section we finally discuss the sheaves of functions and of vector fields associated to an integrable system. Please see again the two motivations mentioned at the beginning of this chapter. We start with the Lagrangian foliation (the central foliation appearing in Motivation 1), and here are the sheaves which encode it:

\begin{definition}
\label{def -central -sheaf} 
Given a completely integrable system ${\mu: (M, \omega)\rightarrow \mathbb{R}^n}$, one defines the following sheaves on ${M}$:
\begin{myenumerate}
\item 
${\CM_\mu}$, \textbf{the central sheaf of ${\mu}$}, is the sheaf of functions
\[
        \CM_\mu(U)
        \defeq
        \big\{
            f \in \calC^\infty(U)
            :
            \{ \mu_i|_{U}, f \}_\omega
            = 0
            \;\forall\;
            i
        \big\}.
    \]

\item 
${\XX_{\mu}}$, \textbf{the Lagrangian sheaf of ${\mu}$}, is the sheaf of vector fields
\[
\XX_{\mu}(U) \defeq
        \big\{
            X\in \gerX(U)
            :
            (d\mu)(X)= 0        
            \big\}.
\]
\end{myenumerate}
\end{definition}
\noindent
We will prove that, while an integrable system comes with various other sheaves of functions (built in various manners out of the ${\mu_i}$'s) and sheaves of vector fields (built in various manners out of the ${X_{\mu_i}}$'s), ${\CM_\mu}$ and ${\XX_{\mu}}$ are the best possible ones. More precisely, we will show:

\begin{proposition}
\label{main -prop -sheavs -int -syst} 
For any completely integrable system ${\mu: (M, \omega)\rightarrow \mathbb{R}^n}$, the central sheaf ${\CM_\mu}$ is closed and self-centralizing (hence also ${\textrm{Inv}_{\omega}}$-closed), ${\XX_{\mu}}$ is strongly involutive and Lagrangian (hence, in particular, also ${\omega}$-closed), and the two correspond to each other via the one-to-one correspondence from Corollary \ref{one-to-one -corr -sheaves}. 
In more detail:
\begin{myenumerate}

\item
${\CM_\mu}$ is the only self-centralizing sheaf of smooth functions on ${M}$ which contains all the ${\mu_i}$'s. 

\item 
${\CM_\mu}$ is the smallest closed sheaf of smooth functions on ${M}$ with the ${\mu_i}$'s. 

\item
${\XX_{\mu}}$ is the only Lagrangian sheaf of vector fields on ${M}$ with the ${X_{\mu_i}}$'s.

\item
${\XX_{\mu}}$ is the smallest strongly involutive sheaf of vector fields with the ${X_{\mu_i}}$'s. 

\item 
${\CM_{\mu}}$ and ${\XX_{\mu}}$ correspond via the bijection between closed sheaves of smooth functions and strongly involutive sheaves of vector fields:
\[ \XX_{\mu}= \Vect(\CM_\mu),\quad   \CM_\mu = \Fun(\XX_{\mu}).\]

\end{myenumerate}
\end{proposition}
\noindent
The proofs are postponed to the end of this section. For now, we prove the following consequence of the independence condition, which is key to many of the arguments to follow. 

\begin{NamedTheorem}[Main Lemma]\label{cis:lem:main}
If ${f,g \in C^\infty}$ are such that ${\set{\mu_i,f}_\omega=\set{\mu_i,g}_\omega = 0}$ for all ${i}$, then ${\set{f,g}_\omega=0}$. 
In other words, ${\ucalC_{\mu}}$ is abelian. 
\end{NamedTheorem}

\begin{proof}
The Hamiltonian vector fields ${X_i\defeq X_{\mu_i}^\omega}$ span the kernel of $d_x\mu$ at each regular point ${x}$. The equations
    \[
        \gerL_{X_i}(f) = \{ \mu_i, f \} = 0
        \quad\forall\, i 
    \]
thus tell us that ${f}$ is constant on the connected components of the regular fibers of ${\mu}$. Since ${\mu:M\to \R^n}$ is locally a submersion near ${x}$, there is a neighborhood ${U\sub M}$ of ${x}$, a neighborhood ${V\sub \R^n}$ of ${\mu(x)}$ and a smooth map ${F:V\to \R^n}$ such that 
    \[
        f\vert_U  
        = 
        F\circ \mu\vert_U.
    \]
From this it follows that ${f}$ and ${g}$ Poisson commute in a neighborhood of ${x}$, i.e.\ 
    \[
        \{ f, g \}_\omega
        =
        \{ F\circ \mu, g \}_\omega
        =
        \sum\nolimits_i
        \frac{ \partial F }{ \partial \mu_i }(\mu)
        \,
        \{ \mu_i , g \}_\omega
        =
        0
    \]
on ${U}$. Since the regular points lie dense in ${M}$, it follows that ${\{f,g\}_\omega=0}$ holds everywhere.
\end{proof}
\noindent
\subsection{The orbit sheaf ${\gerX^\textsc{orb}_{\mu}}$}
Next to ${\XX_{\mu}}$, another interesting sheaf of vector fields is the following:

\begin{definition}
Define the {\bf orbit sheaf ${\gerX^\textsc{orb}_\mu}$} of ${\mu}$ as the sheaf of vector fields on ${M}$ given by
\[
        \gerX^\textsc{orb}_\mu (U)
        \defeq 
        \big\{
            \sum\nolimits_i f^i\, X_{\mu_i}|_U \in \gerX(U)
            :
            f^i \in \calC^\infty(U)
        \big\}.
\]  
\end{definition}
\noindent
The fact that this is indeed a sheaf is based on the remark that the Hamiltonian vector fields ${X_{\mu_i}}$ are linearly independent over ${\ucalC^{\infty}}$, over any open (see below). 
A slightly different way to obtain ${\gerX^\textsc{orb}_\mu}$ is by starting with the ${\calC^{\infty}(M)}$-span of the ${X_{\mu_i}}$'s, remarking that it is a local module, and then considering the associated sheaf (cf. Proposition \ref{prop -basic -sheaves}).
One may say that this is the sheaf associated to the action of ${\mathbb{R}^n}$ on ${M}$: such a sheaf can be defined whenever we have a Lie algebra action or, more generally, a Lie algebroid ${A}$ over ${M}$ (the local module being the image of the anchor map
${\rho: \Gamma(A)\rightarrow \gerX(M)}$). It is well-known that this defines a singular foliation whose leaves are the orbits of the algebroid. For our sheaf ${\gerX^\textsc{orb}_\mu}$, we have:

\begin{proposition}
\label{prop -X -orb -pps} 
For any completely integrable system ${\mu: (M, \omega)\rightarrow \mathbb{R}^n}$, ${\gerX^\textsc{orb}_\mu}$ is a subsheaf of modules of ${\XX_{\mu}}$, it is finitely generated and involutive,  
and it is isotropic, with:
\[ (\gerX^\textsc{orb}_\mu)^{\perp} = \XX_{\mu} .\]
In particular, ${\gerX^\textsc{orb}_{\mu}}$ defines a foliation on ${M}$ in the Androulidakis-Skandalis sense, and the leaves are the orbits of the action of ${\mathbb{R}^n}$ (and they are included in the fibers of ${\mu}$). 

Moreover, 
$
        \gerX_\mu^\textsc{orb}\vert_{M_\textsc{reg}}
        =
        \gerX_\mu\vert_{M_\textsc{reg}}.
    $
\end{proposition}

\noindent
\begin{proof} 
To see that ${\gerX^\textsc{orb}_\mu}$ is indeed a sheaf, the question is whether a (local) vector field that is locally of this type, is of this type on the entire domain of definition. It suffices to remark that
\[  
    \sum\nolimits_i 
    f^i\, X_{\mu_i}|_U= 0 \Longrightarrow f^i= 0.
\]
To check this, one applies ${\omega(\cdot, V)}$ to the right hand side (for an arbitrary vector field ${V}$) to deduce that 
${\sum_i f_i\cdot d\mu_i= 0}$. If some ${f_i}$, say ${f_1}$, is not identically zero, we find on open ${U_0}$ on which ${f_1\neq 0}$. On that open we can them write ${d\mu_1= \sum_{i\geq 2} g_j \cdot d\mu_j}$ for some smooth functions ${g_i}$, and then the wedge product of all the ${(d\mu)_i}$'s would be zero on ${U_0}$. That would contradict the fact that ${\mu}$ is a submersion almost everywhere.

It should be clear now that this is a subsheaf of modules, finitely generated and involutive. Actually, it is of type ${\textrm{Ham}(\ucalC)}$ for some subsheaf ${\ucalC\subset \ucalC^{\infty}_{M}}$- namely the ${\mathbb{R}}$-linear span of the ${\mu_i}$'s that will be discussed in more detail below and denoted ${\ucalC_\mu^\textsc{lin}}$, and we see that we are in the situation of Proposition \ref{AS-fol-curious}. 

While ${\gerX^\textsc{orb}_\mu\subset \XX_{\mu}}$ is clear, to show that the symplectic orthogonal of the first equal the second  one can proceed directly or, again, one can make use of the basic properties of the Hamiltonian envelope (see Lemma \ref{basic -Ham}) to write the left hand side as ${\ucalC_\mu^\textsc{lin}}$. The fact that this equals to ${\XX_{\mu}}$ will be proven in our discussion below about sheaves of functions. 

We now prove the final statement. By definition, ${\mu}$ is a submersion on ${M_\textsc{reg}}$. This implies that ${\ker(\dd\mu)}$ defines a vector bundle of constant rank ${n}$ on ${M_\textsc{reg}}$, and ${\gerX_\mu\vert_{M_\textsc{reg}}}$ is the sheaf of sections of this vector bundle. On the other hand, the Hamiltonian vector fields ${X_{\mu_i}}$ are ${\R}$-linear independent at regular points, so they span a vector bundle ${T^\textsc{orb}}$  of constant rank ${n}$ on ${M_\textsc{reg}}$, and ${\gerX_\mu^\textsc{orb}\vert_{M_\textsc{reg}}}$ is the sheaf of sections of this vector bundle. It is clear that ${T^\textsc{orb}\sub \ker(\dd\mu)}$. They are equal on ${M_\textsc{reg}}$ because they have the same rank. This proves the statement.
\end{proof}

\noindent
\begin{corollary} \label{cor -X -mu -Lagr} One has 
    $
        \gerX_\mu
        =
        \big(
            \gerX_\mu
        \big)^{\perp_\omega}
    $
and ${\gerX_\mu= \Vect(\ucalC_{\mu})}$. 
\end{corollary}

\begin{proof} The inclusion $\big(
            \gerX_\mu
        \big)^{\perp_\omega}\sub \gerX_\mu${ is immediate: if }$V${ is in the left hand side, then }$\omega(V, X)= 0${ for all }$X\in \textrm{Ker}(d\mu)${ hence, in particular, for }$X= X_{\mu_i}${ that means that }$(d\mu_i)(V)= L_V(\mu_i)= \omega(V, X_{\mu_i})= 0${ for all }$i${, i.e.\  }$V\in \gerX_{\mu}$.
Hence we are left with showing that ${\omega}$ vanishes when restricted to ${\gerX_\mu}$. Suppose that ${X,Y\in\gerX_\mu}$. Then ${X\vert_{M_\textsc{reg}}}$ and ${Y\vert_{M_\textsc{reg}}}$ lie in the orbit foliation, so they are of the form
    \[
        X\vert_{M_\textsc{reg}}
        =
        f^i \, X_{\mu_i},
        \quad
        Y\vert_{M_\textsc{reg}}
        =
        g^j \, X_{\mu_j},
        \quad
        f^i,g^j \in C^\infty\vert_{M_\textsc{reg}}.
    \]
We plug this into ${\omega}$ and make the obvious computation:
    \[
        \omega( X, Y )\vert_{M_\textsc{reg}}
        =
        \omega( f^i X_{\mu_i}, g^j X_{\mu_j} )
        =
        f^i g^j \{ \mu_i, \mu_j \}_\omega
        =
        0.
    \]
This shows that ${\omega(X,Y)=0}$ everywhere, since the set of regular points lies dense in ${M}$.

For the last part we have to show that, locally, for a vector field ${X}$, ${X\in\gerX_\mu}$ if and only if ${\gerL_X(f)=0}$ for all ${f\in \CM_\mu}$.
Since ${\{\mu_i,\mu_j\}_\omega=0}$ for all ${i,j}$, the converse implication is clear. For the direct one, one writes again  ${X=f^iX_{\mu_i}}$ on the regular part. Then for all ${f\in \CM_\mu}$, i.e.\ , such that ${\{\mu_i,f\}_\omega=0}$ for all ${i}$,
    \[
        \gerL_Xf 
        = 
        f^i \gerL_{X_{\mu_i}}f 
        = 
        f^i\{ f, \mu_i \}
        = 0
    \]
on ${M_\textsc{reg}}$. Using again that ${M_\textsc{reg}}$ is an open and dense subset of ${M}$, one obtains ${\gerL_X(f)=0}$ everywhere. 
\end{proof}
\noindent
\begin{example}
We give an example of a completely integrable system for which ${\XX_{\mu}}$ is not finitely generated. A notable feature of this example is that it relies on flat functions, so this particular issue takes place outside the category of analytic manifolds.

Let ${\phi:\R\to\R}$ be ${\phi(x)=e^{ -1/x^2}}$ and define a completely integrable system ${\mu:\R^2\to \R}$ by 
    \[
        \mu(x,y)
        =
        \begin{cases}
            \phi(x)y
            &
            x\leq 0,
            \\
            \phi(x)
            &
            x\geq 0.
        \end{cases}
    \]
Let ${X=a(x,y)\,\partial x + b(x,y)\,\partial y}$ be a vector field on ${\R^2}$ such that ${\dd\mu(X)=0}$, i.e.\  ${X\in\gerX_\mu^\textsc{lag}}$.

For negative ${x}$, this means that 
    $
        \phi'ya+\phi b =0.
    $
Since ${\phi'(x) = 2x^{ -3} \phi(x)}$ and since ${\phi(x)}$ is nonvanishing for (strictly) negative ${x}$, this condition is equivalent to 
    $
        2ya+x^3b=0.
    $
Let     
    \[
        Y
        \defeq
        x^3\,\partial x + 2y\,\partial y.
    \]
It follows that ${X=f\, Y}$ for some smooth function ${f}$ and for negative ${x}$. 

For positive ${x}$, ${\dd\mu(X)=0}$ means that ${a(x,y)=0}$. Hence the function ${f}$ must extend by zero to the whole of ${\R^2}$. So far this tells us that
    \[
        X = f\, Y + b \,\partial y.
    \]
where ${f}$ vanishes for ${x}$ positive. But for ${x}$ negative we already concluded that ${X=f\, Y}$, so ${b}$ must vanish there. It follows that ${\gerX_\mu^\textsc{lag}}$ is isomorphic to the submodule of ${C^\infty(\R^2)}$ of functions that are flat along the ${y}$-axis via the map ${\gerX_\mu^\textsc{lag}\to C^\infty : (f\, Y+ b\,\partial y) \mapsto f+b}$. This ${C^\infty(\R^2)}$-module is known to not be finitely generated near the origin.
\end{example}
\noindent
\begin{example}
Even in very simple examples, the orbit and Lagrangian sheaves can be different at singular points, even in cases where both are locally finitely generated. E.g. consider ${\mu(x,y)=x^2}$ as a function ${\R^2\to \R}$. Then the orbit foliation is generated by ${x\,\partial y}$ as a ${C^\infty(\R^2)}$-module, and the Lagrangian foliation is generated by ${\partial y}$. 
\end{example}
\noindent
They are however equal for the normal forms: the elliptic model ${\mu_\text{ell}}$, the hyperbolic model ${\mu_\text{hyp}}$ and focus-focus model pair ${\mu_\text{ff}}$. The computations are worked out in the examples below. For a fixed point of ${\mu}$, i.e.\  a point ${x}$ where ${d_x\mu=0}$, the notion of nondegeneracy basically means that the quadratic terms of ${\mu}$ are up to conjugation only of these three types. These examples suggest that it should be possible to use perhaps the Malgrange preparation theorem to prove that the orbit, and Lagrangian sheaves agree for completely integrable systems with only such nondegenerate singularities.

\begin{example}
Recall the \textbf{elliptic}\index{normal model!elliptic} normal model on ${\R^2}$ with ${\omega_\text{can}}$:
    \[
        \mu_\text{ell}(x,y)
        \defeq
        \tfrac12 \big(
            x^2+y^2
        \big),
        \quad
        X_\text{ell}
        \defeq
        X_{\mu_\text{ell}}
        =
        y \tfrac{\partial}{\partial x}
        -
        x \tfrac{\partial}{\partial y}.
    \]
By definition, the orbit sheaf is generated by ${X_\text{ell}}$. On the other hand, a vector field 
    $
        X
        =
        a\,\partial x 
        + 
        b\,\partial y
    $ 
lies in ${\gerX_{\mu_\text{ell}}}$ if and only if 
    $
        x\, a + y\, b=0.
    $ 
It follows that ${a}$ is divisible by ${y}$, let's say ${a=y\, f}$, and that ${b= -x\, a/y = -x\, f}$. Hence ${X= f\, X_\text{ell}}$, and thus ${\gerX_{\mu_\text{ell}}^\textsc{lag} = \gerX_{\mu_\text{ell}}^\textsc{orb}}$.
\end{example}
\noindent
\begin{example}
Recall the \textbf{hyperbolic}\index{normal model!hyperbolic} normal form on ${\R^2}$ with ${\omega_\text{can}}$:
    \begin{align*}
        \mu_\text{hyp}(x,y)
        \defeq
        x y,
        \quad
        X_\text{hyp}
        \defeq
        X_{\mu_\text{hyp}}
        =
        x \tfrac{\partial}{\partial x}
        -
        y \tfrac{\partial}{\partial y}.
    \end{align*}
Of course the orbit sheaf is generated by ${X_\text{hyp}}$. On the other hand, a vector field 
    $
        X = a\,\partial x + b\, \partial y
    $ 
lies in ${\gerX_{\mu_\text{hyp}}}$ if and only if 
    $
        y\, a + x\, b = 0.
    $ 
Hence in this case ${a}$ is divisible by ${x}$, let's say ${a=x\, f}$, and now ${ b= -y\, a/x = -y\, f}$. Hence ${ X= f\, X_\text{hyp}}$, and thus ${\gerX_{\mu_\text{hyp}}^\textsc{lag} = \gerX_{\mu_\text{hyp}}^\textsc{orb}}$.
\end{example}
\noindent
\begin{example}
Recall the \textbf{focus-focus}\index{normal model!focus-focus} normal form on ${\R^4}$ with ${\omega_\text{can}}$:
    \begin{align*}
        \mu_{\text{ff},1}(x,y)
        &\defeq
        x_1y_1+x_2y_2,
        \\ &
        \quad
        X_{\text{ff},1}
        \defeq
        X_{\mu_{\text{ff},1}}
        =
        -
        x_1 \tfrac{\partial}{\partial x_1}
        +
        y_1 \tfrac{\partial}{\partial y_1}
        -
        x_2 \tfrac{\partial}{\partial x_2}
        +
        y_2 \tfrac{\partial}{\partial y_2},
        \\
        \mu_{\text{ff},2}(x,y)
        &\defeq
        x_1y_2 -x_2y_1,
        \\ &
        \quad
        X_{\text{ff},2}
        \defeq
        X_{\mu_{\text{ff},2}}
        =
        +
        x_2 \tfrac{\partial}{\partial x_1}
        + 
        y_2 \tfrac{\partial}{\partial y_1}
        -
        x_1 \tfrac{\partial}{\partial x_2}
        -
        y_1 \tfrac{\partial}{\partial y_2}.
    \end{align*}
Clearly the orbit sheaf is generated by ${X_{\text{ff},1}}$ and ${X_{\text{ff},2}}$. On the other hand, a vector field 
    $
        X
        =
        \sum\nolimits_{i=1,2}
        \big(
            a_i\, \partial x_i
            +
            b_i\, \partial y_i
        \big)
    $
lies in ${\gerX_{\mu_\text{ff}}^\textsc{lag}}$ if and only if 
    \[
        y_1\, a_1
        +
        x_1\, b_1
        +
        y_2\, a_2
        +
        x_2\, b_2
        =
        0,
        \quad
        y_2\, a_1
        +
        x_1\, b_2
        -
        y_1\, a_2
        -
        x_2\, b_1
        =
        0.
    \]
A simple computation shows that this implies
    \[
        \big(
            y_1^2+y_2^2
        \big)\, a_1
        +
        x_1 
        \big(
            y_1\, b_1 + y_2\, b_2 
        \big)
        +
        x_2 
        \big( 
            y_1\, b_2 - y_2\, b_1 
        \big)
        =
        0.
    \]
Hence ${a_1= -x_1\, f + x_2\, g}$ with
    \[
        f
        \defeq
        \frac{
            y_1\, b_1 + y_2\, b_2
        } {
            y_1^2 + y_2^2
        },
        \quad
        g
        \defeq
        \frac{
            y_2\, b_1 - y_1\, b_2
        } {
            y_1^2 + y_2^2
        }.
    \]
Likewise we find 
    $
        a_2
        =
        - x_2\, f
        - x_1\, g,
    $
and we find 
    $
        b_1
        =
        y_1\, \tilde f + y_2\, \tilde g
    $ 
and 
    $
        b_2
        =
        y_2\, \tilde f 
        -
        y_1\, \tilde g$ 
with
    \[
        \tilde f  
        \defeq
        -
        \frac{x_1a_1+x_2a_2}
        {x_1^2+x_2^2},
        \quad
        \tilde g
        \defeq
        \frac{
            x_2\, a_1 - x_1\, a_2
        } {
            x_1^2 + x_2^2
        }.
    \]
Then we check by hand that ${f=\tilde f}$ and ${g=\tilde g}$, so that ${ X = f\, X_{\text{ff},1} + g\, X_{\text{ff},2}}$.
\end{example}

\noindent
\subsection{More sheaves of vector fields: ${\gerX^\textsc{lin}_\mu}$ and ${\gerX^{\text{ab}\textsc{orb}}_\mu}$}
Recall that the ${X_{\mu_i}}$'s and their restrictions to opens are linearly independent over ${\ucalC^{\infty}_{M}}$, because $M_\textsc{reg}$ is open and dense. So we see that any sheaf of functions ${\ucalC}$ on ${M}$ gives rise to a sheaf of vector fields on ${M}$, the ${\ucalC}$-span of the ${X_{\mu_i}}$, given by:
\begin{equation}\label{new-sheaf-gen-vfds} 
U\mapsto \big\{\sum\nolimits_i f^i\, X_{\mu_i}|_U \in \gerX(U), f^i \in \ucalC(U)\big\}.
\end{equation}
One should be aware however that this is not a sheaf of modules hence, unlike ${\gerX^\textsc{orb}_\mu}$, it can not be defined via the space of global sections (it may even fail to have enough global sections). Various sheaves ${\ucalC}$ will be mentioned below. For now we point out two interesting choices:

\begin{definition}
We define:
\begin{myenumerate}
\item The \textbf{linear sheaf} 
${\gerX^\textsc{lin}_\mu}$ as (\ref{new-sheaf-gen-vfds}) with ${\ucalC}$ being the sheaf of locally constant functions.
\item The \textbf{abelian orbit sheaf}
${\gerX^{\text{ab}\textsc{orb}}_\mu}$ as (\ref{new-sheaf-gen-vfds}) obtained with ${\ucalC}$ being the sheaf  ${\CM_{\mu}}$.
\end{myenumerate}
\end{definition}

\noindent
\begin{proposition} 
\label{Lagr -sheaves -mai -prop} 
One has
\begin{itemize}

\item
$
\gerX_\mu^\textsc{lin} \sub \gerX_\mu^{\text{ab}\textsc{orb}} \sub \gerX_\mu^\textsc{orb} \sub \XX_\mu.
$

\item
${\gerX_\mu^\textsc{lin}}$ and ${\gerX_\mu^{\text{ab}\textsc{orb}}}$ are abelian.

\item
${\gerX_\mu^{\text{ab}\textsc{orb}}}$ is self-centralizing in ${\gerX_\mu^\textsc{orb}}$.

\item    
    $
        \big(
            \gerX_\mu^\textsc{lin}
        \big)^{\perp_\omega}
        =
        \big(
            \gerX_\mu^{\text{ab}\textsc{orb}}
        \big)^{\perp_\omega}
        =
        \big(
            \gerX_\mu^\textsc{orb}
        \big)^{\perp_\omega}
       =
        \XX_\mu.
    $
\end{itemize}
(In particular, since ${\XX_{\mu}}$ is Lagrangian, it is the ${\omega}$-closure of both ${\gerX_\mu^\textsc{lin}}$ and ${\gerX_\mu^{\text{ab}\textsc{orb}}}$ and, as we shall see, it is also the strong involutive closure of the two). Moreover, the sheaves of functions induced by all these sheaves (cf. Definition \ref{def -Fun}) coincides with ${\CM_{\mu}}$: 
\[ \Fun(\gerX^\textsc{lin}_\mu)= \Fun(\gerX^{\text{ab}\textsc{orb}}_\mu)= \Fun(\gerX^\textsc{orb}_\mu)= \Fun(\XX_{\mu})= \CM_{\mu}.\]
\end{proposition}

\noindent
\begin{proof} 
In the sequence of inclusions, the only one that, perhaps, needs a (small) argument is the last one: 
${\gerX_\mu^\textsc{orb} \sub \XX_\mu}$: we check for any ${X=f^i \, X_{\mu_i}}$ that
    \[
        \dd\mu_j(X)
        =
        f^i \, 
        \dd \mu_j
        \big(
            X_{\mu_i}
        \big)
        =
        f^i \,
        \{ \mu_i, \mu_j \}
        = 0,
        \quad\forall\, j.
    \]

We will show that ${\gerX_\mu^\textsc{lin}}$, ${\gerX_\mu^{\text{ab}\textsc{orb}}}$ and ${\gerX_\mu^\textsc{orb}}$ are closed under the Lie bracket. Consider local vector fields ${X=f^i\, X_{\mu_i}}$ and ${Y=g^i\, X_{\mu_i}}$. Note that 
    \[
        [X_{\mu_i},X_{\mu_j}]
        =
        X_{ \{ \mu_i, \mu_j \}_\omega }
        =
        0,
    \]
hence the Lie bracket of ${X}$ and ${Y}$ is given by
    \begin{align*}
        [X,Y]
        &=
        [f^i X_{\mu_i}, g^j X_{\mu_j} ]
        \\
        &=
        f^i \gerL_{X_{\mu_i}}(g^j) X_j
        -
        g^j \gerL_{X_{\mu_j}}(f^i) X_i
        +
        f^ig^j [X_{\mu_i},X_{\mu_j}]
        \\
        &=
        \big(
            f^i \{ g^j, \mu_i \}_\omega
            -
            g^i \{ f^j, \mu_i \}_\omega
        \big)
        X_{\mu_j}.
    \end{align*}
This clearly shows that all three sheaves are closed under the Lie bracket. Moreover, if both the functions ${f^i}$ and ${g^i}$ are constants of motion, as is the case for ${\gerX_\mu^\textsc{lin}}$ and ${\gerX_\mu^{\text{ab}\textsc{orb}}}$, then the terms
    $
        f^i \{ g^j, \mu_i \}_\omega
        -
        g^i \{ f^j, \mu_i \}_\omega
    $
vanish, hence the linear and the abelian orbit sheaves are abelian.

For the statement that ${\gerX_\mu^{\text{ab}\textsc{orb}}}$ is self-centralizing, suppose that \[{X= \sum f_i X_{\mu_i}\in \gerX_{\mu}^\textsc{orb}}\] such that ${[X,Y]=0}$ for any ${Y\in \gerX_\mu^{\text{ab}\textsc{orb}}}$. This means that
    $
        g^i \{ f^j, \mu_i \}_\omega X_{\mu_j} =0
    $
for any choice of ${g^1,\ldots, g^n \in \CM_\mu}$, so in particular ${\{ f^j, \mu_i \}_\omega X_{\mu_j}=0}$ for all ${i}$. At any regular point ${x}$ of ${\mu}$ this means that ${\{ f^j, \mu_i \}_\omega(x)=0}$ for all ${i}$ and ${j}$. Since the regular points of ${\mu}$ are dense in ${M}$, we conclude that ${Y\in\gerX_\mu^{\text{ab}\textsc{orb}}}$.

From the first statement of the proposition we obtain:
    \[
        \big(
            \gerX_\mu
        \big)^{\perp_\omega}
        \sub
        \big(
            \gerX_\mu^\textsc{orb}
        \big)^{\perp_\omega}
        \sub
        \big(
            \gerX_\mu^{\text{ab}\textsc{orb}}
        \big)^{\perp_\omega}
        \sub
        \big(
            \gerX_\mu^\textsc{lin}
        \big)^{\perp_\omega}.
    \]
 Hence, to obtain the first sequence of equalities in the statement, using the fact that ${\gerX_{\mu}}$ is Lagrangian (Corollary \ref{cor -X -mu -Lagr}), it suffices to remark that 
$ 
        \big(
            \gerX_\mu^\textsc{lin}
        \big)^{\perp_\omega}
        \sub
        \gerX_\mu
    ${. Indeed, if     }$
        X\in 
        \big(
            \gerX_\mu^\textsc{lin}
        \big)^{\perp_\omega},
    $
then, in particular 
    \[
        \dd\mu_i(X)=\omega(X,X_{\mu_i})=0,
        \quad\forall\, i.
    \]
    One proceeds similarly for the last part: applying ${\Fun}$ to the original sequence of inclusions we obtain
\[
        \Fun(\gerX_\mu)
        \sub
        \Fun(\gerX_\mu^\textsc{orb})
        \sub
        \Fun(\gerX_\mu^{\text{ab}\textsc{orb}})
        \sub
        \Fun(\gerX_\mu^\textsc{lin});
    \]
it is also clear that ${\Fun(\gerX_\mu^\textsc{lin})\subset \ucalC_{\mu}}$ (even equality), hence it suffices to remark that
${\ucalC_{\mu}\subset  \Fun(\gerX_\mu)}$, but using the second part of Corollary \ref{cor -X -mu -Lagr}, this is the fact that
${\ucalC_{\mu}}$ is contained in its closure (see Definition \ref{def -cl -sheaves -ftcs}). 
\end{proof}
\noindent
\subsection{Sheaves of smooth functions: ${\ucalC_\mu^\textsc{lin}, \ucalC_\mu^\textsc{aff}, \ldots,  \ucalC_\mu^\textsc{fib}}$}
We now pass to sheaves of smooth functions. We start with the question: which are the natural sheaves associated to ${\mu}$? First of all, there are the ones that are `generated' by the ${\mu_i}$'s, where `generation' can be ${\mathbb{R}}$-linear, ringwise, or even ${\calC^{\infty}}$-ringwise. These can be defined all at once, using the pullback construction described in Remark \ref{rk -functions -better}.

\begin{definition} 
We defined the linear, the affine, the polynomial, and the functional span of the ${\mu_i}$'s as the sheaves of functions on ${M}$, denoted 
\[ \ucalC_\mu^\textsc{lin}, \ucalC_\mu^\textsc{aff}, \ucalC_\mu^\textsc{pol}, \ucalC_\mu^{\infty},\]
and defined as the pullbacks via ${\mu}$ (as sheaves of functions - see Remark \ref{rk -functions -better}) of the sheaves  
\[ \ucalC_{\mathbb{R}^n}^{\textsc{lin}},  \ucalC_{\mathbb{R}^n}^{\textsc{aff}},  \ucalC_{\mathbb{R}^n}^{\textsc{pol}},  \ucalC_{\mathbb{R}^n}^{\infty}\]
of linear, affine, polynomial and smooth functions on ${\mathbb{R}^n}$, respectively. 
\end{definition}
\noindent
According to the definition of pullbacks, with $\mathrm{Pol}(\mu)$ the set of polynomials in the $\mu_i$,  we have
\begin{align*} 
\ucalC_\mu^\textsc{pol}(U) &\defeq
            \big\{
                f\in \calC^{\infty}(U): \,\forall\, x\in U, 
                \,\exists\, P \in \mathrm{Pol}(\mu), 
                \\ & \qquad \qquad\qquad
                \,\,\text{s.t.}\,\, f = P(\mu_1, \ldots, \mu_n) 
                \,\,\text{near $x$}
            \big\}
\end{align*}
and similarly for \textbf{${\ucalC_\mu^{\infty}}$} using smooth functions ${P}$. On the other hand, since the restrictions of the ${\mu_i}$'s to any open are linearly independent over ${\mathbb{R}}$, we see that 
\[ \ucalC_\mu^\textsc{lin}(U) \defeq
            \big\{
                   f\in \calC^{\infty}(U): f= \sum\nolimits_i c^i\, \mu_i|_{U}, c^i - \text{ locally constant on}\, U            
            \big\}
\]
(and similarly for the affine sheaf). 

It is instructive to look back at our `Motivation 1'. There, one of the extreme possibilities for defining the notion of equivalence is to require that the symplectomorphism ${\Phi: (M, \omega)\rightarrow (M', \omega'))}$ takes the fibers of ${\mu}$ to the fibers of ${\mu'}$. That means that the base map ${\phi}$ is just a set theoretical bijection. Of course, such bijections do not preserve the sheaf of smooth functions, but the one of all set theoretical functions. Although this sheaf does not quite fit our previous discussions, one can adapt the definition of its pullback. We get yet another sheaf on ${M}$, denoted  
${\ucalC_\mu^\textsc{fib}}$ which is relevant to equivalences that take fibers to fibers. Explicitly:

\begin{definition}
\label{def-c-fib-sheaf} 
Define the sheaf ${\ucalC_\mu^\textsc{fib}}$ by:
\begin{align*} 
    \ucalC_\mu^\textsc{fib}(U) 
    & \defeq
    \big\{
    f\in \calC^{\infty}(U): 
    \textrm{on small enough}\, V\subset U \,\text{open}\, , 
    \\ & \qquad\qquad\qquad
    f|_{V} \,\textrm{is constant on the fibers of}\ \mu|_{V} 
    \big\}.
\end{align*}
(hence each ${x}$ in ${U}$ admits a neighborhood ${V\subset U}$ as above). 
\end{definition}
\noindent
Finally, we also consider the sheaf of function that is closely related to the orbit sheaf ${\gerX_\mu^\textsc{orb}}$; this sheaf arises, and will be relevant, when discussing the orbit equivalence of integrable systems. 

 \begin{definition}\label{def -C -orbit--sheaf} 
Define the sheaf ${\ucalC_\mu^\textsc{orb}}$ by:
\[ \ucalC_\mu^\textsc{orb}(U) \defeq
\big\{
                   f\in \calC^{\infty}(U): X_{f}\in \gerX_\mu^\textsc{orb}(U)\} \big\}.
\]
\end{definition}
\noindent
Hence, a section over ${U}$ is a smooth function ${f}$ on ${U}$ with the property that
\begin{equation}
\label{rk -orbit-sheaf-C} 
X_f= \sum\nolimits_j C^j X_{\mu_j} \quad \textrm{for\ some}\ C^j\in \calC^{\infty}(M).
\end{equation}
This is a sheaf because of the ${\ucalC^{\infty}}$-linear independence of the ${X_{\mu_i}}$'s.

In some sense, `${\ucalC_\mu^\textsc{orb}= \textrm{Ham}^{ -1}(\gerX_\mu^\textsc{orb})}$', and an operation ${\textrm{Ham}^{ -1}}$ can probably be discussed in full generality. 
An interesting fact, which will be useful later on, is that this construction applied to ${\gerX_\mu^\text{ab}\textsc{orb}}$ would give the same result. More precisely:

\begin{lemma}\label{lemma -rk -orbit-sheaf-C} 
If a smooth function ${f\in \ucalC^{\infty}(M)}$ satisfies (\ref{rk -orbit-sheaf-C}), then ${f, C^j\in \ucalC_{\mu}(U)}$.
\end{lemma}

\begin{proof}
Applying the equation to the ${\mu_k}$'s to obtain ${f\in \ucalC_{\mu}(U)}$. Applying the bracket operation 
${[ -, X_{\mu_k}]}$ to the same (\ref{rk -orbit-sheaf-C}) we obtain ${0= \sum_{j} \{\mu_k, C_{j}\} \cdot X_{\mu_{j}})}$; using the again the linear independence of the vector fields ${X_{\mu_j}}$, 
we deduce that ${C^j\in \ucalC_{\mu}(U)}$.
\end{proof}
\noindent
Here are the main properties of these sheaves:

\begin{proposition}\label{pro -all -many -props -sheaves} One has:
\begin{equation}\label{longs -seq -incl -sh -fcts}  
        \ucalC_{\mu}^\textsc{lin}
        \sub
        \ucalC_{\mu}^\textsc{pol}
        \sub
        \ucalC_{\mu}^{\infty}
        \sub
        \ucalC_{\mu}^\textsc{orb}
        \sub
        \CM_\mu, \quad  \ucalC_{\mu}^\textsc{pol}
        \sub
        \ucalC_{\mu}^{\infty}
        \sub
        \ucalC_{\mu}^\textsc{fib}
        \sub
        \CM_\mu, 
        \end{equation}
and all are abelian sheaves (w.r.t.\  ${\{\cdot,\cdot\}_\omega}$). Moreover, their commutators coincide with ${\CM_{\mu}}$:
\[  \big( \ucalC_\mu^\textsc{lin} \big)'
        =
        \big( \ucalC_\mu^\textsc{pol} \big)'
        =
        \big( \ucalC_\mu^{\infty} \big)'
        =
        \big( \ucalC_\mu^\textsc{orb} \big)'
        =
        \big( \ucalC_\mu^\textsc{fib} \big)'
        =
        \big( \ucalC_\mu\big)'
        =
        \CM_\mu.\]
In particular, since ${\CM_{\mu}}$ is self-centralizing, ${\CM_{\mu}}$ is the ${\textrm{Inv}_{\omega}}$-closure of all  the sheaves in (\ref{longs -seq -incl -sh -fcts}).

Next, the sheaves of vector fields induced by them (cf. Definition \ref{def -Vect}) coincide with ${\XX_{\mu}}$:
\begin{align*}
 \Vect(\ucalC_{\mu}^\textsc{lin}) &= \Vect(\ucalC_{\mu}^\textsc{pol})= \Vect(\ucalC_{\mu}^{\infty}) \\ &= \Vect(\ucalC_{\mu}^\textsc{orb})  = \Vect(\ucalC_{\mu}^\textsc{fib})= \Vect(\CM_\mu)= \XX_{\mu}.
\end{align*} 
In particular, since ${\CM_{\mu}}$ is closed, ${\CM_{\mu}}$ is the closure of all the sheaves in (\ref{longs -seq -incl -sh -fcts}). Finally, for the Hamiltonian envelopes (cf. Definition \ref{def -Ham -env}), one also has:
\begin{equation}\label{eqs:Ham -env -sheaves} 
\textrm{Ham}(\ucalC_{\mu}^\textsc{lin})= \textrm{Ham}(\ucalC_{\mu}^\textsc{pol})= \textrm{Ham}(\ucalC_{\mu}^{\infty})= \gerX_\mu^\textsc{orb}.
\end{equation}
\end{proposition}
\noindent
\begin{proof} This is analogous to Proposition \ref{Lagr -sheaves -mai -prop} and, as there, the sequence of inclusion is immediate. And, continuing the analogy, one now applies the commutant operation to obtain a sequence of reverse inclusions; to show that first sequence of equalities, it suffices to show that 
${\ucalC_{\mu}\sub \big( \ucalC_\mu\big)'}$ and ${\big( \ucalC_\mu^\textsc{pol} \big)'\sub \ucalC_{\mu}}$; the first one is the \nameref{cis:lem:main} on page~\pageref{cis:lem:main} while the second one is obvious. 

And similarly applying the operation ${\Vect( -)}$, where we already know that ${\Vect(\ucalC_{\mu})= \gerX_{\mu}}$ (Corollary \ref{cor -X -mu -Lagr}). Hence we are left with showing that ${\Vect(\ucalC_{\mu}^\textsc{lin})\sub \gerX_{\mu}}$ which is again clear. 

For the last part it is clear (and was also remarked in the proof of Proposition \ref{prop -X -orb -pps}) that ${\textrm{Ham}(\ucalC_{\mu}^\textsc{lin})=  \gerX_\mu^\textsc{orb}}$. One is left with showing that ${\textrm{Ham}(\ucalC_{\mu}^{\infty})\sub \gerX_\mu^\textsc{orb}}$. 
That follows from the formula used to
prove that the Hamiltonian envelope is finitely generated (see the proof of Proposition \ref{AS-fol-curious}). 
\end{proof}
\noindent
\begin{proof}[End of proof of Proposition \ref{main -prop -sheavs -int -syst}]
Part 5 was proven in Corollary \ref{cor -X -mu -Lagr} and Proposition \ref{Lagr -sheaves -mai -prop};  we have also seen that ${\gerX_{\mu}}$ is Lagrangian and ${\ucalC_{\mu}}$ is self-centralizing. These imply in particular (by the general discussion on sheaves) that 
${\gerX_{\mu}}$ is closed and strongly involutive and ${\ucalC_{\mu}}$ is closed and ${\textrm{Inv}_{\omega}}$-invariant. Now, 1--4 follow from the formal properties of taking orthogonals, commutants, etc. For instance, if ${\ucalC}$ is a self-centralizing sheaf containing the ${\mu_i}$'s,
we will have ${\ucalC_{\mu}^\textsc{lin}\subset \ucalC}$ and, passing to commutants, ${\ucalC\subset \ucalC_{\mu}}$. Passing again to commutants, we find that ${\ucalC= \ucalC_{\mu}}$. This proves 1, and 3 is completely similar. For 2 (and then 4 analogously) assume that 
${\ucalC}$ is a closed sheaf containing the ${\mu_i}$'s; taking the closures in ${\ucalC_{\mu}^\textsc{lin}\subset \ucalC}$, we find that ${\ucalC_{\mu}\subset \ucalC}$. 
\end{proof}
\noindent
\chapter{The Moser path method for integrable systems}
\label{sec:Equivalence of integrable systems}
\etocsettocdepth{2}
\localtableofcontents

\newpage
\section{Equivalences}
\label{ssec:equivalences -}
\etocsettocstyle{\subsubsection*{Local Contents}}{}
\etocsettocdepth{2}
\localtableofcontents

{\color{white}.}
\newline
In this section we discuss various notions of equivalences between integrable systems. Why several (different) ones?
Start with two integrable systems
\[ \mu: (M, \omega)\rightarrow \mathbb{R}^n, \ \mu': (M', \omega')\rightarrow \mathbb{R}^n.\]
There is no doubt on the very minimal data one should require for the two systems to be equivalent: there should be a symplectomorphism
\[ \Phi: (M, \omega)\rightarrow (M', \omega').\]
But what more? That really depends on which part of the structure induced by the moment maps one is interested in - hence one would like to preserve. Perhaps the most natural scenario is the existence of a base diffeomorphism ${\phi: \mathbb{R}^n\rightarrow \mathbb{R}^n}$, or just from ${\mu(M)}$ to ${\mu'(M')}$, such that 
\[ \mu'\circ \Phi= \phi\circ \mu.\]
However, this turns out to be too restrictive even for proving some basic normal forms. Of course, the existence of ${\phi}$ has several consequences - such as the fact that ${\Phi}$ preserves the orbits of the action of ${\mathbb{R}^n}$, or the fibers of ${\mu}$, or the Lagrangian foliation, etc. Therefore, for other interesting notions of equivalence, one has to decide which of such structures one would like to preserve. Here are some possibilities:
\begin{myenumerate}
\item one requires that the induced actions of ${\mathbb{R}^n}$ are isomorphic (in particular, the orbits will be preserved). 
\item if one is interested in the orbits, remembering that they do not depend so much on the action but on the action algebroids at hand, one requires that the induced action algebroids are preserved.
\item if one is interested only in the `Lagrangian foliations', then one requires that the sheaves encoding the Lagrangian foliation (${\ucalC_{\mu}}$ or ${\gerX_{\mu}}$ from the previous chapter) are preserved.
\end{myenumerate}
One obtains several notions of equivalence, generically called ${\chi}$-equivalences, with 
$ \chi \in \set{ \text{act}, \text{fct}, \text{orb}, \text{fib}, \text{wk}, \text{alg}}, $
discussed below: action, (weak) functional, orbit, fiber, weak and algebraic, respectively. We shall also see that each ${\chi}$-equivalence is characterized by the invariance of certain ${\chi}$-sheaves (of course, the ones already discussed in the previous chapter). 
At one extreme end, one finds the notion of `isomorphism' - which preserve everything (e.g.\ , all the sheaves ${\ucalA_{\mu}}$ for any ${\ucalA}$- see Motivation 1 in section \ref{sub:sec:sing -Lag -fol}):

\begin{definition}
We say that ${(M, \omega, \mu)}$ and ${(M', \omega', \mu')}$ are \textbf{isomorphic} if there exists a symplectomorphism 
${\Phi:(M,\omega)\to (M,\omega')}$ such that ${\mu'\circ \Phi = \mu}$. 
\end{definition}
\noindent
However, we will concentrate on the more interesting relations \[\chi \in \set{ \text{act}, \text{fct}, \text{orb}, \text{fib}, \text{wk}, \text{alg}}.\] 
Here is the way they interact: 

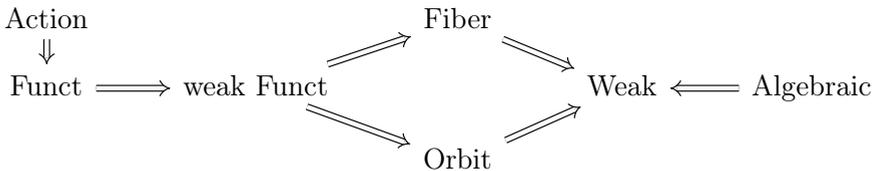
\[ 
\begin{tikzcd}[row sep=small]
    \text{Action}
    \arrow[Rightarrow]{d} & & 
    \text{Fiber} 
    \arrow[Rightarrow]{dr}
    & & \\
    \text{Funct}
    \arrow[Rightarrow]{r}
    & 
    \text{weak Funct}
     \arrow[Rightarrow]{ur} 
    \arrow[Rightarrow]{dr}
    & &
    \text{Weak}
    &    
     \text{Algebraic} 
    \arrow[Rightarrow]{l}
    \\  & &
    \text{Orbit}
    \arrow[Rightarrow]{ur}
    & &
    \end{tikzcd}
\]

\subsection{Action-equivalence}

We start with two integrable systems 
${(M, \omega, \mu)}$ and ${(M', \omega', \mu')}$ as above and we now concentrate on the induced actions of the abelian Lie algebra ${\mathbb{R}^n}$, 
\begin{equation}\label{2 -in -syst -for -reference} 
\rho: \mathbb{R}^n\rightarrow \gerX(M), \ \rho': \mathbb{R}^n\rightarrow \gerX(M').
\end{equation}

\begin{definition}
We say that ${(M, \omega, \mu)}$ and ${(M', \omega', \mu')}$ are \textbf{action equivalent}\index{equivalences of integrable systems!action} if there exists a symplectomorphism 
${\Phi:(M,\omega)\to (M,\omega')}$ and a linear isomorphism ${B: \mathbb{R}^n\rightarrow \mathbb{R}^n}$ such that 
\begin{equation}\label{phi -rho -B}
 \begin{tikzcd}
            \mathbb{R}^n
            \arrow{r}{\rho}
            &
            \gerX(M)
            \\
            \mathbb{R}^n
            \arrow{u}{B}
            \arrow{r}{\rho'}
            &
            \gerX(M') 
            \arrow{u}{\Phi^*}
\end{tikzcd}
\quad \quad \Phi^*\circ \rho'= \rho\circ B .
\end{equation}

\end{definition}
\noindent
Of course, this definition makes sense more generally, 
for arbitrary Lie algebras ${\mathfrak{g}}$ and ${\mathfrak{g}'}$ and Hamiltonian spaces 
\[ \mu: (M, \omega)\rightarrow \mathfrak{g}, \ \mu': (M', \omega')\rightarrow \mathfrak{g}';\]
and it is actually instructive/clarifying to think about this more general setting. Then ${B}$ will be required to be a Lie algebra isomorphism
\[ B: \mathfrak{g}'\rightarrow \mathfrak{g},\]
still satisfying (\ref{phi -rho -B}).

\begin{lemma}
\label{lemma -affine -equivalence} 
Assume that ${M}$ is connected. Considering ${A= B^*: \mathfrak{g}^*\rightarrow \mathfrak{g}^{'\, *}}$, the equation (\ref{phi -rho -B}) is equivalent to the fact that there exists ${\gamma\in \mathfrak{g}^{'\, *}}$ such that the following diagram is commutative:
\[
        \begin{tikzcd}
            M
            \arrow{r}{\mu}
            \arrow{d}{\phi}
            &
            \mathfrak{g}^*
            \arrow{d}{A(\cdot)+\gamma}
            \\
            M'
            \arrow{r}{\mu'}
            &
            \mathfrak{g}^{'\, *}
        \end{tikzcd}
\]        
Moreover, it follows that the element ${\gamma\in \mathfrak{g}^{'\, *}}$ must vanish on ${[\mathfrak{g}', \mathfrak{g}']}$. 

Finally, the condition on ${\gamma}$ together with the fact that ${B}$ is a Lie algebra isomorphism is equivalent to the fact that ${A+ \gamma}$ is a Poisson diffeomorphism.
\end{lemma}

\begin{proof} 
We write the condition (\ref{phi -rho -B}) on an arbitrary element ${v\in \mathfrak{g}}$
\[ \rho(B(v)= \Phi^*(\rho'(v))\]
and we will rewrite it using the moment map condition:
\[ \omega(\rho(u), -)= d\mu_u\]
for all ${u\in \mathfrak{g}}$, where ${\mu_u= \langle \mu, u\rangle}$. For ${u= B(v)}$, our condition gives us
\[ \omega(\Phi^*(\rho'(v)), -)= d\mu_{B(v)}.\]
Since ${\Phi}$ is a symplectomorphism, the first term is ${\omega'(\rho'(v), \Phi_*( -))}$ i.e.\ , using the moment map condition for ${\mu'}$, it is ${\Phi^*(d\mu'_{v})}$. Hence 
\[ d(\mu'_{v}\circ \Phi)= d(\mu_{Bv}). \]
(equality of exact 1-forms on ${M}$). It is also clear that ${\mu_{Bv}= \langle A\circ \mu, v\rangle}$. Since ${M}$ is connected, we find that
\[ \mu'\circ \Phi= A\circ \mu+ \gamma \]
for some constant ${\gamma_v}$ (one for each ${v\in \mathfrak{g}}$); of course, this defines ${\gamma}$, and it is clearly linear. Note that our argument was written so that it can be red also backwards - proving therefore 
the required equivalence. 

We now check that ${\gamma}$ vanishes on an arbitrary bracket ${[u, v]}$ with ${u, v\in \mathfrak{g}}$. We use the defining formula for ${\gamma}$, ${\gamma_{\bullet}= \mu'_{\bullet}\circ \Phi - \mu_{B(\bullet)}}$, and the fact that 
${\{\mu_{u_1}, \mu_{u_2}\}_{\omega}= \mu_{[u_1, u_2]}}$ and similarly for ${\mu'}$. We find
\[ \gamma_{[u, v]}= \{\mu'_{u}, \mu'_{v}\}_{\omega'}\circ \Phi - \mu_{B([u, v])}.\]
Using now that ${\Phi}$ is itself a Poisson map, and ${B}$ is a Lie algebra map we find
\[ \gamma_{[u, v]}= \{\mu'_{u}\circ \Phi, \mu'_{v}\circ \Phi\}_{\omega} - \{\mu_{B(u)}, \mu_{B(v)}\}_{\omega}.\]
Replacing ${\mu'_{u}\circ \Phi}$ by ${\mu_{B(u)}+ \gamma_u}$ and the similar expression for ${v}$, and using that ${\gamma_u}$ and ${\gamma_v}$ are constant, we immediately obtain that, indeed, the right hand side vanishes.
The last part follows similarly (or even by a straightforward local computation). 
\end{proof}
\noindent
\begin{proposition} If ${M}$ is connected, two integrable systems (\ref{2 -in -syst -for -reference}) are action-equivalent if and only if there exists a
symplectomorphism \[{\Phi:(M,\omega)\to (M,\omega')}\] and an affine isomorphism ${\phi: \mathbb{R}^n\rightarrow \mathbb{R}^n}$ such that the following square commutes:
\[
        \begin{tikzcd}
            M
            \arrow{r}{\mu}
            \arrow{d}{\Phi}
            &
            \R^n
            \arrow{d}{\phi}
            \\
            M'
            \arrow{r}{\mu'}
            &
            \R^n.
        \end{tikzcd}
\]
\end{proposition}
\noindent
Hence action-equivalence fits into the framework discussed at the beginning of Section \ref{sub:sec:sing -Lag -fol} when we motivated the various sheaves of functions (Motivation 1 there). Hence, the following should not come as a surprise:

\begin{proposition}\label{cis:prop:aff_eq}
For two completely integrable systems (\ref{2 -in -syst -for -reference}) the following are equivalent:
\begin{myitemize}
\item ${(\omega,\mu)}$ and ${(\omega',\mu')}$ are action-equivalent.
\item There is a symplectomorphism ${\Phi:(M,\omega)\to(M,\omega')}$ s.t.\ ${\ucalC_\mu^\textsc{aff}=\Phi^*\ucalC_{\mu'}^\textsc{aff}}$.
\item There is a symplectomorphism ${\Phi:(M,\omega)\to(M,\omega')}$ s.t.\ ${\gerX_\mu^\textsc{lin}=\Phi^*\gerX_{\mu'}^\textsc{lin}}$.
\end{myitemize}
\end{proposition}
\noindent
\begin{proof}
We can assume that ${M}$ is connected. 
Assume the equivalence first, and write ${B= (b_{i}^{p})}$ so that, on the canonical basis, ${B(e_p)= \sum_i b_{p}^{i} e_i}$. Then the condition (\ref{phi -rho -B}) reads
\[ 
    \Phi^*(X_{\mu^{\, '}_{p}})
    = 
    \sum\nolimits_i 
    b_{p}^{i} X_{\mu_i}.
\]
Similarly, the commutativity of the diagram from Lemma \ref{lemma -affine -equivalence} (i.e.\  the equations ${\gamma_{\bullet}+ \mu_{B(\bullet)}= \mu'_{\bullet}\circ \Phi}$ from the previous proof) reads:
\[ \Phi^*(\mu^{\, '}_{p})= \sum\nolimits_i b_{p}^{i} \mu_i+ \gamma_p,\]
where ${\gamma_p}$ are constants (the components of ${\gamma}$). With these two formulas, one for vector fields and one for functions, holding globally on ${M}$, the statements about the sheaves are immediate. 
For the converses, we find that the previous formulas hold locally. But, as we have already used, the restrictions of the ${X_{\mu_i}}$'s to any open are linearly independent (even over smooth functions), hence the first formula will hold globally. 
One also needs to remark that, by looking at a regular point, the matrix ${(b_{i}^{p})}$ is, indeed, invertible.
\end{proof}

\subsubsection{Orbit equivalence}

As we have mentioned, if one is interested only on the orbits of the action of ${\mathbb{R}^n}$ rather than the action itself, then the relevant object is the induced action algebroid.

\begin{definition} 
We say that two integrable systems ${(M, \omega, \mu)}$ and ${(M', \omega', \mu')}$ are \textbf{orbit equivalent}\index{equivalences of integrable systems!orbit} if there exists an isomorphism between the action Lie algebroids ${\mathbb{R}^n\ltimes M}$ and ${\mathbb{R}^n\ltimes M'}$ induced, covering a symplectomorphism ${\Phi: (M, \omega)\rightarrow (M', \omega')}$.
\end{definition}

\noindent
An isomorphism between the action algebroids is, before anything, an isomorphism of (trivial) vector bundles (covering ${\Phi}$):
\begin{equation}\label{eq -vbisom} 
(\Phi, C): M\times \mathbb{R}^n\rightarrow M'\times \mathbb{R}^n,\ (x, v)\mapsto (\Phi(x), C(x, v));
\end{equation}
hence ${\Phi: M\rightarrow M'}$ is a diffeomorphism and ${C}$ is a smooth map
\[ C: M\rightarrow GL_n .\]
Next, the isomorphism (\ref{eq -vbisom}) should be compatible with the anchor maps (themselves just the infinitesimal actions) - which translate into the following commutative diagram
\[
        \begin{tikzcd}
            M\times \mathbb{R}^n
            \arrow{r}{\rho}
            \arrow{d}{(\Phi, C)}
            &
            TM
            \arrow{d}{d\Phi}
            \\
            M'\times \mathbb{R}^n
            \arrow{r}{\rho'}
            &
            TM'.
        \end{tikzcd}
\]
Finally, (\ref{eq -vbisom}) should be compatible with the Lie algebroid brackets; but this follow automatically from the compatibility with the anchors and the fact that the algebroids under discussion have anchors that are injective almost everywhere. Putting everything together, and using also Lemma \ref{lemma -rk -orbit-sheaf-C}, we are led to the following:

\begin{corollary}
\label{cor -orbut -equiv} 
Two integrable systems ${(M, \omega, \mu)}$ and ${(M', \omega', \mu')}$ are orbit equivalent if and only if there exists a symplectomorphism ${\Phi: (M, \omega)\rightarrow (M', \omega')}$ and a smooth map
${C: M\rightarrow GL_n}$
such that 
\begin{equation}\label{eqorbit-equiv} 
X_{\mu_i}= \sum_{j} C^{j}_{i} \cdot X_{\Phi^*\mu_{j}'}
\end{equation}
(where all the Hamiltonian vector fields are w.r.t.\  ${\omega}$). Moreover, it follows that ${\{\mu_i, \Phi^*\mu_j\}_{\omega}= 0}$ and ${C^{j}_{i}\in \ucalC_{\mu}(M)}$ for all ${i}$ and ${j}$. 
\end{corollary}

The linear independence over smooth functions of the 
Hamiltonian vector fields gives, as remarked in the proof of Proposition \ref{cis:prop:aff_eq}:

\begin{proposition}\label{cis:prop:orb_eq}
For two completely integrable systems (\ref{2 -in -syst -for -reference}) the following are equivalent:
\begin{myitemize}
\item ${(\omega,\mu)}$ and ${(\omega',\mu')}$ are orbit equivalent.
\item There is a ${\Phi:(M,\omega)\to(M,\omega')}$ s.t.\ ${\gerX_\mu^\textsc{orb}=\Phi^*\gerX_{\mu'}^\textsc{orb}}$.
\item There is a ${\Phi:(M,\omega)\to(M,\omega')}$ s.t.\ ${\gerX_\mu^{\textrm{ab}\textsc{orb}}=\Phi^*\gerX_{\mu'}^{\textrm{ab}\textsc{orb}}}$.
\item There is a ${\Phi:(M,\omega)\to(M,\omega')}$ s.t.\ ${\ucalC_\mu^\textsc{orb}=\Phi^*\ucalC_{\mu'}^\textsc{orb}}$.
\end{myitemize}
\end{proposition}

\subsection{Fiber-Equivalence}

We now concentrate on the fibers of the moment map. 

\begin{definition} 
We say that two integrable systems ${(M, \omega, \mu)}$ and ${(M', \omega', \mu')}$ are \textbf{fiber equivalent}\index{equivalences of integrable systems!fiber} if there exists a ${\Phi: (M, \omega)\rightarrow (M', \omega')}$ which takes the fibers of ${\mu}$ to the fibers of ${\mu'}$.
\end{definition}
\noindent
Equivalently, there one has a commutative diagram
\[
        \begin{tikzcd}
            M
            \arrow{r}{\mu}
            \arrow{d}{\Phi}
            &
            \mu(M)
            \arrow{d}{\phi}
            \\
            M'
            \arrow{r}{\mu'}
            &
            \mu'(M').
        \end{tikzcd}
\]
for some \emph{set theoretical} bijection ${\phi}$. Given our intuition when introducing the sheaves ${\ucalC_{\mu}^\textsc{fib}}$, the following should not come as a surprise:

\begin{proposition}\label{prop -fiber -equiv -sheaves} Two completely integrable systems ${(M,\omega)}$ and ${(M,\omega')}$ are fiber-equivalent if and only if there exists a symplectomorphism \[{\Phi:(M,\omega)\to (M,\omega')}\] such that
\[ \Phi^*(\ucalC^{\textsc{fib}}_{\mu'})= \ucalC^{\textsc{fib}}_{\mu}.\]
\end{proposition}
\noindent
\begin{proof}
Given a diffeomorphism ${\Phi}$, when is a smooth function ${f \in C^\infty(U)}$, defined on an open subset ${U\subset M}$, a section of ${\Phi^*(\ucalC^{\textsc{fib}}_{\mu'})}$? Of course, \[{f\circ \Phi^{ -1}: \Phi(U)\rightarrow \mathbb{R}}\] must be in ${\ucalC^{\textsc{fib}}_{\mu'}}$; applying the definition of ${\ucalC^{\textsc{fib}}_{\mu'}}$ but writing the opens in ${M'}$ as ${\Phi(V)}$ with ${V}$-open in ${M}$, we see that the condition is: each point ${x\in M}$ has a neighborhood ${V}$ such that ${f|_{V}}$ is constant on the fibers of ${\mu^{\, '}\circ \Phi|_{V}}$. 

So, in the case of a fiber -equivalence, when ${\mu^{\, '}\circ \Phi|_{V}= \phi\circ \mu|_{V}}$, we see that the condition is that there are such neighborhoods ${V}$ such that ${f|_{V}}$ is constant on the fibers of ${\mu|_{V}}$- hence the equality of the two sheaves.

For the converse, we use the remark we started with for ${f= \mu_i}$- which are clearly sections of ${\ucalC^{\textsc{fib}}_{\mu}}$. We find that any point ${x\in M}$ has a neighborhood ${V}$ such that ${\mu_{i}|_{V}}$ is constant on the fibers of ${\mu^{\, '}\circ \Phi|_{V}}$. Of course, this property is preserved when passing to a smaller open; also, we can apply it to each component ${\mu_i}$ of ${\mu}$. Therefore one can find ${V}$'s such that ${\mu|_{V}}$ is constant on the fibers of ${\mu^{\, '}\circ \Phi|_{V}}$. Using the similar argument but with ${\mu}$ and ${\mu^{\, '}}$ interchanged (and using ${\Phi^{ -1}}$), we see that we can ensure that ${\mu|_{V}}$ and ${\mu^{\, '}\circ \Phi|_{V}}$ have the same fibers, hence a commutative diagram
\[
        \begin{tikzcd}
            V
            \arrow{r}{\mu}
            \arrow{d}{\Phi}
            &
            \mu(V)
            \arrow{d}{\phi_V}
            \\
            \Phi(V)
            \arrow{r}{\mu'}
            &
            \mu'(\Phi(V)).
        \end{tikzcd}
\]
for some set theoretical bijection ${\phi_V}$. Again, such a ${V}$ (and a ${\phi_V}$) can be found around any point in ${M}$. Since ${\phi_V}$ is uniquely determined by the rest of the diagram, we see that they glue together to a global ${\phi}$. 
\end{proof}
\noindent
The previous proposition probably has a more conceptual proof, where the key remark is that the notion of fiber -equivalence is local.

\subsection{Functional Equivalence}

We now come to the equivalence relation which, as we have already mentioned, is probably the most natural one (but too strong for various purposes). 

\begin{definition}\label{def -ftc -equivalence}
Call two integrable systems $(M, \omega, \mu)$ and $(M',\omega',\mu')$ \textbf{weakly functionally equivalent}\index{equivalences of integrable systems!weak functional} if there exists a ${\Phi: (M, \omega)\rightarrow (M', \omega')}$ and a diffeomorphism ${\phi: \mu(M)\rightarrow \mu'(M')}$ such that ${\psi\circ \mu= \mu'\circ \Phi}$, i.e.\ , the following square commutes:
\[
        \begin{tikzcd}
            M
            \arrow{r}{\mu}
            \arrow{d}{\Phi}
            &
            \mu(M)
            \arrow{d}{\phi}
            \\
            M'
            \arrow{r}{\mu'}
            &
            \mu'(M').
        \end{tikzcd}
\]
We say that the systems are \textbf{functionally equivalent}\index{equivalences of integrable systems!functional} if ${\phi: \mu(M)\rightarrow \mu'(M')}$ can be extended to a smooth diffeomorphism between open neighborhoods of ${ \mu(M)}$ and ${\mu'(M')}$.
\end{definition}
\noindent
Here, as before (see e.g.\  the second part of Remark \ref{rk -functions -better}), by a smooth function ${\phi: A\rightarrow B}$ between two subspaces of ${\mathbb{R}^n}$ we mean a map which, around each point in ${a}$, is the restriction of a smooth map defined on an open neighborhood of ${a}$ in ${\mathbb{R}^n}$. Of course, ${\phi}$ is called a diffeomorphism if it is bijective and both ${\phi}$ and ${\phi^{ -1}}$ are smooth.
Recall here that, for {\it closed} subset ${A\subset \R^n}$, any function ${f A\rightarrow \R^n}$ which is smooth in this sense admits an smooth extension to an open neighborhood of ${A}$ (just use a partition of unity argument). However, even in this case, if ${f}$ is a diffeomorphism onto its image it is not clear whether the smooth extension can be realized by a diffeomorphism. 

The nature of this equivalence relation becomes clear if we see ${\R^n}$ as a manifold with the trivial Poisson structure, and ${\Phi}$ and ${\phi}$ as Poisson diffeomorphisms. Functional equivalence tells us that the systems are isomorphic as Poisson maps.

\begin{proposition}\label{prop:funct -equiv}
Two completely integrable systems ${(M,\omega)}$ and ${(M',\omega')}$ are weakly functionally equivalent if and only if there exists a symplectomorphism ${\Phi:(M,\omega)\to (M',\omega')}$ such that
\[ \Phi^*(\ucalC^{\infty}_{\mu'})= \ucalC^{\infty}_{\mu}.\]
\end{proposition}
\noindent
\begin{proof}
We follow the same line of proof as for Proposition \ref{prop -fiber -equiv -sheaves}. We start with a similar remark, to describe which 
smooth functions ${f: U\rightarrow \mathbb{R}}$, defined on an open subset ${U\subset M}$, a section of ${\Phi^*(\ucalC^{\infty}_{\mu'})}$ or, equivalently when is ${f\circ \Phi^{ -1}: \Phi(U)\rightarrow \mathbb{R}}$ a section of ${\ucalC^{\infty}_{\mu'}}$. Writing out the definition of the last sheaf and expressing the opens in ${M}$ as images by ${\Phi}$ of opens in ${M}$ we find the condition: for each point ${x\in U}$ there is a neighborhood ${V\subset U}$ and a smooth function ${g: \mu^{\, '}(\Phi(V))\rightarrow \mathbb{R}}$ such that 
\[ f|_{V}= g\circ \mu^{\, '}\circ \Phi|_{V}.\]

When the two systems are equivalent, since ${\mu^{\, '}\circ \Phi= \phi\circ \mu}$ and since ${g\mapsto g\circ \phi}$ is a bijection between smooth functions ${g}$ on ${\mu^{\, '}(\Phi(V))= \phi(\mu(V))}$ and smooth functions on ${\mu(V)}$, we obtain the desired equality of sheaves.

For the converse, we apply the original remark to ${f= \mu_i}$ (which is clearly in ${\ucalC^{\infty}_{\mu}}$). Then we find, around each point ${x}$, a neighborhood ${V}$ and a smooth map ${g_i: \mu^{\, '}(\Phi(V))\rightarrow \mathbb{R}}$ such that 
${\mu_i|_{V}= g_i\circ \mu^{\, '}\circ \Phi|_{V}}$. We do this for any ${i}$ hence, after eventually shrinking ${V}$, we find a smooth map ${g_V: \mu^{\, '}(\Phi(V))\rightarrow \mathbb{R}}$ fitting into a commutative diagram:
\[
        \begin{tikzcd}
            V
            \arrow{r}{\mu}
            \arrow{d}{\Phi}
            &
            \mu(V)
            \\
            \Phi(V)
            \arrow{r}{\mu'}
            &
            \mu'(\Phi(V))
            \arrow{u}{g_V}
        \end{tikzcd}
\]
Again, the ${g_V}$ is uniquely determined by the rest of the diagram, hence they glue to a global smooth map ${g: \mu^{\, '}(M')\rightarrow \mu(M)}$ so that the following diagram is commutative:
\[
        \begin{tikzcd}
            M
            \arrow{r}{\mu}
            \arrow{d}{\Phi}
            &
            \mu(M)
            \\
            \Phi(M)
            \arrow{r}{\mu'}
            &
            \mu'(\Phi(M))
            \arrow{u}{g}
        \end{tikzcd}
\]
Similarly, interchanging the roles of ${\mu}$ and ${\mu'}$ (and using ${\Phi^{ -1}}$), we find a similar smooth map ${g'}$ going backwards. From the commutativity of the two diagrams, i.e.\  from ${\mu= g\circ \mu'\circ \Phi}$ and ${\mu'= g'\circ \mu\circ \Phi^{ -1}}$, 
one immediately obtains that ${g\circ g'= \textrm{Id}}$ on ${\mu(M)}$ and ${g'\circ g= \textrm{Id}}$ on ${\mu'(M')}$, hence ${g}$ is a diffeomorphism in the sense described above. 
\end{proof}

\subsection{Weak Equivalence}

We now come to the equivalence relation which, when one first encounters, seems to be the least intuitive one. We hope that our discussion (see Remark \ref{rk:our -expl -wek -equiv} below) is more clarifying.

\begin{definition} 
We say that two integrable systems ${(M, \omega, \mu)}$ and ${(M', \omega', \mu')}$ are \textbf{weakly equivalent}\index{equivalences of integrable systems!weak} if there is a symplectomorphism ${\Phi:(M,\omega)\to (M',\omega')}$ such that the ${\mu_i}$ and ${\mu_j'\circ\Phi}$ Poisson commute, i.e.\ , 
\begin{equation}\label{eq -comm -weak-equiv}
        \{
            \mu_i,
            \mu_j' \circ \Phi
        \}_\omega
        =
        0,
        \quad\forall\; i,j.
\end{equation}
\end{definition}
\noindent
This definition is motivated by Eliasson's local form theorems which, in general, provide only such weak equivalences. 

\begin{example}
\label{ex:weak-equiv -strange}  
One should be aware however that weak equivalences do not preserve the singular points. For instance, w.r.t.\  the canonical symplectic structure on ${\R^2}$, both 
${\mu(x,y)\defeq x}$ and ${\mu'(x,y)\defeq x^2}$ are completely integrable systems, the first one regular everywhere and the second one singular at every point on the ${y}$-axis. However,
they are weakly equivalent (with ${\Phi= \textrm{Id}}$):
    \[
        \{
            \mu,\mu'
        \}_{\omega_\text{can}}
        =
        \frac{\partial \mu}{\partial x}
        \frac{\partial \mu'}{\partial y}
        -
        \frac{\partial \mu}{\partial y}
        \frac{\partial \mu'}{\partial x}
        =
        0,
    \]
since neither function depends on the ${y}$-variable. 

It is instructive to try to construct a similar example with ${\mu'= x^2+ y^2}$ (the elliptic case). That does not work, and this is just an indication of the fact that the phenomena mentioned above
cannot occur for nondegenerate singularities. Even more, for such nondegenerate integrable systems, the Williamson type of a singularity is invariant under weak equivalences (see Theorem \ref{weak preserves nondegenerate} below). 
\end{example}
\noindent
The fact that weak equivalence is an equivalence relation is not directly apparent from the definition, but it will follow from the \nameref{cis:lem:main} on page~\pageref{cis:lem:main}.

\begin{lemma}
Weak equivalence is an equivalence relation for completely integrable systems.
\noindent
Moreover, if ${(\omega,\mu)}$ is a completely integrable system and ${\mu':M\to \R^n}$ is a smooth map such that
    $
        \{
            \mu_i,
            \mu_j'
        \}_\omega
        =
        0
    $
for all ${i}$ and ${j}$, then ${(\omega,\mu')}$ is also an integrable system.
\end{lemma}

\begin{proof}
Only transitivity is nontrivial, and it uses that the integrable systems are complete. For suppose that there are integrable systems ${(\omega,\mu)}$, ${(\omega',\mu')}$ and ${(\omega'',\mu'')}$ on ${M}$, ${M'}$ and ${M''}$, respectively, and suppose that there symplectomorphisms ${\Phi:(M,\omega)\to (M',\omega')}$ and ${\Phi':(M',\omega')\to (M'',\omega')}$ such that
    \[
        \{
            \mu_i,
            \mu_j' \circ \Phi
        \}_\omega
        =
        \{
            \mu_i',
            \mu_j'' \circ \Phi'
        \}_{\omega'}
        =
        0,
        \quad\forall\; i,j.
    \]
Since ${\Phi}$ is a symplectomorphism, it follows that also
    $
        \{
            \mu_i'\circ\Phi,
            \mu_j''\circ\Phi'\circ\Phi
        \}_\omega
        =
        0
    $
for all ${i}$ and ${j}$. Since ${(\omega,\mu'\circ\Phi)}$ is also a completely integrable system, it follows from the \nameref{cis:lem:main} on page~\pageref{cis:lem:main} that
    \[
        \{
            \mu_i,
            \mu_j''\circ\Phi'\circ\Phi
        \}_\omega
        =
        0,
        \quad\forall\; i,j.
    \]
This shows that ${(\omega,\mu)}$ and ${(\omega'',\mu'')}$ are weakly equivalent, so weak equivalence is transitive.
The second statement also follows directly from the \nameref{cis:lem:main} on page~\pageref{cis:lem:main}: if 
\[
    {\{\mu_i,\mu'_j \}_\omega=\{\mu_i,\mu'_k\}_\omega=0}
\]
for all ${i}$, then ${\{\mu'_j,\mu'_k\}_\omega=0}$.
\end{proof}
\noindent
\begin{remark}\label{rk:our -expl -wek -equiv}
The intuition for the notion of weak equivalence is that ${\Phi}$ `preserves the Lagrangian foliations'. We hope that our discussion on the various sheaves (and their properties) of an integrable system makes it clear and more intuitive why `the Lagrangian foliation' is encoded by the sheaf ${\ucalC_{\mu}}$ (which is of most importance to this thesis). 
Therefore the condition that ${\Phi}$ preserves the Lagrangian foliation takes the precise form:
\[ \Phi^*\ucalC_{\mu'}= \ucalC_{\mu}.\]
The next proposition shows that this is indeed equivalent to the previous definition; this will also make the fact that weak equivalence is an equivalence relation more transparent. 
\end{remark}
\noindent
\begin{proposition}\label{cis:prop:weak_eq}
For two completely integrable systems (\ref{2 -in -syst -for -reference}) the following are equivalent:
\begin{myitemize}
\item ${(\omega,\mu)}$ and ${(\omega',\mu')}$ are weakly-equivalent.
\item There is a symplectomorphism ${\Phi:(M,\omega)\to(M',\omega')}$ such that ${\ucalC_\mu=\Phi^*\ucalC_{\mu'}}$.
\item There is a symplectomorphism ${\Phi:(M,\omega)\to(M,\omega')}$ such that ${\gerX_\mu=\Phi^*\gerX_{\mu'}}$.
\end{myitemize}
\end{proposition}
\noindent
\begin{proof}
Suppose that ${(\omega,\mu)}$ and ${(\omega',\mu')}$ are weakly equivalent, and choose a symplectomorphism ${\Phi}$ that realizes this. Now if ${f\in \CM_\mu}$, then 
    \[
        \{f\circ\Phi^{ -1},\mu_j'\}_{\omega'}
        =
        \{f,\mu_j'\circ\Phi \}_{\omega}\circ\Phi^{ -1}
        =
        0,
        \quad\forall\, j,
    \]
by the \nameref{cis:lem:main} on page~\pageref{cis:lem:main} and the assumption that ${\phi}$ is a symplectomorphism. This shows that \[\CM_\mu\sub\Phi^*\CM_{\mu'}.\] For the opposite inclusion, if ${g\in \CM_{\mu'}}$, then
    \[
        \{ g\circ\Phi, \mu_j'\circ\Phi \}_\omega
        =
        \{ g, \mu_j' \}_{\omega'}\circ\Phi
        =
        0,
        \quad\forall\, j.
    \]
Since ${(\omega,\mu'\circ\Phi)}$ is also a completely integrable system on ${M}$, and  ${\mu_i\in\CM_{\mu'\circ\Phi}}$, it follows again from the \nameref{cis:lem:main} on page~\pageref{cis:lem:main} that
    $
        \{ g\circ\Phi, \mu_i  \}_\omega
        =
        0
    $
for all $i$. This shows that ${\Phi^*\CM_{\mu'}\sub\CM_\mu}$.

Conversely, the equations (\ref{eq -comm -weak-equiv}) follow right away from any of the 
    \[  
    {\CM_\mu\sub\phi^*\CM_{\mu'}} \quad\text{or}\quad \CM_\mu\supseteq\phi^*\CM_{\mu'}.\] 
The fact that the any of these inclusions must be equality is
a manifestation of the fact ${\CM_{\mu}}$ is the only self-centralizing sheaf that contains the ${\mu_i}$'s. 

Finally, the equivalence between the last two statements is clear given the relationship between ${\CM_{\mu}}$ and ${\gerX_{\mu}}$ described in Proposition \ref{main -prop -sheavs -int -syst} (part 5) and the functoriality of ${\Fun}$ and ${\Vect}$. 
\end{proof}

\subsection{Algebraic Equivalence}

The notions of equivalence that we have discussed so far had a somewhat geometric/intuitive motivation. They are closely related to each other and, as we shall soon see, they all imply weak equivalence. However, to see how much weak equivalence is different from the rest, we are now
describing anther equivalence relation, more algebraic, which still implies weak equivalences but for which the interaction with the other equivalence relations is not so close. 

This notion starts from the remark that, if ${\mu: (M, \omega)\rightarrow \R^n}$ is an integrable system and one changes the components of ${\mu}$ linearly but in a nonconstant fashion, then one may still end up with an integrable system. More precisely, 
defining
$ \mu'_{k}= \sum_i a_{k}^{i} \cdot \mu_i $
where ${A= (a_{k}^{i}): M\rightarrow GL_n(\R)}$ is a smooth map with components ${a_{k}^{i}\in \ucalC_{\mu}(M)}$, then 
$ \mu': (M, \omega)\rightarrow \mathbb{R}^n $
is still an integrable system - as it follows immediately from the \nameref{cis:lem:main} on page~\pageref{cis:lem:main}. 
It is also clear that ${(\omega, \mu)}$ and ${(\omega, \mu')}$ are weakly equivalent. And, as for weak equivalences, singularities may disappear or appear. But, algebraically, 
${(\omega, \mu)}$ and ${(\omega, \mu')}$ look very similar, and they may be called equivalent. We are led to the following:

\begin{definition} 
Call two integrable systems ${(M, \omega, \mu)}$ and ${(M', \omega', \mu')}$ \textbf{algebraically equivalent}\index{equivalences of integrable systems!algebraic} if there is a ${\Phi:(M,\omega)\to (M',\omega')}$ and 
a smooth map ${A= (a_{k}^{i}):M\to\GL(\R^n)}$ with components ${a_{k}^{i}\in \ucalC_{\mu}(M)}$ 
such that ${\mu'\Phi= A\cdot\mu}$, i.e.\ , for each ${k}$:
\[ \Phi^*(\mu'_k)= \sum\nolimits_i a_{k}^{i} \cdot \mu_i \]
\end{definition}
\noindent
\begin{lemma}
Algebraic equivalence is an equivalence relation for completely integrable systems.

\end{lemma}

\begin{proof}
Suppose that ${(\omega,\mu)}$ and ${(\omega',\mu')}$ are completely integrable systems on ${M}$ which are strongly equivalent, i.e.\ , ${A\cdot \mu = \mu'\phi}$ and ${\set{\mu,A}_\omega=0}$. By rewriting the first identity as 
    \[
        \mu\phi^{ -1} 
        = 
        A(\phi^{ -1})^{ -1}\, \mu'
    \]
one sees that, to prove symmetry, it suffices to show that ${A(\phi^{ -1})^{ -1}}$ Poisson commutes with ${\mu'}$. Now observe that
    \[
        \{
            \mu_i,\mu'_j\phi
        \}_\omega
        =
        \{
            \mu_i,A^k_j\,\mu_k
        \}_\omega
        =
        A^k_j\, 
        \{
            \mu_i,\mu_k
        \}_\omega
        +
        \mu_k\,
        \{
            \mu_i,A^k_j
        \}_\omega
        =
        0,
        \quad\forall\, i,j,k.
    \]
Hence by the \nameref{cis:lem:main}  on page~\pageref{cis:lem:main} we conclude that ${\{\mu'\phi,A\}_\omega=0}$. Since ${\phi}$ is a symplectomorphism it follows that 
    $
        \{\mu',A(\phi^{ -1})\}_{\omega'}=0.
    $
The inverse ${A(\phi^{ -1})^{ -1}}$ is a polynomial expression in the components of ${A(\phi^{ -1})}$, hence it also Poisson commutes with ${\mu'}$.

Now suppose that ${(\omega'',\mu'')}$ is another completely integrable system on ${M}$ which is strongly equivalent to ${(\omega',\mu')}$, i.e.\ , ${B\cdot\mu'=\mu''\psi}$ and ${\set{\mu',B}_{\omega'}=0}$. Since
    \[
        B(\phi)\cdot A\cdot \mu
        =
        B(\phi)\cdot \mu'\phi
        =
        \mu''\psi\phi,
    \]
to prove transitivity, we only have to show that ${\{\mu,B(\phi) A\}_\omega=0}$. By the \nameref{cis:lem:main}, using that ${(\omega,\mu'\phi)}$ is also a completely integrable system on ${M}$, and that
    $
        \{\mu'_j\phi, \mu_i \}_\omega=0
    $ 
and 
    $
        \{\mu'\phi, B(\phi)\}_\omega
        =
        \{\mu',B\}_{\omega'}\phi=0,
    $ 
we conclude that ${\{\mu,B(\phi)\}_\omega=0}$. Hence
    \[
        \{ \mu, B(\phi) A \}_\omega
        =
        B(\phi) \{ \mu, A \}_\omega
        +
        \{ \mu, B(\phi) \}_\omega A
        =
        0.
    \]
This shows that strong equivalence is an equivalence relation.
\end{proof}
\noindent
\begin{remark}
Lemma \ref{prop:nondeg_function} suggests yet another interesting sheaf of vector fields associated with an integrable system:
    \[
        \gerY_\mu
        \defeq
        \big\{
            X \in \gerX_M
            \;\;\text{s.t.}\;\;
            \gerL_X(\mu_i) \in \CM_\mu
            \;\forall\; 
            i
        \big\}.
    \]
It is easy to show that ${X \in \gerY_\mu}$ if and only if ${\gerL_X(f)\in \CM_\mu}$ for all ${f \in \CM_\mu}$. This means that ${\gerY_\mu}$ are the derivations ${\CM_\mu}$, and that they form a Lie algebra sheaf. The above proposition shows that also any ${X\in \gerY_\mu}$ has ${X_x=0}$ and ${\dd_xX \in \gerh_{x,\mu}}$ at a nondegenerate fixed point ${x}$.
\end{remark}

\subsection{The local version of equivalences}
\label{The local version of equivalences; invariance under equivalences}

For the local normal forms we look at germs of integrable systems $\mu$ around a point ${x\in M}$. Such a germ will be called a \textbf{local integrable system near ${x}$}\index{integrable system!local} and will be referred to as \textbf{the local integrable system ${(\omega, \mu, x)}$}; so it is defined on a manifold ${M}$ containing ${x}$, but we will not distinguish it from the restrictions to an open neighborhood ${U}$ of ${x}$. 

It is clear that the various notions of equivalences that we have discussed, generically referred to as ${\chi}$-equivalences, with 
\[ \chi \in \set{ \text{act}, \text{orb}, \text{fib},  \text{fct}, \text{wk},  \text{alg}},\]
have an obvious local version.

\begin{definition}
\label{def:local -chi -equivlces} 
Let ${\chi \in \set{ \text{act}, \text{orb}, \text{fib},  \text{fct}, \text{wk},  \text{alg}}}$. Two local integrable systems ${(\omega, \mu, x)}$ and ${(\omega', \mu', x')}$ are \textbf{(locally) ${\chi}$-equivalent} if, when restricted to certain open neighborhoods of ${x}$ and ${x'}$, respectively, the resulting integrable systems are ${\chi}$-equivalent, with an equivalence whose underlying symplectomorphism takes ${x}$ to ${x'}$. 
\end{definition}
\noindent
As alluded to before, the main interest is on comparing an integrable systems, around a fixed point ${x}$, to its normal form around ${x}$. This deserves a definition on its own.

\begin{definition} \label{def:chi -linearizability}
Let 
${\chi \in \set{ \text{act}, \text{fct}, \text{orb},  \text{wk},  \text{alg}}}$. We say that an integrable system ${(\omega, \mu)}$ is \textbf{${\chi}$-linearizable around ${x}$} (or that the local integrable system ${(\omega, \mu, x)}$ is ${\chi}$-linearizable) if
${(\omega, \mu, x)}$ is ${\chi}$-equivalent to its normal model around (cf. Definition \ref{def-chi-lineariz}), interpreted as a local integrable system around the origin. 
\end{definition}
\noindent
The notions of (local) ${\chi}$-equivalences rise the question of \textbf{(local) invariance under ${\chi}$-equivalences}: whenever we have construction that applies to (local) integrable systems, the question is to see how much that construction depends on the ${\chi}$-equivalence class of the system. 

A very good example is that of the hessian Lie algebra ${\gerh_{x, \mu}}$ (as a Lie subalgebra of ${\gersp(T_xM, \omega_x)}$), at a fixed point (see Definition \ref{def -hessian -lie -alg}). Here are the details. Assume that ${(\mu, \omega, x)}$ and ${(\mu', \omega', x')}$ are two local integrable systems around fixed points. A ${\chi}$-equivalence between them is, first of all, 
a symplectomorphism ${\Phi: (M, \omega)\rightarrow (M', \omega')}$ taking ${x}$ to ${x'}$. Of course, ${\Phi}$ induces a canonical isomorphism
\[ \Phi_*: \gersp(T_xM, \omega_x)\rightarrow \gersp(T_{x'}M, \omega'_{x'})\]
and the question is whether this takes ${\gerh_{x,\mu}}$ to ${\gerh_{x',\mu'}}$ or, slightly weaker, to a conjugate of ${\gerh_{x',\mu'}}$ . 
In the first case we say that ${\gerh_{x,\mu}}$ is \textbf{invariant w.r.t.\  ${\chi}$-equivalences}, while in the second case that it is \textbf{ ${\chi}$-invariant up to conjugation w.r.t.\  ${\chi}$-equivalences}. 

It is immediate to see that ${\gerh_{x,\mu}}$ is invariant w.r.t.\  functional equivalences; using the interpretation in terms of the normal representation described above, and the fact 
that the normal representation is a natural construction in the world of Lie algebroids, it follows that ${\gerh_{x,\mu}}$ is invariant also w.r.t.\  
orbit equivalence. The tricky part is for weak equivalences, and that sends us right from the start to the first part of Lemma \ref{prop:nondeg_function} (which, in turn, indicates the necessary conditions).

\begin{proposition} 
Consider only local integrable systems ${(\omega, \mu, x)}$ with the property that
the centralizer of ${\gerh_{x, \mu}}$ is contained in ${\gerh_{x, \mu}}$, and  
the fixed point set of the linear action on ${T_xM}$ is trivial (see (\ref{intertwine-fixed-points})). 
For such integrable system,  ${\gerh_{x, \mu}}$ is a invariant up to conjugation w.r.t.\  weak equivalences. 
\end{proposition}
\noindent
\begin{proof}
Hence ${\{\mu_i,\mu'_j\circ\Phi\}=0}$ for all ${i,j}$, with ${\Phi:(M,\omega)\to (M,\omega')}$.
From Lemma \ref{prop:nondeg_function} and the first hypothesis (see also (\ref{intertwine-fixed-points}) we deduce that \[{\dd_x(\mu'\circ\Phi) = 0}.\] Using also the second hypothesis, the Hessians
    \[
        \hess_x(\mu_j'\circ\Phi)
        =
        j_x^2(\mu_j'\circ\Phi)
        =
        \hess_{\Phi(x)}(\mu'_j)
        \circ 
        (\dd_x\Phi)^{\dtensor 2}
    \]
are ${\R}$-linear combinations of the ${\hh_x\mu_i}$. Here we will write
\[
    {\hh_i= \hess_x(\mu_i)},
    \quad 
    {\hh_j'= \hess_{\Phi(x)}(\mu_j')},
    \quad 
    {A= \dd_x\Phi}
\]
We use again the relation between a symmetric bilinear form ${h}$ and the associated element ${A_h\in \gersp(T_xM,\omega_x)}$: 
    $
        h(v) =  \omega_x(A_h v, v).
    $
Since \[{A:(T_xM,\omega_x)\to (T_{\Phi(x)}M,\omega_{\Phi(x)})}\] is a linear symplectic isomorphism, this means that
    \[
        \hh_x(\mu_j'\circ\Phi)(v)
        =
        \hh_j'\circ A^{\dtensor 2}(v)
        =
        \omega_{\Phi(x)}
            \big(
                A_{\hh_j'}A v, A v
            \big)
        =
        \omega_{x}
            \big(
                (A^{ -1} A_{\hh_j'} A) v, v
            \big)
    \]
Hence the ${A^{ -1}A_{\hh_j'}A}$ are an ${\R}$-linear combination of the ${\hh_i}$. We conclude that
    \[
        \gerh_{\Phi(x), \mu'}
        =
        \Ad_A \gerh_{x, \mu}.
        \qedhere
    \]
\end{proof}
\noindent
Of course, the most interesting case when the previous conditions are satisfied is that of nondegenerate fixed points (see also Corollary \ref{cor-no-fixed-from-willianson}). Remembering that the Cartan subalgebras up to conjugation are classified precisely by their Williamson type (see Remark \ref{rk:Williamson}), we deduce:

\begin{local_theorem}
\label{weak preserves nondegenerate}
Local weak equivalence preserves nondegenerate fixed points and also their Williamson type.
\end{local_theorem}

\subsection{Comparison}

\begin{proposition}
Between completely integrable systems, there are the following implications with respect to the equivalence relations:
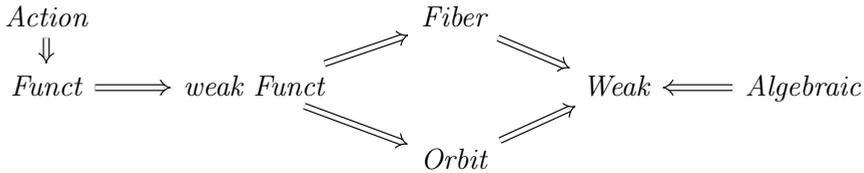
\[ 
\begin{tikzcd}[row sep=small]
    \text{Action}
    \arrow[Rightarrow]{d} & & 
    \text{Fiber} 
    \arrow[Rightarrow]{dr}
    & & \\
    \text{Funct}
    \arrow[Rightarrow]{r}
    & 
    \text{weak Funct}
     \arrow[Rightarrow]{ur} 
    \arrow[Rightarrow]{dr}
    & &
    \text{Weak}
    &    
     \text{Algebraic} 
    \arrow[Rightarrow]{l}
    \\  & &
    \text{Orbit}
    \arrow[Rightarrow]{ur}
    & &
    \end{tikzcd}
\]
\end{proposition}
\noindent

\begin{proof}
It should be clear that the action equivalence implies the functional one and that the weak functional implies the fiber one. That orbit equivalence implies the weak one was noted in Corollary \ref{cor -orbut -equiv}. 

The fact that weak functional equivalence implies orbit equivalence follows from the sheaf characterizations of the two (Propositions \ref{prop:funct -equiv} and \ref{cis:prop:orb_eq}) and the fact that 
the sheaf ${\gerX_\mu^\textsc{orb}}$ can be recovered from ${\textrm{Ham}(\ucalC_{\mu}^{\infty})}$ (see (\ref{eqs:Ham -env -sheaves}) in Proposition \ref{pro -all -many -props -sheaves}).  

Similarly for showing that fiber equivalence implies the weak one: one uses their characterizations in term of sheaves (Proposition \ref{prop -fiber -equiv -sheaves} and \ref{cis:prop:weak_eq}) combined with the
fact that ${\ucalC_\mu}$ (or ${\gerX_{\mu}}$) can be recovered from ${\ucalC_{\mu}^\textsc{fib}}$ (cf. Proposition \ref{pro -all -many -props -sheaves}). 

Finally, algebraic equivalence implies the weak one: assume algebraic equivalence, so that ${A\mu=\mu'\phi}$ with ${\{\mu, A\}_\omega=0}$. Then
    \[
        \{ \mu_i, \mu'_j \phi \}_\omega
        =
        \{ \mu_i, A_j^k \mu_k \}_\omega
        =
        A_j^k \{ \mu_i, \mu_k \}_\omega
        +
        \{ \mu_i, A_j^k \}_\omega \mu_k
        =
        0,
        \quad\forall\; i,j.
    \]
So weak equivalence holds as well. 
\end{proof}
\noindent
The reverse implications do not hold in general.

\begin{example}
The fact that weak equivalence does not imply any of the others follows from the fact that weak equivalences do not preserve singularities, and one can just use the ${\mu}$ and ${\mu'}$ from Example \ref{ex:weak-equiv -strange}; they are not even algebraically equivalent because, 
solving ${a\, \mu = \mu'}$ (and similarly when we have a symplectomorphism) we obtain ${a(x, y)= x}$- which is not invertible at ${x= 0}$. It is actually interesting to think a bit what the general problem is, when trying to achieve algebraic equivalence: while solving ${A\cdot\mu=\mu'}$, the 
discrepancy between the singular points of ${\mu}$ and ${\mu'}$ forces ${A}$ to be a degenerate matrix at the singular points. As we have mentioned, this behavior is radically different at nondegenerate singular points; this suggest that a division algorithm such as the Malgrange preparation theorem might allow us to construct a nondegenerate matrix ${A}$. 
\end{example}
\noindent
\begin{example}
We show that neither functional nor algebraic equivalence imply action equivalence. Consider the elliptic model ${\mu_\text{ell}(x,y)=(x^2+y^2)/2}$ on ${(\R^2,\omega_\text{can})}$, and define 
    \[
        \mu_f 
        \defeq 
        f(\mu_\text{ell})\,\mu_\text{ell}
    \]
for a smooth function ${f\in C^\infty(\R)}$. Note that ${\mu_\text{ell}}$ and ${\mu_f}$ are algebraic equivalent if ${f(z)\neq 0}$ for all ${z\geq 0}$. The function ${f}$ changes the periods of each orbit of ${X_{\mu_\text{ell}}}$.

On the other hand, the map ${\psi: z \mapsto zf(z)}$ is a diffeomorphism if and only if 
    \[
        g(z)
        \defeq
        f(z)
        + 
        z f'(x)
        \neq 
        0
    \]
for all ${z\geq 0}$. Hence, even though ${\mu_\text{ell}}$ and ${\mu_f}$ may not be functionally equivalent for all ${f: \R\to \R^\times}$, there is always a neighborhood of the origin on which they are. 

In any case, the period of the Hamiltonian vector field 
    \[
        X_{\mu_f}
        =
        g(\mu_\text{ell})
        \,
        X_\text{ell}
    \]
at a point ${p\in \R^2}$ is given by ${2\pi/g(\mu_\text{ell}(p))}$, whereas the period of ${X_\text{ell}}$ is always ${2\pi}$. Hence ${\mu_\text{ell}}$ and ${\mu_f}$ are action equivalent if and only if ${g(z)}$ is constant.
\end{example}
\noindent
Finally we have a look at possible relations between functional and strong equivalence.

\begin{example}
Here is an example that shows that algebraic equivalence does not imply fiber or orbit equivalence. Consider the hyperbolic normal form ${\mu_\text{hyp}(x,y) = xy}$ on ${(\R^2,\omega_\text{can})}$ and define
    \[
        \mu_f(x,y) \defeq f(x,y)\, \mu_\text{hyp}(x,y)
    \]
for a nowhere vanishing smooth function ${f\in C^\infty(\R^2)}$ such that ${\set{\mu_\text{hyp},f}=0}$. In other words, look at all integrable systems that are algebraically equivalent to ${\mu_\text{hyp}}$. It follows that ${f}$ is constant along connected components of the level sets of ${\mu_\text{hyp}}$, hence the connected components of ${\mu_\text{hyp}}$ are contained in those of ${\mu_f}$. Yet ${f}$ can be chosen such that ${f(x,y)\neq f( -x, -y)}$, so that ${\mu_f}$ is not the pullback of a function on ${\R}$ along ${\mu_\text{hyp}}$. In that case ${\mu_\text{hyp}}$ and ${\mu_f}$ are not fiber equivalent. 

For orbit equivalence, trying to write ${X_{\mu_f}= C\cdot X_{\mu}}$ one finds \[{C= f+ x \partial_x(f)= f+ y \partial_y(f)},\] but this may fail to be everywhere invertible (e.g.\  for ${f(x, y)= e^{xy}}$).  
\end{example}

\newpage
\section{Deformation cohomology \& Moser path method}
\etocsettocstyle{\subsubsection*{Local Contents}}{}
\etocsettocdepth{2}
\localtableofcontents

\subsection{Deformation cohomology} 
\label{cis:cohomology}

Before going into the details on deformations of integrable systems, and the related topic of normal forms, we will first discuss the relevant algebraic aspects: the deformation complex and its cohomology. 
The deformation complex for integrable systems first appears in the PhD thesis of Miranda \cite{Mir03}, and was later worked out in more detail by her in collaboration with Vũ Ng\d oc \cite{MN05}.

First we recall the \emph{standard} deformation complex from \cite{MN05}, and show the basic fact that the differential squares to zero. After doing these necessary computations we introduce several variations on the standard complex; one for each equivalence relation.

Conceptually, a deformation of an integrable system gives a cocycle in the deformation complex, and the cocycles of two equivalent deformations differ by an exact cocycle. So an equivalence class of deformations gives a class in the deformation cohomology. To make these concepts precise (and true), we have to specify both the equivalence relation and the chain complex. Furthermore, these concepts are only interesting if some (partial) converse holds. In this section we define a chain complex for each equivalence relation and show that we made the `right' choices.
\\

\noindent
Let ${(\omega, \mu)}$ be an integrable system on a manifold ${M}$. Given the Lie algebraic nature present on the deformation complexes, we will identify ${\R^n}$ with the dual of the abelian Lie algebra ${\gerh}$ of dimension ${n}$, writing ${\mu}$ as a moment map 
\[ \mu:M\to \gerh^*.\] 
For ${h\in \gerh}$, we will use the notation 
\[ \mu(h)= \mu_h= \langle \mu, h\rangle \in \calC^{\infty}(M)\]
for the function ${x\mapsto \mu(x)(h)}$. In this way ${\mu}$ becomes a Lie algebra map from ${\gerh}$ to  ${(\calC^{\infty}(M), \{\cdot, \cdot\}_{\omega})}$. Also, with this, the infinitesimal action of ${\gerh}$ on ${M}$ is 
\[ \rho: \gerh\rightarrow \gerX(M), \quad \rho(X)= X^{\omega}_{\mu(h)}= \{\mu(h), \cdot\}_{\omega}.\]
This makes ${\calC^{\infty}(M)}$ into a representation of ${\gerh}$, with action
\[ \gerh\otimes \calC^{\infty}(M)\rightarrow \calC^{\infty}(M), \quad (v, f)\mapsto \gerL_{\rho(h)}(f)= X_{\mu(v)}^{\omega}(f)= \{\mu(v), f\}_{\omega}.\]

The deformation (cochain) complex is built out of two complexes and a map between them:
\begin{myitemize}
\item the first complex is the DeRham complex  ${(\Omega^{\bullet}(M), \dd_\text{dR})}$.
\item the second complex is the Chevalley-Eilenberg complex ${C^{\bullet}(\gerh, \calC^{\infty}(M))}$ of the Lie algebra ${\gerh}$ with coefficients in  ${\calC^{\infty}(M)}$. Hence, in degree ${k}$, 
\[ C^k(\gerh, \calC^{\infty}(M))= \Lambda^k\gerh^*\oplus \calC^{\infty}(M)\]
with its elements interpreted as ${p}$-multilinear antisymmetric maps on ${\gerh}$ with values in ${\calC^{\infty}(M)}$. The Chevalley-Eilenberg  differential, denoted ${\delta}$, is given by
\[ \delta(\beta)(h_0, \ldots, h_k)= \sum_{i= 0}^{k} ( -1)^i\gerL_{\rho(h_i)}(\beta(h_0,\ldots,\widehat{h_i},\ldots, h_k)) .\]
\item the map of complexes
\[ \rho^*: \Omega^{\bullet}(M)\rightarrow C^{\bullet}(\gerh, \calC^{\infty}(M))\]
which, as the notation suggest, comes from composing with the infinitesimal action ${\rho}$:
\[ \rho^*(\alpha)(h_0, h_1, \ldots)= \alpha(\rho(h_0), \rho(h_1), \ldots).\]
\end{myitemize}
By a standard construction in homological algebra, we associate to this datum a new complex ${\calC^\bullet(\omega,\mu)}$, called the mapping cone of ${\rho^*}$, as described in the following diagram:
\[ 
\begin{tikzcd}[row sep=small]
\ldots \arrow{r}{\dd_{\omega,\mu}} & 
    \calC^{k -1}(\omega,\mu) \arrow[equal]{d}\arrow{r}{\dd_{\omega,\mu}}&
    \calC^k(\omega,\mu) \arrow[equal]{d}\arrow{r}{\dd_{\omega,\mu}}&
    \calC^{k+1}(\omega,\mu)\arrow[equal]{d} \arrow{r}{\dd_{\omega,\mu}}
    & \ldots\\
\ldots \arrow{r}{ -\dd_\text{dR}}\arrow{ddr}{\rho^*} & 
    \Omega^{k}(M)
    \arrow{r}{ -\dd_\text{dR}}
    \arrow{ddr}{\rho^*}
   &
    \Omega^{k+1}(M)
    \arrow{r}{ -\dd_\text{dR}}
    \arrow{ddr}{\rho^*}
    &
     \Omega^{k+2}(M)\arrow{r}{ -\dd_\text{dR}} \arrow{ddr}{\rho^*}  
& \ldots\\
&    \bigoplus & \bigoplus & \bigoplus & \\
\ldots \arrow{r}{\delta} & 
    C^{k -1}(\gerh, \calC^{\infty})
    \arrow{r}{\delta}
    &
    C^{k}(\gerh, \calC^{\infty})
    \arrow{r}{\delta}
     &
     C^{k+1}(\gerh, \calC^{\infty})\arrow{r}{\delta}
    & \ldots
     \end{tikzcd}
\]
Remark that the mapping cone only depends on the action ${\rho^*}$.
In degree ${k}$,
\[ \calC^{k}(\omega,\mu)= \Omega^{k+1}(M)\oplus C^{k}(\gerh, \calC^{\infty}),\]
with differential ${\dd_{\omega,\mu}}$ given by
\[\dd_{\omega,\mu}(\alpha, \beta)= ( -\dd_\text{dR}(\alpha), \rho^*(\alpha)+ \delta(\beta)).\]

\begin{definition}\label{def:deformation -complex}
The \textbf{standard deformation complex}\index{deformation complex!standard} of the integrable system ${(\omega, \mu)}$ is the complex
${(\calC^\bullet(\omega,\mu), \dd_{\omega,\mu})}$ described above. Its cohomology is denoted by ${H^k(\omega,\mu)}$.  
\end{definition}
\noindent
To handle the explicit formulas, we will denote the elements of ${\gerh^k}$ by a single letter ${h}$ and we will write ${h_i}$ for its ${i}$-th component.  Accordingly, ${\rho(h)}$ is the ${k}$-multivector field obtained by applying ${\rho}$ to each component. 
Also, for a ${(k+1)}$-tuple ${h \in \gerh^{k+1}}$ we will use the notation 
\[
        h^{(i)}
        \defeq
        (h_0,\ldots,\hat h_i,\ldots,h_k). 
\]
In this way the Chevalley-Eilenberg differential and the chain map ${\rho^*}$ are written:
\[ \delta(\beta)(h)= \sum_{i= 0}^{k} \{\mu(h_i), \beta(h^{(i)})\}_{\omega} ,\quad  \rho^*(\alpha)(h)= \alpha(\rho(h))= \alpha(X^{\omega}_{\rho(h)}).\]

If we write 
\[
        \dd_{\omega,\mu}
        =
        ( -\dd_\text{dR}) \oplus \partial_{\omega,\mu},
 \]
then we obtain the following explicit formula for the second component:
    \[
        \partial_{\omega,\mu}(\alpha,\beta)h
        \defeq
        \sum\nolimits_{i=0}^k
        ( -1)^i
        \big\{ 
            \mu(h_i), 
            \beta (
                h^{(i)}
            )
        \big\}_\omega
        +
        \alpha \big( 
            X^\omega_{\mu(h_0)}, 
            \ldots,
            X^\omega_{\mu(h_k)} 
        \big).
    \]

A standard remark from homological algebra is that one has a short exact sequence of complexes
\[
        0
        \to
        \calC^\bullet(\gerh, \calC^{\infty})
        \to
        \calC^\bullet(\omega,\mu)
        \to
        \Omega^{\bullet+1}(M)
        \to
        0
    \]
which gives rise to a long exact sequence in cohomology:
\[ \ldots \to 
   H^{\bullet}(M) 
   \to 
   H^{\bullet}(\gerh, \calC^{\infty})
   \to
   H^{\bullet}(\omega, \mu)
   \to 
   H^{\bullet+1}(M)
   \to \ldots 
\]
In particular:

\begin{corollary}
\label{cor -def -cplx -contr -base} 
If ${M}$ is contractible then ${H^{\bullet}(\omega, \mu)\cong H^{\bullet}(\gerh, \calC^{\infty})}$ in all positive degrees.
\end{corollary}

This is relevant when considering a local deformations and a local version of this complex - when one can restrict to contractible neighborhoods. For this reason, 
${(C^{\bullet}(\gerh, \calC^{\infty}), \delta)}$ is sometimes called the \textbf{reduced deformation complex}\index{deformation complex!reduced} \cite{MN05}.

Next we need to adapt the deformation complex to the various notion of equivalences. The reason is the following: while, as we shall soon see, a smooth deformation of an integrable system defines a degree one class in the deformation cohomology,
we would like to make sure that a `trivial deformation' gives rise to the zero class (with the hope for some converse). The problem is inside the quotes: triviality of a deformation will depend on the equivalence we use. While this discussion brings us to the degree zero part of the deformation, what we have to do is to adapt this degree zero part to each case. Hence one has to look at the lower part of ${C^{\bullet}(\omega, \mu)}$:
\[
\begin{tikzcd}[row sep=small]
    \calC^{0}(\omega,\mu) \arrow[equal]{d}\arrow{r}{\dd_{\omega,\mu}}&
    \calC^1(\omega,\mu) \arrow[equal]{d}\arrow{r}{\dd_{\omega,\mu}}
    & \ldots\\
    \Omega^{1}(M)
    \arrow{r}{ -\dd_\text{dR}}
    \arrow{ddr}{\rho^*}
   &
    \Omega^{2}(M)
    \arrow{r}{ -\dd_\text{dR}}
    \arrow{ddr}{\rho^*}
    & \ldots\\
\bigoplus & \bigoplus &  \\
    \boldsymbol{C^{0}(\gerh, \calC^{\infty})}
    \arrow{r}{\boldsymbol{\delta^0}}
    &
    C^{1}(\gerh, \calC^{\infty})
    \arrow{r}{\delta}
     & \ldots
     \end{tikzcd}
\]
What we will have to change/adapt is the lower corner that is emphasized in the previous diagram - namely ${C^0(\gerh, \calC^{\infty})= \calC^{\infty}(M)}$- hence also the differential going out of it, 
\begin{equation}\label{eq:addapt -def -complex}
\delta^{0}: \calC^{\infty}(M)\rightarrow \gerh^*\otimes \calC^{\infty}(M).
\end{equation}

\begin{definition}
\label{definitions -adapted -def -cplxs} 
We define the deformation complex of ${(\omega, \mu)}$ adapted to action, functional, weak and algebraic equivalence, respectively, by changing 
the part (\ref{eq:addapt -def -complex}) of the standard deformation complex as follows
\begin{myitemize}
\item 
 For functional equivalence, (\ref{eq:addapt -def -complex}) is replaced by 
\[  \delta^{0}_\text{fct}: \calC^\infty(\gerh^*,\gerh^*)\rightarrow \gerh^*\otimes \calC^{\infty}(M), \quad \delta^{0}_\text{fct}(F)= F\circ \mu.\]
\item 
For action equivalence, (\ref{eq:addapt -def -complex}) is replaced by 
\[  \delta^{0}_\text{act}: \gergl(\gerh^*)\oplus \gerh^*\rightarrow \gerh^*\otimes \calC^{\infty}(M), \quad \delta^{0}_\text{act}(A, \gamma)= A\circ \mu+ \gamma .\]
(think of the domain as a subspace of ${\calC^\infty(\gerh^*,\gerh^*)}$ consisting of affine functions). 
\item For weak equivalence, , (\ref{eq:addapt -def -complex}) is replaced by the inclusion
\[  \delta^{0}_\text{wk}: \gerh^* \otimes \ucalC_{\mu}(M)\hookrightarrow \gerh^*\otimes \calC^{\infty}(M) .\]
\item
For algebraic equivalence, (\ref{eq:addapt -def -complex}) is replaced by 
\[  \delta^{0}_\text{alg}: \gergl(\gerh^*) \otimes \ucalC_{\mu}(M)\rightarrow \gerh^*\otimes \calC^{\infty}(M), \quad \delta^{0}_\text{alg}(B)= B\cdot \mu .\]
(where ${(B\cdot \mu)(x)= B(x)(\mu(x))}$- see  below for notations). 
\item
For orbit equivalence, (\ref{eq:addapt -def -complex}) is replaced by 
\[  \delta^{0}_\text{orb}: \gergl(\gerh^*) \otimes \ucalC^\textsc{orb}_{\mu}(M)\rightarrow \gerh^*\otimes \calC^{\infty}(M), \quad \delta^{0}_\text{orb}(B)= B\cdot \mu .\]
\end{myitemize}
For each of the cases 
$
        \chi 
        \in
        \set{ \text{act}, \text{fct}, \text{orb}, \text{wk}, \text{alg} },
    $
we will denote by ${H^{\bullet}_{\chi}(\omega, \mu)}$ the resulting cohomology. 
\end{definition}
\noindent
Recall here that
\begin{myitemize}
\item ${\ucalC_{\mu}(M)}$ consists of the functions that Poisson commute with all the ${\mu_i}$'s (the global sections of the
central sheaf from Definition \ref{def -central -sheaf}) 
\item ${\ucalC^\textsc{orb}_{\mu}(M)}$ consists of the functions ${f}$ with the property that 
${X_{f}}$ is in the ${\ucalC^{\infty}(M)}$-span of the vector fields ${X_{\mu_i}}$  (the global sections of the
central sheaf from Definition \ref{def -C -orbit--sheaf}). 
\end{myitemize}
Also, in the previous formulas, we interpret the 
elements of the tensor product ${V\otimes \calC^{\infty}(M)}$ as smooth ${V}$-valued functions on ${M}$, for any finite dimensional vector space ${V}$. 
Therefore, also the elements of ${\gerh^* \otimes \ucalC_{\mu}(M)}$ are viewed as functions from ${M}$ to ${\gerh^*}$- the ones 
whose components are in ${\ucalC_{\mu}(M)}$. The notation ${B\cdot \mu}$ suggest to think of the matrix multiplication (with entries smooth functions).

\subsection{The Moser argument (globally)}

In this subsection we focus on the geometric aspects of the deformation theory of integrable systems and we exploit the `Moser path method' or `Moser trick' for integrable systems.
The use of the Moser path method in the context of integrable systems was first exploited by Miranda in her PhD thesis \cite{Mir03}. 

After defining deformations of integrable systems we will prove in detail that a deformation gives rise to a class in the first deformation cohomology group. 
This follows from correctly interpreting the time-derivative of the deformation.
Then we define what it means for a deformation to be (smoothly) trivial, and show that in this case its cohomology class vanishes. 
Conversely, if its cohomology class vanishes for \emph{all} time, then the deformation is smoothly trivial. 
Note that it is not sufficient that the cohomology class, or seemingly even the entire cohomology group, vanishes for only one $t$.
But in the next chapter we show that for normal forms, locally near nondegenerate fixed points, and for \emph{one particular} choice of deformation, it does suffice to prove vanishing of the class for any one small value of $t$.

Many aspects of this story are standard, and have been applied to a variety of other settings, but we think it might interest the reader to find all details in a self-contained story.

\begin{definition} 
A \textbf{deformation}\index{deformation} of a completely integrable system ${(\omega,\mu)}$ on a manifold ${M}$ is a family 
${(\omega_t, \mu_t)}$ of \emph{completely} integrable system on ${M}$ smoothly parametrized by ${t}$ belonging to an open interval 
of type ${( -\epsilon, \epsilon)}$ such that, at ${t= 0}$, ${(\omega^0,\mu^0)=(\omega,\mu)}$.
\end{definition}
\noindent
Of course, smoothly parametrized by ${t}$ means that 
    \[
        \omega
        :
        ( -\epsilon,\epsilon) \times M
        \to
        \wedge^2 T^*M,
        \quad \text{and}\quad 
        \mu
        :
        ( -\epsilon,\epsilon) \times M
        \to
        \R^n,
    \]

\begin{remark}
We want to stress that we require that ${\mu^t}$ is complete for all small ${t}$. This requirement is essential when dealing with equivalences, and it does not follow automatically. Indeed, consider the smooth function ${\mu^\epsilon:\R^2\to\R}$ which vanishes on the set of elements with at most distance ${\epsilon}$ from the union of the axes and is otherwise given by the function 
    \[
        \mu^\epsilon(x,y)
        =
        e^{ -1/(x -\epsilon)^2} e^{ -1/(y -\epsilon)^2}.
    \]   
In this family only ${\mu^0}$ is complete, but the other ${\mu^\epsilon}$ come arbitrarily close to ${\mu^0}$.

Completeness is also not a closed condition. This makes the deformation theory of integrable systems somewhat awkward: it does not completely fit in the framework of `objects satisfying the Maurer-Cartan equation.'
\end{remark}
\noindent
Next we define trivial deformations. As we have already mentioned, this notion depends on a choice of equivalence relation ${\chi}$. In short, a deformation ${(\omega^t,\mu^t)}$ is ${\chi}$-trivial if there is an ${\epsilon > 0}$ such that
${(\omega^t, \mu^t)}$ is ${\chi}$-equivalent to ${(\omega^0, \mu^0)}$, smoothly with respect to ${t}$. Here, smoothness w.r.t.\  ${t}$ means that the data entering ${\chi}$-equivalences (one for each ${t}$) can be chosen to be smooth in ${t}$. Let us spell this out in more detail; for simplicity, we will consider here only
action, functional and weak equivalence. 

\begin{definition}\label{def:chi -trivial -defs}
Let 
${\chi \in \set{ \text{act}, \text{orb}, \text{fib},  \text{fct}, \text{wk},  \text{lag}}}$. A deformation ${(\omega^t,\mu^t)}$ of a completely integrable system on ${M}$ is said to be a \textbf{${\chi}$-trivial deformation}\index{deformation!trivial} if there is an ${\epsilon > 0}$
and a family of ${\chi}$-equivalences between ${(\omega^t, \mu^t)}$ and ${(\omega^0, \mu^0)}$, smoothly parametrized by ${t\in ( -\epsilon, \epsilon)}$.
\end{definition}
 \noindent        

The meaning of ${\chi}$-equivalences smoothly parametrized by ${t}$ is rather obvious. For completeness, here are the details. In all cases of ${\chi}$, we must have a family of symplectomorphisms ${\Phi^t:(M,\omega^0)\to (M,\omega^t)}$; the extra data depends on each case. For \textbf{fiber triviality} and \textbf{weak triviality}, there is no extra data, just the requirement that ${\Phi^t}$ takes the fibers of ${\mu^0}$ to the ones of ${\mu^t}$, or that   ${\{\mu_i^0, \mu_j^t \circ \Phi^t\}= 0}$ for all ${i, j, t}$, respectively. 
For \textbf{action triviality} one requires smooth maps ${A: ( -\epsilon,\epsilon) \to \GL(\R^n)}$ and ${\gamma: ( -\epsilon,\epsilon) \to \R^n}$ such that 
    \[
        A^t \cdot \mu^0 + \gamma^t = \mu^t \circ \Phi^t,
        \quad \forall\;\; t\in ( -\epsilon,\epsilon).
    \]
Similarly for \textbf{orbit triviality} and \textbf{algebraic triviality}. For \textbf{weak functional triviality} one requires a family of diffeomorphisms ${\phi_t : \mu_0(M)\rightarrow \mu_t(M)}$ as in Definition \ref{def -ftc -equivalence} which is required to be smooth in the sense that 
the resulting map ${( -\epsilon,\epsilon)\times \mu_0(M)\rightarrow \mathbb{R}^n}$ is smooth (see the comments after the mentioned definition). \textbf{Functional triviality} should be clear now. 

\begin{remark} 
The previous definition can be slightly generalized, 
so that one talks about ${\chi}$-equivalences between two deformations ${(\omega^t, \mu^t)}$ and ${(\omega'^{,t}, \mu'^{,t})}$ of the same ${(\omega, \mu)}$. 
With this, ${(\omega^t, \mu^t)}$ is ${\chi}$-trivial if and only if it is ${\chi}$-equivalent to the trivial (constant) deformation. 

In such definitions, while the deformations ${(\omega^t, \mu^t)}$ and ${(\omega'^{,t}, \mu'^{,t})}$ coincide at ${t= 0}$ (they are both ${(\omega, \mu)}$), one sometimes requires that the isomorphism between ${(\omega^t, \mu^t)}$ and ${(\omega'^{,t}, \mu'^{,t})}$ is ${\textrm{Id}}$ at ${t= 0}$. In particular, for ${\chi}$-triviality, one sometimes requires that ${\Phi^0= \textrm{Id}}$ in all case, while for action triviality one also requires ${A^0= \textrm{Id}}$, ${\gamma^0= 0}$, etc. 

The reason for this requirement is the following. Starting with a deformation ${(\omega^t, \mu^t)}$ and interpreting it as an element in the deformation complex, we will see in the next proposition that its variation at ${t= 0}$,
\[ \left.\frac{d}{d t}\right|_{t=0} (\omega^t, \mu^t) \in \calC^{1}(\omega, \mu)\]
induces an element in the cohomology ${H^{1}_{\chi}(\omega, \mu)}$ (and similarly for other ${t}$'s). If we want that two ${\chi}$-equivalent deformations induce the same class, it is clear that we have to insist on the previous condition (otherwise the class should be invariant under conjugations induced by self ${\chi}$-equivalences). For simplicity, the next proposition avoids this discussion by looking at the vanishing of the cohomology class (property which is invariant under conjugation). The more general properties of these classes that we have just pointed out have a completely similar proof.
\end{remark}
\noindent
\begin{proposition}\label{prop_families}
Let ${(\omega^t,\mu^t)}$ be a deformation of a completely integrable system ${(\omega, \mu)}$ on a compact manifold ${M}$, and let ${\chi \in \set{ \text{act}, \text{orb}, \text{fib},  \text{fct}, \text{wk},  \text{alg}}}$.

\begin{itemize}
\item
For all ${t\in ( -\epsilon,\epsilon)}$, 
\[
        c^t
        \defeq
       \frac{d}{dt} 
        (\omega^t, \mu^t)
        = ( \frac{d}{dt} \omega^t,  \frac{d}{dt} \mu^t)
        \quad\in \calC^1(\omega^t,\mu^t).
\]
is a closed cochain. 
\item If the deformation is ${\chi}$-trivial, then the induced cohomology classes ${[c^t]\in H^{1}_{\chi}(\omega^t, \mu^t)}$ vanish for small ${t}$.   
\item
Conversely: if ${[c^t]\in H^{1}_{\chi}(\omega^t, \mu^t)}$ is zero smoothly in ${t}$, i.e.\  if the equations
\[ c^t=\dd_{\omega^t,\mu^t}(b^t)\]
 have solutions ${b^t\in \calC^0_\chi(\omega^t, \mu^t)}$ varying smoothly in  ${t}$, then the deformation is ${\chi}$-trivial.
\end{itemize}
\end{proposition}
\noindent
The smooth variation in ${t}$, the smoothness as a map 
    \[
        {b: ( -\epsilon,\epsilon)\to \calC_\chi^0(\omega^t, \mu^t)},
    \]
makes sense either because the last space does not depend on ${(\omega^t,\mu^t)}$
(e.g.\  for ${\text{act}}$) or is inside a space that does not depend on ${t}$ (e.g.\  for ${\calC^{0}_{\text{wk}}\subset \gerh^*\otimes \calC^{\infty}(M)}$). 

Note that, in the proof below, we will see that the vanishing of the class in ${H^{1}_{\text{fct},}(\omega^t, \mu^t)}$, one only needs that the deformation is weak functionally trivial.

\begin{proof}
We have see that ${c^t}$ is closed in ${\calC^1(\omega^t,\mu^t)}$. Regarding its first component, observe that ${\dd\omega^t=0}$ implies that indeed
    \[
        \dd(\partial_t\omega^t)=0.
    \]
For its second we component derive the integrability condition for ${(\omega^t,\mu^t)}$ with respect to time. Observe that the Poisson bracket 
    $
        \set{\cdot,\cdot}_{\omega}
    $ 
does not depend linearly on the symplectic structure, but it does depend linearly on the Poisson structure ${\pi=\omega^{ -1}}$. It is an instance where the language of Poisson geometry is more convenient than symplectic geometry. Let
    \[
        \pi^t \defeq (\omega^t)^{ -1}
    \] 
be the deformation of Poisson structures induced by ${\omega^t}$.
The integrability condition ${\{ \mu^t_i, \mu^t_j \}_{\pi^t}=0}$ gives us
    \[
        \big\{ 
            \partial_t\mu^t_i, \mu^t_j 
        \big\}_{\pi^t}
        +
        \big\{ 
            \mu^t_i, \partial_t\mu^t_j 
        \big\}_{\pi^t}
        +
        \big\{ 
            \mu^t_i, \mu^t_j 
        \big\}_{\partial_t \pi^t}
        =
        \partial_t \big(
        \big\{
            \mu_i^t, \mu_j^t
        \big\}_{\pi^t}
        \big)
        =
        0
    \]
for all ${i,j}$. This starts to look like the correct equation, but we still have to understand the relation between ${\partial_t\pi^t}$ and ${\partial_t\omega^t}$. Any ${2}$-form ${\omega}$ defines an endomorphism 
    \[
        \omega^\sharp: TM\to T^*M,
    \]
and likewise any bivector ${\pi}$ defines an endomorphism 
    \[
        \pi^\flat:T^*M\to TM.
    \]
Note that the assignments ${\omega\mapsto \omega^\sharp}$ and ${\pi\mapsto \pi^\flat}$ are linear and injective, so we may as well compute the relation between ${\partial_t(\pi^{t})^{\flat}}$ and ${\partial_t(\omega^{t})^{\sharp}}$ instead. Since ${(\pi^t)^\flat}$ is by definition the inverse of ${(\omega^t)^\sharp}$, we find
    \[
        \big( 
            \partial_t (\omega^{t})^{\sharp} 
        \big)
        (\pi^{t})^{\flat}
        +
        (\omega^{t})^{\sharp}
        \big( 
            \partial_t (\pi^{t})^{\flat} 
        \big)
        =
        \partial_t \big(
            (\omega^t)^\sharp
            (\pi^t)^\flat
        \big)
        =
        \partial_t(\id)
        = 0.
    \]
From this we derive the formula
    \[
        \partial_t (\pi^{t})^{\flat}
        =
        - (\pi^{t})^{\flat}
        \big(
            \partial_t (\omega^{t})^{\sharp} 
        \big)
        (\pi^{t})^{\flat}.
    \]
The left -hand -side corresponds to the bivector ${\partial_t\pi^t}$, whereas the right -hand -side (which we will denote by ${\varpi^\flat}$ for the moment) corresponds to the bivector that maps a pair of smooth functions ${f}$ and ${g}$ to
    \begin{align*}
        \varpi( \dd f, \dd g )
        & =
        \dd f 
        \big(
            \varpi^\flat ( \dd g )
        \big)
        \\
        & =
        - \dd f 
        \Big(
            \pi^{t\flat}
            \big(
                \partial_t \omega^{t\sharp} 
            \big)
            \pi^{t\flat}
            (\dd g)
        \Big)
        \\
        & =
        \pi^{t\flat}(\dd f)
        \big(
            \partial_t \omega^{t\sharp}
            \pi^{t\flat}(\dd g)
        \big)
        \\
        & =
        \partial_t\omega^{t}
        \big(
            \pi^{t\flat}(\dd f), \pi^{t\flat}(\dd g)
        \big)
        \\
        & =
        \partial_t \omega^t
        \big(
            X^{\omega^t}_f, X^{\omega^t}_g
        \big),
    \end{align*}
where for the third equality we us the antisymmetry of ${\pi^t}$. We conclude that
    \[
        \dd_{\omega^t,\mu^t} c^t
        =
        \big\{ 
            \mu^t_i, \partial_t\mu^t_j 
        \big\}_{\omega^t}
        -
        \big\{ 
            \mu^t_j, \partial_t\mu^t_i 
        \big\}_{\omega^t}
        +
        \partial_t \omega^t(
            X^{\omega^t}_{\mu^t_i},
            X^{\omega^t}_{\mu^t_j}
        )
        = 0.
    \]
This shows that ${c^t = (\partial_t\omega^t,\partial_t\mu^t)}$ is closed in ${\calC^1(\omega^t,\mu^t)}$.

For the second statement we have to discuss each equivalence relation separately. The equivalence relations have some aspects in common however: a smooth family of symplectomorphisms ${\Phi^t: (M,\omega^0)\to (M,\omega^t)}$ and expressions of the form ${\mu^t\circ \Phi^t}$. 
Let ${X^t \in \gerX(M)}$ be the infinitesimal generator of ${\Phi^t}$ and define
    \[
        \alpha^t \defeq \iota_{X^t}(\omega^t).
    \]  
Then deriving ${\omega^0=(\Phi^t)^*\omega^t}$ with respect to time gives
    \[
        0
        =
        \partial_t 
        \big(
            (\Phi^t)^*\omega^t
        \big)
        =
        (\Phi^t)^* \gerL_{X^t}(\omega^t)
        +
        (\Phi^t)^* \partial_t\omega^t
        =
        (\Phi^t)^* \big(
            \dd \iota_{X^t}(\omega^t) + \partial_t\omega^t
        \big).
    \]
This shows that ${\partial_t\omega^t = -\dd \alpha^t}$. Before we get to the case -specific part of the proof, note that the term ${\mu^t \circ \Phi^t}$ appears in each of equivalence relations. Deriving it with respect to time gives 
    \begin{align*}
        \partial_t
        \big(
            \mu^t \circ \Phi^t
        \big)
        & =
        \big(
            \partial_t \mu^t
            +
            \dd \mu^t (X^t)
        \big) \circ \Phi^t
        \\
        & =
        \big(
            \partial_t \mu^t
            +
            \dd\mu^t ( (\pi^t)^\flat  \alpha^t )
        \big) \circ \Phi^t
        \\
        & =
        \big(
            \partial_t \mu^t
            -
            \alpha^t(X_{\mu^t}^{\omega^t})
        \big) \circ \Phi^t.
    \end{align*}

The various ${\chi}$ are treated roughly the same, but each has its own quirks. We give here the details for ${\chi \in \set{ \text{act}, \text{fct}, \text{wk}}}$, which cover all the types of arguments used.

\begin{itemize}

\item
If the deformation is action trivial, so that
    $
        A^t\cdot \mu^0 + \gamma^t = \mu^t\circ \Phi^t,
    $
one finds that
    \begin{align*}
        \partial_t\mu^t \circ \Phi^t
        & =
        \partial_tA^t\cdot \mu^0 
        + 
        \partial_t\gamma^t 
        + 
        \alpha^t(X_{\mu^t}^{\omega^t}) \circ \Phi^t
        \\
        & =
        \partial_t A^t \cdot
        (A^t)^{ -1} \cdot
        \big(
            \mu^t \circ \Phi^t - \gamma^t
        \big)
        + 
        \partial_t\gamma^t 
        + 
        \alpha^t(X_{\mu^t}^{\omega^t}) \circ \Phi^t
        \\
        & =
        \big(
            B^t \cdot \mu^t
            +
            \alpha^t( X_{\mu^t}^{\omega^t} )
            +
            \delta^t
        \big) \circ \Phi^t,
    \end{align*}
where we choose
    \[
        B^t 
        \defeq
        \partial_tA^t \cdot (A^t)^{ -1},
        \quad
        \delta^t
        \defeq
        \partial_t \gamma^t
        -
        B^t (\gamma^t).
    \]
Hence ${(\alpha^t, B^t, \delta^t)}$ is a primitive of ${(\partial_t\omega^t,\partial_t\mu^t)}$ in ${\calC^\bullet_\text{act}(\omega^t,\mu^t)}$. Conversely, if the ${\alpha^t}$, ${B^t}$, and ${\delta^t}$ are given, solve the ordinary differential equations
    \[
        \partial_t\Phi^t = X^t \circ \Phi^t,
        \quad
        \partial_tA^t = B^t \cdot A^t,
        \quad
        \partial_t\gamma^t = \delta^t + B^t(\gamma^t),
    \]
with ${\Phi^0=\id}$, ${A^0=\id}$, ${\gamma^0=0}$. Then the above computations tell us that
    \[
        \partial_t
        \big(
            \mu^t \circ \Phi^t
            -
            A^t \cdot \mu^0
            -
            \gamma^t
        \big)
        = 0.
    \]
Since ${A^0\cdot \mu^0 + \gamma^0 = \mu^0\circ \Phi^0}$, this shows that the deformation is action trivial.

\item
If the deformation is functionally trivial, so that
    $
        \phi^t \circ \mu^0
        =
        \mu^t \circ \Phi^t.
    $
    While each ${\phi^t: \mu^0(M)\rightarrow \mathbb{R}^n}$ is smooth in the sense we discussed above, since ${M}$ is compact ${\mu^0(M)}$ will be closed in ${\R^n}$ hence each
    ${\phi^t}$ has a smooth extension ${\R^n}$. After eventually making ${\epsilon}$ smaller, one can arrange that these extensions vary themselves smoothly also w.r.t.\  ${t}$ (in particular, they have a common domain). For that one uses the  
    smoothness of the family ${\{\phi^t\}}$ and we apply the same extension argument to the entire family viewed as a smooth map defined on ${I\times \mu^0(M)}$- a closed subset of ${\mathbb{R}^n}$,
    where ${I=[ -\epsilon', \epsilon']}$ is a closed interval inside ${( -\epsilon, \epsilon)}$. From now on we use such a smooth family of extensions, still denoted by ${\{\phi^t\}}$. Denoting by ${Y^t}$ 
    the resulting tame -dependent vector field (`infinitesimal generator'), the previous equations give us:
    One finds that
    \begin{align*}
        \partial_t\mu^t \circ \Phi^t
        & =
        Y^t\circ (\phi^t \circ \mu^0)
        +
        \alpha^t( X_{\mu^t}^{\omega^t} )
        \circ \Phi^t
        \\
        & =
        \big(
            Y^t \circ \mu^t 
            + 
            \alpha^t ( X_{\mu^t}^{\omega^t} )
        \big) \circ \Phi^t, 
    \end{align*}
Hence ${(\alpha^t, Y^t)}$ is a primitive of ${(\partial_t\omega^t,\mu^t)}$ in ${\calC^\bullet_\text{fct}(\omega^t,\mu^t)}$. Conversely, if ${\alpha^t}$ and ${Y^t}$ are given, solve the ordinary differential equations
    \[
        \partial_t\Phi^t = X^t\circ \Phi^t,
        \quad
        \partial_t\phi^t = Y^t \circ \phi^t,
    \]
with ${\Phi^0=\id}$ and ${\phi^0=\id}$. We assumed that ${M}$ is compact, hence so is the image ${\mu^0(M)}$. This means that one can choose a small neighborhood ${U\sub \gerh^*}$ of ${\mu^0(M)}$ and a small ${\epsilon >0}$ such that the flow ${\phi^t}$ exists on ${U}$ for all ${t\in ( -\epsilon,\epsilon)}$. The remainder of the argument is as above.

\item
If the deformation is weakly trivial, so that 
    $
        \{
            \mu^0, \mu^t\circ \Phi^t 
        \}_{\omega^0}
        = 0,
    $
one finds that
    \begin{align*}
        0
        & =
        \big\{
            \mu^0, 
            \big(
                \partial_t \mu^t
                -
                \alpha^t( X_{\mu^t}^{\omega^t})
            \big)
            \circ \Phi^t
        \big\}_{\omega^0}
        \\
        & =
        \big\{
            \mu^0 \circ (\Phi^t)^{ -1},
            \partial_t\mu^t
            -
            \alpha^t( X_{\mu^t}^{\omega^t} )
        \big\}_{\omega^t} \circ \Phi^t.
    \end{align*}
Note that we also have 
    $
        \{
            \mu^0 \circ (\Phi^t)^{ -1},
            \mu^t
        \}_{\omega^t} = 0.
    $
The pair ${(\omega^t,\mu^0\circ (\Phi^t)^{ -1})}$ is a completely integrable system for all time ${t}$, so from the \nameref{cis:lem:main} on page~\pageref{cis:lem:main} we conclude that
    \[
        \big\{
            \mu^t,
            \partial_t\mu^t 
            - 
            \alpha^t( X_{\mu^t}^{\omega^t} )
        \big\}_{\omega^t} 
        = 
        0.
    \]
Hence ${(\alpha^t, \partial_t\mu^t + \alpha^t( X_{\mu^t}^{\omega^t} ) )}$ is a primitive of ${(\partial_t\omega^t, \partial_t \mu^t )}$ in ${\calC^\bullet_\text{wk}(\omega^t,\mu^t)}$. Conversely, if ${\alpha^t}$ is given, solve the ordinary differential equation
    \[
        \partial_t\Phi^t = X^t \circ \Phi^t,
    \]
with ${\Phi^0 = \id}$. Retracing the above computations tells us that 
    $
        \partial_t 
        \{ \mu^0, \mu^t \circ \Phi^t \}_{\omega^0}
        = 0. 
    $
Combined with ${\{ \mu^0, \mu^0 \circ \Phi^0 \}_{\omega^0} = 0}$, we conclude that the deformation is weakly trivial. Note that in this argument we have only used the completeness of ${\mu^0}$.
\qedhere

\end{itemize}
\end{proof}

\subsection{The local story}

We briefly look at how to adapt the above deformation theory to the local context, i.e.\  for germs of integrable ${(\omega, \mu)}$ systems around a point $x$. This local version of the theory appears when studying normal forms of integrable systems around singular points (and this will be explained in the next section). 
We start by discussing the local version of the deformation complex. We will also be using the terminology from subsection \ref{The local version of equivalences; invariance under equivalences}: the germ of an integrable system ${(\omega, \mu)}$ around a point ${x}$ will be called \textbf{local integrable system near ${x}$} and will be denoted ${(\omega, \mu, x)}$.

\begin{definition}
Let ${(\omega,\mu, x)}$ be a local integrable system. The \textbf{local}\index{deformation complex!local} standard deformation complex of ${(\omega,\mu, x)}$, denoted ${\calC^\bullet(\omega, \mu, x)}$ is the germ at ${x}$ of the standard deformation.
\end{definition}
\noindent
Therefore, denoting by ${\ucalC^\infty_{x}}$ the spaces of germs at ${x}$ of smooth functions (the stalk at ${x}$ of the sheaf of smooth function) defined on the domain of ${(\omega, \mu)}$, and similarly by ${\Omega^{\bullet}_x}$ for germs of differential forms, we have
\[
        \calC^\bullet(\omega, \mu, x)
        \defeq
        \Omega^{k+1}_{x}
        \oplus
        \big(
            \wedge^k\gerh^*
            \otimes
            \ucalC^\infty_{x}
        \big),
\]
with the differential induced from the one on ${\calC^{\bullet}(\omega, \mu)}$.  

For the various ${\chi}$-version of the deformation complex, one localizes Definition \ref{definitions -adapted -def -cplxs}. More precisely, one proceeds as there but considering now
the following ${\chi}$-modifications (in degree zero) of the Chevalley-Eilenberg complex of ${\gerh}$ with coefficients in ${\ucalC_{x}^{\infty}}$ (for the notations, see also the comments after that definition)
\begin{myitemize}
\item 
 For local functional equivalence we use
\[  \delta^{0}_\text{fct}: \calC^\infty_{\mu(x)}(\gerh^*,\gerh^*)\rightarrow \gerh^*\otimes \ucalC_x, \quad \delta^{0}_\text{fct}(F)= F\circ \mu .\]
(defined on the space of germs around ${\mu(x)\in \gerh^*}$ of smooth ${\gerh^*}$-valued functions). 
\item 
For local action equivalence use:
\[  \delta^{0}_\text{act}: \gergl(\gerh^*)\oplus \gerh^*\rightarrow \gerh^*\otimes \calC^{\infty}_x, \quad \delta^{0}_\text{act}(A, \gamma)= A\circ \mu+ \gamma .\]
\item For local weak equivalence use:
\[  \delta^{0}_\text{wk}: \gerh^* \otimes \ucalC_{\mu, x}\hookrightarrow \gerh^*\otimes \calC^{\infty}_{x} .\]
\item
For local algebraic equivalence use
\[  \delta^{0}_\text{alg}: \gergl(\gerh^*) \otimes \ucalC_{\mu, x}\rightarrow \gerh^*\otimes \calC^{\infty}_{x}, \quad \delta^{0}_\text{alg}(B)= B\cdot \mu .\]
\item 
For local orbit equivalence use
\[  \delta^{0}_\text{orb}: \gergl(\gerh^*) \otimes \ucalC^\textsc{orb}_{\mu, x}\rightarrow \gerh^*\otimes \calC^{\infty}_{x}, \quad \delta^{0}_\text{alg}(B)= B\cdot \mu .\]
\end{myitemize}
The resulting cohomologies will be denoted ${H^{\bullet}_{\chi}(\omega, \mu, x)}$.

\begin{remark} 
When studying local integrable systems ${(\omega, \mu, x)}$, one may impose the rather harmless restriction that ${\mu(x)= 0}$. The only effect that would have on the complexes is only in the case of 
the action equivalence: one can drop the translational term ${\gamma}$ above, hence ${\delta^{0}_\text{act}}$ would be defined on ${\gergl(\gerh^*)\oplus \gerh^*}$. 
\end{remark}
\noindent
The short exact sequence argument used for Corollary \ref{cor -def -cplx -contr -base} can now be used over contractible neighborhoods. 
Paying attention to the induced long exact sequence in low degrees, we find the following:

\begin{corollary} 
For all ${k\geq 2}$ one has ${H^{\bullet}_{\chi, x}(\omega, \mu)\cong H^k(\gerh, \ucalC_{x}^{\infty})}$, the Lie algebra cohomology with coefficients in the germs of smooth functions. 

In low degrees, one has and exact sequence 
\[ 0\rightarrow  H^0_{\chi}(\gerh, \ucalC_{x}^{\infty}) \rightarrow H^{0}_{\chi}(\omega, \mu, x) \rightarrow \ucalC^{\infty}_{x}/\mathbb{R} \rightarrow H^1_{\chi}(\gerh, \ucalC_{x}^{\infty}) \rightarrow H^{1}_{\chi}(\omega, \mu, x) \rightarrow 0,\]
where ${H^{i}_{\chi}(\gerh, \ucalC_{x}^{\infty})}$ is defined by the ${\chi}$-modification (in degree zero) of the Chevalley-Eilenberg complex, as explained above. 
\end{corollary}

One should be aware that, although this corollary may allow us (at least in some cases) to work with smaller complexes to compute ${H^{\bullet}_{\chi, x}(\omega, \mu)}$, for deformations (i.e.\  in degrees ${0}$ and ${1}$) it is still the original complex that one has to use in order to exhibit the controlling cohomology classes.

We now move to deformations of local integrable systems. Instead of trying to develop the language to talk about smooth deformations of germs (and then equivalences between them), we take the direct, explicit (but rather obvious) approach.

\begin{definition} 
Let ${(\omega,\mu, x)}$ be a local completely integrable system, say defined on a manifold ${M}$.
A \textbf{(local) deformation}\index{deformation!local} of ${(\omega,\mu, x)}$ is a 
deformation ${\{(\omega^t, \mu^t)\}}$ of ${(\omega|_{U}, \mu|_{U})}$, for some neighborhood ${U}$ of ${x}$. 
\end{definition}
\noindent
\begin{remark} Hence one has smooth maps
 \[
        \omega
        :
        ( -\epsilon,\epsilon) \times U
        \to
        \wedge^2 T^*M,
        \quad
        \mu
        :
        ( -\epsilon,\epsilon) \times U
        \to
        \R^n.
    \]
It might seem more natural to define local deformations as smooth maps on open neighborhoods ${V\sub \R\times M}$ of ${\{ 0 \}\times \{ x\}}$, instead of the `square' neighborhood ${( -\epsilon,\epsilon)\times U}$. One reason could be that one typically thinks of the domain of a flow like that. However, we think of ${\epsilon}$ as arbitrarily small, so we can always shrink it to get ${V=( -\epsilon,\epsilon)\times U}$.

A more interesting possibility is that allowing the point ${x}$ to carry with respect to ${t}$. We chose not to consider this more general situation for simplicity. Therefore, in the situation of the definition, we will say that ${\{(\omega^t, \mu^t, x)\}}$
is a (local) deformation of ${(\omega, \mu, x)}$. 
\end{remark}
\noindent
Using the notion of local ${\chi}$-equivalence introduced in subsection \ref{The local version of equivalences; invariance under equivalences} (Definition \ref{def:local -chi -equivlces}), 
that of 
${\chi}$-triviality of local deformations, is rather straightforward:

\begin{definition} 
Let ${\chi \in \set{ \text{act}, \text{orb}, \text{fib},  \text{wk},  \text{fct}, \text{alg}}}$. Let ${(\omega^t,\mu^t, x)}$ be a deformation of ${(\omega, \mu, x)}$, assumed to be defined over a manifold ${M}$. We say that the deformation is ${\chi}$-trivial if there exist
\begin{itemize}
\item[1.\ ] an open neighborhood ${U}$ of ${x}$ in ${M}$. 
\item[2.\  ] a smooth map 
\[ \Phi: ( -\epsilon, \epsilon)\times U\rightarrow M \]
inducing symplectomorphisms ${\Phi^t: (U, \omega)\rightarrow (U^t, \omega^t)}$ sending ${x}$ to ${x}$, where ${U^t= \Phi^t(U)}$.
\item[3.] one can complete each ${\Phi^t}$ to a ${\chi}$-equivalence between ${(U, \omega, \mu)}$ and ${(U^t, \omega^t, \mu^t)}$, depending smoothly on the time ${t}$.
\end{itemize}
\end{definition}
\noindent
As in Definition \ref{def:chi -trivial -defs}, the smooth dependence w.r.t.\  ${t}$ of the extra data needed to complete each ${\Phi^t}$ into a ${\chi}$-equivalence (item 3) is the obvious one in each case (see the explanations after the cited definition):
for \textbf{fiber triviality} and \textbf{weak triviality} there is no extra data (just the usual conditions required by the definition of equivalence); 
for \textbf{action triviality}, \textbf{orbit triviality} and \textbf{algebraic triviality} the story is the exactly the same as in the nonlocal case; for \textbf{functional triviality} the extra data is a family of diffeomorphisms 
${\phi^t: \mu^0(U)\rightarrow \mu^t(U^t)\subset \mathbb{R}^n}$ and then the smoothness with respect to ${t}$ is again the smoothness of the resulting map ${( - \epsilon, \epsilon)\times \mu^0(U)\rightarrow \mathbb{R}^n}$). 

By the same argument as before (where the existence of the flows based on the compactness of ${M}$ is replaced by the local existence - which is enough since we work germwise) we deduce:

\begin{proposition}
\label{cis:prop:local_deformations}
Let ${\chi\in\{\text{act},\text{fct},\text{wk} \}}$ and let ${(\omega^t,\mu^t, x)}$ be a deformation of ${(\omega, \mu, x)}$. Then the derivatives of ${(\omega^t, \mu^t)}$ with respect to ${t}$ induce cohomology classes 
\[ [c^t_x]\in H^{1}_{\chi, x}(\omega^t, \mu^t).\]
These classes vanish if the deformation is ${\chi}$-trivial. For the converse, if we assume that:
\begin{itemize}
\item[1.\ ] ${[c^t_x]\in H^{1}_{\chi, x}(\omega^t)}$ is zero smoothly w.r.t.\  ${t}$ in the sense that ${c^t=\dd_{\omega^t,\mu^t}(b^t)}$ have solutions ${b^t\in \calC^0_\chi(\omega^t, \mu^t)}$ that are all defined over a neighborhood of ${x}$, and which vary smoothly in  ${t}$,
 \item[2.\  ] the solution can be chosen s.t.\ the one-form part of ${b^t}$ vanishes at ${x}$,
 \end{itemize}
 then the deformation must be ${\chi}$- trivial. 
 \end{proposition}
\noindent
 
 Note that the smooth variation with respect to ${t}$ has the same meaning as in Proposition \ref{prop_families} (see the comments following it). Actually, condition 1 here just means that there is a neighborhood ${U}$ of ${x}$ on which all
 ${(\omega^t, \mu^t)}$ are defined and such that ${[c^t]\in H^{1}_{\chi}(\omega^t|_{U}, \mu^t|_{U})}$ vanishes smoothly with respect to ${t}$. The extra condition 2 is due to our choice to fix ${x}$. When these conditions are satisfied we say that 
 \textbf{the deformation is cohomologically ${\chi}$-trivial}.

\begin{remark}
\label{proofremark} 
Let us spell out the cohomological triviality condition for ${(\omega^t,\mu^t, x)}$. First of all, there should exist
an open neighborhood ${U\sub M}$ of ${x}$, an ${\epsilon > 0}$, and a smooth family
    \[
        \alpha^t\in \Omega^1(U)
    \]
such that, for all ${t\in ( -\epsilon, \epsilon)}$, 
 \[
\partial_t \omega^t= - \dd \alpha^t, \quad \textrm{and} \quad \alpha^t_x = 0.
\]
Then, case by case, there should further exist: 
\begin{itemize}
\item To be cohomologically action -trivial near ${x}$: also smooth families
    \[
        A^t\in \gergl(\gerh^*), \quad \gamma_t\in \gerh^*
    \]
parametrized by ${t\in ( -\epsilon, \epsilon)}$, such that
\[        
\partial_t\mu^t
= A^t \cdot \mu^t + \gamma_t+ \alpha^t( X_{\mu^t}^{\omega^t}).
\]

\item
To be cohomologically functionally trivial near ${x}$: also an open neighborhood ${V\sub \gerh^*}$ of ${\mu^0(x)}$ and a smooth family 
    \[
        Y^t\in \calC^{\infty}(V, \gerh^*)
    \]
parametrized by ${t\in ( -\epsilon, \epsilon)}$, such that
\[ \partial_t\mu^t = Y^t\circ \mu^t + \alpha^t( X_{\mu^t}^{\omega^t}).\]

\item
To be cohomologically weakly trivial near ${x}$: no further data, just the extra condition: 
\[ 
 \big\{
            \mu_i^t,
           \partial_t\mu^t_j
            +
            \alpha^t(X_{\mu_j^t}^{\omega^t}
        \big\}_{\omega^t}= 0.
\]
for all ${t\in ( -\epsilon, \epsilon)}$ and for all ${i}$ and ${j}$.

\item
To be cohomologically algebraic -trivial near ${x}$: also a smooth family
   \[
        B^t\in \calC^{\infty}(U, \gerh^*)
    \]
parametrized by ${t\in ( -\epsilon, \epsilon)}$, such that
\[ \partial_t\mu^t = B^t\cdot \mu^t + \alpha^t(X_{\mu^t}^{\omega^t}).\]
And completely similarly for orbit equivalence. 
\qedhere
\end{itemize}
\end{remark}

\newpage
\section{Our Main Linearization Theorem}
\etocsettocstyle{\subsubsection*{Local Contents}}{}
\etocsettocdepth{2}
\localtableofcontents

{\color{white}.}
\newline
We now use the cohomological techniques to obtain normal form results for integrable systems. 
The main reason we can do this is the fact that, around a fixed point, any integrable system can be smoothly deformed into the local normal form around the point.
More importantly, once a coordinate chart is chosen, the deformation will be completely canonical and is particularly well adapted to the correspondance of the previous chapter. 

This is a common strategy in the literature, usually done in local coordinates. As we mentioned, it depends on the choice of a chart around $x$. But as we shall show, the resulting cohomology class is canonical. In particular, for a general integrable system ${(\omega, \mu)}$ with fixed point ${x}$, we will construct canonical classes 
   \[
        \text{Lin}_{\chi}(\omega,\mu, x) \in H^1_{\chi}(\omega,\mu, x),
    \]
one for each equivalence relation    $
        \chi 
        \in
        \set{ \text{act}, \text{fct}, \text{orb}, \text{wk}, \text{alg} }
${, which control the }$\chi${-linearizability  of }$(\omega, \mu)${ around }$x$ (see Definition \ref{def-chi-lineariz}). 

Note that in Miranda and Vũ Ng\d oc \cite{MN05} one can find that $H^1_\text{wk}(\omega_\text{can}, \mu_0)$ vanishes. In Chapter \ref{chapter:sheaves_behind} we argue that weak equivalence is the central notion of equivalence for completely integrable systems (in particular for those with only nondegenerate singularities, see also \ref{weak preserves nondegenerate}).
In future work, by drawing inspiration from \cite{CM12, CrFe}, we intend to combine the statement from \cite{MN05} with the methods presented here to provide a concise proof of Theorem \ref{thm -El2}.

\subsection{The linearization class}

We start with the construction of the local deformation of ${(\omega, \mu)}$ into its normal model around ${x}$. For that we fix a chart around ${x}$:
\[
        \Phi: U\to T_xM
    \]
for which ${\Phi(x)=0}$. The chart allows us to define scalar multiplication on ${U}$:
    \[
        m : 
        (0,1] \times U \to U,
        \quad
        m^r(x) 
        \defeq 
        \Phi^{ -1} \big(
            r \cdot\Phi(x)
        \big).
    \]
We now define 
\[ \mu^r: (U, \omega^r)\rightarrow \R^n \]
smoothly parametrized by ${r\in (0, 1]}$  by:
    \[
        \omega^r
        \defeq
        \tfrac{1}{r^2}
        \,
        ( m^{r})^* \omega\in \Omega^2(U), \quad 
        \mu^r
         \defeq
        \tfrac{1}{r^2}
        \,
        ( m^{r})^*\mu 
        =
        \tfrac{1}{r^2}
        \,
        \mu 
        \circ
        m^r \in \calC^{\infty}(U, \R^n). 
    \]
At ${r= 1}$ we obtain the original integrable system. The following is immediate:

\begin{lemma}
${(\omega^r,\mu^r)}$ extends smoothly to ${r=0}$ and ${\mu^0: (U, \omega^0)\rightarrow \mathbb{R}^n}$ is isomorphic, via ${\Phi}$, to the normal model of ${(\omega, \mu)}$ at ${x}$, restricted to ${\Phi(U)}$.
\end{lemma}

Therefore we obtain a class in deformation cohomology for every ${r\in [0,1]}$:
    \[
        c^r(\omega,\mu)
        \defeq
        \big[            
            \tfrac{ \partial } { \partial r }
            (\omega^r,\mu^r)
        \big]
        \quad
        \in
        H^1(\omega^r,\mu^r, x).
    \]
Likewise, we obtain cohomology classes ${c_{\chi}^{r}(\omega,\mu, x)}$ in ${H^1_{\chi}(\omega^r,\mu^r, x)}$ for each equivalence relation
    $
        \chi
        \in
        \big\{
            \text{act},
            \text{fct},
            \text{orb}, 
            \text{wk},
            \text{alg}
        \big\}.
    $ 
We now show that, at ${r= 1}$, i.e.\  for the original system, the resulting class is intrinsic.

\begin{proposition}
For any local integrable system ${(\omega, \mu, x)}$, the cohomology class 
${c^{1}(\omega,\mu, x)\in H^1(\omega,\mu, x)}$ (and similarly ${c_{\chi}^1(\omega,\mu, x)\in H^1_{\chi}(\omega,\mu, x)}$)
is independent of the chosen chart.
\end{proposition}
\noindent
\begin{proof}
Let us denote, inside this proof, by ${m^r}$ the rescale by ${r}$ on a vector space ${V}$ (the relevant one being ${T_xM}$ here) and, for a chart ${\Phi}$ as above, use the notation ${m^r_{\Phi}}$ for the induced rescale: 
\[ m^r_{\Phi}(x) 
        \defeq 
        \Phi^{ -1} \big(
            r \cdot\Phi(x)).\]
We first remark that ${m^r}$ is related to the flow of the standard Euler vector field on ${E}$ (via ${r= e^t}$); the precise form of this relationship that will be useful for us is the remark that, 
over ${V}$, one has 
\[
    \Omega^*(V)\rightarrow \Omega^*(V), 
    \quad
    \eta\mapsto \frac{\partial}{\partial r}\Big\vert_{r=1} \tfrac{1}{r^2}(m^r)^*(\eta)= \gerL_E(\eta) - 2\eta,
\]
for all $\eta\in \Omega^*(V)$
Since ${m^r_{\Phi}= \Phi^{ -1}\circ m_r\circ \Phi}$, we obtain for any form ${\theta}$ on ${M}$:
\begin{align*}
    \frac{\partial}{\partial r}\Big\vert_{r=1} 
    \tfrac{1}{r^2}(m^r_{\Phi})^*(\theta)
    &= \Phi^*\Big( 
        \frac{\partial}{\partial r}\Big\vert_{r=1} 
        \tfrac{1}{r^2}(m^r)^*(\Phi^{ -1})^*(\theta)
    \Big)
    \\
    &= \Phi^*(\gerL_E((\Phi^{ -1})^*\theta) - 2(\Phi^{ -1})^*(\theta))
\end{align*}
Using that ${\Phi^*\gerL_E= \gerL_{\Phi^*E}\phi^*}$ we deduce the working formula
\[ 
    \frac{\partial}{\partial r}\Big\vert_{r=1} 
    \tfrac{1}{r^2}(m^r_{\Phi})^*(\theta)
    =\gerL_{\Phi^*E}(\theta) - 2\theta.
\]
We will apply when ${\eta}$ is the two -form ${\omega}$ and the zero-forms ${\mu^i}$ - so that we get expressions for the resulting cocycles.

Assume now that we have two charts ${\Phi_0,\Phi_1:U\to T_xM}$ as before and consider the two induced cocycles:
\[ c_i= \frac{\partial}{\partial r}\Big\vert_{r=1} \Big( \tfrac{1}{r^2}(m^{r}_{\Phi_i})^*(\omega), \tfrac{1}{r^2}(m^{r}_{\Phi_i})^*(\mu)\Big),\]
with ${i\in \{0, 1\}}$. We compute their difference using the working formula described above. Considering the vector field
${X\defeq \Phi_{1}^{*}(E) - \Phi_{0}^{*}(E)}$ we find:
\[ c_1 - c_0= ( -\gerL_X(\omega), -\gerL_X(\mu)).\]
Hence, considering the one-form ${\alpha= i_X\omega}$ we see that
\[ c_1 - c_0= ( - d\alpha, - i_X(d\mu))= ( - d\alpha, \alpha(X_\mu)),\]
where  for the last equality we used the moment map condition ${d\mu= \omega(X_{\mu}, -)}$. But this precisely shows that ${(\alpha, 0)}$ is a primitive of ${c_1 - c_0}$ in the deformation complex. Since
the second component is zero, it applies to all the other deformation complexes.
\end{proof}

\begin{definition} 
For an integrable system ${\mu: (M, \omega)\rightarrow \mathbb{R}^n}$ \textbf{the linearization class}\index{linearization class} of ${(\omega, \mu)}$ at a fixed point ${x}$, denoted
\[ \textrm{Lin}(\omega,\mu, x)\in H^1(\omega,\mu, x),\]
is defined as the cohomology class from the previous proposition. Similarly we define the ${\chi}$-versions of the linearization class, denoted
\[ \textrm{Lin}_{\chi}(\omega,\mu, x)\in H^1_{\chi}(\omega,\mu, x).\]
\end{definition}

\begin{local_theorem}\label{thm -is -mai1}
Let ${(\omega,\mu)}$ be an integrable system on ${M}$, let ${x\in M}$ be a fixed point, and let 
    $
        {
            \chi \in \{
                \text{act},
                \text{fct},
                \text{wk},
                \text{orb},
                \text{alg}
            \}
        }
    ${. If the integrable system is }$\chi${-linearizable around }$x$ then
    \[
        \textrm{Lin}_{\chi}(\omega,\mu, x)=0
        \,\in
        H_\chi^1(\omega,\mu, x).
    \]
\end{local_theorem}

\begin{proof}
By Proposition \ref{cis:prop:local_deformations}, it suffices to show that if ${(\omega, \chi, x)}$ is ${\chi}$-equivalent to its normal form then the deformation ${(\omega^r,\mu^r)}$ will be ${\chi}$-trivial. We should construct 
a smooth family of local ${\chi}$-equivalences between ${(\omega^r,\mu^r)}$ and ${(\omega^0,\mu^0)}$. 

Without loss of generalities we may assume that ${M=\R^{2n}}$, that ${x=0}$ and ${\mu^r(x)=0}$ for all ${r\in [0,1]}$. In this setting, ${m^r}$ means the scalar multiplication of ${\R^k}$ for some ${k}$.

Before specializing to each of the equivalence relations, we will first discuss the parts they have in common. First suppose there is an open neighborhood ${U\sub M}$ of ${x}$ and a local diffeomorphism ${\Phi:U\to M}$ such that ${\Phi(x)=x}$, and
    \[
        \omega
        =
        \Phi^*\omega^0,
        \quad
        \mu
        =
        \mu^0\circ \Phi.
    \]
This implies that
    \[
        \omega^r
        =
        \tfrac1{r^2} (m^r)^*\omega
        =
        ( \Phi \circ m^r )^* \big( 
            \tfrac1{r^2} \omega^0 
        \big)
        =
        ( m^{1/r} \circ \Phi \circ m^r )^* \omega^0,
    \]
where we used that ${\omega^0}$ is constant, so that 
    $
        \tfrac{1}{r^2}
        \omega^0
        =
        (m^{1/r})^*\omega^0.
    $ 
Likewise,
    \[
        \mu^r
        =
        \tfrac1{r^2} \mu \circ m^r
        =
        \tfrac1{r^2}
        \mu^0 \circ \Phi \circ m^r         
        =
        \mu^0 
        \circ 
        (m^{1/r} \circ \Phi \circ m^r),
    \]
where we used that ${\mu^0}$ is a homogeneous polynomial of degree ${2}$. This suggests we define
    \[
        \Phi^r
        \defeq
        m^{1/r} \circ \Phi \circ m^r,
        \quad
        \forall\; r \in (0,1].
    \]
Since ${\Phi(x)=x}$, the family ${\Phi^r}$ extends smoothly to ${r=0}$ and ${\Phi^0= (d\Phi)_x}$. It gives a smooth local isotopy between ${(\omega^r,\mu^r)}$ and ${(\omega^0,\mu^0)}$. Of course, we are not interested in finding an isotopy relating ${(\omega,\mu)}$ and ${(\omega^0,\mu^0)}$, but this argument is the blue-print for the different cases we will discuss now.

\begin{itemize}

\item
Suppose that ${(\omega,\mu)}$ and ${(\omega^0,\mu^0)}$ are locally action equivalent, i.e.\  there is an open neighborhood ${U\sub M}$ of ${x}$, a local diffeomorphism ${\Phi:U\to M}$ with ${\Phi(x)=x}$, and a automorphism ${A\in \GL(\R^n)}$ such that
    \[
        \omega = \Phi^*\omega^0,
        \quad
        A\cdot \mu = \mu^0\circ\Phi.
    \]
We define ${\Phi^r}$ as above and ${A^r\defeq A}$, so that ${\omega^r = (\Phi^r)^*\omega^0}$ and
    \[
        A^r \cdot \mu^r
        =
        A\cdot m^{1/r^2} \circ\mu\circ m^r
        =
        m^{1/r^2} A\cdot\mu\circ m^r
        =
        \mu^0 \circ \Phi^r.
    \]

\item
Suppose that ${(\omega,\mu)}$ and ${(\omega^0,\mu^0)}$ are locally functionally equivalent, i.e.\  there are open neighborhoods ${U\sub M}$ of ${x}$ and ${V\sub \R^n}$ of ${\mu(x)}$, a local diffeomorphism ${\Phi:U\to M}$ with ${\Phi(x)=x}$, and a local diffeomorphism ${\phi:V\to \R^n}$ such that 
    \[
        \omega = \Phi^*\omega^0,
        \quad
        \phi\circ\mu = \mu^0\circ\Phi.
    \]
Similar to ${\Phi^r}$ we define
    \[
        \phi^r = m^{1/r^2} \circ \phi \circ m^{r^2},
        \quad\forall\; r \in (0,1],
    \]
where the two appearances of ${m_r}$ refers to the rescale centered at the origin and ${\mu(x)}$, respectively. Since ${\phi(\mu(x))= \mu^0(\Phi(x))= 0}$, this family ${\phi^r}$ extends smoothly to ${r=0}$ and ${\phi^0= (d\phi)_x}$. We conclude that ${\omega^r=(\Phi^r)^*\omega^0}$ and
    \[
        \phi^r \circ \mu^r
        =
        m^{1/r^2} \circ \phi\circ\mu \circ m^r
        =
        \mu^0 \circ \Phi^r.
    \]

\item
Suppose that ${(\omega,\mu)}$ and ${(\omega^0,\mu^0)}$ are locally weakly equivalent, i.e.\  there is an open neighborhood ${U\sub M}$ of ${x}$ and a local diffeomorphism ${\Phi:U\to M}$ with ${\Phi(x)=x}$  such that
    \[
        \omega = \Phi^*\omega^0,
        \quad
        \big\{
            \mu_i, \mu_j^0\circ \Phi
        \big\}_{\omega}
        = 0
        \quad\forall\; i,j.
    \]
We define ${\Phi^r}$ as above, so that ${\omega^r=(\Phi^r)^*\omega^0}$ and 
    \begin{align*}
        \big\{
            \mu^r_i, \mu^0_j\circ \Phi^r 
        \big\}_{\omega^r}
        & =
        \big\{
            \tfrac{1}{r^2}
            \mu_i \circ m^r,
            \tfrac{1}{r^2}
            \mu^0_j \circ \Phi\circ m^r
        \big\}_{
            \tfrac{1}{r^2}
            (m^r)^*\omega
        }
        \\
        & =
        \tfrac{1}{r^6}
        \big\{
            \mu_i,
            \mu_j^0 \circ \Phi
        \big\}_{\omega}\circ m^r 
        =0,
        \quad\forall\;
        i,j.
    \end{align*}

\item
Suppose that ${(\omega,\mu)}$ and ${(\omega^0,\mu^0)}$ are locally algebraically equivalent, i.e.\  there is an open neighborhood ${U\sub M}$ of ${x}$, a local diffeomorphism ${\Phi:U\to M}$ with ${\Phi(x)=x}$, and a map ${A: U\to \GL(\R^n)}$ such that
    \[
        \omega= \Phi^*\omega^0,
        \quad
        A\cdot \mu = \mu_0\circ \Phi,
        \quad
        \{
            \mu, A 
        \}_\omega = 0.
    \]
We define ${\Phi^r}$ as above and ${A^r = A\circ m^r}$, so that ${\omega^r=(\Phi^r)^*\omega^0}$ and
    \begin{align*}
        A^r\cdot \mu^r
        & =
        (A\circ m^r) \cdot m^{1/r^2}\circ \mu\circ m^r
        \\
        & =
        m^{1/r^2}\circ (A\circ m^r)\cdot (\mu\circ m^r)
        =
        \mu^0 \circ \Phi^r.
    \end{align*}
Moreover, this satifies
    \begin{align*}
        \{
            \mu^r,
            A^r
        \}_{\omega^r}
        & =
        \{
            \tfrac{1}{r^2}
            \mu \circ m^r,
            A\circ m^r
        \}_{
            \tfrac{1}{r^2}
            (m^r)^*\omega
        }
        \\
        & =
        \tfrac{1}{r^4}
        \{
            \mu,
            A
        \}_{\omega}\circ m^r
        = 0
    \end{align*}
\end{itemize}
The case of orbit equivalence is completely similar. 
\end{proof}

\subsection{The Main Linearization Theorem}

Given the second part of Proposition \ref{cis:prop:local_deformations}, the converse of the last theorem is within reach and quite clear how to attack: given the fact that the family 
${(\omega^r, \mu^r)}$ was obtained from ${(\omega, \mu)}$ by rescaling, it is likely that a solution of the cohomological equations at ${r= 1}$ will give, again using rescaling, a solution for each ${r}$.
This works without much problem - except for one `little thing': we have to make sure that everything can be made smooth at ${r= 0}$. And this will send us right away to the nondegeneracy
condition from Definition \ref{def-weired-nondegen}. While that definition may seem strange at first, or if we were not even aware of it, we would recover it as precisely the condition that is needed 
for smoothness at ${r= 0}$.
We obtain that, except for the case of action-equivalence, the converse of Theorem \ref{thm -is -mai1} holds at nondegenerate fixed points:

\begin{local_theorem}\label{thm -is -mai2}
For any ${{\chi \in \{\text{fct}, \text{wk}, \text{orb}, \text{alg}\}}}$, if ${\mu: (M, \omega)\rightarrow \R^n}$ is a completely integrable system with nondegenerate fixed point ${x}$ at which
\[ \textrm{Lin}_{\chi}(\omega,\mu, x)= 0,\]
then ${(\omega, \mu)}$ is ${\chi}$-linearizable around ${x}$. 
\end{local_theorem}

\begin{proof}

First we fix some notations. For various families ${\{\gamma^r\}}$ of forms, functions, etc, an important role will be played by their derivatives at ${r= 1}$; for that we will use a distinct notation:
\begin{equation}\label{common -eq} 
\dot{\gamma}^1\defeq \frac{\partial}{\partial r}\Big\vert_{r=1}\gamma^r
\end{equation}
Also, as for vector fields (Definition \ref{vert -D -for -X}), for 1-forms ${\alpha}$ we will denote by 
\[ D_{x}(\alpha): T_xM\rightarrow T_{x}^{*}M\]
the vertical derivative - intrinsically defined if ${\alpha(x)= 0}$ (in local coordinates, this just amounts to taking the partial derivatives at ${x}$ of the components of ${\alpha}$). For uniformity, we will also denote by 
\[ D_x(f)\in T_{x}^{*}M\]
the differential at ${x}$ of a function ${f}$- so that one can write
\begin{equation}\label{eq:basic -Dx} 
D_x(\alpha(X))= D_x(\alpha)(X_x)+ \alpha_x\circ D_{x}(\alpha) .
\end{equation}
Note also that, for a one-form ${\alpha}$ with ${\alpha_x= 0}$ and for which the usual differential ${(d\alpha)_x}$ is zero, the operator ${D_x(\alpha)}$ is self-adjoint, hence can be interpreted as an element in the symmetric algebra. It will be denoted 
\[ \sigma_x(\alpha)\in S^2(T^*M).\]
Hence ${\sigma_x(\alpha)(X_x, Y_x)= D_x(\alpha)(X_x)(Y_x)}$ (and is symmetric because ${(d\alpha)_x= 0}$). For instance, when ${\alpha= df}$ for a function ${f}$ (so that ${d\alpha= 0}$ everywhere and the only condition that is left if ${(df)_x= 0}$), one has
\[ \sigma_x(df)= \hess_x(f).\]
While it is clear from the general discussion that for any such ${f}$ one has
\[ \hess_x(\{f, \mu_i\}_{\omega})= \{ \hess_x(f), \hess_x(\mu_i)\}_{\omega_x},\]
let us point out that, more generally:

\begin{lemma}
\label{lemma:tricky -one} 
For any one-form ${\alpha}$ with ${\alpha_x= 0}$ and ${(d\alpha)_x= 0}$ one has
\[ \hess_x(\alpha(X_{\mu_i}))= \{\sigma_x(\alpha), \hess_x(\mu_i)\}_{\omega_x}.\]
\end{lemma}

We now return to the proof of the theorem. Assuming that ${\textrm{Lin}_{\chi}(\omega,\mu, x)}$ vanishes, we will show that a primitive ${(\alpha,\beta)}$ at ${r=1}$ can be extended to a smooth path of (local) primitives ${(\alpha^r,\beta^r)}$, i.e.\  satisfying
    \[
        \frac{\partial}{\partial r}(\omega^r,\mu^r)
        =
        \dd_{\omega^r,\mu^r}(\alpha^r,\beta^r),
        \quad\forall\; r \in [0,1].
\] 
If we can arrange that ${\alpha^{r}_{x}= 0}$, proposition \ref{cis:prop:local_deformations} takes care of the rest. 

We may, and we will, assume that ${\mu(x)= 0}$. We could also assume that ${M= \mathbb{R}^{2n}}$ and ${x= 0}$ (and one should whenever one of the formulas used below is not transparent to the reader); but that would give rise to local formulas whose meaning is not so clear and are hard to manipulate. So, to make the arguments more intuitive, we will work with ${(M, x)}$ as an intrinsic object.

We now discuss the parts of the proof that is common to all the equivalences ${\chi}$. A solution at ${r= 1}$ of our equation has a first component that is the same for all equivalence relations: it is a one-form ${\alpha}$ satisfying
    \[
        \dot{\omega}^1
        =
        - \dd \alpha.
    \]
This implies that
    \begin{align*}
        \frac{\partial}{\partial r} \omega^r
        & = 
        \frac{\partial}{\partial r} \big(
            \tfrac{1}{r^2} (m^r)^*\omega
        \big)
        =
        \tfrac{1}{r^3} (m^r)^* \big(
            \gerL_\zeta\omega - 2 \omega
        \big)
        \\
        & =
        \tfrac{1}{r^3} (m^r)^* \big(
            \frac{\partial}{\partial r}\Big\vert_{r=1}
            \omega^r
        \big)
        = 
        - \dd \big(
            \tfrac{1}{r^3} (m^r)^*\alpha
        \big).
    \end{align*}
This suggests we define
    \[
        \alpha^r
        \defeq
        \tfrac{1}{r^3} (m^r)^*\alpha,
        \quad \forall\; r \in (0,1].
    \]
The `only' problem is to to extend this family smoothly to ${r= 0}$- and this is possible if and only if
\begin{equation}
\label{eq -problem -to -solve -1}
        \alpha_x=0,
        \quad
        D_x(\alpha)=0.
\end{equation}
As we have mentioned, it is this problem that will bring in the conditions on ${\gerh_{x, \mu}}$ (nondegeneracy).

The other component in the deformation complex (the derivatives of ${\mu^r= \tfrac{1}{r^2} \mu\circ m_r}$) is
    \[
        \frac{\partial}{\partial r}\mu^r
        =
        \frac{\partial}{\partial r} \big(
            \tfrac{1}{r^2} 
            \mu \circ m^r
        \big)
        =
        \tfrac{1}{r^3} 
        \big(
            \gerL_\zeta\mu - 2\mu
        \big)\circ m^r
        =
        \tfrac{1}{r^3}
        \big(
            \frac{\partial}{\partial r}\Big\vert_{r=1}
            \mu^r
        \big) \circ m^r.
    \]
Before treating each of the equivalence relations separately, we should understand the relation between ${X_\mu^\omega}$ and ${X_{\mu^r}^{\omega^r}}$, since the latter is a recurring term in the deformation complex of ${(\omega^r,\mu^r)}$. Note that by definition ${X_{\omega^r}^{\mu^r}}$ is the unique vector field ${X^r}$ such that
    \[
        \tfrac{1}{r^2}
        (m^r)^*\big(
            \iota_{(m^r)_*X^r}(\omega)
        \big)
        =
        \iota_{X^r}(\omega^r)
        =
        - \dd\mu^r
        =
        - \tfrac{1}{r^2} (m^r)^*\dd\mu.
    \]
So we find that ${X^{\omega^r}_{\mu^r}=(m^r)^*X_\mu^\omega}$. In particular, this means that
    \[
        \alpha^r(X_{\mu^r}^{\omega^r})
        =
        \tfrac{1}{r^3} \alpha(X_\mu^\omega) \circ m^r.
    \]
This formula already lines up well with the formula for ${\frac{\partial}{\partial r}\mu^r}$ we found above.

\begin{itemize}

\item
Suppose that the class corresponding to weak equivalence near ${x}$ vanishes. Pick a primitive, i.e.\  a one-form ${\alpha}$ as above, defined in an open neighborhood ${U\sub M}$ of ${x}$ and satisfying 
\begin{equation}\label{orig -eq -mthm -wk}
        \dot{\omega}^1 = - \dd\alpha,
        \quad
        \big\{
            \mu_i,
            \dot{\mu}^1_j +\alpha(X_{\mu_j}^{\omega})
        \big\}_{\omega} 
        = 0,
        \quad\forall\; i,j.
\end{equation}
Define ${\alpha^r}$ as above and define
    \[
        \beta^r
        \defeq
        \frac{\partial}{\partial r}\Big\vert_{r=1}
        \mu^r 
        +
        \alpha(X_\mu^\omega).
    \]
At all nonzero ${r}$'s, next to ${- \dd \alpha^r= \omega^r}$ (already proven), we find
    \begin{align*}
        \big\{
            \mu_i^r,
            \frac{\partial}{\partial r}\mu^r_j
            & +
            \alpha^r(X_{\mu_j^r}^{\omega^r}
        \big\}_{\omega^r}
        \\
        & =
        \tfrac{1}{r^7}
        \big\{
            \mu_i\circ m^r,
            \big(
                \frac{\partial}{\partial r}\Big\vert_{r=1}
                \mu_j^r
                +
                \alpha(X_{\mu_j}^\omega)
            \big) \circ m^r
        \big\}_{(m^r)^*\omega}
        \\
        & =
        \tfrac{1}{r^7}
        \big\{
            \mu_i,
            \frac{\partial}{\partial r}\Big\vert_{r=1} 
            \mu_j^r 
            +
            \alpha(X_{\mu_j}^\omega)
        \big\}_\omega \circ m^r
        =
        0.
    \end{align*}
In other words, $
        (
            \alpha^r,
            \beta^r
        )
    $
would be the desired family of primitives. But we still have to ensure the smoothness at ${r= 0}$, i.e.\  the vanishing conditions  (\ref{common -eq}) for ${\alpha}$. First, we claim that any solution ${\alpha}$ of the original equation (\ref{orig -eq -mthm -wk}) must satisfy ${\alpha_x= 0}$. Indeed, the second part of the original equations say that, for all ${j}$, 
\[ \dot{\mu}^1_j + \alpha(X_{\mu_j}^{\omega})\in \ucalC_{\mu}\]
hence, by Proposition \ref{prop:nondeg_function} (since the fixed point set is trivial), the derivative ${D_x}$ at ${x}$ of this function vanishes. The derivative for ${\dot{\mu}^1_j}$ is easily seen to be zero, hence we are left with the derivative of the second term. Using the basic formula (\ref{eq:basic -Dx}) and the fact that ${X_{\mu_j, x}= 0}$, we find that
\[ \alpha_x\circ D_x(X_{\mu_j})= 0.\]
In other words, ${\alpha_x}$ vanishes on the images of all the ${D_x(X_{\mu_j})}$, i.e.\  
\[ \alpha_x\in (T_{x}^{*}M)^{\gerh},\]
which is zero by assumption. Hence ${\alpha_x= 0}$. So far, we have not used the entire statement of Proposition \ref{prop:nondeg_function}: it also tells us that the hessian of ${\dot{\mu}^1_j - \alpha(X_{\mu_j}^{\omega})}$ is in the centralizer of ${\gerh_{x, \mu}}$ hence, by hypothesis, it is in ${\gerh_{x, \mu}}$. Again, the hessian of ${\dot{\mu}^1_j}$ at ${x}$ is easily seen to vanish, hence we are left with the hessian of ${\alpha(X_{\mu_j}^{\omega})}$. For that we use Lemma \ref{lemma:tricky -one} which can be applied since ${(d\alpha)_x= 0}$ (indeed, this is the ${\dot{\omega}_{x}^{1}}$ which is ${0}$ because ${\omega_{x}^{r}}$ is clearly constant w.r.t.\  ${r}$). Hence we obtain
\[ \{ \sigma_x(\alpha),\hess_x(\mu_j)\}_{\omega_0} \in \gerh_{x, \mu}\]
for all ${j}$. Hence, by hypothesis, ${\sigma_x(\alpha)}$ belongs to ${\gerh_{x, \mu}\subset S^2T^{*}_{x}M}$. Hence we can write 
\[ \sigma_x(\alpha)= \sum\nolimits_{i} \lambda_i \hess_x(\mu_i)\]
for real numbers ${\lambda_i}$. Subtracting from ${\alpha}$ the one-form ${\sum_i \lambda_i d\mu_i}$ we now obtain a new solution of (\ref{orig -eq -mthm -wk}) for which also the derivative ${D_x}$ is zero.

\item
Suppose that the class w.r.t.\  functional equivalence vanishes. Pick a primitive, i.e.\  open neighborhoods ${U\sub M}$ of ${x}$ and ${V\sub \R^n}$ of ${0}$ and
    \[
        (\alpha,\beta) \in \Omega^1(U)\oplus C^\infty(V,\R^n),
    \] 
and such that
    \begin{equation}\label{wk -case -starting -eq}
        \dot{\omega}^1
        =
        - \dd\alpha,
        \quad
        \dot{\mu}^1
        =
        \beta \circ \mu 
        +
        \alpha(X_\mu^\omega).
    \end{equation}
Define ${\alpha^r}$ as above and 
    \[
        \beta^r
        \defeq 
        m^{1/r^3}\circ \beta \circ m^{r^2}.
    \]
At all ${r\neq 0}$, it satisfies 
    \begin{align*}
        \frac{\partial}{\partial r}\omega^r 
        & = 
        - \dd\alpha^r,
        \\
        \frac{\partial}{\partial r} \mu^r
        & =
        \tfrac{1}{r^3} \big(
            \frac{\partial}{\partial r}\Big\vert_{r=1}
            \mu^r
        \big) \circ m^r
        \\
        & =
        \tfrac{1}{r^3} \big(
            \beta \circ \mu
            +
            \alpha(X_\mu^\omega)
        \big) \circ m^r
        \\
        & = 
        \beta^r \circ \mu^r
        +
        \alpha^r(X_{\mu^r}^{\omega^r}).
    \end{align*}
i.e.\   
    $
        (
            \alpha^r,
            \beta^r
        )
    $
is a primitive of ${\frac{\partial}{\partial r}(\omega^r,\mu^r)}$. Again, the problem is to make sure that ${\alpha}$ and ${\beta}$ are so that ${\alpha^r}$ and ${\beta^r}$ extend smoothly at ${r= 0}$. That means that, next to the conditions (\ref{eq -problem -to -solve -1}), we have to ensure the similar conditions for ${\beta}$: ${\beta_{0}=0}$, and ${D_{0}(\beta) = 0}$ (recall that we assumed that ${\mu(x)= 0}$). To achieve the conditions on ${\alpha}$ we just use the weak equivalences case discussed above. It can be applied because ${\beta\circ \mu\in \ucalC_{\mu}}$. We conclude that we must have ${\alpha_x= 0}$ and, after changing ${\alpha}$ to another one-form of type ${\alpha - \sum_i \lambda_i d\mu_i}$, one also realizes ${D_x(\alpha)= 0}$. We do have to make sure that this change does not affect the second set of equations in (\ref{wk -case -starting -eq}), but that is clear since ${(d\mu_i)(X_{\mu_j})= 0}$. 

Therefore we may assume we have a solution ${(\alpha, \beta)}$ of (\ref{wk -case -starting -eq}) with ${\alpha_x= 0}$, ${D_x(\alpha)= 0}$. From the equation and the fact that ${\alpha_x= 0}$ it follows right away that ${\beta_x= 0}$. For the other condition on ${\beta}$ we use that 
\[ 0= \hess_x(\beta^i\circ \mu)= \sum\nolimits_k \frac{\partial \beta^i}{\partial y^k}(0) \hess_x(\mu_k). \]
Since the Hessians ${\hess_x(\mu_k)}$ are linearly independent (which is equivalent to the fixed point set being zero) we deduce that ${\frac{\partial \beta^i}{\partial y^k}(0)= 0}$.

\item
Suppose now that the class with respect to algebraic equivalence vanishes. Pick a primitive, i.e.\  an open neighborhood ${U\sub M}$ of ${x}$ and
    \[
        (\alpha,\beta)  
        \in 
        \Omega^1(U) \oplus \calC^\infty(U,\gergl(\R^n)),
    \]
with the components of ${\beta}$ in ${\ucalC_{\mu}}$, and such that
    \[
        \frac{\partial}{\partial r}\Big\vert_{r=1}
        \omega^r
        =
        - \dd\alpha,
        \quad
        \frac{\partial}{\partial r}\Big\vert_{r=1}
        \mu^r
        =
        \beta\cdot \mu
        +
        \alpha(X_\mu^\omega).
    \]
Define ${\alpha^r}$ as above and define
    \[
        \beta^r
        \defeq
        \tfrac{1}{r}
        \beta \circ m^r.
    \]
Again, at ${r\neq 0}$,
    \begin{align*}
        \frac{\partial}{\partial r} \omega^r
        & = 
        - \dd\alpha^r,
        \\
        \frac{\partial}{\partial r} \mu^r
        & =
        \tfrac{1}{r^3} \big(
            \frac{\partial}{\partial r}\Big\vert_{r=1}
            \mu^r
        \big) \circ m^r
        \\
        & =
        \tfrac{1}{r^3} \big(
            \beta \cdot \mu
            +
            \alpha(X_\mu^\omega)
        \big)\circ m^r
        \\
        & =
        \beta^r \cdot \mu^r
        +
        \alpha^r(X_{\mu^r}^{\omega^r}).
    \end{align*}
Hence    $
        (
            \alpha^r,
            \beta^r
        )
    $
is a primitive of ${\frac{\partial}{\partial r}(\omega^r,\mu^r)}$. Again, we have to achieve (\ref{eq -problem -to -solve -1}), as well as ${\beta(x)= 0}$ (for extending ${\beta^r}$ at ${r= 0}$). 
One can apply once more  the discussion from weak equivalences because the components of ${\beta\cdot \mu}$, i.e.\  ${\sum_j \beta^{j}_{i}(x) \mu_{j}(x)}$, are in ${\ucalC_{\mu}}$-
and then we achieve (\ref{eq -problem -to -solve -1}); the argument there also tells us that the Hessian of the components of ${\beta\cdot \mu}$ is equal to ${[\sigma_x(\alpha), \hess_x(\mu_j)\}_{\omega_0}= 0}$. But, since ${\mu(x)= 0}$ and ${(d\mu)_x= 0}$, the Hessian of (the components of) ${\beta\cdot \mu}$ is the (components of) ${\beta(x)\cdot \hess_x(\mu)}$. The resulting equation ${\beta(x)\cdot \hess_x(\mu)}$ implies, as in the functional case, that ${\beta(x)= 0}$.

\item The case of the orbit equivalence follows similarly, by the same type of arguments. 
\qedhere

\end{itemize}

\begin{remark}[conclusion] As mentioned in our thesis outline in Section \ref{sec:outline-thesis} of Chapter \ref{chapter:Intr-and-outline} (item 3), this theorem is Step 1 in a proof to Eliasson's theorems (Theorem \ref{thm -El1} and \ref{thm -El2} from the end of the first chapter) similar to the geometric proof of Conn's linearization theorem from \cite{CrFe}. The other step is the actual vanishing of cohomology. We expect that also this second part can be carried out by standard `soft' methods and we hope to be able to provide the details in the near future.

Let us mention here the test case for the remaining step: the vanishing of the relevant cohomology for the normal models. This has been carried out by Miranda and Vũ Ng\d oc \cite{MN05} for weak equivalences, for each of the three types of nondegenerate singularities. A careful analysis of what happens when we take direct sums can be used to conclude that the same holds for all nondegenerate points (but still the normal models). However, let us mention the proofs we hope for the second step should not only apply generally. We expect them to be easier and more conceptual (certainly not `case by case').
\end{remark}
\end{proof}

\chapter{Pseudogroups}
\label{chap:pseudogroups}

\etocsettocdepth{2}
\localtableofcontents

\noindent
{\color{white}This line is intentionally left blank.}
\newline
Pseudogroups describe the local symmetries of differential equations and of geometric structures.
As pseudogroups are often too general, one has to focus on a more restricted class.
\emph{Lie} pseudogroups, those which are themselves defined by differential equations, are the most interesting such class, as they display a rich interplay between algebra and geometry.
We propose that \emph{closed} pseudogroups already display enough interesting behaviour to be worth further study, and take a step in that direction.
Altough the definition of a closed pseudogroup is not very surprising, to our knowledge it has not gotten any attention in the literature.

Sophus Lie's original studies of Lie pseudogroups \emph{of finite type} gave rise to the modern notion of Lie groups. Central to their success (aside from their ubiquity) is the correspondence between their geometry and algebra via Lie algebras.
In the theorem \ref{pgrp:thm} we show that closed pseudogroups also admit a Lie algebra (sheaf).
This result is best compared with the result that any closed subgroup of a Lie group is a Lie group. Aside from the technical details, its proof follows classical methods involving the flows of sums and Lie brackets of vector fields.

\clearpage
\section{Preliminaries}

Let $X$ be a manifold. 
By a \textbf{local diffeomorphism}\index{local diffeomorphism},
or locally defined diffeomorphism, we mean a pair $(U,\phi)$ consisting of an open subset $U\sub X$ and a smooth open embedding 
    \[
        \phi:U\longhookrightarrow X.
    \]
The usual notion of a \emph{local diffeomorphism} is a smooth map $M\to M$ that has maximal rank everywhere (and thus is everywhere locally a diffeomorphism), whereas for us it means a diffeomorphism that is only defined locally (i.e.\ on an open subset). We have little use for the usual notion and where we do, we spell out its definition. The term \emph{locally defined diffeomorphism} used in repitition quickly becomes too pedantic. Other terms, such as \emph{open embedding}, distract from our point of view: pseudogroups are `local' versions of (subgroups of) the group of diffeomorphisms.

We denote the set of all local diffeomorphisms by $\diff_X$, and the set of those with domain $U$ by $\diff_X(U)$. 
When $U=X$ we will write $\diff(X)$ instead of $\diff_X(X)$.
We will often denote a local diffeomorphism $(U,\phi)$ by $\phi$ and leave the domain implicit.

\subsection{Pseudogroups}

\begin{definition}
\label{def:pseudogroup}
A \textbf{pseudogroup}\index{pseudogroup} is a subset $\pgrp\sub \diff_X$ that satisfies group-like axioms
(identity, composition, and inversion):

\begin{enumerate}

\item[a)]
$\id \in \pgrp(X)$.

\item[b)] 
If $\phi\in \pgrp(U)$ and $\psi\in \pgrp(V)$, then $\phi\circ\psi \in \pgrp(\psi^{-1}(U))$. 

\item[c)]
If $\phi\in \pgrp(U)$, then $\phi^{-1} \in \pgrp(\phi(U))$.

\end{enumerate}

And sheaf-like axioms (restriction and collation):

\begin{enumerate}

\item[d)]
If $\phi\in\pgrp(U)$ and $V\sub U$ is open, then also $\phi\vert_V \in \pgrp(V)$. 
This means in particular that the empty set lies in $\pgrp$.

\item[e)]
If $\phi \in \diff_X(U)$ and there is an open cover $U=\medcup_{i\in I}U_i$ such that $\phi\vert_{U_i} \in\pgrp(U_i)$ for all $i\in I$, then $\phi\in\pgrp(U)$.

\end{enumerate}

If $\pgrp$ satisfies all the above except possibly the collation axiom, then it is called a \textbf{prepseudogroup}.
\end{definition}

In the remainder of this section we discuss various examples of pseudogroups and some methods to construct them. 
We will return to many of these examples later in our discussion.

\begin{example}
\label{exa:pseudogroup_poisson}
Given a symplectic manifold $(X,\omega)$, the local symplectomorphisms form a pseudogroup:
    \[
        \diff_\omega(U)
        \defeq
        \big\{
            \phi \in \diff_X(U)
            \sep
            \phi^*\omega = \omega|_U
        \big\}.
    \]
It is an example of a Lie pseudogroup, one that is determined by a partial differential equation on the local diffeomorphisms (see definition \ref{def:jet_determined}).
\end{example}

\begin{example}
\label{exa:pseudogroup_riem}
Given a Riemannian manifold $(X, g)$, the local isometries
    \[
        \diff_g(U)
        \defeq
        \big\{
            \phi \in \diff_X(U)
            \sep
            \phi^*g = g|_U
        \big\}
    \]
form a pseudogroup. It is well-known \cite{MySte39, Pal57} that the group of global isometries $G\defeq\diff_g(X)$ is a finite dimensional Lie group. 
It thus comes with a faithful action
    $
        G \hookrightarrow \diff(X).
    $
\end{example}

\begin{example}
\label{exa:stabilizer_pseudogroup}
Let $S\sub X$ be a set. The stabilizer of $S$ is a pseudogroup:
    \[
        \diff^S_X(U)
        \defeq
        \big\{
            \phi \in \diff_X(U)
            \sep
            \phi(S \cap U) \sub S,
            \;\;
            \phi^{-1}(S \cap \phi(U)) \sub S
        \big\}.
    \]
This awkward phrasing of what is essentially the condition `$\phi(S)=S$' is necessary when $S$ is not contained in $U$.
If $\pgrp$ is a pseudogroup on $X$, then we denote
    \[
        \pgrp^S(U) \defeq \pgrp(U) \cap \diff_X^S(U).
    \]
\end{example}

\begin{example}
If $\pgrp$ and $\calQ$ are pseudogroups, then so is $\pgrp \cap \calQ$.
\end{example}

\begin{example}
\label{exa:restricted_pseudogroup}
If $\pgrp$ is a pseudogroup on $X$, and $\xi: Y \hookrightarrow X$ is an \emph{open} embedding, 
then we may pull $\pgrp$ back along $\xi$: for any open set $U\sub Y$,
    \[
        \xi^*\pgrp(U)
        \defeq
        \big\{
            \phi \in \diff_Y(U)
            \sep
            \xi\circ\phi\circ\xi^{-1} \in \pgrp(\xi(U))
        \big\}.
    \]
\end{example}

The collation axiom can be described as \textbf{germ-determinacy}\index{germ-determinacy}: two local diffeomorphisms $\phi \in \diff_X(U)$ and $\psi \in \diff_X(V)$ are equivalent as germs at $x$ if there is an open neighborhood $W\sub U\cap V$ of $x$ such that $\phi|_W = \psi|_W$. 
We denote the equivalence class of $\phi$ at $x$ by $\phi_{(x)}$, which is called a germ. 
For any subset $\pgrp\sub \diff_X$, the stalk at $x$ is the set of germs at $x$:
    \[
        \pgrp_{(x)}
        \defeq
        \big\{
            \phi_{(x)}
            \sep
            \phi \in \pgrp(U),
            \;
            x \in U
        \big\}.
    \]
With these concepts in mind, the collation axiom reads that a local diffeomorphism $(U, \phi)$ lies in $\pgrp$ if and only if $\phi_{(x)}\in \pgrp_{(x)}$ for every $x\in U$.

\begin{example}
Suppose $\pgrp\sub \diff_X$ satisfies the group-like and restriction axioms.
Then its sheafification defines a pseudogroup:
    \[
        \overline{\pgrp}^\mathrm{sh}(U)
        \defeq
        \big\{
            \phi \in \diff_X(U)
            \sep
            \phi_{(x)} \in \pgrp_{(x)}
            \;\;\forall\; x\in U
        \big\}.
    \]
\end{example}

\begin{proof}
The sheaf-like and identity axioms are immediate. 
\\
For the composition axiom, if $\phi \in \overline{\pgrp}^\mathrm{sh}(U)$ and $\psi \in \overline{\pgrp}^\mathrm{sh}(V)$, and $x\in \psi^{-1}(U)$, then by definition there are open neighborhoods $U'\sub U$ of $\psi(x)$ and $V'\sub V$ of $x$ such that $\phi|_{U'}\in \pgrp(U')$ and $\psi|_{V'}\in \pgrp(V')$. 
It follows that 
    \[
        (\phi\circ\psi)\big|_{W'} \in \pgrp(W'),
        \qquad
        W' \defeq \psi^{-1}(U')\cap V',
    \]
which shows that 
    $
        (\phi\circ\psi)_{(x)} \in \pgrp_{(x)}.
    $ 
We conclude that $\phi\circ\psi \in \overline{\pgrp}^\mathrm{sh}(\psi^{-1}(U))$. 
The argument is likewise for the inversion axiom.
\end{proof}

\begin{example}
\label{exa:lie_group}
Given a Lie group $G$ and a Lie group action $\rho$ of $G$ on $X$, define a pseudogroup
    \[
        \diff_G(U) \defeq
        \overline{
            \big\{
                \rho(g)|_U \in \diff_X(U) \sep g\in G
            \big\}
        }^\mathrm{sh}
    \]
as follows: a local diffeomorphism $\phi \in \diff_X(U)$ lies in $\diff_G(U)$ if and only if for every $x\in U$ there is an open neighborhood $V\sub U$ of $x$ and a $g\in G$ such that 
    $
        \phi|_V = \rho(g)|_V.
    $
Note that it is the sheafification of the image of $\rho$.

Let $X=\R^m$ and $G\sub\GL_m$, and consider the locally constant maps with values in $G$, seen as a subsheaf of the smooth maps:
    \[
        \underline{G}(U)
        \defeq
        \big\{
            g \in C^\infty(U, G)
            \sep
            \dd g = 0
        \big\}
    \]
A locally constant map $g:x\mapsto g_x$ in $\underline{G}(U)$ defines a smooth map 
    \[
        \phi_g: U\longto X,
        \quad
        \phi_g(x) \defeq g_x(x).
    \]
The map $\phi_g$ is locally a diffeomorphism, since $\dd_x\phi_g = g_x$ is bijective, but may fail to be injective. 
We claim that the pseudogroup $\diff_G(U)$ is the set of those $\phi_g$ that are local diffeomorphisms:
    \[
        \diff_G(U)
        =
        \big\{
            \phi_g
            \sep
            g \in \underline{G}(U)
        \big\}
        \,\cap\,
        \diff_X(U).
    \]
\end{example}

\begin{proof}
Remark that the action $\rho: \mathrm{GL}_m \to \diff(\R^m)$ is locally faithful, i.e.\
if $g,h\in G$ and there is an open set $U\sub \R^m$ such that $\rho(g)|_U = \rho(h)|_U$, then $g=h$.
This is because $\rho$ is a faithful representation, so that
    \[
        \rho(g) = \dd_x\rho(g) = \dd_x\rho(h) = \rho(h)
    \]
for any $x\in U$ implies that $g=h$.

Suppose that $\phi_g$ with $g\in \underline{G}(U)$ is injective. 
Then by definition, for any $x\in U$, there is an open neighborhood $V\sub U$ of $x$ such that $g|_V = g_x$.
This implies that 
    $
        \phi_g|_V = \rho(g_x)|_V,
    $
which shows that $\phi_g \in \diff_G(U)$. 

Now suppose that $\phi\in\diff_G(U)$. 
By definition, for any $x\in U$ there is an open neighborhood $V\sub U$ of $x$ and some $g\in G$ such that 
    \[
        \phi|_V = \rho(g)|_V = \phi_g|_V.
    \]
If also $W\sub U$ and $h\in G$ such that $\phi|_W = \phi_{h}|_W$, then 
    \[
        \rho(g)|_{V\cap W} = \phi|_{V\cap W} = \rho(h)|_{V\cap W}
    \]
implies that $g=h$ whenever the intersection is nonempty.
Hence there is a locally constant $g\in \underline{G}(U)$ such that $\phi=\phi_g$.
\end{proof}

\begin{example}
\label{exa:pseudogroup_of_invariants}
Let $\calC\sub C^\infty$ be a subsheaf of the smooth functions on $X$. 
Then its invariants
    \[
        \diff_\calC(U)
        \defeq
        \big\{
            \phi \in \diff_X(U)
            \sep
            f\circ\phi \in \calC(U)
            \;\;\forall\; f \in \calC(\phi(U))
        \big\}
    \]
satisfy the axioms of a pseudogroup except for the inversion axiom. 
\end{example}

\begin{proof}
The identity, composition, and restriction axioms are immediate.
For the collation axiom, given $f \in \calC(\phi(U))$ and an open cover $U=\cup_i U_i$, suppose that 
    $
        f\circ\phi|_{U_i} \in \calC(U_i)
    $
for all $i$. 
Since $\calC$ is a sheaf, this implies that $f\circ\phi\in\calC(U)$, and thus that $\phi\in \diff_\calC(U)$.
\end{proof}

\begin{example}
To give an example of $\diff_\calC$ that defines a pseudogroup, fix a point $x\in X$ and $\theta\in T^*_xX$, and consider
    \[
        \calC(U) 
        \defeq
        \big\{
            f \in C^\infty(U)
            \sep
            f(x)=0,
            \;
            \dd_xf = \theta
        \big\}
    \]
for any open set $U\sub X$ with $x\in U$, and $\calC(U)=C^\infty(U)$ otherwise.
Then
    \[
        \diff_\calC(U)
        =
        \big\{
            \phi \in \diff_X(U)
            \sep
            \phi(x) = x,
            \;\;
            \theta \circ \dd_x\phi = \theta
        \big\}.
    \]
\end{example}

\begin{proof}
If $\phi\in \diff_X(U)$ and $\phi(x)\neq x$, then we can always choose a function $f\in \calC(\phi(U))$ such that $f(\phi(x))\neq 0$, so then $\phi\notin\diff_\calC(U)$. 
Hence for any $\phi\in \diff_\calC(U)$, this shows that $\phi(x)=x$ and 
    \[
        \theta = \dd_x(f\circ\phi) = \theta \circ\dd_x\phi.
    \]
The other inclusion is immediate.
\end{proof}

\begin{example}
To give an example of $\diff_\calC$ that is not a pseudogroup, consider $X = \R$ and 
    \[
        \calC(U)
        \defeq
        \big\{
            f \in C^\infty(U)
            \sep
            f(x) = 0
            \;\;\forall\;
            x\leq 0
        \big\}
    \]
Note that there is a function $f\in \calC(\R)$ such that $f(x)>0$ for all $x>0$, so any translation $T_r(x) = x - r$ with $r> 0$ lies in $\diff_\calC(\R)$, but its inverse does not.
Instead we should consider the pseudogroup
    \[
        \diff_\calC^\vee(U)
        \defeq
        \big\{
            \phi \in \diff_X(U)
            \sep
            \phi \in \diff_\calC(U)
            \text{ and }
            \phi^{-1} \in \diff_\calC(\phi(U))
        \big\}.
        \qedhere
    \]
\end{example}

\subsection{Lie Algebra Sheaves}

Denote by $\gerX_X$ the sheaf of vector fields on a manifold $X$.

\begin{definition}
A \textbf{Lie algebra sheaf}\index{Lie algebra sheaf} is a subset $\calA \sub \gerX_X$ that satisfies:
\begin{itemize}

\item[a)]
$\calA(U)$ is a Lie subalgebra of $\gerX_X(U)$ for every open $U\sub X$.

\item[b)]
$\calA$ is a subsheaf of $\gerX_X$.
\qedhere

\end{itemize}

\end{definition}

We think of Lie algebra sheaves as the infinitesimal counterparts to pseudogroups in the way that Lie algebras are related to Lie groups. 
However, without additional axioms on the pseudogroups this correspondence cannot be made very precise.
\begin{example}
Given a symplectic manifold $(X, \omega)$, the symplectic vector fields form a Lie algebra sheaf:
    \[
        \gerX_\omega(U)
        \defeq
        \big\{
            v \in \gerX_X(U)
            \sep
            \gerL_v(\omega) = 0
        \big\}
    \]
However, the Hamiltonian vector fields
    \[
        \mathrm{Ham}_\omega(U)
        \defeq
        \big\{
            v \in \gerX_X(U)
            \sep
            i_v\omega = -\dd f
            \text{ for some }
            f \in C^\infty(U)
        \big\}
    \]
do not, since $\mathrm{Ham}_\omega(U) = \gerX_\omega(U)$ for simply connected open subsets $U\sub X$, 
whereas $\mathrm{Ham}_\omega(V) \neq \gerX_\omega(V)$ if $V\sub X$ is more interesting.
Obviously, $\mathrm{Ham}_\omega$ does not satisfy the collation axiom.
\end{example}

\begin{example}
\label{exa:lie_algebra}
Given a Lie algebra $\gerg$ and a Lie algebra action $\rho:\gerg \to X$,
define a Lie algebra sheaf
    \[
        \calA_\gerg(U)
        \defeq
        \overline{
            \big\{
                \rho(a)|_U \in \gerX_X(U)
                \sep
                a \in \gerg
            \big\}
        }^\mathrm{sh}
    \]
as follows: a local vector field $v \in \gerX_X(U)$ lies in $\calA_\gerg(U)$ if and only if for every $x\in U$ there is an open neighborhood $V\sub U$ of $x$ and some $a\in \gerg$ such that $v|_V = \rho(a)|_V$.
Note that it is the sheafification of the image of $\rho$.
\end{example}

\begin{example}
Consider a pair $(\calC, \calB)$ where $\calC$ is a presheaf of smooth functions, $\calB$ is a linear sheaf of smooth functions, and $\calB(U)\sub \calC(U)$ for all open sets $U\sub X$. 
Define
    \[
        \mathrm{Der}_{\calC, \calB}(U)
        \defeq
        \big\{
            v \in \gerX_X(U)
            \sep
            \gerL_{v|_V}(f) \in \calB(V)
            \;\;\forall\,
            f \in \calC(V),
            \;
            V\sub U
        \big\}.
    \]
Then $\mathrm{Der}_{\calC, \calB}$ is a Lie algebra sheaf.
Notable examples are:
\begin{enumerate}

\item[a)]
If $\calB(U)= \{ 0 \}$ for all $U\sub X$, it gives the annihilator of $\calC$:
    \[
        \mathrm{Ann}_\calC
        \defeq
        \mathrm{Der}_{\calC, \{0\}}.
    \]

\item[b)]
If $\calB = \calC$, 
it gives the sheaf derivations of $\calC$:
    \[
        \mathrm{Der}_\calC
        \defeq
        \mathrm{Der}_{\calC, \calC}.
    \]


\end{enumerate}
\end{example}

\begin{proof}
We will denote $\calA\defeq \mathrm{Der}_{\calC, \calB}$.

Let $u, v \in \calA(U)$ and $r \in \R$. 
Then for any open set $V\sub U$ and function $f \in \calC(V)$ we have
    \[
        \gerL_{(u+rv)|_V}(f)
        =
        \gerL_{u|_V}(f) + r \gerL_{v|_V}(f)
        \;\in\;
        \calB(V),
    \]
because we assumed $\calB(V)$ is an $\R$-linear subspace of $C^\infty(V)$. 
Moreover, by Cartan's formula,
    \[
        \gerL_{[u,v]|_V}(f)
        =
        (\gerL_{u|_V} \circ \gerL_{v|_V} - \gerL_{v|_V} \circ \gerL_{u|_V})(f)
        \;\in\;
        \calB(V).
    \]
Hence $u+rv$ and $[u,v]$ lie in $\calA(U)$, and we conclude that $\calA(U)$ is a Lie subalgebra of $\gerX_X(U)$.

$\calA$ is a presheaf by definition: suppose that $v\in \calA(U)$ and $W\sub U$ is an open subset.
Then for any open subset $V\sub W$, we also have $V\sub U$, hence for any $f\in \calC(V)$ we have $\gerL_{v|_V}(f) \in \calB(V)$.
Hence $v|_W \in \mathrm{Der}_\calC(W)$, so we conclude that $\calA$ is a presheaf.

Finally, $\calA$ is a sheaf by the local nature of the Lie derivative: 
suppose that $v\in \gerX_X(U)$ is such that for every $x\in U$ there is an open neighborhood $W\sub U$ of $x$ such that $v|_W\in \calA(W)$.
Let $V\sub U$ be an open subset, and let $f\in \calC(V)$.
Then for a given $x\in U$ choose such an open neighborhood $W\sub U$ of $x$.
Since $\calC$ is a presheaf, $f|_W\in \calC(W)$, and
    \[
        \gerL_{v|_V}(f)\big|_W
        =        
        \gerL_{v|_{V\cap W}}(f|_W) 
        \;\in\; 
        \calB(W).
    \]
Since $\calB$ is a sheaf, this implies $\gerL_{v|_V}(f) \in \calB(V)$.
Hence $v \in \calA(U)$, and we conclude that $\calA$ is a subsheaf of $\gerX_X$.
\end{proof}

\begin{lemma}
Let $\calC$ and $\calB$ be as above. 
Moreover, assume that $\calC$ is a sheaf and that $\calC(U)$ is a $C^\infty(U)$-module for every open set $U\sub X$.
Then
    \[
        \mathrm{Der}_{\calC, \calB}(U)
        =
        \big\{
            v \in \gerX_X(U)
            \sep
            \gerL_v(f) \in \calB(U)
            \;\;\forall\;
            f \in \calC(U)
        \big\}.
    \]
\end{lemma}

\begin{proof}
Let $v \in \gerX_X(U)$ be such that $\gerL_v(f)$ for all $f \in \calC(U)$.
Let $V\sub U$ be an open subset and $f \in \calC(V)$.
For a given $x\in V$, pick a smooth function $\xi: U \!\to\! [0,1]$ that is equal to one on a neighborhood of $x$ and is equal to zero on an open neighborhood of $U\backslash V$.
Since $\calC(V)$ is a $C^\infty(V)$-module, we have $\xi\cdot f\in \calC(V)$.
Moreover, $\xi\cdot f$ extends by zero to a smooth function on $U$ and, since $\calC$ is a sheaf, $\xi\cdot f \in \calC(U)$.
Hence
    \[
        \gerL_{v|_V}(f)\big|_W
        =
        \gerL_{v|_{V\cap W}}(f|_W)
        =
        \gerL_{v|_{V\cap W}}((\xi\cdot f)|_W)
        =
        \gerL_{v|_V}(\xi\cdot f)|_W
        \;\in\;
        \calB(W).
    \]
Since $\calB$ is a sheaf, it follows that $\gerL_{v|_V}(f) \in \calB(V)$.
\end{proof}

\subsection{Isotopies and Vector Fields}

We look at the notion of a smooth path in a pseudogroup. Clearly, such a path should assign to each time $t \in [0,1]$ an element $(U^t,\phi^t)$ of the pseudogroup, and do so in a coherent, smooth way. As usual when defining paths of functions, it is more convenient to consider maps 
    $
        [0,1]\times X\to X,
    $
satisfying some conditions, instead of maps $[0,1]\to C^\infty(X,X)$.

For an open set $U\sub [0,1] \!\times\! X $, and $(t, x) \in U$, introduce the notation
    \[
        U^t
        \defn
        U \cap \big( \set{t} \times X \big),
        \qquad
        U_x
        \defn
        U \cap \big([0,1]\times \set{x}\big).
    \]
Thinking of $t$ as time, we will assume that the open set $U^t$ can only shrink as time passes:
    \[
        U^t\sub U^s,
        \quad\forall\,
        s\leq t.
    \]
This is equivalent to the $U_x$ being intervals containing zero for all $x\in U^0$.

\begin{definition}

A \textbf{local isotopy}\index{local isotopy} is a smooth map $\phi:U\to X$, with ${U\sub [0,1]\times X}$ as above, for which $\phi^0=\id\!|_{U^0}$ and $\phi^t:U^t\to X$ is a local diffeomorphism for all $t$.

We say that a local isotopy $\phi$ lies in $\pgrp$ if $\phi^t \in \pgrp(U^t)$ for all $t$.
\end{definition}

The process of integration assigns to a local time-dependent vector field $v:V\to TX$, where $V\sub [0,1]\times X$, a local isotopy $\varphi_v:U\to X$ given by the ordinary differential equation 
    \[
        \frac{\dd}{\dd t}\varphi_v^t
        =
        v^t\circ \varphi_v^t,
        \qquad
        \varphi^0_v=\id\!|_{U^0}.
    \]
If one requires that $U$ is the largest domain on which $\varphi_v$ can be defined whilst satisfying this equation, then $\varphi_v$ becomes unique. This local isotopy $\varphi_v$ is called the total flow of $v$.

\begin{lemma}
\label{lemma:pgrp:one_parameter}
Let $\pgrp$ be a pseudogroup on $X$ and $v\in\gerX_X(V)$.
Then the total flow $\varphi_v^t$ of $v$ lies in $\pgrp$ if and only if for every $x\in U^0=V$ there is a neighborhood $W\sub U^0$ of $x$ and an $\epsilon > 0$ such that
    \[
        W\sub U^t,
        \quad
        \varphi^t_v\vert_W \in \pgrp(W),
        \quad\forall\;
        t\in [-\epsilon, \epsilon].
    \]
\end{lemma}

\begin{proof}

Suppose that $t>0$ and $x\in U^t$ are given. We will show that there is a neighborhood $W^t\sub U^t$ of $x$ such that $\varphi^t_v|_{W^t}$ lies in $\pgrp(W^t)$ by writing it as a finite composition of maps that we know lie in $\pgrp$. Since this then holds for any $x\in U^t$, it follows that $\varphi^t_v\in \pgrp(U^t)$. Finally, the case where $t<0$ follows by considering the vector field $-v$.

Note that $\varphi^t_v(U^t)\sub U^0$ for all $t$. By applying the hypothesis to $x^r\defeq\varphi^r_v(x)$ for all $r \in [0,t]$, we find a  neighborhood 
    $
        W(r)\sub U^0
    $
of $x^r$ and an $\epsilon(r) > 0$ such that
    \[
        W(r)\sub U^s,
        \quad
        \varphi_v^s \vert_{W(r)} \in \pgrp(W(r)),
        \qquad\forall\;
        s \in [-\epsilon(r), \epsilon(r)].
    \]
The open intervals 
    $
        \big(
            r-\epsilon(r), r+\epsilon(r)
        \big)
    $ 
cover $[0,t]$, so we can choose a finite subcover. 
This subcover can be chosen to be of the form
    \[
        I_i \defeq
        \big(
            r_i-\epsilon(r_i), r_i + \epsilon(r_i)
        \big),
        \quad
        W_i \defeq W(r_i),
        \qquad
        i=0,\ldots, k,
    \]
so that $r_0=0$ and $r_k=t$, and only consecutive intervals overlap. Now pick an intermediate time $q_i$ in each intersection:
    \[
        q_i 
        \in 
        I_i
        \cap 
        I_{i+1},
        \quad
        r_i < q_i < r_{i+1},
        \qquad
        i = 1, \ldots, k-1.
    \]
Since both
    $
        \varphi^{q_{i-1} - r_i}_v|_{W_i}
    $
and
    $
        \varphi^{q_i- r_i}_v|_{W_i}
    $
lie in $\pgrp(W_i)$,
by the group-like axioms of $\pgrp$ we have
    \[
        \varphi^{q_i-q_{i-1}}_v 
        =
        \varphi^{q_i-r_i}_v
        \circ
        (\varphi^{q_{i-1} - r_i}_v)^{-1} 
        \in \pgrp(W'_i)
        \qquad
        i=2,\ldots, k-1,
    \]
on the domain 
    $
        W'_i
        \defn
        \varphi^{ q_{i-1} - r_i }_v(W_{i+1}).
    $
By the composition axiom of $\pgrp$ we have that
    \[
        \varphi^t_v
        =
        \varphi^{t-q_{k-1}}_v
        \circ
        \varphi^{q_{k-1}-q_{k-2}}_v
        \circ
        \cdots
        \circ
        \varphi^{q_2-q_1}_v
        \circ
        \varphi^{q_1}_v
    \]
lies in $\pgrp$ on a neighborhood of $x$. These neighborhoods cover $U^t$, so by the sheaf-like axioms of $\pgrp$ it follows that
    $
        \varphi^t_v : U^t \to X
    $
lies in $\pgrp(U^t)$.
\end{proof}

\begin{question}
What are the \emph{connected components} of a pseudogroup?
\end{question}

\section{Closed Pseudogroups}

We introduce \emph{closed} pseudogroups, and show that they already share some interesting properties with Lie pseudogroups.

We start by recalling the compact-open $C^\infty$-topology on spaces of smooth functions, in particular related to $\diff_X$. This can be done conveniently using the language of jet bundles: at a point $x\in X$, consider all smooth maps $U\to Y$, defined on open neighborhoods $U\sub X$ of $x$, and mapping into a fixed manifold $Y$. Two such maps are $k$-jet equivalent if their value and their derivatives up to order $k$ agree at $x$, where their derivatives are computed by first choosing charts around $x$ and around their value at $x$. The equivalence class of a particular $f:U\to Y$ is called the $k$-jet of $f$ at $x$ and is denoted by $j^k_x(f)$. The $k$-th jet bundle is the bundle of all $k$-jets:
    \[
        J^k(X,Y)
        \defeq
        \medcup_{x\in X}
        \big\{
            j^k_x(f)
            \sep
            f\in C^\infty(U, Y),
            \;
            x\in U\sub X
        \big\}.
    \]
It is a locally trivial fiber bundle over $X$, and the assignment 
    \[
        j^k(f):x\longmapsto j^k_x(f)
    \]
is a smooth local section for any smooth map $f:U\to Y$. Jet bundles allow for coordinate free arguments about partial derivatives of smooth maps.

Let $U\sub X$ be an open subset. 
A basic $C^k$-open subset of $C^\infty(U, Y)$ consists of a compact subset $K\sub U$ and an open subset 
    $
        O
        \sub 
        J^k(X, Y)\vert_{K},
    $
and it is defined as the set
    \[
        \calO(K, O)
        \defn
        \big\{
            f \in C^\infty(U, Y)
            \sep
            j^k(f)(K) \sub O 
        \big\}.
    \]
The \textbf{compact-open $C^k$-topology} is generated by all basic $C^k$-open subsets, and the compact-open $C^\infty$-topology is the union of the compact-open $C^k$-topologies for all $k$.

The strong $C^k$-topology on $\diff_X(U)$ forgoes on the requirement that $K$ is compact: 
its opens are generated by the $\calO(K,O)$ for all $K\sub U$ that are \emph{closed} in $U$.

\begin{definition}
\label{def:closed_pseudogroup}
Call a pseudogroup $\pgrp$ \textbf{closed}\index{pseudogroup!closed} if $\pgrp(U)$ is a closed subset of $\diff_X(U)$ w.r.t.\  the compact-open $C^\infty$-topology for all open sets $U\sub X$.

Likewise, call a Lie algebra sheaf $\calA$ closed if $\calA(U)$ is a closed subset of $\gerX(U)$ w.r.t.\ the compact-open $C^\infty$-topology for all open sets $U\sub X$.
\end{definition}

\begin{example}
$\diff_\omega$ and $\diff_g$ from examples \ref{exa:pseudogroup_poisson} and \ref{exa:pseudogroup_riem} are closed pseudogroups. They all fall under the class discussed in section \ref{sec:jet_determinacy}.
\end{example}

\begin{example}
\label{exa:closed_intersection}
If $\pgrp$ and $\calQ$ are closed pseudogroups, then so is $\pgrp \cap \calQ$.
\end{example}


\begin{example}
\label{exa:closed_lie_group}
Suppose that $X= \R^m$ and $G\sub \mathrm{GL}_m$ is a closed subgroup, then the pseudogroup $\diff_G$ from example \ref{exa:lie_group} is closed.
\end{example}

\begin{proof}
From example \ref{exa:lie_group} we know that 
    \[
        \diff_G(U) = \{ \phi_g \sep g \in \underline{G}(U) \} \medcup \diff_X(U).
    \]
Suppose that $\phi_n \in \diff_G(U)$ is a sequence of local diffeomorphisms that converges to $\phi \in \diff_{\R^m}(U)$.
Then for every $x \in U$ there is an open neighborhood $V\sub U$ such that for all $n$ there is a unique $g_n \in G$ such that $\phi_n|_V = \rho(g_n)|_V$.
Since $\phi_n|_V$ converges to $\phi|_V$ w.r.t. the $C^1$-topology, it follows that $g_n = \dd_y\phi_n$ converges to $\dd_y \phi$ for all $y\in V$.
This shows that $g \defeq \dd_y \phi \in G$ and $\phi|_V = \rho(g)|_V$.
\end{proof}

\begin{example}
$\diff_\calC^\vee$ from example \ref{exa:pseudogroup_of_invariants} is a closed pseudogroup if $\calC$ is a closed subsheaf of the smooth functions on $X$. This follows from the fact that the composition map,
    \[
        c_f:
        \diff_X(U)
        \longto
        C^\infty(U)
        \sep
        \phi \mapsto f\circ \phi,    
    \]
is a $C^k$-to-$C^k$ continuous map for all $f \in C^\infty(\R)$ and all $k$, and so is the inversion map.
\end{example}

The compact-open $C^\infty$-topology is coarser than the strong topology.
So there is a potentially weaker notion of closed pseudogroup, where $\pgrp(U)$ should be a closed subset of $\diff_X(U)$ w.r.t.\  the strong $C^\infty$-topology.

Note that $\diff_X(U)$ is not a closed subset of $C^\infty(U,X)$, similar to how $\GL(\R^m)$ is not a closed subset of $\End(\R^m)$.
It is not an open set either, at least not for the compact-open topologies, but in lemma \ref{lemma:close_to_identity} we describe its open character.
We stress to the reader that $\pgrp$ is considered closed if $\pgrp(U)$ is closed as a subset of $\diff_X(U)$.

We will quickly discuss several other possible definitions and show why they are equivalent. 
To this end we discuss sequences in $\diff_X$, and what it means for a sequence to converge. Namely, a sequence of local diffeomorphisms assigns to every $n\in \N$ a local diffeomorphism $(U_n,\phi_n)\in \diff_X$, possibly varying the domain.

\begin{definition}
We say that a sequence $(U_n, \phi_n) \in \diff_X$ converges on an open subset $U\sub X$ if there exists a $\phi\in \diff_X(U)$ such that for all $x\in U$ there is an open neighborhood $U(x) \sub U$ of $x$ and an $N(x) \in \N$ such that 
    \[
        U(x) \sub U_n,
        \quad\forall\;
        n\geq N(x),
    \]
and such that the sequence $\phi_n\vert_{U(x)}$ converges to $\phi\vert_{U(x)} \in C^\infty(U(x), X)$ w.r.t.\  the compact-open $C^\infty$-topology.
\end{definition}

\begin{lemma}
\label{lemma:pseudogroup_limits}
For any sequence $ \phi_n \in \diff_X(U_n)$ there is a largest open set $U$ on which it converges.
\end{lemma}

\begin{proof}
Let $U$ be the union of all open sets on which $ \phi_n $ converges. 
Since the limit functions agree on overlaps, let $\phi \in C^\infty(U, X)$ be the unique function that extends all limits. 
Remark that the differential of $\phi$ is non-singular, so it suffices to show that $\phi$ is injective. 
Suppose not, and pick $x, y \in U$ such that $\phi(x)=\phi(y)$, and write $z=\phi(x)$. 
Now pick a compact ball $A\sub U$ around $x$ and a compact ball $B\sub U$ around $y$ that do not intersect, and such that $\phi$ is invertible on both $A$ and $B$. 
Then by assumption we have $A, B \sub U_n$ for large enough $n$. 

Now let $A'\sub \mathrm{int}(A)$ be a smaller compact ball around $x$. 
We claim that $\phi_n(A')\sub \phi(A)$ for large enough $n$. 
Suppose not, then there is a sequence $x'_n \in A'$ such that $\phi_n(x'_n) \notin \phi(A)$ for all $n$. 
Since $A'$ is compact, we may assume that $x'_n$ converges to some $x'\in A'$ possibly by passing to a convergent subsequence. 
On one hand we have that $\phi(x') \in \phi(A')$ as the limit of $\phi(x'_n)$, since $\phi(A')$ is closed, while on the other hand we have that $\phi(x') \notin \phi(\mathrm{int}(A))$ as the limit of $\phi_n(x'_n)$. 

By applying lemma \ref{lemma:inverse} to 
    \[
        (\phi|_A)^{-1} \circ (\phi_n\vert_{A'}):
        A' \longto A 
    \]
we obtain that $x\in (\phi|_A)^{-1}(\phi_n(A'))$ and thus that $z \in \phi_n(A)$ for large enough $n$. 
Similarly we have that $z \in \phi_n(B)$, but this contradicts the injectivity of the $\phi_n$.
\end{proof}

\begin{definition}
\label{def:lim_of_loc_seq}
The limit of a sequence $\phi_n \in \diff_X(U_n)$ is the local diffeomorphism $\phi \in \diff_X(U)$ such that the sequence converges to $\phi$ on $U$ and such that $U$ is maximal with this property.
We denote the limit by
    \[
        (U, \phi) = \lim_{n\to \infty} (U_n, \phi_n).
        \qedhere
    \]
\end{definition}

\begin{proposition}
\label{prop:pgrp:closed}
The following are equivalent for a pseudogroup $\pgrp$:
\begin{itemize}

\item $\pgrp$ is closed.

\item $\pgrp$ is closed under taking limits (as defined above).

\item If $\phi \in \diff_X(U)$ is such that for every $x\in U$ there is an open neighborhood $V\sub U$ of $x$ and a sequence $\phi_n \in \pgrp(V)$ converging to $\phi|_V$, then $\phi\in \pgrp(U)$.

\end{itemize}
\end{proposition}

\begin{proof}
Suppose that $\pgrp$ is closed, and that we are given a sequence $(U_n,\phi_n)\in \pgrp$ which converges to $(U,\phi)\in\diff_X$. 
For any $x\in U$, pick an open neighborhood $V\sub U$ whose closure is compact and still contained in $U$. 
By assumption, there is an $N\in \N$ such that $\overline V\sub U_n$, for all $n\geq N$, and $\phi_n$ $C^k$-converges to $\phi$ on $\overline V$. 
So we conclude that $\phi_n\vert_V$ converges to $\phi\vert_V$ in $\diff_X(V)$ w.r.t.\  the $C^\infty$-topology. 
Hence we have $\phi\vert_V\in \pgrp(V)$, and by the sheaf-like axioms of $\pgrp$, we also have $\phi\in \pgrp$.

Now suppose that the second statement holds, and that we are given a $(U,\phi)\in \diff_X$ as described in the third statement. It follows directly that for every $x\in U$ there is an open neighborhood $V\sub U$ of $x$ such that $\phi\vert_V\in \pgrp$. Hence $\phi\in\pgrp$ by the sheaf-like axioms of $\pgrp$.

Finally, suppose that the third statement holds. In particular, this implies that $\pgrp(U)$ is a sequentially closed subset of $\diff_X(U)$.  The topology on $\diff_X(U)$ is first countable. This is evident from its definition via the basis open sets $\calO(K,O)$ and the use of compact exhaustions. Hence it follows that $\pgrp(U)$ is a closed subset of $\diff_X(U)$.
\end{proof}

\subsection{Jet Determinacy}
\label{sec:jet_determinacy}

We briefly discuss the most common class of closed pseudogroups. 
Recall that the collation axiom of pseudogroups can be phrased as being determined by its germs: A local diffeomorphism belongs to the pseudogroup if its germ belongs to it at every point. 
This axiom can be replaced by something stronger by looking at jets instead of the germs.

Let $k\in \N\cup \{\infty\}$. Recall that the $k$-th jet groupoid $J^k\diff_X$ is the subset of $k$-jets that are represented by a local diffeomorphism. This subset has the structure of a Lie groupoid by the functorial properties of $k$-jets. Any local diffeomorphism $(U,\phi)$ defines a local bisection 
    \[
        j^k(\phi):U\to J^k\diff_X.
    \] 
A \textbf{partial differential relation}\index{partial differential relation} on the local diffeomorphisms is a subset
    \[
        R \sub J^k\diff_X.
    \]
A local solution to $R$ is a local diffeomorphism $(U,\phi)\in\diff_X$ such that $j^k(\phi)$ maps into $R$. 
We denote the set of solutions by
    \[
        \sol_R(U)
        \defeq
        \big\{
            \phi\in \diff_X(U)
            \sep
            j^k(\phi)(U)\sub R
        \big\}.
    \]

\begin{example}
The basic $C^k$-open subsets $\calO(K,O)$ are the solution sets of a partial differential relation. In those cases $R$ is the union of $O$ and the complement of $J^k\diff_X\!|_K$.
\end{example}

A pseudogroup $\pgrp$ on $X$ defines an \emph{abstract} subgroupoid of $J^k\diff_X$ by
    \[
        J^k\pgrp
        \defn
        \big\{
            j^k_x(\phi)
            \in J^k\diff_X
            \sep
            (U,\phi)\in\pgrp,
            \;
            x\in U
        \big\}.
    \]
Any abstract full subgroupoid $G\sub J^k\diff_X$ is an example of a partial differential relation, and its set of solutions
    \[
        \pgrp(G)(U)
        \defeq      
        \sol_G(U)
        =
        \big\{
            \phi \in \diff_X(U)
            \sep
            j^k_x(\phi) \in G_x,
            \;\;\forall\;
            x\in U
        \big\}
    \]
defines a pseudogroup on $X$.
By taking $G=J^k\pgrp$, we observe that 
    \[
        \pgrp \sub \pgrp(J^k\pgrp).
    \] 
    
\begin{definition}
\label{def:jet_determined}
A pseudogroup $\pgrp$ is called \textbf{$k$-jet determined}\index{pseudogroup!$k$-jet determined} if 
    \[
        \pgrp=\pgrp(J^k\pgrp),
    \]
In other words, when it satisifies that a $\phi \in\diff_X(U)$ belongs to $\pgrp(U)$ if and only if $j^k_x(\phi)\in J^k_x\pgrp$ for every $x\in U$.

Furthermore, we call it a \textbf{Lie pseudogroup}\index{pseudogroup!Lie} if it is $k$-jet determined for some $k\in \N$, and $J^k\pgrp$ is a closed \emph{Lie} subgroupoid of $J^k\diff_X$ (a smooth embedded submanifold).
\end{definition}

\begin{proposition}
If $\pgrp$ is $k$-jet determined and $J^k\pgrp$ is a closed subset of $J^k\diff_X$, for some $k\in \N\cup\{\infty\}$, then $\pgrp$ is a closed pseudogroup.
\end{proposition}

\begin{proof}
In fact, we show that $\pgrp$ is closed if for some $k\in\N\cup\{\infty\}$ it is $k$-jet determined and $J^k\pgrp$ has closed source-fibers in $J^k\diff_X$, i.e. $J^k_x\pgrp$ is a closed subset of $J^k_x\diff_X$ for all $x\in X$.

Suppose that $\pgrp$ is $k$-jet determined, for finite $k$, and $J^k\pgrp$ has closed source fibers. Let $(U,\phi)$ be a local diffeomorphism in the complement of $\pgrp(U)$. Since $\pgrp$ is $k$-jet determined, this means there is a $x\in U$ such that $j^k_x(\phi)\not\in J^k_x\pgrp$. Recall that $J^k_x\pgrp$ is a closed subset of $J^k_x\diff_X$, so there is an open neighborhood $O\sub J^k\diff_X$ of $j^k_x(\phi)$ which is disjoint from $J^k_x\pgrp$. The set
    $
        \calO(\{x\}, O)
    $
is a compact-open $C^k$-open subset of $\diff_X(U)$. Its elements are not local solutions of the partial differential relation $J^k\pgrp$, so they cannot be elements of $\pgrp(U)$. It follows that $\pgrp(U)$ is a closed subset of $\diff_X(U)$.

Now suppose $\pgrp$ is $\infty$-jet determined and $J^\infty\pgrp$ has closed source-fibers. The former means that a local diffeomorphism $(U,\phi)$ belongs to $\pgrp$ if and only if $j^k_x(\phi)\in J^k_x\pgrp$ for all $x\in U$ and $k\in \N$. The latter means that if $j^\infty_x(\phi) \not\in J^\infty_x\pgrp$, then there is an $k=k(\phi)\in \N$ and an open neighborhood $O\sub J^{k}_x\diff_X$ of $j^{k}_x(\phi)$ which is disjoint from $J^k_x\pgrp$. 
\end{proof}

\begin{question}
Are all closed pseudogroups jet determined?
\end{question}

\subsection{Closure}

By the closure of a pseudogroup $\pgrp$ we will mean the smallest closed pseudogroup containing $\pgrp$.
In account of example \ref{exa:closed_intersection}, the closure is 
    \[
        \overline{\pgrp}
        \defeq
        \bigcap
        \big\{
            \calQ
            \text{ closed pseudogroup with $\pgrp\sub \calQ$}
        \big\}.
    \]
We will attempt to realize the closure more constructively by adding limits.

\begin{lemma}
\label{lemma:pseudogroup_closure}
The closure under limits of a pseudogroup $\pgrp$,
    \[
        \overline{\pgrp}^\mathrm{top}(U)
        \defeq
        \big\{
            \lim_{n\to \infty} \phi_n \in \diff_X(U)
            \sep
            \phi_n \in \pgrp(U_n)
        \big\},
    \]
is a prepseudogroup. 
\end{lemma}

\begin{proof}
Suppose that $\phi_n\in \pgrp(U_n)$ converges to $\phi \in \diff_X(U)$ in the sense of definition \ref{def:lim_of_loc_seq}.
Let $V\sub U$ be an open subset. 
We claim that 
    \[
        \phi_n|_V \in \pgrp(U_n\cap V)
    \]
converges to $\phi|_V$, and thus that $\phi|_V \in \overline{\pgrp}^\mathrm{top}(V)$.
Surely, $\phi_n|_V$ converges to $\phi|_V$ on $V$.
Moreover, suppose $W$ is an open set such that $\phi_n|_V$ converges to some $\psi\in \diff_X(W)$ on $W$.
Then for any $x \in W$ on has that $x \in U_n\cap V$ for large enough $n$.
This already implies that $W\sub V$, and therefore $V$ is the largest open set on which $\phi_n|_V$ converges.

Suppose that moreover $\psi_n\in \pgrp(V_n)$ converges to $\psi \in \diff_X(V)$ in the sense of definition \ref{def:lim_of_loc_seq}.
We will show that $\phi_n\circ\psi_n \in \pgrp(\psi_n^{-1}(U_n))$ converges to $\phi\circ\psi \in \diff_X(\psi^{-1}(U))$ on 
    \[
        W\defeq\psi^{-1}(U).
    \]
Let $x \in W$. 
Then $x\in V$, hence there is an open neighborhood $A \sub V$ of $x$ such that $A\sub V_n$ for large enough $n$, and such that $\psi|_{A} \sub C^\infty(A, X)$ converges to $\psi|_{A}$.
Likewise, $\psi(x)\in U$, hence there is an open neighborhood $C\sub U$ of $\psi(x)$ such that $C\sub U_n$ for large enough $n$, and such that $\phi_n|_C \in C^\infty(C, X)$ converges to $\phi|_C$.
We may pick an open neighborhood $B\sub A$ of $x$ such that $\overline{B}$ is compact, and
    \[
        \overline{B} \sub A,
        \qquad
        \psi(\overline{B})\sub C.
    \] 
By the convergence of $\psi_n$, it follows that $\psi_n(\overline{B})\sub C$ for large enough $n$. 
Recall that the composition map
    \[
        C^\infty(C, X) \times C^\infty(B, C)
        \longto
        C^\infty(B, X)
    \]
is $C^k\!\times\!C^k$-to-$C^k$ continuous w.r.t. the compact open topologies for all $k$.
We conclude that the sequence $(\phi_n\circ\psi_n)|_B$ converges to $(\phi\circ\psi)|_B$ in $C^\infty(B, X)$, and this shows that  
    $
        (\phi_n \circ \psi_n)|_W
        \in
        \pgrp(\psi_n^{-1}(U_n))
    $
converges to $\phi\circ\psi$ on $W$.
Hence 
    $
        \phi\circ\psi \in \overline{\pgrp}^\mathrm{top}(W).
    $

Now we will show that $\phi_n^{-1}$ converges to $\phi^{-1}$ on $\phi(U)$.
\end{proof}

\begin{question}
When is $\overline{\pgrp}^\text{top}$ equal to the closure $\overline{\pgrp}$? And what about its sheafification?
\end{question}

\section{The Lie Algebra Sheaf of a Closed Pseudogroup}
\label{sec:lie_algebra_sheaf}

We define the Lie algebra sheaf associated with a closed pseudogroup. For pseudogroups in general, there are several candidates, but each fails to be a Lie algebra sheaf. In the case of a closed pseudogroup these candidates coincide and form what we will call the Lie algebra sheaf $\calA$ of a closed pseudogroup $\pgrp$.
The first candidate is the set of all local vector fields whose total flow lies in $\pgrp$:
    \[
        \calA_\pgrp^\text{flow}
        \defeq
        \big\{
            v \in \gerX_X
            \sep
            \varphi_v^t \in \pgrp(U^t)
            \;\;\forall\; 
            t\in \R
        \big\}.
    \]
Here $\varphi_v^t:U\to X$ denotes the total flow of $v$. However, by lemma \ref{lemma:pgrp:one_parameter}, one can weaken the maximality condition, and instead only demand that there is an open neighborhood $V$ of $\{ 0\}\times U^0$ such that $\varphi_v^t\vert_{V^t}$ lies in $\pgrp$ for all $t$.
Moreover, its vectors induce one-parameter subpseudogroups of $\pgrp$. Many basic properties about Lie groups related to integration can be proven by studying one-parameter subgroups. However, we believe that $\calA_\pgrp^\text{flow}(U)$ is not a Lie subalgebra in general: it may fail to be closed under the Lie bracket and under addition.

\begin{lemma}
Let $U, V\sub X$ be open sets.
\begin{itemize}

\item
$\calA_\pgrp^\text{flow}$ is a subsheaf of $\gerX_X$.

\item
If $v\in \calA^\text{flow}_\pgrp(V)$ and $\phi\in \pgrp(U)$ then also $\phi_*v\in \calA(\phi(V))$.

\end{itemize}
\end{lemma}

\begin{proof}
That $\calA^\text{flow}_\pgrp$ is a sheaf follows immediately from lemma \ref{lemma:pgrp:one_parameter}.
Because flows behave functorially, we have
    $
        \varphi_{\phi_*v}^t
        =
        \phi \circ \varphi_v^t \circ \phi^{-1}
        \in \pgrp
    $
for all $\phi\in \pgrp$ and $t\in \R$.
This implies that $\phi_*v \in \calA^\text{flow}_\pgrp(\phi(V))$, and thus that $\calA^\text{flow}_\pgrp$ is invariant under push-forwards by $\pgrp$.
\end{proof}

The second candidate is the set of all local vector fields that are the first order approximation of a local isotopy in $\pgrp$. This should not necessarily give a subsheaf of the local vector fields, so we will work with its sheafification. Namely, for an open set $V\sub X$, define
    \[
        \calA_\pgrp(V)
        \sub 
        \gerX_X(V)
    \]
as follows: a vector field $v\in \gerX_X(V)$ lies in $\calA_\pgrp(V)$ if and only if for every $x\in V$ there is an open neighborhood $U\sub V$ of $x$ and a local isotopy $\phi^t \in \pgrp(U)$, for $t\in (-\epsilon, \epsilon)$, such that
    \[
        v|_U
        =
        \tfrac{\dd}{\dd t}\big|_{t=0} \phi^t.
    \]
It has the following properties.

\begin{lemma}
\label{lemma:sheaf_of_vector_spaces}
$\calA_\pgrp(U)$ is a vector subspace of $\gerX_X(U)$ for all open sets $U\sub X$. 
\end{lemma}

\begin{proof}
Closure under scalar multiplication is immediate, since one can rescale time by a scalar and then apply the chain rule. 
Addition follows from composing local isotopies. For suppose that $v,w\in \calA_\pgrp$. Given $x\in X$, fix local isotopies $\Phi$ for $v$ and $\Psi$ for $w$ as in the definition of $\calA_\pgrp$. Let $v^t$ and $w^t$ be the time-dependent vector fields generating $\Phi^t$ and $\Psi^t$ respectively. Note that $v=v^0$ and $w=w^0$.
The composition $\Phi^t\circ \Psi^t$ is still a local isotopy, it is defined on a neighborhood of $(0,x)$ in $\R\!\times\! X$, and it satisfies
    \[
        \tfrac{\dd}{\dd t}
        \big( 
            \Phi^t \circ \Psi^t
        \big)
        =
        v^t \circ \Phi^t \circ \Psi^t 
        + 
        D\Phi^t \circ w^t \circ \Psi^t
        =
        \big(
            v^t + (\Phi^t)_* w^t
        \big) 
        \circ 
        \big(
            \Phi^t \circ \Psi^t
        \big).
    \]
In particular, $v+w=v^0+(\Phi^0)_*w^0$ is an element of $\calA_\pgrp$.
\end{proof}

Observe that 
    $
        \calA_\pgrp^\text{flow}\sub\calA_\pgrp
    $
for any pseudogroup $\pgrp$.
If $\pgrp$ is closed, much more is true.

\begin{theorem}
\label{pgrp:thm}
Let $\pgrp$ be a closed pseudogroup on $X$.

\begin{itemize}

\item 
$\calA_\pgrp^\textrm{flow} = \calA_\pgrp$.

\item 
$\calA_\pgrp$ is a Lie algebra sheaf.

\item 
If $v\in \calA_\pgrp$ and $\phi\in\pgrp$, then $\phi_*v \in \calA_\pgrp$.

\item 
$\calA_\pgrp(V)$ is a closed subset of $\gerX_X(V)$ for all open sets $V\sub X$.

\end{itemize}

\end{theorem}

\begin{proof}
The proof presented here uses the lemmas from section \ref{sec:intermezzo_closed_pseudogroups} and lemma \ref{lemma:pgrp:one_parameter}.
Suppose that $\pgrp$ is closed. Given $v\in\calA_\pgrp(V)$ we need to show that $\varphi^t_v\in \pgrp(U^t)$ for all $t\in\R$, where $\phi_v^t :U^t \to X$ is the total flow of $v$.
By lemma \ref{lemma:pgrp:one_parameter} it suffices to show that for all $x\in V$ there is a neighborhood $W\sub U^0$ of $x$ and an $\epsilon > 0$ such that, for all $t\in (-\epsilon, \epsilon)$,
    \[
        \varphi_v^t\vert_W
        \in
        \pgrp(W).
    \]
Lemma \ref{lemma:time_independent_flows} tells us exactly that: for a given $x\in U^0=V$, pick a local isotopy $\Phi^t$ as given by the definition of $\calA_\pgrp$. 
By that Lemma, the sequence
    $
        \big(
            \Phi^{t/n}
        \big)^n|_B
        \in \pgrp(B)
    $
converges uniformly to $\varphi^t_v|_B$ on a closed ball centered at $x$ in all $C^k$-norms. It follows that $v\in \calA_\pgrp^\text{flow}(V)$.
With the first statement of the theorem proven, the second (except for closure under the Lie bracket) and the third follow from the above lemmas.

Next we will show that $\calA_\pgrp(U)$ is a closed subset of $\gerX_X(U)$. Suppose that $v_n\in \calA_\pgrp(U)$ is a sequence of vector fields that converges to $v\in \gerX_X(U)$ w.r.t.\ the $\textrm{co}^\infty$-topology. Then we must show that $\varphi_v^t \in \pgrp(U^t)$ for all $t$, and by Lemma \ref{lemma:pgrp:one_parameter} it is sufficient to show that for every $x\in U^0$ there is a neighborhood $W\sub U^0$ of $x$ and an $\epsilon>0$ such that $\varphi^t_v|_W \in \pgrp(W)$ for all $t\in [-\epsilon, \epsilon]$. Now choose a chart at $x$, and let $B$ and $C$ be closed balls centered at $x$ such that $B\sub \mathrm{int}(C)$ and $C\sub U^0$. Since the $v_n|_C$ converge to $v|_C$ in all $C^k$-norms, it follows from lemma \ref{lemma:closed_lie_algebra_sheaf} that there is an $\epsilon>0$ such that, for all $t\in (-\epsilon, \epsilon)$, the $\varphi^t_v|_B$ and $\varphi^t|_B$ are defined as maps $B\to C$, and the sequence $\varphi^t_{v_n}|_B$ converges to $\varphi^t_v|_B$ in all $C^k$-norms. Since $\pgrp$ is closed, this shows that $\phi_v^t|_{\mathrm{int}(B)} \in \pgrp(\mathrm{int}(B))$ for all $t\in (-\epsilon, \epsilon)$, and thus that $v\in \calA_\pgrp(V)$.

We have already shown that $\calA_\pgrp$ is a sheaf of vector subspaces of $\gerX_X$. To show that is also closed under the Lie bracket we will use all of the above results. Let $v,w\in\calA_\pgrp(V)$, and let $\varphi^t_v:U^t\to X$ be the total flow of $v$. Define a homotopy of local vector fields by
    \[
        \zeta^t 
        \defeq
        \tfrac 1 t
        \big(
            (\varphi_v^t)_*w
            -
            w|_{V^t}
        \big),
        \qquad
        V^t \defeq \varphi^t_v(U^t)\sub V.
    \]
Note that $\zeta^t \in\calA_\pgrp(V^t)$ for all $t$: $\calA_\pgrp$ is invariant under push-forwards by elemets of $\pgrp$, so 
    $
        {(\varphi_v^t)_*w \in \calA_\pgrp(V^t)},
    $
and $\calA_\pgrp$ is closed under addition and scalar multiplication, so $\zeta^t\in \calA_\pgrp(V^t)$. 
In the limit as $t$ goes to zero we find
    \[
        \lim_{t\to 0}
        \zeta^t
        =
        \tfrac{\dd}{\dd t}\big\vert_{t=0}
        (\varphi_v^t)_*w
        =
        [v,w].
    \]
Since $\calA_\pgrp(V)$ is a closed subset of $\gerX_X(V)$, it must contain $[v, w]$.
\end{proof}

\begin{remark}
There is an alternative way of proving that $\calA_\pgrp$ is closed under the Lie bracket, which is more similar to lemma \ref{lemma:sheaf_of_vector_spaces}.
However, it requires us to introduce local $C^1$-isotopies, 
i.e.\ the maps $\phi:U\to X$, with $U\sub [0,1]\!\times\! X$, such that:
\begin{itemize}

\item 
$\phi:U\to X$ is $C^1$-smooth.

\item 
$\phi^t:U^t\to X$ and $\tfrac{\dd}{\dd t}\phi^t:U^t\to X$ are $C^\infty$-smooth for all $t$.

\end{itemize} 
Moreover, we have to enlarge $\calA_\pgrp$ to allow for local $C^1$-isotopies: a vector field $v\in \gerX_X(V)$ lies in $\calA^1_\pgrp(V)$ if and only if for every $x\in V$ there is an open neighborhood $U\sub V$ of $x$ and a local $C^1$-isotopy $\phi^t\in \pgrp(U)$, for $t\in (-\epsilon, \epsilon)$, such that 
    $
        v|_U
        =
        \tfrac{\dd}{\dd t}
        \big|_{t=0}
        \phi^t.
    $

Note that in general 
    $
        \calA_\pgrp^\text{flow}
        \sub
        \calA_\pgrp
        \sub
        \calA_\pgrp^1.
    $
The following argument actually shows that $\calA_\pgrp$ is closed under the Lie bracket if $\calA_\pgrp^\text{flow}=\calA_\pgrp^1$, but does not require that $\pgrp$ is closed.

Let $v, w\in \calA_\pgrp^\text{flow}(V)$, and let $\varphi_v^t$ and $\varphi_w^t$ be their respective total flows. Define the local isotopy
    \[
        \phi^t
        \defeq
        \varphi_v^t 
        \circ
        \varphi_w^t
        \circ
        \varphi_v^{-t}
        \circ
        \varphi_w^{-t}.
    \]
Its first derivative along $t$ is
    \[
        \tfrac{\dd}{\dd \epsilon} \big|_{\epsilon=0} 
        \phi^{t+\epsilon} \circ (\phi^t)^{-1}
        =
        v
        +
        (
            \varphi_v^t
        )_* w
        -
        (
            \varphi_v^t \circ \varphi_w^t
        )_* v
        -
        (
            \varphi_v^t \circ \varphi_w^t
            \circ
            \varphi_v^{-t}
        )_* w.
    \]
In particular, evaluating at $t=0$ gives $\frac{\dd}{\dd t}\big|_{t=0}\phi^t = 0$, hence $\phi^t$ has no linear term in the $t$ variable. Moreover, further computations show that 
    \[
        \tfrac{\dd^2}{\dd t^2}\big|_{t=0}
        \phi^t 
        = 
        [v,w].
    \]
It follows that $\phi^{\sqrt t}$ is a local $C^1$-isotopy that demonstrates that $[v,w]\in \calA_\pgrp^1$. 
We conclude that $\calA_\pgrp$ is closed under the Lie algebra bracket if $\calA_\pgrp^\text{flow}=\calA_\pgrp^1$.
\end{remark}

\section{Natural Bundles}
\label{sec:natural_bundle}

A natural bundle \cite{EpsThu79, Nij72, PalTer77, Gro86} on $M$ is often described as a submersion $\pi_\textsc{X}:X \!\to\! M$ with a lift of the local diffeomorphisms on $M$.
We look at a variation on this, where only a given pseudogroup $\pgrp$ on $M$ is lifted, and we show that if the lift is continuous and $\pgrp$ is closed then then this induces a closed pseudogroup on the total space of the bundle.

\begin{definition}
A \textbf{$\pgrp$-natural bundle}\index{natural bundle} on $M$ is a submersion $\pi_\textsc{X}: X \!\to\! M$ with a map
    \[
        E: \pgrp \longto \diff_X
    \]
such that $E$ respects the group-like and sheaf-like structures:

\begin{itemize}

\item[a)]
$\phi\in \diff_M(U)$ lifts to a bundle map $E(\phi)$ over $\phi$, i.e.\
    \[
        \begin{tikzcd}
            \pi_\textsc{X}^{-1}(U)
            \arrow[swap]{d}
            \arrow{r}{E(\phi)}
            &
            X
            \arrow{d}
            \\
            U
            \arrow{r}{\phi}
            &
            M.
        \end{tikzcd}
    \]

\item[b)]
The identity $\id\!|_U$ is lifted to the identity, i.e.\
    $
        E(\id\!|_U) 
        = 
        \id\!|_{\pi_\textsc{X}^{-1}(U)}.
    $

\item[c)]
Composition is preserved, i.e.\ $E(\phi\circ\psi) = E(\phi) \circ E(\psi)$ for all $\phi,\psi\in \pgrp$.
\end{itemize}

And such that $E$ respects the smooth structure:

\begin{itemize}

\item[d)]
$E$ is continuous w.r.t. to the compact-open $C^\infty$-topologies: if a sequence of local diffeomorphisms $\phi_n \in \pgrp(U_n)$ converges to $\phi \in \pgrp(U)$ on $U$, then $E(\phi_k)$ converges to $E(\phi)$ on $\pi_\textsc{X}^{-1}(U)$.

\item[e)]
$E$ is smooth: if $\phi^t \in \pgrp(U^t)$ is a smooth family of local diffeomorphisms, then $E(\phi^t)$ is also a smooth family.
\qedhere

\end{itemize}
\end{definition}

Let $\calD$ be the set of all charts into $M$. 
More precisely, for any open set $U\sub \R^m$, let
    $
        \calD(U)
    $
be the set of all open embeddings $U\!\hookrightarrow\!M$.
Since $\calD$ represents the atlas of $M$, it is the means of constructing anything on $M$.
For $d\geq 1$ or $d=\infty$, its space of $d$-jets $J^d\calD$ is a bibundle
    \[
        \R^m \xleftarrow{\;\; l \;\;} J^d\calD \xrightarrow{\;\; r \;\;} M.
    \]
Denote its fibers by $J^d_{x,y}\calD \defeq l^{-1}(x)\medcap r^{-1}(y)$ and $J^d_x\calD \defeq l^{-1}(x)$.

In accordance with \cite{EpsThu79}, for $d=0$, define $J^0_{x,y}\calD$ to be the space of equivalence classes of pointed open embeddings $(\R^m,x) \!\to\! (M,y)$ up to isotopies of $\R^m$ that keep $x$ fixed.
It also forms a (topological) bibundle $J^0\calD$ over $\R^m$ and $M$.
We will not discuss $J^0\calD$ other than to recite the main result of \cite{EpsThu79}.

\begin{lemma}
Any local diffeomorphism $\phi \in \diff_M(U)$ lifts to a local diffeomorphism 
    \[
        \phi^{(d)}:
        J^d\calD \big|_U \longto J^d\calD,
        \quad
        j^d_x(\sigma) \longmapsto j^d_x(\phi\circ\sigma).
    \]
This makes $J^d\calD$ a natural bundle over $M$ for all $d\geq 1$ or $d=\infty$.
\end{lemma}

\begin{proof}
For all $j^d_x(\sigma) \in J^d\calD$ with $\sigma(x) \in U$ we have
    \[
        (\id\!|_U)^{(k)}(j^d_x(\sigma))
        =
        j^d_x(\id\!|_U \circ \sigma)
        =
        j^d_x(\sigma).
    \]
Likewise, for all $\phi, \psi \in \diff_M$  we have
    \begin{align*}
        (\psi^{(d)} \circ \phi^{(k)})(j^d_x(\sigma))
        & =
        \psi^{(d)} 
        \big( 
            j^d_x(\phi\circ\sigma) 
        \big)
        \\
        & =
        j^d_x 
        \big( 
            \psi\circ\phi\circ\sigma
        \big)
        =
        (\psi\circ\phi)^{(d)}(j^d_x(\sigma)).
    \end{align*}
Axioms $d)$ and $e)$ are equally straightforward.
\end{proof}






The proof of the following lemma is immediate.

\begin{lemma}
\label{lemma:natural_bundle_quotients}
Let $X \!\to\! M$ be a $\pgrp$-natural bundle and $G$ a Lie group that acts on $X$ s.t.\
\begin{itemize}

\item 
The action of $G$ on $X$ is proper and free, and $\pi_\textsc{X}(x\cdot g) = \pi_\textsc{X}(x)$.

\item
For all $\phi \in \pgrp(U)$, $x\in \pi_\textsc{X}^{-1}(U)$, and $g\in G$ we have
    \[
        E(\phi)(x \cdot g) = E(\phi)(x)\cdot g.
    \]

\end{itemize}
Then the quotient $X / G$ is also a $\pgrp$-natural bundle.
\end{lemma}

All natural bundles (the $\pgrp$-natrual bundles with $\pgrp=\diff_M$) are known to be of the following type \cite{EpsThu79}.

\begin{example}
\label{exa:associated_jet_bundle}
Recall that $J^d\diff_{\R^m}$ is a transitive Lie groupoid over $\R^m$, and denote its isotropy group (at the origin) by $G_{m, d}$.
Moreover, define $G_{m,0}$ to be equivalence classes up to isotopy of pointed diffeomorphisms $(\R^m,0)\!\to\!(\R^m,0)$, similar to $J^0\calD$.

Remark that $J^d_{0}\calD\!\to\! M$ is a principle $G_{m, d}$-bundle over $M$.
Suppose that $F$ is a smooth manifold and that $G_{m, d}$ acts smoothly on it from the right.
Then the associated bundle
    \[
        X \defeq F \times_{G_{m, d}} J^d_{0}\calD
    \]
defines a natural bundle over $M$ for all $d\geq 1$ and $d=\infty$.
Remark that $J^d\calD$ restricts to a natural bundle $J^d_x\calD$ for any $x\in \R^m$, and this extends to a natural bundle $\widetilde X\defeq F \!\times\! J^d_0\calD$ by
    \[
        \widetilde E:
        \diff_M(U) \longto \diff_{\widetilde X}(F \!\times\! r^{-1}(U)),
        \quad
        \widetilde E(\phi) \defeq \id_F \times \phi^{(d)}.
    \]
Then the quotient $X = \widetilde X/ G_{m,d}$ is a natural bundle over $M$ by lemma \ref{lemma:natural_bundle_quotients}.
\end{example}

\begin{definition}
Let $E\pgrp$ be the pseudogroup on $X$ generated by the image of $E$, i.e.\ for any open subset $U\sub X$, define
    \[
        E\pgrp(U)
        \sub 
        \diff_X(U)
    \]
as follows: 
a local diffeomorphism $\phi\in \diff_X(U)$ lies in $E\pgrp(U)$ if and only if for every $y\in U$ there is an open neighborhood $V\sub U$ of $y$ and a local diffeomorphism $\varphi \in \pgrp(\pi_\textsc{X}(V))$ such that $\phi|_V = E(\varphi)|_V$. 
\end{definition}

\begin{lemma}
If $U\sub X$ is an open subset with connected fibers, then
    \[
        \pgrp(\pi_\textsc{X}(U))
        \longto
        E\pgrp(U),
        \quad
        \varphi
        \longmapsto
        E(\varphi)|_U,
    \]
is a bijection.
\end{lemma}

\begin{proof}
Note that the map is injective for any open subset of $X$.

All difeomorphisms $\phi\in E\pgrp$ are by definition locally fiber-preserving,
and since $U$ has connected fibers, this implies that any $\phi \in E\pgrp(U)$ is fiber-preserving.
If $\phi\in E\pgrp(U)$ has underlying diffeomorphism $\varphi \in \diff_M(\pi_\textsc{X}(U))$, then $\varphi\in \pgrp(\pi_\textsc{X}(U))$ and $\phi = E(\varphi)\vert_U$. 
Indeed, we know that for any $y\in U$ there is a neighborhood $V\sub U$ of $y$ and a 
    $
        \varphi'\in \pgrp(\pi_\textsc{X}(V))
    $ 
such that $\phi\vert_V = E(\varphi')\vert_V$. 
The left-hand side overlies $\varphi\vert_{\pi_\textsc{X}(V)}$ while the right-hand side overlies $\varphi'$.
Hence $\varphi\vert_{\pi_\textsc{X}(V)} = \varphi' \in\pgrp(\pi_\textsc{X}(V))$. 
This holds for any $y\in U$, so by the sheaf-like axioms we have $\varphi\in\pgrp(\pi_\textsc{X}(U))$. 
It is now clear that $\phi=E(\varphi)\vert_U$, and thus that the map is surjective.
\end{proof}

\begin{lemma}
If $\pgrp$ is closed, then so is $E\pgrp$.
\end{lemma}
\begin{proof}
Suppose that $(U_n, \phi_n) \in E\pgrp(U)$ converges to $(U, \phi) \in \diff_{\pi_\textsc{X}}(U)$. 
Let $V\sub U$ be an open set connected fibers and $\overline{V}\sub U$, and such that $\overline{V}$ is compact.
By the convergence, we have that $\overline{V}\sub U_n$ for large enough $n$, and by the above we know that $\phi_n|_V$ and $\phi|_V$ are fiber-preserving.
We will show that $\phi\vert_V$ lies in the image of $E$.
By the above we know that 
    \[
        \phi_n|_V = E(\varphi_n)|_V
    \]
for large enough $n$, where $\varphi_n \in \pgrp(\pi_\textsc{X}(V))$ is the local diffeomorphism underlying $\phi_n|_V$. If $\varphi$ is the local diffeomorphism underlying $\phi|_V$, it follows that $\varphi_n$ converges to $\varphi$ on $\pi_\textsc{X}(V)$.
Namely, for any $x\in \pi_\textsc{X}(V)$, pick a small open neighborhood $W\sub \pi_\textsc{X}(V)$ of $x$ and a local section $\sigma \in \Gamma_V(W)$. 
Then 
    \[
        \varphi_n|_W = \pi_\textsc{X}\circ \phi_n \circ \sigma
    \]
converges to $\varphi|_W$ w.r.t.\ the compact-open $C^\infty$-topology. 
By the collation axiom and because $\pgrp(W)$ is closed in $\diff_M(W)$, it follows that $\varphi \in \pgrp(V)$. 
Since $E$ and the restriction map are continuous, it follows that $\phi_n|_V = E(\varphi_n)|_V$ converges to both $\phi|_V$ and $E(\varphi)|_V$. Since the compact-open $C^\infty$-topology on $\diff_X(U)$ is Hausdorff, we conclude that $\phi|_V = E(\varphi)|_V$.
\end{proof}

%
%
%
%
%
%

\begin{lemma}
If $U\sub X$ is an open subset with connected fibers, then
    \[
        \calA_\pgrp(\pi_\textsc{X}(U))
        \longto
        \calA_{E\pgrp}(U),
        \quad
        v
        \longmapsto
        \tfrac{\dd}{\dd t}\big|_{t=0}
        E(\varphi_v^t)|_U
    \]
is a bijection.
\end{lemma}

\begin{proof}
Note that the map is injective for any open subset of $X$.

For any $w\in \calA_{E\pgrp}(U)$, its flow $\varphi_w^t\in E\pgrp(U^t)$ is locally fiber-preserving.
This implies that
    $
        \dd(\pi_\textsc{X}\circ\varphi_w^t)v=0,
    $
and thus 
    $
        \dd(\dd\pi_\textsc{X} \circ w)v = 0,
    $
for all vertical vectors $v\in T^\mathrm{v}X|_U$.
Since $U\sub X$ has connected fibers, it follows that $\dd\pi_\textsc{X}\circ w$ is constant along the fibers of $U$.
This implies there is a unique smooth local vector field $v\in \gerX_M(\pi_\textsc{X}(U))$ that is $\pi_\textsc{X}$-related to $w$, i.e.\
    \[
        \dd\pi_\textsc{X}\circ w = v \circ\pi_\textsc{X}.
    \]
For any $x\in U$ there is an open neighborhood $W\sub U$ of $x$, and an $\epsilon > 0$, such that $W\sub U^t$ for all $t\in (-\epsilon, \epsilon)$ and such that $W$ has connected fibers.
This means that $\varphi_w^t|_W \in E\pgrp(W)$ is fiber-preserving, and thus there is a family of diffeomorphisms $\psi^t\in \pgrp(\pi_\textsc{X}(W))$ such that 
    \[
        E(\psi^t)|_W = \varphi_w^t|_W,
    \]
for all $t \in (-\epsilon, \epsilon)$.
It easy to see that $\psi^t$ depends smoothly on $t$ and that $v|_{\pi_\textsc{X}(W)} = \tfrac{\dd}{\dd t}\big|_{t=0} \psi^t$.
This implies that $v|_{\pi_\textsc{X}(W)} \in \calA_\pgrp(\pi_\textsc{X}(W))$.
Moreover, because $\varphi_v^t|_{\pi_\textsc{X}(W)} = \psi^t$ and by the axioms of $E$, we have that
    \[
        E(\varphi^t_v)|_W 
        = 
        E(\varphi^t_v|_{\pi_\textsc{X}(W)})|_W 
        =
        E(\psi^t)|_W
        =
        \varphi_w^t|_W. 
    \]
This holds near all $x$, so we conclude that $v\in \calA_{\pgrp}(\pi_\textsc{X}(U))$ and that 
    \[
        E(v)|_U = w.
        \qedhere
    \]
\end{proof}

Epstein \& Thurston \cite{EpsThu79} phrase natural bundles differently: Let $0\leq s\leq r\leq \infty$. For them, a natural bundle is a functor $E$ that maps a $C^r$-manifold $M$ to a $C^s$-manifold $EM$ with a $C^s$-smooth map  
    \[
        \pi_\textsc{EM}:EM \longto M,
    \]
and that maps a $C^r$-smooth map $f:M \!\to\! N$ to a $C^s$-smooth bundle map
    \[
        E(f):EM\longto EN
    \]
over $f$. 
Moreover, they impose the following locality axiom:

\begin{itemize}
\item
    $
        {E(\id\!|_U): EU \!\to\! EM}
    $
is a $C^s$-diffeomorphism onto $\pi_{E(M)}^{-1}(U)$ for any open set $U\sub M$.
\end{itemize}

They prove in order that $\pi_\textsc{EM}$ is a locally-trivial fiber bundle over $M$, that the lift is continuous, and that the natural bundle is an associated jet bundle as described in example \ref{exa:associated_jet_bundle}.

We are only interested in the case where $r=s=\infty$.
Moreover, we restrict $E$ to the category of open subsets of a fixed manifold $M$, denote $X = EM$, and tacitly identify 
    $
        EU 
        = 
        \pi_X^{-1}(U)
    $ 
via the locality axiom.
Recall proposition 7.1 from \cite{EpsThu79} on the continuity of $E$.

\begin{CitedProposition}[{{\cite[proposition 7.1]{EpsThu79}}}]
Let $\phi_m:(\R^n,x) \!\to\! (\R^n,y)$ be a sequence of $C^r$-embeddings such that the $r$-jet of $\phi_m$ at $x$ converges to the $r$-jet of $\phi$ at $x$.
Then $E(\phi_m)|_{\pi^{-1}(x)}$ converges to $E(\phi)|_{\pi^{-1}(x)}$ in the compact-open $C^0$-topology.

In particular, if $\phi_0$ and $\phi_1$ have the same $r$-jet at $x$, then $E(\phi_0)|_{\pi^{-1}(x)} = E(\phi_1)|_{\pi^{-1}(x)}$.
\end{CitedProposition}

\begin{corollary}
If $E$ is a natural bundle in the sense of \cite{EpsThu79} and $\pgrp$ is a jet determined pseudogroup on $M$, then $E\pgrp$ is also jet determined.
\end{corollary}

In remark (2) of \cite{EpsThu79}, Epstein \& Thurston ask if their argument also works for other categories of manifolds, for example symplectic manifolds (and their symplectomorphisms).

\begin{question}
Do axioms (d,e) follow from axioms (a,b,c) for any pseudogroup?
\end{question}

We suspect that one should require the following properties:
\begin{enumerate*}

\item
the pseudogroup is closed,

\item
the pseudogroup is jet-determined,

\item
around any point of $M$ there is a chart such that the pull-back of the pseudogroup along this chart contains all translations.

\end{enumerate*}

\chapter{PDEs with Symmetry}
\label{chapter:PDEs with Symmetry}
\etocsettocdepth{2}
\localtableofcontents

{\color{white}.}
\newline
Many problems in differential geometry, and especially normal form results for various geometric structures, can be reduced to the meta-problem of showing that a fixed solution to a partial differential equation is rigid under a natural group of symmetries. 
To make this precise we will first introduce the main ingredients of this thesis: 
the notion of a partial differential equation with symmetry.

\begin{notation}
\label{notation}
Let $X$ and $Y$ be two smooth manifolds, and let $\mathcal{S}$ be a subpresheaf of the sheaf of smooth maps from $X$ to $Y$, i.e.\ an assignment    
    \[
        \mathcal{S}: U\longmapsto \mathcal{S}(U),
    \]
which associates to an open set $U\sub X$ a subset of $C^\infty(U,Y)$ so that, if $V\sub U$ are open sets, then $f|_V\in \mathcal{S}(V)$ for all $f\in \mathcal{S}(U)$.

We will adopt the following notation: 
if $Z\sub X$ is any subset, we denote by $\mathcal{S}(Z)$ the set of maps of the form $f|_Z$, where $f\in \mathcal{S}(U)$ for some open set $U$ containing $Z$.
This notation will be used throughout the thesis, and several such subpresheaves occur.
\end{notation}

\clearpage
\section{Pseudogroups as Symmetries}

We fix smooth manifolds $X$ and $M$ and a submersion
    \[
        \pi_\textsc{x}\sep X\longto M.
    \]
In this section we describe how a pseudogroup $\pgrp$ on the total space $X$ induces an action on the sheaf $\Gamma_X$ of local sections of $\pi_\textsc{X}$. 
In this way we can think of $\pgrp$ as a set of symmetries of $\Gamma_X$. 
Moreover, if $\pgrp$ is closed, then its Lie algebra sheaf $\calA$ is well-defined and induces an infintesimal action on the `tangent bundle' of $\Gamma_X$.

We begin by briefly discussing certain constructions that appear throughout the thesis, and we give several examples of pseudogroups that are particular to the total space of a submersion.

\begin{notation}
Let $\pi_E: E \to X$ be a submersion and $f:Y\to X$ a smooth map.
We denote the fiber of $E$ at $x\in X$ by $E_x$ (so that $E= \medcup_{x\in X} E_x$), and the pull-back of $E$ along $f$ by
    \[
        E_f
        \defeq
        \medcup_{y\in Y}
        \big(
        E_{f(y)} \times \{ y\}
        \big).
    \]
Moreover, if $Y\sub X$ is an embedded submanifold, we further abbreviate the restriction of $E$ to $Y$ by $E_Y$.
For example:
\begin{enumerate}[label=(\roman*)]

\item[a)]
For an embedded submanifold $Y\sub X$ we denote the restriction by
    \[
        T_YX
        =
        \medcup_{x\in Y} T_xX.
    \]

\item[b)]
For a section $b\in \Gamma_X(M)$ we write
    $
        T^*_bX
        =
        \medcup_{x \in M} T_{b(x)}^*X.
    $

\end{enumerate}
\end{notation}

\begin{definition}
\label{def:vertical_bundle}
The vertical bundle of $X$ will be denoted by
    \[
        VX
        \defeq
        \big\{
            v \in TX
            \sep
            \dd\pi_\textsc{X}(v) = 0
        \big\}.
    \]
Given a section $b \in \Gamma_X(U)$, the pull-back $V_bX$ is a vector bundle over $U$.
A \textbf{vertical tubular neighborhood}\index{vertical tubular neighborhood} of $b$ in $X$ is an open embedding
    \[
        \xi: 
        V_bX
        \longhookrightarrow
        X
    \]
that preserves fibers (mapping $V_{b(x)}X$ into $X_x$), and identifies the zero section of $V_bX$ with $b$.
\end{definition}

Vertical tubular neighborhoods exist for any submersion and any section.
If one thinks of $\Gamma_X$ as an infinite dimensional manifold,
then the space of sections of $V_bX$ play the role of tangent space at $b$,
and vertical tubular neighborhoods provide charts.

\begin{definition}
\label{def:normal_bundle}
Let $Z \sub X$ be an embedded submanifold.
The normal bundle of $Z$ in $X$ will be denoted by
    \[
        N(X, Z)
        \defeq
        \frac{T_ZX}{TZ}.
    \]
We say that $Z$ is a \textbf{fibered submanifold}\index{fibered submanifold} of $X$ if the restriction
    \[
        \pi_\textsc{Z} \defeq \pi_\textsc{X}|_Z:
        Z \longto \pi_\textsc{X}(Z)
    \]
is again a submersion.
In this case the normal bundle of $Z$ in $X$ is also obtained as the quotient of the respective vertical bundles:
    \[
        N(X, Z)
        =
        \frac{ V_ZX }{ VZ }.
    \]
A \textbf{vertical tubular neighborhood} of $Z$ in $X$ is an open embedding
    \[
        \xi: N(X,Z) \longhookrightarrow X
    \]
that preservers fibers and is the identity on $Z$.
\end{definition}

If $Z = b(U)$ is the image of a section $b\in\Gamma_X(U)$, then these concepts agree:
there is a natural inclusion $V_bX \hookrightarrow N(X,b(U))$, and this is an isomorphism because the ranks of these vector bundles are equal.

One can think of the space of fibered submanifolds of $X$ as an infinite dimensional manifold.
Its objects in the connected component of a given fibered manifold $Z$ are the fiber-preserving embeddings $Z\hookrightarrow X$ modulo the fiber-preserving diffeomorphisms of $Z$.
The space of sections of $N(X,Z)$ plays the role of tangent space at $Z$, and the vertical tubular neighborhoods again provide charts.
However, we find it generally easier to work with submersions and normal bundles than with infinite dimensional manifolds.

Let us look at several examples of pseudogroups on $X$.

\begin{example}
\label{exa:fiber_preserving}
The local diffeomorphisms on $X$ that locally map fibers to fibers form a closed pseudogroup:
    \begin{align*}
        \diff_{\pi_\textsc{X}}(U)
        &\defeq
        \big\{
            \phi \in \diff_X(U)
            \sep
            \dd_x(\pi\circ\phi)v = 0
            \;\;\forall\; 
            x \in U
            \;\forall\; v \in V_xX
        \big\}.
    \end{align*}
This is an example of a pseudogroup defined by a differential equation (or better yet, a partial differential relation, see section \ref{sec:jet_determinacy}).
\end{example}

\begin{example}
$\pi_\textsc{X}$ is a natural bundle over $M$ if there is a map
    \[
        E: \diff_M \longto \diff_X
    \]
that respects the group-like and sheaf-like structure (see section \ref{sec:natural_bundle}). 
A typical example of this is the tangent bundle $TM$, or any bundle constructed functorially out of $TM$. 
For any closed pseudogroup $\pgrp$ on $M$, the map $E$ defines a closed pseudogroup
    \[
        E\pgrp(U)
        \sub
        \diff_{\pi_\textsc{X}}(U)
    \]
as follows: a local diffeomorphism $\phi\in \diff_X(U)$ lies in $E\pgrp(U)$ if and only if for every $x\in U$ there is an open neighborhood $V\sub U$ of $x$ and a local diffeomorphism $\varphi \in \pgrp(\pi_\textsc{X}(V))$ such that $\phi|_V = E(\varphi)|_V$. 
Hence it locally maps fibers to fibers.
\end{example}

\begin{example}
\label{exa:vertical_tub_nbhd}
Not every interesting closed pseudogroup on $X$ arises from a natural bundle or otherwise respects the bundle structure of $X$. 
For suppose that $\pgrp$ is a closed pseudogroup on a manifold $Y$, and $M\!\hookrightarrow\! Y$ is a closed embedded submanifold. Let 
    $
        X = N(Y, M)
    $
be the normal bundle of $M$ in $Y$, and fix a tubular neighborhood 
    $
        \xi: X \hookrightarrow Y.
    $
Then the restriction $\xi^*\pgrp$ is a closed pseudogroup on $X$ (see example \ref{exa:restricted_pseudogroup}), which does not respect the bundle structure. 
Sections of $X$ correspond to closed, embedded submanifolds of $Y$ that are close to $M$. The local action of $\xi^*\pgrp$ on these sections can be understood as moving around submanifolds of $Y$. 
\end{example}

\subsection{Local Action}
\label{sect:local action}

Let $\pi_\textsc{X}:X\to M$ be a submersion and let $\pgrp$ be a pseudogroup on $X$.
We will construct a local action 
of $\pgrp$ on the sheaf $\Gamma_X$ of smooth sections of $\pi_\textsc{x}$. The term `action' is used lightly here, since we are not familiar with any formal definition of an action of a pseudogroup on a sheaf. Nonetheless, at the end of this section we make an attempt at writing down axiomatic properties of the concepts presented here.

The action will be `local' in the sense of a local action of a topological group on a space: 
not every local diffeomorphism $\phi\in \pgrp$ can act on a section $b\in \Gamma_X$, but when it is possible (when they are \emph{compatible}), the action is also defined for any $\tilde\phi$ close to $\phi$ and $\tilde b$ close to $b$.
The obstruction to compatibility is partly infinitesimal (a transversality condition, e.g.~\ref{exa:inf_incompatible}) and partly topological (e.g.~\ref{exa:top_incompatible}).
Much of this section is concerned with describing the compatibility condition and showing when it is satisfied.

If $\phi\in \pgrp$ maps fibers to fibers, the action is defined for any section $b\in \Gamma_X$ as the pullback $\phi^*b$. More precisely, if $\phi_0$ is the base map of $\phi$, then define
    \begin{equation}
    \label{for:action_is_pullback}
        \phi^*b
        \defeq
        \phi^{-1} \circ b \circ \phi_0.
    \end{equation}
But elements of the pseudogroup $\pgrp$ typically do not map fibers to fibers:
\begin{enumerate*}[label=(\roman*)]
    \item
        The local diffeomorphisms that map fibers to fibers do not satisfy the sheaf-like axioms of a pseudogroup. 
        Their sheafification is the pseudogroup $\diff_{\pi_\textsc{X}}$ defined in example \ref{exa:fiber_preserving}.
    \item
        By example \ref{exa:vertical_tub_nbhd} it is clear that many interesting pseudogroups on $X$ completely ignore the bundle structure.
\end{enumerate*}

Given $\phi\in \pgrp$ and $b \in \Gamma_X$, we will give conditions under which the image of $\phi^{-1}\circ b$ defines a local section of $X$.
Namely, a local section $b\in \Gamma_X(V)$ can be thought of as an embedded submanifold 
    $
        b(V)\sub X,
    $
and the converse holds under certain conditions.

\begin{lemma}
Let $B\sub X$ be an embedded submanifold and suppose that $W\defeq \pi_\textsc{X}(B)$ is an open set. 
The following are equivalent:
\begin{enumerate}[label=(\roman*)]

\item
$B$ comes from a local section.

\item
    $
        \pi_\textsc{x}\vert_B:B\to W
    $
is a diffeomorphism.

\item
$B$ and $\pi_\textsc{X}^{-1}(x)$ intersect transversally in a unique point for all $x \in W$.

\end{enumerate}
\end{lemma}

\begin{proof}
The implications $(i) \Rightarrow (ii)$ and $(ii) \Rightarrow (iii)$ are immediate, and implication $(iii) \Rightarrow (i)$ follows from the implicit function theorem.
\end{proof}

In the following definition, let $U\sub X$ and $V, W\sub M$ be open sets.

\begin{definition}
\label{def:compatibility}
We call $\phi \in \pgrp(U)$ and $b \in \Gamma_X(V)$  \textbf{compatible}\index{compatibility} on $W$ if
    \[
        B 
        \defeq
        \phi^{-1}(b(V)) \cap \pi_\textsc{X}^{-1}(W)
    \]
comes from a local section. In this case we denote
    \[
        (b\cdot \phi)\vert_{W} 
        \defeq
        (\pi_\textsc{X}|_B)^{-1}
        \;\in\; \Gamma_X(W)
    \]
Moreover, for any $x \in M$ we say that $\phi$ and $b$ compatible near $x$ if there is an open neighborhood $W$ of $x$ on which they are compatible. 
We write
    \[
        V(b, \phi)
        \defeq
        \big\{
            x \in M \sep
            \text{$\phi$ and $b$ are compatible near $x$}
        \big\}
    \]
for the largest open on which $\phi$ and $b$ are compatible, and we denote
    \[
        b\cdot \phi 
        \defeq
        (b \cdot \phi)|_{V(b, \phi)}.
        \qedhere
    \]
\end{definition}

\begin{example}
\label{exa:inf_incompatible}
Let $\pi_\textsc{X}:\R^2 \to \R$ and let $\phi_\theta\in \diff(\R^2)$ be rotation by an angle $\theta$. Then $b\in C^\infty(\R, \R)$ is compatible with $\phi_\theta$ near any point $x$ where $\graph(\dd_xb)$ does not have angle $\pi/2+\theta$ with the horizontal axis.
\end{example}

For $\phi\in \diff_{\pi_\textsc{X}}$ the transversality condition between $B$ and the fibers of $\pi_\textsc{X}$ always holds. However, the compatibility condition between $\phi$ and $b$ on $W$ is not vacuous: 
it tells us that the fiber $\pi_\textsc{X}^{-1}(x)$ over $x\in V$ intersects with the image of $\phi^{-1}\circ b$ in a unique point. 
The following example shows that this does not always hold.

\begin{example}
\label{exa:top_incompatible}
Let $\pi_\textsc{X}=e^{i x}: \R \to S^1$ be the universal covering space of the circle and let $b:S^1\backslash\{0\}\to \R$ be the section that maps onto $(0, 2\pi)$. Then an orientation-preserving $\phi \in \diff(\R)$ is compatible with $b$ only on the set
    \[
        V(b,\phi)=
        \pi_\textsc{X}\big(
            (
                \mathrm{max}(r, s-2\pi),
                \mathrm{min}(r+2\pi, s)
            )
        \big)
    \]
if $r< s \leq r+4\pi$, where $r \defeq\phi^{-1}(0)$ and $s \defeq \phi^{-1}(2\pi)$.
\end{example}

For a local diffeomorphism $\phi\in\pgrp$ that maps fibers to fibers we defined the action by the pullback formula (\ref{for:action_is_pullback}), 
where $\phi_0$ is the  map underlying $\phi$. 
This $\phi_0$ does not exist in general, but we can define a \textbf{base map relative to $b$}\index{relative base map}:
    \[
        \phi_b
        \defeq
        \pi_\textsc{X} \circ \phi \circ (b\cdot\phi):
        V(b,\phi) \longto M.
    \]
It satisfies the equality
    $
        \phi\cdot b = \phi^{-1}\circ b\circ \phi_b.
    $
Its inverse is the map
    \[
        \psi_b
        =
        \pi_\textsc{X} \circ \phi^{-1} \circ b.
    \]
Note that $\psi_b$ is always well-defined, even when $\phi_b$ is not.
Remark that $\phi$ and $b$ are compatible on $W$ if and only if the image of $\psi_b$ contains $W$ and $\psi_b$ is invertible on $\psi_b^{-1}(W)$.

In the following lemma we state some basic properties of the local action.

\begin{lemma}
\label{lemma:restricting_action}
Consider open sets $U, \tilde U\sub X$, and let $\phi\in \pgrp(U)$, let $\psi \in \pgrp(\tilde U)$, and let $b\in \Gamma_X(V)$.
\begin{enumerate}[label=(\roman*)]

\item
Suppose that $\phi$ and $b$ are compatible on $W$, and let 
    $
        U'\sub U,
    $
    $
        V'\sub V,
    $
and
    $ 
        W'\sub W
    $
be open subsets. Then $\phi\vert_{U'}$ and $b\vert_{V'}$ are compatible on $W'$ if and only if
    \[
        (b\cdot\phi)(W') \sub U', 
        \quad
        \phi_b^{-1}(W') \sub V'.
    \]

\item
Acting by the inclusion $\id\!|_U:U\hookrightarrow X$ is restriction to $\pi_\textsc{X}(U\cap b(V))$:
    \[
        V(b, \id\!|_U) = \pi_\textsc{X}\big( U\cap b(V) \big),
        \qquad
        b\cdot \id\!|_U = b|_{V(b, \id\!|_U)}.
    \]

\item
The action is compatible with composition on the open where both sides are defined:
    \[
        \big(
            b \cdot (\phi\circ\psi)
        \big)\big|_W
        =
        b \cdot \phi \cdot \psi \big|_W,
        \quad\text{where}\quad
        W
        \defeq
        V(b\cdot\phi, \psi) 
        \cap 
        V(b, \phi\circ\psi).
    \]
    
\end{enumerate}
\end{lemma}

From one perspective the action of $\pgrp$ on $\Gamma_X$ is defined as a map of sets
    \[
        \Gamma_X \times \pgrp \longto \Gamma_X.
    \]
One could try to interpret the properties in the above lemma as axioms for an action of a pseudogroup on a sheaf. Instead, we choose to interpret the action as something only defined locally: for fixed open sets $U\sub X$ and $V, W\sub M$, and fixed $\phi\in \pgrp(U)$ and $b \in \Gamma_X(U)$ that are compatible on $W$, under some conditions, any $\phi'\in\pgrp(U)$ close to $\phi$ and of $b' \in \Gamma_X(U)$ close to $b$ are compatible on $W$.
We make this precise in the following lemmas.
Their proves are similar to lemmas \ref{lemma:diffeo} and \ref{lemma:global_action}, and are thus omitted.

\begin{lemma}
\label{lemma:close_to_identity}
Let $X_0, X_1\sub X$ be compact sets with dense interior, and $\phi\in\diff_X(X_1)$ a local diffeomorphism, such that   
    \[
        X_0 \sub \phi(\mathrm{int}(X_1)).
    \]
Then there is a $C^1$-open neighborhood 
    $
        \calU\sub C^\infty(X_1,X)
    $
of $\phi$ such that any $\psi \in \calU$ is a diffeomorphism onto its image and 
    \[
        X_0 \sub \psi(\mathrm{int}(X_1)).
    \]
\end{lemma}

%

\begin{lemma}
\label{lemma:local_action}
Let $X_1\sub X$ and $M_0, M_1 \sub M$ be compact sets with dense interior, and $b\in \Gamma_X(M_1)$ a section, such that
    \[
        M_0 \sub \mathrm{int}(M_1),
        \quad
        b(M_0) \sub \mathrm{int}(X_1).
    \]
There are $C^1$-neighborhoods $\calU\sub \diff_X(X_1)$ of $\id\!|_{X_1}$ and $\calV\sub \Gamma_X(M_1)$ of $b$ such that any $\phi\in \calU$ and $c\in \calV$ are compatible on $M_0$.
\end{lemma}


\subsection{Infinitesimal Action}
\label{sec:inf_action}

Recall that a closed pseudogroup $\pgrp$ has an associated a Lie algebra sheaf denoted by $\calA$ (see theorem \ref{pgrp:thm}).
Given a global section $b\in \Gamma_X(M)$, the local action of $\pgrp$ on $\Gamma_X$ we will define a map of sheaves along $b$:
    \[
        \delta_b: 
        \calA
        \longto 
        \Gamma_{V_b X},
    \]
where $VX$ is the vertical bundle of $X$, and
    $
        V_b X
    $
its pullback along $b$.

\begin{definition}
\label{def:inf_action}
The \textbf{infinitesimal action}\index{infinitesimal action} of $\calA$ on $\Gamma_X$ at $b$ is
    \[
        \delta_b :
        \calA(U) \longto \Gamma_{V_b X}({b^{-1}(U)}),
        \qquad
        \delta_b (v)
        \defeq
        \tfrac{\dd }{\dd t}\big|_{t=0}
        (b\cdot \varphi_v^t),
    \]
where $\varphi_v^t:U^t \to X$ is the flow of $v$.
\end{definition}

\begin{lemma}
\label{prop:infinitesimal_action}
$\delta_bv$ is well-defined for any $v\in \calA(U)$, and
    \[
        \qquad
        \delta_b (v)
        =
        (\dd b\circ \dd \pi_\textsc{x} -\id)(v\circ b).
    \]
\end{lemma}

\begin{proof}
For any $x\in b^{-1}(U)$, pick compact balls $M_0, M_1\sub b^{-1}(U)$ centered at $x$, and $X_1 \sub U$ centered at $b(x)$, such that
    \[
        M_0 \sub \mathrm{int}(M_1),
        \quad
        b(M_1) \sub \mathrm{int}(X_1).
    \]
By lemma \ref{lemma:local_action} there is a $C^1$-open neighborhood $\calU\sub \diff_X(X_1)$ of $\id\!|_{X_1}$ such that any $\phi\in \calU$ is compatible with $b|_{M_1}$ on $M_0$.
Since the flow depends smoothly on $t$, there is an $\epsilon >0$ such that $\varphi^t_v|_{X_1} \sub \calU$ for all $t\in (-\epsilon, \epsilon)$.
This implies $(b|_{M_1}\cdot\varphi^t_v|_{X_1})|_{M_0}$ is well-defined for all $t\in (-\epsilon, \epsilon)$, and so is
    \[
        \delta_b(v)\big|_{M_0}
        =
        \tfrac{\dd }{\dd t}\big|_{t=0}
        \big(
            b|_{M_1} \cdot \varphi^t_v|_{X_1}
        \big)\big|_{M_0}.
    \]
Hence the infinitesimal action is well-defined. Moreover, recall that 
    \[
        b\cdot \varphi_v^t
        =
        \varphi_v^{-t}
        \circ
        b
        \circ
        (\varphi_v^t)_b,
        \qquad
        (\varphi_v^{t})_b^{-1}
        =
        \pi_\textsc{x}
        \circ
        \varphi_v^{-t}
        \circ
        b. 
    \]
Hence this formula for $\delta_bv$ follows from the chain rule.
\end{proof}

\clearpage
\section{PDEs with Symmetry}

Given two vector bundles $E$, $F$ over $M$, a linear partial differential operator of order $d$ from $E$ to $F$ is a map of sheaves
    $
        Q: \Gamma_E \to \Gamma_F
    $
that is locally given by
    \[
        Q(e)_j =
        \sum\nolimits_i
        \sum\nolimits_{\abs\alpha \leq d}
        f^{i, \alpha}_j
        \partial^\alpha e_i,
    \]
for some smooth functions $f^{i, \alpha}_j$. Here by `locally' we mean that we work on an open set $U\sub M$ that admits frames $E\vert_U\simeq \R^k\times U$ and $F\vert_U\simeq \R^l\times U$. A section $e\in \Gamma_U E$ can then be seen as a smooth map $e:U\to\R^k$ with component functions $e_i$. In other words, $Q$ is a map that depends $C^\infty$-linearly on the derivatives up to order $d$ of the section of $\Gamma E$. In the language of jet-bundles this means that $Q$ corresponds to a $C^\infty$-linear map of sheaves
    \[
        Q^{(d)}: \Gamma_{J^dE} \longto \Gamma_F
    \]
via the identity
    $
        Q = Q^{(d)} \circ j^d.
    $
Moreover, by being a $C^\infty$-linear map of sheaves, $Q^{(d)}$ corresponds to a vector bundle map
    \[
        q: J^dE \longto F
    \]
by requiring
    $
        Q^{(d)}(\sigma) = q \circ \sigma
    $
for all $\sigma \in \Gamma_{J^dE}$. This gives us also a convenient way of expressing \emph{non-linear} partial differential operators.

Let $\pi_\textsc{X}: X\to M$ and $\pi_\textsc{Y}: Y\to M$ be surjective submersions over $M$. 

\begin{definition}
\label{def:pdo}

A partial differential operator \textbf{(PDO)}\index{partial differential operator} 
of order $d$ is a fiber-preserving map
    \[
        q: J^dX \longto Y,
        \quad
        \pi_\textsc{Y}\circ q =\pi_{J^dX}.
    \]
We associate with it the map of sheaves
    \[
        Q: \Gamma_X \longto \Gamma_Y,
        \quad
        Q(b) \defeq
        q \circ j^d(b).
    \]
A partial differential equation \textbf{(PDE)}\index{partial differential equation} of order $d$ on $\pi_\textsc{X}$ is such a PDO combined with a global section $\zeta \in \Gamma_Y(M)$. 
A section $\zeta$ can also be decribed by its image $Z\defeq \zeta(M) \sub Y$.

We call $b\in \Gamma_X(U)$ a \textbf{solution}\index{solution} of $(Q, Z)$ if it satisfies 
    \[
        Q(b) = \zeta|_U.
    \]
We denote the sheaf of solutions by
    \[
        \sol_{Q, Z}(U) \defeq
        \big\{
            b \in \Gamma_X(U)
            \sep
            Q(b) = \zeta|_U
        \big\}.
        \qedhere
    \]
\end{definition}

In summary, a PDE can be represented by the square
\[
\begin{tikzcd}
    J^dX
    \arrow[->, swap]{d}
    \arrow[->]{r}{q}
    &
    Y
    \arrow[->, start anchor = south west, end anchor = north east, swap]{dl}
    \\
    M
    \arrow{r}{\zeta}
    &
    \arrow[hook, swap]{u}
    Z.
\end{tikzcd}
\]

We wish to combine the notions of pseudogroups and PDEs. A local diffeomorphism $\phi \in \diff_X$ will be considered a symmetry of the sheaf of solutions $\sol(Q, Z)$ if it maps solutions to solutions, i.e.\ if $b \in \sol_{Q, Z}$, then 
    \[
        b \cdot \phi \in \sol_{Q, Z}.
    \]
In general we are not interested in the set of all symmetries of $(Q, Z)$, or the structure of that set. Instead, we will fix a pseudogroup of symmetries as part of our data.

Recall the definitions of a closed pseudogroup (definition \ref{def:closed_pseudogroup}) and the local action (definition \ref{def:compatibility}) of $\pgrp$ on $\Gamma_X$.

\begin{definition}
\label{def:pde_with_symmetry}
A \textbf{PDE with symmetry}\index{partial differential equation!with symmetry} $(Q,Z)$ of order $d$ on $\pi_\textsc{X}$ is
\begin{itemize}

\item
a \emph{closed} pseudogroup $\pgrp$ on $X$,

\item
a PDE $(Q,Z)$ of order $d$ on $\pi_\textsc{X}$,
\end{itemize}

such that
if $\phi \in \pgrp(U)$ and $b\in \sol_{Q, Z}(V)$ are compatible on $W$, then 
    \[
        (b\cdot \phi)|_W
        \in 
        \sol_{Q,Z}(W).
        \qedhere
    \]
\end{definition}

\subsection{Deformation Complex}

We aim to study the interplay between PDEs with symmetry and their infinitesimal counterpart: the deformation complex defined here.
In sections \ref{sec:lie_algebra_sheaf} and \ref{sec:inf_action} we defined the Lie algebra sheaf of a closed pseudogroup and the infinitesimal action respectively.
Here we describe the relation between these objects and the linear data of a PDE with symmetry.


Let $(Q, Z)$ be a PDE with symmetry and $b\in \sol_{Q, Z}(M)$ a global solution.
Recall that $VX$ denotes the vertical bundle of $X$ and $V_bX$ its pull-back along $b$.
Likewise we have the vertical bundle $VY$ and its pull-back $V_\zeta Y$.

\begin{definition}
\label{def:linearization}
The \textbf{linearization}\index{partial differential equation!linearization of} of a PDE with symmetry $(Q, Z)$ at a global solution $b \in \sol_{Q,Z}(M)$ is
    \[
        \delta_b:
        \Gamma_{V_b\! X}(U)
        \longto 
        \Gamma_{V_\zeta\! Y}(U)
        \qquad
        \delta_b(w)
        \defeq
        \tfrac{\dd}{\dd t}
        \big|_{t=0}
            Q(b^{t}),
    \]
where $b^t\in \Gamma_X(U^t)$ is a family of sections with 
    $
        b^0 = b,
    $
and
    $
        \tfrac{\dd}{\dd t}\big|_{t=0}
        b^{t}
        =
        w.
    $
\end{definition}

Intuitively, the kernel of $\delta_b$ represents the tangent space to $\sol_{Q, Z}$ at $b$. 
We are interested in when the infinitesimal action lies open in this kernel.

\begin{definition}
\label{def:deformation_complex}
The \textbf{deformation complex}\index{deformation complex} of $(\pgrp, Q, Z)$ at a global solution $b \in \sol_M(Q, Z)$ is the linear chain complex
    \[
        \calA
        \xrightarrow{\;\; \delta_b \;\;}
        \Gamma_{V_b X}
        \xrightarrow{ \,\delta_b\, }
        \Gamma_{V_\zeta X},
    \] 
where the operators $\delta_b$ are defined in \ref{def:inf_action} and \ref{def:linearization} respectively.
\end{definition}

\begin{lemma}
The above is a chain complex: $\delta_b \circ \delta_b = 0$.
\end{lemma}

\begin{proof}
Let $v\in \calA(U)$ be a vector field with flow $\varphi_v^t \in \pgrp(U^t)$. We have that 
    \[
        \zeta = Q(b\cdot \varphi_v^t)
    \]
equals $\zeta$ over the largest open set $V(b, \varphi_v^t)$ where $b$ and $\varphi_v^t$ are compatible, for $t\in \R$. Hence
    \[
        0 = 
        \tfrac{\partial}{\partial t}\big|_{t=0} \zeta
        = 
        \delta_b \big(
            \tfrac{\partial}{\partial t}\big|_{t=0} (b\cdot \varphi_v^t)
        \big)
        =
        (\delta_b \circ \delta_b)(v)
        \;\in\; \Gamma_{V_\zeta Y}(U).
        \qedhere
    \]
\end{proof}

\clearpage
\section{Local Rigidity}
\label{sec:local_rigidity}

We introduce the notion of nested domains to make precise the aphorism 
\begin{quote}
    \emph{`infinitesimal rigidity implies rigidity'}
\end{quote}
in a \emph{local} setting involving a PDE with symmetry.
The keyword here is local: we will first define \emph{local rigidity} and \emph{local infinitesimal rigidity} in the language of the previous sections.
Then we introduce \emph{nested domains with corners} and reformulate the above concepts.
This then provides the necessary technical details to formulate the \nameref{Main Theorem}.

Given a surjective submersion $\pi_\textsc{X}: X \!\to\! M$, fix the following objects:
\begin{itemize}

\item
a PDE with symmetry $(\pgrp, Q, Z)$,

\item
the Lie algebra sheaf $\calA$ of $\pgrp$,

\item
a global solution $b \in \sol_{Q, Z}(M)$,

\end{itemize}

\subsection{Sketch of Local Rigidity}

Let $K\sub M$ be a compact set.
Loosely speaking, local rigidity of the solution $b$ near $K$ would mean that the \emph{orbit} of $\pgrp$ through $b$ lies \emph{open} in the set of germs of $\sol_{Q,Z}$ at $K$. 
It is difficult to work with the weak notion topology that does exist on the set of germs.
Therefore we opt to work with fixed open subsets and mimick the definition of germs.

\begin{definition}
\label{def:local_rigidity_attempt}
Let $M_0, M_1\sub M$ and $X_0, X_1\sub X$ be \emph{compact} subsets such that
    \[
        M_0
        \sub 
        \mathrm{int}(M_1),
        \qquad
        X_0
        \sub
        \mathrm{int}(X_1),
        \qquad
        b(M_i)
        \sub
        \mathrm{int}(X_i),
        \quad i=0,1.
    \]

\begin{itemize}
\item 
A section $c\in \Gamma_X(M_1)$ is \textbf{locally equivalent}\index{locally equivalent} to $b$ on $M_0$ if there is a local diffeomorphism $\phi\in \pgrp(X_0)$ that is compatible with $b$ on $M_0$, and such that 
    \[
        c\vert_{M_0}
        =
        (b\cdot\phi)\vert_{M_0}.
    \]

\item 
The solution $b$ is \textbf{locally rigid}\index{locally rigid} relative to $(M_1, M_0)$ if there is a $p\in \N$ and a compact-open $C^p$-open neighborhood 
    \[
        \calV
        \sub
        \sol_{Q,Z}(M_1)
    \]
of $b\vert_{M_1}$ such that all
    $
        c\in \calV
    $
are locally equivalent to $b$ on $M_0$.
\qedhere
\end{itemize} 
\end{definition}

\begin{remark}
The above notion of \emph{local equivalence} is not necessarily an equivalence relation.
It may for example fail to be symmetric. 
\end{remark}

We can describe local rigidity more abstractly through non-linear chain complexes.
Recall that by lemma \ref{lemma:local_action} there are $C^1$-open neighborhoods
    \[
        \calU_i \sub \pgrp(X_i)
    \]
of $\id\!|_{X_i}$
such that any $\phi\in \calU_i$ is compatible with $b$ on $M_i$. 
We obtain the sequence of operators
    \[
        \calU_i
        \xrightarrow{ \;\;P\;\; }
        \Gamma_X(M_i)
        \xrightarrow{ \;\;Q\;\; }
        \Gamma_Y(M_i),
        \qquad
        P(\phi)\defeq (b\cdot \phi)\vert_{M_i},
        \quad 
        \phi\in \calU_i.
    \] 
We may think of this as a nonlinear chain complex, since the composition is constant by the axioms of a PDE with symmetry:
    \[
        (Q\circ P)(\phi) = \zeta|_{M_0}.
    \]
The complex can be called \emph{exact} if the image $P(\calU_i)$ lies open in the preimage
    $
        Q^{-1}(\zeta|_{M_i}).
    $
However, there are various reasons why this would be asking for too much.
Instead one could ask for a compact-open $C^p$-open subset $\calV_1\sub \sol_{Q, Z}(M_1)$ and an operator
    \[
        \begin{tikzcd}[row sep=small]
            \calU_1
            \arrow{r}
            \arrow[d,
                -,
                dotted,
                start anchor={
                    [yshift=0.25ex]south
                },
                end anchor={
                    [yshift=-0.25ex]north
                },
            ]
            &
            \calV_1
            \arrow[r]
            \arrow[d,
                -,
                dotted,
                start anchor={
                    [yshift=0.25ex]south
                },
                end anchor={
                    [yshift=-0.25ex]north
                },
            ]
            \arrow[dl,
                start anchor={
                    [xshift=0ex,yshift=.5ex]south west
                },
                end anchor={
                    [xshift=-0.5ex,yshift=-0.5ex]north east
                },
                "H" description
            ]
            &
            \Gamma_Z(M_1)
            \arrow[d,
                -,
                dotted,
                start anchor={
                    [yshift=0.25ex]south
                },
                end anchor={
                    [yshift=-0.25ex]north
                },
            ]
            \\
            \calU_0
            \arrow{r}
            &
            \Gamma_X(M_0)
            \arrow{r}
            &
            \Gamma_Y(M_0)
        \end{tikzcd}
    \]
that splits the complex: $(P\circ H)(c) = c|_{M_0}$ for all solutions $c\in \calV_1$. 

The problem of local rigidity lends itself to some type of inverse function theorem,
where we first solve an infinitesimal version of the problem, obtaining some $h$, and then integrate it to obtain $H$.
We regard the Lie algebra sheaf $\calA$ as the tangent space at the identity to $\pgrp$, and the kernel of the linearization of $Q$ as the tangent space at $b$ to $\sol_{Q,Z}(M)$. 
Therefore we are led to study the chain complex from definition \ref{def:deformation_complex}, for $i=0,1$,
    \[
        \calA(X_i)
        \xrightarrow{ \;\;\delta_b\;\; }
        \Gamma_{V_bX}(M_i)
        \xrightarrow{ \;\;\delta_b\;\; }
        \Gamma_{V_\zeta Y}(M_i).
    \]
The infinitesimal notion of rigidity is that the linear chain complex is exact, $\im{\delta_b}=\ker{\delta_b}$. From the viewpoint of linear algebra this is equivalent to the existence of homotopy operators, i.e. maps $h_i$ such that
    \[
        \delta_b \circ h_1 + h_2\circ \delta_b
        =
        \id.
    \]
Assuming local rigidity, and moreover that $H$ is a differentiable map, we expect that the derivative at $b$ of $H$ is a suitable choice for $h_1$. 
This leads us to the following definition.

\begin{definition}
\label{def:infinitesimal_rigidity_attempt}
We call $b$ \textbf{locally infinitesimally rigid} relative to $(M_1, M_0)$ if there are linear operators
    \[
        \begin{tikzcd}[row sep=small]
            \calA(X_1)
            \arrow[d,
                -,
                dotted,
                start anchor={
                    [yshift=0.25ex]south
                },
                end anchor={
                    [yshift=-0.25ex]north
                },
            ]
            \arrow{r}
            &
            \Gamma_{V_bX} (M_1)
            \arrow[d,
                -,
                dotted,
                start anchor={
                    [yshift=0.25ex]south
                },
                end anchor={
                    [yshift=-0.25ex]north
                },
            ]
            \arrow{r}
            \arrow[dl,
                start anchor={
                    [xshift=1ex,yshift=1ex]south west
                },
                end anchor={
                    [xshift=-1ex,yshift=-1ex]north east
                },
                "h_1" description
            ]
            &
            \Gamma_{V_\zeta Y}(M_1)
            \arrow[d,
                -,
                dotted,
                start anchor={
                    [yshift=0.25ex]south
                },
                end anchor={
                    [yshift=-0.25ex]north
                },
            ]
            \arrow[dl,
                start anchor={
                    [xshift=1ex,yshift=1ex]south west
                },
                end anchor={
                    [xshift=-1ex,yshift=-1ex]north east
                },
                "h_2" description
            ]
            \\
            \calA(X_0)
            \arrow{r}
            &
            \Gamma_{V_bX}(M_0)
            \arrow{r}
            &
            \Gamma_{V_\zeta Y}(M_0).
        \end{tikzcd}
    \]
such that
    $
        (\delta_b\circ h_1
        +
        h_2\circ\delta_b)(w)
        =   
        w|_{M_0}
    $
for all $w\in \Gamma_{V_bX}(M_1)$.
\end{definition}

\subsection{Domains with Corners}

Let $D$ be a smooth manifold without boundary (which in the previous discussion is typically either $X$ or $M$). 

\begin{definition}
\label{def:domain_with_corners}

A \textbf{domain with corners}\index{domain with corners} in $D$ is a compact submanifold $D_1\sub D$ of codimension zero and possibly with corners, i.e.\ it can be covered by charts $(U,\chi )$ of $D$ such that
    \[
        \chi (U\cap  D_1)
        =
        \chi (U)\cap  \left([0,\infty)^k \!\times\! \R^{m-k}\right),
    \]
where the $k$ may depend on the chart.

Denote by $\partial D_1$ the topological boundary of $D_1\sub D$, i.e.\
    \[
        \partial D_1
        \defeq
        D_1\backslash \mathrm{int}(D_1).
        \qedhere
    \]
\end{definition}

Essentially, we adopt the definition of \cite{Joyce} of a manifold with corners, but our notion of boundary is different to the one given there.

\begin{definition}
The \textbf{interior cone}\index{interior cone} of a domain with corners $D_1\sub D$ at $x\in D_1$ is the open cone
    \[
        \calC_x D_1
        \sub 
        T_x D,
    \]
of vectors $v\in T_x D$ such that, for any curve $\gamma^t \in D$ with $\dot{\gamma}^0 =v$, there is an $\epsilon >0$ such that
    \[
        \gamma^t \in \mathrm{int}(D_1),
        \quad\forall\;
        t \in (0, \epsilon).
        \qedhere
    \]
\end{definition}

The following lemma is immediate.

\begin{lemma}
\label{lemma:interior_cone}
Let $(U,\chi )$ be a chart at $x\in D_1$ adapted to $D_1$, i.e.\ a chart as in definition \ref{def:domain_with_corners} with $\chi(x)=0$. 
Then the interior cone is given by
    \[
        \calC_x D_1
        =
        (\dd_x \chi )^{-1}
        \big(
            (0,\infty)^k \!\times\! \R^{m-k}
        \big).
    \]
In particular, for $x\in \mathrm{int}(D_1)$ we have $\calC_x D_1 =T_xD$.

\end{lemma}

\subsection{Nested Domains with Corners}

Controlling the size of domains will be encoded by the following notion.

\begin{definition}
\label{def:nested_domain}
A \textbf{nested domain with corners}\index{nested domain with corners} in $D$ is a family $\{D_r\}_{r\in [0,1]}$ of domains with corners in $D$ such that there exists a smooth family of diffeomorphisms
    \[
        \phi^r:D\diffto D,
        \quad
        r\in [0,1],
    \]
satisfying
    $
        \phi^0=\id_D
    $
and
    $
        \phi^{1-r}(D_1)=D_r,
    $
and whose infinitesimal generator points inwards, i.e.\ 
    \[
        \tilde v^r
        \defeq 
        \tfrac{\dd}{\dd \epsilon}\big|_{\epsilon = 0}
            (\phi^r)^{-1}
            \circ
            \phi^{r+\epsilon}
    \]
maps into the interior cone $\calC D_1$ on $D_1$ for all $r\in [0,1]$.
\end{definition}

\begin{remark}
{\color{white}This line is blank.}
\begin{itemize}

\item
The family of diffeomorphisms $\phi^r$ and $D_1$ determine the entire family $D_r$, but we do not regard $\phi^r$ as part of the data of a nested domain.

\item
We assume that $\phi^r$ is a diffeomorphism of $D$, but it suffices if it is defined locally on an open neighborhood of $D_1$. Suppose that $\phi^r$ is defined on an open neighborhood $U\sub D$ of $D_1$, and let $v^r\in \gerX_D(U)$ be the generator of $\phi^r$. Let $\xi:D\to \R_{\geq 0}$ be a bump function that is one on $D_1$ and zero outside $U$. Then $\xi\cdot v^r$ extends to $D$, and its flow still witnesses that $D_r$ is a nested domain.

\item
Usually the generator of a smooth family of diffeomorphisms $\phi^r$ is the time-dependent vector field $v^r$ given by
    \[
        v^r \defeq 
        \tfrac{\dd}{\dd\epsilon}\big|_{\epsilon = 0}
        \phi^{r+\epsilon} \circ (\phi^r)^{-1}.
    \]
This is related to $\tilde v$ by $\tilde v^r = (\phi^r)^*v^r$ and it agrees if $\phi^r$ is additive.

\end{itemize}
\end{remark}

\begin{example}
Closed polydiscs in $\C^n$ are nested domains with corners:
    \[
        P_r
        \defeq
        \big\{
            z \in \C^n
            \sep
            |z_i| \leq 1+r
            \;\;\forall\;
            i = 1, \ldots, n
        \big\}.
        \qedhere
    \]
\end{example}

\begin{example}
Let $D_1\sub D$ be any domain with corners.
Using a partition of unity and that $\calC_p D_1\sub T_pD_1$ is a convex set, one can construct a compactly supported vector field $v\in \gerX(D)$ such that $v|_{D_1}\in \Gamma_{\calC D_1}$.
If $\varphi_v^t$ denotes the flow of $v$, then 
    \[
        D_{r}\defeq \varphi^{1-r}_v(D_1)
    \]
defines a nested domain with corners in $D$.
\end{example}

\begin{example}
\label{exa:constructing_X_r}
Let $\pi_\textsc{X}:X \!\to\! M$ be a surjective submersion, $b\in \Gamma_X(M)$ a global section, and $\{ M_r \}$ a nested domain with corners in $M$.
We aim to construct a nested domain with corners $\{X_r\}$ in $X$ such that
    \[
        b(M_r) \sub X_r,
        \qquad
        X_r \sub \pi_\textsc{X}^{-1}(M_r),
    \]
for all $r\in [0,1]$.
Let $\xi: V_bX \hookrightarrow X$ be a vertical tubular neighborhood of $b$, i.e.\ a diffeomorphism onto its image that preserves fibers and maps the zero section to $b$. 
Denote the vertical bundle along $b$ by $E\defeq V_bX$.
Choose a vector bundle metric $g$ on $E$, and define
    \[
        X_r \defeq
        \big\{
            e_x \in E_x
            \sep
            |e_x|_g \leq 1+r,
            \;
            x \in M_r
        \big\},
        \quad
        r\in [0,1].
    \]
Then $\{X_r\}$ is a nested domain with corners in $X$.
\end{example}

\begin{proof}
Choose a linear connection $\nabla$ on $E$. 
Call a local section $e\in \Gamma_E(U)$ horizontal at $x$ if $\nabla_{v_x}(e) = 0$ for all $v_x\in T_xM$. 
The lift of a vector field $v \in \gerX(M)$ is the vector field $\hat v$ such that 
    \[
        \hat v(e(x)) = \dd_x e\circ v(x),
    \]
for all $x\in M$ and any local section $e\in \Gamma_E(U)$ that is horizontal at $x$.

Let $\phi^r$ be a family of diffeomorphisms of $M$ corresponding to the nested domain $\{M_r\}$, and let $v^r$ on $M$ be its generator. 
Let $\varphi_{\hat v}^r$ be the flow of the lift $\hat v^r$ of $v^r$, and define a family of diffeomorphisms $\psi^r$ on $E$ by
    \[
        \psi^r(e_x) 
        =
        \tfrac{2-r}2
        \cdot
        \varphi_{\hat v}^r(e_x),
        \quad
        e_x \in E_x,
        \quad
        r \in [0,1].
    \]
Pick the linear connection to be metric, i.e.\ such that
    \[
        \gerL_v g(e_1, e_2) 
        =
        g(\nabla_v(e_1), e_2)
        +
        g(e_1, \nabla_v(e_2)),
        \qquad\forall\;
        v \in \gerX_M
        \quad\forall\;
        e_1, e_2 \in \Gamma_E.
    \]
This implies that, for any vector field $v$ on $M$,
    \[
        (\varphi_{\hat v}^r)^*|.|_g = |.|_g.
    \] 
Namely, for any point $e_x\in E_x$ we may pick a local section $e\in \Gamma_E(U)$ with $x\in U$, horizontal at $x$, and $e(x)=e_x$. Then for any vector field $v$ on $M$,
    \[
        (\gerL_{\hat v}|.|^2_g)_{e_x}
        =
        (\dd_{e(x)}|.|^2_g \circ \dd_xe)(v(x))
        =
        (\gerL_v g(e,e))_x
        =
        2 g_x(\nabla_{v(x)}(e),e_x)
        = 0.
    \]
This already shows that
    $
        \psi^{1-r}(X_1) = X_r.
    $
    
We still have to show that the generator of $\psi^r$ maps into the interior cone of $X^r$. 
Note that the interior cone $\calC X_r$ certainly contains all vectors on $E$ that shrink $g$ and that are projected into $\calC M_r$ by $\pi_\textsc{X}$. 
In other words, 
    \[
        \big\{
            w \in TE
            \sep
            \gerL_w|.|_g^2 < 0
        \big\}
        \medcap
        \dd\pi_E^{-1}(\calC M_r)
        \;\;\sub\;\;
        \calC X_r.
    \]
If $w^r$ on $E$ is the generator of $\psi^r$, we have to show that $\tilde w^r \defeq (\psi^r)^*w^r$ maps into $\calC X_r$ for all $r\in [0,1]$. 
A simple computation shows that
    \[
        w^r
        =
        \tfrac{1}{r-2} \cdot \xi
        +
        \hat v^r,
    \]
where $\xi$ is the Euler vector field on $E$, i.e.\ the generator of multiplication by $e^t$.
This gives
    \[
        \tilde w^r
        =
        \tfrac{1}{r-2}
        \cdot
        (\varphi_{\hat v}^r)^*\xi
        +
        (\varphi_{\hat v}^r)^* \hat v^r.
    \]
We have already shown that $\gerL_{\hat v^r}|.|_g^2 = 0$ and thus also $(\varphi_{\hat v}^r)_*|.|_g^2 = |.|_g^2$.
Since $\gerL_\xi |.|_g^2 > 0$ and $r-2<0$, we conclude that $\tilde w^r$ maps into $\calC X_r$.
\end{proof}

\begin{lemma}
\label{lem:stictily_decreasing}
A nested domain with corners $\{D_r\}$ is strictly decreasing:
    \[
        D_r \sub \mathrm{int}(D_s),
        \quad\forall\;
        0\leq r < s \leq 1.
    \]
\end{lemma}

\begin{proof} 
By definition \ref{def:nested_domain} we have $(\phi^{1-s})^*v^{1-s} \in \calC D_1$ for all $s\in [0,1]$. 
Hence by lemma \ref{lemma:interior_cone} and the compactness of $D_1$, there is a constant $\epsilon>0$ such that
    \[
        (\phi^{1-s})^{-1}\circ  \phi^{1-s+u}(D_1)
        \sub 
        \mathrm{int}(D_1),
    \]
for all $u\in (0,\epsilon)$ and $s\in [0,1]$.
This is equivalent to
    \[
        D_{r}\sub \mathrm{int}(D_{s}), 
        \quad\forall\;
        0<s-r<\epsilon,
    \]
which implies that the nested domain is strictly decreasing.
\end{proof}

\subsection{Local Rigidity \& Tame Operators}
\label{sec:tame_estimates}

Now also fix nested domains $\{ M_r \}$ in $M$ and $\{ X_r \}$ in $X$ such that
    \[
        b(M_r) \sub X_r,
        \qquad\forall\;
        r \in [0,1].
    \]  
We will reformulate definition \ref{def:local_rigidity_attempt} on local rigidity and definition \ref{def:infinitesimal_rigidity_attempt} on local infinitesimal rigidity to be compatible with a choice of nested domains.
At the same time we introduce \emph{tame} operators on nested domains, a variation on the tame operators from \cite{Ham82}.

We start by defining $C^k$-norms on the local sections of an arbitrary vector bundle over $M$.
Then we define norms for local sections of $X$ close to $b$
and, in a similar fashion, for local diffeomorphisms on $X$ close to the identity.
See also section \ref{subsect:comparison} for more details.

Let $\pi_\textsc{E}: E\!\to\! M$ be a vector bundle, and fix vector bundle metrics $m_k$ for all $k\geq 0$ on the fibers of the jet bundle $J^k E \!\to\! M$ such that the projections 
    \[
        \pi_k^{k+1}: J^{k+1} E\longto J^k E
    \]
are norm decreasing. 
Let $K\sub M$ be a compact set with $K=\overline{\mathrm{int}(K)}$. This condition ensures that every section $e\in \Gamma_E(K)$ has a well-defined $k$-jet 
    \[
        j^k(e) \in \Gamma_{J^k E}(K).
    \] 
Now define
    \[
        \norm{e}_{k,K}
        \defeq
        \sup\nolimits_{x\in K}
        \abs{
            j^k(e)_x
        }_{m_k}.
    \]
Moreover, given nested a domain $\{ M_r \}$ and a section $e\in \Gamma_E(M_r)$, simplify the notation to
    \[
        \norm{e}_{k, r}
        \defeq
        \norm{e}_{k, M_r}.
    \]
To define norms on the sections of $X$, pick a vertical tubular neighborhood 
    \[
        \xi: 
        V_bX 
        \!\longhookrightarrow\! 
        X
    \]
around a fixed global section $b \in \Gamma_X(M)$ (see definition \ref{def:vertical_bundle}), and choose $C^k$-norms on $V_bX$ as we did before.
Let 
    \[
        \calV\sub \Gamma_X
    \]
be the subsheaf of local sections that map into the image of $\xi$.
Then for any $c\in \calV(M_r)$ we define
    \[
        \norm{c - b}_{k, r}
        \defeq
        \norm{ \xi^{-1}\circ c }_{k, r}.
    \]
Similarly, to define norms on local diffeomorphisms, pick a vertical tubular neighborhood 
    \[
        \chi: TX \!\longhookrightarrow\! X \!\times\! X
    \]
around the graph of the identity. 
Let 
    \[
        \calU\sub \pgrp
    \]
be the set of local diffeomorphisms whose graph lies in the image $\chi$.
Then for any $\phi \in \calU(X_r)$ we define
    \[
        \norm{\phi - \id}_{k, r}
        \defeq
        \norm{ \chi^{-1}\circ (\id, \phi)}_{k, r}.
    \]

\begin{definition}
\label{def:local_rigidity}
Call a solution $b\in \sol_{Q,Z}(M)$ \textbf{rigid}\index{rigid} relative to $\{ M_r \}$ and $\{ X_r \}$ 
if there is an integer $p\in \N$ and for all $r<s$ there is a $C^p$-open neighborhood 
    \[
        \calV_{s,r} \sub \sol_{Q,Z}(M_s)
    \]
and a (possibly non-linear) operator
    \[
        \begin{tikzcd}[row sep=small]
            \pgrp(X_s)
            \arrow[d,
                -,
                dotted,
                start anchor={
                    [yshift=0.25ex]south
                },
                end anchor={
                    [yshift=-0.25ex]north
                },
            ]
            &
            \calV_{s,r}
            \arrow[d,
                -,
                dotted,
                start anchor={
                    [yshift=0.25ex]south
                },
                end anchor={
                    [yshift=-0.25ex]north
                },
            ]
            \arrow[dl,
                start anchor={
                    [xshift=0ex,yshift=.5ex]south west
                },
                end anchor={
                    [xshift=-0.5ex,yshift=-0.5ex]north east
                },
                "\phi^{s,r}" description
            ]
            \\
            \pgrp(X_r)
            &
            \sol_{Q, Z}(M_r)
        \end{tikzcd}
    \]
such that $\phi^{s,r}(c)$ is compatible with $b$ on $M_r$ for all $c\in\calV_{s,r}$ and we have
    \[
        c|_{M_r} = (b \cdot \phi^{s,r})|_{M_r},
    \]
Call it \textbf{tamely rigid} if there are constants $c_k, d_k, l \geq 0$ such that
    \[
        \norm{\phi^{s,r} - \id}_{k, r}
        \leq
        c_k (s - r)^{-d_k}
        \norm{c - b}_{k+l, s},
        \qquad\forall\;
        c \in \calV_{s,r}.
        \qedhere
    \]
\end{definition}

\begin{definition}
\label{def:local_infinitesimal_rigidity}
Call a solution $b \in \sol_{Q, Z}(M)$ \textbf{infinitesimally rigid}\index{infinitesimally rigid}
relative to $\{ M_r \}$ and $\{ X_r \}$ 
if for all $r < s$ there are linear operators
    \[
        \begin{tikzcd}[row sep=small]
            \calA(X_s)
            \arrow{r}
            \arrow[d,
                -,
                dotted,
                start anchor={
                    [yshift=0.25ex]south
                },
                end anchor={
                    [yshift=-0.25ex]north
                },
            ]
            &
            \Gamma_{V_b X}(M_s)
            \arrow{r}
            \arrow[d,
                -,
                dotted,
                start anchor={
                    [yshift=0.25ex]south
                },
                end anchor={
                    [yshift=-0.25ex]north
                },
            ]
            \arrow[dl,
                start anchor={
                    [xshift=1ex,yshift=1ex]south west
                },
                end anchor={
                    [xshift=-1ex,yshift=-1ex]north east
                },
                "h^{s,r}_1" description
            ]
            &
            \Gamma_{V_\zeta Y}(M_s)
            \arrow[d,
                -,
                dotted,
                start anchor={
                    [yshift=0.25ex]south
                },
                end anchor={
                    [yshift=-0.25ex]north
                },
            ]
            \arrow[dl,
                start anchor={
                    [xshift=1ex,yshift=1ex]south west
                },
                end anchor={
                    [xshift=-1ex,yshift=-1ex]north east
                },
                "h^{s,r}_2" description
            ]
            \\
            \calA(X_r)
            \arrow{r}
            &
            \Gamma_{V_bX}(M_r)
            \arrow{r}
            &
            \Gamma_{V_\zeta Y}(M_r)
        \end{tikzcd}
    \]
such that 
for all $w \in \Gamma_{V_b X}(M_s)$ we have
    \[
        ( \delta_b \circ h_1 + h_2 \circ \delta_b )(w)
        =
        w|_{M_r},
    \]
Call it \textbf{tamely infinitesimally rigid} if moreover there are $c_k, d_k, l_i\geq 0$ such that
    \[
        \norm{ h_i(w) }_{k, r}
        \leq
        c_k (s-r)^{-d_k}
        \norm{ w }_{k+l_i, s}.
        \qedhere
    \] 
\end{definition}

\clearpage
\section{Main Theorem: Simple Version}

Given a submersion $\pi_\textsc{X}: X \to M$, consider the following objects:

\begin{itemize}

\item
a PDE with symmetry $(\pgrp, Q, Z)$ (see definition \ref{def:pde_with_symmetry}),

\item
the Lie algebra sheaf $\calA$ of $\pgrp$ (see theorem \ref{pgrp:thm}),

\item
a global solution $b \in \sol_{Q, Z}(M)$,

\item 
nested domains $\{ M_r \}$ and $\{ X_r \}$ in $M$ and $X$ respectively (see definition \ref{def:nested_domain}) such that
    \[
        b(M_r) \sub X_r,
        \qquad\forall\;
        r \in [0,1].
    \]

\end{itemize}

\begin{NamedTheorem}[Main Theorem (Simple)]
\label{Main Theorem}
Tame infinitesimal rigidity implies rigidity:

Suppose there are linear operators, for all $r<s$,
    \[
        \begin{tikzcd}[row sep=small]
            \calA(X_s)
            \arrow{r}
            \arrow[d,
                -,
                dotted,
                start anchor={
                    [yshift=0.25ex]south
                },
                end anchor={
                    [yshift=-0.25ex]north
                },
            ]
            &
            \Gamma_{V_b X}(M_s)
            \arrow{r}
            \arrow[d,
                -,
                dotted,
                start anchor={
                    [yshift=0.25ex]south
                },
                end anchor={
                    [yshift=-0.25ex]north
                },
            ]
            \arrow[dl,
                start anchor={
                    [xshift=1ex,yshift=1ex]south west
                },
                end anchor={
                    [xshift=-1ex,yshift=-1ex]north east
                },
                "h^{s,r}_1" description
            ]
            &
            \Gamma_{V_\zeta Y}(M_s)
            \arrow[d,
                -,
                dotted,
                start anchor={
                    [yshift=0.25ex]south
                },
                end anchor={
                    [yshift=-0.25ex]north
                },
            ]
            \arrow[dl,
                start anchor={
                    [xshift=1ex,yshift=1ex]south west
                },
                end anchor={
                    [xshift=-1ex,yshift=-1ex]north east
                },
                "h^{s,r}_2" description
            ]
            \\
            \calA(X_r)
            \arrow{r}
            &
            \Gamma_{V_bX}(M_r)
            \arrow{r}
            &
            \Gamma_{V_\zeta Y}(M_r)
        \end{tikzcd}
    \]
and there are constants $c_k, d_k, l_1, l_2 \geq 0$ such that
    \[
            (\delta_b\circ h^{s,r}_1
            +
            h^{s,r}_2\circ \delta_b)(w)
            =
            w|_{M_r},
            \qquad
            \norm{ h_i^{s,r}(v) }_{k,r}
            \leq
            c_k(s-r)^{-d_k}
            \norm{ v }_{k+l_i, s}.
    \]
Then there is an integer $p\in \N$, and for all $r<s$ there is a $C^p$-open neighborhood
    \[
        \calV_{s,r} \sub \sol_{Q,Z}(M_s),
    \]
of $b|_{M_s}$, and a (non-linear) operator
    \[
        \begin{tikzcd}[row sep=small]
            \pgrp(X_s)
            \arrow[d,
                -,
                dotted,
                start anchor={
                    [yshift=0.25ex]south
                },
                end anchor={
                    [yshift=-0.25ex]north
                },
            ]
            &
            \calV_{s,r}
            \arrow[d,
                -,
                dotted,
                start anchor={
                    [yshift=0.25ex]south
                },
                end anchor={
                    [yshift=-0.25ex]north
                },
            ]
            \arrow[dl,
                ->,
                start anchor={
                    [xshift=.5ex,yshift=1.5ex]south west
                },
                end anchor={
                    [xshift=-.5ex,yshift=-.5ex]north east
                },
                "\phi^{s,r}" description
            ]
            \\
            \pgrp(X_r)
            &
            \sol_{Q,Z}(M_r),
        \end{tikzcd}
    \]
such that all $\phi^{s,r}(c)$ and $c$ are compatible on $M_r$ for all $c \in \calV_{s,r}$ and
    \[
        (c\cdot \phi^{s,r}(c))|_{M_r}
        =
        b|_{M_r}.
    \]
Moreover, there are constants $a_k, b_k, l\geq 0$ such that
    \[
        \norm{\phi^{s,r}(c)-\id}_{k,r}
        \leq
        a_k(s-r)^{- b_k}
        \norm{c-b}_{k+l,s}.
    \]
\end{NamedTheorem}

\clearpage
\begin{remark}
The constants $a_k$, $b_k$, $c_k$, and $d_k$ depend only on $k$.
\end{remark}

\begin{remark}
Although it is certainly not optimal, we obtain
    \[
        p = \max(l_1+1, d) + \max(6l_1+5, 4l_2+1).
    \]
Moreover, the $C^p$-open subsets $\calV_{s,r}\sub \sol_{Q,Z}(M_s)$ depend polynomially on $s-r$ in the sense that they are given by an expression 
    \[
        \norm{ c-b }_{p,s} < \gamma(s-r)^{-\delta}
    \]
for some constants $\gamma, \delta > 0$. See inequality (\ref{size}).
\end{remark}

\begin{remark}
It follows from the estimate for $\phi^{s,r}$ that 
    \[
        \phi^{s,r}(b)
        =
        \id\!|_{X_r}.
    \]
\end{remark}

\begin{remark}
In some cases one can construct \emph{better} homotopy operators $\{h^{s,r}_1\}$ and $\{h^{s,r}_2\}$ which do not shrink the parameter of the nested domains.
By the nature of the proof, however, this does not improve the conclusion of the \nameref{Main Theorem}. 

Suppose that for each $r\in [0,1]$ there are linear operators
    \[
        h_1^{r}:
        \Gamma_{V_b X}(M_r)
        \longto 
        \calA(X_r), 
        \qquad 
        h_2^{r}:\Gamma_{V_\zeta Y}(M_r)
        \longto 
        \Gamma_{V_b X}(M_r),
    \]
satisfying the homotopy relation and the tame estimates:
    \[
        \delta_b
        \circ 
        h_1^{r}
        +
        h_2^{r}
        \circ
        \delta_b
        =\id,
        \qquad
        \norm{ h_i^r (v) }_{k,r}
        \leq
        c_k \norm{v}_{k+l_i, r}.
    \]
This fits into the above framework by taking $d_k=0$ and
    \[
        h^{s,r}_1\! (v)\defeq h_1^s(v)\vert_{X_r}, 
        \qquad
        h^{s,r}_2\! (w)\defeq h_2^s(w)\vert_{M_r}.
    \]
\end{remark}

\begin{remark}
Let $X \!\to\! M$ be a natural bundle as in example \ref{exa:associated_jet_bundle}:
    \[
        X = F \times_{G_{m, d}} J^d_0\calD.
    \]
Let $\{M_r\}$ be a shrinking domain on $M$, and $\{X_r\}$ be a shrinking domain on $X$ such that 
    \[
        X_r\sub \pi_\textsc{X}^{-1}(M_r),
    \]
and such that each $X_r$ has connected fibers (see example \ref{exa:constructing_X_r}).
\end{remark}

\begin{lemma}
The family of maps, for $r\in [0,1]$, defined by
    \[
        E_r:
        \calA_\calP(M_r) \longto \calA_{E\pgrp}(X_r),
        \qquad
        E_r(v) 
        \defeq 
        \tfrac{\dd}{\dd t}\big|_{t=0}
        E(\varphi_v^t)|_{X_r},
    \]
is tame: there are constants $c_k \geq 0$ such that
    \[
        \norm{ E_r(v) }_{k, r}
        \leq
        c_k \norm{ v }_{k+d, r}.
    \]
\end{lemma}

\begin{proof}
One can easily see that the map $E: \diff_X \to \diff_{J^d\calD}$ is a general PDE in the sense of definition \ref{def:gen_pde}.
If we wish to be very precise about it, then it is given by
    \begin{align*}
        \widetilde X &= M \!\times\! M \to M,
        \\
        \widetilde Y &= X \!\times\! X \to X, 
        \\
        q_0 &= \pi_\textsc{X}: J^d\calD \longto M,
        \\
        q &: J^d \widetilde X \times_{\!M} X \longto \widetilde Y, 
    \end{align*}
where for any $j^d_x(\phi) \in J^d_x(M, M)$ and $y = (f, j^d_0(\varphi)) \in X_x$,
    \[
        q \big( j^d_x(\phi),  y \big)
        =
        \big(
            f, j^d_0(\phi \circ \varphi)
        \big).
    \]
Since $Q(\id) = \id$, by \nameref{Lemma:A} there is a $C^1$-open neighborhood 
    \[
        \calU(M_s)\sub \diff_M(M_s)
    \]
of $\id\!|_{M_s}$ and there are constants $c_k \geq 0 $ such that
    \[
        \norm{ Q(\phi) - \id }_{k, r }
        \leq
        c_k
        \norm{ \phi - \id }_{k+d, s},
        \quad\forall\;
        \phi \in \calU(M_s).
    \]
The required estimate now follows from \nameref{Lemma:B}(1).
\end{proof}

\begin{remark}
For the \nameref{Main Theorem} one can obtain a conclusion of the form
    \[
        (b\cdot \psi^{s,r}(c))\vert_{M_r}
        =
        c\vert_{M_r}
    \]
instead of what is currently stated (see lemma \ref{lemma:action_on_the_other_side} for a full statement).
This can be done by taking the inverse of $\phi^{s,r}$ (with some care) and standard estimates such as in lemma \ref{lemma:diffeo}.
On the other hand, this can also be deduced by applying the main theorem as described in the following corollary.
\end{remark}

\begin{corollary}
For all $r < s$ there is a $C^p$-open neighborhood 
    \[
        \calU_{s,r}
        \sub
        C^\infty(X_s, X)
    \]
of $\id|_{X_s}$, and an operator
    $
        \psi^{s,r}:
        \calU_{s,r}
        \longto
        \diff_X(X_r)
    $
such that 
    \[
        \psi^{s,r}(\phi)(X_r)\sub X_s,
        \qquad
        \phi \circ \psi^{s,r}(\phi) = \id\!|_{X_s},
    \]
for all $\phi \in \calU_{s,r}$.
Moreover, there are constants $a_k, b_k, l\geq 0$ such that
    \[
        \norm{\psi^{s,r}(\phi) - \id}_{k,r}
        \leq
        a_k (s-r)^{-b_k}
        \norm{\phi - \id}_{k+l, s},
        \qquad\forall\;
        \phi \in \calU_{s,r}.
    \]
\end{corollary}

\begin{proof}
Consider $\widetilde M \defeq X$ and the submersion $\widetilde\pi: \widetilde X \defeq X\!\times\! X \to X$ that maps $(x,y)\mapsto x$.
Then local sections $b \in \Gamma_{\widetilde X}(U)$ correspond to smooth maps $f\in C^\infty(U, X)$ via $b = (\id\!|_U, f)$.
The local diffeomorphisms $\diff_{X}$ lift to a pseudogroup $\widetilde\pgrp$ on $\widetilde X$ where 
    \[
        \psi \in \widetilde\pgrp(U)
    \]
if and only if for every $y \in U$ there is an open neighborhood $V\sub U$ of $x$ and a local diffeomorphism $\phi \in \diff_X(\widetilde\pi(V))$ such that $\psi|_V = (\phi, \id)$.
The corresponding local action of $\widetilde\pgrp$ on $\Gamma_{\widetilde X}$ is simply composition on the right.
Its Lie algebra sheaf $\widetilde\calA$ consists of all horizontal vectors, i.e.\ those $v=(v_1, v_2) \in \gerX_{\widetilde X}$ for which $v_2=0$.
With no further restrictions in mind, we apply the \nameref{Main Theorem} to the PDE with symmetry $(\widetilde\pgrp, 0, 0)$, with $\id\in C^\infty(X, X)$ as global section and nested domains 
    \[
        \widetilde M_r 
        \defeq 
        X_r,
        \qquad
        \widetilde X_r
        \defeq
        X_r \!\times\! X_r.
    \]
By lemma \ref{prop:infinitesimal_action} we find that the corresponding deformation complex is 
    \[
        \widetilde\calA
        \xrightarrow{ \;\; \delta_{\id} \;\; }
        \Gamma_{\widetilde X}
        \xrightarrow{ \;\; 0 \;\; }
        \{ 0 \}, 
    \]
where
    \[
        \delta_{\id}(v)
        =
        v_1\circ ( \id \!\times\! \id),
        \qquad
        v = (v_1, 0) \in \widetilde\calA.
    \]
We conclude from the \nameref{Main Theorem} that for all $r<s$ there is a $C^p$-open neighborhood
    $
        \calU_{s,r} 
        \sub
        C^\infty(X_s, X)
    $
of $\id\!|_{X_s}$, 
and an operator 
    \[
        \psi^{s,r} : \calU_{s,r} \longto C^\infty(X_r, X)
    \]
such that $\psi^{s,r}(\phi)$ and $\phi$ are composable on $X_r$ for all $\phi\in\calU_{s,r}$ and 
    \[
        (\phi \cdot \psi^{s,r}(\phi))|_{X_r}
        =
        \id\!|_{X_r},
        \qquad
        |\psi^{s,r}(\phi) - \id|_{k, r}
        \leq
        a_k(s-r)^{-b_k}
        |\phi - \id|_{k+l, s}.
    \]
Remark that $\psi^{s,r}(\phi)$ and $\phi$ are compatible on $X_r$ if and only if 
    \[
        \psi^{s,r}(\phi)(X_r)
        \sub 
        X_s.
        \qedhere
    \]
\end{proof}

\begin{lemma}
\label{lemma:action_on_the_other_side}
The conclusion of the \nameref{Main Theorem} is equivalent to the following: 
there is an integer $p\in \N$, and for all $r<s$ there is a $C^p$-open neighborhood
    \[
        \calV_{s,r} \sub \sol_{Q,Z}(M_s),
    \]
of $b|_{M_s}$ and a (possibly non-linear) operator
    \[
        \begin{tikzcd}[row sep=small]
            \pgrp(X_s)
            \arrow[d,
                -,
                dotted,
                start anchor={
                    [yshift=0.25ex]south
                },
                end anchor={
                    [yshift=-0.25ex]north
                },
            ]
            &
            \calV_{s,r}
            \arrow[d,
                -,
                dotted,
                start anchor={
                    [yshift=0.25ex]south
                },
                end anchor={
                    [yshift=-0.25ex]north
                },
            ]
            \arrow[dl,
                ->,
                start anchor={
                    [xshift=.5ex,yshift=1.5ex]south west
                },
                end anchor={
                    [xshift=-.5ex,yshift=-.5ex]north east
                },
                "\psi^{s,r}" description
            ]
            \\
            \calU_{s,r}
            & 
            \sol_{Q,Z}(M_r),
        \end{tikzcd}
    \]
such that $\psi^{s,r}(c)$ and $b$ are compatible on $M_r$ for all $c \in \calV_{s,r}$ and
    \[
        (b\cdot \psi^{s,r}(c))\vert_{M_r}
        =
        c\vert_{M_r},
    \]
Moreover, there are constants $a_k, b_k, l\geq 0$ such that
    \[
        \norm{\psi^{s,r}(c)-\id}_{k,r}
        \leq
        a_k(s-r)^{-b_k}
        \norm{c-b}_{k+l,s}.
    \]
\end{lemma}

\begin{proof}
Fix $r < s$ and consider $t \defeq \tfrac12(s+r)$. 
Let $\calV_{s, t}$ and $\phi^{s, t}$ be as in the conclusion of the \nameref{Main Theorem},
and let $\calU_{t, r}$ and $\psi^{t, r}$ be as in the previous corollary.
Since $\phi^{s,t}$ is continuous at $b|_{M_s}$, we may shrink its domain such that it maps into $\calU_{s,r}$:
    \[
        \calV_{s, t}
        \xrightarrow{ \;\; \phi^{s, t} \;\; }
        \calU_{t, r}
        \xrightarrow{ \;\; \psi^{t, r} \;\; }
        \pgrp(X_r).
    \]
Write $\psi^{s,r}(c)\defeq (\psi^{t, r}\circ\phi^{s,t})(c)$ and observe that
    \[
        c|_{M_r} 
        =
        (c \cdot \id\!|_{X_r})|_{M_r}
        =
        \big(
            (c \cdot \phi^{s, t}(c))|_{M_s}
            \cdot \psi^{s,r}(c)
        \big)\big|_{M_r} 
        =
        \big(
            b|_{M_s} \cdot \psi^{s,r}(c)
        \big)|_{M_r},
    \]
and the tame estimates follow from combining the estimates of $\phi^{s, t}$ and $\psi^{t, r}$. 
The converse follows from the same argument.
\end{proof}

\clearpage
\section{Main Theorem: General Version}

We present a more general version of the \nameref{Main Theorem} which will probably cover most interesting applications. 
It differs from the main theorem by allowing a more general notion of PDE.

Recall that a PDE consists of a fiber-preserving map $q: J^dX \to Y$ and a section $\zeta\in \Gamma_Y(M)$.
We previously defined
    \[
        Q(b)\defeq q\circ j^d(b),
        \qquad
        Z \defeq \zeta(M),
    \]
and said that a section $b\in \Gamma_X(V)$ is a solution if and only if $Q(b)=\zeta|_V$.
We will relax this definition by replacing the section $\zeta$ by an arbitrary closed fibered submanifold $Z$,
and by allowing the partial differential operator to change base,
in the sense that $\pi_\textsc{Y}: Y \to N$ may be a submersion over a different base manifold.

\begin{definition}
\label{def:gen_pde}
A \textbf{general PDE}\index{partial differential equation!general} $(Q,Z)$ of order $d$ on $\pi_\textsc{X}$ consists of
\begin{itemize}

\item
a submersion $\pi_\textsc{Y}:Y \to N$,

\item
a smooth map
    $
        q_0: N \to M,      
    $

\item 
a fiber-preserving map
    $
        q: 
        J^d X \!\times_{\!M}\! N
        \to
        Y
    $
over $N$,

\item 
a closed fibered submanifold $Z\sub Y$, i.e.\ a closed submanifold such that 
    $
        \pi_\textsc{Z} \defeq \pi_\textsc{Y}|_Z: Z\to N
    $
is still a submersion.

\end{itemize}
This data induces a map of sheaves over $q_0$ by
    \[
        Q \sep \Gamma_X \longto \Gamma_Y,
        \qquad
        Q(b)
        \defeq 
        q\circ j^db\circ q_0.
    \]
A solution is a local section $b\in \Gamma_X(V)$ such that $Q(b)$ maps into $Z$,
and the sheaf of solutions will be denoted by
    \[
        \sol_{Q,Z}(V)
        \defeq
        \big\{
            b \in \Gamma_X(V)
            \sep
            Q(b) \in \Gamma_Z(q_0^{-1}(V))
        \big\}.
        \qedhere
    \]
\end{definition}

In summary, a general PDE fits into the following diagram:
   \[
    \begin{tikzcd}
        X
        \arrow[->]{d}
        &
        J^dX \times_{\! M} N
        \arrow[->]{d}
        \arrow[->]{r}{q}
        &
        Y
        \arrow[->, start anchor = south west, end anchor = north east]{dl}
        \\
        M
        &
        \arrow[->, swap]{l}{q_0}
        N
        &
        \arrow[hook, swap]{u}
        \arrow{l}
        Z
    \end{tikzcd}
    \]

\clearpage
\begin{example}
Let $(P, \pi)$ be a Poisson manifold. 
Recall that a submanifold $M^m$ is called co-isotropic if and only if $\pi^\sharp$ restricts to a map $(TM)^\circ \!\to\! TM$, where 
    $
        (TM)^\circ \sub T^*_MP
    $
is the conormal bundle of $M\sub P$.
In other words, one requires that the induced map $(TM)^\circ \!\to\! N M$ vanishes,
and by the cannonical isomorphism $TM^\circ \simeq (N M)^*$, this is equivalent to $\pi|_M$ lying in the kernel of the projection
    \[
        p: \wedge^2 T_MP \longto \wedge^2 N M.
    \]
Submanifolds of $P$ diffeomorphic to $M$ can be represented either
intrinsically as parametrized submanifolds or extrinsically as graphs over $M$.

Parametrized submanifolds are represented by embeddings $M \!\hookrightarrow\! P$, which form a strong $C^1$-open subset of $C^\infty(M, P)$.
Consider the Grassmannian bundle of $TP$, pulled back to the trivial bundle $P \!\times\! M$:
    \[
        X \defeq P \!\times\! M \longto M,
        \qquad
        Z \defeq \mathrm{Gr}_m(TP) \!\times\! M \longto X.
    \]
Let $U \sub J^1X$ be the open subset of all one-jets of maximal rank, i.e.\ the $j^1_x(\phi)$ for which $\dd_x\phi$ is of maximal rank.
There is a smooth bundle map over $X$ given by
    \[
        \tau
        \sep
        J^1X \supseteq U
        \longto
        \mathrm{Gr}_m(TP) \!\times\! M,
        \qquad
        \tau( j^1_x(\phi) )
        \defeq
        \dd_x\phi(T_xM).
    \]
Let $E \to \mathrm{Gr}_m(TP)$ be the tautological bundle, and consider the bundle
    \[
        Y 
        \defeq
        \wedge^2 (\underline{TP}/ E) \!\times\! M
        \longto X,
    \]
where $\underline{TP}\defeq TP \times \mathrm{Gr}_k(TP)$.
Since $Y$ is a vector bundle over $Z$, we may identify $Z$ with the image of the zero section, and thus consider it as a submanifold of $Y$.
Moreover, let
    \[
        \sigma
        \sep
        \mathrm{Gr}_m(TP) 
        \longto
        \wedge^2 (\underline{TP} / E)
    \]
be the map obtained by projecting $\pi$ onto $TP/E$.
We conclude that the co-isotropy condition fits into the framework of a general PDE as
    \[
        \begin{tikzcd}
            J^1 X \supseteq U
            \arrow[->]{d}
            \arrow[->]{r}{\sigma \circ \tau}
            &
            Y
            \arrow[
                ->, 
                start anchor = south west, 
                end anchor = north east
            ]{dl}
            \\
            M
            &
            \arrow[hook, swap]{u}
            \arrow{l}
            Z.
        \end{tikzcd}
    \]
\end{example}

\begin{example}
\label{exa:has_coisotropic}
Let $(P, \pi)$ be a manifold and $M$ a submanifold. 
We will turn the previous example on its head: 
recall that a Poisson bivector $\pi\in \gerX^2_P$ has $M$ as a co-isotropic submanifold if $\pi^\sharp$ restricts to a map $(TM)^\circ \!\to\! TM$.
Hence such a bivector needs to satisfy two separate equations: the Poisson equation $[\pi, \pi]=0$ and the above co-isotropy condition.

Consider the bundle map over $P$ given by
    \[
        q
        \sep
        J^1(\wedge^2 TP) \longto \wedge^3 TP,
        \quad
        q(j^1_x\pi)
        \defeq
        [\pi,\pi]_x,
    \]
where $\pi \in \gerX^2_M(U)$, with $x\in U$, is any local representative of $j^1_x\pi$.
This is well-defined because the Schouten-Nijenhuis bracket $[-,-]$ on multivector fields satisfies the Leibniz identity.
Consider also the bundle map
    \begin{align*}
        & r
        \sep
        \wedge^2 T_MP
        \longto
        \wedge^2 NM,
        \quad
        r(\pi_x)
        \defeq
        \pi_x \!\!\! \mod TM.
    \end{align*}
We conclude that this problem fits into the framework of general PDEs as
    \[
    \begin{tikzcd}
        \wedge^2 TM
        \arrow[->]{d}
        &
        J^1(\wedge^2 TM) \sqcup \wedge^2 T_NM
        \arrow[->]{d}
        \arrow[->]{r}{q\;\sqcup\;r}
        &
        \wedge^3 TM \sqcup \wedge^2 NM
        \arrow[->, start anchor = south west, end anchor = north east]{dl}
        \\
        M
        &
        \arrow[->, swap]{l}
        M \sqcup N
        &
        \arrow[hook, swap]{u}
        \arrow{l}
        M \sqcup N.
    \end{tikzcd}
    \]
Here the symbol $\sqcup$ stands for the disjoint union.
\end{example}

\begin{definition}
\label{def:gen_pde_with_symmetry}
A \textbf{PDE with symmetry}\index{partial differential equation!with symmetry} $(\pgrp, Q, Z)$ consists of
\begin{itemize}

\item
a \emph{closed} pseudogroup $\pgrp$ on $X$,

\item
a PDE $(Q,Z)$ on $\pi_\textsc{X}$, where $q_0$ is \emph{proper},

\end{itemize}
such that
if $\phi \in \pgrp(U)$ and $b\in \sol_{Q, Z}(V)$ are compatible on $W$, then 
    \[
        (b\cdot \phi)|_W
        \in 
        \sol_{Q,Z}(W).
        \qedhere
    \]
\end{definition}

Given a submersion $\pi_\textsc{X}: X \to M$, fix the following objects:
\begin{itemize}

\item
a general PDE with symmetry $(\pgrp, Q, Z)$,

\item
the Lie algebra sheaf $\calA$ of $\pgrp$ (see theorem \ref{pgrp:thm}),

\item
a global solution $b \in \sol_{Q, Z}(M)$, and denote $\zeta \defeq Q(b)$,

\item 
nested domains $\{ M_r \}$, $\{ N_r \}$, and $\{ X_r \}$ in $M$, $N$ and $X$ respectively (see definition \ref{def:nested_domain}) such that
    \[
        b(M_r) \sub X_r,
        \qquad
        q_0(N_r) \sub M_r,
        \qquad\forall\;
        r \in [0,1].
    \]
\end{itemize}

Recall that $N(Y, Z)$ denotes the normal bundle of $Z$ in $Y$ and $N_\zeta(Y,Z)$ its pull-back along $\zeta$ (see definition \ref{def:normal_bundle}).

\begin{definition}
The second $\delta_b$ of the deformation complex (see definition \ref{def:deformation_complex}) can be replaced by the linear operator
    \[
        \delta_b:
        \Gamma_{V_bX}(U)
        \longto
        \Gamma_{N_\zeta(Y, Z)}(q_0^{-1}(U)),
        \qquad
        \delta_b(w)
        \defeq
        \tfrac{\dd}{\dd t}\big|_{t=0}
        Q(b^t)
        \!\!\!\mod T_\zeta Z,
    \]
where $b^t \in \Gamma_X(U^t)$ is a family of sections with $b^0=b$ and $\tfrac{\dd}{\dd t}\big|_{t=0} b^t = w$.
\end{definition}

\begin{NamedTheorem}[Main Theorem (General)]
\label{General Main Theorem}
Tame infinitesimal rigidity implies rigidity:

Suppose there are linear operators, for all $r<s$,
    \[
        \begin{tikzcd}[row sep=small]
            \calA(X_s)
            \arrow{r}
            \arrow[d,
                -,
                dotted,
                start anchor={
                    [yshift=0.25ex]south
                },
                end anchor={
                    [yshift=-0.25ex]north
                },
            ]
            &
            \Gamma_{V_b X}(M_s)
            \arrow{r}
            \arrow[d,
                -,
                dotted,
                start anchor={
                    [yshift=0.25ex]south
                },
                end anchor={
                    [yshift=-0.25ex]north
                },
            ]
            \arrow[dl,
                start anchor={
                    [xshift=1ex,yshift=1ex]south west
                },
                end anchor={
                    [xshift=-1ex,yshift=-1ex]north east
                },
                "h^{s,r}_1" description
            ]
            &
            \Gamma_{N_\zeta(Y, Z)}(N_s)
            \arrow[d,
                -,
                dotted,
                start anchor={
                    [yshift=0.25ex]south
                },
                end anchor={
                    [yshift=-0.25ex]north
                },
            ]
            \arrow[dl,
                start anchor={
                    [xshift=1ex,yshift=1ex]south west
                },
                end anchor={
                    [xshift=-1ex,yshift=-1ex]north east
                },
                "h^{s,r}_2" description
            ]
            \\
            \calA(X_r)
            \arrow{r}
            &
            \Gamma_{V_bX}(M_r)
            \arrow{r}
            &
            \Gamma_{N_\zeta(Y, Z)}(N_r)
        \end{tikzcd}
    \]
and there are constants $c_k, d_k, l_1, l_2 \geq 0$, such that
    \[
            (\delta_b\circ h^{s,r}_1
            +
            h^{s,r}_2\circ \delta_b)(w)
            =
            w|_{M_r},
            \qquad
            \norm{ h_i^{s,r}(v) }_{k,r}
            \leq
            c_k(s-r)^{-d_k}
            \norm{ v }_{k+l_i, s}.
    \]
Then there is a $p\in \N$, and for all $r<s$ there is a $C^p$-open neighborhood
    \[
        \calV_{s,r} \sub \sol_{Q,Z}(M_s)
    \]
of $b|_{M_s}$ and a (non-linear) operator
    \[
        \begin{tikzcd}[row sep=small]
            \pgrp(X_s)
            \arrow[d,
                -,
                dotted,
                start anchor={
                    [yshift=0.25ex]south
                },
                end anchor={
                    [yshift=-0.25ex]north
                },
            ]
            &
            \calV_{s,r}
            \arrow[d,
                -,
                dotted,
                start anchor={
                    [yshift=0.25ex]south
                },
                end anchor={
                    [yshift=-0.25ex]north
                },
            ]
            \arrow[dl,
                ->,
                start anchor={
                    [xshift=.5ex,yshift=1.5ex]south west
                },
                end anchor={
                    [xshift=-.5ex,yshift=-.5ex]north east
                },
                "\phi^{s,r}" description
            ]
            \\
            \pgrp(X_r)
            &
            \sol_{Q,Z}(M_r),
        \end{tikzcd}
    \]
such that $\phi^{s,r}(c)$ and $c$ are compatible on $M_r$ for all $c\in \calV_{s,r}$ and
    \[
        (c\cdot \phi^{s,r}(c))\vert_{M_r}
        =
        b\vert_{M_r}.
    \]
Moreover, there are constants $a_k, b_k, l\geq 0$ such that
    \[
        \norm{\phi^{s,r}(c)-\id}_{k,r}
        \leq
        a_k (s-r)^{-b_k}
        \norm{c-b}_{k+l,s}.
    \]
\end{NamedTheorem}

\subsection{Multiple Differential Equations}
\label{remk:more equations}

One may care about studying the solutions of a finite collection of partial differential equations $\{(Q^i,Z^i)\}_{i=1}^l$ given by partial differential operators
    \[
        Q^i:\Gamma_X(M)
        \longto
        \Gamma_{Y^i}(N^i),
    \]
covering base maps $q^i_{0}:N^i \to M$.
The sheaf of solution set is then naturally the intersection of solution sets:
    \[
        \sol_{\{Q^i,Z^i\}}(U)
        =
        \medcap_{i}
        \sol_{Q^i,Z^i}(U).
    \]
Let $\pgrp$ be a closed pseudogroup which preserves $\sol_{\{Q^i,Z^i\}}$ (in the sense of definition \ref{def:pde_with_symmetry}), and fix a global solution $b\in \sol_{\{Q^i, Z^i\}}(M)$. 
The \nameref{General Main Theorem} can be applied in this situation by enlarging the class of manifolds that we are working with to allow $N$, $Y$, and $Z$ to be disjoint unions of finitely many manifolds of possibly different dimensions. 
It should be clear from the proofs, and from the \nameref{Main Theorem for Differential Relations} presented in the next section, that this poses no problem. 
Hence we consider
    \[
    \begin{tikzcd}
        X
        \arrow[->]{d}
        &
        \medsqcup_i (J^{d_i}X_i \!\times_{\!M_i}\! N_i)
        \arrow[->]{d}
        \arrow[->]{r}{\medsqcup_i q^i}
        &
        \medsqcup_i Y^i
        \arrow[->, start anchor = south west, end anchor = north east]{dl}
        \\
        M
        &
        \arrow[->, swap]{l}
        \medsqcup_i N^i
        &
        \arrow[hook, swap]{u}
        \arrow[swap]{l}{\medsqcup_i q_0^i}
        \medsqcup_i Z^i.
    \end{tikzcd}
    \]
Remark that $\medsqcup_iq_0^i$ is proper if and only if all of the maps $q_0^i$ are proper.
The \nameref{General Main Theorem} permits us to choose nested domains $\{X_r\}$, $\{ M_r\}$ and $\{N_r^i \}$ such that 
    \[
        b(M_r)\sub X_r,
        \qquad
        q_0^i(N_r^i)\sub M_r,
        \qquad
        \forall\;
        r \in [0,1]
        \quad\forall\;
        i=1, \ldots, l.
    \]
The second map of the deformation complex thus takes the form
    \begin{align*}
        N_r
        & \defeq
        \medsqcup_i N_r^i,
        \quad
        N_{\zeta}(Y, Z)
        \defeq
        \medoplus_i
        N_{\zeta^i}(Y^i, Z^i),
        \\
        \delta_b
        & \defeq
        (\delta_b^1, \ldots, \delta^l_b)
        \sep
        \Gamma_{V_b X}(M_r)
        \longto 
        \Gamma_{N_\zeta(Y, Z)}(N_r).
    \end{align*}
To satisfy the hypothesis of the theorem one needs to find linear operators
    \begin{align*}
        h_1^{s,r}:
        \Gamma_{V_b X}(M_s) 
        \longto \calA(X_r),
        \quad
        h_2^{s,r,i}:
        \Gamma_{N_{\zeta^i}(Y^i, Z^i)}(N_s^i) 
        \longto 
        \Gamma_{V_b X}(M_r),
    \end{align*}
which satisfy the tameness inequalities, and the homotopy relation:
    \[
        \delta_b \circ h_1^{s,r}
        +
        \sum\nolimits_{i=1}^l h_2^{s,r,i}\circ \delta_b
        =
        \id^{s,r},
    \]
This is equivalent to the usual homotopy relation for the operator
    \[
        h_2^{s,r}
        \defeq
        \sum\nolimits_{i=1}^l
        h_2^{s,r,i}\circ \pr_i:
        \Gamma_{N_\zeta(Y, Z)}(N_s)
        \longto 
        \Gamma_{V_bX}(M_r),
    \]
where the $\pr_i$ denote the projections onto $\Gamma_{N_{\zeta_i}(Y^i, Z^i)}(N_s^i)$.

\clearpage
\section{Main Theorem: Differential Relations}

The differential equations that appear in the \nameref{General Main Theorem}, and in the version of subsection \ref{remk:more equations}, are particular examples of so-called partial differential relations on a submersion. 
A differential relation is the bare minimum of data needed to speak of \emph{imposing a condition on the derivatives of sections}: a subset of the jet-bundle.
We will formulate a version of the main theorem in the language of differential relations,
and we will show that the \nameref{General Main Theorem} is a special case of this theorem.

Some aspects (such as picking a tubular neighborhood) make this statement less cannonical.
Moreover, without a partial differential operator providing us with a linear chain complex, the aphorism 
\emph{infinitesimal rigidity implies rigidity} 
is lost in the details. 
Instead we include certain \emph{quadratic estimates} in the hypothesis that would otherwise follow from tame  infinitesimal rigidity (see \nameref{Lemma:A}).
We start with recalling the definition of a PDR \cite{Gro86}.

Let $\pi_\textsc{X}: X \!\to\! M$ be a submersion.

\begin{definition}
A \textbf{partial differential relation (PDR)}\index{partial differential relation} of order $d$ on $\pi_\textsc{X}$ is a subset 
    \[
        R \sub J^dX.
    \]
A local solution of $R$ is a local section $b\in \Gamma_X(U)$ such that
    \[
        j^db(U)\sub R.
    \]
The sheaf of solutions will be denoted by 
    \[
        \sol_R(U) \defeq
        \big\{
            b \in \Gamma_X(U)
            \sep
            j^db(U) \sub R
        \big\}.
        \qedhere
    \]
\end{definition}

\begin{example}
Recall that a general PDE on $\pi_\textsc{X}$ consists of a submersion $\pi_\textsc{Y}: Y \!\to\! N$, a smooth map $q_0: N\!\to\! M$, a fiber preserving map 
    \[
        q: J^dX \!\times_{\!M}\! N \!\to\! Y
    \]
and a closed fibered submanifold $Z\sub Y$. 
This defines a closed partial differential relation
    \[
        R \defeq 
        \big\{
            j^d_{y}(b) \in J^dX
            \sep
            q_x(j^d_y(b)) \in Z,
            \;\;
            q_0(x) = y            
        \big\}
    \]
such that
    $
        \sol_{Q, Z}(V) = \sol_R(V)
    $
for all $V\sub M$.
Some information is lost, namely the submersion $\pi_Y$, the (higher order) derivatives of $q$ along $R$, and the behaviour of $q$ away from $R$.
\end{example}

\begin{definition}
\label{def:pdr_with_symmetry}
A \textbf{PDR with symmetry}\index{partial differential relation!with symmetry} $(\pgrp, R)$ of order $d$ consists of

\begin{itemize}

\item 
a \emph{closed} pseudogroup $\pgrp$ (see definition \ref{def:closed_pseudogroup}),

\item
a partial differential relation $R$ of order $d$ on $\pi_\textsc{X}$,

\end{itemize}

such that if $\phi \in \pgrp(U)$ and $b\in\sol_R(V)$ are compatible on $W$, then 
    \[
        (b\cdot\phi)|_W \in \sol_R(W).
        \qedhere
    \]
\end{definition}

Given a submersion $\pi_\textsc{X}: X \!\to\! M$, consider the following objects:
\begin{itemize}

\item
a PDR with symmetry $(\pgrp, R)$ of order $d$,

\item
the Lie algebra sheaf $\calA$ of $\pgrp$ (see theorem \ref{pgrp:thm}),

\item
a global solution $b \in \sol_R(M)$,

\item
a vertical tubular neighborhood
    $
        \xi: V_b X \hookrightarrow X
    $
around $b$.

\item 
nested domains $\{ M_r \}$ and $\{ X_r \}$ in $M$ and $X$ respectively (see definition \ref{def:nested_domain}) such that
    \[
        b(M_r) \sub X_r,
        \quad\forall\;
        r \in [0,1].
    \]

\end{itemize}

\begin{NamedTheorem}[Main Theorem for PDRs]
\label{Main Theorem for Differential Relations}
Tame quadratic estimates imply rigidity:

Suppose there is a linear operator, for all $r<s$,
    \[
        \begin{tikzcd}[row sep=small]
            \calA(X_s)
            \arrow{r}
            \arrow[d,
                -,
                dotted,
                start anchor={
                    [yshift=0.25ex]south
                },
                end anchor={
                    [yshift=-0.25ex]north
                },
            ]
            &
            \Gamma_{V_b X}(M_s)
            \arrow[d,
                -,
                dotted,
                start anchor={
                    [yshift=0.25ex]south
                },
                end anchor={
                    [yshift=-0.25ex]north
                },
            ]
            \arrow[dl,
                start anchor={
                    [xshift=1ex,yshift=1ex]south west
                },
                end anchor={
                    [xshift=-1ex,yshift=-1ex]north east
                },
                "h^{s,r}" description
            ]
            \\
            \calA(X_r)
            \arrow{r}
            &
            \Gamma_{V_bX}(M_r)
        \end{tikzcd}
    \]
and there are constants $c_k, d_k, l_1, l_2 \geq 0$ and $\theta >0$, such that
    \[
            \norm{ h^{s,r}(v) }_{k,r}
            \leq
            c_k(s-r)^{-d_k}
            \norm{ v }_{k+l_1, s}
    \]
for all $v \in \Gamma_{V_bX}(M_s)$, and such that
    \begin{equation}
        \label{ineq:quadratic}
        \norm{
            v_{|M_r}
            -
            \delta_b\circ h^{s,r}(v)
        }_{k,r}
        \leq 
        c_k(s-r)^{-d_k}
        \norm{v}_{d,s} 
        \norm{v}_{k+l_2,s}
    \end{equation}
for all $v\in \Gamma_{V_bX}(M_s)$ for which
    $
        \norm{v}_{d,s}
        <
        \theta
    $
and
    $
        (\xi\circ v) \in \sol_R(M_s).
    $

Then there is an integer $p\in \N$, and for all $r<s$ there is a $C^p$-open neighborhood
    \[
        \calV_{s,r} \sub \sol_R(M_s)
    \]
of $b|_{M_s}$ and a (possibly non-linear) operator
    \[
        \begin{tikzcd}[row sep=small]
            \pgrp(X_s)
            \arrow[d,
                -,
                dotted,
                start anchor={
                    [yshift=0.25ex]south
                },
                end anchor={
                    [yshift=-0.25ex]north
                },
            ]
            &
            \calV_{s,r}
            \arrow[d,
                -,
                dotted,
                start anchor={
                    [yshift=0.25ex]south
                },
                end anchor={
                    [yshift=-0.25ex]north
                },
            ]
            \arrow[dl,
                ->,
                start anchor={
                    [xshift=.5ex,yshift=1.5ex]south west
                },
                end anchor={
                    [xshift=-.5ex,yshift=-.5ex]north east
                },
                "\phi^{s,r}" description
            ]
            \\
            \pgrp(X_r)
            & 
            \sol_R(M_r),
        \end{tikzcd}
    \]
such that $\phi^{s,r}(c)$ and $c$ are compatible on $M_r$ for all $c\in \calV_{s,r}$ and
    \[
        (c\cdot \phi^{s,r}(c))\vert_{M_r}
        =
        b\vert_{M_r}.
    \]
Moreover, there are constants $a_k, b_k, l\geq 0$ such that
    \[
        \norm{\phi^{s,r}(c)-\id}_{k,r}
        \leq
        a_k (s-r)^{-b_k}
        \norm{c-b}_{k+l,s}.
        \qedhere
    \]
\end{NamedTheorem}

\chapter{Applications}
\label{chapter: Applications}
\etocsettocdepth{2}
\localtableofcontents

\newpage
\section{Rigidity of Lie Algebras}
Let $(\gerg, \mu)$ be a finite dimensional Lie algebra.

\begin{proposition}
\label{Proposition:lie_algebra_structures}
If $H^2(\gerg, \mathrm{ad}_\mu) = 0$, then there is an open neighborhood 
    \[
        \calV 
        \sub 
        \big\{
            \text{Lie algebra structures $\nu$ on $\gerg$}
        \big\}
    \]
around $\mu$ and a map
    \[
        g : \calV \longto \mathrm{GL}(\gerg)
    \]
such that $\nu \cdot g(\nu) = \mu$ for $\nu \in \calV$, and $g(\mu) = \id$ and $g$ is continuous at $\mu$.
\end{proposition}

\begin{proof}
To apply the \nameref{Main Theorem} we need to specify 
\begin{enumerate*}

\item
the PDE that determines Lie algebra structures,

\item
the closed (pseudo)group that acts on the Lie algebra structures,

\item
auxiliary nested domains and norms,

\item
the homotopy operators.

\end{enumerate*}
We will do so in that order.

Recall that a Lie algebra structure on $\gerg$ is an anti-symmetric bilinear form $\mu : \wedge^2 \gerg \to \gerg$ that satisfies the Jacobi identity $\mathrm{Jac}(\mu)=0$ where
    \begin{align*}
        &
        \mathrm{Jac}(\mu)(x, y, z)
        \defeq
        \mu(\mu(x, y), z)
        +
        \mu(\mu(y, z), x)
        +
        \mu(\mu(z, x), y).
    \end{align*}
The Jacobi identity is a PDE in the sense of definition \ref{def:pdo} as
    \[
        \begin{tikzcd}
            \mathrm{Hom}(\wedge^2\gerg, \gerg)
            \arrow[->]{d}
            \arrow[->]{r}{\mathrm{Jac}}
            &
            \mathrm{Hom}(\wedge^3\gerg, \gerg)
            \arrow[
                ->, 
                start anchor = south west, 
                end anchor = north east
            ]{dl}
            \\
            \{*\}
            &
            \arrow[hook, swap]{u}{0}
            \arrow{l}
            \{*\}.
        \end{tikzcd}
    \]
Remark that the general linear group $\mathrm{GL}(\gerg)$ acts on $\wedge^2 \gerg^* \otimes \gerg$ by
    \[
        (\mu\cdot g)(x, y)
        \defeq
        g^{-1} \mu(gx, gy),
    \]
and this actions sends Lie brackets to Lie brackets.
Let 
    \[
        \diff_{\mathrm{GL}(\gerg)} \sub \diff_{\mathrm{Hom}(\wedge^2\gerg, \gerg)}
    \]
be the pseudogroup on induced by this action (see example \ref{exa:lie_group}).
It is not clear to us whether this pseudogroup is closed, which amounts to proving that the image of the action is a closed subset of $\mathrm{GL}(\wedge^2\gerg^*\otimes\gerg)$.
However, there is a simple method to avoid this issue:
consider instead the equation
    \[
        \begin{tikzcd}
            \gerg
            \times
            \mathrm{Hom}(\wedge^2\gerg, \gerg)
            \arrow[->]{d}
            \arrow[->]{r}{\id \!\times \mathrm{Jac}}
            &
            \gerg
            \times
            \mathrm{Hom}(\wedge^3\gerg, \gerg)
            \arrow[
                ->, 
                start anchor = south west, 
                end anchor = north east
            ]{dl}
            \\
            \{*\}
            &
            \arrow[hook, swap]{u}{(0, 0)}
            \arrow{l}
            \{*\}.
        \end{tikzcd}
    \]
Although the involved spaces are bigger, the solutions are just pairs $(0, \mu)$ where $\mu$ is a Lie bracket.
Now $\mathrm{GL}(\gerg)$ acts on $\gerg \times \mathrm{Hom}(\wedge^2\gerg, \gerg)$ by
    \[
        (v, \mu)\cdot g
        \defeq
        (g^{-1}v, \mu\cdot g),   
    \]
and it still maps solutions to solutions.
This is a faithful representation, hence by example \ref{exa:closed_lie_group} the induced pseudogroup 
    \[
        \diff_{\mathrm{GL}(\gerg)}
        \sub
        \diff_{\gerg \times \mathrm{Hom}(\wedge^2\gerg, \gerg)}
    \]
is closed according to example \ref{exa:closed_lie_group}.
Hence the triple 
    \[
        \pgrp = \diff_{\mathrm{GL}(\gerg)},
        \quad 
        q = \id \!\times \mathrm{Jac},
        \quad
        \zeta = (0,0),
    \]
forms a PDE with symmetry in the sense of definition \ref{def:pde_with_symmetry}.
Its Lie algebra sheaf $\calA$ is the one induced by the infinitesimal action (see example \ref{exa:lie_algebra}):
    \[
        \calA = \calA_{\gergl(\gerg)} \sub \gerX_{\gerg \times \mathrm{Hom}(\wedge^2\gerg, \gerg)}.
    \]

Formally, in order to apply the \nameref{Main Theorem} we have to fix a nested domain $\{ X_r \}$ in $\gerg \times \mathrm{Hom}(\wedge^2 \gerg, \gerg)$.
In the current setting this choice does not matter much, but for the sake of completeness we pick
    \[
        X_r = \{ 
            (x, \nu) \in \gerg\times\mathrm{Hom}(\wedge^2\gerg, \gerg)
            \sep
            | x | \leq 1 + r,
            \;\;
            | \nu - \mu | \leq 1+r
        \}
    \]
for an arbitrary choice of metrics on $\gerg$ and $\mathrm{Hom}(\wedge^2\gerg, \gerg)$.
The other nested domain is trivial: $M_r = \{*\}$.

The infinitesimal action of $\gergl(\gerg)$ on $\gerg\times \mathrm{Hom}(\wedge^2\gerg, \gerg)$ gives maps
    \[
        \rho^r : \mathrm{Hom}(\gerg, \gerg) \longto \calA(X_r),
        \qquad
        \alpha \longmapsto \rho(\alpha)|_{X_r}
    \]
where we identify $\gergl(\gerg)= \mathrm{Hom}(\gerg, \gerg)$. After composing with $\rho^r$ the deformation complex becomes
    \[
        \mathrm{Hom}(\gerg, \gerg)
        \xrightarrow{\;\; (-\id, \delta_\mu) \;\;}
        \gerg \times \mathrm{Hom}(\wedge^2\gerg, \gerg)
        \xrightarrow{\;\; \id \!\times \delta_\mu \;\;}
        \gerg \times \mathrm{Hom}(\wedge^3\gerg, \gerg),
    \]
where $\delta_\mu$ is the Chevally-Eilenberg differential of the adjoint representation:
    \begin{align*}
        \delta_\mu(\alpha)(x, y)
        & =
        \mu(\alpha x, y) - \mu(\alpha y, x) - \alpha \mu(x, y),
        \\
        \delta_\mu(\beta)(x, y, z)
        & =
        \mu(x, \beta(y, z) - \mu(y, \beta(x, z)) + \mu(z, \beta(x, y))
        \\
        & +
        \beta(\mu(x, y), z) - \beta(\mu(x, z), y) + \beta(\mu(y, z), x),
    \end{align*}
for all $\alpha \in \mathrm{Hom}(\gerg, \gerg)$ and $\beta \in \mathrm{Hom}(\wedge^2\gerg, \gerg)$.

Now suppose that $H^2(\gerg, \ad_\mu) = 0$.
Then there are linear maps  
    \[
        \mathrm{Hom}(\gerg, \gerg)
        \xleftarrow{\;\; \tilde h_1 \;\;}
        \mathrm{Hom}(\wedge^2\gerg, \gerg)
        \xleftarrow{\;\; \tilde h_2 \;\;}
        \mathrm{Hom}(\wedge^2\gerg, \gerg)
    \]
such that $\delta_\mu \circ \tilde h_1 + \tilde h_2 \circ \delta_\mu = \id$.
Then define homotopy operators by
    \[
        h_1^r \defeq (0, \rho^r \circ \tilde h_1),
        \qquad
        h_2^r \defeq \id \times \tilde h_2.
    \]
These maps are clearly tame for any choice of $C^k$-norms on $\calA$, hence the hypothesis of 
the \nameref{Main Theorem} is satisfied.
\end{proof}

\newpage
\section{Darboux Theorem}

In symplectic geometry the so-called Moser trick shows that symplectic structures are locally rigid.
We show how to deduce this standard result from the \nameref{Main Theorem}. 

Let $(M,\omega)$ be a symplectic manifold, let $x\in M$, and let $M_r$ be a ball of radius $\epsilon(1+r)$ centered at $x$.

\begin{proposition}
The symplectic form $\omega$ is locally rigid among symplectic forms relative to $\{ M_r \}$, i.e.\ for all $r<s$ there is a $C^p$-open neighborhood 
    \[
        \calV_{s,r} \sub \Omega^2_\mathrm{cl}(M_s)
    \]
of $\omega|_{M_s}$ and an operator
    \[
        \phi^{s,r} : \calV_{s,r} \longto \diff_M(M_r)
    \]
such that $\phi^{s,r}(\varpi)^*\varpi = \omega|_{M_r}$ for all $\varpi\in \calV_{s,r}$.
\end{proposition}

\begin{proof}

Let $X\sub \wedge^2 T^*M$ the open subset of non-degenerate $2$-forms, let $Y= \wedge^3 T^*M$ and let $Q=\dd$ be the de Rham differential. 

Let $\pgrp$ be the closed pseudogroup on $\wedge^2 T^*M$ induced by the natural lift (see section \ref{sec:natural_bundle}). 
Its Lie algebra sheaf $\calA$ identifies with $\gerX_M$ on any open subset $U\sub X$ with connected fibers, i.e.\ for any $v\in \calA(U)$ there is a unique $w\in \gerX(\pi_\textsc{X}(U))$ such that $v = \hat w|_U$, where $\hat w$ is the generator of 
    \[
        (\varphi^t_w)_*: X|_{\pi_\textsc{X}(U)} \!\to\! X.
    \]
Let $\{M_r\}$ be a nested domain around $x$, defined by picking closed balls around $x$.
Let $\{X_r\}$ be any nested domain in $X$ with connected fibers and such that 
    \[
        \omega(M_r)\sub X_r,
        \qquad
        X_r \sub \pi_\textsc{X}^{-1}(M_r),
    \]
for example as constructed in example \ref{exa:constructing_X_r}. 
For every $r\in [0,1]$, we obtain a tame isomorphism
    \[
        \gerX_M(M_r) \xrightarrow{\;\; \simeq \;\;} \calA(X_r).
    \]
After identifying $TM \simeq T^*M$ via $\omega^\sharp$, one finds the deformation complex
    \[
        \Omega^1(M_r)
        \xrightarrow{\;\; \dd \;\;} 
        \Omega^2(M_r) 
        \xrightarrow{\;\; \dd \;\;} 
        \Omega^3(M_r).
    \]
The map first map is given by 
    \[
        \delta_\omega((\omega^\sharp)^{-1}\alpha)
        =
        \gerL_{(\omega^\sharp)^{-1}\alpha}\omega 
        = 
        \dd(i_{(\omega^\sharp)^{-1}\alpha}\omega) 
        = 
        \dd\alpha,
    \]
hence the deformation complex is just the de Rham complex. 
One easily checks that the Poincar\'e Lemma provides tame homotopy operators 
    \[
        h_r^i :\Omega^{i+1}(M_r)\longto \Omega^i(M_r).
        \qedhere
    \]
\end{proof}

Local rigidity implies the existence of Darboux coordinates, as is clearly stated in the following corollary.

\begin{corollary}
\label{corollary:darboux}
For every $x\in M$ there is a chart $\xi :U\to \R^{2n}$ around $x$ such that
    \[
        \xi^*\omega_\mathrm{can}=\omega\vert_U.
    \]
\end{corollary}

\begin{proof}
It is suffices to prove this for $M=\R^{2n}$ and $x=0$. 
Let $B_r$ denote the closed ball of radius $r$ centered at the origin.

After a linear change of coordinates one may assume that $\omega(0)=\omega_\mathrm{can}$. 
Let $m_\epsilon:\R^{2n} \!\to\! \R^{2n}$ denote scalar multiplication by $\epsilon>0$. 
Then 
    \[
        \omega^\epsilon
        \defeq
        \epsilon^{-2}\,
        m_\epsilon^*\omega
    \] 
converges to $\omega_\mathrm{can}$ on $B_2$ as $\epsilon$ tends to $0$.
Hence for some $\epsilon>0$ there is a local diffeomorphism $\phi: B_1 \!\to\! B_2$ such that 
    \[
        \phi^*\omega^\epsilon = \omega_\mathrm{can}|_{B_1}.
    \] 
Since $m_\epsilon^*\omega_\mathrm{can}= \epsilon^2\omega_\mathrm{can}$, this implies
    \[
        (m_\epsilon\phi m_\epsilon^{-1})^*\omega
        =
        (\phi m_\epsilon^{-1})^* 
        \big(
            \epsilon^2\,\omega^\epsilon
        \big)
        =
        (m_\epsilon^{-1})^*
        \big(
            \epsilon^2\,\omega_\mathrm{can}|_{B_1}
        \big)
        =
        \omega_\mathrm{can}|_{B_\epsilon}.
    \]
So the conclusion holds for $\xi = m_\epsilon \circ \phi^{-1} \circ m_\epsilon^{-1}$ and $U = \epsilon \cdot\phi(B_1)$.
\end{proof}

\newpage
\section{Newlander-Nirenberg Theorem}
\label{sect:newlander-nirenberg}
\etocsettocdepth{2}
\localtableofcontents

{\color{white}.}
\newline
The Newlander-Nirenberg Theorem gives holomorphic charts for manifolds with an involutive almost complex structure, therefore realizing it is a genuine complex manifold. Here we show this result follows from the main Theorem. The Newlander-Nirenberg Theorem can be understood as a local normal form problem, as a holomorphic chart is in essence an isomorphism between the almost complex structure and the standard complex structure $I$ on $\bbR^{2n}=\bbC^n$. Hence showing that $I$ is locally rigid, combined with a standard argument, leads to a proof of the Newlander-Nirenberg Theorem.

\subsection{Newlander-Nirenberg}

Recall that an almost complex structure on a manifold $M$ is a endomorphism $J:TM \to TM$ for which $J^2=-\id$. To define the standard example, first identify $\R^{2n}$ with $\bbC^n$ by sending $(x_1,y_1,\ldots)$ to $(x_1+iy_1,\ldots)$, and then multiply by $i$:
    \[
        I
        \sep    
        T \bbC^n
        \longto
        T \bbC^n,
        \quad
        I(v) \defeq i \cdot v.
    \]
Call an atlas \textbf{holomorphic} if every transition function $\calT$ is holomorphic:
    \[
        \dd \calT\circ I
        = 
        I \circ \dd \calT.
    \]
This defines an almost complex structure on $M$ by pulling back $I$ along each chart. A holomorphic atlas is called \textbf{compatible} with a given $J$ if it recovers $J$ in this way.

Recall that $J$ is \textbf{involutive} if $N_J = 0$, where for $X,Y\in \gerX(M)$,
    \[
        N_J(X, Y)
        =
        [X, Y] + J( [JX, Y] + [X, JY]) - [JX, JY].
    \]

\begin{theorem}[Newlander-Nirenberg]
\label{thm_NN}
An almost complex structure $J$ admits a compatible holomorphic atlas if and only if $J$ is involutive.
\end{theorem}

We prove the following rigidity Theorem using the \nameref{General Main Theorem}.
In the next paragraph we will show how this result implies the Newlander-Nirenberg Theorem using standard arguments.

\begin{theorem}
\label{Theorem:rig_NN}
Let $I$ be the standard complex structure on $\mathbb C^n$, and let
    \[
        \qquad
        M_r
        \defeq
        \set{
            z\in \bbC^n
            \sep
            \abs{z_i}\leq 1+r,
            \,\forall\, i=1,\ldots, n
        },
        \quad
        r \in [0,1].
    \] 
There is a $p\in \N$ and for all $r < s$ there is a $C^p$-open neighborhood 
    $
        \calV^{s, r}
    $
of $I|_{M_s}$ inside the space of all involutive almost complex structures on $M_s$, and
a (non-linear) operator
    \[
        \phi^{s, r}:
        \calV^{s, r}
        \longto
        \diff_{\mathbb C^n}(M_r)
    \]
such that 
    \[
        (\phi^{s, r}(J))^*J\vert_{M_r} = I\vert_{M_r},
        \qquad\forall\;
        J \in \calV^{s, r}.
    \]
Moreover, there are constants $c_k, d_k, l \geq 0$ such that
    \[
        \norm{
            \phi^{s, r}(J) - \id
        }_{k, r}
        \leq
        c_k(s-r)^{-d_k}
        \norm{
            J-I
        }_{k+l, s}.
    \]
\end{theorem}

\begin{proof}[Proof of Theorem \ref{thm_NN}]
We prove now the Newlander-Nirenberg Theorem \ref{thm_NN} using the rigidity Theorem \ref{Theorem:rig_NN}.

Using a chart we may assume that $J$ is an involutive almost complex structure on $M_1$.
Moreover, we may assume that $J_0 = I$.
Now consider the smooth family 
    \[
        J^\epsilon_x
        \defeq 
        (m_\epsilon^* J)_x
        =
        \dd_{\epsilon x} m_\epsilon^{-1} \circ J_{\epsilon x} \circ \dd_x m_\epsilon,
        \qquad
        \epsilon \in (0,1],
        \quad
        x \in M_1,
    \]
where $m_\epsilon$ denotes multiplication by the scalar $\epsilon > 0$.
Each $J^\epsilon$ is an involutive almost complex structure on $M_1$.
Moreover, this family extends smoothly to $\epsilon = 0$ with $J^0 = I$.
Therefore, for $\epsilon >0$ small enough, we may applying Theorem \ref{Theorem:rig_NN} to obtain an embedding
    \[
        \phi: M_0 \longto M_1
    \]
satisfying $\phi^*J^\epsilon = I|_{M_0}$.
We conclude that the diffeomorphism
    \[
        m_\epsilon \circ \phi: (U, I) \xrightarrow{\simeq} (V, J)
    \]
is holomorphic, where $U = \mathrm{int}(M_0)$ and $V = \epsilon \cdot \phi(U)$.
\end{proof}

\subsection{The PDE with Symmetry}

Almost complex structures on $M$ can be interpreted as those $\bbC$-linear subbundles $B$ of $T_\bbC M = TM \otimes \bbC $ which are transverse to their complex conjugation:
    \[
        \overline B \oplus B = T_\bbC M.
    \]
Indeed, an almost complex structure $J$ gives a decomposition
    \[
        T_\bbC M = T^{1,0}_JM \oplus T^{0,1}_JM
    \]
into the $+i$ and $-i$ eigenspaces of $J$, respectively. 
Conversely, any $\bbC$-linear subbundle $B\sub T_\bbC M $ of complex rank $n$, transverse to its complex conjugate, defines an almost complex structure: 
first define
    \[
        J'
        \sep
        T_\bbC M \longrightarrow T_\bbC M,
        \qquad
        J'(v)
        \defeq
        \begin{cases}
            i\cdot v
            & v\in \overline B
            \\
            -i\cdot v
            & v\in B.
        \end{cases}
    \]
It clearly satisfies $(J')^2=-1$. 
Moreover, $\overline{J'}=J'$, so its restriction $J\defeq J'\vert_{TM}$ is an almost complex structure on $M$ for which $T^{0,1}_JM=B$.

Consider the complex-rank-$n$ Grassmannian bundle
    \[
        X
        \defeq     
        \gr_{\bbC,n} (T_\bbC M)
        \longto M,
    \]
where $n$ is half the dimension of $M$.
Complex structures can be seen as sections of $X$ which lie in the open subbundle
    \[
        X_{\transverse}
        \defeq
        \big\{
            B_x \in \gr_{\bbC,n}(T_\bbC M)
            \sep
            B_x
            \transverse
            \overline B_x
        \big\}.
    \]
Recall the involutivity condition for almost complex structures seen as $\bbC$-linear subbundles of $T_\bbC M$: 
the Lie bracket extends to a $\bbC$-bilinear bracket on $T_\bbC M$, and one calls $B$ \textbf{involutive} if it is closed under this bracket:
    \[
        [v,w] \in \Gamma_B(M),
        \qquad
        \forall\,
        v,w\in \Gamma_B(M).
    \]
For any given $B$, the Leibniz identity shows that 
    \[
        Q(B)
        \sep 
        \wedge^2 B \longto \nicefrac{T_\bbC M}{B}
        \sep
        (v_x, w_x)
        \mapsto
        [v,w]_x \mod B_x
    \]
does not depend on the choice of $v,w\in \Gamma_B(M)$. 
So involutivity can be written as $Q(B) = 0$.

Consider the tautological bundle $\calT\to X$ whose fiber at $B_x$ is the vector space $B_x$ itself. 
Consider the vector bundle
    \[
        Y \defeq \wedge^2 \calT^* 
            \otimes_\bbC
            \nicefrac{ T_\bbC M }{ \calT }
            \longto X
    \]
Hence $Q$ is a map
    \[
        Q
        \sep
        \Gamma_{X}(M)
        \longto
        \Gamma_{ 
            Y
        }(M)
        ,
    \]
and an almost complex structure $B \in \Gamma_{X_{\transverse}}(M)$ is involutive if and only if the image of $Q(B)$ lies in the zero section of 
    $
        Y\to X.
    $ 
Moreover, $Q(B)$ only depends on the first jet of $B$.
We conclude that the involutivity condition on almost complex structures can be put in the framework of a general PDE (see definition \ref{def:gen_pde}):
    \[
    \begin{tikzcd}
    X_{\transverse}
    \arrow[->]{d}
    &
    J^1X_{\transverse}
    \arrow[->]{d}
    \arrow[->]{r}{q}
    &
    Y
    \arrow[->,dl,
        start anchor = south west, 
        end anchor = north east
    ]
    \\
    M
    &
    \arrow[equal]{l}
    M
    &
    \arrow[hook, swap]{u}{\zeta}
    \arrow{l}
    X_{\transverse}.
    \end{tikzcd}
    \]
However note that the PDE extends to sections of $X$, which has the advantage of being a bundle with compact fibers.

Note that $X$ and $X_{\transverse}$ are natural bundles (see section \ref{sec:natural_bundle}).
Let $\pgrp$ be the pseudogroup on $X$ generated by lifting $\diff_M$, and note that $X_{\transverse}$ is invariant under $\pgrp$.
Denote its Lie algebra sheaf by $\calA$.
Likewise, $Y$ is a natural bundle over $M$, and $Q$ is an equivariant map.
We conclude that $(\pgrp, Q, X)$ is a general PDE with symmetry (see definition \ref{def:gen_pde_with_symmetry}).

Let $\{ M_r \}$ be a nested domain in $M$. 
Since $\pi_\textsc{X}: X\to M$ has compact fibers, the family
    \[
        X_r \defeq \pi_\textsc{X}^{-1}(M_r)
    \]
is a nested domain in $X$.
Since it has connected fibers, the natural lift gives an identification
    $
        \gerX(M_r) \simeq \calA(X_r).
    $
Moreover, for an involutive almost complex structure is given by the following standard result.

\begin{lemma}
The deformation complex of an involutive $B\in \Gamma_{X_{\transverse}}(M)$ is isomorphic to the $\bar\partial$-complex:
    \[
        \Gamma_{\overline B}(M_r)
        \xrightarrow{\;\; \bar\partial \;\;}
        \Gamma_{B^*\otimes \overline B}(M_r)
        \xrightarrow{\;\; \bar\partial \;\;}
        \Gamma_{\wedge^2B^* \otimes \overline B}(M_r),
    \]
where the differentials are given by
    \begin{align*}
        \bar\partial\alpha(v)
        &=
        \pr_{\overline B}
        \;
        [\alpha,v],
        \\
        \bar\partial \beta(v,w)
        &=
        \pr_{\overline B}
        \;
        \left(
            [\beta v,w] 
            + 
            [v,\beta w] 
            -
            \beta[v,w]  
        \right).
    \end{align*}
In the case of the standard complex structure $I$ on $\bbC^n$, with $M_r \sub \bbC^n$, this is the Dolbeault complex 
   \[
        (\Omega^{0,k}(M_r)\otimes \bbC^n, \bar\partial \otimes \id).
   \]   
\end{lemma}

\begin{proof}
The isomorphism $\Gamma_{\overline B}(M_r) \simeq \gerX(M_r)$ is given by
    \[
        \rho: \overline B \longto TM,
        \quad
        \alpha \longmapsto \alpha + \bar \alpha,
    \]
with inverse given by 
    \[
        TM \longhookrightarrow T_\bbC M = \overline B \oplus B \xrightarrow{\;\; \pr_{\overline B \;\;} } \overline B.
    \]
Note that the vertical bundle of $X_{\transverse}$ is isomorphic to $\calT^* \otimes \overline \calT|_{X_{\transverse}}$ via 
    \begin{align*}
        & \lambda : VX_{\transverse} \xrightarrow{\;\; \simeq \;\;} \calT^* \otimes \overline \calT|_{X_{\transverse}},
        \\
        & \lambda\big(\tfrac{\dd}{\dd t}\big|_{t=0} B^t\big)(v)
        =
        \pr_{\overline B}\big( \tfrac{\dd}{\dd t}\big|_{t=0} v^t \big),
    \end{align*}
where $B^t \in X_{\transverse, x}$ is a smooth path with $B^0 = B$, $v\in B$, and $v^t \in B^t$ is a smooth path such that $v^0= v$.
Therefore pulling back the vertical bundle $VX$ along $B\in \Gamma_{X_{\transverse}}(M_r)$ gives the second term in the complex:
    $
        \Gamma_{ B^* \otimes \overline B}(M_r).
    $

Next we identify the first differential of the deformation complex with the $\bar\partial$-operator.
For any $\alpha\in \Gamma_{\overline V}(M_r)$, denote by $\phi^t_{\alpha}$ on the flow on $X$ of the natural lift of $\rho\alpha$. 
Note that the action of $\phi_\alpha^t$ on $B$ is given by
    \[
        B^t \defeq B\cdot \phi_\alpha^t = (\phi_{\rho\alpha}^t)^*B.
    \]
Hence for any $v \in \Gamma_B(M_r)$, we have $(\phi_{\rho\alpha}^t)^*v \in B^t$, so
    \[
        \lambda\big(\tfrac{\dd}{\dd t}\big|_{t=0} B^t\big)(v) 
        =
        \pr_{\overline B} \; [\rho \alpha, v]
        =
        \pr_{\overline B}
        \;
        [\alpha + \bar \alpha, v ]
        =
        \pr_{\overline B}
        \;
        [\alpha,v].
    \]
Here we have used that $B$ is involutive so that $[\bar\alpha, v] \in \Gamma_B(M_r)$.

For the second differential we use that $T_\bbC M = B \oplus \overline B$ to define a tubular neighborhood $U(B)$ of $B$ in $X$.
Namely, define
    \[
        U(B) \defeq
        \big\{
            C_x \in X_x
            \sep
            \pr_{B_x}(C_x) = B_x
        \big\}.
    \]
Then the map
    \[
        B^* \otimes \overline B
        \longto 
        U(B),
        \quad
        \epsilon_x
        \longmapsto
        (\id \oplus \epsilon_x)B_x 
    \]
is an isomorphism. 
The intersection  $U(B)_{\transverse} \defeq U(B)\medcap X_{\transverse}$ corresponds to the $\epsilon_x : B_x \to \overline B_x$ such that the map
    \[
        B_x \oplus \overline B_x \longto B_x \oplus \overline B_x,
        \quad
        a \oplus b \longmapsto (a + \bar\epsilon_x b) \oplus (b + \epsilon_x a)
    \]
is invertible.
The projection onto $B$ restricts to an isomorphism
    \[
        \pr_B: \calT|_{U(B)_{\transverse}} \xrightarrow{\;\; \simeq \;\;} B \times U(B)_{\transverse}.
    \]
And similarly we have an isomorphism
    \[
        \pr_{\overline B} : \overline\calT|_{U(B)_{\transverse}} \xrightarrow{\;\; \simeq \;\;} \overline B \times U(B)_{\transverse}.
    \]
Consider now a smooth family $B^t \in \Gamma_{U(B)_{\transverse}}(M_r)$ with $B^0 = B$.
Let $\epsilon^t$ be the corresponding smooth family in $\Gamma_{B^*\otimes \overline B}(M_r)$. 
Under the above identifications, $Q(B^t)$ corresponds to
    $
        \widetilde Q( \epsilon^t ) \in \Gamma_{\wedge^2 B^* \otimes \overline B}(M_r),
    $
given by
    \[
        \widetilde Q( \epsilon^t )(v, w) 
        =
        \pr_{\overline B} \circ \pr_{\overline B^t} \; [v \oplus \epsilon^t v, w \oplus \epsilon^t w].
    \]
To compute $\tfrac{\dd}{\dd t}\big|_{t=0} \widetilde Q(\epsilon^t)$, we decompose
    \[
        [v \oplus \epsilon^t v, w \oplus \epsilon^t w]
        =
        (a^t + \bar\epsilon^t b^t) \oplus (b^t + \epsilon^t a^t),
    \]  
where $a^t \oplus \epsilon^t a^t \in B^t$ and $b^t \oplus \epsilon^t b^t \in \overline B^t$.
Note that
    \[
        \tfrac{\dd}{\dd t}\big|_{t=0} \widetilde Q(\epsilon^t)(v,w) = \tfrac{\dd}{\dd t }\big|_{t=0} b^t.
    \]
Taking the derivative at $t=0$ of the decomposition, and projecting onto $\overline B$, we obtain
    \[
        \pr_{\overline B}\big(
        [ \dot\epsilon^0 v, w]
        +
        [ v, \dot\epsilon^0 w ]
        \big)
        =
        \tfrac{\dd}{\dd t}\big|_{t=0} b^t
        +
        \dot\epsilon^0 a^0.
    \]
From the decomposition we read $a^0 = \pr_{B}\; [v,w] = [v,w]$, hence
    \[
        \tfrac{\dd}{\dd t}\big|_{t=0} \widetilde Q(\epsilon^t)(v,w)
        =
        \pr_{\overline B}\big(
        [ \dot\epsilon^0 v, w]
        +
        [ v, \dot\epsilon^0 w ]
        -
        \dot\epsilon^0 [v,w]
        \big).
        \qedhere
    \]
\end{proof}

\subsection{The Homotopy Operators}

In order to apply the \nameref{General Main Theorem} to the deformation of involutive almost complex structures we need to construct tame homotopy operators for the chain complex described in the previous paragraph.
The deformation complex for the standard complex structure on $\bbC^n$ consists of $n$ copies of the Dolbeault complex $\Omega^{0,q}(\bbC^n)$ in degree one, so the Dolbeault Lemma tells us it is exact. We recall how to explicitly produce homotopy operators for the Dolbeault complex in degree one and show that they satisfy tame estimates. The groundwork takes place in dimension $n=1$, and the general results follow from algebraic manipulations afterwards. 

Let $D_r \sub \bbC$ be the closed disc of radius $r$ at the origin. The Dolbeault complex on $D_r$ is
    \[
        C^\infty(D_r,\bbC)
        \overset{ \bar\partial }{ \longto }
        C^\infty(D_r, \bbC)
        \longto
        0.
    \]

\begin{proposition}
For any $0< r < s\leq 1$, define the Cauchy-Riemann operator
    \[
        T^{s,r}\!
        \sep
        C^\infty(D_s, \bbC)
        \longto
        C^\infty(D_r, \bbC),
        \quad
        T^{s,r}\! f(z)
        \defeq
        \iint\nolimits_{D_s}
        f(\zeta)
        \frac{ \dd\zeta \dd\bar \zeta }
            { \zeta - z }.            
    \]
It satisfies the properties of a tame homotopy operator:
\begin{itemize}
\item 
    $
        \bar\partial T^{s,r}\! f 
        = 
        f\vert_{D_r}.
    $

\item There are constants $c_k>0$ such that
    $
        \norm{ T^{s,r}\! f }_{k,r}
        \leq
        c_k (s-r)^{-k-2}  
        \norm{ f }_{k,s}.
    $

\end{itemize}
\end{proposition}

\begin{proof}
The first statement is a standard result about the Cauchy-Riemann operator \cite{kodaira}. Nonetheless, most of that argument is reproduced in our estimates. We start by splitting functions on $D_s$ into two cases: functions $f_1$ supported in the interior of $D_s$, and functions $f_2$ which vanish on $D_{r+\epsilon(s-r)}$ for an $\epsilon \in (0,\nicefrac12)$. It is clear that any function $f$ on $D_s$ can be written as $f=f_1+f_2$. Moreover, in the final paragraphs of the proof we show this can be done in a consistent manner so that the assignment $f\mapsto (f_1,f_2)$ satisfies tame estimates. Given the linearity of $T^{s,r}\!$, it is therefore sufficient to consider each case on its own.

Regarding the first case, the change of coordinates $\rho e^{i\theta } = \zeta-z$ gives
    \[
        T^{s,r}\! f_1(z)
        =
        \int_{-\pi}^\pi
        \int_0^\infty
        f_1(z + \rho e^{i\theta} )
        e^{-i \theta }
        \,\dd\rho \,\dd\theta.
    \]
So for any multi-indices $p,q\in \N$,
    \begin{align*}
        &\abs{ 
            \partial^p \bar\partial^q
            \int_{-\pi}^\pi
            \int_0^\infty
            f_1( z+\rho e^{i\theta} )
            e^{-i\theta}
            \,\dd\rho \,\dd\theta
        }
        \\
        &\qquad\qquad =
        \abs{ 
            \int_{-\pi}^\pi
            \int_0^\infty
            \partial^p \bar\partial^q 
            f_1( z+\rho e^{i\theta} )
            e^{-i\theta}
            \,\dd\rho \,\dd\theta
        }
        \\
        &\qquad\qquad \leq
        \volume( \mathrm{supp}(f) )
        \cdot
        \max\nolimits_{z\in \bbC}
        \abs{
            \partial^p\bar\partial^q
            f_1(z)
        }.
    \end{align*}
In particular it follows that 
    $
        \norm{T^{s,r}\! f_1}_{k,r}
        \leq
        \volume( D_s )
        \cdot
        \norm{ f_1 }_{k,s}.
    $

Regarding the second case, the integral
    \[
        T^{s,r}\! f_2(z)
        =
        \iint\nolimits_{D_s}
        f_2(\zeta)  
        \frac{ \dd\zeta \dd\bar\zeta }{ \zeta - z }
        =
        \iint\nolimits_{    
            D_s \backslash D_{r+\epsilon(s-r)} 
        }
        \frac{f_2(\zeta)}{\zeta-z}
        \,\dd\zeta \,\dd\bar\zeta
    \]
is a complex analytic function which converges uniformly for $z\in D_r$, since $\abs\zeta\geq \abs z + \epsilon(s-r)$. Hence $\bar\partial T^{s,r}\! f_2 = 0$ on $D_r$, and the integrals and $\partial$ commute, so that for all $k\in \N$
    \[
        \partial^k
        T^{s,r}\! f_2(z)
        =
        (-1)^kk!
        \iint\nolimits_{ 
            D_s \backslash D_{r+\epsilon(s-r)} 
        }        
        \frac{ f_2(\zeta) }{ (\zeta - z)^{k+1} }
        \,\dd\zeta \,\dd\bar\zeta
    \]
Note that 
    $
        \abs{ \zeta - z }^{k+1} 
        \geq 
        \epsilon^{k+1}(s-r)^{k+1},
    $
so we observe that
    \[
        \abs{ \partial^k T^{s,r}\! f_2(z) }
        \leq
        \volume( D_s\backslash D_r )
        \epsilon^{-k-1}
        (s-r)^{-k-1}
        \sup\nolimits_{z\in D_s} \abs{ f_2(z) }.
    \] 
In particular it follows that 
    \[
        \norm{T^{s,r}\!f_2}_{k,r}
        \leq
        \frac{
            \volume( D_s\backslash D_r )
        }{
            \epsilon^{k+1}
            (s-r)^{k+1}
        }
        \norm{ f_2 }_{0,s},
    \]
which is the very extreme case of a tame estimate where the $C^k$-norm of $T^{s,r}f_2$ ever only depends on the $C^0$-norm of $f_2$. 
Combining both estimates and the argument below completes the proof.

For a set $S\sub D_s$, define the set of smooth functions with support in $S$ as
    \[
        C^\infty(D_s;S)
        \defeq
        \set{
            f\in C^\infty(D_s)
            \sep
            \mathrm{supp}(f)\sub \mathrm{int}(S)
        }.
    \]
We claim the following: For any $\epsilon \in (0,\nicefrac 12)$ there are two families of linear operators
    \begin{align*}
        h^{s,r}_1\!
        &\sep
        C^\infty(D_s)
        \longrightarrow
        C^\infty(D_s;D_{s-\epsilon(s-r)}),
        \\
        h^{s,r}_2\!
        &\sep
        C^\infty(D_s)
        \longrightarrow
        C^\infty(D_s;D_s\backslash D_{r+\epsilon(s-r)}),
    \end{align*}
satisfying the relation
    $
        h^{s,r}_1\!+ h^{s,r}_2\! = \id,
    $
and for which there are positive constants
    $
        c_k>0
    $
so that, for all $k\geq 0$ and all $\nicefrac{1}{2}\leq r<s\leq 1$:
    \[
        \norm{h^{s,r}_i\!f}_{k,s}
        \leq
        c_k
        { (1-2\epsilon)^{-1}(s-r)^{-1} }
        \norm{f}_{k,s}.
    \]

Consider a bump function 
    $
        \chi:\R\to\R_{\geq 0}
    $ 
which is $\chi(t)=1$ for $t\leq 0$, $\chi(t)=0$ for $t\geq 1$, and decreasing on the interval $(0,1)$. Then define a bump function $\chi^{s,r}\!$ on $D_s$ by
    \[
        \chi^{s,r}\!(x)
        \defeq
        \chi\left(
            \frac{
                \abs{x}-(r+\epsilon(s-r))
            }{
                (1-2\epsilon)(s-r)
            }
        \right),
    \]
and the operators $h^{s,r}_i\!$ by
    \[
        h^{s,r}_1\!f
        \defeq        
        \chi^{s,r}\!
        \cdot
        f,
        \qquad
        h^{s,r}_2\!f
        \defeq
        (1-\chi^{s,r}\!)\cdot f.
    \]
It is clear that $h^{s,r}_1\!+h^{s,r}_2\!=\id$, and the challenge lies with proving the estimates. First remark that, using interpolation estimates,
    \[
        \norm{h^{s,r}_1\!f}_{k,s}
        \lesssim
            \norm{\chi^{s,r}\!}_{k,s}
            \norm{ f }_{0,s}
            +
            \norm{\chi^{s,r}\!}_{0,s}
            \norm{ f }_{k,s}
        \lesssim
        \norm{ \chi^{s,r}\!}_{k,s}
        \norm{ f }_{k,s},
    \]
and likewise
    $
        \norm{h^{s,r}_2\!f}_{k,s}
        \lesssim
        \norm{ 1-\chi^{s,r}\! }_{k,s}
        \norm{ f }_{k,s}.
    $ 
For $\nicefrac{1}{2}\leq r < s \leq 1$, define
    \[
        \mu^{s,r}
        \sep
        D_s\backslash \mathrm{int}(D_{\nicefrac12})
        \to
        \R,
        \qquad
        \mu^{s,r}(x)
        \defeq
        \frac{
            \abs{x}-(r+\epsilon(s-r))
        }{
            (1-2\epsilon)(s-r)
        },
    \]
so that $\chi^{s,r}\!$ is extension by zero of $\chi\circ\mu^{s,r}$. We purposefully introduce the lower bound $r\geq\nicefrac 12$ to stay away from the singularity of $x\mapsto \abs x$.
It follows that, for 
    $
        A_s\defeq D_s\backslash \mathrm{int}(D_{\nicefrac12}),
    $
    \[
        \norm{ \chi^{s,r}\! }_{k,s}
        =
        \norm{ \chi\circ\mu^{s,r} }_{k,A_s},
    \]
so we are left with showing that
    $
        \norm{ \chi\circ\mu^{s,r}}_{k,A_s}
        \lesssim
        (1-2\epsilon)^{-1}(s-r)^{-1}.
    $
Now apply the estimates for the composition of smooth maps to get
    \begin{align*}
        \norm{\chi^{s,r}\!}_{k,s}
        & \lesssim
        \norm{ \chi }_{0,\R}
        +
        \norm{ \chi }_{k+1,\R}
        \norm{ \mu^{s,r} }_{0,A_s}
        +   
        \norm{ \chi }_{1,\R}
        \norm{ \mu^{s,r} }_{k,A_s}
        \\
        & \lesssim
        1+
        \norm{ \chi }_{k+1,\R}
        \norm{ \mu^{s,r} }_{k,A_s}
        \\
        & \lesssim
        1 +
        \norm{ \mu^{s,r} }_{k,A_s},
    \end{align*}
where in the last estimate we treat $\norm{\chi}_{k+1,\R}$ as a constant, since $\chi$ is independent of $r$ and $s$. We directly observe that
    \[
        \norm{ \mu^{s,r} }_{k,A_s}^2
        \leq
        (1-2\epsilon)^{-1}(s-r)^{-1}
        \sup\nolimits_{x\in A_s}
        \sum\nolimits_{\abs\alpha\leq k}
        \abs{
            D^\alpha\abs{x}
        }^2.
    \]
Partial derivatives of the absolute value function are always of the form
    \[
        D^\alpha\abs{x}
        =
        \frac{ p_\alpha(x,\abs{x}) }
            { q_\alpha(x,\abs{x}) },
    \]
where $p_\alpha$ and $q_\alpha$ are polynomial expressions in the $x_i$ and in $\abs{x}$. Since
    $
        r \leq \abs{x_i},\abs{x} \leq s,
    $
the absolute value of $D^\alpha\abs{x}$ can be bounded by
    \[
        \abs{ D^\alpha \abs{x} }
        \leq
        \frac{ \tilde{p_\alpha}(s) }
        { \tilde{q_\alpha} (r) }
        \leq
        \frac{ \tilde{p_\alpha}(1) }
        { \tilde{q_\alpha}(\nicefrac12) },  
    \]
where $\tilde{p_\alpha}$ and $\tilde{q_\alpha}$ are polynomials in one variable with nonnegative coefficients. Hence it follows that
    $
        \norm{ \mu^{s,r} }_{k,A_s} 
        \lesssim
        (1-2\epsilon)^{-1}(s-r)^{-1},
    $
and this completes the estimates:
    \begin{align*}
        \norm{ h^{s,r}_1\!f }_{k,s}
        & \lesssim
        (1+ \norm{\mu^{s,r}}_{k,A_s})
        \norm{f}_{k,s}
        \\
        &\lesssim
        \left(
            1 + (1-2\epsilon)^{-1}(s-r)^{-1}
        \right)
        \norm{f}_{k,s}
        \\
        & \lesssim
        (1-2\epsilon)^{-1}(s-r)^{-1} \norm{f}_{k,s}.
        \qedhere
    \end{align*}
\end{proof}

Let $D^n_r \sub \bbC^{n}$ be the closed polydisk of radius $r$ centered at the origin. We recall from \cite{newlander_nirenberg} how to construct homotopy operators in degree one for the Dolbeault complex $(\Omega^{0,q}(D_r^n),\bar\partial)$.

\begin{proposition}
There is a family of operators
    \[
        h^{s,r}_q
        \sep
        \Omega^{0,q}(D^n_s)
        \longto
        \Omega^{0,q-1}(D^n_r),
        \quad
        q=1,2,
    \]
such that
\begin{itemize}

\item
    $
        (\bar\partial h^{s,r}_1\! + h^{s,r}_2\! \bar\partial )
        \alpha
        =
        \alpha\vert_{D^n_r}.
    $

\item
There are constants $c_k>0$ such that
    \[
        \norm{ h^{s,r}_q\!\alpha }_{k,r}
        \leq 
        c_k
         (s-r)^{ -n(k+2) }
        \norm{ \alpha }_{ k+n-1, s }.
    \]

\end{itemize}
\end{proposition}

\begin{proof}

One can recall the general formula for the homotopy operators directly from \cite{newlander_nirenberg}. First we recall the formulas for $\bar\partial$ in the coordinates $z=(z^1,\ldots, z^n)$: let $\bar\partial_j$ denote the Dolbeault operator in the $z^j$-coordinate, then we have
    \[
        \bar\partial(\alpha) 
        = 
        \bar\partial_k\alpha \,\dd\bar z^k,
        \qquad
        \bar\partial(\beta_k \,\dd\bar z^k)
        =
        ( 
            \bar\partial_k \beta_l
            -
            \bar\partial_l \beta_k
        ) 
        \, \dd\bar z^k \wedge \dd\bar z^l.
    \]
Then for a subset 
    $
        J\sub \set{1,\ldots, n},
    $
let $D^J_{s,r}$ be the closed polydisk which is of radius $r$ in any coordinate $z^j$ with $j\in J$, and of radius $s$ in any other coordinate:
    \[
        D^J_{s,r}
        \defeq
        \prod\nolimits_{j=1}^n
        D_{s-\delta^J(j)(s-r)}.
    \]
For a smooth function $f:D^J_{s,r}\to \bbC$, and a $j\not\in J$, denote the Cauchy-Riemann operator by
    \[
        T_j^{s,r} f(z) 
        \defeq
        \frac{1}{2\pi i} 
        \iint_{D_s} 
        f(
            z^1,
            \ldots,
            z^{j-1},
            \zeta,
            z^{j+1},
            \ldots, 
            z^n
        )
        \frac{ 
            \dd\zeta \,\dd\bar\zeta 
        }{
            \zeta- z^j
        }
    \]
for $z \in 
        D^{J\cup \set{j}}_{s,r}$.
From a parameter dependent version of the previous Proposition it follows that:

\begin{itemize}

\item
    $
        \partial_jT^{s,r}_j\! f
        =   
        f\vert_{D^{J\cup\set{j}}_{s,r}}
    $
for all $0<r<s\leq 1$, $j\not\in J$, and 
    $
        f
        \in 
        C^\infty(D^J_{s,r},\bbC).
    $

\item
    $
        \bar\partial_i \bar\partial_j
        =
        \bar\partial_j \bar\partial_i,
    $
    $
        \bar\partial_i T^{s,r}_j\!
        =
        T^{s,r}_j\! \bar\partial_i,
    $
and
    $
        T^{s,r}\!_i T^{s,r}_j\!
        =
        T^{s,r}_j\! T^{s,r}\!_i
    $
whenever $i\neq j$.

\item
There are constants $c_k>0$ such that
    \[
        \norm{ T^{s,r}_j\!\! f }_{k, D^{J\cup\set{j}}_{s,r}}
        \leq
        c_k  (s-r)^{-k-2}
        \norm{ f }_{k, D^J_{s,r} }.
    \]

\end{itemize}
The operators $\bar\partial_j$ and $T^{s,r}_j\!$ will be the building blocks of the homotopy operators $h^{s,r}_q$. 

The first homotopy operator sends
    $
        \beta 
        = 
        \beta_k \dd\bar z^k
        \in
        \Omega^{0,1}(D^n_s)
    $ 
to 
    \[
        h^{s,r}_1\!\beta
        \defeq
        \left.
        \sum\nolimits_k
        \sum\nolimits_{ J \not\ni\, k}
        \frac{ 
            (-1)^J 
        }{
            J+1
        }
        \left(
            \prod\nolimits_{j\in J}
            T^{s,r}_j\! \bar\partial_j
            T^{s,r}_k\! \beta_k
        \right)
        \right\vert_{D^n_r},
    \]
where the sum runs over all subsets $J\sub\set{1,\ldots, n}$ that do not contain $k$. 

The second homotopy operator sends
    $
        \gamma 
        =   
        \gamma_{kl} 
        \, \dd\bar z^k \wedge \dd\bar z^l
        \in
        \Omega^{0,2}(D^n_s)    
    $ 
to
    \[
        h^{s,r}_2\!\gamma
        \defeq
        \left.
        \sum\nolimits_{
            \substack{ k,l \\ k\neq l }
        }
        \sum\nolimits_{ J\not\ni\, k,l }
        \frac{
            (-1)^{ J +1}
        }{
            ( J + 1 ) ( J + 2 )
        }
        \left(
            \prod\nolimits_{j\in J}
            T^{s,r}_j\! \bar\partial_j
            T^{s,r}_l\! \gamma_{kl}
        \right)
        \right\vert_{D^n_r}
        \dd\bar z^k,
    \]
where the sum runs over all subsets $J\sub \set{1,\ldots, n}$ that do not contain $k$ or $l$.
It should be clear to the reader what the formula for $h^{s,r}_q$ should be for $q \geq 3$, but we are only interested in $q=1,2$. 

The estimates for $h_q^{s,r}$ follow directly from the estimates for $T^{s,r}_j\!$ and $\bar\partial_j$. Indeed, the expressions 
    $
        \prod_{j\in J}
        T_j^{s,r}
        \bar\partial_j
        T^{s,r}_k\!
    $
contain at most $n$ operators of the type $T^{s,r}_j\!$, which explains the factor $(s-r)^{-n(k+2)}$, and at most $(n-1)$ operators of the type $\bar\partial_j$, which explains the loss of derivatives. One can obtain slightly better estimates for $h^{s,r}_2\!$, but this isn't important to us.

To compute 
    $
        \bar\partial (h^{s,r}_1\!\beta)
        =
        \bar\partial_p 
        (h^{s,r}_1\!\beta)
        \,\dd\bar z^p,
    $
distinguish in the sum between the cases $k=p$, where one uses
    $
        \bar\partial_pT^{s,r}_p\!
        \beta_p
        =
        \beta_p,
    $
the cases $J \ni p$, where one uses
    $
        \bar\partial_p
        T^{s,r}_p\!
        \bar\partial_p
        =
        \bar\partial_p,
    $
and otherwise. Then on $D^n_r$:
    \begin{align*}
        \bar\partial_p h^{s,r}_1\!\beta
        & =
        \sum_{J \not\ni\, p}
        \frac{ 
            (-1)^J
        }{
            J+1
        }   
        \prod\nolimits_{j\in J}
        T^{s,r}_j\!
        \bar\partial_j
        \beta_p
        \\
        &\qquad +
        \sum_{k\neq p}
        \sum_{J \not\ni\, p}
        \bar\partial_p
        \left(
            \frac{
                (-1)^J
            }{
                J+1
            }
            +
            \frac{
                (-1)^{J+1}
            }{
                J+2
            }
            T^{s,r}_p\!\bar\partial_p
        \right)
        \prod_{j\in J}
        T^{s,r}_j\! 
        \bar\partial_j
        T^{s,r}_k\!
        \beta_k
        \\
        & =
        \sum_{J \not\ni\, p}
        \frac{ 
            (-1)^J
        }{
            J+1
        }   
        \prod\nolimits_{j\in J}
        T^{s,r}_j\!
        \bar\partial_j
        \beta_p
        \\
        &\qquad +
        \sum_{k\neq p}
        \sum_{J \not\ni\, p}
        \frac{
            (-1)^J
        }{
            (J+1)(J+2)
        }
        \prod_{j\in J}
        T^{s,r}_j\! 
        \bar\partial_j
        T^{s,r}_k\!
        \bar\partial_p
        \beta_k.
    \end{align*}
To compute
    $
        h^{s,r}_2\!
        \bar\partial \beta 
        =
        (h^{s,r}_2\!\bar\partial \beta)_p
        \, \dd\bar z^p,  
    $
recall that 
    $
        \bar\partial\beta   
        =
        ( 
        \bar\partial_p \beta_k
        -
        \bar\partial_k \beta_p
        )
        \, \dd\bar z^p \wedge \dd\bar z^k.
    $
Then on $D^n_r$:
    \begin{align*}
        (h^{s,r}_2\!\bar\partial \beta)_p
        & =
        \sum_{k\neq p}
        \sum_{J \not\ni k, p}   
        \frac{
            (-1)^{ J+1 }
        }{
            (J+1)(J+2)
        }
        \prod_{j\in J}
        T^{s,r}_j\!
        \bar\partial_j
        T^{s,r}_k\!
        (\bar\partial_p\beta_k
        -
        \bar\partial_k\beta_p)
        \\
        & =
        -
        \sum\nolimits_{
            \substack{
                J \not\ni\, p
                \\
                J\geq 1
            }
        }
        \frac{ (-1)^J }{ J+1 }
        \prod\nolimits_{j\in J}
        T^{s,r}_j\!
        \bar\partial_j
        \beta_p
        \\
        &\qquad -
        \sum_{k\neq p}
        \sum_{J \not\ni k, p}   
        \frac{
            (-1)^J
        }{
            (J+1)(J+2)
        }
        \prod_{j\in J}
        T^{s,r}_j\!
        \bar\partial_j
        T^{s,r}_k\!
        \bar\partial_p\beta_k.
    \end{align*}
From this it becomes clear that
    $
        (\bar\partial h^{s,r}_1\!
        +
        h^{s,r}_2\!\bar\partial) \beta
        =
        \beta\vert_{D^n_r}.
    $
\end{proof}

\chapter{Proof of the Main Theorem}
\label{chapter:Proof of the Main Theorem}
\etocsettocdepth{2}
\localtableofcontents

\noindent
{\color{white}This line is intentionally left blank.}
\newline
In this section we present the proof of the \nameref{Main Theorem}. Throughout it we fix the objects from the statement of the \nameref{Main Theorem}. 

The discussion will be organized as follows:
\begin{itemize} 

\item
Section \ref{Step 1}: we reduce to the case of the trivial bundle $(Y,Z)=( N\times \R^n, N\times \{0\})$.

\item 
Section \ref{Step 2}: we localize the problem to a tubular neighborhood of $b$.

\item
Section \ref{Step 3}: using a quadratic inequality (whose proof is delayed until section \ref{sect:lemmas}), we reduce the problem to the setting of the 
\nameref{Main Theorem for Differential Relations}

\item 
Section \ref{Step 4}: we state inequalities for flows of vector fields and their action, which will be proven in section \ref{sect:lemmas}

\item 
Section \ref{Step 5}: we state a convergence result for infinite compositions of diffeomorphism, which will be proven in section \ref{sect:lemmas}

\item
Section \ref{Step 6}: we recall the interpolation inequalities and the properties of smoothing operators adapted to our problem

\item 
Section \ref{Step 7}: we give the proof of the 
\nameref{Main Theorem for Differential Relations} using the Nash-Moser fast convergence method.

\end{itemize}

\newpage
\section{Preparations}
\etocsettocdepth{2}
\localtableofcontents

\subsection{Simplifying the Equation}
\label{Step 1}

We show that, given a general PDE, the bundle $\pi_\textsc{Y}:Y \!\to\! N$ and the fibered submanifold $Z\sub Y$ can be replaced by $N\times\R^n$ and $N\times \{ 0\}$ respectively, for some $n\geq 0$. This is based on the following observation.

\begin{lemma}
Let $Z\sub Y$ be a closed embedded submanifold. There exists an $n\in \N$ and a smooth map $\alpha:Y\to \R^n$ such that $Z=\alpha^{-1}(0)$ and
    \[
        \overline{\dd_z\alpha}: T^\mathrm{n}_z(Y,Z) \longto \R^n
    \]
is injective for all $z\in Z$, where $T^\mathrm{n}_z(Y,Z)\defeq T_zY/T_zZ$ is the normal bundle, and $\overline{\dd\alpha}_z$ denotes the map induced by the differential of $\alpha$ at $z$.
\end{lemma}

\begin{proof}
Consider the following objects: 
\begin{itemize}

    \item a vertical tubular neighborhood $\exp: T^\mathrm{n}(Y,Z)\longhookrightarrow Y$ of $Z$, with image denoted by $U$;

    \item a vector bundle embedding 
    $\theta: T^\mathrm{n}(Y,Z)\longhookrightarrow \R^{n-1}\times Z$;

    \item a smooth function $\chi:Y \!\to\! [0,1]$ supported on $U$, and which is equal to one on a neighborhood of $Z$. 
\end{itemize}

These objects exist because $Z$ is closed and embedded. Let $\beta$ be the composition of the maps:
    \[
        \beta:
        U 
        \xrightarrow{\;\exp^{-1}\;}
        T^\mathrm{n}(Y,Z) 
        \xhookrightarrow{ \;\;\;\theta\;\;\; }
        \R^{n-1} \times Z
        \xrightarrow{ \;\;\pr_1\;\; }
        \R^{n-1}.
    \]
Then the map $\alpha \defeq (\chi\beta,(1-\chi)):Y\longto \R^n$ satisfies both requirements.
\end{proof}

In the next lemma we assume the setting of the \nameref{General Main Theorem}.

\begin{lemma}
Let $\alpha$ be as in the previous lemma. Consider the partial differential operator
    \[
        \widetilde{Q}:
        \Gamma_X(M)
        \longto 
        C^\infty(N, \R^n),
        \quad
        \widetilde{Q}(e)
        \defeq 
        \alpha\circ Q(e),
    \]
and $\widetilde Z \defeq \{0\} \!\times\! N$. 
Then the partial differential equation 
    $
        (\widetilde{Q}, \widetilde Z )
    $ 
satisfies
    \[
         \sol_{Q,Z}(U)
         =
         \sol_{\widetilde{Q}, \widetilde Z }(U),
         \quad\forall\;
         U\sub M,
    \]
and it meets the hypothesis of the \nameref{General Main Theorem} if $(Q,Z)$ does.
\end{lemma}
\begin{proof}
Since $\alpha^{-1}(0)=Z$, the two PDEs have indeed the same space of solutions. 
It suffices to construct  operators $\widetilde{h}_2^{s,r}$ such that, together with $h_1^{s,r}$, they satisfy the hypotheses of the \nameref{General Main Theorem}. 
Since 
    $
        \overline{\dd_z\alpha}:
         T^\mathrm{n}_z(Y,Z) \longto \R^n
    $
is injective for $z\in Z$, there exists a vector bundle map
    \[
        \sigma:
        \R^n \!\times\! Z
        \longto 
        T^\mathrm{n}(Y,Z)
        \quad 
        \textrm{s.t.}
        \quad 
        \sigma\circ 
        \overline{\dd\alpha}
        =\id_{ T^\mathrm{n}(Y,Z)}.
    \]
Pulling back these maps via $\zeta\defeq Q(b) \in \Gamma_Z(N)$, we obtain the vector bundle maps
    \[
        \overline{\dd_\zeta\alpha}:
        T^\mathrm{n}_\zeta(Y,Z)
        \longto 
        \R^n \!\times\! N
        \quad 
        \textrm{and}
        \quad
        \sigma_\zeta:
        \R^n \!\times\! N 
        \longto  
        T^\mathrm{n}_\zeta(Y,Z),
    \]
which also satisfy
    $
        \sigma_\zeta
        \circ 
        \overline{\dd_\zeta\alpha}
        =
        \id_{T^\mathrm{n}_\zeta(Y,Z)}.
    $
The linearization of $\widetilde{Q}$ at $b$ is given by
    \[
        \widetilde \delta_b
        =
        \overline{\dd_\zeta\alpha}
        \circ 
        \delta_b:
        \Gamma_{T^\mathrm{n}_b X}(M)
        \longto 
        C^\infty(N, \R^ns),
    \]
hence it is easy to see that the map
    \[
        \widetilde{h}_2^{s,r}
        \defeq 
        h_2^{s,r}
        \circ 
        \sigma_\zeta:
        C^\infty(N_s, \R^n) 
        \longto 
        \Gamma_{T^\mathrm{v}_bX}(M_r)
    \]
satisfies the tame inequalities, and since
    $
        \widetilde{h}_2^{s,r}
        \circ 
        \widetilde\delta_b
        =
        h_2^{s,r}\circ \delta_b,
    $
also the homotopy relation.
\end{proof}

\begin{notation}
From now on we will assume that $Y=\R^n \!\times\! N$ and $Z=\{0\} \!\times\! N$, and we will still denote the partial differential operator by $Q$. 
Hence the PDE simply becomes
    \[
        Q(c)=0.
    \]
\end{notation}

\subsection{Simplifying the Bundle \& Norms}
\label{Step 2}

From now on we denote the vertical bundle of $X$ along $b$ by
    \[
        \pi_E:
        E\defeq T^\mathrm{v}_bX \longto M.
    \]
Fix a vertical tubular neighborhood of $b$ in $X$, i.e. a open embedding
    \[
        \xi:E\longhookrightarrow X,
    \]
that preserves fibers and which sends the zero section of $E$ to $b$. 
Moreover, we simplify our notation by identifying $E$ with its image in $X$. 
This means that $b=0$, and that $\Gamma_E(M)$ is a strong $C^0$-neighborhood of $b$ in $\Gamma_X(M)$.
Since the conclusion of the \nameref{General Main Theorem} only holds for sections of $X$ close to $b$, we may as well only work on $E$.

\begin{notation}
From now on we will assume that $E= X$ is a vector bundle over $M$ and $b=0$.
In particular,
    $
        Q(0)=0,
    $
and the deformation complex becomes
    \[
        \calA(X_r)
        \xrightarrow{ \;\;\delta_0\;\; }
        \Gamma_E(M_r)
        \xrightarrow{ \;\;\delta_0\;\; }
        C^\infty(N_r, \R^n).
    \]
\end{notation}

Fix vector bundle metrics $m_k$, for $k\geq 0$, on the fibers of the jet bundle $J^k E \!\to\! M$ such that the projections 
    \[
        \pi_k^{k+1}: J^{k+1} E\longto J^k E
    \]
are norm decreasing. 
Let $K\sub M$ be a compact set such that $K=\overline{\mathrm{int}(K)}$. This condition ensures that every section $e\in \Gamma_E(K)$ has a well-defined $k$-jet $j^k(e)\in \Gamma_{J^k E}(K)$. 
Define
    \[
        \norm{e}_{k,K}
        \defeq
        \sup\nolimits_{x\in K}
        \abs{
            j^k(e)_x
        }_{m_k}.
    \]
Moreover, given the nested domain $\{ M_r \}$ and $e\in \Gamma_E(M_r)$, simplify the notation to
    \[
        \norm{e}_{k, r}
        \defeq
        \norm{e}_{k, M_r}.
    \]
To define norms on local diffeomorphisms, fix a vertical tubular neighborhood 
    $
        \chi: TX \!\hookrightarrow\! X \!\times\! X
    $
around the graph of the identity. 
Let $\calU\sub \diff_X$ be the strong $C^0$-open subset defined by
    \[
        \calU(U) 
        \defeq
        \big\{
            \phi \in \diff_X(U)
            \sep
            j^0(\phi)(U) \sub \chi(TX)
        \big\}.
    \]
Then for any $\phi \in \calU(X_r)$ we define
    \[
        \norm{\phi}_{k, r}
        \defeq
        \norm{ \chi^{-1}\circ (\id, \phi)}_{k, r}.
    \]
This notation differs slightly from section \ref{sec:local_rigidity} and the statement of the Main Theorems, where we write $\norm{\phi-\id}_{k,r}$ to mean the same thing, for the sake of readability of the proofs.

\begin{notation}
Any constants $c$, $c_k$, and $d_k$ are meant to be \emph{large enough}:
inequalities
    \[
        c_k(s-r)^{-d_k} \norm{ v }_{k,r}
        \leq
        c_k(s-r)^{-d_k} \norm{ w }_{k+l,s}
    \]
should be understood as follows: there are constants $c_k'\geq c_k$ and $d_k'\geq d_k$ such that
    \[
        c_k(s-r)^{-d_k} \norm{ v }_{k,r}
        \leq
        c_k'(s-r)^{-d_k'} \norm{ w }_{k+l,s}.
    \]
Similarly, the constant $\theta > 0$ is meant to be \emph{small enough} such that \nameref{Lemma:A}, \nameref{Lemma:B}, and \nameref{Lemma:C} can be applied where needed.
\end{notation}

\subsection{Lemma A: Quadratic Inequality}
\label{Step 3}

In this section we show that the \nameref{General Main Theorem} follows from the \nameref{Main Theorem for Differential Relations}. 
The core of this argument is \nameref{Lemma:A}, where we show that $\delta_0$ approximates $Q$ up to a quadratic estimate. 
The proof of this result is rather standard: 
it relies on a local computation, the Taylor approximation formula for the map $q$, and the interpolation inequalities (see section \ref{sect:lemmas}).
Recall that $Q$ is a differential operator of order $d$, and $Q(0)=0$.

\begin{NamedTheorem}[Lemma A]
\label{Lemma:A}
There are constants $\theta>0$ and $c_k\geq 0$ such that
    \[
        \norm{ Q(e)-\delta_0(e) }_{k,r}
        \leq 
        c_k \norm{ e }_{d,r} \norm{ e }_{k+d,r},
    \]
for all $k\geq 0$, all $r\in [0,1]$, and all $e\in \Gamma_E(M_r)$ satisfying $\norm{ e }_{d,r}<\theta$.
\end{NamedTheorem}

\begin{lemma}
\label{Lemma:RelationsImpliesMain}
Suppose that $\{ h_1^{s,r} \}$ and $\{ h_2^{s,r} \}$ are families of operators that satisfy the hypothesis of the \nameref{General Main Theorem}:
there are constants $c_k, d_k, l_1, l_2 \geq 0$ such that
    \[  
        (
            \delta_0 \circ h_1^{s,r} + h_2^{s,r} \circ \delta_0
        )(e)
        =
        e|_{M_r},
        \qquad
        \norm{ h_i^{s,r}(e) }_{k, r}
        \leq
        c_k (s-r)^{-d_k} 
        \norm{ e }_{k + l_i, s}.
    \]
Then the hypothesis of the \nameref{Main Theorem for Differential Relations} can be met: there are constants $c_k, d_k, l \geq 0$ and $\theta > 0$ such that
    \[
        \norm{ 
            e\vert_{M_r}
            -
            \delta_0\circ h_1^{s,r}\!(e)
        }_{k,r}
        \leq 
        c_k (s-r)^{-d_k}
        \norm{ e }_{d,s}
        \norm{ e }_{k+l,s}
    \]
for all $e\in \Gamma_E(M_s)$ satisfying 
    $
        \norm{ e }_{d,r}<\theta 
    $
and
    $
        Q(e)=0.
    $
\end{lemma}

\begin{remark}
We obtain $l = d+l_2$.
\end{remark}

\begin{proof}
This follows from the weak homotopy relation, the tameness inequalities for $h_2^{s,r}$, and the previous lemma with $Q(e)=0$: 
    \begin{align*}
        \norm{ 
            e
            -
            (\delta_0\circ h_1^{s,r})(e)
        }_{k,r}
        &\leq
        \norm{ 
            e
            -
            (
                \delta_0\circ h^{s,r}_1
                +
                h^{s,r}_2 \circ \delta_0
            )(e)
        }_{k,r}
        +
        \norm{ (h_2^{s,r} \circ \delta_0)(e) }_{k,r}
        \\
        &\leq 
        c_k(s-r)^{-d_k}
        \norm{ \delta_0(e) }_{k+l_2,s}
        \\
        &\leq 
        c_k(s-r)^{-d_k}
        \norm{ e }_{d,s}
        \norm{ e }_{k+d+l_2,s}.
        \qedhere
    \end{align*}
\end{proof}

\subsection{Lemma B: Flows and the Action}
\label{Step 4}

We give estimates for the flows of vector fields, and for their action on sections. 
The proof of \nameref{Lemma:B} can be found in section \ref{sect:lemmas} (see also \cite{Mar14,MMZ12} for similar statements). 
Recall from section \ref{Step 2} that $\calU(X_r)\sub \diff_X(X_r)$ denotes a fixed $C^0$-open neighborhood of $\id\!|_{X_r}$ on which we have fixed $C^k$-distances $\norm{\cdot}_{k,r}$ for local diffeomorphisms.

\begin{NamedTheorem}[Lemma B]
\label{Lemma:B}
There are constants $\theta>0$ and $c_k\geq 0$ such that for all $r<s$, and all $v\in \gerX(X_s)$ and $e\in \Gamma_E(M_s)$ satisfying
    \[
        \norm{v}_{1,s} < (s-r) \theta,
        \qquad
        \norm{e}_{1,s} < (s-r) \theta,
    \]
the flow $\varphi_v$ of $v$ is defined on $[0,1]\!\times\! X_{(s+r)/2}$ and belongs to $\calU(X_{(s+r)/2})$,
    \[  
        \varphi_v^t:
        X_{(s+r)/2}\longto X_s,
        \qquad 
        \varphi_v^{t} \in \calU(X_{(s+r)/2}),
    \]
and $\varphi_v^t$ and $e$ are compatible on $M_r$ for all $t\in [0,1]$,
    \[
        (e\cdot\varphi_v^t)|_{M_r}
        \in 
        \Gamma_E(M_r).
    \]
Moreover, the time-one flow $\varphi_v:=\varphi_v^1$ satisfies the following inequalities:
\begin{align}
    \norm{\varphi_v}_{k,{(s+r)/2}}
    &\leq 
    c_k \norm{v}_{k,s},
    \tag{1}
    \\
    \norm{e\cdot\varphi_v}_{k,r}
    &\leq 
    c_k
    \big(
        \norm{e}_{k,s}
        +
        \norm{v}_{k,s}
    \big),
    \tag{2}
    \\
    \norm{
        e\cdot\varphi_v
        -
        e
        -
        \delta v
    }_{k,r}
    &\leq 
    c_k
    \big(
        ( \norm{e}_{0,s} + \norm{v}_{0,s} )
        \norm{v}_{k+1,s}
        +
        \norm{v}_{0,s}
        \norm{e}_{k+1,s}
    \big)
    \tag{3}.
\end{align}
\end{NamedTheorem}

\subsection{Lemma C: Infinite Compositions}
\label{Step 5}

We give a convergence result for the infinite composition of diffeomorphisms. 
The proof of \nameref{Lemma:C} can be found in section \ref{sect:lemmas} (see also \cite{Conn85,Mar14} for similar statements).

\begin{NamedTheorem}[Lemma C]
\label{Lemma:C}
There are constants $\theta>0$ and $c_k\geq 0$ such that, 
for all decreasing sequences
    \[
        \{r_{\nu}\}_{\nu\geq 0},
        \quad
        0 \leq r_{\nu+1}\leq r_{\nu}\leq 1,
    \]
and all sequences of local diffeomorphisms
    \[
        \{\phi_{\nu}\in \calU_{X_r}\}_{\nu\geq 1},
        \quad
        \phi_{\nu}:X_{r_{\nu}} \longto X_{r_{\nu-1}}, 
    \]
satisfying
    \[
        \sum\nolimits_{\nu\geq 1}
        \norm{ \phi_{\nu} }_{1,r_\nu}<\theta
        \quad\textrm{and}\quad 
        \sum\nolimits_{\nu\geq 1}
        \norm{ \phi_\nu }_{k,r_\nu}<\infty,
        \quad\forall\; k\geq 0,
    \]
the sequence of smooth maps $\{\psi_\nu \}_{\nu\geq 1}$ given by
    \[
        \psi_\nu
        \defeq 
        \phi_1\circ\phi_2\circ\ldots \circ \phi_\nu \vert_{X_r}:
        X_r \longto X_{r_0},
        \quad\textrm{where}\quad 
        r\defeq\lim_{\nu\to \infty}r_{\nu},
    \]
converges on $X_r$ w.r.t.\ the $C^{\infty}$-topology to a smooth map
    \[
        \psi:X_r \longto X_{r_0},
        \quad 
        \psi 
        \defeq
        \lim_{\nu\to \infty}\psi_\nu \big|_{X_r}.
    \]
Moreover, $\psi$ is a diffeomorphism that lies in $\calU(X_r)$ and there are constants $c_k\geq 0$ such that
    \[
        \norm{ \psi }_{k,r}
        \leq 
        c_k 
        \sum\nolimits_{\nu\geq 1}
        \norm{ \phi_\nu }_{k,r_\nu}.
    \]
\end{NamedTheorem}

\subsection{Smoothing Operators and Interpolation Inequalities}
\label{Step 6}

A main ingredient in the Nash-Moser method is the use of smoothing operators, which account for the \emph{loss of derivatives} during the iteration. The \textbf{smoothing operators}\index{smoothing operators} we will be using are a family of linear operators,
    \[
        S_t^r:
        \Gamma_E(M_r) \longto \Gamma_E(M_r),
        \qquad 
        r\in [0,1],
        \quad
        t>1,
    \]
which satisfy the smoothing inequalities:
    \begin{equation}
        \label{ineq:smoothing}
        \norm{ S_t^r (e) }_{k,r}
        \leq 
        c_k t^l
        \norm{ e }_{k-l,r},
        \qquad 
        \norm{ e-S_t^r (e) }_{k-l,r}
        \leq 
        c_k t^{-l}
        \norm{ e }_{k,r},
    \end{equation}
for all $e\in \Gamma_E(M_r)$ and $0\leq l\leq k$. On sections of a vector bundle with compact base (possibly with corners), the existence of such operators is proven in \cite{Ham82}. In section \ref{sec:lemmas_corners} we give a simple argument to show how one can use a family of smoothing operators $S^1_t$ on $\Gamma_E(M_1)$ to obtain a family $S_t^r$ for all $r\in [0,1]$ as above.

The norms $\norm{ \cdot }_{k,r}$ satisfy the following \textbf{interpolation inequalities}\index{interpolation inequalities}:
    \begin{equation}
    \label{ineq:interpol}
        \norm{ e }_{k,r}^{n-l}
        \leq 
        c_k
        \norm{ e }_{l,r}^{n-k}
        \norm{ e }_{n,r}^{k-l},
        \qquad l \leq k\leq n.
    \end{equation}
This can be proven directly \cite{Conn85}, but it also follows from the smoothing inequalities \cite{Ham82}.

\newpage
\section{The Nash-Moser Iteration}\label{Step 7}
\etocsettocdepth{2}
\localtableofcontents

\noindent
{\color{white}This line is intentionally left blank.}
\newline
We prove the 
\nameref{Main Theorem for Differential Relations} using a Nash-Moser type iteration:
for any local solution $e \in \sol_{Q,Z}$ we define local diffeomorphisms $\{\phi_\nu\}_{\nu\geq 0} \in \pgrp$ whose infinite composition $\psi = \phi_0\circ\phi_1\circ\ldots$ will satisfy the equation $e\cdot \psi = 0$. 
Since $\pgrp$ is a closed pseudogroup, it immediately follows that also $\psi\in\pgrp$.

\subsection{The List of Constants}

Fix a constant $\theta>0$ smaller than the constants denoted also by $\theta$ from the 
\nameref{Main Theorem for Differential Relations}, \nameref{Lemma:B} and \nameref{Lemma:C}.

Consider also a real number $t_0>1$ and let 
    \[
        t_{\nu} \defeq t_{0}^{(3/2)^{\nu}}.
    \]
The value of $t_0$ will be determined so that it is larger than a finite number of constants which will appear during the proof. 
More precisely, we will assume that: \emph{For a finite number of pairs of constants $c, \mu>0$ (more precisely, seven), the following inequality holds:}
    \begin{equation}\label{ineq:t_0}
        c^{\nu}
        t_{\nu}^{- \mu}
        <1,
        \quad
        \textrm{for all}\ \nu\geq 0.
    \end{equation}

We will also use the following numbers, which are expressed in terms of the constants from the statement of the 
\nameref{Main Theorem for Differential Relations},
    \begin{align*}
        p &\defeq \mathrm{max}(l_1+1, d),
        \\
        q &\defeq \mathrm{max}(6l_1+5, 4l_2+1),
        \\
        l &\defeq \mathrm{max}(4l_1+1, l_1+l_2), 
        \\
        n &\defeq \mathrm{max}(l_1+1, l_2+d),
        \\
        b &\defeq \mathrm{max}(d_1+1, d_p, 2d_{p+q}).
    \end{align*}

\begin{remark}
By lemma \ref{Lemma:RelationsImpliesMain}, for the \nameref{General Main Theorem} one should replace $l_2$ by $d+l_2$ in the above constants.
\end{remark}

Fix a pair $0<r<s\leq 1$, and define two sequences of radii $\{ s_\nu \}$ and $\{ r_\nu \}$ such that
    \[ 
        s_{\nu+1} < r_\nu < s_\nu,
        \qquad\forall\;
        \nu\geq 0.
    \]
More specifically, pick
    \[
        s_0\defeq s,
        \qquad  
        s_{\nu+1} \defeq s_\nu - (s-r)3^{-(\nu+1)}, 
        \qquad 
        r_{\nu} \defeq s_{\nu} - \tfrac12 (s-r)3^{-(\nu+1)}.
    \]
Then both sequences converge to $r_{\infty}\defeq\tfrac12 (s+r)$.
Similar to the sequence $t_\nu$, let
    \[
        \epsilon_\nu \defeq (s-r)^{b(3/2)^\nu}.
    \]
Abbreviate the smoothing operators to
    \[
        S_\nu\defeq S^{s_{\nu}}_{\epsilon_\nu^{-1} t_\nu}: 
        \Gamma_E(M_{s_\nu})
        \longto 
        \Gamma_E(M_{s_\nu}),
    \]
and the homotopy operators to
    \[
        h_{\nu}
        \defeq 
        h^{s_\nu, r_\nu}:
        \Gamma_E(M_{s_{\nu}})
        \longto
        \calA(X_{r_\nu}).
    \]

\subsection{The Iteration}
Fix a solution $e\in \sol_R(M_s)$, and assume that $e$ satisfies
    \begin{align}
    \label{size}
   	    \norm{e}_{p,s} 
   	    < 
   	    (s-r)^b
   	    t_0^{-(2l_1+1)},
        \qquad
        \norm{e}_{p+q,s} 
        < 
        (s-r)^{-b}
        t_0^{(2l_1+1)}.
    \end{align}
For all $\nu\geq 0$ we will construct inductively a sequence of solutions of $R$, where
    \[
        e_\nu \in \sol_R(M_{s_\nu}),
        \qquad
        e_{0}\defeq e.
    \]
Assume that $e_\nu$ has been defined, and consider the vector field $v_\nu \in \calA(X_{r_\nu})$ given by
    \[
        v_{\nu}
        \defeq 
        -(h_\nu \circ S_\nu)(e_\nu).
    \]
Moreover, assume that $v_\nu$ and $e_\nu$ satisfy the inequalities
    \begin{equation}
    \label{ineq:flow_action}
        \norm{v_\nu}_{1, r_\nu}
        \leq 
        (r_\nu - s_{\nu+1}) \theta,
        \qquad  
        \norm{e_\nu}_{1, r_\nu}
        \leq 
        (r_\nu - s_{\nu+1}) \theta.
    \end{equation}
Then by \nameref{Lemma:B} the time-one flow $\varphi_{v_\nu}$ of $v_\nu$ is defined as follows:
    \[
        \varphi_{v_\nu}:
        X_{(r_\nu+s_{\nu+1})/2}
        \longto 
        X_{r_\nu},
        \qquad
        \varphi_{v_\nu}
        \in 
        (\pgrp\medcap \calU)(X_{(r_\nu + s_{\nu+1})/2}), 
    \]
Moreover, the local diffeomorphism $\varphi_{v_\nu}$ and the local section $e_\nu |_{M_{r_\nu}}$ are compatible on $M_{s_{\nu+1}}$.
This produces the next term of the sequence:
    \[
        e_{\nu+1}
        \defeq 
        (e_\nu|_{M_{r_\nu}} \cdot \varphi_{v_\nu})|_{M_{s_{\nu+1}}}
        \;\in\; 
        \sol_R(M_{s_{\nu+1}}).
    \]
We will prove the following statement, from which the \nameref{Main Theorem for Differential Relations} follows.

\begin{remark}
The assumptions (\ref{size}) specify the $C^{p+q}$-open neighborhood of $b=0$ from the 
\nameref{Main Theorem for Differential Relations}, while the number $l$ is the tameness degree of the map $e\mapsto \psi$.
\end{remark}

\subsection{The Proof}

For $\nu\geq 0$ we prove inductively the following inequalities:
\begin{itemize}

\item Induction hypothesis $(a)_\nu$: the \emph{weak high-norm} estimate
    \[
        \norm{e_\nu}_{p+q,s_\nu}
        \leq 
        (\epsilon_\nu^{-1} t_\nu)^{2l_1+1}.
    \]

\item Induction hypothesis $(b)_\nu$: the \emph{strong low-norm} estimate
    \[
        \norm{e_\nu}_{p,s_\nu}
        \leq
        (\epsilon_\nu^{-1} t_\nu)^{-(2l_1+1)}.
    \]
\end{itemize}
The second inductive hypothesis ensures that the iteration indeed tends to zero for low-index norms, while the first hypothesis is mostly used to effectively apply the smoothing and interpolation estimates.

The inductive hypotheses $(a)_0$ and $(b)_{0}$ are precisely our base assumptions (\ref{size}) on $e$. 
Assume that for some $\nu\geq 0$ the iteration produces an $e_{\nu}$ satisfying $(a)_\nu$ and $(b)_{\nu}$. 
We will show that (\ref{ineq:flow_action}), $(a)_{\nu+1}$ and $(b)_{\nu+1}$ hold.

First we prove an inequality which we will use repeatedly.

\begin{lemma}
\label{lemma:often}
There are constants $c_k \geq 0$ such that, for any $m\geq 0$,
    \[
        \norm{v_{\nu}}_{k, r_{\nu}}
        \leq 
        c_k (s_\nu - r_\nu)^{-d_k}
        (\epsilon_\nu^{-1} t_\nu)^m
        \norm{ e_\nu }_{k+l_1-m,s_\nu}.
    \]
\end{lemma}

\begin{proof}
We use the tameness inequality of the homotopy operator $h_\nu=h^{s_\nu,r_\nu}$, 
and the smoothing inequalities (\ref{ineq:smoothing}) for any $m \geq 0$:
    \begin{align*}
        \norm{v_\nu}_{k, r_\nu}
        =
        \norm{ (h_{\nu} \circ S_\nu)(e_\nu) }_{k, r_\nu}
        &\leq 
        c_k (s_\nu - r_\nu)^{-d_k}
        \norm{ S_\nu(e_\nu) }_{k+l_1, s_\nu}
        \notag
        \\
        &\leq
        c_k (s_\nu - r_\nu)^{-d_k}
        (\epsilon_\nu^{-1} t_\nu)^m
        \norm{ e_\nu }_{k+l_1-m,s_\nu}.
        \qedhere
    \end{align*}
\end{proof}
\noindent
\begin{proposition}
\label{proposition:induction_hypothesis}
$(b)_\nu$ implies $(\ref{ineq:flow_action})_\nu$.
\end{proposition}

\begin{proof}
For the first inequality of (\ref{ineq:flow_action}) we use lemma \ref{lemma:often} for $k=1$ and $m=0$, 
induction hypothesis $(b)_\nu$ where $p\geq l_1+1$, 
that $b\geq d_1+1$, 
and inequality (\ref{ineq:t_0}):
    \begin{align*}
        \norm{ v_\nu }_{1,r_\nu}
        &\leq 
        c_1 (s_\nu - r_\nu)^{-d_1}
        \norm{ e_\nu }_{l_1+1,s_\nu}
        \\
        &\leq
        c_1 (s_\nu - r_\nu)^{-d_1} 
        (\epsilon_\nu t_\nu^{-1})^{2l_1+1}
        \\
        &\leq
        c_1 2^{d_1} 3^{d_1(\nu+1)}
        (s - r)
        t_\nu^{-(2l_1+1)}
        \big(
            (s - r)^{-(d_1+1)} \epsilon_\nu^{2l_1+1}
        \big)
        \\
        &\leq
        (r_\nu-s_{\nu+1})
        \theta
        \big(
            c_1
            \theta^{-1}
            2^{1+d_1}
            3^{(1+d_1)(\nu+1)} 
            t_\nu^{-(2l_1+1)}
        \big)
        \\
        &
        \leq
        (r_\nu-s_{\nu+1})
        \theta.
    \end{align*}
The second inequality in (\ref{ineq:flow_action}) follows from using induction hypothesis $(b)_\nu$, that $b\geq 1$, and inequality (\ref{ineq:t_0}):
    \begin{align*}
        \norm{ e_\nu }_{1,s_\nu}
        &\leq 
        \norm{ e_\nu }_{p,s_\nu}
        \leq 
        (\epsilon_\nu t_\nu^{-1})^{2l_1 + 1}
        \\
        &\leq
        (s-r)
        t_\nu^{-(2l_1+1)}
        \big(
            (s - r)^{-1} \epsilon_\nu^{2l_1+1}
        \big)
        \\
        &\leq
        (r_\nu - s_{\nu+1}) \theta
        \big(
            2 \theta^{-1} 3^{\nu+1} t_\nu^{-(2l_1+1)}
        \big)
        \leq
        (r_\nu - s_{\nu+1}) \theta.
        \qedhere
    \end{align*}
\end{proof}
\noindent
By \nameref{Lemma:B} we conclude that the step 
    $
        e_{\nu+1}
        \defeq (e_{\nu}|_{M_{r_\nu}} \cdot \varphi_{v_{\nu}})
    $ 
is well-defined on $M_{s_{\nu+1}}$.

\begin{lemma}
\label{lemma:diffnorm}
    $
        \norm{ \varphi_{v_\nu} }_{1,(r_\nu + s_{\nu+1})/2}
        \leq 
        (r_{\nu} - s_{\nu+1}) \theta.
    $
\end{lemma}

\begin{proof}
We may use \nameref{Lemma:B}(1) because of inequality (\ref{ineq:flow_action}), and then we repeat the argument of proposition \ref{proposition:induction_hypothesis}:
    \begin{align*}
        \norm{ \varphi_{v_\nu} }_{1, (r_\nu +s_{\nu+1})/2}
        & \leq 
        c_1
        \norm{ v_\nu }_{1, r_\nu}
        \leq 
        c_1 (s_\nu - r_\nu)^{-d_1}
        \norm{ e_\nu }_{l_1+1,s_\nu}
        \\
        &\leq
        c_1 (s_\nu - r_\nu)^{-d_1} 
        (\epsilon_\nu t_\nu^{-1})^{2l_1+1}
        \\
        &\leq
        c_1 2^{d_1} 3^{d_1(\nu+1)}
        (s - r)
        t_\nu^{-(2l_1+1)}
        \big(
            (s - r)^{-(d_1+1)} \epsilon_\nu^{2l_1+1}
        \big)
        \\
        &\leq
        (r_\nu-s_{\nu+1})
        \theta
        \big(
            c_1
            \theta^{-1}
            2^{1+d_1}
            3^{(1+d_1)(\nu+1)} 
            t_\nu^{-(2l_1+1)}
        \big)
        \\
        & \leq
        (r_\nu-s_{\nu+1})
        \theta.
        \qedhere
    \end{align*}
\end{proof}
\noindent
The use of smoothing operators avoids the \emph{loss of derivatives} in the iteration, which plays a key role in the following lemma.

\begin{lemma}
\label{lemma:no_loss}
There are constants $c_k \geq 0$ such that
    \[
        \norm{ e_{\nu+1} }_{k, s_{\nu+1}}
        \leq 
        c_k (s_\nu - r_\nu)^{-d_k} 
        (\epsilon_\nu^{-1} t_\nu)^{l_1}
        \norm{ e_\nu }_{k,s_\nu}.
    \]
\end{lemma}

\begin{proof}
We use \nameref{Lemma:B}(2) and lemma \ref{lemma:often} for $m=l_1$:
    \[
        \norm{ e_{\nu+1} }_{k,s_{\nu+1}}
        \leq 
        c_k(
            \norm{ e_{\nu} }_{k,s_{\nu}}
            +
            \norm{ v_{\nu} }_{k,r_{\nu}}
        )
        \leq
        c_k (s_\nu - r_\nu)^{-d_k} 
        (\epsilon_\nu^{-1} t_\nu)^{l_1}
        \norm{ e_{\nu} }_{k,s_{\nu}}.
        \qedhere
    \]
\end{proof}
\noindent
Now we prove the weak high-norm estimate.

\begin{proposition}
\label{proposition:hypothesis_a}
$(a)_\nu$ and $(b)_\nu$ imply $(a)_{\nu+1}$.
\end{proposition}

\begin{proof}
We use lemma \ref{lemma:no_loss} for $k=p+q$, induction hypothesis $(a)_{\nu}$, that $b\geq 2 d_{p+q}$, and inequality (\ref{ineq:t_0}):
\begin{align*}
    \norm{ e_{\nu+1} }_{p+q,s_{\nu+1}}
    &\leq 
    c_{p+q} (s_\nu - r_\nu)^{-d_{p+q}}
    (\epsilon_\nu^{-1} t_\nu)^{l_1}
    \norm{ e_\nu }_{p+q,s_\nu}
    \\
    & \leq
    c_{p+q} (s_\nu - r_\nu)^{-d_{p+q}} 
    (\epsilon_\nu^{-1} t_\nu)^{3l_1 + 1}
    \\
    &=
    c_{p+q} 2^{d_k} 3^{d_{p+q}(\nu+1)}
    \big(
        (s-r)^{-d_{p+q}}
        \epsilon_{\nu+1}^{1/2}
    \big)
    \epsilon_{\nu+1}^{-(2l_1+1)}
    t_\nu^{3l_1 + 1}
    \\
    &\leq
    \big(  
        c_{p+q}
        2^{d_k} 
        3^{d_{p+q}(\nu+1)} 
        t_{\nu}^{-1/2}
    \big)
    \epsilon_{\nu+1}^{-(2l_1+1)}
    t_{\nu+1}^{2l_1+1}
    \\
    & \leq
    (\epsilon_{\nu+1}^{-1} t_{\nu+1})^{2l_1+1}.
    \qedhere
\end{align*}
\end{proof}
\noindent
For the strong low-norm estimate we decompose $e_{\nu+1}$ as the sum
    \[
        e_{\nu+1}
        =
        \big(
            e_{\nu}\cdot\varphi_{v_{\nu}}
            -
            e_{\nu}
            -
            \delta_0(v_{\nu})
        \big)
        +
        \big(
            e_{\nu}
            -
            (\delta_0 \circ h_{\nu})(e_{\nu})
        \big)
        +
        \big(
            (\delta_0 \circ h_{\nu} \circ (\id-S_{\nu}))(e_{\nu})
        \big),
    \]
which we denote by $T_1$, $T_2$, and $T_3$. The first two terms admit \emph{quadratic} inequalities.

\begin{lemma}
\label{lemma:T12}
There are constants $c_k \geq 0$ such that
    \[
        \norm{ T_1 }_{k,s_{\nu+1}}
        +
        \norm{ T_2 }_{k,s_{\nu+1}}
        \leq 
        c_k (s_\nu - r_\nu)^{-2 d_{k+1}} (\epsilon_\nu t_\nu^{-1})^{2l_1+1}
        \norm{ e_\nu }_{k+n, s_\nu}.
    \]
\end{lemma}
\begin{proof}
We use \nameref{Lemma:B}(3), lemma \ref{lemma:often} for $m=0$, and then induction hypothesis $(b)_{\nu}$ where $p\geq l_1$:
    \begin{align*}
        \norm{ T_1 }_{k,s_{\nu+1}}
        &\leq 
        c_k \big(
            \norm{ v_\nu }_{0,r_\nu}
            \norm{ e_\nu }_{k+1,s_\nu}
            +
            (
                \norm{ v_\nu }_{0,r_\nu}
                +
                \norm{ e_\nu }_{0,s_\nu}
            )
            \norm{ v_\nu }_{k+1,r_\nu}
        \big)
        \\
        &\leq 
        c_k (s_\nu - r_\nu)^{- (d_0 + d_{k+1})}
        \norm{ e_\nu }_{l_1,s_\nu}
        \norm{ e_\nu }_{k+l_1+1,s_\nu}
        \\
        &\leq 
        c_k (s_\nu - r_\nu)^{- 2 d_{k+1} } 
        (\epsilon_\nu^{-1} t_\nu)^{-(2l_1+1)}
        \norm{ e_\nu }_{k+l_1+1,s_\nu}.
    \end{align*}
By using that $p\geq d$, induction hypothesis $(b)_{\nu}$,  and inequality (\ref{ineq:t_0}) we obtain that
    \[
        \norm{ e_\nu }_{d,s_\nu}
        \leq 
        c_d (\epsilon_\nu t_\nu^{-1})^{2l_1+1}
        <
        (s-r) \theta.
    \]
$(\pgrp,R)$ is a PDR with symmetry, so $e_{\nu}$ remains a solution of $R$. 
Therefore, by the above estimate, we may apply the \emph{quadratic} inequality (\ref{ineq:quadratic}) to $e_{\nu}$. 
This gives the estimate for $T_2$:
    \begin{align*}
        \norm{ T_2 }_{k,s_{\nu+1}}
        &\leq
        c_k (r_\nu - s_{\nu+1})^{-d_k}
        \norm{ e_\nu }_{d,s_\nu}
        \norm{ e_\nu }_{k+d+l_2,s_\nu}
        \\
        &\leq 
        c_k (s_\nu - r_\nu)^{-d_k}
        \norm{ e_\nu }_{d,s_\nu}
        \norm{ e_\nu }_{k+d+l_2,s_\nu}
        \\
        &\leq 
        c_k (s_\nu - r_\nu)^{-d_k} 
        (\epsilon_\nu t_\nu^{-1})^{2l_1+1}
        \norm{ e_\nu }_{k+d+l_2,s_\nu }.
        \qedhere
    \end{align*}
\end{proof}
\noindent
\begin{lemma}
\label{lemma:T3}
For all $k,m\geq 0$ there are constants $c_{k,m} \geq 0$ such that
    \[
        \norm{ T_3 }_{k,s_{\nu+1}}
        \leq 
        c_{k,m} (s_\nu - r_\nu)^{-d_k} 
        (\epsilon_\nu t_\nu^{-1})^m
        \norm{ e_\nu }_{k+l_1+m,s_\nu}.
    \]
\end{lemma}
\begin{proof}
First note that the operator
    $
        \delta_0:
        \calA(X_r)  
        \!\to\! 
        \Gamma_E(M_r)
    $
satisfies an inequality of the form 
    \[
        \norm{ \delta_0(v) }_{k,r}
        \leq 
        c_k
        \norm{ v }_{k,r}.
    \]
Now we use this, the tameness inequality for $h_\nu$, and the smoothing inequalities (\ref{ineq:smoothing}):
    \begin{align*}
        \norm{ T_3 }_{k,s_{\nu+1}}
        &\leq 
        c_k
        \norm{ (h_\nu \circ (\id-S_\nu))(e_\nu) }_{k,r_\nu}
        \\
        & \leq
        c_k (s_\nu - r_\nu)^{-d_k}
        \norm{ e_\nu - S_\nu(e_\nu) }_{k+l_1,s_\nu}
        \\
        &\leq 
        c_{k,m} (s_\nu - r_\nu)^{-d_k} 
        (\epsilon_\nu t_\nu^{-1})^m
        \norm{ e_\nu }_{k+l_1+m,s_\nu}.
        \qedhere
\end{align*}
\end{proof}
\noindent
Now we prove the \emph{strong low-norm} estimate.

\begin{proposition}
\label{proposition:hypothesis_b}
$(a)_\nu$ and $(b)_\nu$ imply $(b)_{\nu+1}$.
\end{proposition}
\begin{proof}
We evaluate $\norm{ e_{\nu+1} }_{p,s_{\nu+1}}$ by using the above lemmas for $k=p$. 
From this, combined with the interpolation inequalities for $p\leq p+n\leq p+q$, induction hypotheses $(a)_{\nu}$ and $(b)_{\nu}$, that $n/q<1/4$, and inequality (\ref{ineq:t_0}) we obtain
    \begin{align*}
        &
        \norm{ T_1 }_{p,s_{\nu+1}}
        +
        \norm{ T_2 }_{p,s_{\nu+1}}
        \leq
        c_p (s_\nu - r_\nu)^{- 2 d_{p+1}} 
        (\epsilon_\nu t_\nu^{-1})^{2l_1+1}
        \norm{ e_\nu }_{p+n, s_\nu}
        \\
        &\leq
        c_p (s_\nu - r_\nu)^{- 2 d_{p+1}} 
        (\epsilon_\nu t_\nu^{-1})^{2l_1+1}
        \norm{ e_\nu }_{p,s_\nu}^{1-n/q}
        \norm{ e_\nu }_{p+q,s_\nu}^{n/q}
        \\
        &\leq 
        c_p (s_\nu - r_\nu)^{-2d_{p+1}}
        (\epsilon_\nu t_\nu^{-1})^{2(2l_1+1)(1-n/q)}
        \\
        &=
        c_p 2^{2d_{p+1}} 9^{d_{p+1}(\nu+1)} 
        \big(
            (s - r)^{-2d_{p+1}}
            \epsilon_\nu^{2(2l_1+1)(1/4 - n/q)}
        \big)
        \epsilon_{\nu+1}^{2l_1+1}
        t_\nu^{-2(2l_1+1)(1-n/q)}
        \\
        &\leq
        \big(   
            c_p 2^{2d_{p+1}+1} 9^{d_{p+1}(\nu+1)} 
            t_\nu^{-2(2l_1+1)(1/4 - n/q)}
        \big)
        \tfrac12
        (\epsilon_{\nu+1} t_{\nu+1}^{-1})^{2l_1+1}
        \\
        &\leq
        \tfrac12 (\epsilon_{\nu+1} t_{\nu+1}^{-1})^{2l_1+1}.
    \end{align*}
We use lemma \ref{lemma:T3} for $k=p$ and $m=q-l_1$, induction hypothesis $(a)_{\nu}$, that $q>6l_1+5/2$, and inequality (\ref{ineq:t_0}):
    \begin{align*}
        \norm{ T_3 }_{p,s_{\nu+1}}
        &\leq 
        c_{p, q-l_1} (s_\nu - r_\nu)^{-d_p} 
        (\epsilon_\nu t_\nu^{-1})^{q-l_1}
        \norm{ e_\nu }_{p+q,s_\nu}
        \\
        &\leq
        c_{p, q-l_1} (s_\nu - r_\nu)^{-d_p}
        (\epsilon_\nu t_\nu^{-1})^{q - 3l_1 -1}
        \\
        &=
        c_{p, q-l_1} 2^{d_p} 3^{d_p(\nu+1)}
        \big(
            (s - r)^{-d_p} \epsilon_\nu^{q-6l_1-5/2}
        \big)
        \epsilon_{\nu+1}^{2l_1+1}
        t_\nu^{3l_1+1-q}
        \\
        & \leq
        \big(
            c_{p, q-l_1}
            2^{d_p+1} 3^{d_p(\nu+1)}
            t_\nu^{6l_1+5/2-q}
        \big)
        \tfrac12
        (\epsilon_{\nu+1} t_{\nu+1}^{-1})^{2l_1+1}
        \\
        & \leq
        \tfrac12
        (\epsilon_{\nu+1} t_{\nu+1}^{-1})^{2l_1+1}.
        \qedhere
    \end{align*}
\end{proof}
\noindent
We conclude that the iteration is well-defined, and $e_{\nu}$ satisfies $(a)_{\nu}$ and $(b)_{\nu}$, for all $\nu\geq 0$.
The next Lemma is used to show convergence of the iteration.

\begin{proposition}
\label{proposition:almost_done}
For all $k\geq 0$ there are constants $a_k, b_k\geq0$ such that
    \[
        \sum\nolimits_{\nu \geq 0}
        \norm{ v_\nu }_{k, r_\nu}
        \leq 
        a_k(s - r)^{-b_k}
        \norm{ e }_{k+l, s}.
    \]
\end{proposition}

\begin{proof}
By iterating the inequality from lemma \ref{lemma:no_loss}, and using that
    \[
        (\epsilon_0^{-1}t_0)
        (\epsilon_1^{-1}t_1)
        \cdots 
        (\epsilon_{\nu-1}^{-1} t_{\nu-1})
        =
        (\epsilon_0^{-1} t_0)^{\sum_{j=0}^{\nu-1} (3/2)^{j}}
        < 
        (\epsilon_0^{-1}t_0)^{2(3/2)^{\nu}}
        =
        (\epsilon_\nu^{-1} t_\nu)^2,
    \]
we obtain the inequality
    \[
        \norm{ e_\nu }_{k,s_\nu}
        \leq 
        c_k^\nu 
        (s_\nu - r_\nu)^{-\nu d_k}
        (\epsilon_\nu^{-1} t_\nu)^{2l_1}
        \norm{ e }_{k,s}.
    \]
We use lemma \ref{lemma:often} for $m=0$, 
lemmas \ref{lemma:T12} and \ref{lemma:T3} for $m=2l_1+1$, 
the above inequality, 
and that $l\geq n+l_1$  and $l\geq 4l_1+1$:
    \begin{align*}
        \norm{ v_{\nu+1} }_{k,r_{\nu+1}}
        &\leq
        c_k (s_{\nu+1} - r_{\nu+1})^{-d_k}
        \norm{ e_{\nu+1} }_{k + l_1,s_{\nu+1}}
        \\
        &\leq
        c_k 3^{d_k} 
        (s_\nu - r_\nu )^{-3 d_k}
        (\epsilon_\nu t_\nu^{-1})^{2l_1+1}
        \big(
            \norm{ e_\nu }_{k+l_1+n,s_\nu}
            +
            \norm{ e_\nu }_{k+4l_1+1,s_\nu}
        \big)
        \\
        &\leq 
        c_k^{\nu+1} 3^{d_k}
        (s_\nu - r_\nu)^{-(\nu+3)d_k}
        \epsilon_\nu t_\nu^{-1}
        \norm{ e }_{k+l,1}
        \\
        &\leq
        \big(
            (c_k2^{d_k})^{\nu+1}
            3^{(\nu^2 + 3\nu + 3)d_k}
            t_0^{-(3/2)^\nu}
        \big)
        (s - r)^{b(3/2)^\nu-(\nu+2)d_k}
        \norm{ e }_{k+l, s}.
    \end{align*}
Since $(3/2)^\nu$ grows faster than any polynomial expression in $\nu$, the proposition follows.
\end{proof}
\noindent
We conclude the proof by applying \nameref{Lemma:C} to the sequence 
    $
        \set{\varphi_{v_\nu}}_{\nu\geq 0}.
    $

\begin{proposition}
\label{prop:iteration}
The sequence
    $
        \psi_\nu
        \defeq 
        \varphi_{v_0}
        \circ \ldots\circ
        \varphi_{v_{\nu}}|_{X_r}
    $
converges w.r.t.\ the $C^\infty$-topology on $\pgrp(X_r)$ to a diffeomorphism $\psi\in \pgrp(X_r)$
such that $\psi$ and $e$ are compatible on $M_r$ and
    \[
        (e\cdot \psi)|_{M_r} = 0.
    \]
Moreover, for all $k\geq 0$ there are constants $a_k, b_k \geq 0$ such that
    \[
	    \norm{\psi}_{k,r} 
        \leq 
        a_k (s-r)^{-b_k} 
        \norm{e}_{k+l,s}.
    \]
\end{proposition}

\begin{proof}
We use lemma \ref{lemma:diffnorm}, \nameref{Lemma:B}(1) and proposition \ref{proposition:almost_done}:
    \begin{align*}
        \sum\nolimits_{\nu\geq 0}
        \norm{ \varphi_{v_\nu} }_{1,s_{\nu+1}}
        & < 
        \tfrac12 (s-r) \theta,
        \\
        \sum\nolimits_{\nu\geq 0}
        \norm{ \varphi_{v_\nu} }_{k,s_{\nu+1}}
        & \leq 
        c_k 
        \sum\nolimits_{\nu\geq 0}
        \norm{ v_\nu }_{k,r_\nu}
        \leq 
        a_k(s-r)^{-b_k} \norm{ e }_{k+l,s}.
    \end{align*}
Therefore we may apply \nameref{Lemma:C} and conclude that the sequence $\psi_{\nu}$ converges uniformly on $X_r$ for all $C^k$-norms to a diffeomorphism $\psi\in\mathrm{Diff}_{X_r}$, which satisfies
    \[
        \norm{ \psi }_{k,r}
        \leq 
        a_k(s-r)^{b_k} \norm{ e }_{k+l,s}.
    \]
Since each $\psi_\nu\in \pgrp(X_0)$, and since $\pgrp(X_0)$ is closed in $\diff_X(X_0)$, it follows that $\psi\in \pgrp(X_0)$. 

Finally, we show that $\psi$ and $e$ are compatible on $M_r$ and that $e\cdot\psi=0$. 
Since 
    \[
        e_{\nu}|_{M_r}
        =
        (e_{\nu-1} \cdot \phi_\nu)|_{M_r}
        =
        (e_{\nu-2} \cdot (\phi_{\nu-1} \circ \phi_\nu))|_{M_r}
        =
        \ldots
        =
        (e \cdot \psi_{\nu})|_{M_r},
    \]
we have that $e_{\nu}(x)=(e\cdot \psi_{\nu})(x)$ for all $x\in M_r$. 
Hence $\psi_{\nu}(e_{\nu}(x))$ belongs to the compact set $e(M_s)$. On the other hand, $e_{\nu}(x)$ converges to $0_x$, and $\psi_{\nu}$ converges to $\psi$ on $X_r$, and therefore $\psi(0_x)\in e(M_s)$. 
So $\psi(M_r)\subset e(M_s)$, which implies that $e\cdot\psi=0$.
 \end{proof}
\noindent

\newpage
\section{Proofs of the Technical Lemmas}\label{sect:lemmas}
\etocsettocdepth{2}
\localtableofcontents

\noindent
{\color{white}This line is intentionally left blank.}
\newline
In this section we prove \nameref{Lemma:A}, \nameref{Lemma:B} and \nameref{Lemma:C}, which were used in the proof of the \nameref{Main Theorem}.
We will prove three different versions of these Lemmas: local versions, global versions for compactly supported functions, and versions for shrinking domains with corners.
These settings correspond to three of the subsections. After the first subsection there is an intermezzo. Here we prove several lemmas that are useful for closed pseudogroups, using the methods developed here.

\subsubsection{The Constant $\theta$}
\label{sec:constant_theta}

In this section we will prove several inequalities, which hold for maps that are `small enough'.
The various constants measuring this `smallness' will all be denoted by the same symbol $\theta>0$, so one should think about $\theta>0$ as the minimum of several constants appearing in the following sections. 
A more precise statement is the following:
\begin{quote}
    \emph{There exists a $\theta>0$ such that all results of this section hold.}
\end{quote}
This constant depends on the operator $Q$, on the shrinking domains, on the choice of norms and on local trivializations.

\subsubsection{The Symbol $\lesssim$}
\label{sec:symbol_lesssim} 

We uphold the following notation in order to simplify estimates.
Let $\mathcal{A}$ be a set, and let 
    $
        p, q:\mathcal{A}\to [0,\infty)
    $
be two maps.
The expression:
    \[
        p(f) \lesssim q(f),
        \quad\forall\;
        f\in \mathcal{A}
    \]
means that there exists a constant $c\geq 0$ such that
    \[
        p(f)\leq cq(f),
        \quad\forall\;
        f\in \mathcal{A}.
    \]

\subsection{Local Versions}

The local versions of \nameref{Lemma:B} and \nameref{Lemma:C} presented here involve many explicit computations and estimates of partial derivatives (and with that, the heart of the matter), whereas the subsequent sections deal with extending these results to different global settings. There is no local version of \nameref{Lemma:A}, as it is dealt with directly in the later sections (using some estimates from this section). 

First we discuss convenient notation to work with partial derivatives, and settle on a choice of $C^k$-norms.
Then we recall the so-called interpolation inequalities, and prove a variant that is often easier to use.
Then we prove estimates for the composition of two maps, and for the inverse of a diffeomorphism. Finally we prove the local versions of Lemma B and C.

\subsubsection{PDOs for Compositions}

For multi-indices
    $
        a=(a_1,\ldots,a_m)\in \N^m
    $
and $b \in \N^m$ we uphold the following notation:
\begin{itemize}

\item
    $
        |a| \defeq a_1+\ldots +a_m.
    $

\item
    $
        a! \defeq a_1! \cdots a_m!.
    $

\item
    $
        a \leq b 
        \Longleftrightarrow
        a_i \leq b_i
        \;\forall\; i.
    $

\item
    $
        {a \choose b}
        \defeq
        {a_1 \choose b_1}
        \cdots
        {a_m \choose b_m}.
    $

\item
    $
        \partial^a \defeq
        \tfrac{ \partial^{a_1} }{ \partial x_1^{a_1} }
        \cdots
        \tfrac{ \partial^{a_m} }{ \partial x_m^{a_m} }
    $
and 
    $
        D^{a} 
        \defeq 
        \tfrac1{a!}
        \partial^a.
    $

\end{itemize}

The last two expressions are partial differential operators acting on smooth functions defined on open subsets of $\R^m$ with values in $\R^n$.

\begin{lemma}
$D^a$ satisfies the following multiplicativity property for $f, g: \R^m \to \R$:
    \begin{equation}
    \label{eqn:D-product}
        D^a(f\cdot g)
        =
        \sum\nolimits\nolimits_{b\leq a}
        D^{b}(f) \cdot D^{a-b}(g).
    \end{equation}
\end{lemma}

In this paragraph we will state some formulas related to the partial derivatives of the composition of two smooth maps.
To keep these formulas compact and readable, we will first introduce a family of auxiliary partial differential operators. 
Namely, for multi-indexes $a\in\N^m$ and $b\in \N^n$ with $|a|\geq |b|\geq 1$, consider the operator 
    \[
        P^{a}_{b}:
        C^\infty(U, \R^n) \longto C^\infty(U),
        \quad U\sub \R^m,
    \]
which sends a smooth function
    $
        v
        =
        (v_1, \ldots, v_n) \in C^\infty(U, \R^n)
    $
to the smooth function
    \[
        P^a_b(v)
        \defeq 
        \sum\nolimits\nolimits_{(a^{i,j})}
        \prod\nolimits_{i=1}^n
        \prod\nolimits_{j=1}^{b_i}
        D^{a^{i,j}}(v_i),
    \]
where the sum runs over all decompositions $(a^{i,j})$ with
    \[
        \sum\nolimits\nolimits_{i=1}^n
        \sum\nolimits\nolimits_{j=1}^{b_i}a^{i,j}
        =a
    \]
and where $a^{i,j}\in \N^m$ with $|a^{i,j}|\geq 1$.
We extend this definition to $|b|=0$ by
    \[
        P^a_0(v) \defeq 
        \left\{
        \begin{array}{ll}
            1, & \text{if } a=0,
            \\
            0, & \text{otherwise.}
        \end{array}
        \right.
    \]

\begin{lemma}
For open sets $U\sub \R^n$ and $V\sub \R^m$, and smooth maps $u: U\to \R$ and $v: V \to U$, we find the formula
    \begin{equation}
    \label{eqn:composition}
        D^a(u\circ  v)
        =
        \sum\nolimits\nolimits_{|b| \leq |a|}(D^bu)\circ  v\cdot  P^{a}_b(v).
    \end{equation}
\end{lemma}
This formula is easily checked, e.g.\ it suffices to check it for $u$ being a polynomial.

\begin{lemma}
For an open set $U\sub \R^m$ and smooth maps $v,w : U \to \R$, we find the formula
    \begin{equation}
    \label{eqn:P-sum}
        P^a_b(v+w) 
        =
        \sum\nolimits\nolimits_{b' \leq b}
        {b \choose b'}
        \sum\nolimits\nolimits_{a' \leq a}
        P^{a'}_{b'}(v) \cdot P^{a-a'}_{b - b'}(w).
    \end{equation}
\end{lemma}

\subsubsection{Norms on Smooth Functions}

Let $K\sub \R^m$ be a compact set with dense interior:
    \[
        K=\overline{\mathrm{int}(K)}.
    \]
\begin{definition}
A function on $K$ with values in $\R^n$ (or any manifold) is said to be smooth if it extends to a smooth function on a neighborhood of $K$.
The space of all such maps is denoted by $C^{\infty}(K,\R^n)$.
\end{definition}

The condition that $K$ has dense interior insures that the operators $D^a$  for $a\in \N^m$ are defined on $C^{\infty}(K,\R^n)$.
For $k\geq 0$ we introduce standard $C^k$-norms:
    \[
        |f|_{k,K} \defeq 
        \sup\nolimits_{x\in K}
        \Big(
            \sum\nolimits\nolimits_{|a|\leq k}
            |D^a f(x)|^2
        \Big)^{1/2},
        \quad f\in C^{\infty}(K,\R^n).
    \]
When the domain of $K$ can be understood from the context we will simplify notation to 
    \[
        |f|_k \defeq |f|_{k,K}.
    \]

\subsubsection{Interpolation Inequalities}

Let $B\sub \R^m$ be an closed ball.
The $C^k$-norms on $C^{\infty}(B,\R^n)$ satisfy the following type of interpolation inequalities:
    \[
        |f|_{j}^{k-i}
        \leq 
        c_{i,j,k}
        |f|_{i}^{k-j}
        |f|_{k}^{j-i},
        \quad
        \textrm{for all}
        \quad
        i \leq j\leq k,
    \]
with constants depending only on the indexes $i, j, k$ and the dimensions $m, n$.
A proof of this result can be found in for example \cite{Conn85}.
Note that, with the notation of \nameref{sec:symbol_lesssim}, the interpolation inequalities become
    \[
        |f|_j^{k-i} \lesssim |f|_{i}^{k-j}|f|_{k}^{j-i}, 
        \quad\forall\; 
        f\in C^{\infty}(B,\R^n).
    \]

Throughout this section we work exclusively with maps defined on closed balls. The reasons for doing so are mostly precisely because the interpolation inequalities hold, and on a few occasions we use that they are convex.
Interpolation inequalities hold for more general subsets; for example, they hold on $C^{\infty}(K,\R^n)$ for a compact manifold with corners $K\sub \R^m$ of codimension zero (compactness can be dropped if one requires all derivatives to be bounded). 

\subsubsection{Generalized Interpolation Inequalities}

The interpolation inequalities will be used often in the following form (see also \cite{Conn85} for a different version of this lemma). Fix an integers $N\geq 0$ and $m_i, n_i \geq 0$ for $1 \leq i \leq N$, and consider smooth maps 
    \[
        f_i \in C^\infty(B_i, \R^{n_i}),
        \quad
        B_i \sub \R^{m_i},
    \]
where the $B_i\sub \R^{m_i}$ are fixed closed balls.
Suppose moreover that we are given indices 
    \[
        0 \leq q_i \leq Q_i
    \]
for $1\leq i \leq N$, and we are interested in bounding the product of the high index norms $|f_i|_{Q_i}$ by an expression involving (mostly) the low index norms $|f_i|_{q_i}$.

\begin{lemma}
\label{lemma:extended_interpolation}
If $M\geq 0$ is an integer such that
    \[
        (Q_1-q_1) + \ldots + (Q_N-q_N) \leq M,
    \]
then the following inequality holds:
    \[
        |f_1|_{Q_1} \cdots |f_N|_{Q_N}
        \lesssim 
        \sum\nolimits\nolimits_{j=1}^N 
        \big(
            |f_{1}|_{q_1}
            \cdots 
            |\widehat{f_j}|_{q_j}
            \cdots  
            |f_N|_{q_N}
        \big)
        |f_j|_{q_j+M},
    \]
for all $f_i\in C^{\infty}(B_i,\R^{n_i})$ with $1\leq i\leq N$.
Here the circumflex means one should exclude that factor from the product.
\end{lemma}

\begin{proof}
By increasing $Q_N$ we may assume that 
    \[
        (Q_1-q_1)+\ldots+(Q_N-q_N)=M.
    \]
By multiplying the interpolation inequalities
    \[
        |f_i|_{Q_i}
        \lesssim 
        |f_i|_{q_i}^{1-\frac{Q_i-q_i}{M}}
        |f_i|_{q_i+M}^{\frac{Q_i-q_i}{M}},
    \]
and writing $1-(Q_i-q_i)/M = \sum\nolimits_{j\neq i}(Q_j-q_j)/M$, we obtain
    \[
        \prod\nolimits_{i=1}^N
        |f_i|_{Q_i}
        \lesssim 
        \prod\nolimits_{i=1}^N
        \Big(
            \big(
                |f_{1}|_{q_1}
                \ldots 
                |\widehat{f_i}|_{q_i}
                \ldots  
                |f_N|_{q_N}
            \big)
            |f_i|_{q_i+M}
        \Big)^{\frac{Q_i-q_i}{M}}.
    \]
The result follows by applying the trivial inequality:
    \[
        x_1^{\lambda_1}\ldots x_N^{\lambda_N}
        \leq 
        x_1+\ldots+x_N,
    \]
where $0\leq x_1,\ldots,x_N$, and where $\lambda_1,\ldots,\lambda_N\in [0,1]$ are such that 
    \[
        \lambda_1+\ldots+\lambda_N=1.
        \qedhere
    \]
\end{proof}

\subsubsection{Estimates for Compositions}

Next we prove estimates for the composition of two maps (other versions of these results can be found in \cite{Conn85}). The estimates in the Lemma below will form the backbone for all other Lemmas in this chapter. They are used explicitly in Lemmas \ref{lemma:inverse}, \ref{lemma:action_local},  \ref{lemma:local_flow_action}, \ref{lemma:infinite_comp}, \ref{lemma:time_independent_flows}, \ref{lemma:cont_action}, \ref{lemma:comparison}, \ref{lemma:A_global}, and \ref{lemma:global_action}.

\begin{lemma}
\label{lemma:local comp}
Fix closed balls
    $
        B\sub \R^m
    $
and
    $
        A\sub \R^m\times\R^n
    $
with $B\times \{0\}\sub \mathrm{int}(A)$, and let
    \[
        \mathrm{i}=\id\times 0: \R^m\longto \R^{m}\times \R^n
    \]
be the standard embedding of $\R^m$ into $\R^m \times \R^n$.
For all
    \[
        f\in C^{\infty}(B,\R^{m}\times\R^n)
        \quad\text{and}\quad
        g\in C^{\infty}(A,\R^l),
    \]
if $|f|_0<\theta$, then $(\mathrm{i}+f)(B)\sub A$ and
    \begin{align*}
        |g\circ (\mathrm{i}+f)|_{k} 
        &\lesssim  
        |g|_{k}(1+|f|_{k}),
        \tag{a}
        \\
        |g\circ (\mathrm{i}+f) - g \circ  \mathrm{i}|_{k} 
        &\lesssim  
        |g|_{k+1}|f|_{k},
        \tag{b}
        \\
        |g\circ  (\mathrm{i}+f)-g\circ  \mathrm{i}-\dd g(\mathrm{i})f|_k
        &\lesssim 
        |g|_{k+2}|f|_0|f|_{k},
        \tag{c}
\intertext{and if in addition $|f|_1<\theta$, then:}
        |g\circ (\mathrm{i}+f)|_{k} 
        &\lesssim  
        |g|_{k}+|g|_1|f|_{k},
        \tag{d}
        \\
        |g\circ (\mathrm{i}+f)-g\circ  \mathrm{i}|_{k} 
        &\lesssim  
        |g|_{k+1}|f|_{0}+|g|_0|f|_{k+1},
        \tag{e}
\end{align*}
where $\dd g(\mathrm{i})f$ denotes the differential of $g$ at $\mathrm{i}$ along $f$, i.e.\
    \[
        \dd g(\mathrm{i})f
        \defeq
        \tfrac{\dd }{\dd \epsilon}\big|_{\epsilon=0}
        g(\mathrm{i}+\epsilon f).
    \]
\end{lemma}

\begin{remark}
There are versions of inequalities (a), (b) and (c) for higher degrees in the following sense.
Let $T_{d}(g,f)$ be the $d$-th Taylor polynomial in $f$ of the composition $g\circ (\mathrm{i}+f)$.
By the same steps as in the proofs below, one can show that for $|f|_0<\theta$,
    \[  
        |g\circ  (\mathrm{i}+f)-T_{d}(g,f)|_k
        \lesssim 
        |g|_{k+d}|f|_0^{d-1}|f|_{k}.
    \]
However, we will only need the inequalities listed above.
\end{remark}

\begin{proof}
If $\theta$ is smaller than the distance between $B\times \{0\}$ and $(\R^m\times\R^n)\backslash A$, then clearly 
    \[
        (\mathrm{i}+f)(B)\sub A.
    \]
Note that $P^{a}_b(\mathrm{i})\neq 0$ if and only if $b=(a,0)$. In that case, the only decomposition $(a^{i,j})$ of $a$ that appears in $P^a_{(a,0)}(\mathrm{i})$ is the one where $a^{i,j}=1$ for all $i,j$, so that $P^a_{(a,0)}(\mathrm{i})=1$.

Let $k\geq 0$, and let $a\in\N^m$ be a multi-index with $|a|= k$.
By applying formulas (\ref{eqn:composition}) and (\ref{eqn:P-sum}) in succession we obtain
    \begin{align}
        &D^a(g\circ (\mathrm{i}+f))
        =
        \sum\nolimits\nolimits_{|b|\leq k}
        (D^bg)\circ (\mathrm{i}+f)
        \cdot
        P^{a}_b(\mathrm{i}+f)
        \nonumber
        \\
        &\qquad =
        \sum\nolimits\nolimits_{|b|\leq k}
        \sum\nolimits\nolimits_{a'\leq a}
        \sum\nolimits\nolimits_{b'\leq b}
        {b \choose b'}
        (D^bg) \circ (\mathrm{i}+f)
        \cdot
        P^{a'}_{b'}(f)
        \cdot
        P^{a-a'}_{b-b'}(\mathrm{i}).
        \nonumber
\intertext{Then by applying the above observation about $P^a_b(\mathrm{i})$ we obtain}
        &\qquad=
        \sum\nolimits\nolimits_{a'+c=a}
        \sum\nolimits\nolimits_{|b'| \leq |a'|}
        {b'+(c,0) \choose b'}
        (D^{b'+(c,0)}g) \circ  (\mathrm{i}+f)
        \cdot
        P^{a'}_{b'}(f)
        \nonumber
        \\
        &\qquad=
        (D^{(a,0)}g) \circ (\mathrm{i}+f)+T_a,
        \label{eqn:comp1}
    \end{align}
where $T_a$ represents the part of the sum where $a'\neq 0$.
The explicit description of $P^{a'}_{b'}$ for a term in $T_a$ implies the following estimate:
    \begin{equation}
    \label{ineq:zbrr}
        |(D^{b'+(c,0)}g)\circ (\mathrm{i}+f)P^{a'}_{b'}(f)|_0
        \lesssim 
        \sum\nolimits |g|_{k+p-s}|f|_{j_1}\cdots |f|_{j_{p}},
    \end{equation}
where $1\leq p\defeq |b'|\leq s\defeq |a'|$ and $1\leq j_1,\ldots, j_{p}$ are such that $j_1+\ldots+j_{p}=s$.

Assume now that $|f|_1< \theta$.
Applying Lemma \ref{lemma:extended_interpolation} to the right-hand side of (\ref{ineq:zbrr}) with
    \[
        q_1=\ldots=q_{p+1}=1,
        \quad
        Q_1=j_1, \ldots, Q_{p}=j_{p},
        Q_{p+1}=k+p-s,
    \]
    \[
        (Q_1-q_1)+\ldots+(Q_{p+1}-1)= k-1,
    \]
we obtain
    \[
        |g|_{k+p-s}|f|_{j_1}\cdots |f|_{j_{p}}
        \lesssim 
        |f|_1^{p-1}(|g|_k|f|_1+|g|_1|f|_k).
    \]
Since $|f|_1\leq \theta$, this gives
    \begin{equation}
    \label{ineq:T_a}
        |T_a|_{0}
        \lesssim 
        |g|_{k}|f|_{1}+|g|_{1}|f|_{k}.
    \end{equation}
The other term in (\ref{eqn:comp1}) is $\lesssim |g|_k$, hence inequality (d) of the Lemma follows.

Note that, since $A$ is convex, 
    \[
        (\mathrm{i}+tf)(B)\sub A,
        \quad\forall\; t\in [0,1].
    \]
Therefore, using formulas (\ref{eqn:comp1}) and (\ref{eqn:composition}) we can write
    \begin{align}
        D^a(g\circ (\mathrm{i}+f)-g\circ  \mathrm{i})-T_a
        &=
        (D^{(a,0)}g)\circ  (\mathrm{i}+f)-(D^{(a,0)}g)\circ  \mathrm{i}
        \nonumber
        \\
        &=
        \sum\nolimits\nolimits_{i=1}^{m+n}
        f_i
        \int_0^1
            \big(
                \frac{\partial}{\partial x_i} 
                \circ  
                D^{(a,0)} g
            \big)
            \circ 
            (\mathrm{i}+tf)
        \,\dd t.
        \label{equation:integral}
    \end{align}
For $|f|_1<\theta$, using formulas (\ref{equation:integral}) and (\ref{ineq:T_a}), and Lemma \ref{lemma:extended_interpolation}, we obtain inequality (e):
    \begin{align*}
        |g\circ (\mathrm{i}+f)-g\circ  \mathrm{i}|_{k}
        &\lesssim 
        |g|_{k+1}|f|_0+|g|_{k}|f|_{1}+|g|_{1}|f|_{k}
        \\
        &\lesssim
        |g|_{k+1}|f|_0+|g|_{0}|f|_{k+1}.
    \end{align*}

From now on we will only assume that $|f|_0<\theta$.
For $(a)$, applying Lemma \ref{lemma:extended_interpolation} to each term of formula (\ref{ineq:zbrr}) gives
    \[
        |g|_{k+p-s}|f|_{j_1}\cdots |f|_{j_{p}}
        \lesssim 
        |g|_{k+p-s} |f|_0^{p-1}|f|_p
        \lesssim 
        |g|_k|f|_k.
    \]
Since the other term in formula (\ref{eqn:comp1}) is $\lesssim |g|_k$, we obtain inequality (a):
    \[
        |g\circ  (\mathrm{i}+f)|_k
        \lesssim 
        |g|_k(1+|f|_k).
    \]

For $(b)$, by applying Hadamard's Lemma we can write
    \begin{equation}
    \label{Hadamard}
        g\circ  (\mathrm{i}+f)-g\circ  \mathrm{i}
        =
        \sum\nolimits\nolimits_{i=1}^{m+n}
        f_ih^i,
        \qquad
        h^{i} \defeq 
        \int_0^1
        \frac{\partial g}{\partial x_i} \circ  (\mathrm{i}+tf) \,\dd t
    \end{equation}
Applying inequality (a) of the Lemma to the composition $\tfrac{\partial}{\partial x_i} g\circ  (\mathrm{i}+tf)$ we obtain
    \begin{equation}
    \label{eq:brute_1}
        |h^{i}|_k
        \lesssim 
        |g|_{k+1}(1+|f|_k).
    \end{equation}
Let $a\in\N^m$ be a multi-index with $|a|=k$.
By applying $D^a$ to formula (\ref{Hadamard}) we obtain
    \begin{align*}
        D^a(g\circ  (\mathrm{i}+f)-g\circ  \mathrm{i})
        =
        \sum\nolimits\nolimits_{a' \leq a}
        \sum\nolimits\nolimits_{i=1}^{m+n}
        D^{a'}(f_i)
        \cdot
        D^{a-a'}(h^i).
    \end{align*}
Using this, inequality (\ref{eq:brute_1}), and Lemma \ref{lemma:extended_interpolation} we obtain inequality (b):
    \begin{align*}
        |g\circ  (\mathrm{i}+f)-g\circ  \mathrm{i}|_k
        &\lesssim 
        \sum\nolimits\nolimits_{j_1+j_2=k}
        |f|_{j_1}|g|_{j_2+1}(1+|f|_{j_2})
        \\
        &\lesssim
        |g|_{k+1}
        \sum\nolimits\nolimits_{j_1+j_2=k}
        |f|_{j_1}(1+|f|_{j_2})
        \\
        &\lesssim
        |g|_{k+1}
        \sum\nolimits\nolimits_{j_1+j_2=k}
        (|f|_{j_1}+|f|_0|f|_{j_1+j_2})
        \\
        &\lesssim
        |g|_{k+1}|f|_{k}.
\end{align*}

Inequality (c) is proven similarly: denoting 
    \[
        \delta(g,f)
        \defeq 
        g\circ  (\mathrm{i}+ f)-g\circ  \mathrm{i}-\dd g(\mathrm{i})f
    \]
and applying Hadamard's Lemma twice we obtain
    \[
        \delta(g,f)=
        \sum\nolimits\nolimits_{i,j=1}^{m+n}
        f_if_jh^{i,j},
        \qquad
        h^{i,j}\defeq 
        \int_0^1
            (1-t)\frac{\partial^2 g}{\partial x_i \partial x_j} \circ  (\mathrm{i}+tf) 
        \,\dd t
    \]
By applying inequality (a) to the integrand we obtain
    \begin{equation}
    \label{eq:brute}
        |h^{i,j}|_k
        \lesssim 
        |g|_{k+2}(1+|f|_k).
    \end{equation}
As in the proof of (b), by using the multiplicativity property of the operators $D^a$ as described in formula (\ref{eqn:D-product}), then estimate (\ref{eq:brute}), and finally Lemma \ref{lemma:extended_interpolation} we obtain (c):
    \[|
        \delta(g,f)|_k
        \lesssim 
        \sum\nolimits\nolimits_{j_1+j_2+j_3=k}
        |f|_{j_1}
        |f|_{j_2}
        |g|_{j_3+2}
        (1+|f|_{j_3})
        \lesssim
        |g|_{k+2}|f|_0|f|_{k}.
        \qedhere
    \]
\end{proof}

From the proof of the Lemma we also extract the following Corollary.
Here we refrain from using the symbol $\lesssim$ in order to make an estimate with a more precise description of the constant appearing in front of one of the terms.
This Corollary is exclusively needed for Lemma \ref{lemma:infinite_comp} (the local version of \nameref{Lemma:C}) and Lemma \ref{lemma:time_independent_flows} (the intermezzo on flows of time-dependent vector fields).

\begin{corollary}
\label{coro:local comp}
Let $B$ and $C$ be closed balls in $\R^m$ such that $B\sub \mathrm{int}(C)$.
There are constants $c_k\geq 0$, such that for all smooth maps
    \[
        f\in C^{\infty}(B,\R^{m})
        \quad \mathrm{and}\quad 
        g\in C^{\infty}(C,\R^l),
    \]
satisfying $|f|_1<\theta$ we have that $(\id+f)(B)\sub C$ and that
    \[
        |g\circ (\id+f)|_k
        \leq 
        |g|_{k} + c_k
        (|g|_{k}|f|_{1}+|g|_1|f|_k).
    \]
\end{corollary}

\begin{proof}
As seen in the proof of the previous Lemma (with $n=m$ and $C=A$), for all multi-indices $a\in\N^m$ with $|a|\leq k$ we can write
    \[
        D^a(g\circ  (\id+f))-D^a(g)\circ  (\id+f)
        = T_a.
    \]
Recall estimate (\ref{ineq:T_a}): 
    \[
        |T_a|_k\lesssim (|g|_{k}|f|_{1}+|g|_1|f|_k).
    \]
Therefore we find constants $c_k\geq 0$ with the property that for any $x\in B$ there exists $y\in A$ (namely $y=x+f(x)$) such that
    \[
        |D^a(g\circ  (\id+f))(x)|
        \leq 
        |D^a(g)(y)|+ c_k(|g|_{k}|f|_{1}+|g|_1|f|_k).
    \]
This clearly implies the Corollary.
\end{proof}

\subsubsection{Estimates for Inverses} 
Next we prove estimates for inverses of diffeomorphisms (see also \cite{Conn85} for a similar result). They are used explicitly in Lemmas \ref{lemma:flow} and \ref{lemma:action_local} (the local versions of \nameref{Lemma:B}), Lemma \ref{lemma:diffeo} (on smooth maps close to the identity), and also Lemmas \ref{lemma:pseudogroup_limits} and \ref{lemma:pseudogroup_closure}.

\begin{lemma}
\label{lemma:inverse}
Let $B$ and $C$ be two closed balls in $\R^m$ such that $B\sub \mathrm{int}(C)$.
Any
    \[
        \id+g:C\longrightarrow\R^m
        \quad\textrm{satisfying}\quad  
        |g|_{1,C}<\theta
    \]
is a diffeomorphism onto its image, its image contains $B$, and its inverse
    \[
        \id+f
        \defeq 
        (\id+g)^{-1}|_{B}:
        B\longto C,
    \]
satisfies
    \begin{equation}
    \label{ineq:norm_inverse}
        |f|_{k}\lesssim |g|_{k},
        \quad\forall\; k\geq 0.
    \end{equation}
\end{lemma}

\begin{proof}
Let $x,y\in C$ be such that $x+g(x)=y+g(y)$.
Then
    \[
        |x-y|
        =
        |g(x)-g(y)|
        \leq 
        \sum\nolimits\nolimits_{i=1}^m
        \int_0^1
            |D^i(g)(x+t(y-x))(x_i-y_i)|
        \,\dd t
        \leq 
        |g|_1|x-y|.
    \]
Hence for $0<\theta<1$ we have $x=y$, and therefore $\id+g$ is injective.
Next, for $\theta<1/2$ we have that, for $x\in C$ and $v\in \R^m$, $\dd_x g$ satisfies
    \[|
        \dd_x g(v)|
        \leq 
        \sum\nolimits\nolimits_{i=1}^{m}
        |D^i_x(g)||v_i|
        \leq 
        |g|_1|v|
        \leq 
        |v|/2.
    \]
Therefore $\id+\dd g$ is invertible, and therefore $\id+g$ is a diffeomorphism onto its image.
For later use, note that the operator norm of $(\id+\dd g)^{-1}$ satisfies:
    \begin{equation}
    \label{eq:op_norm}
        |(\id+\dd g)^{-1}|
        =
        \big|
            \sum\nolimits\nolimits_{n\geq 0} (-1)^n (\dd g)^{n}
        \big|
        \leq 
        \sum\nolimits\nolimits_{n\geq 0} 2^{-n}=2.
    \end{equation}
We write 
    \[
        (\id+f)\defeq (\id+g)^{-1}:(\id+g)(C)\longto C.
    \]
Then $f$ satisfies $f=-g\circ (\id+f)$, which gives inequality (\ref{ineq:norm_inverse}) for $k=0$.

We prove the general case by induction.
Consider a multi-index $a\in \N^m$, with $|a|=k$.
Note that $P^{a}_b(\id)=1$ if $a=b$ and $P^{a}_b(\id)=0$ if $a\neq b$.
Therefore, by the calculation from the proof of Lemma \ref{lemma:local comp} we have
    \begin{equation}
    \label{eqn:deriv_of_inverse}
        -D^a(f)
        =
        \sum\nolimits\nolimits_{a=a'+c}
        \sum\nolimits\nolimits_{|b'|\leq |a'|}
        {b'+c \choose b'}
        (D^{b'+c}g) \circ  (\id+f)
        \cdot
        P^{a'}_{b'}(f).
    \end{equation}
Let $\Sigma_a$ be the sum of terms involving only derivatives of $f$ of order strictly smaller than $k$.
By applying the inductive hypothesis (that $|f|_{j}\lesssim |g|_j$ for $j<k$) the terms in $\Sigma_a$ can be estimated by
    \[
        |\Sigma_a|_0
        \lesssim
        \sum\nolimits\nolimits_{0\leq p\leq s\leq k} 
        |g|_{k+p-s}
        |g|_{j_1}
        \cdots
        |g|_{j_{p}},
    \]
where $p\defeq |b'|\leq s\defeq |a'|$, $1\leq j_1,\ldots, j_{p}$, are such that $j_1+\ldots+j_{p}=s$.
Note that $k-p+s\geq 1$ (unless $k=s$ and $p=0$, but in that case $k=0$, which contradicts $k\geq 1$).
By applying Lemma \ref{lemma:extended_interpolation} we obtain
    \[
        |\Sigma_a|_0
        \lesssim 
        \sum\nolimits\nolimits_{1\leq p\leq k}
        |g|_1^{p}
        |g|_k
        \lesssim 
        |g|_k.
    \]
The only terms in the right-hand side of formula (\ref{eqn:deriv_of_inverse}) which involve derivatives of $f$ of degree $k$ are those where $a'=a$, $c=0$ and $|b'|=1$. Hence we can write
    \[
        D^a(f)
        +
        \sum\nolimits\nolimits_{i=1}^m
        \frac{\partial g}{\partial x_i} \circ (\id+f)D^a(f^i)
        =
        -\Sigma_a.
    \]
The left-hand side is the differential of $g$ applied to the vector $D^a(f)$. Therefore,
    \[
        D^a(f)  
        =
        -(\id+\dd g)^{-1}_{(\id+f)}\Sigma_a.
    \]
Using inequality (\ref{eq:op_norm}) we obtain that
    \[
        |D^a(f)|_0
        \lesssim 
        |\Sigma_a|_0
        \lesssim 
        |g|_k.
    \]
This implies that $f$ has bounded $k$-th derivatives, and shows the inductive hypothesis 
    \[
        |f|_k\lesssim |g|_k.
    \]

Finally, we show that for $\theta$ small enough we have $B\sub (\id+g)(C)$.
Let $\epsilon>0$ be smaller than the distance between $B$ and $\R^m\backslash C$.
Then, for any $x\in B$, the closed ball $B_{\epsilon}(x)$ of radius $\epsilon$ around $x$ lies in $C$.
Consider the function
    \[
        \lambda:B_{\epsilon}(x)\longrightarrow \R,
        \quad
        \lambda(y) \defeq |y+g(y)-x|^2.
    \]
For $\epsilon/2>\theta>0$, we have that $\epsilon>2|g|_0$.
Therefore, for $z\in \partial B_{\epsilon}(x)$, we have that:
    \[
        |z+g(z)-x|
        \geq 
        |x-z|-|g|_0
        = 
        \epsilon-|g|_0
        > 
        |g|_0
        \geq |g(x)|,
    \]
thus $\lambda(z)>\lambda(x)$.
This shows that $\lambda$ attains its minimum in $\mathrm{int}(B_{\epsilon}(x))$.
At such a minimum-point $y$, we have that:
    \[
        0
        =
        \dd_y\lambda
        =
        2\dd_y(\id+g)(y+g(y)-x).
    \]
Since $\dd_y(\id+g)$ is injective, $y+g(y)=x$.
Hence, $x\in (\id+g)(C)$.
\end{proof}

\subsubsection{Local version of Lemma B(1)}

The following is a local version of the first part of \nameref{Lemma:B}. It is used explicitly in Lemmas \ref{lemma:local_flow_action} and \ref{lemma:B_1_global}. In fact, it has been slightly extended beyond the scope of \nameref{Lemma:B} to allow for time-dependent vector fields, as needed for Lemmas \ref{lemma:time_independent_flows}, \ref{lemma:cont_action}, and \ref{lemma:closed_lie_algebra_sheaf}.

Let $C\sub \R^m$ be a closed ball. We will use the standard identification $TC\cong C\times \R^m$, which gives an identification between vector fields on $C$ and smooth maps $C\to \R^m$, i.e.\
    \[
        \gerX(C)\cong  C^\infty(C,\R^m).
    \]
Using this isomorphism we introduce $C^k$-norms $|\cdot|_{k}=|\cdot|_{k,C}$ on $\gerX(C)$.
Moreover, for a time-dependent vector field $v^t\in \gerX(C)$, with $t\in [0,1]$, we will write
    \[
        |v|_{k}\defeq \sup\nolimits_{t\in [0,1]}|v^t|_k,
    \]
and similar notations will be used for smooth families $f^t\in  C^\infty(C,\R^n)$.

\begin{lemma}
\label{lemma:flow}
Let $B$ and $C$ be two closed balls in $\R^m$ such that $B\sub \mathrm{int}(C)$. For any time dependent vector field $v^t\in \gerX(C)$, with $t\in [0,1]$, satisfying
    \[
        |v|_{1} < \theta,
    \]
the flow $\varphi^{t}_v\defeq \varphi^{t,0}_v$ of $v$ exists for all $0\leq t\leq 1$ as a map from $B$ to $C$:
    \[
        \varphi_v^{t}
        =
        \id+f^t:
        B\longto C.
    \]
Moreover, $f^t$ satisfies
    \[
        |f^t|_{k}\lesssim t|v|_k.
    \]
\end{lemma}

\begin{proof}
The flow satisfies the integral formula
    \begin{equation}
    \label{eqn:int formula}
        f^t 
        =
        \int_{0}^t
            v^{\tau}\circ(\id+f^{\tau}) 
        \,\dd\tau,
    \end{equation}
wherever both sides are defined.

For $x\in B$, the flow line $\varphi_v^{t}(x)$ either stays in $\mathrm{int}(C)$ for all $t\geq 0$, or there is a smallest $t_{0}\in (0,1]$ which lies on the boundary, i.e.\ $\varphi^{t_0}_v(x)\in \partial C$. 
If such a $t_0$ exists, then (\ref{eqn:int formula}),
    \[
        |x-\varphi_{v}^{t_0}(x)|=|f^{t_0}(x)|
        \leq
        t_0|v|_{0}<t_0\theta.
    \]
Thus taking $\theta>0$ smaller than the distance between $B$ and $\R^m\backslash C$, we obtain the contradiction that $t_0>1$. Hence the flow exists up to time one as a map from $B$ to $C$.

Equality (\ref{eqn:int formula}) implies that $|f^t|_{0}\leq t|v|_{0}$.
Now we prove by induction on $k\geq 1$ that the partial derivatives of order $k$ of $f^t$ satisfy 
    \[
        |f^t|_k\lesssim t|v|_k.
    \] 
Let $a\in \N^m$ be a multi-index with $|a|\leq k$. As in the proof of Lemma \ref{lemma:inverse}, apply $D^a$ to integral formula (\ref{eqn:int formula}) to obtain
    \begin{equation}
    \label{eqn:deriv_of_flow}
        D^a(f^t)
        =
        \sum\nolimits_{a=a'+c}
        \sum\nolimits_{|b'|\leq |a'|}
        {b'+c \choose b'}
        \int_0^t
            (D^{b'+c}v^{\tau})
            \circ
            (\id+f^{\tau})
            P^{a'}_{b'}(f^{\tau})
        \,\dd\tau.
    \end{equation}
Let $\Sigma_a$ denote the terms in this sum which involve only derivatives of $f^\tau$ of order strictly less than $k$. As in the proof of Lemma \ref{lemma:inverse}, by applying the inductive hypothesis to these terms one obtains that
    \[
        |\Sigma_a|_0
        \lesssim 
        t|v|_k.
    \]
The terms in formula (\ref{eqn:deriv_of_flow}) containing derivative of order $k$ are
    \[
        \sum\nolimits_{i=1}^m
        \int_0^t
            (D^iv^{\tau}) 
            \circ 
            (\id+f^\tau)
            D^{a}(f_i^{\tau})
        \,\dd\tau.
    \]
We conclude that there is a constant $c_k\geq 0$ such that, for all $x\in B$,
    \[
        |D^a(f^t)_x|
        \leq 
        c_k t |v|_k
        +
        |v|_1
        \int_0^t
            |D^a(f^{\tau})_x|
        \,\dd\tau.
    \]
Next we apply the following version of Gr\"onwall's inequality: if $u:[0,1]\to\R$ is a continuous function, and there are continuous functions $\alpha,\beta:[0,1]\to \R$ with $\beta(t)\geq 0$ such that
    \[
        u(t)
        \leq 
        \alpha(t)
        +
        \int_0^{t}
            \beta(\tau)
            u(\tau)
        \,\dd\tau,
        \quad\forall\;
        t\in [0,1],
    \]
then $u$ satisfies 
    $
        u(t)
        \leq 
        \alpha(t)
        +
        \int_0^t
            \alpha(\tau)
            \beta(\tau)
            e^{\int_\tau^t
                \beta(s)
            \,\dd s}
        \,\dd\tau.
    $
In our case it gives
\[|D^a(f^t)_x|\leq c_kt|v|_{k}+\int_0^tc_k\tau|v|_{k}|v|_{1}e^{(t-\tau)|v|_1}\dd \tau.\]
Hence $|f^t|_{k}\lesssim t|v|_k$, and this concludes the proof.
\end{proof}

\subsubsection{Local Version of Lemma B(2,3)}

In this section we prove two lemmas. The first Lemma deals with local diffeomorphisms of the trivial bundle
    $
        \underline\R^n 
        \defeq
        \R^m\times\R^n
    $
over
    $
        \R^m.
    $
We show that any map $C^1$-close to the identity is a local diffeomorphisms and is compatible with any section $C^1$-close to the zero section. Here we mean compatibility in the sense of definition \ref{def:compatibility} (extended to closed balls).
This Lemma is used explicitly in Lemmas \ref{lemma:local_flow_action} and \ref{lemma:global_action}.
The second Lemma is a local version of the second and third parts of \nameref{Lemma:B}. It is only used explicitly in Lemma \ref{lemma:B_23_global}.

Let $B, C\sub \R^m$ and $A\sub \R^m\times \R^n$ be closed balls. We identify local sections of the trivial bundle
    $
        \underline\R^n 
    $ 
with smooth maps $\R^m\to \R^n$, i.e.\
    \[
        \Gamma_C(\underline\R^n)
        \simeq
        C^\infty(C, \R^n).
    \]
Using this isomorphism we introduce $C^k$-norms
    $
        |\cdot|_k \defeq |\cdot|_{k,C}
    $
on $\Gamma_C(\underline\R^n)$. By abuse of notation, we also use the $C^k$-norms 
    $
        |\cdot|_k = |\cdot|_{k,B}
    $
on $\Gamma_B(\underline\R^n)$.

\begin{lemma}
\label{lemma:action_local}
Consider closed balls $B,C\sub \R^m$ and $A\sub \R^m\times \R^n$ such that
    \[
        B\times \{0\}
        \sub 
        \mathrm{int}(A),
        \quad 
        A\sub \mathrm{int}(C)\times \R^n.
    \]
For any smooth maps
    $
        e: C\to \R^n
    $
and
    $
        \phi
        =
        \id+g:A\to\R^{m}\times\R^n
    $
satisfying
    \[
        |e|_{1}<\theta,
        \quad
        |g|_{1}<\theta,
    \]
we have that $\phi$ is a diffeomorphism onto its image, and its right action on $e$ is defined on $B$:
    \[
        e\cdot\phi: B\longto \R^n.
    \]
Moreover, they satisfy the inequalities
    \begin{equation}
    \label{ineq:action_local1}
        |e\cdot \phi|_{k}
        \lesssim 
        |e|_{k}+|g|_{k}.
    \end{equation}
\end{lemma}%

\begin{proof}
By Lemma \ref{lemma:inverse} we can choose $\theta>0$ such that $\phi$ is a diffeomorphism onto its image. 
Next, choose $\theta>0$ such that $|g|_0$ is smaller than the distance between $A$ and $(\R^m\backslash \mathrm{int}(C))\times \R^n$. 
This implies that
    $
        \phi(A)\sub C\times \R^n,
    $
which insures that the following map is well-defined:
    \[
        \psi\defeq\id+(0,h):A\longto \R^m\times \R^n,
        \quad
        h(x,y)\defeq g_2(x,y)-e(x+g_1(x,y))+e(x),
    \]
where $g=(g_1,g_2)$. Also define 
    \[
        \widetilde{e}:A\longto \R^n,
        \quad
        \widetilde{e}(x,y)\defeq e(x).
    \]
From the definition of $h$ we obtain the estimate
    \[
        |h|_k
        \lesssim 
        |g|_k
        +
        |\widetilde{e}\circ(\id+g)|_k
        +
        |e|_k.
    \]
Since $|\widetilde{e}|_k=|e|_k$, $|g|_1<\theta$, and $|e|_1<\theta$, by using Lemma \ref{lemma:local comp}(d) we obtain
    \begin{equation}
    \label{estim:h}
        |h|_k\lesssim |g|_k+|e|_k.
    \end{equation}
In particular, 
    $
        |h|_1\lesssim (|g|_1+|e|_1)/2< \theta,
    $
and so after shrinking $\theta$ we may apply Lemma \ref{lemma:inverse} and conclude that $\psi$ is a diffeomorphism onto its image. 
Moreover, by that Lemma we may assume that $B\times D\sub \psi(A)$, where $D\sub \R^n$ is a closed ball around $0$. Therefore after shrinking $\theta$, we have that
    \[
        (\id\times e)(B)\sub B\times D \sub \psi(A).
    \]
Note that the inverse of $\psi$ takes the form
    $
        \psi^{-1}=\id+(0,l)
    $
for some map 
    $
        l: B\times D \longto A.
    $

Next we claim that for every $x\in B$ there exits a unique $z\in C$ and a unique $y\in \R^n$ which satisfy the equation
    \begin{equation}
    \label{eq:defining}
        \phi(x,y)=(z,e(z)).
    \end{equation}
This equation is equivalent to:
    \[
        z=x+g_1(x,y),
        \quad
        e(z)=y+g_2(x,y),
    \]
and also to
    \[
        z=x+g_1(x,y),
        \quad 
        \psi(x,y)=(x,e(x)).
    \]
Hence the claim follows because $\psi$ is a diffeomorphism. Moreover, the solutions depend smoothly on $x$, as they are given by $y:B\longto \R^n$ and $z:B\longto C$ given by 
    \begin{align*}
        y(x)
        &=
        e(x)+l(x,e(x)), 
        \\
        z(x)
        &=
        x+g_1(x,e(x)+l(x,e(x))).
    \end{align*}
The defining equation (\ref{eq:defining}) of these maps gives us
    \begin{equation}
    \label{eq:defining1}
        \phi^{-1}\circ (\id\times e)\circ z
        =
        \id\times y.
    \end{equation}
By projecting onto the first component we obtain
    \[
        (\phi^{-1})_e\circ z
        = \id, 
        \qquad
        (\phi^{-1})_e
        \defeq \mathrm{pr}_1\circ \phi^{-1}\circ (\id\times e).
    \]
This shows that $z$ is a diffeomorphism onto its image $\Omega\defeq z(B)$ with inverse
    \[
        z^{-1}
        =
        (\phi^{-1})_e:
        \Omega\diffto B.
    \]
This implies that the action of $\phi$ on $(\id\times e)$ is well-defined on $B$ and is given by
    \[
        e\cdot \phi
        =
        \mathrm{pr}_2
        \circ 
        \phi^{-1}
        \circ 
        (\id\times e)
        \circ 
        (\phi^{-1})^{-1}_e
        =
        \mathrm{pr}_2
        \circ 
        \phi^{-1}
        \circ 
        (\id\times e)
        \circ 
        z
        \stackrel{(\ref{eq:defining1})}{=}
        y.
    \]
Finally, we prove inequality (\ref{ineq:action_local1}). 
By Lemma \ref{lemma:inverse} for $\psi$ and by (\ref{estim:h}) we obtain
    \begin{equation}
    \label{estim:l}
        |l|_k \lesssim |g|_k+|e|_k.
    \end{equation}
Now recall that 
    $
        e\cdot \phi=y=e+l(\id\times e).
    $ 
So by writing 
    $
        \id\times e
        =
        \id\times 0 + 0\times e
    $ 
we can apply Lemma \ref{lemma:local comp}(d), and by using also (\ref{estim:l}) we obtain the required estimate:
    \[
        |e\cdot \phi|_k
        \lesssim 
        |e|_{k}+|l|_k+|l|_1|e|_k
        \lesssim 
        |e|_k+|g|_k.
        \qedhere
    \]
\end{proof}

Next we prove the local versions of \nameref{Lemma:B}(2,3).

\begin{lemma}
\label{lemma:local_flow_action}
Consider closed balls $A$, $B$, and $C$ as in Lemma \ref{lemma:action_local}. 
For all vector fields $v\in \gerX(A)$ and all smooth maps $e:C\longto \R^n$ satisfying
    \[
        |v|_{1}<\theta,
        \quad 
        |e|_{1}<\theta,
    \]
the action of the time-one flow $\varphi_v$ of $v$ on $e$ is defined on $B$:
    \[
        e\cdot\varphi_v:
        B\longto \R^n.
    \]
Moreover, it satisfies the inequalities
    \begin{align*}
        |e\cdot\varphi_v|_{k}
        &\lesssim 
        |e|_{k}+|v|_{k},
        \\
        |e\cdot\varphi_v-e-\delta v|_{k}
        &\lesssim 
        (|v|_{0}+|e|_{0})
        |v|_{k+1}+|v|_{0}|e|_{k+1}.
    \end{align*}
\end{lemma}

\begin{proof}
Consider a smaller closed ball $A'\sub \mathrm{int}(A)$ such that still $B\times \{0\}\sub\mathrm{int}(A')$. 
For small enough $\theta$, by Lemma \ref{lemma:flow}, the flow of $v$ is defined for all $t\in [0,1]$ as a map
    \[
        \varphi_v^t
        =
        \id+f^t:
        A'\longto A,
    \]
and it satisfies 
    $
        |f^t|_k
        \lesssim 
        |v|_k.
    $ 
In particular, by choosing $\theta$ small enough, we may assume that $e$ and $\varphi_v=\id+f$ satisfy the assumptions of Lemma \ref{lemma:action_local}, but with $A'$ instead of $A$. 
We conclude that $e\cdot\varphi_v$ is defined on $B$, and that the first inequality holds:
    \[
        |e\cdot\varphi_v|_k
        \lesssim 
        |e|_k+|f|_k
        \lesssim |e|_k+|v|_k.
    \]
Now decompose 
    $
        f^t
        =
        (f^t_1,f^t_2)
    $ 
and 
    $
        v
        =
        (v_1,v_2).
    $
As in the proof of Lemma \ref{lemma:action_local}, the following family is well-defined for all $t\in [0,1]$:
    \[
        \psi^t
        \defeq
        \id+0\times h^t:
        A'\longto \R^m\times \R^n,
        \quad
        h^t(x,y)
        \defeq 
        f_2^t(x,y)-e(x+f_1^t(x,y))+e(x).
    \]
Note that $\psi^0=\id$ and that $\psi^t$ is a smooth isotopy of local diffeomorphisms. 
The inverse of $\psi^t$ will be denoted by
    \[
        \chi^t
        \defeq 
        (\psi^t)^{-1}, 
        \quad
        \chi^t=\id+0\times l^t.
    \]
Moreover, to shorten notation we write
    \[
        \epsilon_k\defeq |e|_k,
        \quad
        \nu_k\defeq |v|_k.
    \]
Lemma \ref{lemma:flow} implies that
    \begin{equation}
    \label{ineq:zaction0}
        |f^t|_k\lesssim \nu_k,
    \end{equation}
and therefore, by (\ref{estim:h}) and (\ref{estim:l}) in the proof of Lemma \ref{lemma:action_local}, we have
    \begin{equation}
    \label{ineq:zaction1}
        |h^t|_k
        \lesssim 
        \nu_k+\epsilon_{k}, 
        \quad 
        |l^t|_k\lesssim \nu_k+\epsilon_{k}.
    \end{equation}

Next we claim that the following inequality holds:
    \begin{equation}
    \label{ineq:zaction2}
        |\psi^1-\id-\dot{\psi}^0\circ \psi^1|_k
        \lesssim 
        (\epsilon_0+\nu_0)\nu_{k+1}
        +
        \nu_0\epsilon_{k+1}
    \end{equation}
For this, note that
    $
        \dot{\psi}^0
        =
        (0,v_2-\dd e\circ v_1),
    $
and therefore
    \[
        \psi^1-\id-\dot{\psi}^0\circ\psi^1
        =
        (0,h^1-(v_2-\dd e(v_1))\circ\psi^1).
    \]
We split this expression into four groups, which will be treated separately:
    \begin{equation}
    \label{eq:zaction}
        (f^1_2-v_2)
        -
        (\widetilde{e}\circ(\id+f^1)
        -
        \widetilde{e})
        -
        (v_2\circ \psi^1-v_2)
        +
        \dd e(v_1)\circ\psi^1,
    \end{equation}
where again $\widetilde{e}(x,y)=e(x)$. 
For the first term we use the integral formula (\ref{eqn:int formula}) to write
    \[
        f^1-v
        =
        \int_0^1
            (v\circ(\id+f^t)-v)
        \,\dd t.
    \]
Applying Lemma \ref{lemma:local comp}(e) to the integrand we obtain
    \[
        |f^1-v|_k
        \lesssim 
        \nu_0\nu_{k+1}.
    \]
For the second term of (\ref{eq:zaction}), since $|\widetilde{e}|_k=\epsilon_k$, we apply Lemma \ref{lemma:local comp}(e) again, and obtain
    \begin{equation}
    \label{ineq:zaction4}
        |\widetilde{e}\circ(\id+f^1)-\widetilde{e}|_k
        \lesssim 
        \nu_0\epsilon_{k+1}+\epsilon_0\nu_{k+1}
    \end{equation}
Also for the third term of (\ref{eq:zaction}) we apply Lemma \ref{lemma:local comp} (e), and then (\ref{ineq:zaction1}), and obtain
    \begin{equation*}
        |v_2\circ(\id+h^1)-v_2|_k
        \lesssim 
        (\nu_0+\epsilon_0)\nu_{k+1}+\nu_0\epsilon_{k+1}.
    \end{equation*}
For the fourth and last term of (\ref{eq:zaction}), we first evaluate the norm of $\dd e(v_1)$. 
Note that the expression $\dd e(v_1)$ is a bilinear map in $e$ and $v_1$. 
Therefore by applying the multiplicativity relation (\ref{eqn:D-product}) for the operators $D^a$ followed by Lemma \ref{lemma:extended_interpolation} we obtain
    \[
        |\dd e(v_1)|_k
        \lesssim 
        \sum\nolimits_{j=0}^k
        \epsilon_{j+1}\nu_{k-j}
        \lesssim 
        \epsilon_1\nu_k+\epsilon_{k+1}\nu_{0}
        \lesssim 
        \epsilon_0\nu_{k+1}+\nu_0\epsilon_{k+1}.
    \]
Finally, by using Lemma \ref{lemma:local comp}(d), Lemma \ref{lemma:extended_interpolation} and (\ref{ineq:zaction1}), we obtain
    \begin{align}
        |\dd e(v_1)\circ \psi^1|_k
        &\lesssim 
        |\dd e(v_1)|_k+|\dd e(v_1)|_1|h^1|_k
        \label{ineq:zaction6}
        \\
        &\lesssim 
        |\dd e(v_1)|_k(1+|h^1|_{1})+|\dd e(v_1)|_0|h^1|_{k+1}
        \nonumber
        \\ 
        &\lesssim
        (
            \epsilon_0\nu_{k+1}
            +
            \epsilon_{k+1}\nu_{0}
        )
        (1+\nu_1+\epsilon_1)
        +
        \epsilon_{1} \nu_{0}
        (\epsilon_{k+1} + \nu_{k+1})
        \nonumber
        \\
        &\lesssim
        (\nu_0+\epsilon_0)
        \nu_{k+1}
        +
        \nu_0\epsilon_{k+1}.
        \nonumber
\end{align}
Thus we conclude that (\ref{ineq:zaction2}) holds. 
For later use, let us mention the following inequality:
    \begin{equation}
    \label{ineq:zaction7}
        |\psi^1-\id-\dot{\psi}^0\circ \psi^1|_0
        \lesssim 
        \nu_0.
    \end{equation}
This follows from inspecting the terms in (\ref{eq:zaction}) once more: for the first and the third this clearly holds; the other two are bounded by $\nu_0\epsilon_1$ (for the second, see the proof of (\ref{ineq:zaction4}), and the fourth follows directly), and we have $\epsilon_1<\theta$.

Next we show that inequality (\ref{ineq:zaction2}) implies the following inequality:
    \begin{equation}
    \label{ineq:action6}
        |\chi^1-\id-\dot{\chi}^0|_k
        \lesssim 
        (\nu_0+\epsilon_0)\nu_{k+1}+\nu_0\epsilon_{k+1}.
    \end{equation}
By differentiating $\chi^t\circ\psi^t=1$ at $t=0$ we obtain that
    $
        \dot{\chi}^0+\dot{\psi}^0 = 0,
    $
and therefore
    \begin{equation}
    \label{ineq:action7}
        \chi^1-\id-\dot{\chi}^0
        =
        \chi^1-\id+\dot{\psi}^0
        =
        -(\psi^1-\id-\dot{\psi}^0\circ\psi^1)\circ \chi^1.
    \end{equation}
This, together with inequality (\ref{ineq:zaction7}), implies that
    \begin{equation}
    \label{ineq:zaction8}
        |\chi^1-\id-\dot{\chi}^0|_0\lesssim \nu_0.
    \end{equation}
To obtain (\ref{ineq:action6}), we apply Lemma \ref{lemma:local comp}(d), Lemma \ref{lemma:extended_interpolation}(to the second term), and then inequalities (\ref{ineq:zaction7}), (\ref{ineq:zaction1}), and (\ref{ineq:zaction2}) in order:
    \begin{align*}
        |\chi^1-\id-\dot{\chi}^0|_k
        &\lesssim 
        |\psi^1-\id-\dot{\psi}^0\circ \psi^1|_k
        +
        |l^1|_k
        |\psi^1-\id-\dot{\psi}^0\circ\psi^1|_1
        \\
        &\lesssim
        |\psi^1-\id-\dot{\psi}^0\circ \psi^1|_k
        (1+|l^1|_1)
        +
        |l^1|_{k+1}
        |\psi^1-\id-\dot{\psi}^0\circ\psi^1|_0
        \\
        &\lesssim
        (\nu_0+\epsilon_0)
        \nu_{k+1}+\nu_0\epsilon_{k+1}.
\end{align*}

Finally, by the proof of Lemma \ref{lemma:action_local} we have that
    \begin{align*}
        e\cdot \varphi_v^t
        &=
        e+l^t\circ(\id\times e)
        =
        \mathrm{pr}_2\circ\chi^t\circ(\id\times e),
        \\    
        \delta v
        &=
        \tfrac{\dd}{\dd t}\big|_{t=0}
        (e\cdot\varphi_v^t)
        =
        \mathrm{pr}_2\circ \dot{\chi}^0\circ(\id\times e).
    \end{align*}
Hence we have that
    \[
        e\cdot\varphi_v
        -
        e-\delta v
        =
        \mathrm{pr_2}
        \circ
        \left(
            \chi^1-\id-\dot{\chi}^0
        \right)
        \circ 
        (\id\times 0+ 0\times e).
    \]
By applying Lemmas \ref{lemma:local comp}(d) and \ref{lemma:extended_interpolation}, and inequalities (\ref{ineq:action6}) and (\ref{ineq:zaction8}), we obtain
    \begin{align*}
        |e\cdot\varphi_v-e-\delta v|_k
        &\lesssim 
        |(\chi^1-\id-\dot{\chi}^0)\circ (\id\times0+0\times e)|_k
        \\
        &\lesssim 
        |\chi^1-\id-\dot{\chi}^0|_k
        + 
        \epsilon_k|\chi^1-\id-\dot{\chi}^0|_1
        \\
        &\lesssim 
        |\chi^1-\id-\dot{\chi}^0|_k
        (1+\epsilon_1)
        + 
        \epsilon_{k+1}
        |\chi^1-\id-\dot{\chi}^0|_0
        \\
        &\lesssim 
        (\epsilon_0+\nu_0)
        \nu_{k+1}
        +
        \nu_0\epsilon_{k+1}.
        \qedhere
    \end{align*}
\end{proof}

\subsubsection{Local version of Lemma C}

In this section we prove a local version of \nameref{Lemma:C}.

\begin{lemma}
\label{lemma:infinite_comp}
Let $B, C \sub \R^m$ be closed balls such that $B\sub \mathrm{int}(C)$. 
For all sequences of smooth maps
    \[
        \id+f_{\nu}:C\longto \R^m, 
        \quad
        \nu\geq 1,
    \]
satisfying
    \[
        \sum\nolimits_{\nu\geq 1}|f_{\nu}|_{1}<\theta
        \quad\textrm{and}\quad
        \sum\nolimits_{\nu\geq 1}|f_{\nu}|_{k}<\infty,
        \quad\forall\; k\geq 0,
    \]
we have that the compositions
    \[
        \id+g_{\nu}
        \defeq 
        (\id+f_{1})\circ(\id+f_{2})\circ \ldots \circ (\id+f_{\nu})
    \]
are well-defined on $B$, and converge in all $C^k$-norms on $B$ to a smooth map
    \[
        \id+g_{\infty}
        \defeq 
        \lim_{\nu\to \infty}
        \big(
            \id+g_{\nu}|_B
        \big):
        B\longto C.
    \]
Moreover, it satisfies
    \[
        |g_{\infty}|_{k}
        \lesssim 
        \sum\nolimits_{\nu\geq 1}
        |f_{\nu}|_{k}.
    \]
\end{lemma}

\begin{proof}
Let $\theta$ be smaller than the distance between $B$ and $\R^m\backslash C$. 
We claim that
    \[
        (\id+f_{\nu-\mu})\circ\ldots\circ(\id+f_{\nu})(B)
        \sub C,
        \quad\forall\;
        0\leq \mu < \nu.
    \]
For suppose that this fails. Then there would be an $x_0\in B$ and a $0\leq \mu<\nu$ such that
    \[
        x_{\mu+1}
        \defeq 
        (\id+f_{\nu-\mu})\circ\ldots\circ(\id+f_{\nu})(x_0)
        \notin 
        C,
    \]
while for $0\leq i< \mu$ we would have that
    \[
        x_{i+1}
        \defeq 
        (\id+f_{\nu-i})\circ\ldots\circ(\id+f_{\nu})(x_0)
        \in C.
    \]
This would imply that
    \[
        |x_{\mu+1}-x_0|
        \leq 
        \sum\nolimits_{i=0}^{\mu}
        |x_{i+1}-x_{i}|
        =
        \sum\nolimits_{i=0}^{\mu}
        |f_{\nu-i}(x_{i})|
        \leq 
        \sum\nolimits_{j=1}^{\mu}
        |f_{\nu-j}|_{0}<\theta,
    \]
which gives the contradiction that $x_{\mu+1}\in C$. We conclude that the claim holds. In particular, $g_{\nu}$ is defined on $B$, and
    $
        (\id+g_{\nu})(B)\sub C.
    $
Define
    \[
        \sigma_{\nu}^k
        \defeq 
        \sum\nolimits_{i=1}^{\nu}
        |f_{i}|_{k},
        \qquad
        \sigma^k
        \defeq 
        \lim_{\nu\to\infty}
        \sigma_{\nu}^k
        =
        \sum\nolimits_{\nu\geq 1}
        |f_{\nu}|_{k}.
    \]
First we show that there are constants $c_k\geq 0$ such that
    \begin{equation}
    \label{eqn:g_nu is bounded}
        |g_{\nu}|_{k}
        \leq 
        c_k\sigma^k_{\nu}.
    \end{equation}
Note that $g_1=f_1$, and that in general we have the following recursive formula:
    \[
        g_{\nu+1}
        =
        f_{\nu+1} + g_{\nu}\circ (\id + f_{\nu+1}).
    \]
In particular, this implies that
    \[
        |g_{\nu+1}|_{0}
        \leq
        |f_{\nu+1}|_{0} + |g_{\nu}|_{0}.
    \]
By iterating this we obtain inequality (\ref{eqn:g_nu is bounded}) for $k=0$. 
For $k\geq 1$, by applying Corollary \ref{coro:local comp} to the recursive formula we obtain
    \begin{equation}
    \label{eqn:iterative inequality}
        |g_{\nu+1}|_{k}
        \leq
        |f_{\nu+1}|_{k}
        (1+c_k|g_{\nu}|_{1})
        +
        |g_{\nu}|_{k}
        (1+c_k|f_{\nu+1}|_{1}),
    \end{equation}
for some constants $c_k>0$ that depend only on $k$. For $k=1$, this gives
    \begin{align*}
        |g_{\nu+1}|_{1}
        &\leq 
        |f_{\nu+1}|_{1}
        +
        |g_{\nu}|_{1}
        (1+c_1|f_{\nu+1}|_{1})
        \\
        &\leq
        |f_{\nu+1}|_{1}
        +
        |g_{\nu}|_{1}
        e^{c_1|f_{\nu+1}|_{1}}.
    \end{align*}
Since the following argument will be used twice, let us define
    \[
        x_{\nu} \defeq |g_{\nu}|_{1}, 
        \quad 
        y_{\nu} \defeq |f_{\nu}|_{1},
        \quad 
        \epsilon_{\nu} \defeq c_1|f_{\nu}|_{1}.
    \]
With this notation the above inequality becomes
    \[  
        x_{\nu+1}
        \leq 
        y_{\nu+1}
        +
        x_{\nu}
        e^{\epsilon_{\nu+1}}.
    \]
Thus we obtain
    \[
        x_{\nu+1}
        e^{
            -\sum\nolimits_{i=1}^{\nu+1}
            \epsilon_{i}
        }
        \leq 
        y_{\nu+1}
        e^{
            -\sum\nolimits_{i=1}^{\nu+1}
            \epsilon_{i}
        }
        +
        x_{\nu}
        e^{
            -\sum\nolimits_{i=1}^{\nu}
            \epsilon_{i}
        }
        \leq 
        y_{\nu+1}
        +
        x_{\nu}
        e^{
            -\sum\nolimits_{i=1}^{\nu}
            \epsilon_{i}
        }.
    \]
Iterating this, and using that $x_1e^{-\epsilon_1}\leq x_1\leq y_1$, we obtain
    \[
        x_{\nu}
        e^{-\sum\nolimits_{i=1}^{\nu}\epsilon_{i}}
        \leq 
        \sum\nolimits_{i=1}^{\nu}
        y_{i}.
    \]
Hence we obtain
    \[
        x_{\nu}
        \leq 
        e^{\sum\nolimits_{i=1}^{\nu}\epsilon_{i}}
        \sum\nolimits_{i=1}^{\nu}
        y_{i}
        \leq 
        e^{\sum\nolimits_{i\geq 1}\epsilon_{i}}
        \sum\nolimits_{i=1}^{\nu}
        y_{i}.
    \]
Recall that, by the hypothesis of the Lemma, we have 
    $
        \sum\nolimits_{i\geq 1}
        \epsilon_i
        \leq 
        c_1\sigma^1
        \leq 
        c_1\theta,
    $ 
so we obtain inequality (\ref{eqn:g_nu is bounded}) for $k=1$. 
Moreover, it implies that 
    $
        |g_{\nu}|_{1}\leq c_1.
    $ 
Using this in inequality (\ref{eqn:iterative inequality}), for $k\geq 1$, we obtain an inequality of the form
    \begin{align*}
        |g_{\nu+1}|_{k}
        &\leq
        c_k |f_{\nu+1}|_{k}
        +
        |g_{\nu}|_{k}
        (1+c_k|f_{\nu+1}|_{1})
        \\
        &\leq 
        c_k |f_{\nu+1}|_{k}
        +
        |g_{\nu}|_{k}
        e^{c_k|f_{\nu+1}|_{1}}.
    \end{align*}
Now applying the argument from before with
    \[
        x_{\nu} \defeq |g_{\nu}|_{k},
        \quad
        y_{\nu} \defeq c_k|f_{\nu}|_{k}, 
        \quad
        \epsilon_{\nu} \defeq c_k|f_{\nu}|_{1},
    \]
gives inequality (\ref{eqn:g_nu is bounded}) for all $k\geq 1$.

By using Lemma \ref{lemma:local comp}(e) we obtain
    \begin{align*}
        |g_{\nu+1}-g_{\nu}|_{k}
        &\leq 
        |f_{\nu+1}|_{k}
        +
        c_k (|f_{\nu+1}|_{k+1}|g_{\nu}|_{0}
        +
        |f_{\nu+1}|_{0}|g_{\nu}|_{k+1})
        \\
        &\leq 
        c_k |f_{\nu+1}|_{k+1} |g_{\nu}|_{k+1}
        \\
        &\leq 
        c_k|f_{\nu+1}|_{k+1}\sigma^{k+1},
\end{align*}
where in the last step we have used (\ref{eqn:g_nu is bounded}). On the right-hand-side we have the terms of an absolutely convergent series. Therefore the sequence $\{g_{\nu}\}_{\nu\geq 1}$ is Cauchy with respect to all $C^k$-norms, which implies that it converges in all $C^k$-norms uniformly to a smooth map $g_{\infty}:B\to \R^m$. Since, for all $x\in B$, $x+g_{\nu}(x)\in C$ it follows that $\id+g_{\infty}$ takes values in $C$. Finally, taking $\nu\to\infty$ in inequality (\ref{eqn:g_nu is bounded}) gives the final estimate of the Lemma.
\end{proof}

\subsection{Intermezzo: Closed Pseudogroups}
\label{sec:intermezzo_closed_pseudogroups}

In this section we prove three lemmas related to closed pseudogroups.

\begin{lemma}
\label{lemma:time_independent_flows}
Let $B$ and $C$ be two closed balls in $\R^m$ such that $B\subset \mathrm{int}(C)$.
Let $v^t\in \mathfrak{X}(C)$, with $t\in [0,1]$, be a time dependent vector field satisfying
    \[
        |v|_{1}<\theta,
    \]
and denote its flow by $\varphi^t_v$. Then, for all $n\geq 1$ and all $0\leq t\leq 1$, we have that the map     
    \[
        (\varphi^{t/n}_v)^n
        \defeq
        \varphi^{t/n}_v \circ \ldots \circ \varphi^{t/n}
    \]
is defined on $B$ as a smooth map
    \[
        (\varphi_v^{t/n})^n:B\longto C,
    \]
and the sequence $(\varphi_v^{t/n})^n$ converges uniformly on $B$ in all $C^k$-norms to the flow of the (time independent) vector field $v^0$, i.e.\
    \[
        \lim_{n\to \infty}(\varphi_v^{t/n})^n
        =
        \varphi_{v^0}^t.
    \]
\end{lemma}

\begin{proof}
Assume that $B$ is the closed ball $B_r$ of radius $r>0$ around $0\in \R^m$, and let $\rho>0$ be such that $B_{r+\rho}\subset \mathrm{int}(C)$. By Lemma \ref{lemma:flow}, we may assume that $\varphi_v^t$ is defined on $B_{r+\rho}$ for all $t\in [0,1]$. By the proof of Lemma \ref{lemma:flow}, $\varphi_v^{t}$ displaces a point by a distance less than $t\theta$. Thus by shrinking $\theta$ we may assume that
\[\varphi_v^{t/n}\big(B_{r+\frac{i-1}{n}\rho}\big)\subset B_{r+\frac{i}{n}\rho},\ \ \ \textrm{for}\ \ t\in [0,1]\ \textrm{and}\ i=1,\ldots, n.\]
This implies that $(\varphi_v^{t/n})^n$ is defined on $B$, for all $t\in [0,1]$ and all $n\geq 1$.

\noindent
For clarity of exposition, we replace $v$ by the compactly supported vector field 
    \[
        w \defeq \chi\cdot v,
    \] 
where $\chi$ is a smooth function such that $\mathrm{supp}(\chi)\subset \mathrm{int}(C)$ and $\chi|_{B_{r+\rho}}=1$. We have that $|w|_{k}\lesssim |v|_{k}$, that $\varphi_w^t:C\to C$ is defined for all $t\in [0,1]$, and that (by the above) $(\varphi_w^{t/n})^n|_B=(\varphi_v^{t/n})^n|_B$ for all $t\in [0,1]$ and all $n\geq 1$. Thus it suffices to show that $(\varphi_w^{t/n})^n$ converges uniformly to $\varphi_{w^0}^t$ in all $C^k$-norms on $C$.

For $t\in [0,1]$ and $n\geq 0$, we define
    \[
        (\varphi_w^t)^n
        \id+f^t_n:
        C\longto C, 
        \quad
        \delta^t_{n+1}
        \defeq
        f^t_{n+1}-f^t_n:
        C\longto\R^m.
    \]
Next we prove that, for $k\geq 1$ there are constants $c_k>0$ such that
    \begin{equation}
    \label{eq:uniform bound all n}
        |f_m^{t/n}|_k
        \leq 
        c_k|w|_k, 
        \quad
        \textrm{for all}\ t\in [0,1] \text{ and } 0\leq m\leq n.
    \end{equation}
First note that
    \[
        \delta^t_{m+1}
        =
        (\varphi_w^t)^{m+1}-(\varphi_w^t)^{m}
        =
        \big(
            (\varphi_w^t)^{m}
            -
            (\varphi_w^t)^{m-1}
        \big)
        \circ
        (\varphi_w^t)
        =
        \delta^t_m \circ (\id+f^t_1).
    \]
By Lemma \ref{lemma:flow}, we have that $|f^t_1|_{1}\leq c t|w|_1<c\theta$. Therefore we may apply Lemma \ref{coro:local comp} to obtain that
    \[
        |\delta^t_{m+1}|_1
        \leq 
        |\delta^t_m|_1(1+ct|w|_1).
    \]
Define $\epsilon\defeq ct|w|_1$. By iterating the above we obtain
    \[
        |\delta^t_{m}|_1
        \leq 
        |f^t_1|_1 (1+\epsilon)^{m-1}
        \leq 
        \epsilon (1+\epsilon)^{m-1}.
    \]
Therefore,
    \[
        |f^t_m|_1
        =
        |\sum\nolimits_{i=1}^m
        \delta^t_i|_1
        \leq 
        (1+\epsilon)^{m}-1.
    \]
Next we apply Lemma \ref{coro:local comp} to the higher order norms:
    \begin{align*}
        |\delta^t_{m+1}|_k
        &\leq 
        |\delta^t_m|_k (1+c_kt|w|_1)
        +
        c_k |\delta^t_m|_1 |f^t_1|_k
        \\
        &\leq 
        |\delta^t_m|_k
        (1+c_k\epsilon)
        +
        c_k t |w|_k
        \epsilon(1+\epsilon)^{m-1}
        \\
        &\leq 
        |\delta^t_m|_k
        (1+c_k\epsilon)
        +
        c_k t |w|_k
        \epsilon (1+c_k\epsilon)^{m-1},
    \end{align*}
where we assumed that $c_k>1$. Equivalently,
    \begin{align*}
        (1+c_k\epsilon)^{1-m}
        |\delta^t_{m+1}|_k
        \leq 
        (1+c_k\epsilon)^{2-m}
        |\delta^t_m|_k
        +
        c_k \epsilon t|w|_k.
    \end{align*}
By iterating this inequality, and using again Lemma \ref{lemma:flow}, we obtain that
    \begin{align*}
        (1+c_k\epsilon)^{1-m}
        |\delta^t_{m+1}|_k
        &\leq 
        (1+c_k\epsilon)
        |\delta^t_1|_k
        +
        (m-2) c_k \epsilon t|w|_k
        \\
        &\leq
        (1+(m-1)c_k\epsilon) c_k t |w|_k
        \\
        &\leq
        (1+c_k\epsilon)^{m-1} c_k t|w|_k,
    \end{align*}
Hence we obtain that
    \[
        |\delta^t_{m+1}|_k
        \leq 
        (1+c_k\epsilon)^{2(m-1)}
        c_k t|w|_k
        \leq 
        (1+c_k\epsilon)^{2m} c_k t|w|_k,
    \]
which holds also for $m=0$. 
This gives
    \[
        |f^t_m|_k
        =
        |\sum\nolimits_{i=1}^m\delta^t_i|_k
        \leq 
        \frac{(1+c_k\epsilon)^{2m}-1}{\epsilon}c_kt|w|_k
        \leq 
        \frac{e^{2c_km\epsilon}-1}{\epsilon}c_kt|w|_k.
    \]
Substituting $\epsilon=ct|w|_1$ back we obtain
    \[
        |f^t_m|_k
        \leq 
        \frac{e^{c_kmt|w|_1}-1}{|w|_1}c_k|w|_k.
    \]
Using that the function $(e^x-1)/x$ is increasing for $x\geq 0$, we obtain inequality (\ref{eq:uniform bound all n}):
    \[
        |f^{t/n}_m|_k
        \leq 
        \frac{e^{c_k|w|_1}-1}{|w|_1}c_k|w|_k
        \leq 
        \frac{e^{c_k\theta}-1}{\theta}
        c_k^2|w|_k
        \leq 
        c_k|w|_k.
    \]

Next consider the telescopic sum
    \begin{align*}
        (\varphi_w^{t/n})^n-\varphi_{w^0}^{t}
        &=
        \sum\nolimits_{m=0}^{n-1}
        (\varphi_w^{t/n})^{m+1}
        \circ
        \varphi_{w^0}^{t(n-m-1)/n}
        - 
        (\varphi_w^{t/n})^{m}
        \circ
        \varphi_{w^0}^{t(n-m)/n}
        \\
        &=
        \sum\nolimits_{m=0}^{n-1}
        \Big( 
            (\varphi_w^{t/n})^{m}\circ \chi^{t/n}- (\varphi_w^{t/n})^{m}
        \Big)
        \circ
        \varphi_{w^0}^{t(n-m)/n},
\end{align*}
where $\chi^s\defeq\varphi_w^{s}\circ \varphi_{w^0}^{-s}$. To each term in the sum we first apply Lemma \ref{lemma:local comp}(a), then Lemma \ref{lemma:flow} (we can because $|w^0|_1<\theta$), and then Lemma \ref{lemma:local comp}(b) (we can because of (\ref{eq:uniform bound all n})), and finally inequality (\ref{eq:uniform bound all n}). We obtain
    \begin{align*}
        |(\varphi_w^{t/n})^n-\varphi_{w^0}^{t}|_k
        &\leq 
        \sum\nolimits_{m=0}^{n-1}
        \Big|
            \big(
                (\varphi_w^{t/n})^{m}\circ \chi^{t/n}- (\varphi_w^{t/n})^{m}
            \big)
            \circ\varphi_{w^0}^{t(n-m)/n}
        \Big|_k
        \\
        &\lesssim 
        \sum\nolimits_{m=0}^{n-1} 
        |(\varphi_w^{t/n})^{m}\circ \chi^{t/n}- (\varphi_w^{t/n})^{m}|_k(1+|w^0|_k)
        \\
        &\lesssim
        \sum\nolimits_{m=0}^{n-1} 
        |(\varphi_w^{t/n})^{m}\circ \chi^{t/n}- (\varphi_w^{t/n})^{m}|_k(1+|w|_k)
        \\
        &\lesssim
        \sum\nolimits_{m=0}^{n-1}
        |f^{t/n}_m|_{k+1}|\chi^{t/n}-\mathrm{id}|_k(1+|w|_k)
        \\
        &\lesssim 
        n\cdot |\chi^{t/n}-\mathrm{id}|_k|w|_{k+1}(1+|w|_k).
    \end{align*}
Note that the function $\chi^s-\mathrm{id}$ vanishes to second order at $s=0$, therefore $\chi^s-\mathrm{id}=s^2\kappa^s$ for some smooth function $\kappa^s:C\to \R^m$. We obtain
    \[
        |(\varphi_w^{t/n})^n-\varphi_{w^0}^{t}|_k
        \lesssim 
        \tfrac{ 1}{n}
        |\kappa|_{k}|w|_{k+1}
        (1+|w|_k),
    \]
which proves the Lemma.
\end{proof}

The following implies continuity of the flow.

\begin{lemma}
\label{lemma:cont_action}
Let $B$ and $C$ be two closed balls in $\R^m$ such that $B\subset \mathrm{int}(C)$.
For all vector fields $u,v\in \mathfrak{X}(C)$ satisfying 
    \[
        |u|_2<\theta,
        \quad 
        |v|_2<\theta, 
    \]
the flow 
    $
        \varphi^{t}_{u+sv}:B\to C
    $
is defined for all $-1\leq s,t\leq 1$ and satisfies
    \[
        |\varphi^t_{u+v}-\varphi^t_{u}|_k
        \lesssim 
        |v|_k(1+|u|_{k+1}+|v|_{k+1})^2.
    \]
\end{lemma}
\begin{proof}
The claim about the flow being well-defined follows from Lemma \ref{lemma:flow}. 

Recall the variational formula for the derivative of the flow \cite[\S 1.5 (1.5.2)]{DK}:
    \[
        \tfrac{\dd}{\dd s}\varphi^t_{u+sv}
        =
        \int_0^t
            (\varphi_{u+sv}^{-\tau})^*(v)
        \,\dd\tau
        \circ \varphi^t_{u+sv}.
    \]
This gives
    \[
        \varphi^t_{u+v}
        -
        \varphi^t_{u}
        =
        \int_0^1 \int_0^t
            (\varphi_{u+\sigma v}^{-\tau})^*(v) \circ \varphi^t_{u+\sigma v}
        \,\dd\sigma \,\dd\tau.
    \]
We apply Lemma \ref{lemma:local comp}(d) and then Lemma \ref{lemma:flow} to the integrand:
    \[
        |(\varphi_{u+\sigma v}^{-\tau})^*(v) \circ \varphi^t_{u+\sigma v}|_{k}
        \lesssim
        |(\varphi_{u+\sigma v}^{-\tau})^*(v)|_k(1+|u|_k+v|_k).
    \]
Further, using \cite[Lemma 3.11]{Mar14} to estimate the pullback of vector fields, we obtain
    \[
        |(\varphi_{u+\sigma v}^{-\tau})^*(v)|_k
        \lesssim 
        |v|_k(1+|v|_{k+1}+|u|_{k+1}).
    \]
These inequalities imply the statement.
\end{proof}

The next result implies that the Lie algebra sheaf of a closed pseudogroup is closed.

\begin{lemma}
\label{lemma:closed_lie_algebra_sheaf}
Let $B$ and $C$ be two closed balls in $\R^m$ such that $B\subset \mathrm{int}(C)$. Let $v_{n}\in \mathfrak{X}(C)$, $n\geq 1$, be a sequence of vector fields converging uniformly in all $C^k$-norms to a vector field $v\in \mathfrak{X}(C)$. There is $\epsilon>0$ such that, for $|t|<\epsilon$ and for $n$ large enough, the flows $\varphi_{v_n}^t$ and $\varphi_v^t$ are defined on $B$, and the sequence $\varphi_{v_n}^t$ converges on $B$ uniformly in all $C^k$-norms to $\varphi_v^t$.
\end{lemma}

\begin{proof}
Let $n_0$ be such that, for $n>n_0$, $|v_n|_1\leq2|v|_1$. Denote $\epsilon\defeq\theta/(2|v|_1+1)$. The sequence
$w_n\defeq\epsilon v_n$ converges uniformly to $w\defeq\epsilon v$, and, for $n>n_0$, it satisfies $|w_n|_1<\theta$ and $|w|_1<\theta$. Thus, by Lemma \ref{lemma:flow}, for $n>n_0$ $\varphi_{w_n}^t=\varphi_{v_n}^{t\epsilon}$ and $\varphi_w^{t}=\varphi_v^{t\epsilon}$ are defined as maps $B\to C$, for all $t\in [-1,1]$. Thus, we need to show that $\varphi_{w_n}^t$ converges to $\varphi_w^t$ on $B$, for all $t\in [-1,1]$. This follows directly from Lemma \ref{lemma:cont_action}.
\end{proof}

\subsection{Global Versions for Compactly Supported Functions}
\label{subsect:comparison}

In this subsection, we prove global versions of Lemmas A, B and C.
We start with a discussion on different types of $C^k$-norms.

\subsubsection{Standard Norms}
Let $\pi_\textsc{F}:F\to X$ be a vector bundle, and
fix smooth norms $n_k$, for $k\geq 1$, on its $k$-th jet bundles $J^kF\to X$.
These can be used to define $C^k$-norms on spaces of sections:
    \[
        \|e\|_{k,K}
        \defeq 
        \sup\nolimits\nolimits_{x\in K}
        n_k\big( j^k_x e \big),
        \quad 
        e\in \Gamma_KF,
    \]
where $K\sub X$ is a compact subset with dense interior, i.e.\ 
    $
        K=\overline{\mathrm{int}K}.
    $
We consider these $C^k$-norms as being the standard norms on $\Gamma_KF$. They are also the norms we use in section \ref{sec:tame_estimates} prior to the statement of the \nameref{Main Theorem}.

Consider a second set of norms $n'_k$, with induced $C^k$-norms $\|\cdot \|_{k,K}'$.
Since $K$ is compact, the two families of norms $n_k$ and $n'_k$ are uniformly equivalent over $K$. Hence the resulting $C^k$-norms are equivalent for all $k$, i.e.\
    \begin{equation}
    \label{eq:standard}
        \|e\|_{k,K}
        \lesssim 
        \|e\|_{k,K}'
        \lesssim\|e\|_{k,K},
        \quad\forall\;
        e\in \Gamma_KF.
\end{equation}

\subsubsection{Standard Norms for Diffeomorphisms}
\label{subsub:norms_for_diffeo}

Given a fibre bundle, one way to construct $C^k$-norms for sections close to a fixed section is by using a vertical tubular neighborhood around the section, and thus reducing the construction to the setting of a vector bundle.

We explain how this works for local maps $X\longto X$ close to the identity.
Note first that local maps $X\longto X$ can be identified with local sections of the trivial bundle 
    \[
        \mathrm{pr}_1:X\times X\longto X.
    \]
Consider a Riemannian metric $m$ on $X$.
The corresponding exponential mapping $\mathcal{E}$ gives an open bundle embedding
    \[
        V(m)\longhookrightarrow  X\times X,
        \quad
        v\mapsto (\pi_{TX}(v),\mathcal{E}(v)),\]
where $V(m)\sub TX$ is a neighborhood of the zero section.
Let $K\sub X$ be a compact subset with dense interior, i.e.\ $K = \overline{\mathrm{int}(K)}$.
Then the push forward of vector fields along $\mathcal{E}$ gives a linear isomorphism
    \begin{equation}
    \label{exp_iso}
        \mathcal{E}_*:\mathcal{V}_K\diffto \mathcal{U}_K,
    \end{equation}
between $C^0$-neighborhoods $\mathcal{V}_K\sub \mathfrak{X}(K)$ of $0$ and
$\mathcal{U}_K\sub C^{\infty}(K,X)$ of $\id\!|_{K}$.
For example, if $\rho:X\to (0,\infty)$ denotes the injectivity radius of $m$ and $\delta:X\times X\to [0,\infty)$ the distance function induced by $m$, then one can take
    \begin{align*}
        \mathcal{U}_K
        &\defeq
        \big\{
            v\in \gerX(K)
            \sep 
            |v(x)|_m < \rho(x), 
            \;\forall\;
            x\in K
        \big\}, 
        \\
        \mathcal{V}_K
        &\defeq
        \big\{
            \phi\in C^{\infty}(K,X)
            \sep 
            \delta(x,\phi(x))<\rho(x), 
            \;\forall\;
            x\in K
        \big\}.
\end{align*}

Consider $C^k$-norms $\|\cdot \|_{k,K}$ on $\mathfrak{X}(K)$ coming from norms $n_k$ on the jet bundles of $TX$, as defined in the previous section.
We transport these norms via $\mathcal{E}_*$ to $C^k$-norms on $\mathcal{U}_K$:
    \[
        \|\phi\|_{k,K}
        \defeq 
        \|v\|_{k,K},
        \quad \textrm{where } \phi=\mathcal{E}_*(v)
        \textrm{ and }
        v\in \mathcal{V}_K.
    \]

\subsubsection{Norms from Local Trivializations}

Let $\pi_\textsc{F}:F\to X$ be a vector bundle. In this section we describe an alternative way to construct $C^k$-norms on the sections of $F$, using explicit local trivializations of $F$, and then we show that these norms are equivalent to the standard ones. Later we will see that the norms described here are more convenient for some computations. A similar discussion can be found in a remark of \cite{Ham82}.

\begin{definition}\label{defi:local_triv}
A \textbf{non-linear local trivialization}\index{non-linear frame} of $F$ consists of
    \[
        \zeta=(U,B,\chi ,V,\Xi),
    \]
where:
\begin{itemize}

\item 
$\chi :U\to \R^m$ is a chart on $X$.

\item 
$B\sub U$ is a compact set such that $\chi (B)$ is a closed ball in $\R^m$.

\item 
$V$ is an open set in $F$ such that
    \[
        z_\textsc{F}(U)\sub V\sub \pi_\textsc{F}^{-1}(U),
    \]
where $z_\textsc{F}:X\to F$ denotes the zero-section.

\item 
$\Xi:V\to \R^n$ is a \textbf{non-linear frame}, i.e.\ for each $x\in \pi_\textsc{F}(V)$, this map restricts to an open embedding of $V_x\defeq \pi_\textsc{F}^{-1}(x)\cap  V$ into $\R^n$. We define
    \[
        \Xi_x\defeq \Xi|_{V_x}:V_x\hookrightarrow \R^n.
    \]
\end{itemize}
Let such a non-linear local trivialization $\zeta$ of $F$ be given. 
If $e\in\Gamma_{B}F$ satisfies $e(B)\sub V$, then we say that $e$ is \textbf{admissible}\index{admissible} for $\zeta$. In this case, its \textbf{local representation} is defined as
    \[
        e^{\zeta}
        \defeq 
        \Xi \circ  e\circ  \chi ^{-1}-\Xi \circ  z_F\circ  \chi ^{-1}:
        \chi (B)\longto \R^n.
    \]
For admissible sections we define the corresponding $C^k$-norms by
    \[
        |e|_{k,B}^{\zeta}
        \defeq
        |e^{\zeta}|_{k,\chi (B)}
        \defeq
        \sup\nolimits_{x\in \chi (B)}
        \Big(
            \sum\nolimits_{|a|\leq k}
            |(D^ae^{\zeta})_x|^2
        \Big)^{1/2}.
    \]
Finally, if $V=F|_U$ and $\Xi_p$ is linear for all $p\in U$, we call $\zeta$ a \textbf{linear local trivialization}.
Note that any section is admissible for linear local trivializations.
\end{definition}


\begin{definition}
Consider a finite family of local trivializations of $F$,
    \[
        \calZ
        =
        \big\{
            \zeta^i=(U^i,B^i,\chi ^i,V^i,\Xi^i)
        \big\}_{i\in I},
        \quad
        I \textrm{ finite,}
    \]
and define
    $
        K\defeq \cup _{i\in I}B^i.
    $
A section $e\in \Gamma_{K}F$ is said to be $\calZ$-admissible if for each $i\in I$ its restriction $\zeta^i|_{B^i}$ is $\zeta^i$-admissible.
For such sections we define the following $C^k$-norms:
    \[
        |e|_{k,K}^{\calZ}
        \defeq 
        \sum\nolimits_{i\in I}
        |e|_{k,B^i}^{\zeta^i}.
        \qedhere
    \]
\end{definition}

Next we show that these norms are equivalent to the standard ones.

\begin{lemma}
\label{lemma:comparison}
Any $e\in \Gamma_{K}F$ satisfying $\|e\|_{0,K}<\theta$ is $\calZ$-admissible, and \[\|e\|_{k,K}\lesssim |e|_{k,K}^{\calZ}\lesssim\|e\|_{k,K}.\]
\end{lemma}

\begin{proof}
It suffices to prove the result for a $\calZ$ with a single local trivialization ${\zeta=(U,B,\chi ,V,\Xi)}$.
The existence of a $\theta>0$ such that $\|e\|_{0,K}<\theta$ implies that $e$ is $\calZ$-admissable is obvious.
We associate to $\zeta$ a second trivialization $\zeta'$ which is obtained by linearizing the frame:
    \[
        \zeta' \defeq (U,B,\chi ,F|_U,\Xi'), 
        \quad 
        \Xi'(v)\defeq 
        \tfrac{\dd}{\dd \epsilon}\big|_{\epsilon=0}\Xi(\epsilon v).
    \]
Since $\Xi'$ is linear, we have $V'=F|_U$, and any section $e\in \Gamma_{B}F$ is admissible for $\zeta'$. Moreover, the induced norms 
    $
        |e|_{k,B}^{\zeta'}
    $ 
can be described also by norms on the jet bundles (namely, the pullback of the standard norms on the jet bundle of $\R^m\times\R^n\to \R^m$).
Therefore, we immediately obtain inequalities of the form
    \[
        \|e\|_{k,B}
        \lesssim
        |e|_{k,B}^{\zeta'}
        \lesssim
        \|e\|_{k,B}.
    \]
Hence it suffices to check equivalence of the norms $|\cdot |_{k,B}^{\zeta'}$ and $|\cdot |_{k,B}^{\zeta}$ for any two (non-linear) local trivializations $\zeta$ and $\zeta'$.
For simplicity, assume that $\Xi \circ  z_F=0$ (note that translating $\Xi$ by $\Xi \circ  z_F\circ  \pi_F$ does not change the norms).
If $e\in \Gamma_{B}F$ satisfies $\|e\|_{0,B}<\theta$, then its local representatives are related as follows:
    \begin{equation}
    \label{eqn:local_relation}
        e^{\zeta'}
        =
        g\circ  (\id\times e^{\zeta}),
    \end{equation}
where $g$ is the change-of-frame map
    \[
        g:O \longto \R^n, 
        \quad
        g(\chi (x),v)
        \defeq
        \Xi'_{x}\circ (\Xi_{x})^{-1}(v),
    \]
where $O\sub \chi (U)\times \R^n$ is an open neighborhood of $\chi (U)\times \{0\}$.
By our assumption, $g(x,0)=0$ for all $x\in \chi (U)$.
Since $\chi (B)$ is a smooth closed ball, we may apply Lemma \ref{lemma:local comp}(b) to formula (\ref{eqn:local_relation}) with $f=(0,e^{\zeta})$, and we obtain that
    \[
        |e|_{k,B}^{\zeta'} 
        \lesssim
        |e|_{k,B}^{\zeta}.
    \]
Interchanging the role of $\zeta$ and $\zeta'$, the conclusion follows.
\end{proof}

\subsubsection{Norms for Diffeomorphisms from Charts}

\begin{definition}
Consider a triple $\zeta=(U,B,\chi )$, where $\chi :U\to \R^m$ is a chart on $X$, and $B\sub U$ is a compact subset such that $\chi (B)$ is a closed ball in $\R^m$.
If a smooth map $\phi:B\longto X$ satisfies
    $
        \phi(B)\sub U,
    $
then we say that $\phi$ is \textbf{admissible} for $\zeta$.
In this case, its \textbf{local representation} is the smooth map
    \[
        \id+f^{\zeta}:\chi (B)\longto \R^m, 
        \quad
        \id+f^{\zeta}\defeq \chi \circ  \phi \circ  \chi ^{-1},
    \]
and the corresponding $C^k$-(semi)-norm is defined by
    \[
        |\phi|_{k,B}^{\zeta}
        \defeq 
        |f^{\zeta}|_{k,\chi (B)}.
        \qedhere
    \]
\end{definition}

\begin{definition}
Consider a finite family of such triples
$\calZ=\{\zeta^i=(U^i,B^i,\chi ^i)\}_{i\in I}$,
and define
    $K\defeq \cup _{i\in I}B^i.$
If $\phi\in C^{\infty}(K,X)$ is $\calZ$-\textbf{admissible}\index{admissible} if $\phi|_{B^i}$ is admissible for all $i\in I$; and, in this case, we define its $C^k$-norm with respect to $\calZ$ as follows
    \[  
        |\phi|_{k,K}^{\calZ}
        \defeq 
        \sum\nolimits_{i\in I}
        |\phi|_{k,B^i}^{\zeta^i}.
        \qedhere
    \]
\end{definition}

Clearly, all these notions are particular cases of the notions discussed in the previous section.
This can be made precise by using the Riemannian exponential map $\mathcal{E}$ from subsection \ref{subsub:norms_for_diffeo}, so that these norms can be transferred to norms on vector fields.
On this space, they correspond to norms coming from nonlinear trivializations of $TX$ as defined in Definition \ref{defi:local_triv}.
Thus the following version of Lemma \ref{lemma:comparison} holds.

\begin{lemma}
\label{lemma:comparison_diffeo}
Any $\phi\in \mathcal{U}_K\sub C^{\infty}(K,X)$ satisfying $\|\phi\|_{0,K}<\theta$ is $\calZ$-admissible. Moreover, for any
$\phi\in C^{\infty}(K,X)$ which is $\calZ$-admissible and satisfies $|\phi|^{\calZ}_{0,K}<\theta$ we have that $\phi\in \mathcal{U}_K$.
In either case, the resulting norms are equivalent, i.e.\
    \[
        \|\phi\|_{k,K}
        \lesssim 
        |\phi|_{k,K}^{\calZ}
        \lesssim\|\phi\|_{k,K}.
    \]
\end{lemma}

\subsubsection{Global Version of Lemma A}

Consider a surjective submersion $\pi_\textsc{E}:E\to M$ and a partial differential operator
    \[
        Q:
        \Gamma_{M}E\longto C^{\infty}(N,\R^n), 
        \quad 
        Q(e)_x 
        \defeq 
        q_1((j^de)_{q_0(x)},x),
        \quad\forall\;
        x\in N,
    \]
constructed out of two maps (as per definition \ref{def:gen_pde})
    \[
        q_0:N\longto M, 
        \qquad
        q_1: J^dE \times_{M}N \longto \R^n.
    \]
We will assume that $q_1$ vanishes along the zero-section, i.e.\ $q_1(0_{q_0(x)},x)=0$ for all $x\in N$.

Let $K\sub N$ and $L\sub M$ be two compact subsets with dense interiors satisfying $q_0(K)\sub \mathrm{int}(L)$.
By applying $Q$ and then restricting to $K$, we have an induced map
    \[
        Q:\Gamma_{L}E\longrightarrow \Gamma_{K}(N\times \R^n),
    \]
which we also denote by $Q$.
The following version of \nameref{Lemma:A} holds.

\begin{lemma}
\label{lemma:A_global}
For all $e\in \Gamma_{L}E$ satisfying $\|e\|_{d,L}<\theta$ we have that
\[\|Q(e)-D_0Q(e)\|_{k,K}\lesssim \|e\|_{d,L}\|e\|_{k+d,L}.\]
\end{lemma}

\begin{proof}
Note that $Q$ is a composition of simpler PDOs, namely $Q=Q_1\circ  Q_0\circ  j^d$, where $Q_1$ is of degree $0$ and covers the identity, $Q_0$ is the pullback along $q_0$, and $j^d$ is the canonical $d$-th jet map.
Accordingly, our proof also splits into several steps.

\textbf{Step 1.} 
We show that it suffices to prove the Lemma for $d=0$.
For this, note that the map $Q$ and its differential at $e=0$ factor as
    \[
        Q=P\circ  j^d,
        \quad
        D_0Q=D_0P\circ  j^d,
    \]
where $P$ is the partial differential operator of degree 0, i.e.
    \[
        P:\Gamma_{M}J^dE\longto C^{\infty}(N,\R^n), 
        \quad
        P(\alpha)_x
        \defeq 
        q_1(\alpha_{q_0(x)},x), 
        \qquad\forall\;
        x\in N.
    \]
Suppose that the Lemma has been proven for PDOs of degree $0$. Then by applying it to $P$ we obtain that there is a $\theta>0$ so that, if  
    $
        j^de\in \Gamma_{L}J^d E
    $
satisfies $\|\alpha\|_{0,L}<\theta$, then
    \[
        \|(Q-D_0Q)(e)\|_{k,K}
        =
        \|(P-D_0P)(j^de)\|_{k,K}
        \lesssim
        \|j^de\|_{0,L}
        \|j^de\|_{k,L}.
    \]
Since $\|j^de\|_{k,L}\lesssim \|e\|_{k+d,L}$, the conclusion follows.

We assume from now on that $d=0$, i.e.\ $Q(e)_x=q_1(e\circ  q_0(x),x)$.

\textbf{Step 2.} 
Since $K$ and $L$ are compact with dense interior, and $q_0(K)\sub \mathrm{int}(L)$, we can find another compact set with dense interior
$K'\sub N$ such that
    \[
        K\sub \mathrm{int}(K')
        \textrm{ and }
        q_0(K')\sub \mathrm{int}(L).
    \]
Let $Q_0$ denote the pullback along $q_0$, i.e.\
    \[
        Q_0:
        \Gamma_{M}E\longto \Gamma_{N}(E\times_{M} N),
        \quad  
        e\mapsto e\circ  q_0.
    \]
For $k\geq 0$ this map induces a vector bundle map covering the identity, i.e.\
    \[
        q_0^k: 
        (J^kE)\times_MN\longto J^k(E\times_MN),
        \quad 
        (j^k_{q_0(x)}e,x) \mapsto j^k_x(e\circ  q_0).
    \]
Let $n_k$ denote the norm on $J^kE$ and let $\tilde{n}_k$ denote the norm on $J^k(E\times_MN)$, inducing the $C^k$-norms on sections of $E$ and $E\times_MN$ respectively.
Since $K'$ is compact, there exists a constant $c_k>0$ such that
    \[
        \tilde{n}_k(q_0^k(\alpha,x))
        \leq 
        c_k n_k(\alpha),
        \quad\forall\;
        x\in K',
        \;\forall\;
        \alpha\in J^k_{q_0(x)}E.
    \]
This gives estimates of the form
    \[
        \|Q_0(e)\|_{k,K'}
        \lesssim 
        \|e\|_{k,L}, 
        \quad\forall\;
        e\in \Gamma_{L}E.
    \]

\textbf{Step 3.} 
We show that it suffices to prove the Lemma for $N=M$ and $q_0=\mathrm{id}$.
Note that the map $Q$ and its differential at $e=0$ factor as
    \[
        Q=Q_1\circ  Q_0, 
        \quad
        D_0Q=D_0Q_1\circ  Q_0,
    \]
where $Q_0$ is as in Step 2, and $Q_1$ is induced by a fibre bundle map, i.e.
    \[
        Q_1:\Gamma_{N}(E\times_MN)\longto C^{\infty}(N,\R^n),
        \quad
        Q_1(\beta)\defeq q_1\circ  \beta.
    \]
Suppose that the Lemma has been proven for PDOs with $q_0=\id$. If $e\in \Gamma_{L}E$ satifies $\|e\|_{0,L}<\theta$, then by Step 2 we have $\|Q_0(e)\|_{0,K'}<\theta$.
Therefore, by applying the Lemma to $Q_0(e)$, the operator $Q_1$, and the sets $K \sub K'$, we obtain that
    \begin{align*}
        \|(Q-D_0Q)(e)\|_{k,K}
        &=
        \|(Q_1-D_0Q_1)(Q_0(e))\|_{k,K}
        \\
        &\lesssim 
        \|Q_0(e)\|_{0,K'}\|Q_0(e)\|_{k,K'}
        \\
        &\lesssim
        \|e\|_{0, L} \|e\|_{k, L}.
    \end{align*}

\textbf{Step 4.}
We have reduced the Lemma to the case
    \[
        Q = Q_1:
        \Gamma_{K'}F\longto C^{\infty}(K,\R^n),
        \quad
        Q(e)=q_1\circ  e|_{K},
    \]
where $F=E\times_{M}N$, and $q_1:F\to \R^n$ is a smooth map that vanishes along the zero-section.
Let $\zeta=(U,B,\chi ,F|_U,\Xi)$ be a linear local trivialization of $F$ as per Definition \ref{defi:local_triv}.
This induces the (somewhat tautological) linear local trivialization 
    \[
        \zeta_1
        \defeq 
        (U, B, \chi, U\times \R^n, \mathrm{pr}_2)
    \]
of $N\times\R^n$.
Denote the local representation of $q_1$ in these trivializations by
    \[
        q_1^{\zeta}:\chi (U)\times \R^l\longto \R^n.
    \]
Then for $e\in \Gamma_{B}F$ the local representation of $Q(e)-D_0Q(e)$ is
    \[
        (Q(e)-D_0Q(e))^{\zeta_1}
        =
        q_1^{\zeta}(\id\times e^{\zeta})
        -
        \dd_0q_1^{\zeta}(\id\times e^{\zeta}),
    \]
where 
    \[
        \dd_0q_1^{\zeta}(x,v)
        \defeq
        \tfrac{\dd}{\dd \epsilon}\big|_{\epsilon=0}
        q_1^{\zeta}(x,\epsilon v).
    \]
Note also that $q_1^{\zeta}(\id\times 0)=0$.
We apply Lemma \ref{lemma:local comp}(c) to estimate the right-hand-side, and obtain that there is a $\theta>0$ such that, for all $e\in\Gamma_{B}F$ satisfying $\|e\|_{0,B}\leq \theta$,
    \[
        |Q(e)-D_0Q(e)|_{k,B}^{\zeta_1}
        \lesssim 
        |e|_{0,B}^{\zeta}|e|^{\zeta}_{k,{B}}.
    \]
Since $K$ is compact with dense interior, and $K\sub \mathrm{int}(K')$, there exists a finite family $\calZ=\{\zeta^i\}$ of linear local trivializations as above  satisfying
\[K\sub K_{\calZ}\defeq \cup _{i\in I}B^i\sub K'.\]
By applying Lemma \ref{lemma:comparison} twice and the above argument to each $\zeta^i$ we obtain
    \[
        \|Q(e)-D_0Q(e)\|_{k,K_{\calZ}}
        \lesssim
        |Q(e)-D_0Q(e)|^{\calZ}_{k,K_{\calZ}}
        \lesssim
        |e|^{\calZ}_{0,K_{\calZ}}
        |e|^{\calZ}_{k,K_{\calZ}}
        \lesssim 
        \|e\|_{0,K_{\calZ}}\|e\|_{k,K_{\calZ}},
    \]
for all $e\in \Gamma_{K_{\calZ}}F$ satisfying $\|e\|_{0,K_{\calZ}}<\theta$.
This trivially implies that
    \[
        \|Q(e)-D_0Q(e)\|_{k,K}
        \lesssim 
        \|e\|_{0,K'}\|e\|_{k,K'},
    \]
for all $e\in \Gamma_{K'}F$ satisfying $\|e\|_{0,K'}<\theta$, and concludes the proof.
\end{proof}

\subsubsection{Compact Supports}
%
The support of a map $\phi:O\to X$ is the closure of the open set where $\phi\neq \id$.
We denote by $C^{\infty}_c(O,X)$ the compactly supported smooth maps $O\to X$, and by $C^\infty_K(X,X)$ the space of smooth maps whose support lies in the set $K\sub X$.
Such maps can be extended by the identity to smooth maps in $C^{\infty}(X,X)$, and thus we regard
    $
        C^{\infty}_c(O,X)
        \sub 
        C^{\infty}(X,X).
    $
    
Likewise, for a vector bundle $\pi_\textsc{F}: F\to X$, the support of a section $e\in \Gamma_OF$ is the closure of the open set where $e\neq 0$. We denote by $\Gamma_c(O, F)$ the compactly supported sections of $F|_O$, and by    
    \[
        \Gamma_K(O,F)
        \defeq
        \big\{
            e \in \Gamma_OF
            \sep
            \mathrm{supp}(e)\sub K
        \big\}
    \]
the sections of $F|_O$ with support in the set $K\sub X$.

\subsubsection{Recognizing Diffeomorphisms}

We show that maps which are $C^1$-close to the identity are diffeomorphisms.

We remind the reader of the conventions for \nameref{sec:constant_theta} and \nameref{sec:symbol_lesssim} on p.\pageref{sec:constant_theta}, and of the exponential mapping $\calE_*:\calV_K \to \calU_K$ defined on p.\pageref{exp_iso}.

\begin{lemma}
\label{lemma:diffeo}
Let $K\sub X$ be a compact set.
Any map $\phi\in C^{\infty}_{K}(X,X)\cap  \mathcal{U}_{K}$ satisfying $\|\phi\|_{1,K}<\theta$ is a diffeomorphism of $X$, and moreover, $\phi^{-1}\in C^{\infty}_{K}(X,X)$ satisfies
    \[
        \|\phi^{-1}\|_{k,K}
        \lesssim 
        \|\phi\|_{k,K}.
    \]
\end{lemma}

\begin{proof}
Fix a finite family of charts $\{(U^i,\chi ^i)\}_{i\in I}$ and three families of compact sets
    \[
        B_2^i\sub \mathrm{int}(B_1^i),
        \quad
        B_1^i\sub \mathrm{int}(B_0^i),
        \quad
        B_0^i\sub U^i,
        \quad\forall\; i\in I,
    \]
such that:
\begin{itemize}

\item 
$\chi ^i(B_a^i)\sub \R^m$ is a closed ball for all $a=0,1,2$ and $i\in I$.

\item 
For all $i,j\in I$ we have that
    \[
        B_1^i \cap  B_1^j \neq \emptyset
        \;\;\Longrightarrow\;\;
        B_2^j\sub \mathrm{int}(B_0^i).
    \]
    
\item 
$K\sub \cup _{i\in I}B_2^i$.

\end{itemize}
Such a cover can be easily constructed.
In particular, the second condition can be ensured by requiring that there are points $x_i\in X$ and a small enough number $\epsilon>0$ such that
    \[
        B_2^i\sub B_{\epsilon}(x_i),
        \quad 
        B_1^i\sub B_{2\epsilon}(x_i), 
        \quad
        B_{5\epsilon}(x_i)\sub \mathrm{int}(B_0^i),
    \]
where $B_{r}(x)$ denotes the closed ball of radius $r$ centered at $x$ for a fixed metric $m$ on $X$.

We write $K_a\defeq\cup _{i\in I}B_a^i$ for $a=0,1,2.$ 
Consider also the families of charts
    \[
        \calZ_a=\{\zeta_a^i\defeq (U^i,B_a^i,\chi ^i)\}_{i\in I},
        \quad
        a=0,1,2,
    \]
with the resulting norms $|\cdot |^{\calZ_a}_{k,K_a}$.

Since $\phi$ is supported inside $K\sub K_0$, we also have that $\phi\in \mathcal{U}_{K_0}$, and that 
    $
        \|\phi\|_{k,K_0}
        =
        \|\phi\|_{k,K}.
    $
Thus by taking $\theta$ small enough we may assume that $\phi$ is $\calZ_0$-admissible
Denote the corresponding local expressions by
    \[
        \id+f^i
        =
        \chi ^i \circ \phi \circ (\chi ^i)^{-1}:
        \chi ^i(B_0^i)\longto \chi ^i(U^i).
    \]
Since we have
    \[
        |f^i|_{1,\chi ^i(B_0^i)}
        \lesssim 
        |\phi|^{\calZ_0}_{1,K_0}
        \lesssim 
        \|\phi\|_{1,K_0}<\theta,
    \] 
we may apply Lemma \ref{lemma:inverse} to conclude that $\id+f^i$ is a diffeomorphism onto its image and 
    \[
        \chi ^i(B^i_1)\sub (\id+f_i)(\chi ^i(B^i_0)).
    \]
The former implies that $\phi$ is a local diffeomorphism whose restriction to $B^i_0$ is injective for all $i\in I$.
The latter implies that 
    \[
        K_1=\cup _{i\in I}B^i_1\sub \phi(X),
    \]
and since $\phi$ is the identity outside of $K\sub K_1$, we deduce that $\phi$ is onto.

We will show that $\phi$ is injective.
By taking $\theta$ even smaller we may assume that, for all $i\in I$,
    \[
        \phi(B_2^i)\sub B_1^i.
    \]
Assume now that $\phi(x)=\phi(y)$ for some $x,y\in K_2$, and let $i,j\in I$ be such that $x\in B_2^i$ and $y\in B_2^j$.
By the above we have $\phi(x)=\phi(y)\in B_1^i\cap  B_1^j$.
Our assumptions on the sets $B^i_a$ then imply that $B_2^j\sub B_0^i$.
Hence $x,y\in B_0^i$, and therefore $x=y$.
Next assume $\phi(x)=\phi(y)$ for some $x\in K_2$ and $y\notin K_2$, and assume that $x\in B_2^i$.
Then $y=\phi(y)=\phi(x)\in B^i_1$, hence $x,y\in B_1^i$, and we conclude again that $x=y$.

The estimates for the local expressions of $\phi^{-1}$
    \[
        (\id+f_i)^{-1}:
        \chi ^i(B_1^i)\longto \chi ^i(U^i)
    \]
have been proven in Lemma \ref{lemma:inverse}, and since $\phi$ is supported inside $K_2$, the desired estimates follow from that Lemma.
\end{proof}

\subsubsection{Global Version of Lemma B(1)}

In this section we extend Lemma \ref{lemma:flow} to a global version of \nameref{Lemma:B}(1) for vector fields with compact support.

We denote by $\gerX_{K}(X) = \Gamma_K(X, TX)$ the space of vector fields on $X$ with support in $K$.

\begin{lemma}
\label{lemma:B_1_global}
Let $K\sub X$ be a compact set.
For any vector field $v\in \gerX_{K}(X)$ satisfying 
    $
        \|v\|_{1,K} < \theta
    $
we have that its time-one flow $\varphi_v\in C^{\infty}_{K}(X,X)$ satisfies
    \[
        \|\varphi_v\|_{k,K}
        \lesssim
        \|v\|_{k,K}.
    \]
\end{lemma}
\begin{proof}
Fix a finite family of charts $\{(U^i,\chi ^i)\}_{i\in I}$ and two families of compact sets,
    \[
        B_1^i \sub \mathrm{int}(B_0^i),
        \quad
        B_0^i \sub U^i,
        \quad i\in I,
    \]
such that $\chi ^i(B_0^i)$ and $\chi ^i(B_1^i)$ are closed balls in $\R^m$ and such that
    $
        K\sub \cup _{i\in I}B_1^i.
    $
Write
    \[
        K_a
        \defeq
        \cup _{i\in I}B_a^i, 
        \quad
        \calZ_a
        \defeq
        \big\{
            \zeta_a^i=(U^i,B_a^i,\chi ^i)
        \big\}_{i\in I},
        \quad
        a=0,1.
    \]
Denote the local representatives of $v$ in the charts $(U^i,\chi ^i)$ by
    \[
        v^i
        \defeq
        \dd \chi ^i\circ  v\circ  (\chi ^i)^{-1}
        \in \mathfrak{X}(\chi ^i(U^i)).
    \]
These are the local representations of $v$ with corresponding to the linear trivializations of $TX$ given by the the family $\{(U^i,B^i_0,TU^i, \chi ^i,\dd \chi ^i)\}_{i\in I}$. Therefore, by Lemma \ref{lemma:comparison} we obtain
    \begin{equation}
    \label{ineq:vf_local_1}
        |v^i|_{k,\chi ^i(B^i_0)}
        \lesssim 
        \|v\|_{k,K_0}
        =
        \|v\|_{k,K}
    \end{equation}
Thus by choosing $\theta>0$ small enough we may apply Lemma \ref{lemma:flow} to the $v^i$ and the balls $\chi ^i(B_1^i)\sub \chi ^i(B_0^i)$.
We conclude that the flow of $v^i$ is defined for $t\in [0,1]$ as a map
    \[
        \id+f^t_{v^i}:
        \chi ^i(B_1^i)\longto \chi ^i(B_0^i),
    \]
Moreover, using also inequality (\ref{ineq:vf_local_1}), it satisfies
    \begin{equation}
    \label{eqn:flow2}
        |f^t_{v^i}|_{k,\chi ^i(B_1^i)}
        \lesssim 
        |v^i|_{k,\chi ^i(B_0^i)}
        \lesssim 
        \|v\|_{k,K}.
    \end{equation}
This shows that, for $t\in [0,1]$, the flow $\varphi^t_{v}\in C^{\infty}_{K}(X,X)$ is admissible for the family $\calZ_1$ with local representatives $\id+f^t_{v^i}$.
By inequality (\ref{eqn:flow2}) and Lemma \ref{lemma:comparison_diffeo} we obtain also that
    \[
        \|\varphi_{v}\|_{k,K}
        =
        \|\varphi_{v}\|_{k,K_1}
        \lesssim 
        \|v\|_{k,K}.
        \qedhere
    \]
\end{proof}

\subsubsection{Global Version of Lemmas B(2,3)}

In this section we extend Lemma \ref{lemma:local_flow_action} to a global version of \nameref{Lemma:B}(2, 3) for diffeomorphisms and sections with compact support.

Let $\pi_\textsc{X}:X\to M$ be a surjective submersion, and fix a global section $b\in\Gamma_M X$. We regard $M$ as a submanifold of $X$ via the embedding $b:M\hookrightarrow X$. Let $E\sub X$ be an open subbundle around $M$ that is endowed with a vector bundle structure for which $M$ is the zero-section. In other words, $E$ is a vertical tubular neighborhood of $M$ in $X$. We will write $\pi_\textsc{E} \defeq \pi_\textsc{X}|_E$.

\begin{lemma}
\label{lemma:global_action}
Let $L\sub X$ be a compact set and denote $K\defeq L\cap  M$.
For all diffeomorphisms ${\phi\in C^{\infty}_{L}(X,X)\cap \mathcal{U}_L}$ and sections $e\in \Gamma_K(M, E)$ satisfying
    \[
        \|\phi\|_{L,1}<\theta
        \quad\textrm{and}\quad
        \|e\|_{K,1}<\theta
    \]
we have that $\phi$ on $e$ are compatible on $M$ (in the sense of definition \ref{def:compatibility}):
    \[
        e\cdot\phi\in \Gamma_K(M, E).
    \]
Moreover, they satisfy
    \[
        \|e\cdot \phi\|_{k,K}
        \lesssim
        \|\phi\|_{k,L}
        +
        \|e\|_{k,K}.
    \]
\end{lemma}

\begin{proof}
Fix the following objects:
\begin{itemize}
\item A finite family of charts $\{\chi ^i:U^i\to \R^m\}_{i\in I}$ on $M$.
                                                                                                                                                                                                                                                                                                                                                                                                                                                                                                                                                                                                                                                                                                                                                                                                                                                                                                                                                                                                                                                                                                                                                                                                                                                                                                                                                                                                                                                                                                                                                                                                                                                                                                                                                                                                                                                                                                                                                                                                                                                                                                                                                                                                                                                                                                                                                                                                                                                                                                                                                                                                                                                                                                                                                                                                                                                                                                                                                                                                                                                                                                                                                                                                                                                                                                                                                                                                                                                                                                                                                                                                                                                                                                                                                              \item Compact sets $B^i, C^i\sub U^i$, for $i\in I$, that satisfy
 \begin{itemize}
    \item  $B^i\sub \mathrm{int}(C^i)$;
    \item $\chi ^i(B^i)$ and $\chi ^i(C^i)$ are closed balls in $\R^m$;
    \item $K\sub \cup _{i\in I}B^i$.
 \end{itemize}
\item Linear frames on $E$ above the $U^i$, for $i\in I$, i.e.\ maps $\Xi^i:E|_{U^i}\to \R^n$ with induced bundle trivializations denote by
        \[
            \widehat{\chi }^i
            \defeq 
            (\chi ^i\circ  \pi)\times \Xi^i:
            E|_{U^i}\longto \R^m\times \R^n.
        \]
\item Compact sets $A^i\sub E|_{U^i}$, for $i\in I$, such that $\widehat{\chi }^i(A^i)$ is a closed ball  in $\R^{m}\times \R^n$ and
        \[
            B^i \sub \mathrm{int}(A^i),
            \quad
            A^i\sub \mathrm{int}(\pi^{-1}(C^i)).
        \]
\end{itemize}

For the $\phi$ and $e$ to be compatible we need to show that $\phi$ is a diffeomorphism of $X$ and that $\pi \circ  \phi^{-1}\circ  e$ is a diffeomorphism of $M$.
By Lemma \ref{lemma:diffeo}, we may assume that $\phi$ is a diffeomorphism of $X$, and moreover, that
    \begin{equation}
    \label{eq:inverse}
        \phi^{-1}\in C^{\infty}_{L}(X,X)\cap  \mathcal{U}_L,
        \quad
        \|\phi^{-1}\|_{L,k}\lesssim \|\phi\|_{L,k}.
    \end{equation}
Since the inverse is supported inside $L$, this implies that by shrinking $\theta$ we may assume that 
    \[
        \phi^{-1}(A^i)\sub E|_{U^i}.
    \]
Denote the local expression of $\phi^{-1}$ by
    \[
        \id+f^i_1\times f^i_2
        \defeq
        \widehat{\chi }^i\circ  \phi^{-1}\circ  (\widehat{\chi }^i)^{-1}: 
        \widehat{\chi }^i(A^i)\longto \R^m\times \R^n.
    \]
Similarly, since $\|e\|_{0,K}<\theta$ and $B^i\sub \mathrm{int}(A^i)$, we may assume that $e(B^i)\sub A^i$.
Denote the local expression of $e$ by
    \[
        e^i
        \defeq
        \Xi^i\circ  e\circ  (\chi ^i)^{-1}:
        \chi ^i(C^i)\longto \R^n.
    \]
We obtain that the composition
    \[
        \phi_e
        \defeq 
        \pi \circ  \phi^{-1}\circ e:
        M\longrightarrow M
    \]
maps $B^i$ into $C^i$ and has the local expression
    \[
        \id+f^i_1\circ  (\id\times e^i)
        =
        \chi ^i\circ  \phi_e \circ  (\chi ^i)^{-1}:
        \chi ^i(B^i)\longrightarrow \R^m.
    \]
Applying Lemmas \ref{lemma:local comp}(d), then \ref{lemma:comparison} and \ref{lemma:comparison_diffeo}, and finally using (\ref{eq:inverse}) we obtain
    \begin{align}
    \label{ineq:domain_action}
        |f^i_1\circ  (\id\times e^i)|_{k,\chi ^i(B^i)}
        &\lesssim
        |f^i_1|_{k, \widehat\chi^i(A^i)} 
        + 
        |f^i_1|_{1, \widehat\chi^i(A^i)} 
        |e^i|_{k, \chi^i(B^i)}
        \notag
        \\
        &\lesssim
        \|\phi^{-1}\|_{k, L} 
        +
        \theta \| e\|_{k, K}
        \notag
        \\
        &\lesssim 
        \|\phi\|_{k,L}+\|e\|_{k,K}.
    \end{align}
If $x\in M\backslash K$, then $e(x)=x$ and $\phi^{-1}(x)=x$, hence $\phi_e(x)=x$.
Thus
    \[
        \phi_e\in C^{\infty}_{K}(M,M).
    \]
On the other hand, since $K\sub \cup _{i\in I}B^i$ we can use the family of charts 
    \[
    \calZ\defeq \{(U^i, B^i,\chi ^i)\}_{i\in I}
    \]
(for which $\phi_e$ is admissible) to calculate the norms of $\phi_e$. Then by inequality (\ref{ineq:domain_action}) and again by Lemma \ref{lemma:comparison_diffeo} we obtain that
    \[
        \|\phi_e\|_{k,K}
        \lesssim
        \|\phi\|_{k,L}
        +
        \|e\|_{k,K}.
    \]
In particular, after shrinking $\theta$, by Lemma \ref{lemma:diffeo} we may assume that $\phi_e$ is a diffeomorphism, and clearly that $\phi_e^{-1}\in C^{\infty}_{K}(M,M)$.
This implies that the action is globally defined, i.e.\ $\phi$ and $e$ are compatible on $M$, and one checks as above that the result has support in $K$:
    \[
        e\cdot \phi
        =
        \phi^{-1}\circ  e\circ  \phi_e^{-1}
        \in 
        \Gamma_K(M,E).
    \]
Finally, we will apply Lemma \ref{lemma:action_local} to obtain the estimate of the Lemma.
Consider the local expression of $\phi$:
    \[
        \id+g^i
        \defeq
        \widehat{\chi }^i\circ  \phi \circ  (\widehat{\chi }^i)^{-1}: 
        \widehat{\chi }^i(A^i)\longto \R^m\times \R^n.
    \]
Note that $|g_i|_{1,\widehat{\chi }^i(A^i)}\lesssim \|\phi\|_{1,L}$.
Thus we may apply Lemma \ref{lemma:action_local} to the maps $\id+g_i$ and $e_i$ and the balls $\chi ^i(B^i)$, $\chi ^i(C^i)$ and $\widehat{\chi }^i(A^i)$ to conclude that the local action $e^i\cdot  (\id\times g^i)$ between the local expressions is defined as a map
    \[  
        e^i\cdot  (\id+g^i): 
        \chi ^i(B^i)\longrightarrow \R^n.
    \]
Moreover, it satisfies
    \[
        |e^i\cdot  (\id+g^i)|_{k,\chi ^i(B^i)}
        \lesssim 
        |e^i|_{k,\chi ^i(C^i)}
        +
        |g^i|_{k,\chi ^i(A^i)}.
    \]
But since the action of $e\cdot  \phi$ is globally defined, we see that the maps $e^i\cdot  (\id\times g^i)$ are the local expressions of $e\cdot  \phi$, i.e.
    \[
        \Xi^i\circ  (e\cdot  \phi)\circ  (\chi ^i)^{-1}
        =
        e^i\cdot  (\id+ g^i).\]
By applying Lemma \ref{lemma:comparison} once more, we obtain the estimates from the statement.
\end{proof}

Finally, we give the following global version of \nameref{Lemma:B}(2,3).

\begin{lemma}
\label{lemma:B_23_global}
Let $L\sub X$ be a compact set, and denote $K\defeq L\cap  M$.
For all $v\in \gerX_{L}(X)$ and $e\in \Gamma_K(M, E)$ satisfying
    \[
        \|v\|_{L,1} < \theta
        \quad\textrm{and}\quad
        \|e\|_{K,1} < \theta,
    \]
we have that the time-one flow $\varphi_v$ on $e$ are compatible on $M$ (in the sense of definition \ref{def:compatibility}):
    \[
        e\cdot  \varphi_v\in \Gamma_{M}^{K}E.
    \]
Moreover, they satisfy
    \begin{align*}
        \|e\cdot \varphi_v\|_{k,K}
        &\lesssim
        \|v\|_{k,L}+\|e\|_{k,K}
        \\
        \|e\cdot \varphi_v-e-\delta v\|_{k,K}
        &\lesssim
        (\|v\|_{0,L} + \|e\|_{0,K})
        \|v\|_{k+1,L}
        +
        \|v\|_{0,L} \|e\|_{k+1,K}.
    \end{align*}
\end{lemma}

\begin{proof}
Consider charts, linear frames, and trivializations as in the proof of Lemma \ref{lemma:global_action}.
The fact that the action $e\cdot  \varphi_v$ is defined, i.e. $\varphi_v$ and $e$ are compatible on $M$, and that it satisfies the first inequality follows directly from Lemmas \ref{lemma:B_1_global} and \ref{lemma:global_action}.
Exactly as in the proof of Lemma \ref{lemma:global_action}, we can localize the second inequality.
For this, note that it is important that the frames $\Xi^i$ are linear, so that the local representatives of the infinitesimal action $\delta v$ is the infinitesimal action of the local representatives $v^i$ of $v$, i.e. so that
    \[
        (\delta v)^i = \delta (v^i).
    \]
Finally, applying Lemma \ref{lemma:local_flow_action} to the local expressions, one obtains the second inequality.
\end{proof}

\subsubsection{Global Version of Lemma C}

In this section we extend Lemma \ref{lemma:infinite_comp} to a global version of \nameref{Lemma:C} for diffeomorphisms with compact support.

\begin{lemma}
\label{lemma:C_global}
Let $K\sub X$ be a compact set.
For all sequences of smooth maps ${\phi_{\nu}\in C^{\infty}_{K}(X,X)}$ satisfying
    \[
        \sum\nolimits_{\nu\geq 1}
        \|\phi_{\nu}\|_{1,K} < \theta
        \quad\textrm{and}\quad
        \sum\nolimits_{\nu\geq 1}
        \|\phi_{\nu}\|_{k,K}<\infty,
        \quad\forall\; k\geq 0,
    \]
we have that the sequence of compositions
    \[
        \psi_{\nu}
        \defeq 
        \phi_1 \circ \ldots \circ  \phi_{\nu}\in C^{\infty}_{K}(X,X)
    \]
converges in all $C^k$-norms on $K$ to a smooth map
    \[
        \psi
        \defeq 
        \lim_{\nu\to \infty}\psi_{\nu}\in C^{\infty}_{K}(X,X).
    \]
Moreover, it satisfies
    \[
        \|\psi\|_{k,K}
        \lesssim 
        \sum\nolimits_{\nu\geq 1}
        \|{\phi}_{\nu}\|_{k,K}.
    \]
\end{lemma}

\begin{proof}
Fix a finite family of charts $\{(U^i,\chi ^i)\}_{i\in I}$ and two families of compact sets,
    \[
        B_1^i \sub \mathrm{int}(B_0^i),
        \quad
        B_0^i \sub U^i,
        \quad i\in I,
    \]
such that $\chi ^i(B_0^i)$ and $\chi ^i(B_1^i)$ are closed balls in $\R^m$ and such that
    $
        K\sub \cup _{i\in I}B_1^i.
    $
Write
    \[
        K_a
        \defeq
        \cup _{i\in I}B_a^i, 
        \quad
        \calZ_a
        \defeq
        \big\{
            \zeta_a^i=(U^i,B_a^i,\chi ^i)
        \big\}_{i\in I},
        \quad
        a=0,1.
    \]
By taking $\theta$ small enough we may assume that $\phi_{\nu}$ is $\calZ_0$-admissible.
Denote the corresponding local representation by
    \[
        \id+f_{\nu}^i
        \defeq 
        \chi ^i \circ \phi_{\nu} \circ (\chi ^i)^{-1}:
        \chi ^i(B^i_0)\longto \chi ^i(U^i).
    \]
Using Lemma \ref{lemma:comparison_diffeo} and that $\phi_{\nu}$ is supported in $K_1$ we obtain
    \begin{equation}
    \label{eq:proof_lemma_c}
        \sum\nolimits_{\nu\geq 1}
        |f^i_{\nu}|_{k,\chi ^i(B_0^i)}
        =
        \sum\nolimits_{\nu\geq 1}
        |\phi_{\nu}|^{\zeta_0^i}_{k,B_0^i}
        \lesssim 
        \sum\nolimits_{\nu\geq 1}
        |\phi_{\nu}|^{\calZ_0}_{k,K_0}
        \lesssim 
        \sum\nolimits_{\nu\geq 1}
        \|\phi_{\nu}\|_{k,K}.
    \end{equation}
Thus we may apply Lemma \ref{lemma:infinite_comp} to the sequence $\{\id+f_{\nu}^i\}_{\nu\geq 1}$ and the balls $\chi ^i(B^i_1)\sub \chi ^i(B^i_0)$.
We conclude that, for $\theta$ small enough, the following composition is defined between the indicated balls,
    \[
        \id+g_{\nu}^i
        \defeq 
        (\id+f^i_{1})
        \circ \ldots \circ 
        (\id+f^i_{\nu})|_{\chi ^i(B^i_1)}:
        \chi ^i(B^i_1)\longto \chi ^i(B^i_0),
    \]
and that it converges in all $C^k$-norms to a smooth map
    \[
        \id+g_{\infty}^i
        \defeq 
        \lim_{\nu\to\infty}(\id+g^i_{\nu}):
        \chi ^i(B^i_1)\longto \chi ^i(B^i_0).
    \]
Moreover, using also (\ref{eq:proof_lemma_c}), this map satisfies
    \begin{equation}
    \label{eq:lemma_C}
        |g_{\infty}^i|_{k,\chi ^i(B^i_1)}
        \lesssim 
        \sum\nolimits_{\nu\geq 1}
        \|\phi_{\nu}\|_{k,K}.
    \end{equation}
Moreover, in the proof of Lemma \ref{lemma:infinite_comp} we have shown that, for all $0\leq \mu<\nu$, the composition is defined as a map between the balls $\chi^i(B^i_1)$ and $\chi^i(B^i_0)$, i.e.\
    \[
        (\id+f^i_{\nu-\mu})
        \circ \ldots \circ 
        (\id+f^i_{\nu}):
        \chi ^i(B^i_1)\longto \chi ^i(B^i_0).
    \]
This implies that
    $
        \phi_{\nu-\mu}
        \circ \ldots \circ
        \phi_{\nu}(B^i_1)
        \sub 
        B^i_0
    $
for all $i$.
In particular, the composition
    \[
        \psi_{\nu}
        \defeq 
        \phi_{1}
        \circ \ldots \circ
        \phi_{\nu}
        \in C^{\infty}_{K}(X,X)
    \]
is admissible for the family $\calZ_1$ with local representatives $\id+g_{\nu}^i$.
We conclude that the sequence $\psi_{\nu}$ converges in all $C^k$-norms to a map $\psi\in C^{\infty}_{K}(X,X)$.
Since $\psi_{\nu}(B_1^i)\sub B_0^i$, we also have that
$\psi(B_1^i)\sub B_0^i$. Hence $\psi$ is also $\calZ_1$-admissible, with local representatives $\id+g^i_{\infty}$.
Therefore, by Lemma \ref{lemma:comparison_diffeo} and inequality (\ref{eq:lemma_C}), we obtain
    \[
        \|\psi\|_{k,K}
        =
        \|\psi\|_{k,K_1}
        \lesssim 
        |\psi|^{\calZ_1}_{k,K_1}
        =
        \sum\nolimits_{i}
        |g^i_{\infty}|_{k,\chi ^i(B_i^1)}
        \lesssim 
        \sum\nolimits_{\nu\geq 1}
        \|\phi_{\nu}\|_{k,K}.
        \qedhere
    \]
\end{proof}

\subsection{Global Versions on Shrinking Domains with Corners}
\label{sec:lemmas_corners}

In this section we discuss shrinking domains with corners, and prove \nameref{Lemma:A}, \nameref{Lemma:B} and \nameref{Lemma:C} in this setting.
The proofs are based on the key lemma \ref{lem:extension}, which allows us to tamely extend sections over a shrinking domain to global, compactly supported sections; reducing these lemmas to their global analogs discussed in the previous subsection.

\begin{lemma}
Let $D$ be a smooth manifold, and $\{ D_r \}$ a nested domain with corners in $D$.
Then for any Riemannian distance $\delta$ on $D$ there exists $\theta>0$ such that the following implication holds for all $r<s$:
    \[
        \big( 
            x\in D_r \textrm{ and } \delta(x,y)<\theta(s-r)
        \big)
        \quad\Longrightarrow\quad 
        y\in D_s.
    \]
\end{lemma}

\begin{proof}
Assume the second statement to be false.
Then there exist sequences 
    \[
        0\leq r_n<s_n\leq 1,
        \quad 
        x_n\in D_{r_n},
        \quad
        y_n\in D,
        \quad
        \theta_n>0
    \]
such that $\lim_{n\to \infty }\theta_n=0$ and $\delta(x_n,y_n)<\theta_n(s_n-r_n)$ but $y_n\notin D_{s_n}$.
By passing to convergent subsequences we obtain
    \[
        r \defeq \lim_{n\to \infty }r_n,
        \quad
        s \defeq \lim_{n\to \infty }s_n,
        \quad
        z \defeq \lim_{n\to \infty} x_n,
    \]
with $r\leq s$. Since $\lim_{n\to \infty}\delta(x_n,y_n)=0$, we also have $\lim_{n\to \infty} y_n=z$.

First assume that $r<s$.
For large enough $n$ we have $(r+s)/2 < s_n$, and thus by the previous, 
    \[
        D_r \sub \mathrm{int}(D_{(r+s)/2}),
        \quad
        D_{(r+s)/2}\sub \mathrm{int}(D_{s_n}).
    \]
Note that $z \in D_r$, since it is the limit of $x_n\in D_{r_n}$. We conclude that $y_n\in D_{(r+s)/2}$ for large enough $n$, but this contradicts with $y_n \notin D_{s_n}$.

Next assume that $s=r$, and define 
    \begin{align*}
        q_n 
        &\defeq 
        (\phi^{1-r_n})^{-1}(x_n) \;\in\; D_1,
        \qquad
        q_n'  \defeq (\phi^{1-s_n})^{-1}(x_n) \;\in\; D_1,
        \\
        p_n
        &\defeq 
        (\phi^{1-s_n})^{-1}(y_n) \;\notin\; D_1.
    \end{align*}
These three sequences converge to the same boundary point 
    \[
        p\defeq (\phi^{1-r})^{-1}(z) \;\in\; \partial D_1.
    \]
Note that the metrics $\delta$ and $(\phi^{\lambda})_*\delta$ are equivalent in a compact neighborhood of $D_1$, uniformly in $\lambda\in [0,1]$. Therefore we obtain
    \begin{align*}
        \delta
        (
            p_n, q_n'
        )
        &=
        \delta 
        \big(
            p_n,
            (\phi^{1-s_n})^{-1} \circ \phi^{1-r_n}(q_n)
        \big)
        \\
        &=
        ((\phi^{1-s_n})_* \delta) 
        \big(
            \phi^{1-s_n}(p_n),
            \phi^{1-r_n}(q_n)
        \big)
        \\
        &\leq
        c\delta(x_n,y_n)
        \leq 
        c\theta_n(s_n-r_n).
\end{align*}
Consider a chart $(U,\chi )$ adapted to $D_1$ and centered at $p$:
    \[
        \chi (U\cap  D_1)
        =
        \chi (U)\cap  
        \big(
            [0,\infty)^k\times \R^{m-k}
        \big),
        \quad
        \chi(p) = 0.
    \]
By lemma \ref{lemma:interior_cone} we may identify
    $
        \calC_p D_1
        =
        (0,\infty)^k\times \R^{m-k}.
    $
Moreover, $q_n$ and $p_n$ belong to $U$ for $n$ large enough.
To simplify notation we denote by $\epsilon_n\in \R^m$ any sequence converging to $0$.
In the chart $(U,\chi )$ the above inequality can be expressed as
    \[
        p_n
        -
        q_n'
        =
        \epsilon_n (s_n-r_n).
    \]
The Taylor expansion of $q_n'$ with respect to $r_n$ at $s_n$ gives
    \begin{align*}
        q_n' - q_n
        &=
        (\phi^{1-s_n})^{-1}
        \circ  
        \phi^{1-r_n}(q_n)
        -
        q_n
        \\
        &=
        (s_n-r_n)
        ((\phi^{1-s_n})^* v^{1-s_n})_{q_n}
        +
        \epsilon_n(s_n-r_n)
        \\
        &=
        (s_n-r_n)
        ((\phi^{1-r})^*v^{1-r})_p
        +
        \epsilon_n(s_n-r_n).
    \end{align*}
Here we used that $\epsilon (s_n-r_n)^2 = \epsilon_n(s_n-r_n)$ for any $\epsilon \in \R^m$.
We conclude
    \[
        \lim_{n\to\infty}
        \frac{p_n-q_n}{s_n-r_n}
        =
        ((\phi^{1-r})^*v^{1-r})_p
        \;\in\; \calC_p D_1.
    \]
Since $\calC_p D_1$ is an open cone, this implies that $p_n-q_n\in \calC_p D_1$ for $n$ large enough.
On the other hand, in our adapted chart, $D_1$ is given by $\calC_p D_1$, and since $q_n\in D_1$, it follows that 
    \[
        p_n
        =
        (p_n-q_n)
        +
        q_n
        \in \calC_p D_1.
        \qedhere
    \]

\end{proof}

\subsubsection{Norms and the Symbol $\lesssim$}

Let $\{D_r\}_{r\in [0,1]}$ be a shrinking domain with corners in $D$.
We will adopt the simpler notation for $C^k$-norms of sections $e\in \Gamma_F(D_r)$:
\[\|e\|_{k,r}\defeq \|e\|_{k,D_r}.\]

In all results discussed below, the constants subsumed by the symbol $\lesssim$ are independent on the index $r\in [0,1]$ (or on the indexes $r,s\in [0,1]$).
This dependence is lost already in the extension lemma \ref{lem:extension}, on which all other results are based upon.

\subsubsection{Extension Operators}

Next, we show that sections on a shrinking domain with corners can be extended to compactly supported functions in a uniformly ``tame'' way.
The first part of the proof is based on that of
\cite[Corollary 1.3.7, pg.138]{Ham82}.

\begin{lemma}
\label{lem:extension}
Let $F$ be a vector bundle over $D$, let $\{D_r\}_{r\in [0,1]}$ be a shrinking domain with corners in $D$, and consider a compact set $K\sub D$ such that $D_1\sub \mathrm{int}(K)$.
There exist linear maps:
\[\sigma_r:\Gamma_F(D_r)\longrightarrow \Gamma_K(D, F),\quad r\in [0,1],\]
satisfying
\[\|\sigma_r(e)\|_{k,K}
	\|e\|_{k,r}, \ \ \ \ \sigma_r(e)\big|_{D_r}=e,\]
for all $e\in \Gamma_F(D_r)$.
\end{lemma}

\begin{proof}
We construct first $\sigma_1$, by adapting the argument in \cite[Corollary 1.3.7, pg.138]{Ham82}.
Let $\zeta=(U,B,C,\chi ,\Xi,\rho)$ consist of:
\begin{itemize}
\item a chart $(U,\chi )$ on $D$, which is adapted to $D_1$:
\[\chi (U\cap  D_1)=\chi (U)\cap  \R^m_l, \]
where we have denoted
\[\R^m_l\defeq [0,\infty)^{l}\times\R^{m-l},\]
\item compact sets $B$, $C$ such that $B\sub \mathrm{int}(C)$,
\item a linear frame $\Xi:F|_U\to \R^n$ with associated trivialization:
\[\widehat{\chi }\defeq (\chi \circ \pi)\times\Xi:F|_U\to \R^m\times \R^n,\]
\item a map $\rho\in C^{\infty}(\R^m)$ such that $\rho|_{\chi (B)}=1$ and $\rho$ is supported inside $C$.
\end{itemize}

We first construct an extension map
    \[
        \sigma_{\zeta}:
        \Gamma_{B\cap  D_1}({D_1}, F) \longrightarrow \Gamma_{C}({D}, F).
    \]
Let $e\in \Gamma_{B\cap  D_1}({D_1}, F)$, and denote its local representation by
\[e^{\zeta}\defeq \Xi \circ  e\circ \chi ^{-1}:\chi (U)\cap  \R^m_l\longrightarrow \R^n.\]
Note that $e^{\zeta}$ is supported inside the compact set $\chi (B)\cap \R^m_l$.
By setting $e^{\zeta}$ to be zero on $\R^m_l\backslash \chi (B)$, it can be regarded as a map
\[e^{\zeta}: \R^m_l\longrightarrow \R^n,\]
which has support in $\chi (B)$.
Next, we extend $e^{\zeta}$ to a map
\[e^{\zeta}_1:\R^m_{l-1}\to \R^n,\]
defined for
\[(a,-x,b)\in [0,\infty)^{k-1}\times (-\infty,0)\times \R^{m-l}=\R^m_{l-1}\backslash\R^m_{l}\]
as follows:
\[e^{\zeta}_1(a,-x,b)\defeq \int_0^{\infty}\phi(t)e^{\zeta}(a,tx,b)\dd t,\]
where $\phi$ is function used and defined in \cite[Corollary 1.3.7, pg.138]{Ham82}.
Repeating this procedure, one obtains maps $e^{\zeta}_i$ defined on 
    \[
        [0,\infty)^{l-i}\times (-\infty,0)\times \R^{m-l+i},
        \quad 
        0\leq i\leq l,
    \]
with $e^{\zeta}=e^{\zeta}_0$, and with $e_{i+1}^{\zeta}$ extending $e_i^{\zeta}$.

Next, note that $\rho \cdot  e_{l}^{\zeta}$ extends $e^{\zeta}$ and is supported in $\chi (C)$.
Define
\[\sigma_{\zeta}(e)\in \Gamma_{C}({D}, F)\]
to be the section with local representative $\rho \cdot  e_{l}^{\zeta}$.
Using
  the equivalence of norms coming from jet spaces (\ref{eq:standard}), the multiplicativity rule (\ref{eqn:D-product}), that the map $\phi$ has globally bounded derivatives, as explained in
\cite[Corollary 1.3.7, pg.138]{Ham82}, and then again (\ref{eq:standard}), we obtain that the map $\sigma_{\zeta}$ satisfies:
\[\|\sigma_{\zeta}(e)\|_{k,C} \lesssim |\rho \cdot  e^{\zeta}_l|_{k,\chi (C)}\lesssim |e^{\zeta}_l|_{k,\R^m}\lesssim |e^{\zeta}|_{k,\R^m_k\cap  \chi (B)}
	\|e\|_{k,D_1\cap  B}
.\]

Consider now a finite family of local trivializations 
    \[
        \{\zeta^i=(U^i,B^i,C^i,\chi ^i,\Xi^i,\rho^i)\}_{i\in I}
    \]
with the properties from above,
so that we construct the extension maps
    \[
        \sigma_{\zeta^i}:
        \Gamma_{D_1\cap  B^i}({D^1},F) 
        \longto 
        \Gamma_{C^i}({D}, F).
    \]
Moreover, we choose this family such that
\[D_1\sub \cup _{i\in I}\mathrm{int}(B^i), \ \ \cup _{i\in I}C^i\sub K.\]

Consider a partition of unity $\{\eta^i\}_{i\in I}$ on $D_1$ subordinate to the cover $\{\mathrm{int}(B^i)\}_{i\in I}$.
We define the extension map
\[\sigma_1:\Gamma_F(D_1)\longrightarrow \Gamma_{K}(D, F),\]
\[\sigma_1(e)\defeq \sum\nolimits_{i\in I}\sigma_{\zeta^i}(\eta^i\cdot  e).\]
It is straightforward that $\sigma_1$ satisfies the required estimates.

Let $\{\phi^r\}_{r\in [0,1]}$ be a family of diffeomorphisms of $D$ which generates the shrinking domain, hence $D_{r}=\phi^{1-r}(D_1)$.
By applying the tube lemma to the map $\phi:[0,1]\times X\to X$, one can easily prove the existence of a compact set $L$ such that
\[D_1\sub \mathrm{int}(L) \ \ \ \textrm{and}\ \ \ \phi^{1-r}(L)\sub K,\ \ \textrm{for all}\ r\in [0,1].\]
Applying the Tube lemma there is a neighborhood $O$ of $D_1$ such that
Clearly, we may assume that the map $\sigma_1$ constructed above maps to $\Gamma_{L}(D, F)$.

Consider a smooth family of vector bundle isomorphisms
\[\widehat{\phi}^r:F\longrightarrow F, \ \ \ r\in [0,1],\]
such that $\widehat{\phi}^0=\id_F$ and such that $\widehat{\phi}^r$ covers $\phi^r$.
For example, one can set $\widehat{\phi}^r(y)$ to be the parallel transport with respect to a linear connection of $y$ above the path $\phi^t(\pi(y))$, $t\in [0,r]$.

Define the extension operator by conjugating with the pullback by $\widehat{\phi}^{1-r}$,
\[\sigma_r\defeq (\widehat{\phi}^{1-r})_*\circ  \sigma_1\circ  (\widehat{\phi}^{1-r})^*,\]
where these maps have the following domains:
\begin{align*}
    (\widehat{\phi}^{1-r})^*
    &:
    \Gamma_F(D_r) \longrightarrow \Gamma_F(D_1),
    \\
    \sigma_1
    &:
    \Gamma_F(D_1) \longrightarrow \Gamma_L(D,F),
    \\
    (\widehat{\phi}^{1-r})_*
    &:
    \Gamma_{L}(D, F) \longrightarrow \Gamma^{\phi_{1-r}(L)}(D, F) \sub \Gamma_{K}(D, F).
\end{align*}
Clearly, this is an extension operator:
\[\sigma_r(e)|_{D_r}=e \ \ \textrm{for} \ e\in \Gamma_F(D_r).\]
It remains to verify the tameness inequalities.
First note that there is an induced isomorphisms between the $k$-th jet bundles, $j^k\widehat{\phi}^s:J^kF\to J^kF$ covering $\phi^s$, defined by the relation:
\[(j^k_x\widehat{\phi}^s)(j^k_xe)=j^k_{\phi_s(x)}(\widehat{\phi}^s_*(e)),\ \ \ e\in \Gamma_F(D).\] Let $n_k^s$ be the pullback via this map of the norm $n_k$ on $J^kF$ used to define the $C^k$-norms.
Then $n_k^s$ is a smooth family of norms on $J^kF$. Therefore, for any compact subset $C\sub D$ we can find constants $c(k,C)>0$ such that
\[n_k(\alpha)/c(k,C)\leq  n_k^s(\alpha)\leq  c(k,C) n_k(\alpha), \]
for all $\alpha\in J^kF|_{C}$ and $s\in [0,1]$.
If $C$ has dense interior, $C=\overline{\mathrm{int}(C)}$, this implies that the induced linear isomorphisms
\[(\widehat{\phi}^{s})_*:\Gamma_F(C)\longrightarrow \Gamma_F(\phi^s(C))\]
satisfy the following estimates:
\[\|e\|_{k,C}
	\|(\widehat{\phi}^{s})_*(e)\|_{k,{\phi}^{s}(C)}
	\|e\|_{k,C}, \ \ \ e\in \Gamma_F(C).\]
Applying these inequalities for $C=D_1$ and for $C=L$ in the composition
\[\sigma_r=(\widehat{\phi}^{1-r})_*\circ  \sigma_1\circ  (\widehat{\phi}^{1-r}_*)^{-1}\]
we obtain that the inequality from the statement follows from the respective inequality for $\sigma_1$.
\end{proof}

\subsubsection{Extension Operators for Maps}
\label{subsub:extension_maps}

Let $\{X_r\}_{r\in [0,1]}$ be a shrinking domain with corners in $X$.
Let $K\sub X$ be a compact subset such that $X_1\sub \mathrm{int}(K)$.
We build extension operators for maps $X_r\to X$ as follows.
Denote the extension operators for $TX$ from lemma \ref{lem:extension} by
\[\sigma'_r:\mathfrak{X}(X_r)\longrightarrow \mathfrak{X}_K(X).\]
Conjugating by the isomorphism (\ref{exp_iso}) we find extension operators
\begin{align*}
\sigma_r:\mathcal{U}_{X_r,K}\longrightarrow C^{\infty}_K(X,X), \\
\phi\mapsto \mathcal{E}_*\circ \sigma_r'\circ  \mathcal{E}^{-1}_*(\phi),
\end{align*}
which are defined on the $C^0$-neighborhood of the inclusion $X_r\hookrightarrow X$:
\[\mathcal{U}_{X_r,K}\defeq \mathcal{U}_{X_r}\cap  \mathcal{E}_*((\sigma_r')^{-1}(\mathcal{V}_K)).\]
These maps also satisfy:
\begin{equation}\label{eq:ext}
\|\sigma_r(\phi)\|_{k,K}
	\|\phi\|_{k,r}, \ \ \ \ \sigma_r(\phi)|_{X_r}=\phi.
\end{equation}

\subsubsection{Smoothing Operators}

We use the lemma \ref{lem:extension} and some results from \cite{Nash} to obtain \textbf{smoothing operators}\index{smoothing operators}.
\begin{lemma}
Let $\{D_r\}_{r\in [0,1]}$ be a shrinking domain with corners.
There exists a family of linear maps:
\[S_t^r:\Gamma_F(D_r)\longrightarrow \Gamma_F(D_r), \ \ r\in [0,1], \ t> 1,\]
which satisfy the smoothing inequalities:
\[\|S_t^r e\|_{k,r}\lesssim t^{l}\|e\|_{k-l,r}  \ \  \ \|e-S_t^re\|_{k-l,r}\lesssim t^{-l}\|e\|_{k,r},\]
for all $e\in \Gamma_F(D_r)$ and all $0\leq l\leq k$.
\end{lemma}
\begin{proof}
Assume first that $F=\R\times D$.
Consider the following objects:
the extension map $\sigma_r:C^{\infty}(D_r)\to C^{\infty}_K(D)$ from the previous lemma with $K$ compact and $D_1\sub \mathrm{int}(K)$; a proper embedding $i:D\hookrightarrow \R^{N}$, where we identify $D\cong  i(D)$; a tubular neighborhood $T(D)\sub \R^N$ with projection $p:T(D)\to D$; an open set $V$ such that $K\sub V\sub T(D)$ with $\overline{V}$ compact; a smooth function $\rho:\R^N\to [0,1]$ such that $\rho|_{K}= 1$ and with $\mathrm{supp}(\rho)\sub V$.
We consider the following map:
\[\lambda_r:C^{\infty}(D_r)\longrightarrow C^{\infty}_{\overline{V}}(\R^m), \ \ \lambda_r(f)\defeq \rho \cdot  (\sigma_r(f)\circ  p).\]
It is easy to see that $\lambda_r$ satisfies
\[\|\lambda_r(f)\|_{k,\overline{V}}
	\|f\|_{k,r}, \ \ \ \lambda_r(f)|_{D_r}=f.\]
Using the family of smooth functions $K_t:\R^N\to \R$ constructed in \cite{Nash}, define the smoothing operators as follows:
\[S_t^r:C^{\infty}(D_r)\to C^{\infty}(D_r), \ \ \ S_t^r(f)=\left(K_t \ast \lambda_r(f)\right)|_{D_r},\]
where $\ast$ denotes the convolution product.
In \cite{Nash} it is shown that these operators satisfy indeed the smoothing inequalities.

For $F\cong \R^n\times D$ the above construction can be performed component-wise.
For general $F$ there exists a second vector bundle $F'$ such that 
    \[
        F\oplus F'\cong \R^n\times D.
    \]
Let $i_F:F\to \R^n\times D$ be the corresponding inclusion map, and 
    \[
        p_F:\R^n\times D\to F
    \]
the corresponding projection map.
If $S^r_t$ are smoothing operators for $\R^n\times D$, then 
    $
        (p_F)_*\circ  S^r_t\circ  (i_F)_*
    $
are smoothing operators for $F$.
\end{proof}

The norms $\|\cdot \|_{k,r}$ satisfy the following \textbf{interpolation inequalities}\index{interpolation inequalities}:
\[\|e\|_{k,r}^{n-l}
	\|e\|_{l,r}^{n-k}\|e\|_{n,r}^{k-l}, \ \ \ \  l \leq k\leq n.\]
This can be proven directly, as in \cite{Conn85}, but it is also a consequence of the existence of smoothing operators, as in \cite[Corollary 1.4.2, pg.176]{Ham82}.

\subsubsection{Proof of lemma A}
\label{sec:proof_of_lemma_A}

Let $L\sub M$ be a compact set such that $M_1\sub \mathrm{int}(L)$.
Then $q_0^{-1}(\mathrm{int}(L))$ is an open neighborhood of the compact set $N_1$. Therefore, there exists a compact set $K$ such that $N_1\sub \mathrm{int}(K)$ and $K\sub q_0^{-1}(\mathrm{int}(L))$, hence $q_0(K)\sub \mathrm{int}(L)$.
We consider the extension operators associated to $L$:
    \[
        \sigma_r:
        \Gamma_{M_r}E
        \longrightarrow \Gamma_{M}^LE.
    \]
Let $e\in \Gamma_{M_r}E$ be such that $\|e\|_{d,r}<\theta$.
Since $\|\sigma_r(e)\|_{d,L}
	\|e\|_{d,r}$, after shrinking $\theta$, we may apply lemma \ref{lemma:A_global} to obtain
\begin{align*}
\|Q(e)-D_0Q(e)\|_{k,r}&\leq \|Q(\sigma_r(e))-D_0Q(\sigma_r(e))\|_{k,K} \\
&\lesssim\|\sigma_{r}(e)\|_{d,L}\|\sigma_{r}(e)\|_{k+d,L}\lesssim\|e\|_{d,r}\|e\|_{k+d,r}.
\end{align*}
This concludes the proof of \nameref{Lemma:A}.

\subsubsection{Proof of lemma B}
\label{sec:proof_of_lemma_B}

Let $v\in \gerX(X_s)$ be a vector field as in the statement:
\[\|v\|_{0,s}<(s-r)\theta, \ \ \ \|v\|_{1,s}<\theta.\]
Fix a compact subset $L\sub X$ such that $X_1\sub \mathrm{int}(L)$.
Using the corresponding extension operator, we define
\[\widetilde{v}\defeq \sigma_s(v)\in \mathfrak{X}_L(X).\]
By the properties of the extension operators, $\|\widetilde{v}\|_{k,L}
	\|v\|_{k,s}$.
For $k=0$, this allows us to apply lemma \ref{lemma:B_1_global} to $\widetilde{v}$. 
Hence its flow satisfies 
    \[
        \varphi_{\widetilde{v}}^t\in C^{\infty}_{L}(X,X)\cap  \mathcal{U}_L
    \]
and
\begin{equation}\label{ineq:tame_flow}
\|\varphi^t_{\widetilde{v}}\|_{k,L}\lesssim\|\widetilde{v}\|_{k,L}
	\|v\|_{k,s}, \ \ \textrm{for }\ t\in [0,1].
\end{equation}
In particular, $\|\varphi^t_{\widetilde{v}}\|_{0,s}\leq (s-r)c\theta$, for some constant $c>0$.
This property implies that, for any Riemannian metric on $X$, there is a constant $C>0$ such that $\varphi^t_{\widetilde{v}}$ cannot displace points by a distance of more than $(s-r)C\theta$.
This and lemma \ref{lem:stictily_decreasing} imply that, after shrinking $\theta$, we may assume that
\[\varphi^t_{\widetilde{v}}(X_r)\sub X_s.\]
Since $v|_{X_s}=\widetilde{v}|_{X_s}$, the above property implies that, for $x\in X_r$, $\varphi^t_{\widetilde{v}}(x)$ is also a flow line of $v$.
So the flow of $v$ is defined for $t\in [0,1]$ on $X_r$ and it coincides with that of $\widetilde{v}$:
\[\varphi^t_{\widetilde{v}}|_{X_r}=\varphi^t_{v}|_{X_r}:X_r\longrightarrow X_s.\]
The estimates (\ref{ineq:tame_flow}) imply the estimates (1) from \nameref{Lemma:B}.

Next, let $e\in \Gamma_{M_s}E$ be a section as in the statement:
\[\|e\|_{0,s}<(s-r)\theta, \ \ \ \|e\|_{1,s}<\theta.\]
Consider the compact set $K\defeq L\cap  M$.
Then
\[M_1\sub X_1\cap  M\sub \left(\mathrm{int}(L)\right)\cap  M\sub \mathrm{int}(L\cap  M)=\mathrm{int}(K).\]
We use the extension operator corresponding $K$, we define:
\[\widetilde{e}\defeq \sigma_s(e)\in \Gamma_K(M, E),\]
which satisfies $\|\widetilde{e}\|_{k,K}
	\|e\|_{k,s}$.
By shrinking $\theta$ if necessary, we may apply lemma \ref{lemma:B_23_global} to $K$, $L$, $\widetilde{v}$ and $\widetilde{e}$ and conclude that the action of $\varphi_{\widetilde{v}}$ on $\widetilde{e}$ is defined:
    \[
        \widetilde{e} \cdot  \varphi_{\widetilde{v}} \in \Gamma_K(M, E),
    \]
and, using also the estimates for the extension operators, that it satisfies:
\[\|\widetilde{e}\cdot \varphi_{\widetilde{v}}\|_{k,K}\lesssim\|v\|_{k,s}+\|e\|_{k,s}.\]
\[\|\widetilde{e}\cdot \varphi_{\widetilde{v}}-\widetilde{e}-\delta \widetilde{v}\|_{k,K}\lesssim(\|v\|_{0,s}+\|e\|_{0,s})\|v\|_{k+1,s}+\|v\|_{0,s}\|e\|_{k+1,s}.\]

Finally, we show that the right action of $\varphi_v$ is defined on $e$ over $X_r$, and
\begin{equation}\label{eq:domain3}
(e\cdot  \varphi_v)|_{X_r}=(\widetilde{e}\cdot \varphi_{\widetilde{v}})|_{X_r}.
\end{equation}
Then the conclusion will follow immediately, because the following holds:
\[(\delta\widetilde{v})|_{X_r}=(\delta v)|_{X_r}.\]

First note that, by shrinking $\theta$, we may assume that 
\begin{equation}\label{eq:domain1}
\varphi_{\widetilde{v}}|_{X_{(r+s)/2}}=\varphi_{v}|_{X_{(r+s)/2}}:X_{(r+s)/2}\longrightarrow X_s.
\end{equation}
By the estimates for $\widetilde{e}\cdot \varphi_{\widetilde{v}}$, there exists a constant $c>0$ such that 
\[\|\widetilde{e}\cdot \varphi_{\widetilde{v}}\|_{0,K}\lesssim(s-r)c\theta.\]
In particular, this shows that $\widetilde{e}\cdot \varphi_{\widetilde{v}}$ sends a point in $M_r$ to a point at distance at most $C(s-r)\theta$, for some $C>0$ and some fixed Riemannian distance on $X$.
Thus, by shrinking $\theta$, we may apply lemma \ref{lem:stictily_decreasing} to obtain that
\begin{equation}\label{eq:domain}
\widetilde{e}\cdot \varphi_{\widetilde{v}}(M_r)\sub X_{(s+r)/2},
\end{equation}
where we have used that 
\[M_r\sub X_r.\]
Next, recall from the proof of lemma \ref{lemma:global_action} that the map 
\[\omega=\pi \circ  \varphi^{-1}_{\widetilde{v}}\circ  \widetilde{e}\in C^{\infty}_K(M,M)\]
satisfies
\[\|\omega\|_{k,K}
	\|\varphi_{\widetilde{v}}\|_{k,L}+\|\widetilde{e}\|_{k,K}.\]
Applying also lemma \ref{lemma:diffeo}, we obtain that there exists a constant $c>0$ such that 
$\|\omega^{-1}\|_{0,K}\leq (s-r)c\theta$.
Applying again lemma \ref{lem:stictily_decreasing} as before, we obtain that
\begin{equation}\label{eq:domain2}
\omega^{-1}(M_r)\sub M_{s}.
\end{equation}
Finally, we will apply lemma \ref{lemma:restricting_action} for the sets:
\[O'=M_{r}\sub O=M,\ \ U'=X_{(r+s)/2}\sub U=X, \  \ V'=M_{s}\sub V=M.\]
Note that, in the notation of lemma \ref{lemma:restricting_action}, \[f=\widetilde{e}\cdot \varphi_{\widetilde{v}} \  \ \textrm{and} \ \ \sigma=\omega^{-1},\] 
thus (\ref{eq:domain}) and (\ref{eq:domain2}) give the assumptions of the lemma:
\[f(O')\sub U'\ \ \ \textrm{and}\ \ \ \sigma(O')\sub V'.\]
We have that $\widetilde{e}|_{V'}=e$ and, by (\ref{eq:domain1}), that $\varphi_{\widetilde{v}}|_{U'}=\varphi_{v}|_{U'}$. Therefore, the lemma \ref{lemma:restricting_action} implies (\ref{eq:domain3}) which concludes the proof of \nameref{Lemma:B}.

\subsubsection{Proof of lemma C}
\label{sec:proof_of_lemma_C}

Let $\phi_{\nu}:X_{r_{\nu}}\to X_{r_{\nu-1}}$ be a sequence of maps satisfying the conditions of \nameref{Lemma:C}.
Let $K\sub X$ be a compact subset such that $X_1\sub \mathrm{int}(K)$, and consider the corresponding extension operators from subsection \ref{subsub:extension_maps}:
\[\sigma_{r}:\mathcal{U}_{X_r,K}\longrightarrow C^{\infty}_{K}(X,X), \ \ \  r\in[0,1].\]
By shrinking $\theta$, we may assume that $\phi_{\nu}\in \mathcal{U}_{X_{r_{\nu}},K}$.
Let
\[\widetilde{\phi}_{\nu}\defeq \sigma_{r_{\nu}}(\phi_{\nu})\in C^{\infty}_K(X,X).\]
By the properties of the extension operators (\ref{eq:ext}), we may assume that the sequence $\widetilde{\phi}_{\nu}$ satisfies the hypothesis of lemma \ref{lemma:C_global}.
Therefore the sequence
\[\widetilde{\psi}_{\nu}\defeq \widetilde{\phi}_{1}\circ \ldots \circ  \widetilde{\phi}_{\nu}\in C^{\infty}_K(X,X)\]
converges in all $C^k$-norms to a map $\widetilde{\psi}\in C^{\infty}_K(X,X)$, and moreover, using also (\ref{eq:ext}), estimates of the following form hold:
\begin{equation}\label{eq:proof_lemma_c_1}
\|\widetilde{\psi}\|_{k,K}
	\sum\nolimits_{\nu\geq 1}\|\widetilde{\phi}_{\nu}\|_{k,K}\lesssim
\sum\nolimits_{\nu\geq 1}\|\phi_{\nu}\|_{k,r_{\nu}}.
\end{equation}
In particular, $\|\widetilde{\psi}\|_{1,K}\leq c \theta$, for some constant $c$; thus, after shrinking $\theta$, lemma \ref{lemma:diffeo} implies that $\widetilde{\phi}$ is a diffeomorphism of $X$.

Finally, since $\widetilde{\phi}_{\nu}|_{X_{\nu}}=\phi_{\nu}$, and $\phi_{\nu}(X_{r_\nu})\sub X_{r_{\nu-1}}$, we have that:
\[\widetilde{\psi}_{\nu}|_{X_{r_{\nu}}}=\phi_1\circ \ldots \circ \phi_{\nu}.\]
We conclude that the sequence
\[\psi_{\nu}\defeq \phi_1\circ \ldots \circ \phi_{\nu}|_{X_{r_{\infty}}}:X_{r_{\infty}}\longrightarrow X_{r_{0}}\] converges in all $C^k$-norms to
\[\psi\defeq \widetilde{\psi}|_{X_{r_{\infty}}}:X_{r_{\infty}}\longrightarrow X_{r_{0}}.\]
Since $\widetilde{\psi}$ is a diffeomorphism, $\psi$ is a diffeomorphism onto its image. Hence equation (\ref{eq:proof_lemma_c_1}) implies the estimates of the statement.
This concludes the proof of \nameref{Lemma:C}.

\makeatletter
\@openrightfalse
\makeatother

\chapter*{Acknowledgments}
\addcontentsline{toc}{chapter}{Acknowledgments}

First I thank Marius for taking me in as a PhD student.
Marius, thank you for your continued effort, even though our communication has not always been as smooth as the things we study.
I wish we would have understood each other better.
I truly appreciate your love for maths, exceptional talent for exposition, what you do for the community, and the great atmosphere you create in your research group.

Next I thank Ionu\cb{t} for being my mathematics big brother and colleague.
Your clarity of thought and sheer tenacity for maths have always amazed me.
When we talked about maths, you could always explain to me what I meant.

To the members of the reading committee, Erik van den Ban, Marco Zambon, San Vũ Ng\d oc, Nguyen Tien Zung, and Eva Miranda:
thank you for the time and effort invested in reading this thesis, and for your comments and suggestions. It is not an easy read.

Next I thank mom, dad, Kirsten, Yasha, Maggie, Piotr, Ralph, Tera, João, Aline, Davide and Federica for their friendship and moral support.
\\
\emph{`The world breaks everyone and afterward many are strong at the broken places.'}
\\
I'm also thankful for the times with my other friends: Amparo, Arjen, Boris, Dana, Daniele, David, Dima, Florian, Francesco, Ionu\cb{t}, Ivan, Joey, Jules, KaYin, Lai, Lauran \& Laurent, Maria, Matias, Michelle, Mike, Ori, Pedro, Peer, Richard, Stefan, and Valentijn.
Soon I'll have time for you.

\makeatletter
\@openrightfalse
\makeatother

\chapter*{Summary}
\addcontentsline{toc}{chapter}{Summary}


\paragraph{Shapes and spaces}
This is a thesis about shapes and their appearance under gradual transformation.
Any person has a basic understanding of what shapes are: lines, triangles, circles, spheres, etc.
But in geometry we care about making truthful statements about shapes, and this requires precise assumptions (known as axioms).
Roughly, these assumptions should balance two things:
\begin{enumerate*}
    \item capture some human intuition of what shapes are,
    \item allow for interesting consequences through logical deduction.
\end{enumerate*}

A popular, and perhaps the most fundamental, description of a shape is a \emph{`(topological) space'}.
The basic building block of a space is the smallest and simplest of shapes: a point.
For a mathematician a point is not the tip of a pen or the dot it makes on a sheet of paper.
Instead, a \emph{`point'} is the abstract idea of an arbitrarily small dot, without surface or dimension (its dimension is zero).
From there, a space is a set (or collection, or class) of points with \emph{extra structure}.
This structure must then satisfy specific axioms that, roughly, capture a concept of distance and cohesion between points.

The concept of a space is the product of many years of research by great minds.
This has lead to a short list of axioms that are fairly easy to comprehend but have surprisingly strong consequences.
Perhaps its main strength is its ubiquity: all shapes we can think of are examples of spaces, but so are many abstract mathematical concepts.
This means that any statement proven about spaces will find application in many fields, also where we might not expect the importance of geometry.
But in this lies also its weakness: a statement about spaces cannot be very \emph{strong}, because it must be general.

A big part of maths is balancing the amount of generality (or abstraction):
too abstract, and we say nothing about everything; too specific, and we cannot see the trees for the forest.
One way to be more specific is restricting \emph{which} spaces we study.
For example, a (for me) essential restriction is the \emph{`Hausdorff'} condition:
it demands that any two points in the space can be sufficiently separated from each other.
Any shape you draw on a sheet of paper is Hausdorff, but many spaces are not, and this already shows just how general spaces are.

Another way to be more specific is to require the spaces have \emph{extra structure}.
A popular example of this is the concept of a \emph{`manifold'}:
for considerable time, and for most people, the earth is known to be approximatley spherical.
Yet for practical reasons we often describe its topography with charts, which are projections of the sphere on a \emph{flat} sheet.
No single chart can completely cover the earth without presenting locations on earth twice.
Moreover, cartography is a notoriously difficult subject:
the mapmaker must choose whether to preserve distance, or angle, or area, etc., based on the type of projection
(and transitions between types of charts is even more startling).
For these reasons it can be convenient to collect many \emph{`charts'} in an \emph{`atlas'}.
This is exactly the extra structure that comes with a manifold:
a collection of all projections (without double information) of the space onto something flat.

Many shapes we can think of are manifolds, but not all.
Actually, I mean that we all agree on \emph{what} the extra structure is --
we are not just restricting to a smaller class of shapes.
The purpose of adding extra structure was to be able to make stronger statements about the class of shapes.
And perhaps the single most important emergent property of manifolds is the concept of \emph{`dimension'}.
Points are zero dimensional, lines and circles one, and spheres (like the \emph{surface} of the earth) are two-dimensional.
And here the true advantage of this abstract approach to shape shows itself:
a circle on a sheet of paper and a circle in space both are one dimensional,
because it is just the same abstract circle embedded in different in different environments.
The trick here is that  \emph{`dimension'} is the number of coordinates that you need to make a chart.
For example, charts of the earth only have the directions up/down and left/right.
A coordinate is just a number, and dimension is just the number of independent coordinates we keep track of.
So we can now speak of shapes of any possible dimension, or even of infinite dimension.

From here geometers usually specialize by restricting \emph{which} charts are allowed in an atlas.
Some prefer charts that don't change distances, or volumes, or angles, etc.
These differences are large enough that they lead to entirely different fields of mathematics, where entire communitees and careers are based upon.
Perhaps the most important distinction lies between \emph{`algebraic geometry'}, where charts are polynomials, 
and \emph{`differential geometry'}, where charts are \emph{`smooth'}.
This thesis focusses on smooth manifolds.
Roughly, this means that the shapes do not have corners (circles and spheres are okay, but triangles and squares are not).
Slightly more technically, it means that charts have derivatives: we can talk about the speed and direction of a point moving around on the shape.
Just try to imagine what your speed in a corner of a triangular race track would be.

\paragraph{First theme: what are classes of shapes?}

We are mostly interested in \emph{classes} of shapes.
For example, the class of all triangles in a two-dimensional plane (an infinite sheet of paper),
or of only the equilateral triangles.
An interesting phenomenom is that such classes are again spaces themselves, by zooming out and thinking of each shape in the class as a point.
The extra structure of \emph{this} space describes the distance and cohesion between shapes.

The following is a classical example.
Think of the class of all straight lines in the plane passing through a common point.
These lines can be described by picking one line and rotating it with the special point as its axis.
After rotating 180 degrees the line comes back onto itself (it is now flipped, but still the same line).
This means that each line can be described by the \emph{angle} we should rotate, or in other words,
by the \emph{points} of a circle.
I mean to say that this class of lines is really just a circle.
You might say we only rotate 180 degrees while a circle is 360 degrees (should it not be half of a circle?).
The answer is: the point on the circle moves twice as fast as the line.

For the main theorems, which span the second half of this thesis, we have spent much thought on how to describe extra structure of classes of shapes.
The above example is particularly nice, because it is just a circle: a smooth finite-dimensional manifold. 
But often these spaces do not have finite dimension and their atlas is hard to describe.
So we prefer to find a different way of dealing with them.
We have opted to focus on what I would call \emph{`geometric structures'}:
shapes with extra structure that satisfies a \emph{smooth} set of rules. 
\\

\noindent
In technical terms, they are the sections of a submersion that satisfy a smooth partial differential equation.
But it cannot be any kind of equation because its interpretation should relate to geometry.
So in short the answer to the question of this paragraph is:
a class of geometric structures has the structure of a submersion with a differential equation of a geometric nature. 
I do not pretend this is a new idea, but only that it is important.

\paragraph{Second theme: how do we look at shapes?}

As a simple and somewhat silly example of a shape with extra structure, think of a triangle on a plane and now imagine it is colored (we may pick a color for each point of the triangle).
The extra structure is the painting of the triangle, and for each triangle there are many ways of painting it.

A more interesting example is a sphere (the surface of the earth), where at each point we put an arrow, describing the speed and direction of the wind.
Such extra structure is called a \emph{`vector field'}.

The first half of this thesis specializes on one particular class of structures, called \emph{`completely integrable systems'}.
Like how vector fields describe the motion of the wind, integrable systems describe the motion of a point of mass according to the laws of classical dynamics.
\emph{Completely} integrable systems are a particularly simple class.

In the previous paragraph we saw a bit of the recursive nature of geometry:
classes of shapes are themselves shapes.
Now I would like to point out another one:
shapes with extra structure are themselves just different shapes (without extra structure).
For example, think of a colored triangle on a plane, 
and suppose that, like me, you are somewhat color blind: 
the triangle is only colored in gray scale.
Another way to describe this triangle is the following: 
imagine that the shade of gray represents \emph{height}, and the triangle is  a piece of mountainous landscape sticking out of the plane.
Now we may forget about its color and only think of it as a misshapen triangle.
We are still studying the same shape, but are looking at it from a different perspective.
\\

\noindent
A more serious example would be to describe vector fields as embedded submanifolds of the tangent bundle (with some restrictions, but perhaps you may want to drop these).
And coming back to the first half of the thesis, perhaps you would prefer to think of integrable systems as Poisson maps, or even coistropic submanifolds of the product Poisson manifolds.

We may look at geometric structures as sections of a submersion, satisfying a differential equation, or as embedded submanifolds of the total space.
For the main theorems, we discovered it is useful to not pick one or the other perspective, but freely switch between the two.
Namely, the first perspective is often the one we natural start with,
while the second is useful for describing \emph{'symmetries'} of geometric structures.


\paragraph{Third theme: deformations \& symmetry}



Moving on, what does it mean to change a shape gradually?
For a triangle on a two-dimensional plane this is straightforward: you can move it around, rotate it, make it smaller or bigger, or change the lengths of any of its sides (and change its shape accordingly). 
And all of this should be done in a gradual manner, meaning no jumps and no jagged motions.
In a way, what it means to change something is always clear (it means going from one shape in the class to another), but the meaning of \emph{gradual} is extra information on the class of shapes: geometric structure.
For this thesis the meaning of gradual is somewhat fixed: it has to be completely smooth, meaning that the trajectory, its speed, acceleration, etc, are all continuous (i.e. without jumps).



As we go about pushing around the triangle, we can ask ourselves, are we really achieving something here? Are we \emph{truly} changing the triangle?
This is where the notion of symmetry enters the picture, and the answer is (as always) \emph{it depends on your perspective}.
At one extreme, we could say that any triangle is completely unique. 
Then any change really creates a new triangle, and not much can be said.
At the other extreme, we could say every triangle is essentially the same.
Then all change is inessential, and not much can be said.

A more sensible person might say that the location of the triangle doesn't matter, nor does its orientation. 
In other words, \emph{moving} and \emph{rotating} it does not matter, but changing its shape and size still does.

This makes us aware of two different viewpoints on symmetry.
We can describe which shapes are essentially the same, thus dividing them into different classes. 
In mathematics this is called \emph{`equivalence'}.
We can also describe precisely which forms of change are essential (and which are not).
In a way this is putting a constraint on the change, and means we are choosing a space of essential changes.
In the mathematical literature we associate with this the terms \emph{`(Lie) group'} and \emph{`action'}, or more generally \emph{`pseudogroup'}.
Our approach fits the latter: we study pseudogroups that describe the (to us) essential symmetries.
\\

\noindent
Putting this all together, we are studying the following:
\begin{enumerate}
    \item A space of shapes, or \emph{`submersion'} with \emph{`differential equation'}.
    \item A space of changes, or \emph{`pseudogroup'}.
\end{enumerate}

\makeatletter
\@openrightfalse
\makeatother
\chapter*{Samenvatting}
\addcontentsline{toc}{chapter}{Samenvatting}

\paragraph{Vormen en ruimtes}
Dit proefschrift gaat over ruimtelijke vormen en geleidelijke verandering. Iedereen heeft er wel een beeld bij wat vormen zijn: lijnen, driehoeken, cirkels, bollen, etc. Maar in de meetkunde willen we zorgvuldige uitspraken maken over vormen, en dit behoeft exacte aannames (zogenaamde axioma’s). Dergelijke aannames moeten grofweg in het volgende voorzien:
\begin{enumerate*}
\item een menselijke intuïtie omvatten van wat vormen zijn,
\item interessante gevolgen toelaten door middel van logische deductie.
\end{enumerate*}

Een populaire, en misschien de meest fundamentele, beschrijving van een vorm is een \emph{'(topologische) ruimte'}. De bouwsteen van een ruimte is de kleinste en meeste eenvoudige vorm: een punt. Voor een wiskundige is een punt niet de punt van een potlood, of de stip die het op een vel papier maakt. Namelijk, een \emph{`punt'} is het abstracte idee van een willekeurig kleine stip, die geen oppervlakte of dimensie heeft (of eigenlijk, dimensie nul). Nu is een \emph{ruimte} een verzameling (of collectie) van punten met \emph{extra structuur}. Deze structuur moet aan specifieke axioma's voldoen die grofweg de begrippen afstand en samenhang behelzen.

Het concept van een ruimte vloeit voort uit vele jaren onderzoek door grote geesten. Het resultaat is een klein lijstje aan axioma's die redelijk eenvoudig te begrijpen zijn, maar verrassend sterke gevolgen hebben. De omvang is wellicht de grootste kracht: alle vormen die wij kunnen bedenken, en nog veel meer abstracte wiskundige begrippen, zijn ruimtes. Dit betekent dat een uitspraak over ruimtes toepassing zal vinden in veel vakgebieden, ook sommigen waar je misschien meetkunde zou verwachten. Maar daar ligt ook de zwakte: een uitspraak over ruimtes kan nooit echt \emph{sterk} zijn, want het zal zo algemeen zijn.

Voor een groot deel wil je in de wiskunde juist dit balanceren: te abstract, en we zeggen niets over alles; te specifiek, en we zien door de bomen het bos niet meer. Eén manier om specifieker te zijn is door te specificeren \emph{welke} ruimtes we bestuderen. Bijvoorbeeld, voor mij is \emph{'Hausdorff'} een noodzakelijke conditie: het betekent dat twee punten altijd een stukje van elkaar afliggen. Elke vorm die je op een vel papier tekent is Hausdorff, maar niet elke ruimte, en dit laat al zien hoe belachelijk algemeen ruimtes zijn.

Een andere manier om specifieker te zijn is door \emph{extra structuur} te eisen. Een populair voorbeeld hiervan is het begrip \emph{`variëteit'}: voor enige tijd, en voor de meeste mensen, is het bekent dat de aarde ongeveer een bol is. Maar toch is het vaak praktisch om z'n topografie met kaarten te beschrijven, welke projecties van de bol op een \emph{plat} vlak zijn. Geen enkele kaart kan de aarde volledig bedekken zonder sommige locaties dubbel aan te geven. Verder is cartografie een berucht moeilijk onderwerp: de kaartenmaker moet kiezen of ze óf afstanden, óf hoeken, óf inhoud wil behouden, afhankelijk van het soort kaart (en dan nog maar niet te spreken over de overgangen tussen verschillende soorten kaarten). Om deze redenen is het praktisch om \emph{`kaarten'} te verzamelen in een \emph{`atlas'}. Dit is exact de extra structuur van een variëteit: een collectie van projecties (zonder dubbele informatie) van de ruimte naar iets plats.

Veel vormen die we kunnen bedenken zijn variëteiten, maar niet alle. Wat ik hiermee bedoel is dat het voor iedereen duidelijk is \emph{wat} de extra structuur is -- we kijken niet alleen naar minder vormen. 
\\

\noindent
Het doel van extra structuur toevoegen was om we sterkere uitspraken te kunnen maken. Wellicht de belangrijkste extra eigenschap van een variëteit is het begrip \emph{`dimensie'}. Voor een wiskundige zijn punten nul-, lijnen en cirkels één-, en bollen tweedimensionaal. Hier toont het voordeel van een abstracte beschrijving van vormen zich: een cirkel op een vel papier en een cirkel in de ruimte zijn beiden eendimensionaal, want het is dezelfde abstracte cirkel die zich simpelweg in andere omstandigheden bevindt. 

De truc hier is dat \emph{`dimensie'} óók een abstract begrip is: het is aantal coördinaten dat je minimaal nodig hebt om een kaart te kunnen maken. Bijvoorbeeld, op een kaart van de aarde heb je de richten opwaarts/neerwaarts en linksom/rechtsom, dus een bol is tweedimensionaal. Een coördinaat is gewoon een getal, en dimensie is gewoon het aantal van zulke getallen. We kunnen nu dus praten over vormen van elke willekeurige dimensie, zelfs oneindig.

Vanaf hier specialiseren meetkundigen zich vaak verder door aan te geven \emph{welke} kaarten toegestaan zijn. Sommigen hebben voorkeur voor kaarten die afstanden, of hoeken, of volumes, enz., behouden. Deze verschillen zijn soms groot genoeg dat gehele vakgebieden op het een of het ander gespecialiseerd zijn. Het grootste verschil bestaat tussen de \emph{`algebraïsche meetkunde'}, waar alle kaarten polynomen zijn, en de \emph{`differentieerbare meetkunde'} waar alle kaarten \emph{`glad'} zijn. In dit proefschrift kijken we naar gladde variëteiten. Dit betekent grofweg dat de vormen geen hoeken hebben (cirkels zijn oké, maar driehoeken dan weer niet). Iets meer technisch betekent dit dat kaarten afgeleiden hebben: we kunnen praten over de snelheid en richting van een punt dat door een vorm beweegt. Probeer je maar eens voor te stellen wat je snelheid zou moeten zijn om door de hoeken van een driehoekig parcour te racen.

\paragraph{Eerste thema: wat zijn klassen van vormen?}
We zijn voornamelijk geïnteresseerd in \emph{klassen} van vormen. We kijken bijvoorbeeld naar alle driehoeken op een vlak, of alleen naar de gelijkzijdige driehoeken. Wat dit interessant maakt is dat zulke klassen zelf ook weer ruimtes zijn. Dit zie je door uit te zoomen en te doen alsof elke vorm in de klasse een punt is. De extra structuur op \emph{deze} ruimte beschrijft de afstand en samenhang tussen vormen.

Hier volgt een klassiek voorbeeld. Stel je alle rechte lijnen in een vlak voor die door een gemeenschappelijk punt gaan. Deze klasse van lijnen kan ook beschreven worden door één van de lijnen uit te kiezen en deze rond dat punt te draaien. Na 180 graden te draaien komt de lijn weer op zichzelf terecht (hij is nu omgekeerd, maar dat maakt niet uit). Dit betekent dat elke lijn wordt beschreven door de \emph{hoek} die we moeten draaien, of in andere woorden, door een \emph{punt} van een cirkel. Ofwel, deze klasse van lijnen is gewoon een cirkel. Je kunt hier bezwaar op hebben omdat we maar 180 graden draaien, in plaats van 360 (is het niet een halve cirkel?). Het antwoord hierop luidt: de punt op de cirkel draait twee keer zo snel als de lijn. 

Voor onze hoofdstellingen, die de tweede helft van dit proefschrift beslaan, hebben we lang nagedacht over de extra structuur die aanwezig is op een klasse van vormen. Het bovenstaande voorbeeld is bijzonder mooi omdat het slechtst een cirkel is: een eindigdimensionale gladde variëteit. Maar vaak zijn deze ruimten van klassen niet eindig dimensionaal en zijn de kaarten erg ingewikkeld. Daarom hebben we ervoor gekozen om ze op een andere manier te beschrijven. We werken met zogenaamde \emph{`meetkundige structuren'}: vormen met extra structuur die aan een \emph{gladde} lijst regels voldoet. 
\\

\noindent
In technische termen zijn ze de sneden van een submersie die aan een gladde partiële differentiaalvergelijking voldoen. Maar dit kan niet gewoon elke vergelijking zijn: het moet een meetkundige interpretatie hebben.
Dus kortweg is het antwoord op de vraag van deze paragraaf: een klasse van \emph{meetkundige structuren} heeft de extra structuur van een submersie met een differentiaalvergelijking van meetkundige aard. Ik beweer niet dat dit nieuwe ideeën zijn, enkel dat ze belangrijk zijn voor dit proefschrift.

\paragraph{Tweede thema: hoe kijken we naar vormen?} 
Stel je, als een simpel en misschien wat sullig voorbeeld van een vorm met extra structuur, een driehoek voor waarbij we elk punt een kleur geven. 
Een vergelijkbaar maar interessanter voorbeeld is een bol (het oppervlak van de aarde) waarbij we aan alle punten pijlen hangen die de snelheid en richting van de wind beschrijven. Een dergelijke extra structuur heet een \emph{`vectorveld'}.

In de vorige paragraaf zagen we de recursieve aard van meetkunde opduiken:
klassen van vormen zijn zelf weer vormen.
Nu wil ik graag op een ander soort recursie wijzen:
vormen met extra structuur zijn zelf weer andere vormen (zonder extra structuur).
Bijvoorbeeld, stel je weer de gekleurde driehoek voor en dat je, net als ik, wat kleurenblind bent: de driehoek heeft alleen grijstinten.
De driehoek kan nu een berglandschap voorstellen, waarbij elke grijstint een bepaalde hoogte aangeeft. Nu kunnen we de kleur van de driehoek vergeten en simpelweg deze vervormde driehoek onthouden. We bestuderen nog dezelfde vorm, maar vanuit een ander perspectief.
\\

\noindent
Een serieuzer voorbeeld is om vectorvelden als deelvariëteiten van de raakbundel te interpreteren (en mogelijk de extra eigenschappen te vergeten). Nog een voorbeeld is, om even op de eerste helft van het proefschrift terug te komen, om intgreerbare systemen als Poissonafbeeldingen, of zelfs als coïsotrope deelvariëteiten van de product Poissonvariëteit, te beschouwen.

Kortom, we kunnen meetkundige structuren als sneden van een submersie beschouwen, of als deelvariëteiten van de totale ruimte. 
Voor onze hoofdstellingen bleek het handig te zijn om ons vrijelijk te bewegen tussen deze twee perspectieven. 
Een vraagstuk doet zich vaak vanuit het eerste perspectief voor, en het tweede is geschikter om \emph{`symmetriën'} mee te beschrijven.

\paragraph{Derde thema: deformaties \& symmetrie}
Wat betekent het om een vorm geleidelijk te veranderen?
Voor een driehoek in een vlak is dit wel duidelijk:
je kan het verplaatsen, draaien, kleiner of groter maken, of de lengtes van één van de zijdes aanpassen. En dit alles moet op een geleidelijke manier gedaan worden, dus zonder sprongen of schokkende bewegingen.
In een zekere zin is het altijd duidelijk wat het betekent om een vorm te veranderen (je gaat van één vorm naar een andere in de klasse), maar de betekenis van \emph{geleidelijk} is extra structuur op de klasse van vormen zelf.
In dit proefschrift ligt deze structuur min-of-meer vast: het moet volledig glad zijn, wat betekent dat de baan, de snelheid, de versnelling, enz., allemaal geen sprongen maken.

Als we zo bezig zijn de driehoek te veranderen kunnen we onszelf afvragen: 
wat bereiken we hier nu mee? Verandert er wel \emph{echt} wat aan de driehoek?
Hier duikt het begrip symmetrie op, en het antwoord is (zoals altijd): het hangt van je perspectief af. Aan de ene kant kun je zeggen dat elke driehoek uniek is -- dan geeft elke verandering een totaal nieuwe driehoek. Aan de andere kant kun je zeggen dat elke driehoek ongeveer hetzelfde is -- dan doet geen enkele verandering er toe. In beide gevallen kun je verder weinig zeggen.

Een redelijk mens zou zeggen dat alleen de locatie of stand van een driehoek er niet toe doen.
Of anders gezegd, dat het \emph{verplaatsen} en \emph{draaien} van een driehoek niet uitmaakt, maar het aanpassen van de zijden en grootte nog wel.

Dit maakt ons bewust van twee verschillende perspectieven op symmetrie.
Enerzijds kunnen we beschrijven welke vormen in essentie hetzelfde zijn, door vormen nog verder in klassen onder te verdelen.
In de wiskunde heet dit een \emph{`equivalentie'}.
Anderzijds kunnen we aangeven welke soorten veranderingen essentieel zijn (en welke niet). 
In de wiskunde associëren we dit met de begrippen \emph{`(Lie) groep'} en \emph{`actie'} of, algmener, het begrip \emph{`pseudogroup'}.
Onze aanpak is die laatste: we bestuderen pseudogroepen die (voor ons) de essentiële symmetriën vastleggen.
\\

\noindent
Al met al bestuderen we het volgende:
\begin{enumerate}

\item
Een ruimte van vormen: een \emph{`submersie'} met \emph{`differentiaalvergelijking'}.

\item
Een ruimte van veranderingen: een \emph{`pseudogroep'} op de submersie.

\end{enumerate}

\cleardoublepage
\phantomsection

\cleardoublepage
\phantomsection
\addcontentsline{toc}{chapter}{Index}
\printindex

\end{document}